\documentclass[10pt, reqno]{amsart}

\usepackage{graphicx}
\usepackage{times}
\usepackage[colorlinks=true,linkcolor=blue,citecolor=blue]{hyperref}%

\usepackage{xcolor}
\usepackage{hyperref}
\hypersetup{
    linktoc=all,
    colorlinks=true
}
\usepackage{stmaryrd}

\newtheorem{theorem}{Theorem}[section]

\newtheorem{lemma}[theorem]{Lemma}
\newtheorem{corollary}[theorem]{Corollary}
\newtheorem{prop}[theorem]{Proposition}

\newtheorem{ass}[theorem]{Assumption}

\newtheorem{pprop}[theorem]{Pseudo-Proposition}

\theoremstyle{definition}
\newtheorem{definition}[theorem]{Definition}

\theoremstyle{remark}
\newtheorem{remark}[theorem]{Remark}
\numberwithin{equation}{section}

\usepackage[top=0.5in, bottom=0.75in, left=0.75in, right=0.75in]{geometry}

\usepackage[scr]{rsfso}
\usepackage[english]{babel}
\usepackage{fancyhdr}
\usepackage{amsmath}
\usepackage{amsthm}
\usepackage{amssymb}
\usepackage{eucal}
\usepackage{comment}
\usepackage{enumitem}
\setlist{leftmargin=*}
\usepackage[integrals]{wasysym}
\newcommand\nc{\newcommand}
\nc{\on}{\operatorname}
\nc{\E}{\mathbf{E}}
\nc{\R}{\mathbb R}
\nc{\C}{\mathbb C}
\nc{\Q}{\mathbb Q}
\nc{\Z}{\mathbb Z}
\nc{\N}{\mathbb N}
\nc{\F}{\mathbb F}
\nc{\wt}{\widetilde}
\nc{\ol}{\overline}
\nc{\short}[3]{0 \longrightarrow #1 \longrightarrow #2 \longrightarrow #3 \longrightarrow 0}
\nc{\pd}[2]{\frac{\partial #1}{\partial #2}}
\nc{\rnc}{\renewcommand}
\nc{\e}{\varepsilon}
\nc{\DMO}{\DeclareMathOperator}
\nc{\grad}{\nabla}
\nc{\Exp}{\mathbf{Exp}}
\nc{\fsp}{\fontdimen2\font=2.1pt}

\rnc{\leq}{\leqslant}
\rnc{\geq}{\geqslant}
\rnc{\d}{\mathrm{d}}
\rnc{\O}{\mathrm{O}}
\rnc{\exp}{\mathbf{Exp}}
\newenvironment{nouppercase}{%
  \renewcommand{\uppercasenonmath}[1]{}}{}
\title{\Large Kardar-Parisi-Zhang Equation from Long-Range Exclusion Processes}

\author{\large Kevin Yang}

\usepackage{setspace}
\begin{document}
\setstretch{1.0}
\fsp
\begin{nouppercase}
\maketitle
\end{nouppercase}

\begin{abstract}
\fsp We prove here that the height function associated to non-simple exclusion processes with arbitrary jump-length converges to the solution of the Kardar-Parisi-Zhang SPDE under suitable scaling and renormalization. This extends the work of \cite{DT} for arbitrary jump-length and \cite{GJ16} for the non-stationary regime. Thus we answer a ``Big Picture Question" from the AIM workshop on KPZ and also expand on the almost empty set of non-integrable and non-stationary particle systems for which weak KPZ universality is proven. We use an approximate microscopic Cole-Hopf transform as in \cite{DT} but we develop tools to analyze local statistics of the particle system via local equilibrium and work of \cite{GJ16}. Local equilibrium is done via the one-block step in \cite{GPV} for path-space/\emph{dynamic} statistics.
\end{abstract}

{\hypersetup{linkcolor=blue}
\setcounter{tocdepth}{1}
\tableofcontents}

\section{Introduction}
\fsp The Kardar-Parisi-Zhang SPDE, which we call the KPZ equation, is an SPDE whose statistics are conjectured to be universal among a large class of rough dynamic interfaces including fluctuations of burning fronts, bacterial growth colonies, and crack formation. In the physics literature this was unrigorously shown in \cite{KPZ} but mathematical proof is a major open problem. To begin a precise discussion, we write down the KPZ equation with an effective diffusivity $\alpha\in\R_{>0}$/effective asymmetry $\alpha'\in\R$:
\small\begin{align}
\partial_{T} \mathbf{h} \ = \ \frac{\alpha}{2} \partial_{X}^{2} \mathbf{h} \ - \ \frac{\alpha'}{2} | \partial_{X} \mathbf{h} |^{2} \ + \ \alpha^{\frac12} \xi. \label{eq:KPZ}
\end{align}\normalsize\normalsize
The $\xi$-term is a Gaussian space-time white noise of delta-covariance $\E\xi_{T,X}\xi_{S,Y}=\delta_{T-S}\delta_{X-Y}$. The KPZ equation is a \emph{singular} SPDE because of both the roughness of the $\xi$-term and the nonlinear dependence on the slope $\partial_{X}\mathbf{h}$. To motivate the universality problem of interest here, both from a mathematical and physics perspective, we give the following interpretation of ``singular".
\begin{itemize}[leftmargin=*]
\item Consider a growth model $\wt{\mathbf{h}}$ given by the solution of a version of \eqref{eq:KPZ} but where the quadratic function of the slope is instead an arbitrary nonlinear function. To get the KPZ equation from the $\wt{\mathbf{h}}$-equation, we may Taylor expand the nonlinear function there as a function of the slope $\partial_{X}\wt{\mathbf{h}}$. The first two terms, which are constant and linear in $\partial_{X}\wt{\mathbf{h}}$, can be removed by elementary transformations, and the leading-order term that remains is the quadratic in \eqref{eq:KPZ}. Though this argument may hold water at a heuristic level, it ultimately gives an incorrect effective asymmetry $\alpha'$ because of the singular nature of \eqref{eq:KPZ}. Dealing with this singular equation rigorously is a major goal of \cite{Hai13} which was later generalized into a theory of regularity structures in \cite{Hai14}. Rigorous application of regularity structures to make the current bullet point on universality \emph{correct} is done in \cite{HQ}.
\item We emphasize that though regularity structures were successfully implemented to confirm universality of the KPZ equation for SPDE growth models in \cite{HQ}, it depends on the noise in these SPDEs to be space-time white noise. This allows one to build solutions of said SPDEs using explicit Gaussian-based objects. Work of \cite{HS} extends universality to more general continuum noises. Extensions of universality via regularity structures to a few types of semi-discrete noise are done in \cite{EH,Matetski}. However, this does not include noises that are relevant for non-simple exclusions. In particular, using regularity structures for universality in the context of general particle systems with a genuinely discrete flavor is open.
\end{itemize}
The previous bullet point illustrates difficulties/interesting aspects of the proposed universality of the KPZ equation. The rest of this introduction before we introduce the particle system of interest in this paper is organized as follows. First we introduce a solution to KPZ which will allow us to avoid dealing with the singular features of the KPZ equation \eqref{eq:KPZ}. Then we record a list of results concerning convergence to KPZ for special \emph{integrable} or \emph{solvable} particle systems. We conclude with progress beyond solvable models and the contributions of this paper. Later in this introduction section we will discuss additional background.

To deal with the singular nature of the KPZ equation, instead of using the regularity structures in \cite{Hai14} we will instead employ the following Cole-Hopf transform/solution from \cite{BG}. First, for $\alpha\in\R_{>0}$ and $\lambda\in\R$ we introduce the \emph{stochastic heat equation}:
\small\begin{align}
\partial_{T} \mathbf{Z} \ = \ \frac{\alpha}{2} \partial_{X}^{2} \mathbf{Z} \ - \ \lambda \alpha^{\frac12} \mathbf{Z}\xi. \label{eq:SHE}
\end{align}\normalsize\normalsize
%
\begin{itemize}[leftmargin=*]
\item We specialize to $\lambda=\alpha'\alpha^{-1}$ and \emph{define} the solution of the KPZ equation \eqref{eq:KPZ} as $\mathbf{h}\overset{\bullet}=-\lambda^{-1}\log\mathbf{Z}$. We clarify this log-transform is well-defined because \eqref{eq:SHE} admits a continuous solution via Ito calculus for its Duhamel form, and for positive initial data it remains positive with probability 1; see \cite{Mu}. We will also call the stochastic heat equation \eqref{eq:SHE} the ``SHE".
\item We must also include an infinite renormalization/counter-term in $\mathbf{h}$ to handle singular aspects of KPZ; see Remark \ref{remark:ch2RG}.
\end{itemize}
In \cite{BG} the authors show that the height function associated to a nearest-neighbor ASEP/"asymmetric simple exclusion process" model converges to the Cole-Hopf solution of KPZ, and since then a number of similar results were obtained in \cite{CGST}, \cite{CST}, and \cite{CT}, for example. The key input for these papers is that the height function for the microscopic particle system exhibits an \emph{algebraic duality}; the exponential of the height function satisfies an \emph{exact} microscopic version of the stochastic heat equation. Duality of such particle systems therefore provides a direction towards establishing the KPZ equation scaling limit for microscopic height functions while only encountering stochastic analysis of the SHE and thus \emph{without} directly addressing the singular features of the KPZ equation; recall from earlier that addressing these singular aspects for SPDEs in the context of particle systems is open. However duality is an indication of integrability or solvability, and the ``set of integrable models" is very sparse and rare among all particle systems, so universality of the KPZ equation is unlikely to be solved with ideas based solely on integrability.

In a step outside the set of integrable models, the authors in \cite{DT} prove universality of the KPZ equation for the height functions associated to a non-simple generalization of ASEP in \cite{BG}. The exponential of the height function solves a microscopic SHE as in \cite{BG} but with additional error terms. In \cite{DT}, the authors assume the maximal jump-length in the non-simple model is at most 3 in which case the aforementioned error terms can be addressed with standard ideas of hydrodynamic limits. In a nutshell, our contribution is dealing with these errors for arbitrary jump-lengths. The approach that we will take to doing this is analyzing these errors using a homogenization strategy based on ideas from the first step in a general ``Boltzmann-Gibbs principle" originally introduced in \cite{BR84}; we cite \cite{GJ15} for a refined version of this Boltzmann-Gibbs principle. The Boltzmann-Gibbs principle is generally accessible for systems with invariant measure initial conditions; for general non-stationary initial measures it becomes a difficult problem. One version of the Boltzmann-Gibbs principle was established in \cite{CYau} for non-equilibrium initial measures. However, the proofs within \cite{CYau} are not applicable in this article because of a few stochastic analytic difficulties for the SHE that come from the singular features of the KPZ equation. Thus, we introduce another mechanism to access parts of a non-equilibrium Boltzmann-Gibbs principle based on adapting the classical one-block scheme in \cite{GPV} to path-space/dynamic statistics of the particle system. This gives another \emph{probabilistic} and \emph{general} homogenization tool to establishing KPZ equation scaling limits for particle systems.

To reiterate, our first main goal is to extend \cite{DT} to arbitrary maximal jump-length, answering, in large part, one ``Big Picture Problem" from the AIM workshop on KPZ. The second is to develop general tools for non-equilibrium particle systems.
\subsection{The Model}
The particle system we study is a generalization of the non-simple exclusion process studied in \cite{DT} but for arbitrary-length steps in the particle random walk; we cite \cite{DT} for this entire subsection. We give its generator below.
\begin{itemize}[leftmargin=*]
\item Using spin-notation adopted previously in \cite{DT}, given any sub-lattice $\mathfrak{I} \subseteq \Z$ define $\Omega_{\mathfrak{I}} = \{\pm 1\}^{\mathfrak{I}}$. Observe such an association prescribes a mapping $\mathfrak{I} \mapsto \Omega_{\mathfrak{I}}$ for which any containment $\mathfrak{I}\subseteq\mathfrak{I}'$ induces the canonical projection $\Omega_{\mathfrak{I}'} \to \Omega_{\mathfrak{I}}$ given by
\small\begin{align}
\Pi_{\mathfrak{I}'\to\mathfrak{I}}: \Omega_{\mathfrak{I}'} \to \Omega_{\mathfrak{I}}, \quad (\eta_{x})_{x \in \mathfrak{I}'} \mapsto (\eta_{x})_{x \in \mathfrak{I}}.
\end{align}\normalsize\normalsize
We abuse notation and denote $\Pi_{\mathfrak{I}} \overset{\bullet}= \Pi_{\mathfrak{I}' \to \mathfrak{I}}$ for any $\mathfrak{I} \subseteq \mathfrak{I}' \subseteq \Z$. We adopt the physical interpretation that for any $\eta \in \Omega_{\Z}$, the value $\eta_{x} = 1$ indicates the presence of a particle located at $x \in \Z$ and that $\eta_{x} = -1$ indicates the absence of a particle. 
\item We introduce a maximal jump-length $\mathfrak{m}\in\Z_{>0}\cup\{+\infty\}$ and two sets of coefficients/speeds; note ellipticity $\alpha_{1} > 0$ below:
\small\begin{align}
\mathrm{A} \ = \ \left\{ \alpha_1,\ldots,\alpha_{\mathfrak{m}} \in \R_{\geq0}: \alpha_{1} > 0, \ {{\sum}}_{k = 1}^{\mathfrak{m}} \alpha_k = 1\right\}, \quad \Gamma \ = \ \left\{ \gamma_k \in \R \right\}_{k = 1,\ldots,\mathfrak{m}}.
\end{align}\normalsize\normalsize
For any pair of sites $x,y \in \Z$, we denote by $\mathfrak{S}_{x,y}$ the generator for a speed-1 symmetric exclusion process on the bond $\{x,y\}$. We specify the generator $\mathfrak{S}^{N,!!}$ of our dynamic for $N \in \Z_{>0}$ large via its action on a generic functional $\varphi: \Omega_{\Z} \to \R$:
\small\begin{align}
\mathfrak{S}^{N,!!}\varphi(\eta) \ \overset{\bullet}= \ N^{2} \sum_{k = 1}^{\mathfrak{m}} \alpha_{k} \sum_{x \in \Z} \left(\frac12+\frac12N^{-\frac12}\gamma_{k}\mathbf{1}_{\eta_{x}=-1}\mathbf{1}_{\eta_{x+k}=1} - \frac12N^{-\frac12}\gamma_{k}\mathbf{1}_{\eta_{x}=1}\mathbf{1}_{\eta_{x+k}=-1}\right) \mathfrak{S}_{x,x+k}\varphi(\eta).
\end{align}\normalsize\normalsize
We will denote by $\eta_{T}$ the particle configuration observed after time-$T$ evolution under the $\mathfrak{S}^{N,!!}$ dynamic. To be clear, every superscript $!$ denotes another scaling factor of $N$ for any operator for the entirety of this paper.
\end{itemize}
\begin{definition}
Provided any time $T \in \R_{\geq 0}$, let us define $\mathbf{h}^{N}_{T,0}$ to be $2$ times the net flux of particles across the bond $\{0,1\}$, with leftward traveling particles counting as positive flux. We also define the following \emph{height function} from \cite{DT}:
\begin{subequations}
\small\begin{align}
\mathbf{h}_{T,x}^{N} \ &\overset{\bullet}= \ \mathbf{h}_{T,0}^{N} \ + \ \mathbf{1}_{x\geq1}{{\sum}}_{y=1,\ldots,x} \eta_{T,y} \ - \ \mathbf{1}_{x<1}{{\sum}}_{y=x,\ldots,0}\eta_{T,y}.
\end{align}\normalsize\normalsize
\end{subequations}
\end{definition}
The height function $\mathbf{h}^{N}$ becomes the solution of the KPZ equation in the large-$N$ limit under appropriate renormalization. This is the main theorem of the current paper. Looking more closely at what showing this scaling limit entails, we address first what the effective diffusivity $\alpha\in\R_{>0}$ and the effective asymmetry $\alpha'\in\R$ should be. This tells us what the limit KPZ stochastic PDE for $\mathbf{h}^{N}$ should be, and this also tells us how to define the corresponding microscopic Cole-Hopf transform. Indeed, let us remark from our discussion of SHE above that defining the Cole-Hopf transform for the limit KPZ equation requires knowing the ratio $\alpha'\alpha^{-1}\in\R$. We will not perform the calculation in here, although to this end we reference the \emph{KPZ scaling theory} calculation in \cite{DT} immediately after (1.8) therein. Ultimately, the effective diffusivity and asymmetry are those of the particle random walk.
\begin{definition}
We define $\alpha \overset{\bullet}= {{\sum}}_{k=1}^{\mathfrak{m}}k^{2}\alpha_{k}$ and $\alpha'\overset{\bullet}={{\sum}}_{k=1}^{\mathfrak{m}}k\alpha_{k}\gamma_{k}$. We also define $\lambda\overset{\bullet}=\alpha'\alpha^{-1}$ following notation in \cite{DT}.
\end{definition}
\begin{definition}
Define the \emph{microscopic Cole-Hopf transform} denoted by $\mathbf{Z}^{N}$ by the following analog of the Cole-Hopf transform of the continuum SHE/KPZ equation. In the following, the growth speed $\mathfrak{v}_{N} \in \R$ is defined in (1.29) in \cite{DT} for arbitrary $\mathfrak{m}$:
\small\begin{align}
\mathbf{Z}_{T,x}^{N} \ \overset{\bullet}= \ \exp\left(-\lambda N^{-\frac12}\mathbf{h}_{T,x}^{N}+N\mathfrak{v}_{N}T\right).
\end{align}\normalsize\normalsize
We will realize $\mathbf{h}^{N}$ and $\mathbf{Z}^{N}$ as functions on $\R_{\geq0}\times\R$ via piecewise linear interpolation of their values on $\R_{\geq0}\times\Z$.
\end{definition}
\begin{remark}\label{remark:ch2RG}
Technically $\mathbf{Z}^{N}$ is the microscopic Cole-Hopf transform of the \emph{renormalized} height function $\mathbf{h}^{N}$ with counter-term of speed $N\mathfrak{v}_{N}\gg1$. The renormalization/counter-term is a microscopic indication of singular features of the KPZ equation.
\end{remark}
\subsection{Main Theorem}
The primary result of the paper is a scaling limit for $\mathbf{Z}^{N}$ under a large class of initial probability measures defined as follows. We emphasize this class of initial measures is also of interest in \cite{BG} and \cite{DT} for example.
\begin{definition}\label{definition:NS}
We say a probability measure $\mu_{0,N}$ on $\Omega_{\Z}$ is \emph{near stationary}, if the following moment bounds hold with respect to $\mu_{0,N}$ \emph{uniformly} in $N\in\Z_{>0}$ provided any $p \in \Z_{\geq 1}$ and $0 \leq\vartheta<\frac12$:
\small\begin{align}
\sup_{x \in \Z} \E|\mathbf{Z}_{0,x}^{N}|^{2p} + \sup_{x,y \in \Z} N^{2p\vartheta}\left(|x-y|^{-2p\vartheta}\E|\mathbf{Z}_{0,x}^{N}-\mathbf{Z}_{0,y}^{N}|^{2p}\right) \ \lesssim_{p,\vartheta} \ 1.
\end{align}\normalsize\normalsize
Moreover, we require $\mathbf{Z}_{0,NX}^{N} \to_{N\to\infty} \mathbf{Z}_{0,X}$ locally uniformly for some continuous initial data $\mathbf{Z}_{0,\bullet}:\R\to\R_{\geq0}$.
\end{definition}
Before we can present the main result, we must first introduce a few assumptions which we package as one. The first part of the following set of assumptions is a \emph{finite} maximal jump-length that does not depend on $N\in\Z_{>0}$. We comment on what may be done to relax this constraint though this just amounts to technical adjustments throughout a few different parts of the paper. We will take the assumption of finite maximal jump-length to make this paper more readable. The second part of the following set of assumptions is more serious. It asserts that the speed of the asymmetric jumps in the particle random walk approximately satisfies a linear constraint from \cite{DT}. We actually borrow and improve on such constraint for asymmetric jumps in the particle system from \cite{DT}. In spirit of universality, the constraint should not necessarily be there. It would be interesting to remove it.
\begin{ass}\label{ass:ass}
\fsp The maximal jump-length $\mathfrak{m}$ is uniformly bounded and independent of $N\in\Z_{>0}$. Moreover, we have the a priori estimate $\sup_{k=1,\ldots,\mathfrak{m}}\alpha_{k}|\gamma_{k}-\bar{\gamma}_{k}| \lesssim N^{-1/2}$, where the ``specialized" speeds of asymmetric jumps are given by 
\small\begin{align}
\alpha_{k} \bar{\gamma}_{k} \ &\overset{\bullet}= \ 2 \lambda{{\sum}}_{\ell = k}^{\mathfrak{m}}\alpha_{\ell}\left(\ell k^{-1}-1\right) \ + \ \lambda\alpha_{k}.
\end{align}\normalsize\normalsize
\end{ass}
In the following statement of our main result, we will employ the Skorokhod space of $\mathbf{D}_{1}\overset{\bullet}=\mathbf{D}([0,1],\mathbf{C}(\mathbb{K}))$ of cadlag paths valued in the Banach space of continuous functions $\mathbf{C}(\mathbb{K})$ on an arbitrary but fixed compact set $\mathbb{K}\subseteq\R$. We cite \cite{Bil} for details.
\begin{theorem}\label{theorem:KPZ}
\fsp Under near-stationary initial data, the process $\mathbf{Z}_{T,NX}^{N}$ is tight in the large-$N$ limit with respect to the Skorokhod topology on $\mathbf{D}_{1}$. All limit points are the solution to \emph{SHE} with parameters $\alpha,\lambda \in \R$ defined earlier and initial data $\mathbf{Z}_{0,\bullet}$.
\end{theorem}
\begin{itemize}[leftmargin=*]
\item There is nothing special about time $1$. It may be replaced by any fixed positive time. We also allow $\mathbb{K}$ to be any compact set.
\item In \cite{DT}, and even in \cite{BG} and related works, the class of near-stationary initial probability measures considered allow for the a priori estimates given in Definition \ref{definition:NS} to grow in some exponential-linear fashion. We did not allow for that here. However, our methods will still hold for this larger set of initial data. The differences are almost cosmetic and depend only on diffusive tails of a discretization of the classical Gaussian heat kernel as in \cite{DT}. We focus on the above class of initial data introduced in Definition \ref{definition:NS} without the exponential growth just to make this paper more readable.
\item We have assumed finite maximal jump-length for the particle random walk in our paper. We can actually remove this with a strategy based on the following outline \emph{if} we assume the sequences $\{\alpha_{k}\}_{k\in\Z_{>0}}$ and $\{\alpha_{k}\gamma_{k}\}_{k\in\Z_{>0}}$ have all moments as measures on $\Z_{>0}$. Take any $\e\in\R_{>0}$ as small as we want but independent of $N\in\Z_{>0}$. Because both $\{\alpha_{k}\}_{k\in\Z_{>0}}$ and $\{\alpha_{k}\gamma_{k}\}_{k\in\Z_{>0}}$ admit all moments, the speed of jump of length more than $N^{\e}$ is at most $\kappa_{C}N^{-C}$ for any $C\in\R_{>0}$, so roughly speaking we can forget about all jumps of length more than $N^{\e}$, and $\mathfrak{m}=N^{\e}$. We are almost in a situation of finite maximal jump-length, but $\mathfrak{m}=N^{\e}$ still grows in $N\in\Z_{>0}$ even if slowly. It turns out that the presence of $\mathfrak{m}=N^{\e}$ only affects estimates which are power-savings in $N\in\Z_{>0}$ and their dependence on $\mathfrak{m}=N^{\e}$ is polynomial. Thus these extra $N^{10\e}$ factors, for example, are negligible. 
\end{itemize}
\subsection{Narrow-Wedge Initial Measure}
The analysis in this paper can be adjusted to treat exclusion processes considered herein where the initial probability measure on the set of particle configurations is \emph{not} near-stationary or anything nearby, but rather the \emph{narrow-wedge} initial measure/configuration. This initial measure for the set of particle configurations gives rise to a microscopic Cole-Hopf transform with large-$N$ limit the Dirac point mass; for a detailed discussion, see \cite{ACQ} and \cite{C11}.

Because of the distributional nature of Dirac point masses the adjustments we need in order to treat exclusion processes here but with narrow-wedge initial measure are nontrivial and technical since they require adjustments of the hydrodynamic-limit-input of this paper and not just the stochastic analytic inputs as was the case for \cite{DT}. Adding these would require several detailed technical arguments. For this reason we will defer the extension to narrow-wedge initial measure to a separate paper. 
\subsection{Background}
We have mentioned already the open problem of implementing the theory of regularity structures outside stochastic PDEs in \cite{Hai14} and \cite{HQ} and to height functions of interacting particle systems. We now discuss a different approach to weak KPZ universality known as a \emph{theory of energy solutions}. This approach was thoroughly explored in \cite{GJ15} by Goncalves and Jara. The approach via energy solutions is designed around a nonlinear martingale problem for the KPZ equation or its avatar in the stochastic Burgers equation. As much of previous literature on hydrodynamic limits and their fluctuations for interacting particle systems was also based on martingale problems, such nonlinear martingale problem was engineered to fit in the ``same framework" to confirm the universality of \eqref{eq:KPZ} for many non-integrable interacting particle systems. In particular, it does not apply the Cole-Hopf transform and instead directly makes sense of the singular features of KPZ. The theory of energy solutions, however, depends on the model being at/\emph{very} close to some invariant measure to make sense of these singular features of KPZ. In particular, the theory of energy solutions applies only to \emph{stationary} interacting particle systems while much of the interest in the current work is proving a KPZ scaling limit for \emph{non-stationary} interacting particle systems. 
\subsection{Organization}
We will not be able to discuss the actual content of this paper until the end of Section \ref{section:Framework} at which point we will have set up an \emph{approximate} microscopic version of SHE for $\mathbf{Z}^{N}$. We instead give the following high-level outline for now.
\begin{itemize}[leftmargin=*]
\item In Section \ref{section:Framework}, we re-develop the framework in Section 2 of \cite{DT}. The strategy for the proof of Theorem \ref{theorem:KPZ} is also given. 
\item In Section \ref{section:HydroStuff}, we establish local equilibrium via entropy production in infinite-volume along with equilibrium calculations.
\item In Section \ref{section:Ctify}, we introduce compactification of the $\mathbf{Z}^{N}$-dynamics. This is effectively done by heat kernel estimates. 
\item In Section \ref{section:D1B}, we develop our main technical contribution which we call a \emph{dynamic variation} of the one-block strategy.
\item In Section \ref{section:KPZ1}, we establish preliminary time-regularity estimates for the microscopic Cole-Hopf transform $\mathbf{Z}^{N}$.
\item In Section \ref{section:KPZ2}, we use the dynamical one-block strategy and time-regularity in a multiscale analysis to get a ``key" estimate.
\item In Section \ref{section:KPZ3}, we prove Theorem \ref{theorem:KPZ} for near-stationary data using the above ``key" estimate. 
\end{itemize}
In the appendix we record auxiliary heat kernel and martingale estimates which are based on Proposition A.1/Corollary A.2 in \cite{DT} and Lemma 3.1 in \cite{DT}. We also include a list of notation in the appendix for consult while reading this paper. Lastly, this paper is long consequence of trying to provide enough explanation for ideas in here and clarify each of the many big points. In particular, we try to leave no seemingly abstract estimate/construction without a word about why it is helpful/what it is doing. We also often explain proofs of various results to clarify the technical details here which adds to the length of the paper as well.
\subsection{Acknowledgements}
The author thanks Amir Dembo for discussion and advice, and Li-Cheng Tsai and Stefano Olla for useful discussion. The author also thanks anonymous referees for immensely useful feedback on earlier versions.
\subsection{Comments on Notation}
We have an index for notation we use often in this paper in the appendix section of this paper. We give here a few pieces of notation from the aforementioned appendix section that are more commonly used.
\begin{itemize}[leftmargin=*]
\item A ``universal" constant is one that depends on nothing beyond possibly fixed data of the particle system, for example speeds $\{\alpha_{k}\}_{k=1}^{\mathfrak{m}}$. When we refer to a constant as ``arbitrarily/sufficiently small but universal", we mean arbitrarily/sufficiently small depending only on a uniformly bounded number of universal constants. The same is true for ``arbitrarily/sufficiently large but universal" constants except the word ``small" is replaced by ``large". The reader is invited to take ``arbitrarily/sufficiently small but universal" constants to be $999^{-999}$ and to take ``arbitrarily/sufficiently large but universal" constants to be $999^{999}$.
\item Provided any $a,b\in\R$, we define the discretized interval $\llbracket a,b\rrbracket=[a,b]\cap\Z$.
\item For any finite set $I$, let $\kappa_{I}\in\R$ be a constant depending \emph{only} on $I$. We also define a \emph{normalized} sum $\wt{{{\sum}}}_{i\in I} \overset{\bullet}= |I|^{-1}{{\sum}}_{i\in I}$.
\item The ``microscopic time-scale" is order $N^{-2}$ and the ``macroscopic time-scale" is order $1$. The ``mesoscopic time-scales" are any time-scales that are between these two time-scales. Similarly the ``microscopic length-scale" is order 1 and the ``macroscopic length-scale" is order $N$. The ``mesoscopic length-scales" are between these two length-scales.
\item Script font is used for operators. Fraktur font is used for particle system data. Bold font is used for PDE-type objects.
\item Finally, starting in Section \ref{section:KPZ1}, the subscript ``$\mathrm{st}$" for random times $\mathfrak{t}_{\mathrm{st}}\in\R_{\geq0}$ stands for ``stopping time".
\end{itemize}
%
%
%
\section{Approximate Microscopic Stochastic Heat Equation}\label{section:Framework}
We recap framework developed in \cite{DT} with adjustments catered to our analysis in this paper. We then provide an outline of overcoming obstacles discussed in \cite{DT} that limit the maximal jump-length in \cite{DT}. We will employ here invariant probability measures for relevant exclusion processes from Definition \ref{definition:ensembles}. These are canonical and grand-canonical ensembles.
\subsection{Quantities of Interest}
First we introduce notation for functionals for which we can conduct probabilistic analysis.
\begin{definition}
\fsp For any $\mathfrak{f}:\Omega_{\Z}\to\R$, define its \emph{support} as the smallest subset $\mathfrak{I}\subseteq\Z$ so that $\mathfrak{f}$ depends only on $\eta_{x}$ for $x\in\mathfrak{I}$.
\end{definition}
\begin{definition}
For $X\in\Z$, define $\tau_{X}:\Omega_{\Z}\to\Omega_{\Z}$ to shift a configuration so that $(\tau_{X}\eta)_{Z}=\eta_{Z+X}$ for all $(\eta,Z)\in\Omega_{\Z}\times\Z$.
\end{definition}
We define three classes of functionals/coefficients below. The first is a generalization of ``weakly vanishing" terms in \cite{DT}.
\begin{definition}
A functional $\mathfrak{w}: \R_{\geq 0} \times \Z \times \Omega_{\Z} \to \R$ is \emph{weakly vanishing} if $|\mathfrak{w}| \lesssim N^{-\beta}$ for some $\beta \in \R_{>0}$ universal, or if:
\begin{itemize}[leftmargin=*]
\item For all $(T,X,\eta) \in \R_{\geq 0} \times \Z \times \Omega_{\Z}$, we have $\mathfrak{w}_{T,X}(\eta) = \mathfrak{w}_{0,0}(\tau_{-X}\eta_{T})$ for some ``reference functional" $\mathfrak{w}_{0,0}: \Omega_{\Z} \to \R$.

\item We have $\E^{\mu_{0,\Z}} \mathfrak{w}_{0,0} = 0$, where $\mu_{0,\Z}$ is the product Bernoulli measure on $\Omega_{\Z}$ defined by $\E^{\mu_{0,\Z}}\eta_{x}=0$ for all $x \in \Z$.

\item We have the deterministic bound $|\mathfrak{w}_{0,0}| \lesssim 1$ uniformly in $N \in \Z_{>0}$ and $\eta \in \Omega_{\Z}$ with a universal implied constant. 
\item The support of $\mathfrak{w}_{0,0}$ is uniformly bounded, so that it is contained in an interval of length independent of $N\in\Z_{>0}$.
\end{itemize}
\end{definition}
\emph{At the level of hydrodynamic limits}, weakly vanishing terms are negligible as near-stationary initial measures imply that the global $\eta$-density is roughly $0$. This is the defining property of weakly vanishing terms in \cite{DT}. The one-block and two-blocks steps in the proof of Lemma 2.5 in \cite{DT} will also apply for weakly vanishing terms defined above. We return to this in the proof for Theorem \ref{theorem:KPZ}, but we note here weakly vanishing terms will not give any difficulties that were not already treated in \cite{DT}.
\begin{definition}
A functional $\mathfrak{g}: \R_{\geq 0} \times \Z \times \Omega_{\Z} \to \R$ is a \emph{pseudo-gradient} if the following conditions are satisfied:
\begin{itemize}[leftmargin=*]
\item For all $(T,X,\eta) \in \R_{\geq 0} \times \Z \times \Omega_{\Z}$, we have $\mathfrak{g}_{T,X}(\eta) = \mathfrak{g}_{0,0}(\tau_{-X}\eta_{T})$ with reference functional $\mathfrak{g}_{0,0}:\Omega_{\Z}\to\R$.
\item We have the deterministic bound $|\mathfrak{g}_{0,0}| \lesssim 1$ uniformly in $N \in \Z_{>0}$ and $\eta \in \Omega_{\Z}$. The support of $\mathfrak{g}_{0,0}$ is uniformly bounded.
\item We have $\E^{\mu_{\varrho,\mathfrak{I}}^{\mathrm{can}}} \mathfrak{g}_{0,0} = 0$ for any canonical ensemble parameter $\varrho \in \R$ and subset $\mathfrak{I} \subseteq \Z$ containing the support of $\mathfrak{g}_{0,0}$.
\end{itemize}
\end{definition}
If $\mathfrak{g}_{0,0}$ is a discrete gradient, so that $\mathfrak{g}_{0,0} = \tau_{-\mathfrak{j}}\mathfrak{f}_{0,0} - \mathfrak{f}_{0,0}$ for a local functional $\mathfrak{f}_{0,0}:\Omega_{\Z}\to\R$ satisfying the required uniform boundedness and support condition and $\mathfrak{j}\in\Z$ is uniformly bounded, then $\mathfrak{g}_{0,0}$ is the reference functional for a pseudo-gradient. To show this, uniform boundedness of $\mathfrak{j}\in\Z$, of $\mathfrak{f}_{0,0}$, and of the support of $\mathfrak{f}_{0,0}$ guarantee the second bullet point for $\mathfrak{g}_{0,0}$. To see the last bullet point for $\mathfrak{g}_{0,0}$, it suffices to see that the canonical ensembles in that third bullet point are ``invariant under shifts"; on any subset containing supports of $\tau_{-\mathfrak{j}}\mathfrak{f}_{0,0}$ and $\mathfrak{f}_{0,0}$, with respect to any canonical ensemble $\tau_{-\mathfrak{j}}\mathfrak{f}_{0,0}$ and $\mathfrak{f}_{0,0}$ are equal in law.

For a non-gradient example of a pseudo-gradient we give the cubic nonlinearity in Proposition 2.3 in \cite{DT} in the case where the maximal jump-length is $\mathfrak{m}\geq4$. This cubic functional evaluated at the particle system at $(0,0) \in \R_{\geq0}\times\Z$ satisfies required estimates for the reference functional in the second bullet point in the definition of pseudo-gradients that we gave above. This can be checked directly because this cubic nonlinearity is the sum of a uniformly bounded $\mathfrak{m}$-dependent number of difference of cubic monomials in spins contained in some neighborhood of length-scale at most $10\mathfrak{m}$. To justify the interesting vanishing-in-expectation requirement in the third bullet point in the definition above, we observe that canonical ensembles are invariant under swapping spins at deterministic points. Thus, if we take $\eta_{1}\eta_{2}\eta_{3}-\eta_{-1}\eta_{-2}\eta_{-3}$, for example, in expectation with respect to any canonical ensemble containing $\{\pm1,\pm2,\pm3\}$ we may replace $\eta_{-1}\eta_{-2}\eta_{-3} \mapsto \eta_{1}\eta_{2}\eta_{3}$ at the level of expectations.
\begin{definition}
A functional $\wt{\mathfrak{g}}: \R_{\geq 0} \times \Z \times \Omega_{\Z} \to \R$ is said to have a \emph{pseudo-gradient factor} if it is uniformly bounded and we have a factorization of functionals $\wt{\mathfrak{g}} = \mathfrak{g}\cdot\mathfrak{f}$ such that the following constraints are satisfied:
\begin{itemize}[leftmargin=*]
\item We have $\mathfrak{f}_{T,X}(\eta) = \mathfrak{f}_{0,0}(\tau_{-X}\eta_{T})$. The support of $\mathfrak{f}_{0,0}$ is bounded but may be $N$-dependent. We still require $|\mathfrak{f}_{0,0}| \lesssim 1$.
\item The factor $\mathfrak{g}$ is a pseudo-gradient, and the $\eta$-wise supports of $\mathfrak{g}_{0,0}$ and $\mathfrak{f}_{0,0}$ are disjoint subsets in $\Z$.
\item The factor $\mathfrak{f}_{0,0}$ is an average of terms that are each a product of $\eta$-spins times uniformly bounded/deterministic constants.
\end{itemize}
\end{definition}
Since the supports of the pseudo-gradient factor and the functional $\mathfrak{f}$ in the last class of functionals are disjoint subsets, it is easy to see $\wt{\mathfrak{g}}$-functionals above are also pseudo-gradients. However, observe that the support of the reference functional $\mathfrak{f}_{0,0}$ is allowed to grow with $N\in\Z_{>0}$ so the same is true of $\wt{\mathfrak{g}}$. Our analysis for pseudo-gradients deteriorates in the length-scale of their support so we will not be able to efficiently study these $\wt{\mathfrak{g}}$-terms via the same ideas. The point of introducing the previous class of functionals is to highlight the probing of a pseudo-gradient factor with uniformly bounded support we will need to do.
\subsection{Approximate SHE}
Dynamics of the microscopic Cole-Hopf transform will be driven by the following heat operators.
\begin{definition}\label{definition:HEAT}
Let $\mathbf{H}^{N}$ be the heat kernel with $\mathbf{H}_{S,S,x,y}^{N} = \mathbf{1}_{x=y}$ solving the semi-discrete parabolic equation below where the operator $\mathscr{L}^{!!}$, which is also defined below, acts on the backwards spatial variable of the heat kernel:
\small\begin{align}
\partial_{T}\mathbf{H}_{S,T,x,y}^{N} \ &= \ \mathscr{L}^{!!}\mathbf{H}_{S,T,x,y}^{N}
\end{align}\normalsize\normalsize
The $\mathscr{L}^{!!}$ operator is the discrete-type Laplacian defined below in which $\wt{\alpha}_{k} = \alpha_{k} + \mathscr{O}(N^{-1})$ are defined in Lemma 1.2 of \cite{DT}:
\small\begin{align}
\mathscr{L}^{!!} \ &\overset{\bullet}= \ 2^{-1}{{{\sum}}}_{k=1}^{\mathfrak{m}} \wt{\alpha}_{k} \Delta_{k}^{!!}.
\end{align}\normalsize\normalsize
We introduced $\Delta_{k}\varphi_{x} \overset{\bullet}= \varphi_{x+k}+\varphi_{x-k}-2\varphi_{x}$ and $\Delta_{k}^{!!} \overset{\bullet}= N^{2}\Delta_{k}$ given any function $\varphi: \Z \to \R$. Additionally for $(T,x) \in \R_{\geq 0}\times\Z$, we define space-time and spatial heat/convolution-operators acting on space-time test functions $\varphi: \R_{\geq 0} \times \Z \to \R$:
\small\begin{align}
\mathbf{H}_{T,x}^{N}\varphi \ &\overset{\bullet}= \ \mathbf{H}_{T,x}^{N}\varphi_{S,y} \ \overset{\bullet}= \ \mathbf{H}_{T,x}^{N}\varphi_{\bullet,\bullet} \ \overset{\bullet}= \ \int_{0}^{T}{{{\sum}}}_{y\in\Z}\mathbf{H}_{S,T,x,y}^{N}\cdot\varphi_{S,y} \ \d S \\
\mathbf{H}_{T,x}^{N,\mathbf{X}}\varphi \ &\overset{\bullet}= \ \mathbf{H}_{T,x}^{N,\mathbf{X}}\varphi_{0,y} \ \overset{\bullet}= \ \mathbf{H}_{T,x}^{N,\mathbf{X}}\varphi_{0,\bullet} \ \overset{\bullet}= \ {{{\sum}}}_{y\in\Z}\mathbf{H}_{0,T,x,y}^{N} \cdot \varphi_{0,y}.
\end{align}\normalsize\normalsize
We occasionally use convexity of the spatial heat operator, so for any norm $\|\|$ and $p\in\R_{\geq1}$ we have $\|\mathbf{H}_{T,x}^{N,\mathbf{X}}\varphi\|^{p} \leq \mathbf{H}_{T,x}^{N}\|\varphi\|^{p}$.
\end{definition}
The main result of this section is an SDE-type equation for $\mathbf{Z}^{N}$; it borrows largely from Section 2 of \cite{DT}.
\begin{prop}\label{prop:Duhamel}
\fsp Consider $\e_{X,1}>0$ arbitrarily small but universal and $\beta_{X} \overset{\bullet}= \frac13+\e_{X,1}$. We have, with notation defined after,
\begin{subequations}
\small\begin{align}
\d\mathbf{Z}_{T,x}^{N} \ &= \ \mathscr{L}^{!!}\mathbf{Z}_{T,x}^{N}\d T + \mathbf{Z}_{T,x}^{N}\d\xi_{T,x}^{N} + \Phi_{T,x}^{N,2}\d T + \Phi_{T,x}^{N,3}\d T \label{eq:Duhamel} \\
\mathbf{Z}_{T,x}^{N} \ &= \ \mathbf{H}_{T,x}^{N,\mathbf{X}}\mathbf{Z}_{0,\bullet}^{N} + \mathbf{H}_{T,x}^{N}(\mathbf{Z}^{N}\d\xi^{N}) + \mathbf{H}_{T,x}^{N}(\Phi^{N,2}) + \mathbf{H}_{T,x}^{N}(\Phi^{N,3}). \label{eq:IntGT}
\end{align}\normalsize\normalsize
\end{subequations}
The first $\Phi^{N,2}$-term contains the ``pseudo-gradient content" in the $\mathbf{Z}^{N}$-equation. We use notation to be defined after:
\small\begin{align}
\Phi^{N,2}_{T,x} \ &\overset{\bullet}= \ N^{\frac12}\mathscr{A}_{N^{\beta_{X}}}^{\mathbf{X},-}(\mathfrak{g}_{T,x}) \cdot \mathbf{Z}_{T,x}^{N} + N^{\beta_{X}}\left(\wt{{{\sum}}}_{\mathfrak{l}=1,\ldots,N^{\beta_{X}}} \wt{\mathfrak{g}}_{T,x}^{\mathfrak{l}}\right)\mathbf{Z}_{T,x}^{N} +  N^{-\frac12}\wt{{{\sum}}}_{\mathfrak{l}=1,\ldots,N^{\beta_{X}}} \grad_{-7\mathfrak{l}\mathfrak{m}}^{!}(\mathfrak{b}_{T,x}^{\mathfrak{l}}\mathbf{Z}_{T,x}^{N})
\end{align}\normalsize\normalsize
The $\Phi^{N,3}$-term contains the ``weakly-vanishing content" in the $\mathbf{Z}^{N}$-equation for which we also use notation defined after:
\small\begin{align}
\Phi^{N,3}_{T,x} \ &\overset{\bullet}= \ \mathfrak{w}_{T,x}\mathbf{Z}_{T,x}^{N} + {{{\sum}}}_{k=-2\mathfrak{m}}^{2\mathfrak{m}}c_{k}\grad_{k}^{!}(\mathfrak{w}_{T,x}^{k}\mathbf{Z}_{T,x}^{N}).
\end{align}\normalsize\normalsize
%
\begin{itemize}[leftmargin=*]
\item The martingale integrator $\d\xi^{N}$ is defined in \emph{(2.4)} of \cite{DT}. It generalizes to any maximal-length in straightforward fashion.
\item We define a spatial-average ``operator" where $\mathfrak{g}_{T,x} = \mathfrak{g}_{0,0}(\tau_{-x}\eta_{T})$ is a pseudo-gradient whose support has size at most $5\mathfrak{m}$:
\small\begin{align}
\mathscr{A}_{N^{\beta_{X}}}^{\mathbf{X},-}(\mathfrak{g}_{T,x}) \ &\overset{\bullet}= \ \wt{{{\sum}}}_{\mathfrak{l}=1,\ldots,N^{\beta_{X}}}\tau_{-7\mathfrak{l}\mathfrak{m}}\mathfrak{g}_{T,x}.
\end{align}\normalsize\normalsize
The summands in the $\mathscr{A}^{\mathbf{X},-}(\mathfrak{g})$-term have disjoint supports as spatial shifts are with respect to multiples of $7\mathfrak{m}$. The superscript ``$\mathbf{X},-$" emphasizes a spatial average in the negative spatial direction. The support of $\mathscr{A}^{\mathbf{X},-}(\mathfrak{g})$ has size $\mathscr{O}(N^{\beta_{X}})$.
\item The terms $\wt{\mathfrak{g}}^{\mathfrak{l}}_{T,x} = \wt{\mathfrak{g}}^{\mathfrak{l}}_{0,0}(\tau_{-x}\eta_{T})$ have support of length $\mathscr{O}(N^{\beta_{X}})$. Each $\wt{\mathfrak{g}}^{\mathfrak{l}}$-term admits a pseudo-gradient factor.
\item The terms $\mathfrak{b}^{\mathfrak{l}}$ are uniformly bounded. The term $\mathfrak{w}$ is the sum of a uniformly bounded number of weakly vanishing terms.
\item The terms $\mathfrak{w}^{k}$ for $|k|\leq2\mathfrak{m}$ are weakly vanishing, and the deterministic coefficients $c_{k} \in \R$ are uniformly bounded.
\item Given $k \in \Z$, define the discrete gradient $\grad_{k}\varphi_{x} \overset{\bullet}= \varphi_{x+k} - \varphi_{x}$ for any $\varphi: \Z \to \R$ and its continuum rescaling $\grad_{k}^{!} \overset{\bullet}= N\grad_{k}$.
\end{itemize}
\end{prop}
\begin{remark}\label{remark:ch2Duhamel}
The multiples $5\mathfrak{m}$ and $7\mathfrak{m}$ can be replaced with any uniformly bounded multiples; it will not change our analysis of \eqref{eq:Duhamel}. All we need is $\mathscr{A}^{\mathbf{X},-}$ in Proposition \ref{prop:Duhamel} is an average of pseudo-gradients with \emph{disjoint} supports.
\end{remark}
\begin{remark}
With respect to any canonical ensemble pseudo-gradients are fluctuations, so averaging them in $\mathscr{A}^{\mathbf{X},-}$ above will provide better cancellation as the length-scale of averaging increases. The error in replacing the pseudo-gradient $\mathfrak{g}$ by a spatial-average will grow in the length-scale of spatial-averaging which is where the second term in $\Phi^{N,2}$ comes from; this will come in the proof of Proposition \ref{prop:Duhamel}. Lastly, cancellation at canonical ensembles will eventually give cancellation at non-equilibrium.
\end{remark}
\begin{lemma}\label{lemma:Duhamel1}
 We have the exact identity ${{\sum}}_{k=1}^{\mathfrak{m}} k \alpha_{k}\gamma_{k} = {{\sum}}_{k=1}^{\mathfrak{m}}k\alpha_{k}\bar{\gamma}_{k}$.
\end{lemma}
\begin{proof}[Proof of \emph{Proposition \ref{prop:Duhamel}}]
The stochastic equation \eqref{eq:Duhamel} of SDE-type is derived using almost entirely Proposition 2.3 from \cite{DT} and the derivation of $\mathbf{Z}^{N}$-dynamics in \cite{DT} that was done prior to Proposition 2.3 in \cite{DT}. We provide only extra ingredients in additional spatial averaging and reorganizing with Taylor expansion of the exponential formula for the microscopic Cole-Hopf transform similar to the proof for Proposition 2.3 in \cite{DT}. We give the additional ingredients below starting with notation. 
\begin{itemize}[leftmargin=*]
\item Define $\mathfrak{f}_{\mathfrak{k}}^{\mathfrak{m}}\overset{\bullet}=\tau_{\mathfrak{k}\mathfrak{m}}\mathfrak{f}_{T,x}$ for any $\mathfrak{k}\in\Z$ and $\mathfrak{f}:\Omega_{\Z}\to\R$ as we only shift by $\mathfrak{m}$-multiples. We also define $\grad_{\mathfrak{k}}^{\mathfrak{m}}\overset{\bullet}=\grad_{\mathfrak{k}\mathfrak{m}}$.
\item We also declare that all functionals, including the microscopic Cole-Hopf transform, are evaluated at $(T,x)$.
\end{itemize}
Proposition 2.3 from \cite{DT} gives the desired result but with cubic nonlinearities in place of the $(\bar{\mathbf{Z}}^{N})^{-1}\Phi^{N,2}$-terms, up to another difference that we address at the end of this proof that concerns quadratic polynomials in spins. The cubic nonlinearities have support contained in $\llbracket-\mathfrak{m},\mathfrak{m}\rrbracket\subseteq\Z$. We will replace this cubic nonlinearity with its spatial average on the length-scale $N^{\beta_{X}}$. The error terms, which come from Taylor expansion, are the remaining terms in $\Phi^{N,2}$ and additional weakly vanishing terms.
 
The cubic nonlinearities in Proposition 2.3 of \cite{DT} are pseudo-gradients as canonical ensembles are permutation-invariant. If $\mathfrak{c}$ denotes the contribution of these nonlinearities, an elementary discrete ``product/Leibniz rule" gets the following identities for which we define $\wt{\mathfrak{c}}\overset{\bullet}=\mathfrak{c}_{-2}^{\mathfrak{m}}$ and for which we give a little more explanation afterwards:
\small\begin{align}
N^{\frac12}\mathfrak{c}\mathbf{Z}^{N} \ &= \ N^{\frac12}\mathfrak{c}^{\mathfrak{m}}_{-2}\mathbf{Z}^{N} + N^{\frac12}\mathfrak{c}_{-2}^{\mathfrak{m}}\left(\grad_{-2}^{\mathfrak{m}}\mathbf{Z}^{N}\right) - N^{\frac12}\grad_{-2}^{\mathfrak{m}}(\mathfrak{c}\mathbf{Z}^{N}) \label{eq:Duhamel1} \\
&= \ N^{\frac12}\wt{\mathfrak{c}}_{-7\mathfrak{l}}^{\mathfrak{m}}\mathbf{Z}^{N} \ + \ N^{\frac12}\wt{\mathfrak{c}}_{-7\mathfrak{l}}^{\mathfrak{m}}\left(\grad_{-7\mathfrak{l}}^{\mathfrak{m}}\mathbf{Z}^{N}\right) - N^{\frac12}\grad_{-7\mathfrak{l}}^{\mathfrak{m}}\left(\wt{\mathfrak{c}}\mathbf{Z}^{N}\right) + N^{\frac12}\mathfrak{c}_{-2}^{\mathfrak{m}}\left(\grad_{-2}^{\mathfrak{m}}\mathbf{Z}^{N}\right) - N^{\frac12}\grad_{-2}^{\mathfrak{m}}(\mathfrak{c}\mathbf{Z}^{N}). \label{eq:Duhamel2}
\end{align}\normalsize\normalsize
The second line \eqref{eq:Duhamel2} follows from applying the same discrete product/Leibniz rule used to obtain \eqref{eq:Duhamel1} but applied to the first term within the RHS of \eqref{eq:Duhamel1} and with $-2$ replaced by $-7\mathfrak{l}$. We will now match each term in \eqref{eq:Duhamel} to either a weakly vanishing term contributing to $\Phi^{N,3}$ or one of the terms in $\Phi^{N,2}$ at least after averaging \eqref{eq:Duhamel2} over $\mathfrak{l}\in\llbracket1,N^{\beta_{X}}\rrbracket$.
\begin{itemize}[leftmargin=*]
\item Up to uniformly bounded error terms that are of order $N^{-1/2}$, the second term within the RHS of \eqref{eq:Duhamel1} is the product of linear polynomials in the spins and the cubic polynomial $\mathfrak{c}$, all supported in some neighborhood with size at most $5\mathfrak{m}$. We emphasize there is no $N^{1/2}$ in this product. This product is also without any constant term because multiplying the cubic polynomial by a linear polynomial cannot cancel spin-factors to obtain a constant. As it is a polynomial in spins supported in the same set of size $5\mathfrak{m}$, it is weakly vanishing as any polynomial of spins without constant term vanishes in expectation with respect to $\mu_{0,\Z}$.
\item The final term from the RHS of \eqref{eq:Duhamel1} has the form of a gradient term in $\Phi^{N,3}$ for $\mathfrak{k}=-2\mathfrak{m}$ since it is an \emph{unscaled} gradient acting on the product of a uniformly bounded functional and $\mathbf{Z}^{N}$ which is then multiplied by the $N^{1/2}$-factor.
\item We move to the first three terms in \eqref{eq:Duhamel2}. Defining $\mathfrak{b}^{\mathfrak{l}} \overset{\bullet}= \wt{\mathfrak{c}}$ turns the third term in \eqref{eq:Duhamel2} into a gradient-term in $\Phi^{N,2}$.
\item As $\mathfrak{c}$ is a pseudo-gradient so is $\wt{\mathfrak{c}}=\mathfrak{c}_{-2}^{\mathfrak{m}}$. Thus the first term in \eqref{eq:Duhamel2} is the first term in $\Phi^{N,2}$ after we average over $\mathfrak{l}\in\llbracket1,N^{\beta_{X}}\rrbracket$.
\item Analysis of the cubic nonlinearity $\mathfrak{c}$ in Proposition 2.3 of \cite{DT} now amounts to computing the second term in \eqref{eq:Duhamel2}. For this we compute the gradient of the microscopic Cole-Hopf transform by Taylor expansion of its exponential formula in terms of $\eta$-spins. Taylor expansion in this fashion is done in the proof of Proposition 2.3 in \cite{DT} for example. The result of such Taylor expansion gives a representation of the second term in \eqref{eq:Duhamel2} that we describe as follows. First we emphasize the support of the $\wt{\mathfrak{c}}_{-7\mathfrak{l}}^{\mathfrak{m}}$-factor in the second term in \eqref{eq:Duhamel2} is contained strictly to the left of $x-7\mathfrak{l}\mathfrak{m}$. Taylor expansion gives
\small\begin{align}
N^{\frac12}\wt{\mathfrak{c}}_{-7\mathfrak{l}}^{\mathfrak{m}}\left(\grad_{-7\mathfrak{l}}^{\mathfrak{m}}\mathbf{Z}^{N}\right) \ &= \ {\sum}_{\mathfrak{k}=1}^{\infty} N^{-\frac12\mathfrak{k}+\frac12}\lambda^{\mathfrak{k}}(\mathfrak{k}!)^{-1}\left(\wt{\mathfrak{c}}_{-7\mathfrak{l}}^{\mathfrak{m}}\left(\grad_{-7\mathfrak{l}}^{\mathfrak{m}}\mathbf{h}^{N}\right)^{\mathfrak{k}}\right) \\
&= \ {\sum}_{\mathfrak{k}=1}^{\infty} N^{-\frac12\mathfrak{k}+\beta_{X}\mathfrak{k}+\frac12}\lambda^{\mathfrak{k}}(\mathfrak{k}!)^{-1}\left(\wt{\mathfrak{c}}_{-7\mathfrak{l}}^{\mathfrak{m}}\left(N^{-\beta_{X}}\grad_{-7\mathfrak{l}}^{\mathfrak{m}}\mathbf{h}^{N}\right)^{\mathfrak{k}}\right). \label{eq:Duhamel3}
\end{align}\normalsize\normalsize
The equation \eqref{eq:Duhamel3} may be interpreted as the Taylor series for $\mathbf{Z}^{N}$ and recalling $\mathbf{Z}^{N}$ is the exponential of $\mathbf{h}^{N}$. We now observe that the $\mathbf{h}^{N}$-gradient is a linear polynomial in $\eta$-spins which are contained a neighborhood with length at most $\kappa_{\mathfrak{m}}N^{\beta_{X}}$ strictly to the right of $x-7\mathfrak{l}\mathfrak{m}$ and thus disjoint from the support of $\wt{\mathfrak{c}}^{\mathfrak{m}}_{-7\mathfrak{l}}$. We thereby additionally observe that per sum-index $\mathfrak{k}\in\Z_{\geq1}$ the corresponding summand on the far RHS is a functional with pseudo-gradient factor $\wt{\mathfrak{c}}^{\mathfrak{m}}_{-7\mathfrak{l}}$. Indeed observe the remaining functional-factor $N^{-\beta_{X}}\grad_{-7\mathfrak{l}}^{\mathfrak{m}}\mathbf{h}^{N}$ is uniformly bounded if $|\mathfrak{l}|\lesssim N^{\beta_{X}}$ since the $\eta$-spins are uniformly bounded. Additionally, the prefactor in the $\mathfrak{k}$-summand is at most $N^{\beta_{X}}N^{-(\mathfrak{k}-1)/2+(\mathfrak{k}-1)\beta_{X}}$ times uniformly bounded factors. Thus the infinite series within the far RHS of \eqref{eq:Duhamel3} is summable as $\beta_{X}<\frac12$ so it is a functional with pseudo-gradient factor $\wt{\mathfrak{c}}_{-7\mathfrak{l}}^{\mathfrak{m}}$ which is then scaled by $N^{\beta_{X}}$. After we average \eqref{eq:Duhamel3} over $\mathfrak{l}\in\llbracket1,N^{\beta_{X}}\rrbracket$ we may match the second term from \eqref{eq:Duhamel2} to the second term in $\Phi^{N,2}$. Alternatively instead of taking the entire infinite series into the functional-with-pseudo-gradient-factor/second term in $\Phi^{N,2}$ we may cut this series off at $\mathfrak{k}=4$ with an error that is uniformly vanishing in the large-$N$ limit.
\end{itemize}
We now return to the last difference between Proposition 2.3 from \cite{DT} and Proposition \ref{prop:Duhamel} concerning quadratic polynomials in $\eta$-spins that we remarked on at the beginning of this proof. Within Proposition 2.3 from \cite{DT} this quadratic polynomial was absorbed as a weakly vanishing term. However because we take the weaker Assumption \ref{ass:ass} we must treat it differently here. First we introduce some notation to define this quadratic polynomial of interest.
\begin{itemize}[leftmargin=*]
\item Define the length-$k$ neighborhood $\mathfrak{I}_{x,k}\overset{\bullet}=\{z_{1},z_{2}\in\Z: z_{1}\leq x<z_{2}, \ z_{2}-z_{1}=k\}$ and $\alpha_{k}\wt{\gamma}_{k}\overset{\bullet}=\alpha_{k}\gamma_{k}-\alpha_{k}\bar{\gamma_{k}} = \mathscr{O}(N^{-1/2})$.
\end{itemize}
This quadratic polynomial is the sum of terms defined below where we sum over $k\in\llbracket1,\mathfrak{m}\rrbracket$ and the neighborhood $\mathfrak{I}_{x,k}$:
\small\begin{align}
N\alpha_{k} \wt{\gamma}_{k} \eta_{T,z_{1}} \eta_{T,z_{2}} \ &= \ N\alpha_{k} \wt{\gamma}_{k}(\eta_{T,z_{1}}\eta_{T,z_{2}}- \eta_{T,x} \eta_{T,x+1}) + N\alpha_{k}\wt{\gamma}_{k}\eta_{T,x} \eta_{T,x+1} \ \overset{\bullet}= \ \Phi_{1;k,z_{1},z_{2}}+\Phi_{2;k}. \label{eq:Duhamel4}
\end{align}\normalsize\normalsize
Observe the $\Phi_{1;}$-term is a pseudo-gradient because $\eta$-spins are exchangeable with respect to any canonical ensemble. This was the justification for the cubic nonlinearity in Proposition 2.3 in \cite{DT} being a pseudo-gradient as well. In particular after we sum over the two finite sets $k\in\llbracket1,\mathfrak{m}\rrbracket$ and $\mathfrak{I}_{x,k}$ we get another pseudo-gradient since pseudo-gradients are closed under uniformly bounded linear combinations. Thus we may employ the same decomposition and expansions for the sum of $\Phi_{1;}$-terms as those we used to address the cubic nonlinearity in Proposition 2.3 in \cite{DT}. We emphasize Assumption \ref{ass:ass} implies $\Phi_{1;}$ is order $N^{1/2}$ like the aforementioned cubic nonlinearity. Meanwhile the second $\Phi_{2;k}$-terms are all independent of the $\mathfrak{I}_{x,k}$-variables. Therefore after we sum over $\mathfrak{I}_{x,k}$, we get from $\Phi_{2;k}$ the factor $k\alpha_{k}\wt{\gamma}_{k}$ times a $k$-independent factor. Summing over $k\in\llbracket1,\mathfrak{m}\rrbracket$ and applying Lemma \ref{lemma:Duhamel1} shows that $\Phi_{2;k}$-terms within the far RHS of \eqref{eq:Duhamel4} ultimately contribute zero to the $\mathbf{Z}^{N}$-stochastic equation.
\end{proof}
\begin{proof}[Proof of \emph{Lemma \ref{lemma:Duhamel1}}]
By definition of $\lambda \in \R$, it suffices to prove ${\sum}_{k = 1}^{\mathfrak{m}} k \alpha_{k} \bar{\gamma}_{k} = \lambda {\sum}_{k = 1}^{\mathfrak{m}} k^{2} \alpha_{k}$. By definition of $\bar{\gamma}_{k}$, observe
\small\begin{align}
{\sum}_{k = 1}^{\mathfrak{m}} k \alpha_{k} \bar{\gamma}_{k} \ = \ 2 \lambda {\sum}_{k = 1}^{\mathfrak{m}} {\sum}_{\ell = k}^{\mathfrak{m}} \ (\ell - k) \alpha_{\ell} \ + \ \lambda {\sum}_{k = 1}^{\mathfrak{m}} k \alpha_{k} \ &= \ 2 \lambda {\sum}_{k = 1}^{\mathfrak{m}} {\sum}_{\ell = k}^{\mathfrak{m}} \ell \alpha_{\ell} \ - \ 2 \lambda {\sum}_{k = 1}^{\mathfrak{m}} k {\sum}_{\ell = k}^{\mathfrak{m}} \alpha_{\ell} \ + \ \lambda {\sum}_{k = 1}^{\mathfrak{m}} k \alpha_{k}. \nonumber
\end{align}\normalsize\normalsize
We rewrite both double summations on the far RHS by accumulating the resulting coefficients for all $\alpha_{k}$ with $k \in \Z_{>0}$. In the first double summation, we obtain $k \alpha_{k}$ a total of $k$-times provided any $k \in \Z_{>0}$. Inside the second double summation, we grab one copy of $j\alpha_{k}$ for each $j \in \llbracket 1, k \rrbracket$.  Combining these two observations with elementary calculations gives
\small\begin{align}
2 \lambda {\sum}_{k = 1}^{\mathfrak{m}} {\sum}_{\ell = k}^{\mathfrak{m}} \ell \alpha_{\ell} \ - \ 2 \lambda {\sum}_{k = 1}^{\mathfrak{m}} k {\sum}_{\ell = k}^{\mathfrak{m}} \alpha_{\ell} \ = \ 2 \lambda {\sum}_{k = 1}^{\mathfrak{m}} \left( k^{2} - {\sum}_{j = 1}^{k} j \right) \alpha_{k} \ &= \ \lambda {\sum}_{k = 1}^{\mathfrak{m}} (k^{2} - k) \alpha_{k}.
\end{align}\normalsize\normalsize
Combining the previous two displays completes the proof.
\end{proof}
\subsection{Strategy}
We now discuss ``proof" of Theorem \ref{theorem:KPZ}. We clarify that we will not present results here that we actually get in this paper. The reasons for this are technical and instead we will get slightly adapted quantitative versions of the following; we discuss these reasons at the end of this strategy discussion. For this reason, all results discussed \emph{here} will be ``pseudo-results".

The proof of Theorem \ref{theorem:KPZ} is built on the following key result that effectively says we can forget about the $\Phi^{N,2}$-contribution in the stochastic equation for $\mathbf{Z}^{N}$ from Proposition \ref{prop:Duhamel}. We discuss its importance and implications afterwards.
\begin{pprop}\label{pprop:S1}
\fsp Define a norm $\|\varphi\|_{\mathfrak{t};\mathbb{X}}\overset{\bullet}=\sup_{0\leq\mathfrak{s}\leq\mathfrak{t}}\sup_{x\in\mathbb{X}}|\varphi_{\mathfrak{s},x}|$ for any $\mathfrak{t}\in\R_{\geq0}$ and $\mathbb{X}\subseteq\R$. We have convergence in probability $\|\mathbf{H}^{N}(\Phi^{N,2})\|_{1;\Z} \to_{N\to\infty} 0$ where $\Phi^{N,2}$ is defined in the statement of \emph{Proposition \ref{prop:Duhamel}}.
\end{pprop}
Due to standard linear procedure, it is enough to analyze the process $\mathbf{Y}^{N}$ that solves the same stochastic equation as $\mathbf{Z}^{N}$ from Proposition \ref{prop:Duhamel} but without the $\Phi^{N,2}$-term and replacing all factors of $\mathbf{Z}^{N}$ with $\mathbf{Y}^{N}$ and while keeping the same initial data. More precisely it is left to show $\mathbf{Y}^{N}$ converges to the solution of SHE. At this point, outside of the \emph{exact} microscopic version of the SHE corresponding to the first two terms in the stochastic equation of $\mathbf{Y}^{N}$ from Proposition \ref{prop:Duhamel}, there is only weakly-vanishing data. Thus the SHE scaling limit for $\mathbf{Y}^{N}$ can be computed using the same ideas from hydrodynamic limits that were used in the proof of Theorem 1.1 in \cite{DT}. There is a caveat to this that is ultimately negligible; see the end of the proof of Proposition \ref{prop:KPZ2}.

We discuss elements of the would-be-proof of Pseudo-Proposition \ref{pprop:S1} below. For simplicity, let us assume that the gradient term in $\Phi^{N,2}$ is equal to 0. To control such a gradient term, we use a simple summation-by-parts argument and apply regularity estimates for the heat kernel in the heat operator $\mathbf{H}^{N}$ like in \cite{DT}. Let us also assume the order $N^{\beta_{X}}$-term in $\Phi^{N,2}$ in Proposition \ref{prop:Duhamel} is equal to 0. It turns out that all steps we require to control this $N^{\beta_{X}}$-term will be steps we need to control the leading-order $N^{1/2}$-term in $\Phi^{N,2}$ in Proposition \ref{prop:Duhamel} anyway. The first step to control the order $N^{1/2}$-term in $\Phi^{N,2}$ is an introduction-of-cutoff.
\begin{pprop}\label{pprop:S2}
\fsp Define the cutoff-spatial average $\mathscr{C}^{\mathbf{X},-}$ as $ \mathscr{A}^{\mathbf{X},-}(\mathfrak{g})$ in $\Phi^{N,2}$ in \emph{Proposition \ref{prop:Duhamel}} but only when this spatial average is at most $N^{\e-\beta_{X}/2}$ in absolute value. On the complement event, define $\mathscr{C}^{\mathbf{X},-}$ to be 0. We have, in probability,
\small\begin{align}
{\lim}_{N\to\infty} \|\mathbf{H}^{N}(N^{\frac12}|\mathscr{A}^{\mathbf{X},-}_{N^{\beta_{X}}}(\mathfrak{g}_{S,y})-\mathscr{C}_{S,y}^{\mathbf{X},-}|\mathbf{Z}^{N})\|_{1;\Z} \ = \ 0.
\end{align}\normalsize\normalsize
\end{pprop}
It turns out that such a step will be unnecessary for the order $N^{\beta_{X}}$-term in $\Phi^{N,2}$ from Proposition \ref{prop:Duhamel} just because it is lower-order. We will explain both the would-be-proof and utility of the previous Pseudo-Proposition \ref{pprop:S2} in future subsections. The key second step towards the would-be-proof of Pseudo-Proposition \ref{pprop:S1} is the following replacement-by-time-average.
\begin{pprop}\label{pprop:S3}
\fsp Given $\mathfrak{t}\in\R_{>0}$, define $\mathscr{A}^{\mathfrak{t};\mathbf{T},+}_{S,y} \overset{\bullet}= \mathfrak{t}^{-1}\int_{0}^{\mathfrak{t}} \mathscr{C}^{\mathbf{X},-}_{S+\mathfrak{r},y} \d\mathfrak{r}$ and for $\mathfrak{t}=0$ we instead define $\mathscr{A}^{\mathfrak{t};\mathbf{T},+} = \mathscr{C}^{\mathbf{X},-}$. In words $\mathscr{A}^{\mathfrak{t};\mathbf{T},+}$ is the time-average of $\mathscr{C}^{\mathbf{X},-}$ with respect to time-scale $\mathfrak{t}\in\R_{\geq0}$ in the ``positive time-direction". 

Defining $\mathfrak{t}_{\max}\overset{\bullet}=N^{-1}$, we have the convergence in probability
\small\begin{align}
{\lim}_{N\to\infty} \|\mathbf{H}^{N}(N^{\frac12}(\mathscr{A}^{\mathfrak{t}_{\max};\mathbf{T},+}_{S,y}-\mathscr{C}^{\mathbf{X},-}_{S,y})\mathbf{Z}^{N})\|_{1;\Z} \ = \ 0.
\end{align}\normalsize\normalsize
\end{pprop}
The conclusion of Pseudo-Proposition \ref{pprop:S3} is replacement of the spatial-average-with-cutoff $\mathscr{C}^{\mathbf{X},-}$ defined in the statement of Pseudo-Proposition \ref{pprop:S2} with its time-average on time-scale $\mathfrak{t}_{\max}=N^{-1}$. To take advantage of such replacement we would need to estimate the time-average $\mathscr{A}^{\mathfrak{t}_{\max};\mathbf{T},+}$. Ultimately, we would get the following.
\begin{pprop}\label{pprop:S4}
 Admit the setting of \emph{Pseudo-Proposition \ref{pprop:S3}}. We have the convergence in the probability
\small\begin{align}
{\lim}_{N\to\infty} \|\mathbf{H}^{N}(N^{\frac12}\mathscr{A}_{S,y}^{\mathfrak{t}_{\max};\mathbf{T},+}\mathbf{Z}^{N})\|_{1;\Z} \ = \ 0.
\end{align}\normalsize\normalsize
\end{pprop}
Pseudo-Proposition \ref{pprop:S1} would now follow from the triangle inequality for the norm $\|\|_{1;\Z}$. We now explain the would-be-proofs for each of Pseudo-Proposition \ref{pprop:S2}, Pseudo-Proposition \ref{pprop:S3}, and Pseudo-Proposition \ref{pprop:S4}. 
\subsection{Strategy -- Local Equilibrium}
Before we start the would-be-proofs for Pseudo-Proposition \ref{pprop:S2}, Pseudo-Proposition \ref{pprop:S3}, and Pseudo-Proposition \ref{pprop:S4}, we first introduce key invariant measure calculations. The first is a large-deviations-type estimate which we establish in Lemma \ref{lemma:LDP} and Corollary \ref{corollary:LDP}, and the second is the Kipnis-Varadhan ``Brownian" inequality which we establish in Lemma \ref{lemma:KV}, Lemma \ref{lemma:SpectralH-1}, and Lemma \ref{lemma:H-1SpectralPGF}.
\begin{lemma}\label{lemma:LE1}
Suppose the particle system starts at any ``canonical ensemble" invariant measure so that the measure on particle configurations on any finite subset of $\Z$ is a convex combination of canonical ensembles in \emph{Definition \ref{definition:ensembles}} on that subset. We have
\small\begin{align}
{\lim}_{N\to\infty}N^{\frac12}\E|\mathscr{A}^{\mathbf{X},-}_{N^{\beta_{X}}}(\mathfrak{g}_{S,y})-\mathscr{C}^{\mathbf{X},-}_{S,y}| \ = \ 0.
\end{align}\normalsize\normalsize
\end{lemma}
\begin{lemma}\label{lemma:LE2}
 Suppose the particle system starts at a canonical ensemble invariant measure as in \emph{Lemma \ref{lemma:LE1}}. For any $\mathfrak{t}\in\R_{\geq0}$,
\small\begin{align}
N^{1/2}\E|\mathscr{A}_{S,y}^{\mathfrak{t};\mathbf{T},+}| \ \lesssim \ N^{-1/2}\mathfrak{t}^{-1/2}N^{-\beta_{X}/2+\e}.
\end{align}\normalsize\normalsize
\end{lemma}
The proofs of Lemmas \ref{lemma:LE1} and \ref{lemma:LE2} depend heavily on invariant measures. However, we do not need these bounds \emph{pointwise} in space-time but only after the law of the particle system is \emph{averaged} against the heat kernel in space-time because of the ``heat-operator-integrated" structure of the proposed estimates within Pseudo-Proposition \ref{pprop:S2}, Pseudo-Proposition \ref{pprop:S3}, and Pseudo-Proposition \ref{pprop:S4}. From \cite{GPV}, in this \emph{averaged} sense the law of the particle system is close to a convex combination of canonical ensemble invariant measures at least on mesoscopic space-time scales. We now note that the statistics we address in Lemma \ref{lemma:LE1} and Lemma \ref{lemma:LE2} are mesoscopic statistics. Thus, we expect that Lemma \ref{lemma:LE1} and Lemma \ref{lemma:LE2} still hold beyond the canonical ensemble initial measures when the absolute values therein are integrated against the heat kernel.

However, there are obstructions in actually implementing such a local equilibrium idea that concern how local equilibrium works at a more technical but still important level. We will refer to them in the would-be-proof of Pseudo-Proposition \ref{pprop:S4}.
\begin{itemize}[leftmargin=*]
\item Following \cite{GPV}, we reduce to local equilibrium using entropy production estimates. Because of the asymmetry in the model not being sufficiently weak, and because of the infinite-volume features of the lattice $\Z$ where the exclusion process evolves, establishing entropy production estimates in our setting is noticeably more difficult; Section \ref{section:HydroStuff} addresses this.
\item We reduce to a local equilibrium by comparing the averaged law of the particle system to an invariant measure at the level of relative entropy using the entropy inequality in Appendix 1.8 in \cite{KL} which connects relative entropy to large deviations, for example. We also require a log-Sobolev inequality for the symmetric simple exclusion process established in \cite{Yau} to control relative entropy by Dirichlet form in order to use entropy production from the first bullet point. The LSI in \cite{Yau} is quadratic in the length-scale we want to reduce to local equilibrium on, so to reduce to local equilibrium on bigger length-scales, we need better entropy production or control, at a level of large deviations, on the functional that we are estimating in expectation.
\end{itemize}
\subsection{Pseudo-Proposition \ref{pprop:S2}}
Assume Lemma \ref{lemma:LE1} holds without expectation and the $\mathbf{Z}^{N}$-process is uniformly bounded as it is supposed to look like the SHE solution. By Lemma \ref{lemma:LE1}, there is nothing to do for invariant measure initial data. However, because the statistic we estimate in Lemma \ref{lemma:LE1} is mesoscopic, in general we can finish by reducing to local equilibrium.
\subsection{Pseudo-Proposition \ref{pprop:S3}}
We use the fast-variable/slow-variable idea from stochastic homogenization. This proposes that on the time-scale on which we want to replace $\mathscr{C}^{\mathbf{X},-}$ by time-average, the heat kernel and the $\mathbf{Z}^{N}$-process both exhibit little variation. The error for such a replacement-by-time-average is thus controlled by time-regularity of the heat kernel times $\mathscr{C}^{\mathbf{X},-}$ and time-regularity of the $\mathbf{Z}^{N}$-process times the size of $\mathscr{C}^{\mathbf{X},-}$. The heat kernel is smooth with an integrable ``enough" short-time singularity. On the other hand, we also expect that $\mathbf{Z}^{N}$ has time-regularity matching that of the SHE solution, and this is Holder regularity of exponent $(\frac14)^{-}$. Thus, the time-regularity for $\mathbf{Z}^{N}$ is worse than that for the heat kernel, so we control the error given by time-regularity of $\mathbf{Z}^{N}$ times $\mathscr{C}^{\mathbf{X},-}$. For any time-scale $\mathfrak{t}\in\R_{\geq0}$, we thus expect the resulting error to be at most, recalling $\mathscr{C}^{\mathbf{X},-}$ has an a priori bound by definition in Pseudo-Proposition \ref{pprop:S2}, of order roughly equal to:
\small\begin{align}
N^{1/2}\|\mathscr{C}^{\mathbf{X},-}\|_{\infty}\mathfrak{t}^{1/4} \ \lesssim \ N^{1/2-\beta_{X}/2+\e}\mathfrak{t}^{1/4} \ \lesssim \ N^{1/3}\mathfrak{t}^{1/4}. \label{eq:S31}
\end{align}\normalsize\normalsize
Recall $\beta_{X}=1/3+\e_{X,1}$ in the statement of Proposition \ref{prop:Duhamel}. Pseudo-Proposition \ref{pprop:S3} proposes that we take $\mathfrak{t}=\mathfrak{t}_{\max}=N^{-1}$ for which the RHS of the estimate \eqref{eq:S31} blows up. To remedy this we will appeal to the following multiscale procedure. Roughly it proposes an initial replacement of $\mathscr{C}^{\mathbf{X},-}$ by its time-average on \emph{some} time-scale and applying Lemma \ref{lemma:LE2} and local equilibrium to boost our a priori bound for $\mathscr{C}^{\mathbf{X},-}$ after time-average. First assume for simplicity that Lemma \ref{lemma:LE2} holds without expectation.
\begin{itemize}[leftmargin=*]
\item The bound \eqref{eq:S31} tells us we can pick the time-scale $\mathfrak{t}=\mathfrak{t}_{1,1} = N^{-2+\e_{1}}$ for $\e_{1}\in\R_{>0}$ arbitrarily small but still universal, so we first replace $\mathscr{C}^{\mathbf{X},-}$ with its time-average $\mathscr{A}^{\mathfrak{t}_{1,1};\mathbf{T},+}$ with respect to this preliminary time-scale $\mathfrak{t}_{1,1}=N^{-2+\e_{1}}$.
\item Let us now replace the time-average $\mathscr{A}^{\mathfrak{t}_{1,1};\mathbf{T},+}$ with respect to the time-scale $\mathfrak{t}_{1,1}=N^{-2+\e_{1}}$ with the time-average $\mathscr{A}^{\mathfrak{t}_{1,2};\mathbf{T},+}$ with respect to the time-scale $\mathfrak{t}_{1,2}=N^{\e_{1}}\mathfrak{t}_{1,1}$ inside of the heat operator integrated against both the heat kernel and the $\mathbf{Z}^{N}$-process. To this end, we first observe that $\mathscr{A}^{\mathfrak{t}_{1,2};\mathbf{T},+}$ is basically a time-average of $\mathscr{A}^{\mathfrak{t}_{1,1};\mathbf{T},+}$ on the latter/larger time-scale $\mathfrak{t}_{1,2}$. Therefore, the error we get is \eqref{eq:S31} as before but now with $\mathscr{A}^{\mathfrak{t}_{1,1};\mathbf{T},+}$ in place of $\mathscr{C}^{\mathbf{X},-}$ and with $\mathfrak{t}=\mathfrak{t}_{1,2}=N^{\e_{1}}\mathfrak{t}_{1,1}$. Lemma \ref{lemma:LE2} gives the following ultimately vanishing bound for the error in this second ``replacement-by-time-average" step:
\small\begin{align}
N^{1/2}\|\mathscr{A}^{\mathfrak{t}_{1,1};\mathbf{T},+}\|_{\infty}\mathfrak{t}_{1,2}^{1/4} \ \lesssim \ N^{-1/2-\beta_{X}/2+\e}\mathfrak{t}_{1,1}^{-1/2}\mathfrak{t}_{1,2}^{1/4} \ \lesssim \ N^{-2/3+\e}\mathfrak{t}_{1,1}^{-1/4}\mathfrak{t}_{1,2}^{1/4} \cdot \mathfrak{t}_{1,1}^{-1/4} \ = \ N^{-2/3+\e+\e_{1}}\mathfrak{t}_{1,1}^{-1/4} \ \lesssim \ N^{-1/6+\e+\e_{1}}. \label{eq:S32}
\end{align}\normalsize\normalsize
\item The rest of the multiscale procedure is then replacing $\mathscr{A}^{\mathfrak{t}_{1,n};\mathbf{T},+}$ with $\mathscr{A}^{\mathfrak{t}_{1,n+1};\mathbf{T},+}$, in which $\mathfrak{t}_{1,n+1}=N^{\e_{1}}\mathfrak{t}_{1,n}$, until we arrive at the final time-scale $\mathfrak{t}_{\max}=N^{-1}$. As $\mathfrak{t}_{1,1}=N^{-2+\e_{1}}$, we only require a $\e_{1}$-dependent number of steps to get to $\mathfrak{t}_{\max}=N^{-1}$.
\end{itemize}
\subsection{Pseudo-Proposition \ref{pprop:S4}}
Lemma \ref{lemma:LE2} with $\mathfrak{t}=\mathfrak{t}_{\max}=N^{-1}$ yields $N^{1/2}|\mathscr{A}^{\mathfrak{t}_{\max};\mathbf{T},+}| \lesssim N^{-\beta_{X}/2+\e}$, so the proposed bound in Pseudo-Proposition \ref{pprop:S4} holds if the exclusion process starts at a canonical ensemble invariant measure. We would then like to perform a reduction to local equilibrium since $\mathscr{A}^{\mathfrak{t}_{\max};\mathbf{T},+}$ is a mesoscopic statistic. However, there is a serious obstruction that cannot be easily circumvented. To explain this, let us observe that for longer time-scales $\mathfrak{t}\in\R_{\geq0}$, the statistic $\mathscr{A}^{\mathfrak{t};\mathbf{T},+}$ depends on spins on a larger $\mathfrak{t}$-dependent block. For example, the LSI for exclusion processes from \cite{Yau} grows in the length-scale of the set on which the exclusion process lives. For the time-scale $\mathfrak{t}_{\max}=N^{-1}$, the associated length-scale is too large for us to apply local equilibrium. We resolve this via the following multiscale method that makes local equilibrium accessible by slowly boosting a priori large-deviations bounds for time-averages. Recall better large-deviations estimates help with local equilibrium.
\begin{itemize}[leftmargin=*]
\item The issue with local equilibrium was that the time-scale $\mathfrak{t}_{\max}=N^{-1}$ was too long. With this in mind, we first consider shorter time-scales. We write $\mathscr{A}^{\mathfrak{t}_{\max};\mathbf{T},+}$ as an average of time-shifted time-averages on a time-scale $0<\mathfrak{t}_{2,1}\leq\mathfrak{t}_{\max}$:
\small\begin{align}
\mathscr{A}^{\mathfrak{t}_{\max};\mathbf{T},+}_{S,y} \ &= \ \wt{\sum}_{\mathfrak{l}=0}^{\mathfrak{t}_{\max}\mathfrak{t}_{2,1}^{-1}-1}\mathscr{A}^{\mathfrak{t}_{2,1};\mathbf{T},+}_{S+\mathfrak{l}\mathfrak{t}_{2,1},y}. \label{eq:S41}
\end{align}\normalsize\normalsize
The identity \eqref{eq:S41} follows by noting that averaging on a large time-scale is the same as averages on a smaller time-scale and averaging all the small-scale averages together. If $\mathfrak{t}_{2,1}\in\R_{>0}$ is a ``short enough" time-scale, we may analyze every summand within the RHS of \eqref{eq:S41} using Lemma \ref{lemma:LE2} and a reduction to local equilibrium. However, if we choose $\mathfrak{t}_{2,1}\in\R_{>0}$ too small our estimates for the summands on the RHS of \eqref{eq:S41} will not vanish in the large-$N$ limit when equipped with the additional factor of $N^{1/2}$ in the proposed estimate in the statement of Pseudo-Proposition \ref{pprop:S4}. Because we cannot pick $\mathfrak{t}_{2,1}\in\R_{>0}$ large enough to get a vanishing estimate for the summands on the RHS of \eqref{eq:S41} while still being able to employ local equilibrium, we instead do the following with motivation given shortly. Take $\beta_{2,1} = \frac12\beta_{X}-\e+\e_{2}$; building on \eqref{eq:S41}, we introduce cutoff:
\small\begin{align}
\mathscr{A}^{\mathfrak{t}_{\max};\mathbf{T},+}_{S,y} \ = \ \wt{\sum}_{\mathfrak{l}=0}^{\mathfrak{t}_{\max}\mathfrak{t}_{2,1}^{-1}-1}\mathscr{A}^{\mathfrak{t}_{2,1};\mathbf{T},+}_{S+\mathfrak{l}\mathfrak{t}_{2,1},y} \ &= \ \wt{\sum}_{\mathfrak{l}=0}^{\mathfrak{t}_{\max}\mathfrak{t}_{2,1}^{-1}-1}\mathscr{A}^{\mathfrak{t}_{2,1};\mathbf{T},+}_{S+\mathfrak{l}\mathfrak{t}_{2,1},y}\mathbf{1}_{|\mathscr{A}^{\mathfrak{t}_{2,1};\mathbf{T},+}_{S+\mathfrak{l}\mathfrak{t}_{2,1},y}|\lesssim N^{-\beta_{2,1}}} \label{eq:S42} \\
&+ \ \wt{\sum}_{\mathfrak{l}=0}^{\mathfrak{t}_{\max}\mathfrak{t}_{2,1}^{-1}-1}\mathscr{A}^{\mathfrak{t}_{2,1};\mathbf{T},+}_{S+\mathfrak{l}\mathfrak{t}_{2,1},y}\mathbf{1}_{|\mathscr{A}^{\mathfrak{t}_{2,1};\mathbf{T},+}_{S+\mathfrak{l}\mathfrak{t}_{2,1},y}|\gtrsim N^{-\beta_{2,1}}}. \label{eq:S42}
\end{align}\normalsize\normalsize
The first sum on the far RHS of the identity \eqref{eq:S42} gives each $\mathscr{A}^{\mathfrak{t}_{2,1};\mathbf{T},+}$-summand therein an \emph{improved} a priori upper bound of $N^{-\beta_{2,1}}$ which is $N^{-\e_{2}}$ better than the a priori bound of $N^{-\beta_{X}/2+\e}$ for $\mathscr{A}^{\mathfrak{t}_{2,1};\mathbf{T},+}$ that is inherited from the $\mathscr{C}^{\mathbf{X},-}$ it averages. For the second sum on the far RHS of \eqref{eq:S42}, we pick $\mathfrak{t}_{2,1}\in\R_{>0}$ small enough to apply local equilibrium but large enough so that the summands are negligible with high-probability with respect to canonical ensembles. To see where this negligible feature of the summands in this second sum on the far RHS of \eqref{eq:S42} comes from, we observe that if $\mathfrak{t}_{2,1}\in\R_{>0}$ is sufficiently large then Lemma \ref{lemma:LE2} tells us the event $|\mathscr{A}^{\mathfrak{t}_{2,1};\mathbf{T},+}|\gtrsim N^{-\beta_{2,1}}$ happens with low-probability. We emphasize that the point for such step is to accomplish a lesser goal. We will not estimate $\mathscr{A}^{\mathfrak{t}_{2,1};\mathbf{T},+}$ by $N^{-1/2}$ but rather by $N^{-\beta_{2,1}}$. We want a less sharp estimate, so we may pick a shorter time-scale $0<\mathfrak{t}_{2,1}\leq\mathfrak{t}_{\max}$ that does the job and for which we can apply local equilibrium.
\item At the end of the previous bullet point, we are now left with the first sum on the far RHS of \eqref{eq:S42}. This is an average of time-averages $\mathscr{A}^{\mathfrak{t}_{2,1};\mathbf{T},+}$ with a priori upper bound cutoff of $N^{-\beta_{2,1}}$ that improves on the $N^{-\beta_{X}/2+\e}$-cutoff on $\mathscr{C}^{\mathbf{X},-}$ by a factor of $N^{-\e_{2}}$. As each of these time-averages comes with the same upper bound cutoff of $N^{-\beta_{2,1}}$ we can group these time-averages on time-scale $\mathfrak{t}_{2,1}\in\R_{>0}$ into $\mathfrak{t}_{\max}\mathfrak{t}_{2,2}^{-1}$ many time-averages on a to-be-determined time-scale $\mathfrak{t}_{2,2}$ that is larger than $\mathfrak{t}_{2,1}$, and each scale $\mathfrak{t}_{2,2}$-averages also inherits the a priori upper bound cutoff of $N^{-\beta_{2,1}}$ as it averages terms with such cutoff. So
\small\begin{align}
\wt{\sum}_{\mathfrak{l}=0}^{\mathfrak{t}_{\max}\mathfrak{t}_{2,1}^{-1}-1}\mathscr{A}^{\mathfrak{t}_{2,1};\mathbf{T},+}_{S+\mathfrak{l}\mathfrak{t}_{2,1},y}\mathbf{1}_{|\mathscr{A}^{\mathfrak{t}_{2,1};\mathbf{T},+}_{S+\mathfrak{l}\mathfrak{t}_{2,1},y}|\lesssim N^{-\beta_{2,1}}} \ \approx \ \wt{\sum}_{\mathfrak{l}=0}^{\mathfrak{t}_{\max}\mathfrak{t}_{2,2}^{-1}-1}\mathscr{A}^{\mathfrak{t}_{2,2};\mathbf{T},+}_{S+\mathfrak{l}\mathfrak{t}_{2,2},y}\mathbf{1}_{|\mathscr{A}^{\mathfrak{t}_{2,2};\mathbf{T},+}_{S+\mathfrak{l}\mathfrak{t}_{2,2},y}|\lesssim N^{-\beta_{2,1}}}. \label{eq:S43}
\end{align}\normalsize\normalsize
The error terms corresponding to the approximation \eqref{eq:S43} look sufficiently like the second sum on the far RHS of \eqref{eq:S42}, it turns out after unfolding definitions. To be just a little more precise, the error terms are given by scale-$\mathfrak{t}_{2,1}$ time-averages with both upper and lower bound cutoffs. We then pick $\beta_{2,2} = \beta_{2,1} + \e_{2}$ and improve our a priori $N^{-\beta_{2,1}}$-bound similar to \eqref{eq:S42}:
\small\begin{align}
\wt{\sum}_{\mathfrak{l}=0}^{\mathfrak{t}_{\max}\mathfrak{t}_{2,2}^{-1}-1}\mathscr{A}^{\mathfrak{t}_{2,2};\mathbf{T},+}_{S+\mathfrak{l}\mathfrak{t}_{2,2},y}\mathbf{1}_{|\mathscr{A}^{\mathfrak{t}_{2,2};\mathbf{T},+}_{S+\mathfrak{l}\mathfrak{t}_{2,2},y}|\lesssim N^{-\beta_{2,1}}} \ &= \ \wt{\sum}_{\mathfrak{l}=0}^{\mathfrak{t}_{\max}\mathfrak{t}_{2,2}^{-1}-1}\mathscr{A}^{\mathfrak{t}_{2,2};\mathbf{T},+}_{S+\mathfrak{l}\mathfrak{t}_{2,2},y}\mathbf{1}_{|\mathscr{A}^{\mathfrak{t}_{2,2};\mathbf{T},+}_{S+\mathfrak{l}\mathfrak{t}_{2,2},y}|\lesssim N^{-\beta_{2,2}}}\nonumber\\
&+\wt{\sum}_{\mathfrak{l}=0}^{\mathfrak{t}_{\max}\mathfrak{t}_{2,2}^{-1}-1}\mathscr{A}^{\mathfrak{t}_{2,2};\mathbf{T},+}_{S+\mathfrak{l}\mathfrak{t}_{2,2},y}\mathbf{1}_{N^{-\beta_{2,2}}\lesssim|\mathscr{A}^{\mathfrak{t}_{2,2};\mathbf{T},+}_{S+\mathfrak{l}\mathfrak{t}_{2,2},y}|\lesssim N^{-\beta_{2,1}}}. \nonumber
\end{align}\normalsize\normalsize
Again, if we pick $\mathfrak{t}_{2,2}\in\R_{>0}$ sufficiently large but not too much larger than $\mathfrak{t}_{2,1}\in\R_{>0}$, then we can use Lemma \ref{lemma:LE2} and local equilibrium to argue the second sum within the RHS of the previous display is negligible with high-probability similar to our reasoning to control the second sum within the far RHS of \eqref{eq:S42}. We now discuss why we can reduce to a local equilibrium on this larger time-scale $\mathfrak{t}_{2,2}\in\R_{>0}$. The indicator functions attached to summands in the second sum within the RHS of the previous display come equipped with a priori $N^{-\beta_{2,1}}$-upper bounds that are $N^{-\e_{2}}$-\emph{better} than our a priori upper bounds for the summands in the second sum on the far RHS of \eqref{eq:S42}. The a priori upper bounds are deterministic and at the level of large-deviations trivially. Because we have better a priori upper bounds of $N^{-\beta_{2,1}}$ for the $\mathscr{A}^{\mathfrak{t}_{2,2};\mathbf{T},+}$-functional we are estimating in expectation, we can reduce bounds of scale $\mathfrak{t}_{2,2}$-data to local equilibrium as $\mathfrak{t}_{2,2}\in\R_{>0}$ is only slightly larger than the previous time-scale $\mathfrak{t}_{2,1}\in\R_{>0}$. We emphasize the error terms in the second sum on the RHS of the last display are time-averages with a priori upper bound and lower bound cutoffs that are $N^{\e_{2}}$-off from each other.
\item We iterate these improvements on time-scale and a priori upper bounds until we arrive at the a priori upper bound $N^{-1/2-\e}$. Because we improve our a priori upper bounds by $N^{-\e_{2}}$ at every step, we require only a $\e_{2}$-dependent number of iterations. All error terms in such scheme are time-averages with a priori upper and lower bound cutoffs which differ by a factor of $N^{\e_{2}}$. Moreover, all time-scales on which we perform a time-average will be at most the $\mathfrak{t}_{\max}$-scale from Pseudo-Proposition \ref{pprop:S3}.
\end{itemize}
\subsection{Technical Comments}
We start with why/in what way Pseudo-Propositions \ref{pprop:S2}, \ref{pprop:S3}, and \ref{pprop:S4} are \emph{pseudo}-results.
\begin{itemize}[leftmargin=*]
\item For technical reasons it will be convenient to ``compactify" the microscopic Cole-Hopf transform $\mathbf{Z}^{N}$. The current microscopic Cole-Hopf transform solves a stochastic equation on the infinite set $\R_{\geq0}\times\Z$. We will replace it by the solution of the same stochastic equation but ``periodized-in-space" onto a set $\R_{\geq0}\times\mathbb{T}_{N}$ as the $\|\|_{1;\mathbb{T}_{N}}$-norm is easier to work with than $\|\|_{1;\Z}$. Here $\mathbb{T}_{N}\subseteq\Z$ is a torus that is much larger than the macroscopic length-scale $N$. Thus, it should not change any scaling limits. 
\item In view of the previous bullet point, we will only prove estimates for $\Phi^{N,2}$ with respect to $\|\|_{1;\mathbb{T}_{N}}$-norms and not $\|\|_{1;\Z}$-norms.
\item The details behind the version of Pseudo-Proposition \ref{pprop:S1} that we actually get are different at a technical level. In particular, we will actually take a different, but still morally similar, value for $\mathfrak{t}_{\max}$ from Pseudo-Proposition \ref{pprop:S3}.
\item The stochastic equation for $\mathbf{Z}^{N}$ in Proposition \ref{prop:Duhamel} is multiplicative in $\mathbf{Z}^{N}$, so it is not enough to study only local functionals . We must also control $\|\mathbf{Z}^{N}\|_{1;\mathbb{T}_{N}}$. This ends up being a technical point we treat with a continuity argument from PDE.
\end{itemize}
Let us now give an outline of the rest of this paper.
\begin{itemize}[leftmargin=*]
\item In Section \ref{section:HydroStuff} we construct local equilibrium and give invariant measure estimates. In Section \ref{section:Ctify} we ``compactify" $\Z\to\mathbb{T}_{N}$.
\item In Section \ref{section:D1B}, we get the spatial/time-average estimates alluded to in the would-be-proofs of Pseudo-Propositions \ref{pprop:S2}, \ref{pprop:S3}, and \ref{pprop:S4}. We do this via local equilibrium and a ``dynamical one-block step"/dynamic version of the one-block step in \cite{GPV}.
\item In Section \ref{section:KPZ1} we establish time-regularity estimates for $\mathbf{Z}^{N}$ using mostly standard ideas that are present in \cite{DT}, for example.
\item We establish a strong version of Pseudo-Proposition \ref{pprop:S1} in Section \ref{section:KPZ2} following basically the previously outlined strategy. In Section \ref{section:KPZ3} we apply this strong estimate to pretend $\Phi^{N,2}\approx0$ with high probability. We then follow the arguments in \cite{DT}.
\end{itemize}
\begin{itemize}[leftmargin=*]
\item Section \ref{section:HydroStuff}, Section \ref{section:D1B}, Section \ref{section:KPZ1}, and Section \ref{section:KPZ2} are all technical. The reader is invited to skim these sections and first take their results for granted and then read the paper in its written order to get a ``blackbox" proof of Theorem \ref{theorem:KPZ}. For technical aspects behind showing the important ingredients within Section \ref{section:HydroStuff}, Section \ref{section:D1B}, Section \ref{section:KPZ1}, and Section \ref{section:KPZ2}, the reader is invited to read Section \ref{section:KPZ2} followed by Section \ref{section:HydroStuff}, Section \ref{section:D1B}, then Section \ref{section:KPZ1}. We have organized Section \ref{section:HydroStuff}, Section \ref{section:D1B}, Section \ref{section:KPZ1}, and Section \ref{section:KPZ2} to provide proofs of main results first, deferring proofs of technical manipulations/ingredients until the end of the section.
\end{itemize}
%
%
%
\section{Local Equilibrium Estimates}\label{section:HydroStuff}
This section is a preliminary discussion consisting of tools from hydrodynamic limits. We organize it as follows.
\begin{itemize}[leftmargin=*]
\item We will first introduce the key invariant measures for relevant exclusion processes. These are the canonical ensembles and grand-canonical ensembles that we used to define the three classes of functionals at the beginning of Section \ref{section:Framework}.
\item We move to entropy production in Proposition \ref{prop:EProd}. We use it to perform ``reduction to local equilibrium" in Lemma \ref{lemma:LE}.
\item We then move to invariant measure calculations that will be important \emph{after} our reduction to local equilibrium. These tools are given in Lemma \ref{lemma:KV}, Lemma \ref{lemma:SpectralH-1}, Lemma \ref{lemma:H-1SpectralPGF}, Lemma \ref{lemma:LDP}, and Corollary \ref{corollary:LDP}.
\end{itemize}
\begin{definition}\label{definition:ensembles}
For $\mathfrak{I} \subseteq \Z$ and $\varrho \in \R$, let the canonical ensemble/measure $\mu_{\varrho,\mathfrak{I}}^{\mathrm{can}}$ be the uniform measure on the hyperplane
\small\begin{align}
\Omega_{\varrho,\mathfrak{I}} \ &\overset{\bullet}= \ \left\{\eta\in\Omega_{\mathfrak{I}}: \ \wt{\sum}_{x\in\mathfrak{I}}\eta_{x} = \varrho \right\} \ \subseteq \ \Omega_{\mathfrak{I}}.
\end{align}\normalsize\normalsize
We define the grand-canonical ensemble/measure $\mu_{\varrho,\mathfrak{I}}$ as a product measure on $\Omega_{\mathfrak{I}} = \{\pm1\}^{\mathfrak{I}}$ such that $\E\eta_{x} = \varrho$ for all $x\in\mathfrak{I}$. We note that $\varrho\in\R$ denotes \emph{spin density} and not \emph{particle density} in our conventions.
\end{definition}
\begin{definition}\label{definition:REDF}
Consider any subset $\mathfrak{I} \subseteq \Z$ and parameter $\varrho \in \R$, and suppose that $\mathfrak{f}:\Omega_{\mathfrak{I}}\to\R_{\geq0}$ is a probability density with respect to either the canonical ensemble $\mu_{\varrho,\mathfrak{I}}^{\mathrm{can}}$ or the grand-canonical ensemble $\mu_{\varrho,\mathfrak{I}}$ depending on the context below.
\begin{itemize}[leftmargin=*]
\item The \emph{relative entropy} with respect to $\mu_{\varrho,\mathfrak{I}}^{\mathrm{can}}$ and $\mu_{\varrho,\mathfrak{I}}$, respectively, are the functionals
\small\begin{align}
\mathfrak{H}_{\varrho,\mathfrak{I}}^{\mathrm{can}}(\mathfrak{f}) \ \overset{\bullet}= \ \E^{\mu_{\varrho,\mathfrak{I}}^{\mathrm{can}}}\mathfrak{f}\log\mathfrak{f} \quad \mathrm{and} \quad \mathfrak{H}_{\varrho,\mathfrak{I}}(\mathfrak{f}) \ &\overset{\bullet}= \ \E^{\mu_{\varrho,\mathfrak{I}}}\mathfrak{f}\log\mathfrak{f}.
\end{align}\normalsize\normalsize
\item The \emph{Dirichlet form} with respect to $\mu_{\varrho,\mathfrak{I}}^{\mathrm{can}}$ and $\mu_{\varrho,\mathfrak{I}}$, respectively, are defined as follows in which we recall that $\mathfrak{S}_{x,y}$ denotes the generator for the symmetric exclusion process with speed 1 on the bond $\{x,y\}\subseteq\Z$. Recall the set $\mathrm{A}$ of symmetric speeds $\alpha_{\bullet}$:
\begin{subequations}
\small\begin{align}
\mathfrak{D}_{\varrho,\mathfrak{I}}^{\mathrm{can}}(\mathfrak{f}) \ &\overset{\bullet}= \ - \E^{\mu_{\varrho,\mathfrak{I}}^{\mathrm{can}}}\mathfrak{f}^{\frac12} \cdot {\sum}_{x,y \in \mathfrak{I}} \alpha_{|x-y|} \mathfrak{S}_{x,y} \mathfrak{f}^{\frac12} \ = \ -{\sum}_{x,y\in\mathfrak{I}}\alpha_{|x-y|} \E^{\mu_{\varrho,\mathfrak{I}}^{\mathrm{can}}}\mathfrak{f}^{\frac12}\mathfrak{S}_{x,y}\mathfrak{f}^{\frac12} \\
\mathfrak{D}_{\varrho,\mathfrak{I}}(\mathfrak{f}) \ &\overset{\bullet}= \ - \E^{\mu_{\varrho,\mathfrak{I}}}\mathfrak{f}^{\frac12} \cdot {\sum}_{x,y \in \mathfrak{I}} \alpha_{|x-y|} \mathfrak{S}_{x,y} \mathfrak{f}^{\frac12} \ = \ -{\sum}_{x,y\in\mathfrak{I}}\alpha_{|x-y|} \E^{\mu_{\varrho,\mathfrak{I}}}\mathfrak{f}^{\frac12}\mathfrak{S}_{x,y}\mathfrak{f}^{\frac12}.
\end{align}\normalsize\normalsize
\end{subequations}
\end{itemize}
\end{definition}
\begin{remark}
The Dirichlet form is usually defined without a square root of $\mathfrak{f}$, but we adopt it as it appears frequently here.
\end{remark}
\subsection{Entropy Production}
We present here an entropy production bound that will be the key ingredient for local equilibrium that is not already established somewhere else. First, we introduce notation to be used both in this section and beyond.
\begin{definition}
Provided any $T \in \R_{\geq0}$, let us define $\mu_{T,N}$ as the law of the particle system on $\Omega_{\Z}$ after time-$T$ evolution with initial measure $\mu_{0,N}$. Define the Radon-Nikodym derivative $\mathfrak{f}_{T,N}\d\mu_{0,\Z} = \d\mu_{T,N}$. 
\end{definition}
\begin{definition}
For any random variable or probability measure $\phi$, define $\phi^{\mathfrak{I}} \overset{\bullet}= \E^{\mu_{0,\Z}}_{\mathfrak{I}}\phi$ as conditional expectation conditioning on $\eta_{x}$ for $x \in \mathfrak{I} \subseteq \Z$. If $\phi$ is a measure, this conditional expectation is its marginal on the joint spins $\eta_{x}$ for $x\in\mathfrak{I}\subseteq\Z$.
\end{definition}
\begin{prop}\label{prop:EProd}
 Take any $T \in \R_{\geq0}$ uniformly bounded, and consider any arbitrarily small though universal constant $\e \in \R_{>0}$. Uniformly over all initial probability measures $\mu_{0,N}$ on $\Omega_{\Z}$, we get the time-averaged grand-canonical Dirichlet form bound
\small\begin{align}
T^{-1}\int_{0}^{T} \mathfrak{D}_{0,\mathfrak{I}_{N}}(\mathfrak{f}_{S,N}^{\mathfrak{I}_{N}}) \ \d S \ &\lesssim_{\e} \ T^{-1}N^{-\frac34 + 5\e} + N^{-\frac34+5\e} \ \lesssim \ T^{-1}N^{-\frac34+5\e}.
\end{align}\normalsize\normalsize
We introduced $\mathfrak{I}_{N} \overset{\bullet}= \llbracket -N^{\frac54 + \e}, N^{\frac54 + \e} \rrbracket \subseteq \Z$. The second bound above follows as $T\in\R_{\geq0}$ is uniformly bounded.
\end{prop}
Proposition \ref{prop:EProd} is referred to as an ``entropy production bound" even though it estimates a time-averaged Dirichlet form. This is because the upper bound comes from estimating terms appearing in the time derivative, or growth, of relative entropy.

Without the $\e$-dependent terms in the exponents in the statement of Proposition \ref{prop:EProd}, the conclusion is entropy production estimates that would match that of a \emph{symmetric} simple exclusion process on a torus of length $N^{5/4}$ roughly speaking. Thus one interpretation of Proposition \ref{prop:EProd} is somehow controlling the contribution of the asymmetry on the length-scale $N^{5/4}$ by that of the symmetric part of the exclusion process. We note that doing so upon replacing the length-scale $N^{5/4}\to N^{3/2}$ is not difficult as we show in Lemma \ref{lemma:EProdSoP} for reasons which amount to $\frac32=2-\frac12$ and the relative strength of the asymmetry is $N^{-1/2}$.
\subsection{Local Equilibrium}
First, we introduce notation that will be used both in and beyond the current section. The following construction is of a space-time averaged probability density in spirit of the one-block scheme in \cite{GPV}.
\begin{definition}\label{definition:AvMeasures}
Recall the law of the particle system $\mu_{\bullet,N}$ with initial measure $\mu_{0,N}$. Recall $\mathfrak{f}_{\bullet,N}\d\mu_{0,\Z} \overset{\bullet}= \d\mu_{\bullet,N}$. Provided times $\mathfrak{t}_{0},T \in \R_{\geq 0}$ and a subset $\mathfrak{I} \subseteq \Z$, define the following space-time averaged probability density where $\tau_{x}\mathfrak{f}_{\mathfrak{t},N}(\eta) \overset{\bullet}= \mathfrak{f}_{\mathfrak{t},N}(\tau_{x}\eta)$:
\small\begin{align}
\bar{\mathfrak{f}}_{\mathfrak{t}_{0},T,\mathfrak{I}} \ &\overset{\bullet}= \ T^{-1} \int_{0}^{T} \wt{\sum}_{x \in \mathfrak{I}} \tau_{x}\mathfrak{f}_{\mathfrak{t}_{0}+S,N} \ \d S.
\end{align}\normalsize\normalsize
We also define $\d\bar{\mu}_{\mathfrak{t}_{0},T,\mathfrak{I}} \overset{\bullet}= \bar{\mathfrak{f}}_{\mathfrak{t}_{0},T,\mathfrak{I}}\d\mu_{0,\Z}$ and $\d\bar{\mu}_{T,\mathfrak{I}} \overset{\bullet}= \d\bar{\mu}_{0,T,\mathfrak{I}}$ and $\bar{\mathfrak{f}}_{T,\mathfrak{I}} \overset{\bullet}= \bar{\mathfrak{f}}_{0,T,\mathfrak{I}}$. Provided any subset $\mathfrak{I}_{2} \subseteq \Z$ and any $\varrho \in \R$, define ``spin-density" probabilities and the ``$\mathfrak{I}_{2}$-local" relative entropy and Dirichlet form that we use in proving local equilibrium:
\begin{subequations}
\small\begin{align}
\mathfrak{p}_{\mathfrak{t}_{0},T,\mathfrak{I}}^{\mathfrak{I}_{2},\varrho} \ &\overset{\bullet}= \ \E^{\mu_{0,\mathfrak{I}_{2}}} \left( \bar{\mathfrak{f}}_{\mathfrak{t}_{0},T,\mathfrak{I}}^{\mathfrak{I}_{2}}\mathbf{1}_{\wt{\sum}_{x \in \mathfrak{I}_{2}}\eta_{x} = \varrho} \right) \\
\bar{\mathfrak{H}}_{\varrho,\mathfrak{I}_{2}}^{\mathrm{can}}\left( \bar{\mathfrak{f}}_{\mathfrak{t}_{0},T,\mathfrak{I}}^{\mathfrak{I}_{2}} \right) \ &\overset{\bullet}= \ \mathfrak{H}_{\varrho,\mathfrak{I}_{2}}^{\mathrm{can}} \left( \frac{\d\bar{\mu}_{\mathfrak{t}_{0},T,\mathfrak{I}}^{\mathfrak{I}_{2}}}{\d\mu_{\varrho,\mathfrak{I}_{2}}^{\mathrm{can}}} \left(\mathfrak{p}_{\mathfrak{t}_{0},T,\mathfrak{I}}^{\mathfrak{I}_{2},\varrho}\right)^{-1}\mathbf{1}_{\wt{\sum}_{x \in \mathfrak{I}_{2}}\eta_{x} = \varrho}\right) \\
\bar{\mathfrak{D}}_{\varrho,\mathfrak{I}_{2}}^{\mathrm{can}}\left( \bar{\mathfrak{f}}_{\mathfrak{t}_{0},T,\mathfrak{I}}^{\mathfrak{I}_{2}} \right) \ &\overset{\bullet}= \ \mathfrak{D}_{\varrho,\mathfrak{I}_{2}}^{\mathrm{can}}\left( \frac{\d\bar{\mu}_{\mathfrak{t}_{0},T,\mathfrak{I}}^{\mathfrak{I}_{2}}}{\d\mu_{\varrho,\mathfrak{I}_{2}}^{\mathrm{can}}} \left(\mathfrak{p}_{\mathfrak{t}_{0},T,\mathfrak{I}}^{\mathfrak{I}_{2},\varrho}\right)^{-1}\mathbf{1}_{\wt{\sum}_{x \in \mathfrak{I}_{2}}\eta_{x} = \varrho}\right).
\end{align}\normalsize\normalsize
\end{subequations}
The first term is the probability under $\d\bar{\mu}_{\mathfrak{t}_{0},T,\mathfrak{I}}$ of the hyperplane in $\Omega_{\mathfrak{I}_{2}}$ with spin density $\varrho\in\R$. The second term is the relative entropy of $\d\bar{\mu}_{\mathfrak{t}_{0},T,\mathfrak{I}}$ with respect to a canonical ensemble after projecting to $\mathfrak{I}_{2}$-marginals and conditioning on the hyperplane with appropriate spin density $\varrho\in\R$. The third term is the same as the second but Dirichlet form instead of relative entropy.
\end{definition}
\begin{lemma}\label{lemma:LE}
 Suppose the functional $\varphi$ has support $\mathfrak{I}_{\varphi} \subseteq \mathfrak{I}_{N}$ where $\mathfrak{I}_{N} \subseteq \Z$ is the block defined above in \emph{Proposition \ref{prop:EProd}} but with possibly different but still arbitrarily small and universal exponent $\e\in\R_{>0}$. Given any $\kappa \in \R_{>0}$ and again for $\e \in \R_{>0}$ arbitrarily small but universal, we have the following ``reduction to equilibrium" on $\mathfrak{I}_{\varphi}$ uniformly over $\mathfrak{t}_{0},T\lesssim1$:
\small\begin{align}
\E^{\mu_{0,\Z}} \left(\bar{\mathfrak{f}}_{\mathfrak{t}_{0},T,\mathfrak{I}_{N}}^{\mathfrak{I}_{\varphi}}\cdot\varphi\right) \ &\lesssim_{\e} \ T^{-1} \kappa^{-1} N^{-\frac34+10\e} |\mathfrak{I}_{N}|^{-1} |\mathfrak{I}_{\varphi}|^{3} \ + \ \kappa^{-1} \sup_{\varrho \in \R} \log \E^{\mu_{\varrho,\mathfrak{I}_{\varphi}}^{\mathrm{can}}} \exp\left(\kappa \varphi\right).
\end{align}\normalsize\normalsize
\end{lemma}
We provide the proof of Lemma \ref{lemma:LE} towards the end of the current section as with the proof of Proposition \ref{prop:EProd}. The proof of Lemma \ref{lemma:LE} is an application of the classical entropy inequality in Appendix 1.8 of \cite{KL}, the log-Sobolev inequality of \cite{Yau}, and the entropy production bounds in Proposition \ref{prop:EProd}. We point out that the dependence on the support of the functional $\varphi$ we take expectation of in the statement of Lemma \ref{lemma:LE} is cubic on the RHS of the estimate therein. This reflects additional difficulty in comparison with local equilibrium on bigger length-scales. We also point out that if we knew $|\varphi|\lesssim\kappa^{-1}$ with universal implied constant at least at a level of large deviations, the second term on the RHS of the estimate in the statement of Lemma \ref{lemma:LE} could be reduced to just an expectation of $|\varphi|$ times universal factors courtesy of uniform sublinearity of the exponential on compact sets and concavity of the logarithm. We use this feature of the estimate in the statement of Lemma \ref{lemma:LE} later in Section \ref{section:D1B}.
\subsection{Invariant Measure Estimates}
We begin this discussion with the following local construction.
\begin{definition}\label{definition:PeriodicGenerator}
Consider a subset $\mathfrak{I} \subseteq \Z$ and define the following infinitesimal generator in which $+_{\mathfrak{I}}$ denotes addition on $\mathfrak{I}$ equipped with periodic boundary conditions $\inf\mathfrak{I} = \sup\mathfrak{I}+1$ and thus realized as a torus:
\small\begin{align}
\mathfrak{S}_{\mathfrak{I}}^{N,!!}\mathfrak{f}(\eta) \ &\overset{\bullet}= \ N^{2} \sum_{k = 1}^{\mathfrak{m}} \sum_{x \in \mathfrak{I}} \alpha_{k} \left( 1 + \frac12N^{-\frac12}\gamma_{k} \frac{1-\eta_{x}}{2}\frac{1 + \eta_{x+_{\mathfrak{I}}k}}{2} - \frac12N^{-\frac12}\gamma_{k}\frac{1+\eta_{x}}{2}\frac{1-\eta_{x+_{\mathfrak{I}}k}}{2} \right) \mathfrak{S}_{x,x+_{\mathfrak{I}}k}\mathfrak{f}(\eta).
\end{align}\normalsize\normalsize
We denote by $\eta_{\mathfrak{I};\bullet}$ the $\mathfrak{I}$-local/$\mathfrak{I}$-periodic stochastic process valued in $\Omega_{\mathfrak{I}}$ with infinitesimal generator $\mathfrak{S}_{\mathfrak{I}}^{N,!!}$. It is an exclusion process of particles that jump by the same local rules as the original exclusion process but with periodic boundary on the set $\mathfrak{I}$.
\end{definition}
We first give a general Kipnis-Varadhan inequality to control time-averages for stationary exclusion processes by a Sobolev norm. We then control the Sobolev norm. Lastly, we present martingale LDP-type bounds to study spatial mixing of canonical ensembles in the context of pseudo-gradients. We note that these results are all basically classical.
\begin{lemma}\label{lemma:KV}
Fix any $\mathfrak{I} \subseteq \Z$ and $\varrho \in \R$. For any $\varphi: \Omega_{\mathfrak{I}} \to \R$ and $\mathfrak{t}_{\mathrm{av}} \in \R_{\geq 0}$, we have the following that we explain after:
\small\begin{align}
\E^{\mu_{\varrho,\mathfrak{I}}^{\mathrm{can}}}|\mathscr{A}_{\mathfrak{I};\mathfrak{t}_{\mathrm{av}}}^{\mathbf{T},+}(\varphi(\eta_{\mathfrak{I};\bullet}))|^{2} + \E^{\mu_{\varrho,\mathfrak{I}}^{\mathrm{can}}}\sup_{0\leq\mathfrak{t}\leq\mathfrak{t}_{\mathrm{av}}}\mathfrak{t}^{2}\mathfrak{t}_{\mathrm{av}}^{-2} |\mathscr{A}_{\mathfrak{I};\mathfrak{t}}^{\mathbf{T},+}(\varphi(\eta_{\mathfrak{I};\bullet}))|^{2} \ &\lesssim \ \mathfrak{t}_{\mathrm{av}}^{-1} \|\varphi\|_{\dot{\mathbf{H}}_{\varrho,\mathfrak{I}}^{-1}}^{2}.
\end{align}\normalsize\normalsize
Above, we have introduced the following time-averages and Sobolev norms adapted to the $\mathfrak{I}$-local stochastic process:
\begin{itemize}[leftmargin=*]
\item Define the following time/dynamic-average operator acting on $\varphi$ with initial time $\bullet \in \R_{\geq 0}$ and time-scale $\mathfrak{t}\in\R_{\geq0}$:
\small\begin{align}
\mathscr{A}_{\mathfrak{I};\mathfrak{t}}^{\mathbf{T},+}(\varphi(\eta_{\mathfrak{I};\bullet})) \ &\overset{\bullet}= \ \mathbf{1}_{\mathfrak{t}\neq0}\mathfrak{t}^{-1}\int_{0}^{\mathfrak{t}}\varphi(\eta_{\mathfrak{I};\bullet+\mathfrak{r}}) \ \d\mathfrak{r} \ + \ \mathbf{1}_{\mathfrak{t}=0}\varphi(\eta_{\mathfrak{I};\bullet}).
\end{align}\normalsize\normalsize
The superscript ``$\mathbf{T},+$" denotes time-average in the positive time-direction. We will use and introduce this notation again later throughout the paper; see \emph{Section \ref{section:D1B}}, for example. We also clarify that the maximal process on the LHS of the proposed estimate is actually the supremum over time-scales $0\leq\mathfrak{t}\leq\mathfrak{t}_{\mathrm{av}}$ of the integral of $\varphi(\eta_{\mathfrak{I};\bullet})$ on the interval $[0,\mathfrak{t}]$ then weighted by $\mathfrak{t}_{\mathrm{av}}^{-1}$, and then squaring this integral and $\mathfrak{t}_{\mathrm{av}}$-weight. In particular, we only get a squared time-average if $\mathfrak{t}=\mathfrak{t}_{\mathrm{av}}$.
\item The Sobolev norm $\dot{\mathbf{H}}^{-1}$ is defined by the following, where $\bar{\mathfrak{S}}_{\mathfrak{I}}^{N,!!}$ denotes $\mathfrak{S}_{\mathfrak{I}}^{N,!!}$ without $\mathfrak{S}_{x,x+_{\mathfrak{I}}k}$-terms for which $x+_{\mathfrak{I}}k \neq x+k$ and also sets $\gamma_{k} = 0$, so $\bar{\mathfrak{S}}^{N,!!}_{\mathfrak{I}}$ is the infinitesimal generator for a finite-range symmetric exclusion process on $\mathfrak{I}\subseteq\Z$ suppressing all jumps outside the block $\mathfrak{I}\subseteq\Z$. We note that this finite-range symmetric exclusion process is also $N^{2}$ speed:
\small\begin{align}
\|\varphi\|_{\dot{\mathbf{H}}_{\varrho,\mathfrak{I}}^{-1}}^{2} \ &\overset{\bullet}= \ \sup_{\psi: \Omega_{\mathfrak{I}}\to\R} \left( 2 \E^{\mu_{\varrho,\mathfrak{I}}^{\mathrm{can}}}(\varphi \psi) \ + \ \E^{\mu_{\varrho,\mathfrak{I}}^{\mathrm{can}}}\psi\bar{\mathfrak{S}}^{N,!!}_{\mathfrak{I}}\psi \right).
\end{align}\normalsize\normalsize
\item The expectation $\E^{\mu_{\varrho,\mathfrak{I}}^{\mathrm{can}}}$ of the ``path-space" functional $|\mathscr{A}_{\mathfrak{I};\mathfrak{t}_{\mathrm{av}}}^{\mathbf{T},+}\varphi(\eta_{\mathfrak{I};\bullet})|^{2}$ and its running sup is equal to $\E^{\mu_{\varrho,\mathfrak{I}}^{\mathrm{can}}}\E_{\eta_{\mathfrak{I};\bullet}}$, where 
\begin{itemize}[leftmargin=*]
\item The inner expectation is with respect to the path-space measure induced by the $\mathfrak{S}_{\mathfrak{I}}^{N,!!}$-dynamic with initial spins $\eta_{\mathfrak{I};\bullet} \in \Omega_{\mathfrak{I}}$.
\item The outer expectation is over the above initial spin configuration $\eta_{\mathfrak{I};\bullet} \in \Omega_{\mathfrak{I}}$ with respect to the canonical ensemble $\mu_{\varrho,\mathfrak{I}}^{\mathrm{can}}$.
\end{itemize}
\end{itemize}
\end{lemma}
\begin{proof}
The first term in the claimed bound is controlled by the second by taking the latter supremum at $\mathfrak{t} = \mathfrak{t}_{\mathrm{av}}$. We note $\mu_{\varrho,\mathfrak{I}}^{\mathrm{can}}$ is an invariant measure for the $\eta_{\mathfrak{I}}$-process, which can be checked by standard procedure, and we cite Lemma 2.4 in \cite{KLO}.
\end{proof}
To use Lemma \ref{lemma:KV} we use the next set of estimates. The first is an orthogonality estimate that is certainly true for variance and comes by the spatial mixing of the canonical ensemble. The point is that orthogonality holds for the Sobolev norm too. The second estimate is a spectral gap. It estimates the Sobolev norm by a variance with constant adapted to diffusive speed-scaling in the process and the support of the functional at hand. We cite Proposition 3.3 and Proposition 3.4 in \cite{GJ15} for a proof.
\begin{lemma}\label{lemma:SpectralH-1}
Fix any $\mathfrak{I} \subseteq \Z$ and $\varrho \in \R$. Suppose we have a collection of bounded functions $\varphi_{1},\ldots,\varphi_{J}: \Omega_{\mathfrak{I}} \to \R$ so that
\begin{itemize}[leftmargin=*]
\item The functional $\varphi_{j}$ has support $\mathfrak{I}_{j} \subseteq \mathfrak{I}$, and the supports $\mathfrak{I}_{1},\ldots,\mathfrak{I}_{J} \subseteq \mathfrak{I}$ are mutually disjoint.
\item Provided any $j \in \llbracket 1,J \rrbracket$ and any $\wt{\varrho} \in \R$, we have the canonical ensemble mean zero condition $\E^{\mu_{\wt{\varrho},\mathfrak{I}_{j}}^{\mathrm{can}}} \varphi_{j} = 0$.
\end{itemize}
We define the average $\Phi_{J,N} \overset{\bullet}= \wt{\sum}_{j=1}^{J} \varphi_{j}$. We recall $\wt{\sum}$ denotes an averaged sum. We have
\small\begin{align}
\sup_{\varrho\in\R}\|\Phi_{J,N}\|_{\dot{\mathbf{H}}_{\varrho,\mathfrak{I}}^{-1}}^{2} \ &\lesssim \ J^{-1}\sup_{j=1,\ldots,J}\sup_{\varrho_{j}\in\R} \| \varphi_{j} \|_{\dot{\mathbf{H}}_{\varrho_{j},\mathfrak{I}_{j}}^{-1}}^{2} \\
\sup_{\varrho_{j}\in\R} \| \varphi_{j} \|_{\dot{\mathbf{H}}_{\varrho_{j},\mathfrak{I}_{j}}^{-1}}^{2}  \ &\lesssim \ N^{-2} |\mathfrak{I}_{j}|^{2}\sup_{\varrho_{j} \in \R} \E^{\mu_{\varrho_{j},\mathfrak{I}_{j}}^{\mathrm{can}}} |\varphi_{j}|^{2}.
\end{align}\normalsize
\end{lemma}
The orthogonality bound will be useful as we will take $\varphi_{j} = \tau_{-7j\mathfrak{m}}\mathfrak{g}$ in a future application; recall in Proposition \ref{prop:Duhamel} that these $\varphi_{j}$-terms have disjoint and uniformly bounded support. Let us also point out the double supremum on the RHS of the first estimate in the statement of Lemma \ref{lemma:SpectralH-1}. Orthogonality would usually have an average over $j\in\llbracket1,J\rrbracket$ instead of the supremum over these indices. Propositions 3.3 and 3.4 in \cite{GJ15} allow for an average instead of supremum on the RHS of the first estimate in the statement of Lemma \ref{lemma:SpectralH-1}. However, we will not need an average, and the average certainly implies an upper bound in terms of the sup as well. The $N^{-2}$-factor in the second bound of Lemma \ref{lemma:SpectralH-1} reflects time-scaling and speed to invariant measure.

Our last bound treats pseudo-gradients of large supports but whose pseudo-gradient factors have small supports, for which Lemma \ref{lemma:SpectralH-1} is suboptimal in support length.
\begin{lemma}\label{lemma:H-1SpectralPGF}
 If $\varphi: \Omega_{\Z} \to \R$ has support $\mathfrak{I}_{\varphi} \subseteq \Z$ and admits pseudo-gradient factor $\mathfrak{g}: \Omega_{\Z} \to \R$ with support $\mathfrak{I} \subseteq \Z$, then the Sobolev norm of the functional $\varphi$ is controlled by that of its pseudo-gradient factor. Precisely, with a universal implied constant
\small\begin{align}
\sup_{\varrho\in\R}\| \varphi \|_{\dot{\mathbf{H}}^{-1}_{\varrho,\mathfrak{I}_{\varphi}}} \ &\lesssim \ \sup_{\varrho\in\R}\| \mathfrak{g} \|_{\dot{\mathbf{H}}^{-1}_{\varrho,\mathfrak{I}}}.
\end{align}\normalsize\normalsize
\end{lemma}
\begin{proof}
First observe that it suffices to assume $\varphi \overset{\bullet}= \mathfrak{g} \cdot \mathfrak{f}$ with $\mathfrak{f} = \prod_{i=1}^{\ell}\eta_{x_{i}}$. Indeed, we generally have $\varphi = \mathfrak{g} \cdot \mathfrak{f}$ for some local functional $\mathfrak{f}: \Omega_{\Z} \to \R$ that is an average of such products times uniformly bounded and deterministic constants. It then suffices to apply the triangle inequality for the Sobolev norm which is elementary to deduce via its variational definition. We also note that $\mathfrak{I}\subseteq\Z$ is disjoint from the support $\{x_{i}\}_{i=1}^{\ell}$ of $\mathfrak{f}$ by definition of admitting a pseudo-gradient factor. We start as follows:
\begin{itemize}[leftmargin=*]
\item First, for any local $\psi: \Omega_{\Z} \to \R$, define $\hat{\psi} \overset{\bullet}= \psi \cdot \prod_{i = 1}^{\ell} \eta_{x_{i}} = \psi \cdot \mathfrak{f}$, so in particular $\varphi = \hat{\mathfrak{g}}$ under our notation in this proof.
\item We recall the local $\bar{\mathfrak{S}}^{N,!!}$-generators and Sobolev norms in Lemma \ref{lemma:KV}. By negative semi-definiteness of $\mathfrak{S}_{x,y}$-generators, we have the following inequality of quadratic forms for any $\psi:\Omega_{\mathfrak{I}_{\varphi}}\to\R$; recall the containment $\mathfrak{I}\subseteq\mathfrak{I}_{\varphi}$:
\small\begin{align}
\E^{\mu_{\varrho,\mathfrak{I}_{\varphi}}^{\mathrm{can}}} \psi\bar{\mathfrak{S}}_{\mathfrak{I}_{\varphi}}^{N,!!}\psi \ \leq \ \E^{\mu_{\varrho,\mathfrak{I}_{\varphi}}^{\mathrm{can}}}\psi\bar{\mathfrak{S}}_{\mathfrak{I}}^{N,!!}\psi \ \leq \ \E^{\mu_{\varrho,\mathfrak{I}_{\varphi}}^{\mathrm{can}}} \hat{\psi}\bar{\mathfrak{S}}_{\mathfrak{I}}^{N,!!}\hat{\psi}.
\end{align}\normalsize\normalsize
The last bound follows as $\bar{\mathfrak{S}}_{\mathfrak{I}}^{N,!!}$ does not act on spins outside $\mathfrak{I} \subseteq \Z$, as $\eta_{x}^{2} = 1$ for all $x \in \Z$, and as $\{x_{i}\}_{i=1}^{\ell}$ is disjoint from $\mathfrak{I}$.
\end{itemize}
With the above notation and observation, we get the following with explanation given afterwards:
\small\begin{align}
\| \varphi \|_{\dot{\mathbf{H}}^{-1}_{\varrho,\mathfrak{I}_{\varphi}}}^{2} \ \leq \ \sup_{\psi:\Omega_{\mathfrak{I}_{\varphi}}\to\R} \left( 2 \E^{\mu_{\varrho,\mathfrak{I}_{\varphi}}^{\mathrm{can}}} (\varphi \psi) \ + \ \E^{\mu_{\varrho,\mathfrak{I}_{\varphi}}^{\mathrm{can}}} \psi\bar{\mathfrak{S}}_{\mathfrak{I}_{\varphi}}^{N,!!}\psi \right) \ &\leq \ \sup_{\psi:\Omega_{\mathfrak{I}_{\varphi}}\to\R} \left(2 \E^{\mu_{\varrho,\mathfrak{I}_{\varphi}}^{\mathrm{can}}} (\mathfrak{g}\hat{\psi}) \ + \ \E^{\mu_{\varrho,\mathfrak{I}_{\varphi}}^{\mathrm{can}}} \hat{\psi}\bar{\mathfrak{S}}_{\mathfrak{I}}^{N,!!}\hat{\psi} \right).
\end{align}\normalsize\normalsize
In the final bound, we note $\varphi\psi = \mathfrak{g}\hat{\psi}$ again because $\eta_{x}^{2}=1$. This last bound is controlled by $\|\mathfrak{g}\|^{2}$ via convexity of the Dirichlet form and replacing $\hat{\psi}$ on the far RHS of the above by its conditional expectation/projection on $\mathfrak{I} \subseteq \Z$. The result follows.
\end{proof}
Let us now leverage spatial mixing of canonical ensembles in the context of pseudo-gradients supported on disjoint subsets to get LDP estimates for spatial-averages of pseudo-gradients. This will only be relevant for studying order $N^{1/2}$-terms in $\Phi^{N,2}$ in Proposition \ref{prop:Duhamel}; recall there was no analog for order $N^{\beta_{X}}$-terms in $\Phi^{N,2}$ for Pseudo-Proposition \ref{pprop:S2}. 
\begin{lemma}\label{lemma:LDP}
 Retain the setting of \emph{Lemma \ref{lemma:SpectralH-1}}, and additionally assume the following a priori estimate:
\begin{itemize}[leftmargin=*]
\item We have the deterministic estimate $|\varphi_{j}| \lesssim 1$ uniformly in $\eta \in \Omega_{\Z}$ and $j\in\llbracket1,J\rrbracket$ with universal implied constant.
\end{itemize}
For any $\mathfrak{C} \in \R_{>0}$ and $\varrho \in \R$, we have the following sub-Gaussian tail bound with universal implied constants:
\small\begin{align}
\mathbf{P}^{\mu_{\varrho,\mathfrak{I}}^{\mathrm{can}}} [|\Phi_{J,N}| \gtrsim \mathfrak{C}] \ &\lesssim \ \exp\left(-J\mathfrak{C}^{2}\right).
\end{align}\normalsize\normalsize
\end{lemma}
\begin{proof}
Consider the filtration of $\sigma$-algebras $\sigma(\mathfrak{I}_{1},\ldots,\mathfrak{I}_{j})$ for $j \in \llbracket 1,J \rrbracket$. Under the canonical ensemble $\mu_{\varrho,\mathfrak{I}}^{\mathrm{can}}$ the $\varphi_{j}$-quantities are uniformly bounded martingale increments since the projection of any canonical ensemble onto a smaller block is a convex combination of canonical ensembles. We then use the Azuma martingale inequality.
\end{proof}
\begin{corollary}\label{corollary:LDP}
 Retain the setting of \emph{Lemma \ref{lemma:LDP}}. For any $\e_{1},\e_{2},C\in\R_{>0}$, with universal implied constants we have
\small\begin{align}
\E^{\mu_{\varrho,\mathfrak{I}}^{\mathrm{can}}} \exp\left(N^{-\e_{1}}J|\Phi_{J,N}|^{2}\right) \cdot \mathbf{1}\left(|\Phi_{J,N}| \gtrsim N^{\e_{2}}J^{-\frac12}\right) \ &\lesssim_{\e_{1},\e_{2},C} \ N^{-C}.
\end{align}\normalsize\normalsize
\end{corollary}
\begin{proof}
By Lemma \ref{lemma:LDP} $\Phi_{J,N}$ has zero mean and is sub-Gaussian of variance of order $J^{-1}$, so it suffices to pretend $\Phi_{J,N}$ is indeed Gaussian with zero mean and variance of order $J^{-1}$. The result then follows by straightforward calculations for Gaussians.
\end{proof}
The rest of the section is dedicated to proof of Proposition \ref{prop:EProd} and Lemma \ref{lemma:LE}. We invite the reader to skip to the beginning of the upcoming Section \ref{section:Ctify} as the proof of Proposition \ref{prop:EProd}, especially, is quite technical.
\subsection{Proof of Proposition \ref{prop:EProd}}
We first give ingredients needed for proof of Proposition \ref{prop:EProd} with brief descriptions of what they say and why they are true. We then combine these ingredients to get Proposition \ref{prop:EProd} and close with their proofs. For convenience, let us declare all expectations in this subsection are taken with respect to the grand-canonical measure $\mu_{0,\Z}$.

The first step we take is the following a priori estimate that serves as an easier but weaker version of Proposition \ref{prop:EProd}. Roughly speaking, such a priori estimate reduces to an exclusion process on a torus of length-scale which is basically $N^{3/2}$. The explain the exponent $3/2$, the particles in the exclusion process perform symmetric random walks of speed $N^{2}$ and a speed $N^{3/2}$ asymmetry. Thus, we expect only information at length-scale $(N^{2})^{1/2}+N^{3/2}\lesssim N^{3/2}$ to matter.
\begin{lemma}\label{lemma:EProdSoP}
 For $\e\in\R_{>0}$ arbitrarily small but universal, define $\mathfrak{I}_{N,+}\overset{\bullet}= \llbracket-N^{\frac32+\e},N^{\frac32+\e}\rrbracket\subseteq\Z$. We also define the following:
\begin{itemize}[leftmargin=*]
\item Consider the perturbation $\wt{\mu}_{0,N}$ as a probability measure on the product $\Omega_{\Z} = \Omega_{\Z\setminus\mathfrak{I}_{N,+}}\times\Omega_{\mathfrak{I}_{N,+}}$. In words, the measure $\wt{\mu}_{0,N}$ is defined by taking the measure $\mu_{0,N}$, cutting it off outside $\mathfrak{I}_{N,+}$, and gluing the grand-canonical measure $\mu_{0,\Z}$ outside $\mathfrak{I}_{N,+}$:
\small\begin{align}
\d\wt{\mu}_{0,N} \ \overset{\bullet}= \ \d\mu_{0,\Z\setminus\mathfrak{I}_{N,+}} \otimes \d\mu_{0,N}^{\mathfrak{I}_{N,+}}.
\end{align}\normalsize\normalsize
\item Analogously define $\wt{\mu}_{T,N}$ as the measure/law after time-$T$ evolution with initial measure $\wt{\mu}_{0,N}$, and define $\wt{\mathfrak{f}}_{T,N}\d\mu_{0,\Z}\overset{\bullet}=\d\wt{\mu}_{T,N}$. 
\end{itemize}
We first have the following a priori global entropy production estimate for the perturbation initial measure $\wt{\mu}_{0,N}$:
\small\begin{align}
\int_{0}^{T} \mathfrak{D}_{0,\Z}(\wt{\mathfrak{f}}_{S,N}) \ \d S \ &\lesssim \ N^{-\frac12+\e}. \label{eq:EProdSoP1}
\end{align}\normalsize\normalsize
We have comparison of Dirichlet forms on $\mathfrak{I}_{N} \overset{\bullet}= \llbracket -N^{\frac54 + \e}, N^{\frac54 + \e} \rrbracket$ in \emph{Proposition \ref{prop:EProd}} for any $C>0$ and $0\leq T\leq100$:
\small\begin{align}
|\mathfrak{D}_{0,\mathfrak{I}_{N}}(\mathfrak{f}_{T,N}^{\mathfrak{I}_{N}}) - \mathfrak{D}_{0,\mathfrak{I}_{N}}(\wt{\mathfrak{f}}_{T,N}^{\mathfrak{I}_{N}})| \ &\lesssim_{C,\e} \ N^{-C}. \label{eq:EProdSoP0}
\end{align}\normalsize\normalsize
\end{lemma}
The second bound \eqref{eq:EProdSoP0} is the assertion that information beyond the aforementioned length-scale $N^{3/2}$, up to $N^{\e}$-factors, will not matter for Dirichlet forms. The first estimate above exploits the ``stationary cutoff" beyond the length-scale $N^{3/2+\e}$.

We now implement the multiscale strategy in the proofs of Lemma 4.1 in \cite{DT} and of Theorem 2.1 in \cite{LM}. It relates entropy production on a block to that on a slightly-larger block. To explain this in more detail, we require the following constructions.
\begin{definition}\label{definition:EProdProof}
During the proof of Proposition \ref{prop:EProd} given in this subsection, we use the following constructions.
\begin{itemize}[leftmargin=*]
\item Given $\mathfrak{l} \in \Z_{\geq0}$, define $e_{N,\mathfrak{l}} \overset{\bullet}= \exp(-N^{-5/4-\e}\mathfrak{l})$ and $\mathfrak{I}_{N,\mathfrak{l}} \overset{\bullet}= \llbracket-N^{\frac54+\e} - \mathfrak{l}\cdot N^{\e}, N^{\frac54+\e}+\mathfrak{l}\cdot N^{\e}\rrbracket \subseteq \Z$. This family of subsets is a progressive ``fattening" of the original $\mathfrak{I}_{N}$-subset in Proposition \ref{prop:EProd}. The weights $e_{N,\mathfrak{l}}$ will be pieces of a multiscale analysis that controls entropy production on the block $\mathfrak{I}_{N,\mathfrak{l}}$ by that on $\mathfrak{I}_{N,\mathfrak{l}+\mathfrak{j}}$; see proofs of Lemma 4.1 in \cite{DT} and Theorem 2.1 in \cite{LM}.
\item For any $x,y \in \Z$, define the projected densities $\mathfrak{f}_{\mathfrak{l}}\overset{\bullet}=\mathfrak{f}_{T,N}^{\mathfrak{I}_{N,\mathfrak{l}}}$ and $\mathfrak{f}_{\mathfrak{l},x,y}\overset{\bullet}=\mathfrak{f}_{T,N}^{\mathfrak{I}_{N,\mathfrak{l}}\cup\{x,y\}}$. The former has entropy production related to Dirichlet forms of the latter. The latter is not in the proof of Lemma 4.1 in \cite{DT}, but it is in the proof of Theorem 2.1 in \cite{LM}.
\item Define the boundary $\partial\mathfrak{I}_{N,\mathfrak{l}} \overset{\bullet}= \{(x,y) \in \Z: \ x\in\mathfrak{I}_{N,\mathfrak{l}}, y\not\in\mathfrak{I}_{N,\mathfrak{l}}, x < y\} \cup \{(x,y) \in \Z: \ x\not\in\mathfrak{I}_{N,\mathfrak{l}}, y\in\mathfrak{I}_{N,\mathfrak{l}}, x < y\}$. Note that pairs of points in the boundary $\partial\mathfrak{I}_{N,\mathfrak{l}}$ do \emph{not} have to be nearest neighbors. Roughly, entropy production on $\mathfrak{I}_{N,\mathfrak{l}}$ forces us to look at the boundary $\partial\mathfrak{I}_{N,\mathfrak{l}}$ as particles can jump in/out of $\mathfrak{I}_{N,\mathfrak{l}}$; see proofs of Theorem 2.1 in \cite{LM} and Lemma 4.1 of \cite{DT}.
\end{itemize}
\end{definition}
As mentioned earlier, we relate entropy production on block $\mathfrak{I}_{N,\mathfrak{l}}$ and entropy production on a small neighborhood $\mathfrak{I}_{N,\mathfrak{l}+\mathfrak{j}}$. This is the lemma below, which introduces boundary terms $\Phi_{\mathrm{Sym}},\Phi_{\wedge,1},\Phi_{\wedge,2}$ that we treat with another pair of lemmas.
\begin{lemma}\label{lemma:EProd2}
Provided any $\mathfrak{l}\in\Z_{\geq1}$, we have the following entropy production differential inequality:
\small\begin{align}
\partial_{T}\mathfrak{H}_{0,\mathfrak{I}_{N,\mathfrak{l}}}(\mathfrak{f}_{\mathfrak{l}}) + 2N^{2}\mathfrak{D}_{0,\mathfrak{I}_{N,\mathfrak{l}}}(\mathfrak{f}_{\mathfrak{l}}) \ &\leq \ \Phi_{\mathrm{Sym}} + \Phi_{\wedge,1} + \Phi_{\wedge,2}. \label{eq:EProd21}
\end{align}\normalsize\normalsize
The RHS comes from jumps at $\partial\mathfrak{I}_{N,\mathfrak{l}}$. In particular $\Phi_{\mathrm{Sym}}$ comes from jumps at $\partial\mathfrak{I}_{N,\mathfrak{l}}$ in the symmetric part of the process and $\Phi_{\wedge,1},\Phi_{\wedge,2}$ come from asymmetric such jumps. Recall $\mathfrak{S}_{x,y}$ is the generator of the speed 1 symmetric exclusion on $\{x,y\}\subseteq\Z$:
\begin{subequations}
\small\begin{align}
\Phi_{\mathrm{Sym}} \ &\overset{\bullet}= \ 2N^{2}\E\sum_{(x,y)\in\partial\mathfrak{I}_{N,\mathfrak{l}}} \alpha_{|y-x|}\left((\mathfrak{f}_{\mathfrak{l},x,y}-\mathfrak{f}_{\mathfrak{l}})\mathfrak{f}_{\mathfrak{l}}^{-1/2} \mathfrak{S}_{x,y}(\mathfrak{f}_{\mathfrak{l}}^{1/2})\right) \\
\Phi_{\wedge,1} \ &\overset{\bullet}= \ N^{\frac32}\E\sum_{(x,y)\in\partial\mathfrak{I}_{N,\mathfrak{l}}}\mathbf{1}_{x\leq y}\alpha_{|y-x|}\gamma_{|y-x|}\mathbf{1}_{\eta_{x}=-1}\mathbf{1}_{\eta_{y}=1}\left((\mathfrak{f}_{\mathfrak{l},x,y}-\mathfrak{f}_{\mathfrak{l}})\mathfrak{f}_{\mathfrak{l}}^{-1/2} \mathfrak{S}_{x,y}(\mathfrak{f}_{\mathfrak{l}}^{1/2})\right) \\
\Phi_{\wedge,2} \ &\overset{\bullet}= \ -N^{\frac32}\E\sum_{(x,y)\in\partial\mathfrak{I}_{N,\mathfrak{l}}}\mathbf{1}_{x\leq y}\alpha_{|y-x|}\gamma_{|y-x|}\mathbf{1}_{\eta_{x}=1}\mathbf{1}_{\eta_{y}=-1}\left((\mathfrak{f}_{\mathfrak{l},x,y}-\mathfrak{f}_{\mathfrak{l}})\mathfrak{f}_{\mathfrak{l}}^{-1/2} \mathfrak{S}_{x,y}(\mathfrak{f}_{\mathfrak{l}}^{1/2})\right).
\end{align}\normalsize\normalsize
\end{subequations}
\end{lemma}
\begin{lemma}\label{lemma:EProdSym}
Retaining the context of the current subsection, with a universal implied constant we have the following control of boundary-terms in $\Phi_{\mathrm{Sym}}$ coming from symmetric jumps in the particle system in terms of data on larger $\mathfrak{I}_{N,\mathfrak{l}+\mathfrak{j}}$ blocks:
\small\begin{align}
|\Phi_{\mathrm{Sym}}| \ &\lesssim \ N^{\frac32}{\sum}_{(x,y)\in\partial\mathfrak{I}_{N,\mathfrak{l}}}\alpha_{|x-y|}{\sum}_{\mathfrak{j}=1}^{N^{5/4}} \E |\mathfrak{S}_{x,y}(\mathfrak{f}_{\mathfrak{l}+\mathfrak{j}}^{\frac12})|^{2} \ + \ N^{\frac54}.
\end{align}\normalsize\normalsize
\end{lemma}
\begin{lemma}\label{lemma:EProdWedge}
Retaining the context of the current subsection, we have $|\Phi_{\wedge,1}| + |\Phi_{\wedge,2}|\lesssim N^{3/2}(\mathfrak{D}_{0,\Z}(\mathfrak{f}_{T,N}))^{1/2}$.
\end{lemma}
\begin{proof}[Proof of \emph{Proposition \ref{prop:EProd}}]
By the second estimate in Lemma \ref{lemma:EProdSoP}, we may take $\mathfrak{f}=\wt{\mathfrak{f}}$ with the notation therein. For convenience we take $T=1$. The argument for general uniformly bounded $T\in\R_{\geq0}$ follows identically. Lemmas \ref{lemma:EProd2}, \ref{lemma:EProdSym}, and \ref{lemma:EProdWedge} give
\small\begin{align}
\partial_{T}\mathfrak{H}_{0,\mathfrak{I}_{N,\mathfrak{l}}}(\mathfrak{f}_{\mathfrak{l}}) \ + \ 2N^{2}\mathfrak{D}_{0,\mathfrak{I}_{N,\mathfrak{l}}}(\mathfrak{f}_{\mathfrak{l}}) \ &\lesssim \ N^{\frac32}\sum_{(x,y)\in\partial\mathfrak{I}_{N,\mathfrak{l}}}\alpha_{|x-y|}{\sum}_{\mathfrak{j}=1}^{N^{5/4}} \E |\mathfrak{S}_{x,y}(\mathfrak{f}_{\mathfrak{l}+\mathfrak{j}}^{\frac12})|^{2} \ + \ N^{\frac54} \ + \ N^{\frac32}(\mathfrak{D}_{0,\Z}(\mathfrak{f}_{T,N}))^{\frac12}. \label{eq:EProd1}
\end{align}\normalsize\normalsize
We integrate \eqref{eq:EProd1} in time while recalling the ``worst-case" relative entropy bound $\mathfrak{H}_{0,\mathfrak{I}_{N,\mathfrak{l}}}(\mathfrak{f}_{\mathfrak{l}}) \lesssim |\mathfrak{I}_{N,\mathfrak{l}}| \lesssim N^{5/4+\e}+\mathfrak{l}\cdot N^{\e}$ to get the following integrated estimate, from which the second line follows from the Cauchy-Schwarz inequality applied to the last integral on the RHS of the first line. The last line follows by the global bound/the first estimate of Lemma \ref{lemma:EProdSoP}:
\small\begin{align}
\int_{0}^{1}\mathfrak{D}_{0,\mathfrak{I}_{N,\mathfrak{l}}}(\mathfrak{f}_{\mathfrak{l}}) \ \d T \ &\lesssim \ N^{-\frac34+\e}+\mathfrak{l}\cdot N^{-2+\e} \ + \ \int_{0}^{1}\left(N^{-\frac12}\sum_{(x,y)\in\partial\mathfrak{I}_{N,\mathfrak{l}}}\alpha_{|x-y|}{\sum}_{\mathfrak{j}=1}^{N^{5/4}} \E |\mathfrak{S}_{x,y}(\mathfrak{f}_{\mathfrak{l}+\mathfrak{j}}^{\frac12})|^{2}\right)\d T \ + \ N^{-\frac12}\int_{0}^{1}(\mathfrak{D}_{0,\Z}(\mathfrak{f}_{T,N}))^{\frac12} \ \d T \nonumber \\
&\lesssim \ N^{-\frac34+\e}+\mathfrak{l}\cdot N^{-2+\e} + \int_{0}^{1}\left(N^{-\frac12}\sum_{(x,y)\in\partial\mathfrak{I}_{N,\mathfrak{l}}}\alpha_{|x-y|}{\sum}_{\mathfrak{j}=1}^{N^{5/4}} \E |\mathfrak{S}_{x,y}(\mathfrak{f}_{\mathfrak{l}+\mathfrak{j}}^{\frac12})|^{2}\right)\d T  +  N^{-\frac12}\left(\int_{0}^{1}\mathfrak{D}_{0,\Z}(\mathfrak{f}_{T,N}) \ \d T\right)^{\frac12} \nonumber \\
&\lesssim \ N^{-\frac34+\e}+\mathfrak{l}\cdot N^{-2+\e} + \int_{0}^{1}\left(N^{-\frac12}\sum_{(x,y)\in\partial\mathfrak{I}_{N,\mathfrak{l}}}\alpha_{|x-y|}{\sum}_{\mathfrak{j}=1}^{N^{5/4}} \E |\mathfrak{S}_{x,y}(\mathfrak{f}_{\mathfrak{l}+\mathfrak{j}}^{\frac12})|^{2}\right)\d T. \label{eq:EProd2}
\end{align}\normalsize\normalsize
We average over all scales $\mathfrak{l}\in\Z_{\geq1}$ against $e_{N,\mathfrak{l}}$-weights like the proofs of Theorem 2.1 in \cite{LM} and Lemma 4.1 in \cite{DT}. Because we average, the first two terms in \eqref{eq:EProd2} are okay. The proofs of Theorem 2.1 in \cite{LM} and Lemma 4.1 in \cite{DT} tell us the following for the third term in \eqref{eq:EProd2}. This double sum is given by ``local" Dirichlet forms/jumps at the boundary $\partial\mathfrak{I}_{N,\mathfrak{l}}$ of projections $\mathfrak{f}_{\mathfrak{l}+\mathfrak{j}}$ corresponding to $\mathfrak{I}_{N,\mathfrak{l}+\mathfrak{j}}$ subsets which are no more than slightly larger than the length-scale $|\mathfrak{I}_{N,\mathfrak{l}}|$. When we sum over all scales $\mathfrak{l}\in\Z_{\geq1}$ as in proof of Theorem 2.1 in \cite{LM}/Lemma 4.1 in \cite{DT} against exponential weights $e_{N,\mathfrak{l}}$, the ``local Dirichlet forms"/third term in \eqref{eq:EProd2} gets washed out by the $\mathfrak{D}_{0,\mathfrak{I}_{N,\mathfrak{l}+\mathfrak{j}}}(\mathfrak{f}_{\mathfrak{l}+\mathfrak{j}})$-terms from the $\mathfrak{I}_{N,\mathfrak{l}+\mathfrak{j}}$-equation (\eqref{eq:EProd1} for $\mathfrak{l}+\mathfrak{j}$ in place of $\mathfrak{l}$). Upon recalling the $e_{N,\mathfrak{l}}$-weights defined in Definition \ref{definition:EProdProof}, this paragraph combined with the previous estimate \eqref{eq:EProd2} provides 
\begin{align*}
N^{-\frac54}{\sum}_{\mathfrak{l}=1}^{\infty}e_{N,\mathfrak{l}} \int_{0}^{1}\mathfrak{D}_{0,\mathfrak{I}_{N,\mathfrak{l}}}(\mathfrak{f}_{\mathfrak{l}}) \ \d T \ &\lesssim \ N^{-\frac54}{\sum}_{\mathfrak{l}=1}^{\infty}e_{N,\mathfrak{l}}\left( N^{-\frac34+\e}+\mathfrak{l}\cdot N^{-2+\e}\right) \ \lesssim \ N^{-\frac34+5\e}.
\end{align*}
The last bound uses the elementary bounds $\sum_{\mathfrak{l}=1}^{\infty}e_{N,\mathfrak{l}} \lesssim N^{5/4+\e}$ and $\sum_{\mathfrak{l}=1}^{\infty}e_{N,\mathfrak{l}}\mathfrak{l}\lesssim N^{5/2+2\e}$. Like Lemma 4.1 in \cite{DT} and Theorem 2.1 in \cite{LM}, this last estimate combined with convexity of the Dirichlet form completes the proof of Proposition \ref{prop:EProd}.
\end{proof}
It remains to get Lemma \ref{lemma:EProdSoP}, Lemma \ref{lemma:EProd2}, Lemma \ref{lemma:EProdSym}, and Lemma \ref{lemma:EProdWedge}. We defer their proofs to the end of this section.
\subsection{Proof of Lemma \ref{lemma:LE}}
We will first decompose the post-projection space-time averaged measure $\bar{\mu}_{\mathfrak{t}_{0},T,\mathfrak{I}_{N}}^{\mathfrak{I}_{\varphi}}$ into conditional expectations onto the hyperplanes of Definition \ref{definition:ensembles}, write them as densities with respect to appropriate canonical ensembles, and apply the entropy inequality in Appendix 1.8 of \cite{KL} for constant $\kappa \in \R_{>0}$ for these densities with respect to the canonical ensembles. This gives the following preliminary identity followed by an entropy inequality estimate:
\small\begin{align}
\E^{\mu_{0,\Z}} \left(\bar{\mathfrak{f}}_{\mathfrak{t}_{0},T,\mathfrak{I}_{N}}^{\mathfrak{I}_{\varphi}} \cdot \varphi\right) \ = \ \E^{\bar{\mu}_{\mathfrak{t}_{0},T,\mathfrak{I}_{N}}^{\mathfrak{I}_{\varphi}}}\varphi \ &= \ \sum_{\varrho \in \R} \mathfrak{p}_{\mathfrak{t}_{0},T,\mathfrak{I}_{N}}^{\mathfrak{I}_{\varphi},\varrho}\E^{\mu_{\varrho,\mathfrak{I}_{\varphi}}^{\mathrm{can}}} \left(\varphi \cdot \frac{\d\bar{\mu}_{\mathfrak{t}_{0},T,\mathfrak{I}_{N}}^{\mathfrak{I}_{\varphi}}}{\d\mu_{\varrho,\mathfrak{I}_{\varphi}}^{\mathrm{can}}}[\mathfrak{p}_{\mathfrak{t}_{0},T,\mathfrak{I}_{N}}^{\mathfrak{I}_{\varphi},\varrho}]^{-1}\mathbf{1}_{\wt{\sum}_{x\in\mathfrak{I}_{\varphi}}\eta_{x}=\varrho}\right)\\
&\lesssim \ \kappa^{-1}\sum_{\varrho \in \R} \mathfrak{p}_{\mathfrak{t}_{0},T,\mathfrak{I}_{N}}^{\mathfrak{I}_{\varphi},\varrho} \bar{\mathfrak{H}}_{\varrho,\mathfrak{I}_{\varphi}}^{\mathrm{can}}( \bar{\mathfrak{f}}_{\mathfrak{t}_{0},T,\mathfrak{I}_{N}}^{\mathfrak{I}_{\varphi}}) \ + \ \kappa^{-1}\sum_{\varrho \in \R} \mathfrak{p}_{\mathfrak{t}_{0},T,\mathfrak{I}_{N}}^{\mathfrak{I}_{\varphi},\varrho} \log \E^{\mu_{\varrho,\mathfrak{I}_{\varphi}}^{\mathrm{can}}}\exp\left(\kappa \varphi\right).
\end{align}\normalsize\normalsize
Sums over $\R$ are over finite sets of possible spin densities on $\mathfrak{I}_{\varphi}$. As $\mathfrak{p}$-terms sum to 1 on the finite set of possible $\varrho \in \R$ because they are probabilities under the same probability measure of disjoint hyperplanes, we get 
\small\begin{align}
\kappa^{-1}{\sum}_{\varrho \in \R} \mathfrak{p}_{\mathfrak{t}_{0},T,\mathfrak{I}_{N}}^{\mathfrak{I}_{\varphi},\varrho} \log \E^{\mu_{\varrho,\mathfrak{I}_{\varphi}}^{\mathrm{can}}}\exp\left(\kappa \varphi\right) \ &\leq \ \kappa^{-1}{\sup}_{\varrho \in \R} \log \E^{\mu_{\varrho,\mathfrak{I}_{\varphi}}^{\mathrm{can}}}\exp\left(\kappa \varphi\right).
\end{align}\normalsize\normalsize
Concerning the relative entropy, we apply the LSI within Theorem A in \cite{Yau} to introduce a $\mathfrak{I}_{\varphi}$-local Dirichlet form estimate in terms of canonical ensemble Dirichlet forms which glue to a grand-canonical Dirichlet form:
\small\begin{align}
\kappa^{-1}{\sum}_{\varrho \in \R} \mathfrak{p}_{\mathfrak{t}_{0},T,\mathfrak{I}_{N}}^{\mathfrak{I}_{\varphi},\varrho}\bar{\mathfrak{H}}_{\varrho,\mathfrak{I}_{\varphi}}^{\mathrm{can}}\left(\bar{\mathfrak{f}}_{\mathfrak{t}_{0},T,\mathfrak{I}_{N}}^{\mathfrak{I}_{\varphi}}\right) \ \lesssim \ \kappa^{-1} |\mathfrak{I}_{\varphi}|^{2} {\sum}_{\varrho \in \R} \mathfrak{p}_{\mathfrak{t}_{0},T,\mathfrak{I}_{N}}^{\mathfrak{I}_{\varphi},\varrho} \bar{\mathfrak{D}}_{\varrho,\mathfrak{I}_{\varphi}}^{\mathrm{can}}\left(\bar{\mathfrak{f}}_{\mathfrak{t}_{0},T,\mathfrak{I}_{N}}^{\mathfrak{I}_{\varphi}}\right) \ &= \ \kappa^{-1} |\mathfrak{I}_{\varphi}|^{2} \mathfrak{D}_{0,\mathfrak{I}_{\varphi}}\left(\bar{\mathfrak{f}}_{\mathfrak{t}_{0},T,\mathfrak{I}_{N}}^{\mathfrak{I}_{\varphi}}\right)
\end{align}\normalsize\normalsize
where the last identity follows from the definition of the $\bar{\mathfrak{D}}$-functional, or equivalently by definition of conditional probability and conditional expectation. Via convexity of the Dirichlet form in Corollary 10.3 in Appendix 1.10 of \cite{KL}, as in the classical one-block estimate in Theorem 2.4 in \cite{GPV}, the final Dirichlet form is controlled by the LHS of the bound of Proposition \ref{prop:EProd} but with an additional factor of $|\mathfrak{I}_{N}|^{-1}$ from spatial averaging and $|\mathfrak{I}_{\varphi}|$ from redundancies in counting contributions from jumps. Here in our application of Proposition \ref{prop:EProd}, we choose $\e\in\R_{>0}$ therein to be $2\e$ where the $\e\in\R_{>0}$-factor in $2\e$ is the $\e$-value we picked in the current Lemma \ref{lemma:LE}. Equivalently, we use Proposition \ref{prop:EProd} for twice the $\e$-exponent we are using here. This gives
\small\begin{align}
\kappa^{-1} |\mathfrak{I}_{\varphi}|^{2} \mathfrak{D}_{0,\mathfrak{I}_{\varphi}}(\bar{\mathfrak{f}}_{\mathfrak{t}_{0},T,\mathfrak{I}_{N}}^{\mathfrak{I}_{\varphi}}) \ &\lesssim_{\e} \ T^{-1} \kappa^{-1} N^{-\frac34+10\e} |\mathfrak{I}_{N}|^{-1} |\mathfrak{I}_{\varphi}|^{3}.
\end{align}\normalsize\normalsize
We clarify that we use Proposition \ref{prop:EProd} with the initial measure therein to be $\mu_{\mathfrak{t}_{0},N}$ here. This completes the proof. \qed
\subsection{Proofs of Technical Estimates}
In this last subsection, we provide proofs of the technical results Lemma \ref{lemma:EProdSoP}, Lemma \ref{lemma:EProd2}, Lemma \ref{lemma:EProdSym}, and Lemma \ref{lemma:EProdWedge} that we used in the proof of Proposition \ref{prop:EProd}. We prove them in the order they are written.
\begin{proof}[Proof of \emph{Lemma \ref{lemma:EProdSoP}}]
The first estimate \eqref{eq:EProdSoP1} is standard. We cite Theorem 9.2 in Appendix 1 of \cite{KL} and use the observation that $\mathfrak{H}_{0,\Z}(\wt{\mathfrak{f}}_{0,N}) \lesssim N^{3/2+\e}$ and use the diffusive scaling of $N^{2}$ in the infinitesimal generator $\mathfrak{S}^{N,!!}$. To obtain \eqref{eq:EProdSoP0}, we construct a coupling via two-species process to relate $\mathfrak{f}$ and $\wt{\mathfrak{f}}$ from which we get a Dirichlet form comparison by a total variation bound.
\begin{itemize}[leftmargin=*]
\item Consider a coupling of initial spin configurations distributed via $\mu_{0,N}$ and $\wt{\mu}_{0,N}$ respectively by sampling a first configuration via $\mu_{0,N}$, and then sampling a configuration via $\wt{\mu}_{0,N}$ by taking the projection of the first configuration via $\mu_{0,N}$ onto $\mathfrak{I}_{N,+} \subseteq \Z$ and sampling the remaining spins via the grand-canonical ensemble $\mu_{0,\Z\setminus\mathfrak{I}_{N,+}}$. Recall $\mathfrak{I}_{N,+}$ in the statement of Lemma \ref{lemma:EProdSoP}.
\item We couple the above configurations in the following fashion. First, we observe the underlying symmetric exclusion process may be realized as attaching Poisson clocks to every \emph{bond} connecting points inside $\Z$. Every step in the symmetric exclusion process, or equivalently every ringing of any Poisson clock, then corresponds to swapping the spins at those two points in the corresponding bond. We will first couple the dynamics of the symmetric exclusion model underlying the total dynamic of the two configurations by coupling these ``spin-swap" dynamics. We emphasize that this is not the basic coupling. In particular, one configuration swaps possibly identical spins along a bond via this ``spin-swap" dynamic whenever the other one does.
\item The remaining Poisson clocks associated to the totally asymmetric finite-range exclusion process are equipped with the basic coupling, so particles in the two configurations jump together when possible when the totally asymmetric part rings.
\item Initially, all discrepancies between these two configurations are supported outside $\mathfrak{I}_{N,+} \subseteq \Z$. We then consider the following decomposition $\Z\setminus\mathfrak{I}_{N,+} = \cup_{\mathfrak{l} = 1} \mathfrak{I}_{N,\mathfrak{l}}$, where each sub-lattice $\mathfrak{I}_{N,\mathfrak{l}}$ satisfies $|\mathfrak{I}_{N,\mathfrak{l}}| \lesssim N^{2}$ with universal implied constant and the distance between $\mathfrak{I}_{N,\mathfrak{l}}$ and $\mathfrak{I}_{N}$ is at least $N^{3/2+\e}\mathfrak{l}$ times universal factors. Discrepancies among the two initial configurations originally in $\mathfrak{I}_{N,\mathfrak{l}}$ must then travel a distance of at least $N^{3/2+\e}\mathfrak{l}$, times universal factors, to land in $\mathfrak{I}_{N}\overset{\bullet}= \llbracket -N^{5/4+\e}, N^{5/4+\e} \rrbracket$.

\item Under the above semi-basic coupling, the dynamics of any tagged discrepancy are a \emph{free and unsuppressed} symmetric finite-range random walk of speed $N^{2}$ and randomly suppressed/killed asymmetric finite-range random walk of speed $N^{3/2}$. We emphasize the free/unsuppressed nature of the symmetric random walk with speed $N^{2}$ comes from coupling the symmetric parts of the two-species exclusion process via spin-swap bond dynamics rather than the basic coupling.

\item So, the probability we find a discrepancy among the respective particle configurations originally in $\mathfrak{I}_{N,\mathfrak{l}}\subseteq\Z$ that propagates into $\mathfrak{I}_{N} \subseteq \mathfrak{I}_{N,+} \subseteq \Z$ before time 100 is bounded above by the probability that a symmetric finite-range random walk with an asymmetry of speed of order $N^{3/2}$ traveled a distance $N^{3/2+\e}\mathfrak{l}$. Because each $\mathfrak{I}_{N,\mathfrak{l}}$ subset has at most $N^{2}$-many discrepancies by considering the size of this block, the probability that we find any discrepancy appearing in $\mathfrak{I}_{N} \subseteq \mathfrak{I}_{N,+} \subseteq \Z$ that originally came from $\mathfrak{I}_{N,\mathfrak{l}}$ is at most the previous exit probability for random walks with an additional factor $N^{2}$ coming from the union bound over all possible discrepancies originally in $\mathfrak{I}_{N,\mathfrak{l}}$. We then take another union bound and sum over all subsets $\mathfrak{I}_{N,\mathfrak{l}}$ with $\mathfrak{l} \in \Z_{\geq 1}$. This probability also provides an estimate for total variation distance between $\wt{\mathfrak{f}}^{\mathfrak{I}_{N}}$ and $\mathfrak{f}^{\mathfrak{I}_{N}}$ as is usual for couplings, because each is a probability density with respect to $\mu_{0,\mathfrak{I}_{N}}$ for the law of one of the two configurations at time $T\in\R_{\geq0}$. If $\Sigma$ is the previous random walk with $\Sigma_{0} = 0$, the Doob maximal inequality then gives, for any $C\in\R_{>1}$, the total variation bound
\small\begin{align}
\sup_{0\leq T\leq100} \E^{\mu_{0,\mathfrak{I}_{N}}}|\wt{\mathfrak{f}}_{T,N}^{\mathfrak{I}_{N}} - \mathfrak{f}_{T,N}^{\mathfrak{I}_{N}}| \ \lesssim \ {\sum}_{\mathfrak{l}=1}^{\infty} \ N^{2} \mathbf{P}\left(\sup_{0\leq\mathfrak{t}\leq200}|\Sigma_{\mathfrak{t}}| \gtrsim N^{\frac32+\e}\mathfrak{l}\right) \ \lesssim \ {\sum}_{\mathfrak{l}=1}^{\infty} \ \mathfrak{l}^{-C} N^{2} N^{-C\e} N^{-\frac32C}\E|\Sigma_{200}|^{C} \ &\lesssim_{C} \ N^{2-C\e}.\nonumber
\end{align}\normalsize\normalsize
\end{itemize}
Above, we observe $\E|\Sigma_{200}|^{C}\lesssim_{C}N^{\frac32C}$ by standard random walk estimates because it is symmetric of speed $N^{2}$ with speed-$N^{3/2}$ asymmetry. We now estimate the Dirichlet form in terms of the LHS. In what follows all expectations are taken with respect to the invariant measure $\mu_{0,\Z}$. Note $\mu_{0,\Z}$ is invariant under spin-swaps; thus it is an invariant measure under the $\mathfrak{S}_{x,y}$-generators. We now make a list of elementary but useful observations for the next step.
\begin{itemize}[leftmargin=*]
\item First observe the deterministic identities $a^{2}-b^{2}=(a-b)(a+b)$ and $(a+b)^{2}\lesssim a^{2}+b^{2}$ with $a,b\in\R_{\geq0}$. Given $a,b\in\R_{\geq0}$, we also have $(a-b)^{2}\lesssim|a^{2}-b^{2}|$. If $a,b=0$  this is trivial, and otherwise $(a-b)^{2}\lesssim|a-b|(a+b)\cdot|a-b||a+b|^{-1} \lesssim |a^{2}-b^{2}|$.
\item Let $\sigma_{x,y}:\Omega_{\Z}\to\Omega_{\Z}$ denote the involution that swaps spins at the deterministic points $x,y\in\Z$. In particular, we have the identity $\mathfrak{S}_{x,y}\varphi = \varphi\circ\sigma_{x,y} - \varphi$ for any functional $\varphi:\Omega_{\Z}\to\R$. We additionally note that the measure $\mu_{0,\Z}$ is $\sigma_{x,y}$-invariant for the same reason it is $\mathfrak{S}_{x,y}$-invariant. Thus, we have $\E|\mathfrak{S}_{x,y}\varphi|^{2}\lesssim \E|\varphi\circ\sigma_{x,y}|^{2} + \E|\varphi|^{2} \lesssim \E|\varphi|^{2}$.
\item As $\mathfrak{f}$ and $\wt{\mathfrak{f}}$ are probability densities, the previous bullet point gives $\E|\mathfrak{S}_{x,y}\mathfrak{f}^{1/2}|^{2} \lesssim \E\mathfrak{f}=1$ and $\E|\mathfrak{S}_{x,y}\wt{\mathfrak{f}}^{1/2}|^{2} \lesssim \E\wt{\mathfrak{f}}=1$. 
\item The first bullet point and an elementary calculation give $|\mathfrak{S}_{x,y}\varphi-\mathfrak{S}_{x,y}\psi|^{2} = |\varphi\circ\sigma_{x,y}-\psi\circ\sigma_{x,y}|^{2} + |\varphi-\psi|^{2}$.
\end{itemize}
Using these bullet points and the Cauchy-Schwarz inequality gives the following with more explanation given after:
\small\begin{align}
\E|\mathfrak{S}_{x,y}(\mathfrak{f}_{T,N}^{\mathfrak{I}_{N}})^{\frac12}|^{2} - \E|\mathfrak{S}_{x,y}(\wt{\mathfrak{f}}_{T,N}^{\mathfrak{I}_{N}})^{\frac12}|^{2} \ &\lesssim \ \E(|\mathfrak{S}_{x,y}(\mathfrak{f}_{T,N}^{\mathfrak{I}_{N}})^{\frac12}-\mathfrak{S}_{x,y}(\wt{\mathfrak{f}}_{T,N}^{\mathfrak{I}_{N}})^{\frac12}| \cdot|\mathfrak{S}_{x,y}(\mathfrak{f}_{T,N}^{\mathfrak{I}_{N}})^{\frac12}+\mathfrak{S}_{x,y}(\wt{\mathfrak{f}}_{T,N}^{\mathfrak{I}_{N}})^{\frac12}|) \\
&\lesssim \ \left( \E |\mathfrak{S}_{x,y}(\mathfrak{f}_{T,N}^{\mathfrak{I}_{N}})^{\frac12} - \mathfrak{S}_{x,y}(\wt{\mathfrak{f}}_{T,N}^{\mathfrak{I}_{N}})^{\frac12}|^{2}\right)^{\frac12}\left(\E|\mathfrak{S}_{x,y}(\mathfrak{f}_{T,N}^{\mathfrak{I}_{N}})^{\frac12} + \mathfrak{S}_{x,y}(\wt{\mathfrak{f}}_{T,N}^{\mathfrak{I}_{N}})^{\frac12}|^{2}\right)^{\frac12} \\
&\lesssim \ \left( \E |\mathfrak{S}_{x,y}(\mathfrak{f}_{T,N}^{\mathfrak{I}_{N}})^{\frac12} - \mathfrak{S}_{x,y}(\wt{\mathfrak{f}}_{T,N}^{\mathfrak{I}_{N}})^{\frac12}|^{2}\right)^{\frac12}\left(\E|\mathfrak{S}_{x,y}(\mathfrak{f}_{T,N}^{\mathfrak{I}_{N}})^{\frac12}|^{2} + \E|\mathfrak{S}_{x,y}(\wt{\mathfrak{f}}_{T,N}^{\mathfrak{I}_{N}})^{\frac12}|^{2}\right)^{\frac12}\\
&\lesssim \ \left(\E|(\mathfrak{f}_{T,N}^{\mathfrak{I}_{N}})^{\frac12}-(\wt{\mathfrak{f}}_{T,N}^{\mathfrak{I}_{N}})^{\frac12}|^{2} +\E|(\mathfrak{f}_{T,N}^{\mathfrak{I}_{N}}\circ\sigma_{x,y})^{\frac12}-(\wt{\mathfrak{f}}_{T,N}^{\mathfrak{I}_{N}}\circ\sigma_{x,y})^{\frac12}|^{2} \right)^{\frac12}\\
&\lesssim \ \left(\E|\mathfrak{f}_{T,N}^{\mathfrak{I}_{N}}-\wt{\mathfrak{f}}_{T,N}^{\mathfrak{I}_{N}}| + \E|\mathfrak{f}_{T,N}^{\mathfrak{I}_{N}}\circ\sigma_{x,y}-\wt{\mathfrak{f}}_{T,N}^{\mathfrak{I}_{N}}\circ\sigma_{x,y}|\right)^{\frac12} \\
&\lesssim \ \left(\E|\mathfrak{f}_{T,N}^{\mathfrak{I}_{N}}-\wt{\mathfrak{f}}_{T,N}^{\mathfrak{I}_{N}}|\right)^{\frac12}.
\end{align}\normalsize\normalsize
The Dirichlet form sums at most $N^{100}$ such terms, so the last two displays and elementary asymptotics yield \eqref{eq:EProdSoP0}. We justify the previous list of estimates in case of interest. The first estimate above follows from applying the difference of squares in the first bullet point from our list with $a=\mathfrak{S}_{x,y}\mathfrak{f}^{1/2}$ and $b=\mathfrak{S}_{x,y}\wt{\mathfrak{f}}^{1/2}$. The second line follows from the Cauchy-Schwarz inequality. The third line follows from applying the second bound in the first bullet point preceding the above list of estimates for $a=\mathfrak{S}_{x,y}\mathfrak{f}$ and $b=\mathfrak{S}_{x,y}\wt{\mathfrak{f}}$ within the second expectation factor. The fourth line follows by the third bullet point from the list preceding the above estimates applied for the second expectation factor and applying the final bullet point in said list for the first expectation factor. The fifth bound follows by $(a-b)^{2}\lesssim|a^{2}-b^{2}|$ we gave in that list. The last bound follows by $\sigma_{x,y}$-invariance of $\mu_{0,\Z}$.
\end{proof}
\begin{proof}[Proof of \emph{Lemma \ref{lemma:EProd2}}]
The proposed inequality follows immediately from (4.14) and (4.15) within the proof of Lemma 4.1 in \cite{DT}. We emphasize that the proof of Lemma 4.1 in \cite{DT} does not depend on the maximal jump-length $\mathfrak{m}\in\Z_{\geq1}$.
\end{proof}
\begin{proof}[Proof of \emph{Lemma \ref{lemma:EProdSym}}]
By Cauchy-Schwarz $ab \lesssim N^{3/4}a^{2} + N^{-3/4}b^{2}$ for $a = \mathfrak{S}_{x,y}\mathfrak{f}_{\mathfrak{l}}^{1/2}$ and $b=(\mathfrak{f}_{\mathfrak{l},x,y}-\mathfrak{f}_{\mathfrak{l}})\mathfrak{f}_{\mathfrak{l}}^{-1/2}$, we have
\small\begin{align}
N^{2}\E\sum_{(x,y)\in\partial\mathfrak{I}_{N,\mathfrak{l}}} \alpha_{|y-x|} (\mathfrak{f}_{\mathfrak{l},x,y}-\mathfrak{f}_{\mathfrak{l}})\mathfrak{f}_{\mathfrak{l}}^{-\frac12} \mathfrak{S}_{x,y}\mathfrak{f}_{\mathfrak{l}}^{\frac12} \ &\lesssim \ N^{\frac{11}{4}}\E\sum_{(x,y)\in\partial\mathfrak{I}_{N,\mathfrak{l}}} \alpha_{|y-x|} |\mathfrak{S}_{x,y}\mathfrak{f}_{\mathfrak{l}}^{\frac12}|^{2} \ + \ N^{\frac54}\E\sum_{(x,y)\in\partial\mathfrak{I}_{N,\mathfrak{l}}} \alpha_{|y-x|}|\mathfrak{f}_{\mathfrak{l},x,y}-\mathfrak{f}_{\mathfrak{l}}|^{2} \mathfrak{f}_{\mathfrak{l}}^{-1}. \label{eq:EProdSym1}
\end{align}\normalsize\normalsize
We first bound the first term on the RHS of \eqref{eq:EProdSym1} by convexity of the Dirichlet form; this lets us replace the probability density $\mathfrak{f}_{\mathfrak{l}}$ in this term by projections onto larger subsets. Combining this with the observation $\frac{11}{4}=\frac32+\frac54$ gives
\small\begin{align}
N^{\frac{11}{4}}\E\sum_{(x,y)\in\partial\mathfrak{I}_{N,\mathfrak{l}}} \alpha_{|y-x|} |\mathfrak{S}_{x,y}\mathfrak{f}_{\mathfrak{l}}^{\frac12}|^{2} \ = \ N^{\frac{11}{4}}\sum_{(x,y)\in\partial\mathfrak{I}_{N,\mathfrak{l}}} \alpha_{|y-x|} \E|\mathfrak{S}_{x,y}\mathfrak{f}_{\mathfrak{l}}^{\frac12}|^{2} \ &\lesssim \ N^{\frac32}\sum_{(x,y)\in\partial\mathfrak{I}_{N,\mathfrak{l}}}\alpha_{|y-x|}{\sum}_{\mathfrak{j}=1}^{N^{\frac54}} \E |\mathfrak{S}_{x,y}\mathfrak{f}_{\mathfrak{l}+\mathfrak{j}}^{\frac12}|^{2}. \label{eq:EProdSym2}
\end{align}\normalsize\normalsize
The RHS of \eqref{eq:EProdSym2} appears in the upper bound of the result Lemma \ref{lemma:EProdSym} we are currently proving. We move to the second term on the RHS of \eqref{eq:EProdSym1}. For this, we first observe the inequality $\mathfrak{f}_{\mathfrak{l},x,y}\mathfrak{f}_{\mathfrak{l}}^{-1}\lesssim1$. This follows from the fact that $\mathfrak{f}_{\mathfrak{l}}$ is the conditional expectation of $\mathfrak{f}_{\mathfrak{l},x,y}$ over a uniformly bounded set with uniformly positive probabilities. This was used in the proof of Lemma 4.1 in \cite{DT} for example. As $\mathfrak{f}_{\mathfrak{l},x,y}$ and $\mathfrak{f}_{\mathfrak{l}}$ are probability densities, we use $|a-b|^{2}\lesssim a^{2}+b^{2}$ with $a=\mathfrak{f}_{\mathfrak{l},x,y}$ and $b=\mathfrak{f}_{\mathfrak{l}}$ to get
\small\begin{align}
N^{\frac54}\E{\sum}_{(x,y)\in\partial\mathfrak{I}_{N,\mathfrak{l}}} \alpha_{|y-x|}|\mathfrak{f}_{\mathfrak{l},x,y}-\mathfrak{f}_{\mathfrak{l}}|^{2} \mathfrak{f}_{\mathfrak{l}}^{-1} \ &\lesssim \ N^{\frac54}{\sum}_{(x,y)\in\partial\mathfrak{I}_{N,\mathfrak{l}}} \alpha_{|y-x|}\E|\mathfrak{f}_{\mathfrak{l},x,y}|^{2} \mathfrak{f}_{\mathfrak{l}}^{-1} \ + \ N^{\frac54}{\sum}_{(x,y)\in\partial\mathfrak{I}_{N,\mathfrak{l}}} \alpha_{|y-x|} \E\mathfrak{f}_{\mathfrak{l}} \ \lesssim \ N^{\frac54}. \label{eq:EProdSym3}
\end{align}\normalsize\normalsize
The last bound follows by finite support of $\alpha$-coefficients and inequalities $\E|\mathfrak{f}_{\mathfrak{l},x,y}|^{2} \mathfrak{f}_{\mathfrak{l}}^{-1} \lesssim \E\mathfrak{f}_{\mathfrak{l},x,y}=1$ and $\E\mathfrak{f}_{\mathfrak{l}}=1$.
\end{proof}
\begin{proof}[Proof of \emph{Lemma \ref{lemma:EProdWedge}}]
We study $\Phi_{\wedge}\overset{\bullet}=\Phi_{\wedge,1}$. The proof to estimate $\Phi_{\wedge,2}$ follows from identical considerations because the sign and the direction of the asymmetric jump will not matter in what follows. Via the finite support of $\alpha,\gamma$ coefficients, we get
\small\begin{align}
|\Phi_{\wedge}| \ &\lesssim \ N^{\frac32} \sum_{(x,y)\in\partial\mathfrak{I}_{N,\mathfrak{l}}} \alpha_{|x-y|}\gamma_{|x-y|} \E |\mathfrak{f}_{\mathfrak{l},x,y}+\mathfrak{f}_{\mathfrak{l}}|\cdot(\mathfrak{f}_{\mathfrak{l}})^{-\frac12}\cdot|\mathfrak{S}_{x,y}(\mathfrak{f}_{\mathfrak{l}}^{\frac12})| \ \lesssim \ N^{\frac32}\sup_{(x,y)\in\partial\mathfrak{I}_{N,\mathfrak{l}}} \E |\mathfrak{f}_{\mathfrak{l},x,y}+\mathfrak{f}_{\mathfrak{l}}|\cdot(\mathfrak{f}_{\mathfrak{l}})^{-\frac12}\cdot|\mathfrak{S}_{x,y}(\mathfrak{f}_{\mathfrak{l}}^{\frac12})|. \label{eq:EProdWedge1}
\end{align}\normalsize\normalsize
We now apply the Cauchy-Schwarz inequality with respect to the expectation to get
\small\begin{align}
N^{\frac32}\E |\mathfrak{f}_{\mathfrak{l},x,y}+\mathfrak{f}_{\mathfrak{l}}|\cdot(\mathfrak{f}_{\mathfrak{l}})^{-1/2}\cdot|\mathfrak{S}_{x,y}(\mathfrak{f}_{\mathfrak{l}}^{1/2})| \ &\lesssim \ N^{\frac32}\left(\E |\mathfrak{f}_{\mathfrak{l},x,y}+\mathfrak{f}_{\mathfrak{l}}|^{2}\mathfrak{f}_{\mathfrak{l}}^{-1}\right)^{1/2}\left(\E|\mathfrak{S}_{x,y}(\mathfrak{f}_{\mathfrak{l}}^{\frac12})|^{2}\right)^{1/2} \ \lesssim \ N^{\frac32}\left(\E|\mathfrak{S}_{x,y}(\mathfrak{f}_{\mathfrak{l}}^{\frac12})|^{2}\right)^{1/2}. \label{eq:EProdWedge2}
\end{align}\normalsize\normalsize
The last estimate \eqref{eq:EProdWedge2} follows by the three bounds $|\mathfrak{f}_{\mathfrak{l},x,y}+\mathfrak{f}_{\mathfrak{l}}|^{2}\mathfrak{f}_{\mathfrak{l}}^{-1} \lesssim \mathfrak{f}_{\mathfrak{l},x,y}^{2}\mathfrak{f}_{\mathfrak{l}}^{-1} + \mathfrak{f}_{\mathfrak{l}}$ and $\E|\mathfrak{f}_{\mathfrak{l},x,y}|^{2} \mathfrak{f}_{\mathfrak{l}}^{-1} \lesssim \E\mathfrak{f}_{\mathfrak{l},x,y}=1$ and $\E\mathfrak{f}_{\mathfrak{l}}=1$ in the proof of Lemma \ref{lemma:EProdSym}. The last expectation on the far RHS of \eqref{eq:EProdWedge2} is bounded by the global $\mathfrak{D}_{0,\Z}$-Dirichlet form of $\mathfrak{f}_{\mathfrak{l}}$ as each bond in the Dirichlet form contributes non-negative amount and $\mathfrak{S}_{x,y}$ corresponds to just one bond in the $\mathfrak{D}_{0,\Z}$-form. Last, the global $\mathfrak{D}_{0,\Z}$-Dirichlet form of $\mathfrak{f}_{\mathfrak{l}}$ is bounded by that of the unprojected global Radon-Nikodym derivative $\mathfrak{f}_{T,N}$ of the time-$T$ probability measure $\mu_{T,N}$ by convexity of the Dirichlet form. We combine this with \eqref{eq:EProdWedge1} and \eqref{eq:EProdWedge2} to finish the proof.
\end{proof}
%
%
%
\section{Analytic Compactification}\label{section:Ctify}
We replace the microscopic Cole-Hopf transform with a ``compactified" version at the level of stochastic heat-operator-type equations such as \eqref{eq:IntGT}. We start by ``compactifying" the heat operators. Below $\e_{\mathrm{cpt}} \in \R_{>0}$ is arbitrarily small but universal.
\begin{definition}\label{definition:HEATCPT}
Define $\mathbb{T}_{N} \overset{\bullet}= \llbracket-N^{\frac54+\e_{\mathrm{cpt}}},N^{\frac54+\e_{\mathrm{cpt}}}\rrbracket$ and let $\bar{\mathbf{H}}^{N}$ be the heat kernel on $\R_{\geq0}^{2}\times\mathbb{T}_{N}^{2}$ satisfying
\small\begin{align}
\partial_{T}\bar{\mathbf{H}}_{S,T,x,y}^{N} \ &= \ \bar{\mathscr{L}}^{!!}\bar{\mathbf{H}}_{S,T,x,y}^{N} \quad\mathrm{and}\quad \bar{\mathbf{H}}_{S,S,x,y} \ = \ \mathbf{1}_{x=y}.
\end{align}\normalsize
Above, we have introduced $\bar{\mathscr{L}}^{!!} \overset{\bullet}= 2^{-1}\sum_{k=1}^{\mathfrak{m}}\wt{\alpha}_{k}\bar{\Delta}_{k}^{!!}$, in which $\bar{\Delta}_{k}^{!!} \overset{\bullet}= N^{2}\bar{\Delta}_{k}$ and $\bar{\Delta}_{k}\varphi_{x} \overset{\bullet}= \varphi_{x+k}+\varphi_{x-k}-2\varphi_{x}$ given any test function $\varphi:\mathbb{T}_{N} \to \R$. Here, the barred discrete differential operators mean that addition/subtraction are both defined on $\mathbb{T}_{N}$ with periodic boundary conditions $\inf\mathbb{T}_{N} = \sup\mathbb{T}_{N} + 1$, thus realizing $\mathbb{T}_{N}$ as a torus. For $\varphi:\R_{\geq0}\times\mathbb{T}_{N}\to\R$, we also define
\begin{subequations}
\small\begin{align}
\bar{\mathbf{H}}_{T,x}^{N}\varphi \ &\overset{\bullet}= \ \bar{\mathbf{H}}_{T,x}^{N}\varphi_{S,y} \ \overset{\bullet}= \ \bar{\mathbf{H}}_{T,x}^{N}\varphi_{\bullet,\bullet} \ \overset{\bullet}= \ \int_{0}^{T}{\sum}_{y\in\mathbb{T}_{N}}\bar{\mathbf{H}}_{S,T,x,y}^{N} \cdot \varphi_{S,y} \ \d S \\
\bar{\mathbf{H}}_{T,x}^{N,\mathbf{X}}\varphi \ &\overset{\bullet}= \ \bar{\mathbf{H}}_{T,x}^{N,\mathbf{X}}\varphi_{0,y} \ \overset{\bullet}= \ \bar{\mathbf{H}}_{T,x}^{N,\mathbf{X}}\varphi_{0,\bullet} \ \overset{\bullet}= \ {\sum}_{y\in\mathbb{T}_{N}}\bar{\mathbf{H}}_{0,T,x,y}^{N} \cdot \varphi_{0,y}.
\end{align}\normalsize\normalsize
\end{subequations}
\end{definition}
\begin{remark}\label{remark:ch2HKE}
Although $\bar{\mathbf{H}}^{N}$ is not the full-line heat kernel $\mathbf{H}^{N}$ for which heat kernel estimates are established in Proposition A.1 in \cite{DT}, the heat kernel $\bar{\mathbf{H}}^{N}$ admits a classical explicit representation given by spatial translations of $\mathbf{H}^{N}$. Thus, Proposition A.1 in \cite{DT} holds for $\bar{\mathbf{H}}^{N}$ for uniformly bounded times upon replacing the distance on $\Z$ with the geodesic distance on $\mathbb{T}_{N}$. The same is true if we replace ``Proposition A.1" with ``Corollary A.2" in \cite{DT} for the same reason. We organize this in Lemma \ref{lemma:HKE}.
\end{remark}
The aforementioned ``compactification" of the microscopic Cole-Hopf transform defined here is basically given by replacing heat operators in \eqref{eq:IntGT} by their compactifications above. We also perturb the initial data; see Remark \ref{remark:ch2ChiT=0}. Roughly speaking, like \cite{DT} our analysis of stochastic heat-type equations will use regularity of their initial data, and our perturbation of the initial data in Definition \ref{definition:ChiTorus} below ensures such regularity with respect to the geodesic distance on the corresponding geometry $\mathbb{T}_{N}$. However, the perturbations are not detectable in any $\mathbf{C}(\mathbb{K})$ in the large-$N$ limit so the details of these perturbations are not delicate.
\begin{definition}\label{definition:ChiTorus}
Define $\wt{\mathbf{Z}}^{N}\overset{\bullet}=\mathbf{Z}^{N}-\bar{\mathbf{Z}}^{N}$ where $\bar{\mathbf{Z}}^{N}$ is the solution to the following equation with terms to be defined afterwards:
\small\begin{align}
\bar{\mathbf{Z}}_{T,x}^{N} \ &\overset{\bullet}= \ \bar{\mathbf{H}}_{T,x}^{N,\mathbf{X}}(\chi_{\bullet}\mathbf{Z}_{0,\bullet}^{N}) + \bar{\mathbf{H}}_{T,x}^{N}(\bar{\mathbf{Z}}^{N}\d\xi^{N}) + \bar{\mathbf{H}}_{T,x}^{N}(\bar{\Phi}^{N,2}) + \bar{\mathbf{H}}_{T,x}^{N}(\bar{\Phi}^{N,3}). \label{eq:QBarEquation}
\end{align}\normalsize\normalsize
Above $\chi:\Z\to\R_{\geq0}$ is a cutoff function that satisfies the following support and derivative-type constraints:
\begin{itemize}[leftmargin=*]
\item We have $\chi_{y}=1$ for $|y|\leq\frac13\sup\mathbb{T}_{N}$ and $\chi_{y}=0$ for $|y|\geq\frac23\sup\mathbb{T}_{N}$, so the support of $\chi$ is contained \emph{in the interior of} $\mathbb{T}_{N}$.
\item We have a macroscopic-length-scale Lipschitz bound $|\chi_{y}-\chi_{y'}|\leq N^{-1}|y-y'|$ and a uniform bound $|\chi_{y}|\leq1$ for all $y,y'\in\Z$.
\end{itemize}
Defining $N^{-1}\bar{\grad}_{k}^{!}\varphi = \varphi_{x+k} - \varphi_{x}$, where barred-gradients are gradients on $\mathbb{T}_{N}$ with periodic boundary conditions, we have
\small\begin{align}
\bar{\Phi}^{N,2}_{T,x} \ &\overset{\bullet}= \ N^{\frac12}\mathscr{A}_{N^{\beta_{X}}}^{\mathbf{X},-}(\mathfrak{g}_{T,x}) \cdot \bar{\mathbf{Z}}_{T,x}^{N} + N^{\beta_{X}}\left(\wt{\sum}_{\mathfrak{l}=1,\ldots,N^{\beta_{X}}} \wt{\mathfrak{g}}_{T,x}^{\mathfrak{l}}\right)\bar{\mathbf{Z}}_{T,x}^{N} + N^{-\frac12}\wt{\sum}_{\mathfrak{l}=1,\ldots,N^{\beta_{X}}} \bar{\grad}_{-7\mathfrak{l}\mathfrak{m}}^{!}(\mathfrak{b}_{T,x}^{\mathfrak{l}}\bar{\mathbf{Z}}_{T,x}^{N}).
\end{align}\normalsize\normalsize
The term $\bar{\Phi}^{N,3}$ is the compactification of $\Phi^{N,3}$ with weakly vanishing content defined in the statement of Proposition \ref{prop:Duhamel}:
\small\begin{align}
\bar{\Phi}^{N,3}_{T,x} \ &\overset{\bullet}= \ \mathfrak{w}_{T,x}\bar{\mathbf{Z}}_{T,x}^{N} + {\sum}_{k=-2\mathfrak{m}}^{2\mathfrak{m}}c_{k}\bar{\grad}_{k}^{!}(\mathfrak{w}_{T,x}^{k}\bar{\mathbf{Z}}_{T,x}^{N}).
\end{align}\normalsize\normalsize
We recall the functionals inside $\bar{\Phi}^{N,2}$ and $\bar{\Phi}^{N,3}$, namely upon dropping all $\bar{\mathbf{Z}}^{N}$-terms therein, are defined in Proposition \ref{prop:Duhamel}. We emphasize that although $\bar{\mathbf{Z}}^{N}$ evolves on $\mathbb{T}_{N}$, the particle system and its functionals evolve via the global $\mathfrak{S}^{N,!!}$-dynamic on $\Omega_{\Z}$.
\end{definition}
\begin{remark}\label{remark:ch2Jumps}
The term $\bar{\mathbf{Z}}^{N}\d\xi^{N}$ is the martingale associated to the Poisson process whose jumps are those of $\d M$ in (2.4) of \cite{DT} then scaled by $\bar{\mathbf{Z}}^{N}$ at the same space-time point. We will make similar constructions later in the paper.
\end{remark}
\begin{remark}\label{remark:ch2ChiT=0}
The $\chi$-cutoff guarantees the initial data of $\bar{\mathbf{Z}}^{N}$ has spatial regularity with respect to geodesic distance on $\mathbb{T}_{N}$, inherited via a priori near-stationary regularity and $\chi$-regularity, \emph{but without changing initial data in a fashion that is detectable at macroscopic length-scales}. Indeed, this $\chi$-cutoff allows us to forget about any boundary conditions, given its support, while doing nothing at lengths of order $|\mathbb{T}_{N}|$, which is order $N^{5/4+\e_{\mathrm{cpt}}}$, about the origin; see Definition \ref{definition:ChiTorus}.
\end{remark}
The primary goal of this section is the following estimate which compares $\mathbf{Z}^{N}$ and $\bar{\mathbf{Z}}^{N}$. In particular, one consequence of this next result is a comparison of these two processes on compact space-time sets which are the sets of interest in Theorem \ref{theorem:KPZ}.
\begin{prop}\label{prop:Ctify}
 Take any $C_{1}>0$ arbitrarily large but universal. We consider the following data for any $\mathfrak{t}\in\R_{\geq0}$ and $\mathbb{X}\subseteq\Z$.
\begin{itemize}[leftmargin=*]
\item Define an ``interior" $\wt{\mathbb{T}}_{N} \overset{\bullet}= \llbracket-N^{\frac54+\frac12\e_{\mathrm{cpt}}},N^{\frac54+\frac12\e_{\mathrm{cpt}}}\rrbracket\subseteq\mathbb{T}_{N} = \llbracket-N^{\frac54+\e_{\mathrm{cpt}}},N^{\frac54+\e_{\mathrm{cpt}}}\rrbracket$ far from the boundary; $|\wt{\mathbb{T}}_{N}|\ll|\mathbb{T}_{N}|$.
\item Define the time-discretized norm $[\varphi]_{\mathfrak{t};\mathbb{X}} \overset{\bullet}= \sup_{\mathfrak{I}_{\mathfrak{t}}\times\mathbb{X}}|\varphi_{T,x}|$ with time-discretization $\mathfrak{I}_{\mathfrak{t}} \overset{\bullet}= \{\mathfrak{t}\mathfrak{j}N^{-100}\}_{\mathfrak{j}=0}^{N^{100}}$.
\end{itemize}
There exists $C_{2}>0$ depending only on $C_{1}$ satisfying $C_{2}\gtrsim C_{1}$ with universal implied constant if $C_{1}>0$ is sufficiently big such that for any deterministic time-horizon $0\leq\mathfrak{t}_{\mathfrak{f}} \leq 1$, we have the following outside an event of probability at most $\kappa_{C_{1},C_{2}}N^{-C_{2}}$:
\small\begin{align}
[\wt{\mathbf{Z}}^{N}]_{\mathfrak{t}_{\mathfrak{f}};\wt{\mathbb{T}}_{N}} \ &\lesssim \  N^{-C_{1}}. \label{eq:Ctify}
\end{align}\normalsize\normalsize
\end{prop}
\begin{remark}
We eventually upgrade the estimate in Proposition \ref{prop:Ctify} to one on the supremum over the semi-discrete set $[0,1]\times\wt{\mathbb{T}}_{N}$ as opposed to the supremum over a fully discrete set. This is plausible given that for time-scales well below the microscopic time-scale we expect to see very little occur. However, because the stochastic equation for $\bar{\mathbf{Z}}^{N}$ is multiplicative in $\bar{\mathbf{Z}}^{N}$, this will require a priori estimates for $\bar{\mathbf{Z}}^{N}$ which we are not quite ready to establish. We mention this, however, to explain the eventual utility behind Proposition \ref{prop:Ctify}. Such upgrade from the supremum over a totally discrete set to the supremum over a semi-discrete set will all be done in this final section of this paper concerning the proof of Theorem \ref{theorem:KPZ}; see Lemma \ref{lemma:KPZ11} and its proof for details. In particular, the reader is invited to treat this section as a reason to study the compactification of $\bar{\mathbf{Z}}^{N}$ until we return to upgrading the comparison in Proposition \ref{prop:Ctify} to an estimate for the supremum over the aforementioned semi-discrete set.
\end{remark}
\begin{remark}\label{remark:ch2Ctify1}
To illustrate the ideas behind Proposition \ref{prop:Ctify} and where the proof comes from, let us pretend
\small\begin{align}
\mathbf{Z}^{N} \ = \ \mathbf{H}^{N,\mathbf{X}}\mathbf{Z}_{0,\bullet}^{N} \quad\text{and}\quad \bar{\mathbf{Z}}^{N}\ = \ \bar{\mathbf{H}}^{N,\mathbf{X}}\mathbf{Z}_{0,\bullet}^{N}. 
\end{align}\normalsize\normalsize
In this case, the difference $\mathbf{Z}^{N}-\bar{\mathbf{Z}}^{N}$ on $\wt{\mathbb{T}}_{N}$ is controlled by comparing the pair of heat kernels $\mathbf{H}^{N}$ and $\bar{\mathbf{H}}^{N}$ on uniformly bounded time sets with backwards spatial variable in $\wt{\mathbb{T}}_{N}$. These heat kernels correspond to random walks that differ with exponentially small probability when the backwards spatial variable/initial starting point is in the set $\wt{\mathbb{T}}_{N}$. This basically gives Proposition \ref{prop:Ctify} in this ``idealized" case. We will deal with the other terms in the stochastic equations for $\mathbf{Z}^{N}$ and $\bar{\mathbf{Z}}^{N}$ perturbatively.
\end{remark}
\subsection{Stochastic Fundamental Solutions}
In spirit of Remark \ref{remark:ch2Ctify1}, we first consider the following fundamental solutions.
\begin{definition}\label{definition:CtifySFS}
We define $\mathbf{J}^{N}$ as a space-time process on $\R_{\geq0}^{2}\times\Z^{2}$ via $\mathbf{J}_{S,S,x,y}^{N} = \mathbf{1}_{x=y}$ and, for any $T\in\R_{\geq0}$ satisfying $T\geq S$,
\small\begin{align}
\mathbf{J}_{S,T+\mathfrak{t},x,y}^{N} \ &= \ \mathbf{H}_{\mathfrak{t},x}^{N,\mathbf{X}}\mathbf{J}_{S,T,\bullet,y}^{N} + \mathbf{H}_{\mathfrak{t},x}^{N}(\mathbf{J}_{S,T+\bullet,\bullet,y}^{N}\d\xi^{N}) + \mathbf{H}_{\mathfrak{t},x}^{N}(\Phi^{N,2}_{S,T+\bullet,\bullet,y}) + \mathbf{H}_{\mathfrak{t},x}^{N}(\Phi^{N,3}_{S,T+\bullet,\bullet,y}).
\end{align}\normalsize\normalsize
Above, we introduced the following $\Phi^{N,2}$/pseudo-gradient content from Proposition \ref{prop:Duhamel}, but adapted to $\mathbf{J}^{N}$, for any $T\geq S$:
\small\begin{align}
\Phi_{S,T,x,y}^{N,2} \ &\overset{\bullet}= \ N^{\frac12}\mathscr{A}_{N^{\beta_{X}}}^{\mathbf{X},-}(\mathfrak{g}_{T,x}) \cdot \mathbf{J}_{S,T,x,y}^{N} + N^{\beta_{X}}\left(\wt{\sum}_{\mathfrak{l}=1,\ldots,N^{\beta_{X}}} \wt{\mathfrak{g}}_{T,x}^{\mathfrak{l}}\right)\mathbf{J}_{S,T,x,y}^{N} + N^{-\frac12}\wt{\sum}_{\mathfrak{l}=1,\ldots,N^{\beta_{X}}} \grad_{-7\mathfrak{l}\mathfrak{m}}^{!}(\mathfrak{b}_{T,x}^{\mathfrak{l}}\mathbf{J}_{S,T,x,y}^{N}).
\end{align}\normalsize\normalsize
Again the $\Phi^{N,3}$-term contains weakly vanishing content from Proposition \ref{prop:Duhamel} but with a multiplicative $\mathbf{J}^{N}$-factor:
\small\begin{align}
\Phi_{S,T,x,y}^{N,3} \ &\overset{\bullet}= \ \mathfrak{w}_{T,x}\mathbf{J}_{S,T,x,y}^{N} + {\sum}_{k=-2\mathfrak{m}}^{2\mathfrak{m}}c_{k}\grad_{k}^{!}(\mathfrak{w}_{T,x}^{k}\mathbf{J}_{S,T,x,y}^{N}).
\end{align}\normalsize\normalsize
We analogously define $\bar{\mathbf{J}}^{N}$ as a space-time field on $\R_{\geq0}^{2}\times\mathbb{T}_{N}^{2}$ via $\bar{\mathbf{J}}_{S,S,x,y}^{N} = \mathbf{1}_{x=y}$ and, for any $T\in\R_{\geq0}$ satisfying $T\geq S$,
\small\begin{align}
\bar{\mathbf{J}}_{S,T+\mathfrak{t},x,y}^{N} \ &= \ \bar{\mathbf{H}}_{\mathfrak{t},x}^{N,\mathbf{X}}\bar{\mathbf{J}}_{S,T,\bullet,y}^{N} + \bar{\mathbf{H}}_{\mathfrak{t},x}^{N}(\bar{\mathbf{J}}_{S,T+\bullet,\bullet,y}^{N}\d\xi^{N}) + \bar{\mathbf{H}}_{\mathfrak{t},x}^{N}(\bar{\Phi}^{N,2}_{S,T+\bullet,\bullet,y}) + \bar{\mathbf{H}}_{\mathfrak{t},x}^{N}(\bar{\Phi}^{N,3}_{S,T+\bullet,\bullet,y}).
\end{align}\normalsize\normalsize
We introduced the following data for any $T\geq S$ where gradients $\bar{\grad}$ are defined with respect to addition on the torus $\mathbb{T}_{N}$:
\small\begin{align}
\bar{\Phi}_{S,T,x,y}^{N,2} \ &\overset{\bullet}= \ N^{\frac12}\mathscr{A}_{N^{\beta_{X}}}^{\mathbf{X},-}(\mathfrak{g}_{T,x}) \cdot \bar{\mathbf{J}}_{S,T,x,y}^{N} + N^{\beta_{X}}\left(\wt{\sum}_{\mathfrak{l}=1,\ldots,N^{\beta_{X}}} \wt{\mathfrak{g}}_{T,x}^{\mathfrak{l}}\right)\bar{\mathbf{J}}_{S,T,x,y}^{N} + N^{-\frac12}\wt{\sum}_{\mathfrak{l}=1,\ldots,N^{\beta_{X}}} \bar{\grad}_{-7\mathfrak{l}\mathfrak{m}}^{!}(\mathfrak{b}_{T,x}^{\mathfrak{l}}\bar{\mathbf{J}}_{S,T,x,y}^{N}).
\end{align}\normalsize\normalsize
The following $\bar{\Phi}^{N,3}$ is the compactification of the $\Phi^{N,3}$ in the $\mathbf{J}^{N}$-equation introduced above:
\small\begin{align}
\bar{\Phi}_{S,T,x,y}^{N,3} \ &\overset{\bullet}= \ \mathfrak{w}_{T,x}\bar{\mathbf{J}}_{S,T,x,y}^{N} + {\sum}_{k=-2\mathfrak{m}}^{2\mathfrak{m}}c_{k}\bar{\grad}_{k}^{!}(\mathfrak{w}_{T,x}^{k}\bar{\mathbf{J}}_{S,T,x,y}^{N}).
\end{align}\normalsize\normalsize
We conclude by defining the difference in fundamental solutions $\wt{\mathbf{J}}^{N}\overset{\bullet}=\mathbf{J}^{N}-\bar{\mathbf{J}}^{N}$.
\end{definition}
Elementary calculation of the forwards-time-differential of $\mathbf{J}^{N}$ and $\bar{\mathbf{J}}^{N}$ gives the following SDE-type equations:
\begin{subequations}
\small\begin{align}
\d_{\mathfrak{t}}\mathbf{J}_{S,\mathfrak{t},x,y}^{N} \ &= \ \mathscr{L}^{!!}\mathbf{J}_{S,\mathfrak{t},x,y}^{N}\d\mathfrak{t} \ + \ \mathbf{J}_{S,\mathfrak{t},x,y}^{N}\d\xi_{\mathfrak{t},x}^{N} \ + \ \Phi_{S,\mathfrak{t},x,y}^{N,2}\d\mathfrak{t} \ + \ \Phi_{S,\mathfrak{t},x,y}^{N,3}\d\mathfrak{t} \\
\d_{\mathfrak{t}}\bar{\mathbf{J}}_{S,\mathfrak{t},x,y}^{N} \ &= \ \bar{\mathscr{L}}^{!!}\bar{\mathbf{J}}_{S,\mathfrak{t},x,y}^{N}\d\mathfrak{t} \ + \ \bar{\mathbf{J}}_{S,\mathfrak{t},x,y}^{N}\d\xi_{\mathfrak{t},x}^{N} \ + \ \bar{\Phi}_{S,\mathfrak{t},x,y}^{N,2}\d\mathfrak{t} \ + \ \bar{\Phi}_{S,\mathfrak{t},x,y}^{N,3}\d\mathfrak{t}.
\end{align}\normalsize\normalsize
\end{subequations}
%
\begin{itemize}[leftmargin=*]
\item The operators $\mathscr{L}^{!!}$ and $\bar{\mathscr{L}}^{!!}$ are operators from the construction of the pair of heat operators $\mathbf{H}^{N}$ and $\bar{\mathbf{H}}^{N}$, respectively. These operators act on $\mathbf{J}^{N}$ and $\bar{\mathbf{J}}^{N}$ through the backwards spatial variable $x\in\Z$ and $x\in\mathbb{T}_{N}$, respectively.
\end{itemize}
Provided Remark \ref{remark:ch2Ctify1}, we aim to get diffusive tail estimates for fundamental solutions $\mathbf{J}^{N}$ and $\bar{\mathbf{J}}^{N}$. To this end, throughout this section we use the following exponential weights, which are exponentials of non-negative quantities, to quantify diffusive tails.
\begin{definition}\label{definition:tShortweights}
Define $e_{N,x,y} \overset{\bullet}= \exp(N^{-3/4}\mathbf{d}_{x,y})$ and $e_{N,x,y}^{\mathrm{cpt}}\overset{\bullet}=\exp(N^{-3/4}\mathbf{d}_{x,y}^{\mathrm{cpt}})$ where $\mathbf{d}^{\mathrm{cpt}}$ is geodesic distance on the torus $\mathbb{T}_{N}$ and $\mathbf{d}$ is the usual absolute-value-distance on $\Z$. We also define $\|\|_{\omega;p}$ as the $p$-norm with respect to all randomness.
 \end{definition}
\begin{prop}\label{prop:CtifySFS}
 Provided any $p \in \R_{\geq 1}$, uniformly in $S,T \in \R_{\geq 0}$ such that $S \leq T \leq 1$, we have
\small\begin{align}
{\sup}_{x,y\in\Z}e_{N,x,y}\|\mathbf{J}_{S,T,x,y}^{N}\|_{\omega;2p} \ + \ {\sup}_{x,y\in\mathbb{T}_{N}}(e_{N,x,y}^{\mathrm{cpt}})\|\bar{\mathbf{J}}_{S,T,x,y}^{N}\|_{\omega;2p} \ &\lesssim \ \exp\left(\kappa_{p}N^{1/2}\right).\label{eq:CtifySFSI}
\end{align}\normalsize\normalsize
Recall $\wt{\mathbb{T}}_{N} \overset{\bullet}= \llbracket-N^{\frac54+\frac12\e_{\mathrm{cpt}}},N^{\frac54+\frac12\e_{\mathrm{cpt}}}\rrbracket$ in \emph{Proposition \ref{prop:Ctify}}. Uniformly in $x \in \wt{\mathbb{T}}_{N}$ and $y \in \mathbb{T}_{N}$, we have the comparison
\small\begin{align}
\|\wt{\mathbf{J}}_{S,T,x,y}^{N}\|_{\omega;2p} \ \overset{\bullet}= \ \|\mathbf{J}_{S,T,x,y}^{N}-\bar{\mathbf{J}}_{S,T,x,y}^{N}\|_{\omega;2p} \ &\lesssim_{p} \ \exp\left(-N^{1/2}\right).\label{eq:CtifySFSIII}
\end{align}\normalsize\normalsize
\end{prop}
\begin{remark}
We note \eqref{eq:CtifySFSI} and \eqref{eq:CtifySFSIII} agree with Laplace transforms for $\mathbf{H}^{N}$ and $\bar{\mathbf{H}}^{N}$ and Gaussian kernels of time-scale $N^{2}$.
\end{remark}
\begin{proof}[Proof of \emph{Proposition \ref{prop:Ctify}}]
For convenience let us first ignore $\chi$ from the initial data of $\bar{\mathbf{Z}}^{N}$. By uniqueness of solutions to defining linear equations of $\mathbf{Z}^{N}$ and $\bar{\mathbf{Z}}^{N}$, we obtain the following fundamental solution representation for the difference $\wt{\mathbf{Z}}^{N}=\mathbf{Z}^{N} - \bar{\mathbf{Z}}^{N}$:
\small\begin{align}
\wt{\mathbf{Z}}_{T,x}^{N} \ = \ {\sum}_{y\in\Z}\mathbf{J}_{0,T,x,y}^{N}\mathbf{Z}_{0,y}^{N} - {\sum}_{y\in\mathbb{T}_{N}}\bar{\mathbf{J}}_{0,T,x,y}^{N}\mathbf{Z}_{0,y}^{N} \ &= \ {\sum}_{y\in\mathbb{T}_{N}}\wt{\mathbf{J}}_{0,T,x,y}^{N}\mathbf{Z}_{0,y}^{N} \ + \ {\sum}_{y\in\Z\setminus\mathbb{T}_{N}}\mathbf{J}_{0,T,x,y}^{N}\mathbf{Z}_{0,y}^{N}. \label{eq:Ctify21}
\end{align}\normalsize\normalsize
By Proposition \ref{prop:CtifySFS}, the Holder inequality, a priori bounds for near-stationary data, and $|\mathbb{T}_{N}|\lesssim N^{2}$, for $p \in \R_{\geq 1}$ and $C_{1} \in \R_{>0}$ and $x\in\wt{\mathbb{T}}_{N}$ we get the following pair of upper bounds with further explanation that we detail afterwards:
\small\begin{align}
\|\left({\sum}_{y\in\mathbb{T}_{N}} \wt{\mathbf{J}}_{0,T,x,y}^{N}\mathbf{Z}_{0,y}^{N}\right)\|_{\omega;p} \ \leq \ {\sum}_{y\in\mathbb{T}_{N}}\|\wt{\mathbf{J}}_{0,T,x,y}^{N}\|_{\omega;2p}\|\mathbf{Z}_{0,y}^{N}\|_{\omega;2p} \ \lesssim_{p} \ {\sum}_{y\in\mathbb{T}_{N}}\|\wt{\mathbf{J}}_{0,T,x,y}^{N}\|_{\omega;2p} \ &\lesssim_{p,C_{1}} \ N^{-2C_{1}}. \label{eq:Ctify!}
\end{align}\normalsize\normalsize
The first bound in \eqref{eq:Ctify!} follows by the triangle inequality and the Holder inequality $\|\varphi\psi\|_{\omega;p}\leq\|\varphi\|_{\omega;2p}\|\psi\|_{\omega;2p}$. The second inequality follows from the uniform bound on $\|\mathbf{Z}_{0,\bullet}^{N}\|_{\omega;2p}$ from the definition of near-stationary initial data. In the third step, we apply the second estimate in Proposition \ref{prop:CtifySFS} and then we multiply by $|\mathbb{T}_{N}|\lesssim N^{2}$ to account for the sum over $y\in\mathbb{T}_{N}$. Because the exponentially small upper bound from the second estimate in Proposition \ref{prop:CtifySFS} beats out any power of $N\in\Z_{>0}$, we get the far RHS of \eqref{eq:Ctify!} as an upper bound for the far LHS. In particular the aforementioned $|\mathbb{T}_{N}|$-factor can be ignored to get \eqref{eq:Ctify!}.

As $|\mathbb{T}_{N}|\gg|\wt{\mathbb{T}}_{N}|$, for $x\in\wt{\mathbb{T}}_{N}$ and $y\not\in\mathbb{T}_{N}$, note $\mathbf{d}_{x,y} \geq\kappa N^{5/4+\e_{\mathrm{cpt}}}$ with $\kappa>0$ universal, so $e_{N,x,y}^{-1}\leq\exp(-\kappa N^{1/2+\e_{\mathrm{cpt}}})$. Thus, we get the same upper bound for the second term on the far RHS of \eqref{eq:Ctify21} by using \eqref{eq:CtifySFSI} which we explain afterwards:
\small\begin{align}
\|\left({\sum}_{y\in\Z\setminus\mathbb{T}_{N}}\mathbf{J}_{0,T,x,y}^{N}\mathbf{Z}_{0,y}^{N}\right)\|_{\omega;p} \ \lesssim_{p} \ {\sum}_{y\in\Z\setminus\mathbb{T}_{N}}\|\mathbf{J}_{0,T,x,y}^{N}\|_{\omega;2p} \ \lesssim \ {\sum}_{y\in\Z\setminus\mathbb{T}_{N}}\exp\left(\kappa_{p}N^{1/2}\right)e_{N,x,y}^{-1} \ \lesssim_{p,C_{1}} \ N^{-2C_{1}}. \label{eq:Ctify!!}
\end{align}\normalsize\normalsize
The first bound follows again by the triangle inequality for the $\|\|_{\omega;p}$-norm, the Holder inequality $\|\varphi\psi\|_{\omega;p}\leq\|\varphi\|_{\omega;2p}\|\psi\|_{\omega;2p}$, and moment bounds for initial data. The second/third bounds follow by the diffusive tail in the first bound in Proposition \ref{prop:CtifySFS} combined with an elementary summation estimate resembling tail probabilities of Brownian motion. This bound also requires the observation that the negative exponent in such tail estimate dwarfs the positive exponent $N^{1/2}$ on the RHS of the first bound in Proposition \ref{prop:CtifySFS} when $x\in\wt{\mathbb{T}}_{N}$ and we sum outside $\mathbb{T}_{N}$, again since for $x\in\wt{\mathbb{T}}_{N}$ and $y\not\in\mathbb{T}_{N}$, we have $\mathbf{d}_{x,y} \gtrsim N^{5/4+\e_{\mathrm{cpt}}}$. Using the last display with \eqref{eq:Ctify21} and \eqref{eq:Ctify!} gives the following where we control a sup by an \emph{unaveraged} sum of absolute values:
\small\begin{align}
\|\|\wt{\mathbf{Z}}_{T,x}^{N}\|_{\mathfrak{t}_{\mathfrak{f}};\wt{\mathbb{T}}_{N}}\|_{\omega;p}\ \leq \ \|\left({\sum}_{\mathfrak{I}_{\mathfrak{t}_{\mathfrak{f}}}\times\wt{\mathbb{T}}_{N}}|\wt{\mathbf{Z}}_{T,x}^{N}|\right)\|_{\omega;p} \ \leq \ {\sum}_{\mathfrak{I}_{\mathfrak{t}_{\mathfrak{f}}}\times\wt{\mathbb{T}}_{N}}\|\wt{\mathbf{Z}}_{T,x}^{N}\|_{\omega;p} \ &\lesssim_{p,C_{1}} \ |\mathfrak{I}_{\mathfrak{t}_{\mathfrak{f}}}||\wt{\mathbb{T}}_{N}|N^{-2C_{1}}. \label{eq:Ctify22}
\end{align}\normalsize\normalsize
Combining the estimate \eqref{eq:Ctify22} with the Chebyshev inequality gives the desired estimate \eqref{eq:Ctify} if we choose $C_{2} = C_{1}-103$ and then choose $C_{1}\in\R_{>0}$ sufficiently/arbitrarily large but universal for example. To make this precise, the Markov inequality and \eqref{eq:Ctify22} with $p=1$ gives the following estimate in which all implied constants depend only on $C_{1}$ and $C_{2}=C_{1}-103$:
\small\begin{align}
\mathbf{P}\left(\|\wt{\mathbf{Z}}_{T,x}^{N}\|_{\mathfrak{t}_{\mathfrak{f}};\wt{\mathbb{T}}_{N}}\ \gtrsim \ N^{-C_{1}}\right) \ \lesssim \ N^{C_{1}}\|\|\wt{\mathbf{Z}}_{T,x}^{N}\|_{\mathfrak{t}_{\mathfrak{f}};\wt{\mathbb{T}}_{N}}\|_{\omega;1} \ \lesssim \ N^{C_{1}}|\mathfrak{I}_{\mathfrak{t}_{\mathfrak{f}}}||\wt{\mathbb{T}}_{N}|N^{-2C_{1}} \ \lesssim \ N^{-C_{1}}N^{100+2} \ \lesssim \ N^{-C_{2}}.
\end{align}\normalsize\normalsize
We choose $C_{1}\in\R_{>0}$ sufficiently/arbitrarily large but universal to complete the proof in the case where we ignore $\chi$/set $\chi\equiv1$. The error we get when we ignore $\chi$ comes from the initial data $\bar{\mathbf{Z}}^{N}_{0,\bullet}$. It is ultimately the following error in replacing $\chi$ by 1:
\small\begin{align}
|\mathrm{Error}| \ &\leq \ {\sum}_{y\in\mathbb{T}_{N}}|\bar{\mathbf{J}}_{0,T,x,y}^{N}| |\chi_{y}-1|\mathbf{Z}_{0,y}^{N}.
\end{align}\normalsize\normalsize
By $|\mathbb{T}_{N}|\gg|\wt{\mathbb{T}}_{N}|$ and definition of $\chi$ in Definition \ref{definition:ChiTorus}, the support of $\chi_{y}-1$ is $\gtrsim N^{5/4+\e_{\mathrm{cpt}}}$ away from $x\in\wt{\mathbb{T}}_{N}$; we had a similar observation prior to \eqref{eq:Ctify!!}. We also have $|\chi_{y}-1|\leq2$, so we can use tail bounds for $\bar{\mathbf{J}}^{N}$ in the first bound in Proposition \ref{prop:CtifySFS} and a priori bounds for near-stationary initial data, like in our proof of \eqref{eq:Ctify!!}, to account for ignoring $\chi$.
\end{proof} 
\subsection{Proof of Proposition \ref{prop:CtifySFS}}
The first step is a short-time moment estimate with a proof almost identical to that of Proposition 3.2 in \cite{DT}. The short-time nature allows us to control both $\mathbf{J}^{N}$ and $\bar{\mathbf{J}}^{N}$ on exponential scales $e_{N,x,y}$ and $e_{N,x,y}^{\mathrm{cpt}}$, respectively.
\begin{lemma}\label{lemma:CtifySFS1}
 Take $S,T \in \R_{\geq 0}$ with $S \leq T \leq 1$ and define $\mathfrak{t}_{\mathrm{short}} \overset{\bullet}= N^{-1/2}$. For $p \in \R_{\geq 1}$, we have, recalling \emph{Definition \ref{definition:tShortweights}},
\begin{subequations}
\small\begin{align}
\sup_{0 \leq \mathfrak{t} \leq \mathfrak{t}_{\mathrm{short}}} \sup_{x,y \in \Z} e_{N,x,y}\|\mathbf{J}_{S,T+\mathfrak{t},x,y}^{N}\|_{\omega;2p} \ &\lesssim_{p} \ \sup_{x,y\in\Z}e_{N,x,y}\|\mathbf{J}_{S,T,x,y}^{N}\|_{\omega;2p} \label{eq:CtifySFS1I} \\
\sup_{0 \leq \mathfrak{t} \leq \mathfrak{t}_{\mathrm{short}}} \sup_{x,y \in \mathbb{T}_{N}} e_{N,x,y}^{\mathrm{cpt}}\|\bar{\mathbf{J}}_{S,T+\mathfrak{t},x,y}^{N}\|_{\omega;2p} \ &\lesssim_{p} \ \sup_{x,y\in\mathbb{T}_{N}}e_{N,x,y}^{\mathrm{cpt}}\|\bar{\mathbf{J}}_{S,T,x,y}^{N}\|_{\omega;2p}. \label{eq:CtifySFS1II}
\end{align}\normalsize\normalsize
\end{subequations}
\end{lemma}
We defer proof of Lemma \ref{lemma:CtifySFS1} to the end of the section. The first bound in Proposition \ref{prop:CtifySFS} will come from iterating Lemma \ref{lemma:CtifySFS1}. We now turn to the second ingredient for the proof of Proposition \ref{prop:CtifySFS}, a Duhamel-type formula.
\begin{lemma}\label{lemma:CtifySFS2}
 Provided any $S,T \in \R_{\geq 0}$ satisfying $S \leq T$ and provided any $x,y \in \mathbb{T}_{N}$, we have, with notation explained after,
\small\begin{align}
\bar{\mathbf{J}}_{S,T,x,y}^{N} \ &= \ \mathbf{J}_{S,T,x,y}^{N} \ + \ \int_{S}^{T}{\sum}_{w\in\mathbb{T}_{N}}\bar{\mathbf{J}}_{R,T,x,w}^{N} \cdot N^{2} {\sum}_{|\mathfrak{l}| \leq 10N^{\beta_{X}}\mathfrak{m}} \left(\mathfrak{f}_{R,w}^{\mathfrak{l},1}\mathbf{J}_{S,R,w+\mathfrak{l},y}^{N} + \mathfrak{f}_{R,w}^{\mathfrak{l},2}\mathbf{J}_{S,R,w\wt{+}\mathfrak{l},y}^{N}\right) \ \d R.
\end{align}\normalsize\normalsize
The functionals $\mathfrak{f}^{\mathfrak{l},i}$ have supports outside the ball $\mathfrak{B}$ of radius $N^{\frac54+\frac23\e_{\mathrm{cpt}}}$ around $0 \in \Z$ and they are uniformly bounded. Moreover $+$ denotes addition on $\Z$ and $\wt{+}$ denotes addition on the torus $\mathbb{T}_{N}$ in the statement and proof of this result.
\end{lemma}
\begin{proof}[Proof of \emph{Proposition \ref{prop:CtifySFS}}]
As we noted after the statement of Lemma \ref{lemma:CtifySFS1}, the estimate \eqref{eq:CtifySFSI} follows by its validity at $T = S$, which follows almost entirely by definition, then iterating Lemma \ref{lemma:CtifySFS1} and gathering factors a total of order $N^{1/2}$-many times. To prove \eqref{eq:CtifySFSIII}, recall $\beta_{X}\leq1$ and $S,T\leq1$. By Lemma \ref{lemma:CtifySFS2} and the Holder inequality $\|\varphi\psi\|_{\omega;p}\leq\|\varphi\|_{\omega;2p}\|\psi\|_{\omega;2p}$, we get
\small\begin{align}
\|\wt{\mathbf{J}}_{S,T,x,y}^{N}\|_{\omega;p} \ &\lesssim_{\mathfrak{m}} \ N^{3}{\sup}_{S\leq R\leq T}\left({\sum}_{w\in\mathbb{T}_{N}}\mathbf{1}_{w\not\in\mathfrak{B}} \|\bar{\mathbf{J}}_{R,T,x,w}^{N}\|_{\omega;2p} \cdot {\sup}_{w\in\Z}\|\mathbf{J}^{N}_{S,R,w,y}\|_{\omega;2p}\right). \label{eq:0}
\end{align}\normalsize\normalsize
We emphasize that the cutoff indicator function inside the supremum on the RHS of the previous display follows by the support constraints on the $\mathfrak{f}^{\mathfrak{l},1}$ and $\mathfrak{f}^{\mathfrak{l},2}$ functionals from the statement of Lemma \ref{lemma:CtifySFS2}. We assume $S,T \in \R_{\geq 0}$ satisfy $S \leq T \leq 1$ and take any $x \in \wt{\mathbb{T}}_{N}$ and $y \in \mathbb{T}_{N}$. For this case, we observe if $w\not\in\mathfrak{B}$ then it satisfies $|w| \gtrsim N^{5/4+2\e_{\mathrm{cpt}}/3}$, thus $\mathbf{d}_{x,w}^{\mathrm{cpt}}\gtrsim N^{5/4+\e_{\mathrm{cpt}}/2}$, so for these points $w\not\in\mathfrak{B}$ the diffusive tail estimate \eqref{eq:CtifySFSI} yields the following estimate which says the diffusive tails beat the exponentially growing factor on the RHS of \eqref{eq:CtifySFSI}, rendering such exponentially factor asymptotically irrelevant. Here $\kappa\gtrsim1$:
\small\begin{align}
\mathbf{1}_{w\not\in\mathfrak{B}}\|\bar{\mathbf{J}}_{R,T,x,w}^{N}\|_{\omega;2p} \ \lesssim \ \exp\left(\kappa_{p}N^{1/2}\right)\mathbf{1}_{w\not\in\mathfrak{B}}e_{N,x,y}^{-1} \ \lesssim \ \exp\left(\kappa_{p}N^{1/2}\right)\exp\left(-\kappa N^{1/2+\e_{\mathrm{cpt}}/2}\right). \label{eq:1}
\end{align}\normalsize\normalsize
We now use \eqref{eq:CtifySFSI} and uniform boundedness of $e_{N,\bullet,\bullet}^{-1}$-factors to obtain the following estimate with universal implied constant. Roughly speaking it forgets the diffusive off-diagonal/tail behavior of $\mathbf{J}^{N}$ and uses only the exponential control for its growth. As we note again shortly such exponential control is dwarfed by the negative-exponent factor on the far RHS of \eqref{eq:1}:
\small\begin{align}
\|\mathbf{J}^{N}_{S,R,w,y}\|_{\omega;2p} \ \lesssim \ \exp\left(\kappa_{p}N^{1/2}\right). \label{eq:2}
\end{align}\normalsize\normalsize
We plug \eqref{eq:1} and \eqref{eq:2} into \eqref{eq:0} to get the following for which we recall $|\mathbb{T}_{N}|\lesssim N^{2}$ and then bound the supremum on the RHS \eqref{eq:0} by $|\mathbb{T}_{N}|$ times the space-time supremum of the product of $\|\|_{\omega;2p}$-norms. We ultimately obtain the estimate below in which the final conclusion is that the negative-exponent factor on the far RHS of \eqref{eq:1} dwarfs all other factors:
\small\begin{align}
\|\wt{\mathbf{J}}_{S,T,x,y}^{N}\|_{\omega;p} \ \lesssim_{\mathfrak{m}} \ N^{3}|\mathbb{T}_{N}|\exp\left(2\kappa_{p}N^{1/2}\right)\exp\left(-\kappa N^{1/2+\e_{\mathrm{cpt}}/2}\right) \ \lesssim_{p} \ \exp\left(-N^{1/2+\e_{\mathrm{cpt}}/3}\right) \ \leq \ \exp\left(-N^{1/2}\right).
\end{align}\normalsize\normalsize
The second estimate on the far RHS of the previous display follows by elementary asymptotics in $N\in\Z_{>0}$ as $\kappa_{p}\lesssim_{p}1$.
\end{proof}
We start the proof of the auxiliary results in Lemma \ref{lemma:CtifySFS1} and Lemma \ref{lemma:CtifySFS2} with notation that captures diffusive tail behavior of the heat kernels $\mathbf{H}^{N}$ and $\bar{\mathbf{H}}^{N}$ and that will be used in the proof of Lemma \ref{lemma:CtifySFS1}. Let us first recall Definition \ref{definition:tShortweights}.
\begin{definition}\label{definition:Exps}
Given any $\kappa \in \R_{>0}$, any $S,T \in \R_{\geq 0}$ so that $S \leq T$, and any $x,y \in \Z$ or $\mathbb{T}_{N}$, define the following exponential weights in which $\mathbf{d}_{x,y}=|x-y|$ is usual distance on $\Z$, in which $\mathfrak{s}_{S,T}=|T-S|$, and in which $\mathbf{d}^{\mathrm{cpt}}$ is geodesic distance on $\mathbb{T}_{N}$:
\small\begin{align}
e_{S,T,x,y}^{N,\kappa} \ \overset{\bullet}= \ \exp\left(\kappa\frac{\mathbf{d}_{x,y}}{N\mathfrak{s}_{S,T}^{1/2}\vee1}\right), \quad e_{S,T,x,y}^{N,\kappa,\mathrm{cpt}} \ \overset{\bullet}= \ \exp\left(\kappa\frac{\mathbf{d}_{x,y}^{\mathrm{cpt}}}{N\mathfrak{s}_{S,T}^{1/2}\vee1}\right).
\end{align}\normalsize\normalsize
We will require later for the proof of Lemma \ref{lemma:CtifySFS1} the following elementary summation estimate, which follows by a geometric series bound, and the following elementary control on $e_{N,x,y}$ and $e_{N,x,y}^{\mathrm{cpt}}$ for short times. We take $\kappa\in\R_{>0}$ strictly positive:
\begin{subequations}
\small\begin{align}
(N\mathfrak{s}_{S,T}^{1/2} \vee 1)^{-1}\left({\sum}_{w\in\Z}e_{S,T,x,w}^{N,-\kappa} \ + \ {\sum}_{w\in\mathbb{T}_{N}}e_{S,T,x,w}^{N,-\kappa,\mathrm{cpt}}\right) \ &\lesssim_{\kappa} \ 1 \label{eq:Exps1} \\
\mathbf{1}_{\mathfrak{t}\leq N^{-1/2}}\left(e_{0,\mathfrak{t},x,y}^{N,-2}e_{N,x,y}^{2} \ + \ e_{0,\mathfrak{t},x,y}^{N,-2,\mathrm{cpt}}(e_{N,x,y}^{\mathrm{cpt}})^{2}\right)\ &\lesssim \ 1. \label{eq:Exps2}
\end{align}\normalsize\normalsize
\end{subequations}
\end{definition}
\begin{proof}[Proof of \emph{Lemma \ref{lemma:CtifySFS1}}]
We use notation of Definitions \ref{definition:tShortweights}, \ref{definition:Exps}; recall $\mathfrak{t}_{\mathrm{short}}=N^{-1/2}$. We note the following estimates for both $\mathbf{H}^{N}$ and $\bar{\mathbf{H}}^{N}$ which follow by Lemma \ref{lemma:HKE} as we explain soon. Just for this proof, we define $\mathfrak{t}_{N}\overset{\bullet}= N^{2}\mathfrak{t} \vee 1$ as a function of $\mathfrak{t}\in\R_{\geq0}$. The following estimates effectively give a short-time scaling of heat kernels for random walks with speed $N^{2}$ with exponential off-diagonal weight that is of diffusive type. In the estimates $8$ and $2$ are just arbitrary positive numbers:
\begin{subequations}
\small\begin{align}
\sup_{0\leq\mathfrak{t}\leq\mathfrak{t}_{\mathrm{short}}}\sup_{x,y\in\Z} \mathfrak{t}_{N}^{\frac12}e_{N,x,y}^{2} e_{0,\mathfrak{t},x,y}^{N,8}\mathbf{H}_{0,\mathfrak{t},x,y}^{N} \ + \ \sup_{0\leq\mathfrak{t}\leq\mathfrak{t}_{\mathrm{short}}}\sup_{x,y\in\mathbb{T}_{N}} \mathfrak{t}_{N}^{\frac12} (e_{N,x,y}^{\mathrm{cpt}})^{2} e_{0,\mathfrak{t},x,y}^{N,8,\mathrm{cpt}}\bar{\mathbf{H}}_{0,\mathfrak{t},x,y}^{N} \ &\lesssim \ 1 \label{eq:CtifySFS11} \\
\sup_{0\leq\mathfrak{t}\leq\mathfrak{t}_{\mathrm{short}}}\sup_{x,y\in\Z} \mathfrak{t}_{N}e_{N,x,y}^{2} e_{0,\mathfrak{t},x,y}^{N,8}\sup_{|k|\lesssim\mathfrak{m}}|\grad_{k}^{!}\mathbf{H}_{0,\mathfrak{t},x,y}^{N}| \ + \ \sup_{0\leq\mathfrak{t}\leq\mathfrak{t}_{\mathrm{short}}}\sup_{x,y\in\mathbb{T}_{N}} \mathfrak{t}_{N} (e_{N,x,y}^{\mathrm{cpt}})^{2} e_{0,\mathfrak{t},x,y}^{N,8,\mathrm{cpt}}\sup_{|k|\lesssim\mathfrak{m}}|\bar{\grad}_{k}^{!}\bar{\mathbf{H}}_{0,\mathfrak{t},x,y}^{N}| \ &\lesssim_{\mathfrak{m}} \ N. \label{eq:CtifySFS11a}
\end{align}\normalsize\normalsize
\end{subequations}
Indeed, for $0\leq\mathfrak{t}\leq\mathfrak{t}_{\mathrm{short}}$, by the off-diagonal bound \eqref{eq:HKENash} for $\kappa=10$ and the bound $e_{0,\mathfrak{t},x,y}^{N,-2} \lesssim e_{N,x,y}^{-2}$ from Definition \ref{definition:Exps},
\small\begin{align}
\mathfrak{t}_{N}^{1/2}\mathbf{H}^{N}_{0,\mathfrak{t},x,y} \ \lesssim \ e_{0,\mathfrak{t},x,y}^{N,-10} \ \lesssim \ e_{N,x,y}^{-2}e_{0,\mathfrak{t},x,y}^{N,-8}.
\end{align}\normalsize\normalsize
This provides \eqref{eq:CtifySFS11}; for \eqref{eq:CtifySFS11a} we instead use \eqref{eq:HKEXR} in Lemma \ref{lemma:HKE} and $\mathfrak{m}\lesssim1$.

We will now prove \eqref{eq:CtifySFS1I}. The proof of \eqref{eq:CtifySFS1II} follows by identical methods upon elementary replacements like replacing $\Z$ by $\mathbb{T}_{N}$. We will basically copy the proof of Proposition 3.2 in \cite{DT}. By definition of $\mathbf{J}^{N}$ in Definition \ref{definition:CtifySFS}, we get, for $\|\|=\|\|_{\omega;2p}$,
\small\begin{align}
\|\mathbf{J}_{S,T+\mathfrak{t},x,y}^{N}\|^{2} \ &\lesssim \ \|\mathbf{H}_{\mathfrak{t},x}^{N,\mathbf{X}}\mathbf{J}_{S,T,\bullet,y}^{N}\|^{2} + \|\mathbf{H}_{\mathfrak{t},x}^{N}(\mathbf{J}_{S,T+\bullet,\bullet,y}^{N}\d\xi^{N})\|^{2} + \|\mathbf{H}_{\mathfrak{t},x}^{N}(\Phi^{N,2}_{S,T+\bullet,\bullet,y})\|^{2} + \|\mathbf{H}_{\mathfrak{t},x}^{N}(\Phi^{N,3}_{S,T+\bullet,\bullet,y})\|^{2}. \label{eq:CtifySFS12}
\end{align}\normalsize\normalsize
Following the proof of Proposition 3.2 in \cite{DT}, \eqref{eq:CtifySFS11} and the elementary inequality $e_{N,x,y}\leq e_{N,x,w}e_{N,w,y}$, which follows by the triangle inequality, and the $\Z$-sum bound \eqref{eq:Exps1} for $S=0$ and $\kappa=8$ from Definition \ref{definition:Exps} and convexity of $\mathbf{H}^{N,\mathbf{X}}$ give the following spatial-heat-operator bound; recall $\mathfrak{t}_{N}=N^{2}\mathfrak{t}\vee1$ from earlier in this proof and $e$-weights in Definition \ref{definition:tShortweights}:
\small\begin{align}
e_{N,x,y}^{2}\|\mathbf{H}_{\mathfrak{t},x}^{N,\mathbf{X}}\mathbf{J}_{S,T,\bullet,y}^{N}\|_{\omega;2p}^{2} \ \leq \ e_{N,x,y}^{2}{\sum}_{w\in\Z}\mathbf{H}_{0,\mathfrak{t},x,w}^{N} \cdot \|\mathbf{J}_{S,T,w,y}^{N}\|_{\omega;2p}^{2} \ &\lesssim \ {\sum}_{w\in\Z}e_{N,x,w}^{2}\mathbf{H}_{0,\mathfrak{t},x,w}^{N} \cdot e_{N,w,y}^{2}\|\mathbf{J}^{N}_{S,T,w,y}\|_{\omega;2p}^{2} \\
&\lesssim \ \mathfrak{t}_{N}^{-1/2}{\sum}_{w\in\Z}e_{0,\mathfrak{t},x,w}^{N,-8} \cdot {\sup}_{w\in\Z}e_{N,w,y}^{2}\|\mathbf{J}_{S,T,w,y}^{N}\|_{\omega;2p}^{2} \\
&\lesssim \ {\sup}_{w\in\Z}e_{N,w,y}^{2}\|\mathbf{J}_{S,T,w,y}^{N}\|_{\omega;2p}^{2}. \label{eq:CtifySFS13}
\end{align}\normalsize\normalsize
We now want to use the martingale bound Lemma 3.1 in \cite{DT} as with the proof of Proposition 3.2 in \cite{DT}. However, since the martingale inequality of Lemma 3.1 from \cite{DT} is specific to the microscopic Cole-Hopf transform $\mathbf{Z}^{N}$, and in particular does not generalize to arbitrary products between any adapted process and the martingale increment $\d\xi^{N}$ in (2.4) of \cite{DT}, we instead use Lemma \ref{lemma:MG} to get the first line below; the lines after follow by $e_{N,x,y}^{2}e_{N,w,y}^{-2} \leq e_{N,x,w}^{2}$ and \eqref{eq:CtifySFS11} and $\mathfrak{t}_{N}^{-1/2}\sum_{w\in\Z}e_{0,\mathfrak{t},x,w}^{N,-8}\lesssim1$:
\small\begin{align}
e_{N,x,y}^{2}\|\mathbf{H}_{\mathfrak{t},x}^{N}(\mathbf{J}_{S,T+\bullet,\bullet,y}^{N}\d\xi^{N})\|_{\omega;2p}^{2} \ &\lesssim_{p} \ e_{N,x,y}^{2}\int_{0}^{\mathfrak{t}} \left(\sup_{0\leq\mathfrak{r}\leq\mathfrak{t}}\sup_{w\in\Z}e_{N,w,y}^{2}\| \mathbf{J}_{S,T+\mathfrak{r},w,y}^{N}\|_{\omega;2p}^{2}\right)\mathfrak{s}_{R,\mathfrak{t}}^{-\frac12}{\sum}_{w\in\Z}\mathbf{H}_{R,\mathfrak{t},x,w}^{N} e_{N,w,y}^{-2} \ \d R \\
&\lesssim_{p} \ \int_{0}^{\mathfrak{t}}\mathfrak{s}_{R,\mathfrak{t}}^{-\frac12}\left(\sup_{0\leq\mathfrak{r}\leq\mathfrak{t}}\sup_{w\in\Z}e_{N,w,y}^{2}\|\mathbf{J}_{S,T+\mathfrak{r},w,y}^{N}\|_{\omega;2p}^{2}\right){\sum}_{w\in\Z}e_{N,x,w}^{2}\mathbf{H}_{R,\mathfrak{t},x,w}^{N} \ \d R \\
&\lesssim_{p} \ \int_{0}^{\mathfrak{t}}\mathfrak{s}_{R,\mathfrak{t}}^{-\frac12}\left(\sup_{0\leq\mathfrak{r}\leq\mathfrak{t}}\sup_{w\in\Z}e_{N,w,y}^{2}\|\mathbf{J}_{S,T+\mathfrak{r},w,y}^{N}\|_{\omega;2p}^{2}\right)(\mathfrak{s}_{R,\mathfrak{t}})_{N}^{-\frac12}{\sum}_{w\in\Z}e_{R,\mathfrak{t},x,w}^{N,-8} \ \d R \\
&\lesssim \ \int_{0}^{\mathfrak{t}}\mathfrak{s}_{R,\mathfrak{t}}^{-\frac12}\left(\sup_{0\leq\mathfrak{r}\leq\mathfrak{t}}\sup_{w\in\Z}e_{N,w,y}^{2}\|\mathbf{J}_{S,T+\mathfrak{r},w,y}^{N}\|_{\omega;2p}^{2}\right) \ \d R. \label{eq:CtifySFS14}
\end{align}\normalsize\normalsize
We have used the notation $\mathfrak{s}_{S,R}^{-1/2} \overset{\bullet}= |S-R|^{-1/2}$ within Definition \ref{definition:Exps}. We also clarify our earlier notation $(\mathfrak{s}_{R,\mathfrak{t}})_{N} = N^{2}\mathfrak{s}_{R,\mathfrak{t}}\vee1$. For the third term on the RHS of \eqref{eq:CtifySFS12}, we first observe that the heat kernel $\mathbf{H}^{N}$ is a probability measure on $\Z$ with respect to the forwards spatial-variable. We then apply the Cauchy-Schwarz inequality with respect to the space-time convolution ``integral" via heat kernel. As $\mathfrak{t}\lesssim N^{-1/2}$, calculations similar to the derivations of \eqref{eq:CtifySFS13} and \eqref{eq:CtifySFS14} give, for $\Psi$ defined shortly, the bounds
\small\begin{align}
e_{N,x,y}^{2} \|\mathbf{H}_{\mathfrak{t},x}^{N}(\Phi^{N,2}_{S,T+\bullet,\bullet,y})\|_{\omega;2p}^{2} \ &\lesssim \ e_{N,x,y}^{2} N\mathfrak{t} \int_{0}^{\mathfrak{t}}\sum_{w\in\Z}\mathbf{H}_{R,\mathfrak{t},x,w}^{N} \cdot \|\mathbf{J}^{N}_{S,T+R,w,y}\|_{\omega;2p}^{2} \ \d R + \Psi \\
&\lesssim \ N^{\frac12}\int_{0}^{\mathfrak{t}}\sum_{w\in\Z}e_{N,x,w}^{2}\mathbf{H}_{R,\mathfrak{t},x,w}^{N} \cdot e_{N,w,y}^{2}\|\mathbf{J}_{S,T+R,w,y}^{N}\|_{\omega;2p}^{2} \ \d R + \Psi \\
&\lesssim \ N^{\frac12}\int_{0}^{\mathfrak{t}}(\mathfrak{s}_{R,\mathfrak{t}})_{N}^{-\frac12} \sum_{w\in\Z}e_{R,\mathfrak{t},x,w}^{N,-8} \cdot e_{N,w,y}^{2}\|\mathbf{J}_{S,T+R,w,y}^{N}\|_{\omega;2p}^{2} \ \d R + \Psi \\
&\lesssim \ N^{\frac12}\int_{0}^{\mathfrak{t}} \sup_{w\in\Z}e_{N,w,y}^{2}\|\mathbf{J}_{S,T+R,w,y}^{N}\|_{\omega;2p}^{2} \ \d R + \Psi. \label{eq:CtifySFS15}
\end{align}\normalsize\normalsize
The last quantity $\Psi$ is slightly more subtle and in particular unaddressed in Proposition 3.2 in \cite{DT}. It is defined as
\small\begin{align}
\Psi \ &\overset{\bullet}= \ e_{N,x,y}^{2}\mathfrak{t}\int_{0}^{\mathfrak{t}}\sum_{w\in\Z}\mathbf{H}_{R,\mathfrak{t},x,w}^{N} \cdot \wt{\sum}_{\mathfrak{l}=1}^{N^{\beta_{X}}} N \|\grad_{-7\mathfrak{l}\mathfrak{m}}(\mathfrak{b}^{\mathfrak{l}}_{T+R,w}\mathbf{J}_{S,T+R,w,y}^{N})\|_{\omega;2p}^{2} \ \d S.
\end{align}\normalsize\normalsize
We will not use summation-by-parts for $\Psi$. We instead expand the gradient while recalling the $\mathfrak{b}^{\mathfrak{l}}$-functionals are uniformly bounded. This bounds the $\|\|_{\omega;2p}$-term in the $\Psi$-definition by the same norm of $\mathbf{J}^{N}$ evaluated at points $7\mathfrak{l}\mathfrak{m}$-apart. Noting $e_{N,w,y}^{2} \lesssim e_{N,w,w+\mathfrak{k}}^{2}e_{N,w+\mathfrak{k},y}^{2}\lesssim e_{N,w+\mathfrak{k},y}^{2}$ for all $\mathfrak{k}\in\Z$ with $|\mathfrak{k}| \lesssim N^{\beta_{X}}$, where $\beta_{X} = \frac13+\e_{X,1} \leq \frac12$, and recalling $\mathfrak{t}\lesssim N^{-1/2}$, this paragraph gives the following when combined with the same considerations as in derivations of \eqref{eq:CtifySFS13}, \eqref{eq:CtifySFS14}, and \eqref{eq:CtifySFS15}:
\small\begin{align}
\Psi \ &\lesssim \ N\mathfrak{t}\int_{0}^{\mathfrak{t}}\sum_{w\in\Z}e_{N,x,w}^{2}\mathbf{H}_{R,\mathfrak{t},x,w}^{N} \cdot \wt{\sum}_{\mathfrak{l}=1}^{N^{\beta_{X}}}e_{N,w,y}^{2}\||\mathbf{J}_{S,T+R,w,y}^{N}| + |\mathbf{J}_{S,T+R,w-7\mathfrak{l}\mathfrak{m},y}^{N}|\|_{\omega;2p}^{2} \ \d R \\
&\lesssim \ N^{\frac12}\int_{0}^{\mathfrak{t}}(\mathfrak{s}_{R,\mathfrak{t}})_{N}^{-\frac12} \sum_{w\in\Z}e_{R,\mathfrak{t},x,w}^{N,-8} \cdot \wt{\sum}_{\mathfrak{l}=1}^{N^{\beta_{X}}}e_{N,w,y}^{2}\||\mathbf{J}_{S,T+R,w,y}^{N}| + |\mathbf{J}_{S,T+R,w-7\mathfrak{l}\mathfrak{m},y}^{N}|\|_{\omega;2p}^{2} \ \d R \\ 
&\lesssim \ N^{\frac12}\int_{0}^{\mathfrak{t}}(\mathfrak{s}_{R,\mathfrak{t}})_{N}^{-\frac12} \sum_{w\in\Z}e_{R,\mathfrak{t},x,w}^{N,-8} \cdot \wt{\sum}_{\mathfrak{l}=1}^{N^{\beta_{X}}}\left(e_{N,w,y}^{2}\|\mathbf{J}_{S,T+R,w,y}^{N}\|_{\omega;2p} + e_{N,w-7\mathfrak{l}\mathfrak{m},y}^{2}\|\mathbf{J}_{S,T+R,w-7\mathfrak{l}\mathfrak{m},y}^{N}\|_{\omega;2p}^{2}\right) \ \d R \\
&\lesssim \ N^{\frac12} \int_{0}^{\mathfrak{t}} \sup_{w\in\Z}e_{N,w,y}^{2}\|\mathbf{J}_{S,T+R,w,y}^{N}\|_{\omega;2p}^{2} \ \d R. \label{eq:CtifySFS17}
\end{align}\normalsize\normalsize
Lastly, we treat the final term from the RHS of \eqref{eq:CtifySFS12}. Following the weakly vanishing-term and gradient-term estimates from the proof of Proposition 3.2 in \cite{DT}, which resemble the proof of \eqref{eq:CtifySFS15} but with an additional summation-by-parts, we get
\small\begin{align}
e_{N,x,y}^{2}\|\mathbf{H}_{\mathfrak{t},x}^{N}(\Phi^{N,3}_{S,T+\bullet,\bullet,y})\|_{\omega;2p}^{2} \ &\lesssim \ \int_{0}^{\mathfrak{t}}\sum_{w\in\Z}e_{N,x,w}^{2}\left(\mathbf{H}_{R,\mathfrak{t},x,w}^{N} + \sup_{|k|\lesssim\mathfrak{m}}|\grad_{k}^{!}\mathbf{H}_{R,\mathfrak{t},x,w}^{N}| \right)\cdot e_{N,w,y}^{2}\|\mathbf{J}_{S,T+R,w,y}^{N}\|_{\omega;2p}^{2} \ \d R \\
&\lesssim \ \int_{0}^{\mathfrak{t}}\sup_{w\in\Z}e_{N,w,y}^{2}\|\mathbf{J}_{S,T+R,w,y}^{N}\|_{\omega;2p}^{2} \sum_{w\in\Z}e_{N,x,w}^{2}\left(\mathbf{H}_{R,\mathfrak{t},x,w}^{N} + \sup_{|k|\lesssim\mathfrak{m}}|\grad_{k}^{!}\mathbf{H}_{R,\mathfrak{t},x,w}^{N}| \right) \ \d R \\
&\lesssim \  \int_{0}^{\mathfrak{t}}\sup_{w\in\Z}e_{N,w,y}^{2}\|\mathbf{J}_{S,T+R,w,y}^{N}\|_{\omega;2p}^{2} \cdot \left( (\mathfrak{s}_{R,\mathfrak{t}})_{N}^{-\frac12} \sum_{w\in\Z}e_{R,\mathfrak{t},x,w}^{N,-8} \ + \ N(\mathfrak{s}_{R,\mathfrak{t}})_{N}^{-1} \sum_{w\in\Z}e_{R,\mathfrak{t},x,w}^{N,-8} \right) \ \d R \\
&\lesssim \ \int_{0}^{\mathfrak{t}}\left(1+N(\mathfrak{s}_{R,\mathfrak{t}})_{N}^{-\frac12}\right) \sup_{w\in\Z}e_{N,w,y}^{2}\|\mathbf{J}_{S,T+R,w,y}^{N}\|_{\omega;2p}^{2} \ \d R \\
&\lesssim \ \int_{0}^{\mathfrak{t}}\mathfrak{s}_{R,\mathfrak{t}}^{-\frac12} \sup_{w\in\Z}e_{N,w,y}^{2}\|\mathbf{J}_{S,T+R,w,y}^{N}\|_{\omega;2p}^{2} \ \d R. \label{eq:CtifySFS16}
\end{align}\normalsize\normalsize
The previous display additionally requires \eqref{eq:CtifySFS11a} to treat the gradient of the heat kernel that we have not touched yet for this proof along with the estimate $\mathfrak{t}_{N}^{-1/2}\sum_{w\in\Z}e_{0,\mathfrak{t},x,w}^{N,-8}\lesssim1$ that we have been using frequently for this proof. We also use the estimate $(\mathfrak{s}_{R,\mathfrak{t}})_{N} \gtrsim N^{2}\mathfrak{s}_{R,\mathfrak{t}}$ which follows by definition and the estimate $1\geq\mathfrak{s}_{R,\mathfrak{t}}$ for $R\leq\mathfrak{t}\leq1$ to get the last bound \eqref{eq:CtifySFS16}. Otherwise, we proceed as in the derivations of \eqref{eq:CtifySFS13}, \eqref{eq:CtifySFS14}, \eqref{eq:CtifySFS15}, and \eqref{eq:CtifySFS17} to arrive at \eqref{eq:CtifySFS16} above. We now define 
\small\begin{align}
\Phi_{R} \ \overset{\bullet}= \ \sup_{w\in\Z}e_{N,w,y}^{2}\|\mathbf{J}_{S,T+R,w,y}^{N}\|_{\omega;2p}^{2}.
\end{align}\normalsize\normalsize
We gather estimates in \eqref{eq:CtifySFS12}, \eqref{eq:CtifySFS13}, \eqref{eq:CtifySFS14}, \eqref{eq:CtifySFS15}, \eqref{eq:CtifySFS17}, and \eqref{eq:CtifySFS16} to deduce
\small\begin{align}
\Phi_{\mathfrak{t}} \ &\lesssim_{p} \ \Phi_{0} + \int_{0}^{\mathfrak{t}}\mathfrak{s}_{R,\mathfrak{t}}^{-\frac12}\sup_{0\leq\mathfrak{r}\leq\mathfrak{t}}\Phi_{\mathfrak{r}} \ \d R + \int_{0}^{\mathfrak{t}}\mathfrak{s}_{R,\mathfrak{t}}^{-\frac12}\Phi_{R} \ \d R + N^{\frac12}\int_{0}^{\mathfrak{t}}\Phi_{R} \ \d R \ \lesssim \ \Phi_{0} + \mathfrak{t}^{\frac12}\sup_{0\leq\mathfrak{r}\leq\mathfrak{t}}\Phi_{\mathfrak{r}} + \int_{0}^{\mathfrak{t}}\mathfrak{s}_{R,\mathfrak{t}}^{-\frac12}\Phi_{R} \ \d R + N^{\frac12}\int_{0}^{\mathfrak{t}}\Phi_{R} \ \d R. \label{eq:CtifySFS18}
\end{align}\normalsize\normalsize
The first integral on the far RHS of \eqref{eq:CtifySFS18} has a singular factor in the integrand, but it is integrable so that we can still apply the Gronwall inequality. The Gronwall inequality ultimately gives the following uniformly in $0\leq\mathfrak{t}\leq N^{-1/2}$:
\small\begin{align}
\sup_{0\leq\mathfrak{t}\leq N^{-1/2}}\Phi_{\mathfrak{t}} \ \lesssim_{p} \ \Phi_{0} + \sup_{0\leq\mathfrak{t}\leq N^{-1/2}}\mathfrak{t}^{\frac12}\sup_{0\leq\mathfrak{r}\leq \mathfrak{t}}\Phi_{\mathfrak{r}} \ \lesssim \ \Phi_{0} + N^{-\frac14}\sup_{0\leq\mathfrak{r}\leq N^{-1/2}}\Phi_{\mathfrak{r}}.
\end{align}\normalsize\normalsize
We complete the proof upon moving the second term on the far RHS to the far LHS. 
\end{proof}
\begin{proof}[Proof of \emph{Lemma \ref{lemma:CtifySFS2}}]
We proceed as with classical parabolic semigroups and check the RHS of the claimed identity solves the defining SDE-type stochastic equation for $\bar{\mathbf{J}}^{N}$ given an appropriate choice of functionals. This would provide the desired result as the initial data at $T = S$ match on both sides of the proposed identity, from which we employ the elementary uniqueness for solutions to the linear SDE-type equation for $\bar{\mathbf{J}}^{N}$. In particular, all of our reasoning besides tracking the support of both $\mathfrak{f}^{\mathfrak{l},1}$ and $\mathfrak{f}^{\mathfrak{l},2}$ functionals will be algebraic. Defining $\Phi_{T}$ as the entire RHS of the proposed identity and $\d_{T}$ as the $T$-time differential,
\small\begin{align}
\d_{T}\Phi_{T} \ &= \ \d_{T}\mathbf{J}_{S,T,x,y}^{N} \ + \ \d_{T}\int_{S}^{T}\sum_{w\in\mathbb{T}_{N}}\bar{\mathbf{J}}_{R,T,x,w}^{N} \cdot N^{2}\sum_{|\mathfrak{l}|\leq10N^{\beta_{X}}\mathfrak{m}}\left(\mathfrak{f}_{R,w}^{\mathfrak{l},1}\mathbf{J}_{S,R,w+\mathfrak{l},y}^{N} + \mathfrak{f}_{R,w}^{\mathfrak{l},2}\mathbf{J}_{S,R,w\wt{+}\mathfrak{l},y}^{N}\right) \ \d R. \label{eq:CtifySFS21}
\end{align}\normalsize\normalsize
Denote the first term on the RHS of \eqref{eq:CtifySFS21} by $\Phi_{1}$ and the second term by $\Phi_{2}$. Now define $\Psi_{R,w}^{N,\mathfrak{l}} \overset{\bullet}= \mathfrak{f}_{R,w}^{\mathfrak{l},1}\mathbf{J}_{S,R,w+\mathfrak{l},y}^{N} + \mathfrak{f}_{R,w}^{\mathfrak{l},2}\mathbf{J}_{S,R,w\wt{+}\mathfrak{l},y}^{N}$ as what is integrated in $\Phi_{2}$ on the RHS of \eqref{eq:CtifySFS21}. We start by computing $\Phi_{2}$ using the chain rule, where the first term on the RHS of \eqref{eq:CtifySFS22} comes by applying $\d_{T}$ to $\bar{\mathbf{J}}^{N}$ and the second term comes by applying $\d_{T}$ to the integral while recalling $\bar{\mathbf{J}}^{N}_{T,T,x,y} = \mathbf{1}_{x=y}$:
\small\begin{align}
\Phi_{2} \ &= \ \int_{S}^{T}\sum_{w\in\mathbb{T}_{N}}\d_{T}\bar{\mathbf{J}}_{R,T,x,w}^{N} \cdot N^{2}\sum_{|\mathfrak{l}|\leq10N^{\beta_{X}}\mathfrak{m}}\Psi_{R,w}^{N,\mathfrak{l}} \ \d R \ + \ N^{2} \sum_{|\mathfrak{l}| \leq 10N^{\beta_{X}}\mathfrak{m}}\Psi_{T,x}^{N,\mathfrak{l}} \ \d T. \label{eq:CtifySFS22}
\end{align}\normalsize\normalsize
Let $\bar{\mathscr{T}}$ be the linear operator acting on functions $\varphi:\R_{\geq0}\times\mathbb{T}_{N} \to \R$ so $\d_{T}\bar{\mathbf{J}}^{N}_{S,T,x,y} = \bar{\mathscr{T}}\bar{\mathbf{J}}^{N}_{S,T,x,y} \d T$. By \eqref{eq:CtifySFS21} and \eqref{eq:CtifySFS22} and recalling the definition of $\Phi_{T}$ as the RHS of the proposed estimate, we get the following which we explain afterwards:
\small\begin{align}
\d_{T}\Phi_{T} \ &= \ \d_{T}\mathbf{J}_{S,T,x,y}^{N} + \left(\int_{S}^{T}\sum_{w\in\mathbb{T}_{N}}\bar{\mathscr{T}}\bar{\mathbf{J}}_{R,T,x,w}^{N} \cdot N^{2}\sum_{|\mathfrak{l}|\leq10N^{\beta_{X}}\mathfrak{m}}\Psi_{R,w}^{N,\mathfrak{l}} \ \d R\right)\d T + N^{2} \sum_{|\mathfrak{l}|\leq 10N^{\beta_{X}}\mathfrak{m}}\Psi_{T,x}^{N,\mathfrak{l}} \d T \nonumber \\
&= \ \d_{T}\mathbf{J}_{S,T,x,y}^{N} + \bar{\mathscr{T}}\left(\int_{S}^{T}\sum_{w\in\mathbb{T}_{N}}\bar{\mathbf{J}}_{R,T,x,w}^{N} \cdot N^{2}\sum_{|\mathfrak{l}|\leq10N^{\beta_{X}}\mathfrak{m}}\Psi_{R,w}^{N,\mathfrak{l}} \ \d R\right)\d T + N^{2} \sum_{|\mathfrak{l}|\leq 10N^{\beta_{X}}\mathfrak{m}}\Psi_{T,x}^{N,\mathfrak{l}} \d T \nonumber \\
&= \ \bar{\mathscr{T}}\left(\mathbf{J}_{S,T,x,y}^{N} + \int_{S}^{T}\sum_{w\in\mathbb{T}_{N}}\bar{\mathbf{J}}_{R,T,x,w}^{N} \cdot N^{2}\sum_{|\mathfrak{l}|\leq10N^{\beta_{X}}\mathfrak{m}}\Psi_{R,w}^{N,\mathfrak{l}} \ \d R\right)\d T + \left(\d_{T}\mathbf{J}_{S,T,x,y}^{N} - \bar{\mathscr{T}}\mathbf{J}_{S,T,x,y}^{N}\d T\right) + N^{2} \sum_{|\mathfrak{l}|\leq 10N^{\beta_{X}}\mathfrak{m}}\Psi_{T,x}^{N,\mathfrak{l}} \d T \nonumber \\ 
&= \ \bar{\mathscr{T}}\Phi_{T}\d T + \left(\d_{T}\mathbf{J}_{S,T,x,y}^{N} - \bar{\mathscr{T}}\mathbf{J}_{S,T,x,y}^{N}\d T\right) + N^{2}\sum_{|\mathfrak{l}|\leq 10N^{\beta_{X}}\mathfrak{m}}\Psi_{T,x}^{N,\mathfrak{l}} \d T. \label{eq:CtifySFS23}
\end{align}\normalsize\normalsize
To be precise, the first line above follows from combining \eqref{eq:CtifySFS21} and \eqref{eq:CtifySFS22}. The second identity above follows by pulling out the time-independent $\bar{\mathscr{T}}$-operator, acting on $\bar{\mathbf{J}}^{N}$ through the $x$-variable, outside the integral. The third line replaces $\d_{T}\mathbf{J}^{N}$ with $\bar{\mathscr{T}}\mathbf{J}^{N}$ and isolates the resulting error in such replacement. The last line follows by definition of $\Phi_{T}$.

Recalling the definition of $\Psi^{N,\mathfrak{l}}$, it is left to show that the second and third terms within \eqref{eq:CtifySFS23} cancel for appropriate choices of functionals $\mathfrak{f}^{\mathfrak{l}}$. To this end, first observe that the second term in \eqref{eq:CtifySFS23} vanishes as long as the usual distance between $x \in \mathbb{T}_{N}$ and the boundary of $\mathbb{T}_{N}$ is at least $10N^{\beta_{X}}\mathfrak{m}$, thus for these points we may choose $\mathfrak{f}^{\mathfrak{l},i} = 0$. On the other hand, for points within $10N^{\beta_{X}}\mathfrak{m}$ of the boundary, the second term involves contributions from $\mathbf{J}^{N}$ evaluated at sites in a neighborhood in $\mathbb{T}_{N}$ of radius $10N^{\beta_{X}}\mathfrak{m}$ under the geodesic distance on $\mathbb{T}_{N}$ and sites in a neighborhood in $\Z$ of the same radius of $10N^{\beta_{X}}\mathfrak{m}$ but under the usual distance on $\Z$, all afterwards scaled by $N^{2}$ and additional uniformly bounded functionals. This paragraph may be checked just by unfolding definition of the local terms/operators $\d_{T}\mathbf{J}^{N}$ and $\bar{\mathscr{T}}$. The second set of sites where we pick $\mathfrak{f}^{\mathfrak{l}} \neq 0$ are necessarily bounded below by $N^{5/4+2\e_{\mathrm{cpt}}/3}$ in absolute value by required proximity to the boundary of $\mathbb{T}_{N}$. This completes the proof.
\end{proof}
%
%
%
\section{Dynamical One-Block Analysis}\label{section:D1B}
We study the first two/non-gradient terms in $\bar{\Phi}^{N,2}$ of Definition \ref{definition:ChiTorus} by ``local equilibrium" in Lemma \ref{lemma:LE}, invariant measure bounds in Lemmas \ref{lemma:KV}, \ref{lemma:SpectralH-1}, \ref{lemma:H-1SpectralPGF}, and Corollary \ref{corollary:LDP}, then combined with analytic inputs of heat kernel estimates for $\bar{\mathbf{H}}^{N}$.
\begin{itemize}[leftmargin=*]
\item We first introduce estimates of main interest in this section.  The proofs for all results will be deferred until after statements of all main results; we invite the reader, at least in a first reading, to skip these technical proofs and go to Section \ref{section:KPZ1}. We clarify that there are five propositions in this current section, the last four of which carry corollaries that will be crucial to the proof of Theorem \ref{theorem:KPZ}. We give the proof of every corollary after its statement as each such proof will be almost trivial.
\item The proof of each main result in this section will each require a few preliminary ingredients. We present these preliminaries for each result and combine them to give the proof of the corresponding result. We defer the proof of \emph{all} preliminary lemmas to the end of this section to avoid obscuring the idea behind proofs of main results with a host of technical manipulations.
\item The results are given in the following order. First, we will introduce a mechanism to replace the order $N^{1/2}$ spatial average of pseudo-gradients by a CLT-type cutoff that resembles Pseudo-Proposition \ref{pprop:S2}. We then give an estimate that we eventually apply to estimate errors in the multiscale time-replacement strategy discussed in the would-be-proof for Pseudo-Proposition \ref{pprop:S3}. Third, we present estimates to analyze the error terms, which we recall are time-averages of pseudo-gradient content with two-sided a priori bounds, from the multiscale ``gluing" that we discussed in the would-be-proof of Pseudo-Proposition \ref{pprop:S4}. We will write versions of the estimates for the first order $N^{1/2}$ term in $\bar{\Phi}^{N,2}$ and then for the order $N^{\beta_{X}}$-terms in $\bar{\Phi}^{N,2}$. We then organize resulting corollaries that turn the aforementioned estimates into multiscale results. Lastly, we give proofs.
\end{itemize}
We inherit notation of Proposition \ref{prop:Duhamel}. We invite the reader to go back to Proposition \ref{prop:Duhamel} for notation throughout this section. 
\subsection{Static Analysis}
Let us first introduce the following cutoff basically from the statement of Pseudo-Proposition \ref{pprop:S2}.
\begin{definition}\label{definition:S1B}
Recall $\beta_{X}=\frac13+\e_{X,1}$ and $\mathscr{A}^{\mathbf{X},-}$ in Proposition \ref{prop:Duhamel}. Set $\mathbf{1}[\mathscr{E}_{S,y}^{\mathbf{X},-}]\overset{\bullet}=\mathbf{1}[|\mathscr{A}_{N^{\beta_{X}}}^{\mathbf{X},-}(\mathfrak{g}_{S,y})|\lesssim N^{-\frac12\beta_{X}+\frac12\e_{X,1}}]$ and 
\small\begin{align}
\mathscr{C}_{N^{\beta_{X}}}^{\mathbf{X},-}(\mathfrak{g}_{S,y}) \ &\overset{\bullet}= \ \mathscr{A}_{N^{\beta_{X}}}^{\mathbf{X},-}(\mathfrak{g}_{S,y}) \cdot \mathbf{1}[\mathscr{E}_{S,y}^{\mathbf{X},-}].
\end{align}\normalsize\normalsize
\end{definition}
The following result is a quantitative and $\mathbb{T}_{N}$-adapted version of Pseudo-Proposition \ref{pprop:S2}. Its proof, as we will see towards the second half of this section, depends on the observation that the introduction of the CLT-type cutoff at canonical ensembles does something only with \emph{exponentially low} probability. We then apply the local equilibrium reduction in Lemma \ref{lemma:LE}.
\begin{prop}\label{prop:S1B}
 With $\mathscr{A}^{\mathbf{X},-}_{N^{\beta_{X}}}(\mathfrak{g})$ in \emph{Proposition \ref{prop:Duhamel}} and $\mathscr{C}_{N^{\beta_{X}}}^{\mathbf{X},-}(\mathfrak{g})$ in \emph{Definition \ref{definition:S1B}}, there exists $\beta_{\mathrm{univ}} \in \R_{>0}$ universal so that
\small\begin{align}
\E\|\bar{\mathbf{H}}_{T,x}^{N}(N^{\frac12}|\mathscr{A}_{N^{\beta_{X}}}^{\mathbf{X},-}(\mathfrak{g}_{S,y}) - \mathscr{C}_{N^{\beta_{X}}}^{\mathbf{X},-}(\mathfrak{g}_{S,y})|)\|_{1;\mathbb{T}_{N}} \ &\lesssim \ N^{-\beta_{\mathrm{univ}}}.
\end{align}\normalsize\normalsize
\end{prop}
The utility of Proposition \ref{prop:S1B} is in a priori bounds on time-averages of $\mathscr{C}^{\mathbf{X},-}$. This is our only advantage of $\mathscr{C}^{\mathbf{X},-}$.
\subsection{Dynamic Analysis IA}
We move onto estimates behind implementation of a quantitative $\mathbb{T}_{N}$-adapted version of Pseudo-Proposition \ref{pprop:S3}. We recall that this pseudo-proposition justifies the replacement of the spatial-average-with-cutoff $\mathscr{C}^{\mathbf{X},-}$ from Definition \ref{definition:S1B} with its time-average on some mesoscopic time-scale. We additionally recall that the strategy behind the would-be-proof of Pseudo-Proposition \ref{pprop:S3} amounts to a multiscale scheme given by a replacement-by-time-average on progressively larger time-scales until we reach the ``maximal" mesoscopic time-scale. We start with notation for time-averages.
\begin{definition}\label{definition:D1B1A}
Take $\mathfrak{t}_{\mathrm{av}}\in\R_{\geq0}$ and define the following time-average $\mathscr{A}^{\mathbf{T},+}$ acting on the spatial-average and its cutoff:
\begin{subequations}
\small\begin{align}
\mathscr{A}_{\mathfrak{t}_{\mathrm{av}}}^{\mathbf{T},+}\mathscr{A}_{N^{\beta_{X}}}^{\mathbf{X},-}(\mathfrak{g}_{S,y}) \ &\overset{\bullet}= \ \mathbf{1}_{\mathfrak{t}_{\mathrm{av}}\neq0}\mathfrak{t}_{\mathrm{av}}^{-1}\int_{0}^{\mathfrak{t}_{\mathrm{av}}}\mathscr{A}_{N^{\beta_{X}}}^{\mathbf{X},-}(\mathfrak{g}_{S+\mathfrak{r},y}) \d\mathfrak{r} \ + \ \mathbf{1}_{\mathfrak{t}_{\mathrm{av}}=0}\mathscr{A}_{N^{\beta_{X}}}^{\mathbf{X},-}(\mathfrak{g}_{S,y}) \\
\mathscr{A}_{\mathfrak{t}_{\mathrm{av}}}^{\mathbf{T},+}\mathscr{C}_{N^{\beta_{X}}}^{\mathbf{X},-}(\mathfrak{g}_{S,y}) \ &\overset{\bullet}= \ \mathbf{1}_{\mathfrak{t}_{\mathrm{av}}\neq0}\mathfrak{t}_{\mathrm{av}}^{-1}\int_{0}^{\mathfrak{t}_{\mathrm{av}}}\mathscr{C}_{N^{\beta_{X}}}^{\mathbf{X},-}(\mathfrak{g}_{S+\mathfrak{r},y}) \d\mathfrak{r} \ + \ \mathbf{1}_{\mathfrak{t}_{\mathrm{av}}=0}\mathscr{C}_{N^{\beta_{X}}}^{\mathbf{X},-}(\mathfrak{g}_{S,y}).
\end{align}\normalsize\normalsize
\end{subequations}
Let us clarify that the first time-average $\mathscr{A}^{\mathbf{T},+}\mathscr{A}^{\mathbf{X},-}$ will not be used until later in a future subsection concerning the proof for a result dedicated towards the would-be-proof of Pseudo-Proposition \ref{pprop:S4}. We think of $\mathscr{A}^{\mathbf{T},+}$ as a time-averaging operator.
\end{definition}
\begin{prop}\label{prop:D1B1A}
 Take any deterministic $\mathfrak{t}_{\mathrm{av}} \in \R_{\geq0}$ satisfying $N^{-2} \lesssim \mathfrak{t}_{\mathrm{av}} \lesssim N^{-1}$. There exists $\beta_{\mathrm{univ}} \in \R_{>0}$ universal so that
\small\begin{align}
\mathfrak{t}_{\mathrm{av}}^{1/4} \E\|\bar{\mathbf{H}}_{T,x}^{N}(N^{\frac12}|\mathscr{A}_{\mathfrak{t}_{\mathrm{av}}}^{\mathbf{T},+}\mathscr{C}_{N^{\beta_{X}}}^{\mathbf{X},-}(\mathfrak{g}_{S,y})|)\|_{1;\mathbb{T}_{N}} \ &\lesssim \ N^{-\beta_{\mathrm{univ}}}.
\end{align}\normalsize\normalsize
\end{prop}
The factor $\mathfrak{t}_{\mathrm{av}}^{1/4}$ will come from time-regularity of the $\bar{\mathbf{Z}}^{N}$-process, so the nature of Proposition \ref{prop:D1B1A} matches that of \eqref{eq:S32}.
\subsection{Dynamic Analysis IIA}
We now focus on estimates for the multiscale idea in the would-be-proof for Pseudo-Proposition \ref{pprop:S4}. As discussed earlier in the would-be-proof of Pseudo-Proposition \ref{pprop:S4}, this means we will study time-averages of cutoffs $\mathscr{C}^{\mathbf{X},-}$ of spatial averages $\mathscr{A}^{\mathbf{X},-}$ of pseudo-gradients by analyzing two-sided cutoffs for these time-averages. First, some notation to make presenting these two-sided cutoff estimates more convenient.
\begin{definition}\label{definition:D1B2A}
Recall the constructions of time-averages from Definition \ref{definition:D1B1A}. Define the following pair of events controlling maximal processes and time-averages. We clarify that in these following events, the process of interest is the supremum of the \emph{integral} of $\mathscr{C}^{\mathbf{X},-}$-terms over all time-scales $0\leq\mathfrak{t}\leq\mathfrak{t}_{\mathrm{av}}$ that are then weighted by $\mathfrak{t}_{\mathrm{av}}^{-1}$. These are only time-averages if $\mathfrak{t}=\mathfrak{t}_{\mathrm{av}}$:
\begin{subequations}
\small\begin{align}
\mathbf{1}[\mathscr{G}_{\lesssim;S,y}^{\beta_{-}}] \ &\overset{\bullet}= \ \mathbf{1}\left({\sup}_{0\leq\mathfrak{t}\leq\mathfrak{t}_{\mathrm{av}}}\mathfrak{t}\cdot\mathfrak{t}_{\mathrm{av}}^{-1}|\mathscr{A}_{\mathfrak{t}}^{\mathbf{T},+}\mathscr{C}_{N^{\beta_{X}}}^{\mathbf{X},-}(\mathfrak{g}_{S,y})|\lesssim N^{-\beta_{-}}\right) \\
\mathbf{1}[\mathscr{G}_{\geq;S,y}^{\beta_{+}}] \ &\overset{\bullet}= \ \mathbf{1}\left({\sup}_{0\leq\mathfrak{t}\leq\mathfrak{t}_{\mathrm{av}}}\mathfrak{t}\cdot\mathfrak{t}_{\mathrm{av}}^{-1}|\mathscr{A}_{\mathfrak{t}}^{\mathbf{T},+}\mathscr{C}_{N^{\beta_{X}}}^{\mathbf{X},-}(\mathfrak{g}_{S,y})|\geq N^{-\beta_{+}}\right).
\end{align}\normalsize\normalsize
\end{subequations}
We will eventually take $\beta_{-},\beta_{+}\in\R_{\geq0}$ deterministic such that $\beta_{-}\leq\beta_{+}$. The implied constant in the first event is universal.

Take any deterministic $\mathfrak{t}_{N,\e}\in\R_{\geq0}$ and define the following ``two-sided cutoffs" for the time-averages in Definition \ref{definition:D1B1A}:
\begin{subequations}
\small\begin{align}
\mathscr{C}_{\mathfrak{t}_{\mathrm{av}};\beta_{-},\beta_{+}}^{\mathbf{T},+,\mathfrak{t}_{N,\e},1}\mathscr{C}_{N^{\beta_{X}}}^{\mathbf{X},-}(\mathfrak{g}_{S,y}) \ &\overset{\bullet}= \ \mathscr{A}_{\mathfrak{t}_{\mathrm{av}}}^{\mathbf{T},+}\mathscr{C}_{N^{\beta_{X}}}^{\mathbf{X},-}(\mathfrak{g}_{S+\mathfrak{t}_{N,\e},y})  \cdot \mathbf{1}[\mathscr{G}_{\lesssim;S+\mathfrak{t}_{N,\e},y}^{\beta_{-}}]\mathbf{1}[\mathscr{G}_{\geq;S,y}^{\beta_{+}}] \\
\mathscr{C}_{\mathfrak{t}_{\mathrm{av}};\beta_{-},\beta_{+}}^{\mathbf{T},+,\mathfrak{t}_{N,\e},2}\mathscr{C}_{N^{\beta_{X}}}^{\mathbf{X},-}(\mathfrak{g}_{S,y}) \ &\overset{\bullet}= \  \mathscr{A}_{\mathfrak{t}_{\mathrm{av}}}^{\mathbf{T},+}\mathscr{C}_{N^{\beta_{X}}}^{\mathbf{X},-}(\mathfrak{g}_{S,y})  \cdot \mathbf{1}[\mathscr{G}_{\lesssim;S,y}^{\beta_{-}}]\mathbf{1}[\mathscr{G}_{\geq;S+\mathfrak{t}_{N,\e},y}^{\beta_{+}}].
\end{align}\normalsize\normalsize
\end{subequations}
In words, the first $\mathscr{C}^{\mathbf{T},+,1}$-operator adjusts the $\mathscr{A}^{\mathbf{T},+}$-operator with forward time-shift by $\mathfrak{t}_{N,\e}\in\R_{\geq0}$. It also introduces an upper bound cutoff of order $N^{-\beta_{-}}$ for this shifted time-average, and then it introduces the lower bound cutoff of $N^{-\beta_{+}}$ for the \emph{unshifted} time-average. Meanwhile the second $\mathscr{C}^{\mathbf{T},+,2}$-operator adjusts the $\mathscr{A}^{\mathbf{T},+}$-operator by introducing an upper bound cutoff of order $N^{-\beta_{-}}$ and then a lower bound cutoff of $N^{-\beta_{+}}$ for a $\mathfrak{t}_{N,\e}$-\emph{shifted} time-average. Note that by taking $\mathfrak{t}=\mathfrak{t}_{\mathrm{av}}$ in the suprema defining $\mathscr{G}^{\beta_{-}}$-events in $\mathscr{C}^{\mathbf{T},+,1}$ and $\mathscr{C}^{\mathbf{T},+,2}$ above, we get the following deterministic bound with a universal implied constant:
\small\begin{align}
|\mathscr{C}_{\mathfrak{t}_{\mathrm{av}};\beta_{-},\beta_{+}}^{\mathbf{T},+,\mathfrak{t}_{N,\e},1}\mathscr{C}_{N^{\beta_{X}}}^{\mathbf{X},-}(\mathfrak{g}_{S,y})| + |\mathscr{C}_{\mathfrak{t}_{\mathrm{av}};\beta_{-},\beta_{+}}^{\mathbf{T},+,\mathfrak{t}_{N,\e},2}\mathscr{C}_{N^{\beta_{X}}}^{\mathbf{X},-}(\mathfrak{g}_{S,y})| \ \lesssim \ N^{-\beta_{-}}.
\end{align}\normalsize\normalsize
\end{definition}
We emphasize that these $\mathscr{C}^{\mathbf{T},+,\mathfrak{i}}$-cutoffs will be error terms that we get as the result of starting with an a priori cutoff for time-averages of $\mathscr{C}^{\mathbf{X},-}$-terms, improving cutoffs slightly, gluing time-averages with slightly improved estimates to time-averages on slightly bigger time-scales with the same slightly improved estimates, and iterating these steps starting with, again, improving slightly the last a priori estimates. This is the multiscale procedure in the would-be-proof of Pseudo-Proposition \ref{pprop:S4}. The only difference with these $\mathscr{C}^{\mathbf{T},+,\mathfrak{i}}$-cutoffs and the error terms presented in the would-be-proof for Pseudo-Proposition \ref{pprop:S4} is that the previous $\mathscr{C}^{\mathbf{T},+,\mathfrak{i}}$-cutoffs are defined with cutoffs for \emph{maximal processes} associated to the time-averages in Definition \ref{definition:D1B1A}. Also the additional time-shift $\mathfrak{t}_{N,\e}\in\R_{\geq0}$ is absent from the would-be-proof of Pseudo-Proposition \ref{pprop:S4}. This point/difference is mostly technical, especially because the time-shift $\mathfrak{t}_{N,\e}$ will basically be of the same order as the time-average scale $\mathfrak{t}_{\mathrm{av}}$.

The main result for this current subsection is the following control for these $\mathscr{C}^{\mathbf{T},+,\mathfrak{i}}$-cutoffs against the heat operator $\bar{\mathbf{H}}^{N}$ with an appropriate $\mathbb{T}_{N}$-adapted norm. We will emphasize here only the relationship between the time-scales for time-averages and the exponents for the $\mathscr{C}^{\mathbf{T},+,\mathfrak{i}}$-cutoffs in the following result: given better a priori upper bounds, and thus for $\beta_{-}\in\R_{\geq0}$ bigger, we can reduce to local equilibrium with Lemma \ref{lemma:LE} for a longer time-scale $\mathfrak{t}_{\mathrm{av}}\in\R_{\geq0}$ for the time-average. With a longer time-scale $\mathfrak{t}_{\mathrm{av}}\in\R_{\geq0}$ for the time-average, we can upgrade the exponent $\beta_{-}\in\R_{\geq0}$ to a slightly bigger $\beta_{+}\in\R_{\geq0}$ and then control the error since \emph{at invariant measures/statistical equilibrium} the lower-bound cutoff of $N^{-\beta_{+}}$ is negligible with sufficiently high probability if we pick a sufficiently large time-scale $\mathfrak{t}_{\mathrm{av}}\in\R_{\geq0}$ for the time-average.
\begin{prop}\label{prop:D1B2A}
 Consider deterministic data $\beta_{\pm} \in \R_{\geq0}$ and $\mathfrak{t}_{\mathrm{av}} \in \R_{\geq0}$ and $\mathfrak{t}_{N,\e}\in\R_{\geq0}$ satisfying the following constraints.
\begin{itemize}[leftmargin=*]
\item We have $\beta_{-} \geq \frac12\beta_{X} - \e_{X,2}$ and $\beta_{+} = \beta_{-} +\e_{X,3}$ for $\e_{X,2},\e_{X,3}>0$ sufficiently small but universal. We have $\beta_{+} \leq \frac23$.
\item Define $\mathfrak{t}_{\mathrm{av}} \overset{\bullet}= N^{-1 - \beta_{X} - \frac12 \beta_{-} + \beta_{+} + \e_{X,4}}$ for $\e_{X,4} \in \R_{>0}$ arbitrarily small but uniformly bounded below.
\item Consider $\mathfrak{t}_{N,\e}\in\R_{\geq0}$ such that $0\leq\mathfrak{t}_{N,\e}\lesssim N^{\e}\mathfrak{t}_{\mathrm{av}}$ with $\e\in\R_{>0}$ arbitrarily small but universal.
\end{itemize}
Define $\|\varphi\|_{\bullet;\mathbb{X}} = \|\varphi_{\mathfrak{t},x}\|_{\bullet;\mathbb{X}} = \sup_{0\leq\mathfrak{t}\leq\bullet}\sup_{x\in\mathbb{X}}|\varphi_{\mathfrak{t},x}|$. There exists $\beta_{\mathrm{univ}} \in \R_{>0}$ universal so for any deterministic $0\leq\mathfrak{t}_{\mathfrak{s}},\mathfrak{t}_{\mathfrak{f}} \leq 2$, we have the following estimate, in which $\e_{X}=\max_{i}\e_{X,i}\vee\e$, whose statement we clarify afterwards:
\small\begin{align}
{\sup}_{\mathfrak{i}=1,2} \E\|\bar{\mathbf{H}}_{T,x}^{N}(N^{\frac12}|\mathscr{C}_{\mathfrak{t}_{\mathrm{av}};\beta_{-},\beta_{+}}^{\mathbf{T},+,\mathfrak{t}_{N,\e},\mathfrak{i}}\mathscr{C}_{N^{\beta_{X}}}^{\mathbf{X},-}(\mathfrak{g}_{S+\mathfrak{t}_{\mathfrak{s}},y})|)\|_{\mathfrak{t}_{\mathfrak{f}};\mathbb{T}_{N}} \ &\lesssim \ N^{-\beta_{\mathrm{univ}}} + N^{-\frac34\beta_{-}+10\e_{X}}+N^{-\frac12+10\e_{X}}+N^{-\beta_{\mathrm{univ}}}\mathfrak{t}_{\mathfrak{f}}.
\end{align}\normalsize\normalsize
We emphasize the norm on the LHS is a supremum over the $(T,x)$-variables  in the heat operator. The space-time variables $(S,y)$ are the integration variables in the heat operator and the constant $\mathfrak{t}_{\mathfrak{s}}\in\R_{\geq0}$ is a deterministic time-shift for the particle system data. Lastly, the sup over $\mathfrak{i}\in\{1,2\}$ is over the two choices of time-average cutoffs $\mathscr{C}^{\mathbf{T},+,1}$ and $\mathscr{C}^{\mathbf{T},+,2}$ from \emph{Definition \ref{definition:D1B2A}}.
\end{prop}
\begin{remark}
We could have given the previous result for $\mathfrak{t}_{\mathfrak{f}}=1$ and its utility for the current paper would not have changed. We stated it for general $\mathfrak{t}_{\mathfrak{f}}\in\R_{\geq0}$ as this will be important for the narrow-wedge initial measure to be addressed in a future article. Similarly, the time-shift $\mathfrak{t}_{\mathfrak{s}}\in\R_{\geq0}$ will play almost no role in the proof of Proposition \ref{prop:D1B2A} and is there because of what the terms that we will eventually need to control with Proposition \ref{prop:D1B2A} actually are. See Section \ref{section:KPZ2}/the proof of Lemma \ref{lemma:Step3A} for details.
\end{remark}
We eventually take $\e_{X,4}$ in the statement of Proposition \ref{prop:D1B2A} $N$-dependent, but bounded below uniformly, for presentational convenience in Section \ref{section:KPZ2}. We note this again in Corollary \ref{corollary:D1B2A} which contains a version of Proposition \ref{prop:D1B2A} we will use later.
\subsection{Dynamic Analysis IB}
This current subsection and the next are all basically versions of Proposition \ref{prop:D1B1A}/Proposition \ref{prop:D1B2A} but for order $N^{\beta_{X}}$-terms in $\bar{\Phi}^{N,2}$. The clarifying remarks concerning Proposition \ref{prop:D1B1A} and Proposition \ref{prop:D1B2A} apply here as well. We start by basically replicating Proposition \ref{prop:D1B1A} but for order $N^{\beta_{X}}$-terms in $\bar{\Phi}^{N,2}$. In particular, we justify replacement of $\wt{\mathfrak{g}}^{\mathfrak{l}}$-terms of order $N^{\beta_{X}}$ with their time-averages on a mesoscopic time-scale, and the key estimate to this end is given in Proposition \ref{prop:D1B1B}.
\begin{definition}\label{definition:D1B1B}
 Consider any $\mathfrak{t}_{\mathrm{av}}\in\R_{\geq0}$ and define  the following time-average-operator $\mathscr{A}^{\mathbf{T},+}$ acting on the functional $\wt{\mathfrak{g}}^{\mathfrak{l}}$ from Proposition \ref{prop:Duhamel} that admits a uniformly bounded pseudo-gradient factor:
\small\begin{align}
\mathscr{A}_{\mathfrak{t}_{\mathrm{av}}}^{\mathbf{T},+}(\wt{\mathfrak{g}}_{S,y}^{\mathfrak{l}}) \ &\overset{\bullet}= \ \mathbf{1}_{\mathfrak{t}_{\mathrm{av}}\neq0}\mathfrak{t}_{\mathrm{av}}^{-1}\int_{0}^{\mathfrak{t}_{\mathrm{av}}}\wt{\mathfrak{g}}_{S+\mathfrak{r},y}^{\mathfrak{l}} \d\mathfrak{r} \ + \ \mathbf{1}_{\mathfrak{t}_{\mathrm{av}}=0}\wt{\mathfrak{g}}_{S,y}^{\mathfrak{l}}.
\end{align}\normalsize\normalsize
\end{definition}
\begin{prop}\label{prop:D1B1B}
 Consider any deterministic time-scale $\mathfrak{t}_{\mathrm{av}}>0$ satisfying $N^{-2} \lesssim \mathfrak{t}_{\mathrm{av}} \lesssim N^{-1}$. There exists a universal constant $\beta_{\mathrm{univ}}>0$ so we have the following expectation bound for a $\|\|_{1;\mathbb{T}_{N}}$-norm uniformly over terms-with-pseudo-gradient-factors: 
\small\begin{align}
{\sup}_{\mathfrak{l}=1,\ldots,N^{\beta_{X}}}\mathfrak{t}_{\mathrm{av}}^{1/4} \E\|\bar{\mathbf{H}}_{T,x}^{N}(N^{\beta_{X}}|\mathscr{A}_{\mathfrak{t}_{\mathrm{av}}}^{\mathbf{T},+}(\wt{\mathfrak{g}}_{S,y}^{\mathfrak{l}})|)\|_{1;\mathbb{T}_{N}} \ &\lesssim \ N^{-\beta_{\mathrm{univ}}}.
\end{align}\normalsize\normalsize
\end{prop}
\subsection{Dynamic Analysis IIB}
We basically replicate Proposition \ref{prop:D1B2A} for order $N^{\beta_{X}}$-terms in $\bar{\Phi}^{N,2}$ in this subsection.
\begin{definition}\label{definition:D1B2B}
Recall the constructions of time-averages in Definition \ref{definition:D1B1B}. We define the following pair of events controlling maximal processes and time-averages. Again, like Definition \ref{definition:D1B2A} the following events are concerned with suprema of \emph{integrals} of $\wt{\mathfrak{g}}^{\mathfrak{l}}$-terms over time-scales $0\leq\mathfrak{t}\leq\mathfrak{t}_{\mathrm{av}}$ that are then weighted by $\mathfrak{t}_{\mathrm{av}}^{-1}$. These are only time-averages if $\mathfrak{t}=\mathfrak{t}_{\mathrm{av}}$:
\begin{subequations}
\small\begin{align}
\mathbf{1}[\mathscr{F}_{\lesssim;S,y}^{\beta_{-}}] \ &\overset{\bullet}= \ \mathbf{1}\left({\sup}_{0\leq\mathfrak{t}\leq\mathfrak{t}_{\mathrm{av}}}\mathfrak{t}\cdot\mathfrak{t}_{\mathrm{av}}^{-1}|\mathscr{A}_{\mathfrak{t}}^{\mathbf{T},+}(\wt{\mathfrak{g}}_{S,y}^{\mathfrak{l}})|\lesssim N^{-\beta_{-}}\right)\\
\mathbf{1}[\mathscr{F}_{\geq;S,y}^{\beta_{+}}] \ &\overset{\bullet}= \  \mathbf{1}\left({\sup}_{0\leq\mathfrak{t}\leq\mathfrak{t}_{\mathrm{av}}}\mathfrak{t}\cdot\mathfrak{t}_{\mathrm{av}}^{-1}|\mathscr{A}_{\mathfrak{t}}^{\mathbf{T},+}(\wt{\mathfrak{g}}_{S,y}^{\mathfrak{l}})|\geq N^{-\beta_{+}}\right).
\end{align}\normalsize\normalsize
\end{subequations}
Take any deterministic $\mathfrak{t}_{N,\e}\in\R_{\geq0}$ and define the following ``two-sided cutoffs" for the time-averages in Definition \ref{definition:D1B1B}:
\begin{subequations}
\small\begin{align}
\mathscr{C}_{\mathfrak{t}_{\mathrm{av}};\beta_{-},\beta_{+}}^{\mathbf{T},+,\mathfrak{t}_{N,\e},1}(\wt{\mathfrak{g}}_{S,y}^{\mathfrak{l}}) \ &\overset{\bullet}= \ \mathscr{A}_{\mathfrak{t}_{\mathrm{av}}}^{\mathbf{T},+}(\wt{\mathfrak{g}}_{S+\mathfrak{t}_{N,\e},y}^{\mathfrak{l}})\cdot\mathbf{1}[\mathscr{F}_{\lesssim;S+\mathfrak{t}_{N,\e},y}^{\beta_{-}}]\mathbf{1}[\mathscr{F}_{\geq;S,y}^{\beta_{+}}] \\
\mathscr{C}_{\mathfrak{t}_{\mathrm{av}};\beta_{-},\beta_{+}}^{\mathbf{T},+,\mathfrak{t}_{N,\e},2}(\wt{\mathfrak{g}}_{S,y}^{\mathfrak{l}}) \ &\overset{\bullet}= \ \mathscr{A}_{\mathfrak{t}_{\mathrm{av}}}^{\mathbf{T},+}(\wt{\mathfrak{g}}_{S,y}^{\mathfrak{l}})\cdot\mathbf{1}[\mathscr{F}_{\lesssim;S,y}^{\beta_{-}}]\mathbf{1}[\mathscr{F}_{\geq;S+\mathfrak{t}_{N,\e},y}^{\beta_{+}}].
\end{align}\normalsize\normalsize
\end{subequations}
\end{definition}
\begin{prop}\label{prop:D1B2B}
 Consider deterministic data $\beta_{\pm} \in \R_{\geq0}$ and $\mathfrak{t}_{\mathrm{av}} \in \R_{\geq0}$ and $\mathfrak{t}_{N,\e}\in\R_{\geq0}$ satisfying the following constraints.
\begin{itemize}[leftmargin=*]
\item We have $\beta_{+} = \beta_{-} + \e_{X,2}$ for $\e_{X,2} \in \R_{>0}$ arbitrarily small but universal. We also have $\beta_{+} \leq 5/12$.
\item Define $\mathfrak{t}_{\mathrm{av}} \overset{\bullet}= \ N^{-\frac54-\frac12\beta_{-}+\beta_{+}+\e_{X,3}}$ for $\e_{X,3} \in \R_{>0}$ arbitrarily small but uniformly bounded below.
\item Consider $\mathfrak{t}_{N,\e}\in\R_{\geq0}$ such that $0\leq\mathfrak{t}_{N,\e}\lesssim N^{\e}\mathfrak{t}_{\mathrm{av}}$ with $\e\in\R_{>0}$ arbitrarily small but universal.
\end{itemize}
There exists $\beta_{\mathrm{univ}}>0$ universal so that for any deterministic $0\leq\mathfrak{t}_{\mathfrak{s}},\mathfrak{t}_{\mathfrak{f}} \leq 2$, we have the following in which $\e_{X}=\max_{i}\e_{X,i}\vee\e$:
\small\begin{align}
{\sup}_{\mathfrak{i}=1,2}{\sup}_{\mathfrak{l}=1,\ldots,N^{\beta_{X}}}\E\|\bar{\mathbf{H}}_{T,x}^{N}(N^{\beta_{X}}|\mathscr{C}_{\mathfrak{t}_{\mathrm{av}};\beta_{-},\beta_{+}}^{\mathbf{T},+,\mathfrak{t}_{N,\e},\mathfrak{i}}(\wt{\mathfrak{g}}_{S+\mathfrak{t}_{\mathfrak{s}},y}^{\mathfrak{l}})|)\|_{\mathfrak{t}_{\mathfrak{f}};\mathbb{T}_{N}} \ &\lesssim \ N^{-\beta_{\mathrm{univ}}} + N^{-\frac18-\frac34\beta_{-}+10\e_{X}} + N^{-\frac12+10\e_{X}} + \mathfrak{t}_{\mathfrak{f}}N^{-\beta_{\mathrm{univ}}}. \nonumber
\end{align}\normalsize
\end{prop}
\subsection{Consequences}
We first present the corollary corresponding to Proposition \ref{prop:D1B2A} as it tells us the ``maximal" time-scale for the time-average we want to replace the spatial-average-with-cutoff $\mathscr{C}^{\mathbf{X},-}$ by while being able to estimate each error term from the multiscale idea for Pseudo-Proposition \ref{pprop:S4}. This will be the maximal time-scale we can study with local equilibrium.
\begin{corollary}\label{corollary:D1B2A}
 Consider deterministic data $\{\beta_{m}\}_{m=0}^{\infty}$ and $\{\mathfrak{t}_{\mathrm{av},m}\}_{m=1}^{\infty}$ and $\{\mathfrak{t}_{N,\e,m}\}_{m=1}^{\infty}$ and $\mathfrak{m}_{+} \in \Z_{\geq0}$ defined as follows.
\begin{itemize}[leftmargin=*]
\item Define $\beta_{0} \overset{\bullet}= \frac12\beta_{X} - \e_{X,2}$ and $\beta_{m} \overset{\bullet}= \beta_{m-1} + \e_{X,3}$ for $\e_{X,2},\e_{X,3}>0$ sufficiently small but universal and $m \in \Z_{\geq1}$.
\item Define $\mathfrak{t}_{\mathrm{av},m} \overset{\bullet}= N^{-1 - \beta_{X} - \frac12 \beta_{m-1} + \beta_{m} + \e_{X,4}}$ for $m \in \Z_{\geq 1}$ and $\e_{X,4}>0$ arbitrarily small but uniformly bounded below and chosen so that $\mathfrak{t}_{\mathrm{av},m+1}$ is an integer multiple of $\mathfrak{t}_{\mathrm{av},m}$. The integer-multiple-constraint will only be useful for convenience in \emph{Section \ref{section:KPZ2}}. This may force $\e_{X,4}$ to be $N$-dependent; this will not matter as long as $\e_{X,4}$ is small but uniformly bounded below.
\item Take $\mathfrak{t}_{N,\e,m}\geq0$ to be any deterministic time satisfying $0\leq\mathfrak{t}_{N,\e,m} \lesssim N^{\e}\mathfrak{t}_{\mathrm{av},m}$ with $\e>0$ arbitrarily small but universal.
\item Define $\mathfrak{m}_{+} \in \Z_{\geq 0}$ as the smallest positive integer for which $\beta_{\mathfrak{m}_{+}} > \frac12+3\e_{X,3}$. 
\end{itemize}
For some universal constant $\beta_{\mathrm{univ}} \in \R_{>0}$ and for any deterministic $0\leq\mathfrak{t}_{\mathfrak{s}}\leq2$, we have the estimate
\small\begin{align}
{\sup}_{m=1,\ldots,\mathfrak{m}_{+}} {\sup}_{\mathfrak{i}=1,2} \E\|\bar{\mathbf{H}}_{T,x}^{N}(N^{\frac12}|\mathscr{C}_{\mathfrak{t}_{\mathrm{av},m};\beta_{m-1},\beta_{m}}^{\mathbf{T},+,\mathfrak{t}_{N,\e,m},\mathfrak{i}}\mathscr{C}_{N^{\beta_{X}}}^{\mathbf{X},-}(\mathfrak{g}_{S+\mathfrak{t}_{\mathfrak{s}},y})|)\|_{2;\mathbb{T}_{N}} \ &\lesssim \ N^{-\beta_{\mathrm{univ}}}.
\end{align}\normalsize\normalsize
We have $\mathfrak{m}_{+}\lesssim1$ and $1 \leq \mathfrak{t}_{\mathrm{av},m}\mathfrak{t}_{\mathrm{av},m-1}^{-1} \leq N^{\e_{X,3}}$ and $\mathfrak{t}_{\mathrm{av},\mathfrak{m}_{+}} \lesssim N^{-13/12+\e_{X,5}}$ for $\e_{X,5}>0$ arbitrarily small but universal.
\end{corollary}
\begin{proof}
Each expectation in the proposed bound satisfies constraints of Proposition \ref{prop:D1B2A} for $\mathfrak{t}_{\mathfrak{f}}=2$, so it suffices to show $\mathfrak{m}_{+} \lesssim 1$ and verify the claims about $\mathfrak{t}_{\mathrm{av},m}$. As $\beta_{m} - \beta_{m-1} = \e_{X,3}$, we get the bound $\mathfrak{m}_{+} \lesssim \e_{X,3}^{-1}$. This also gives $\mathfrak{t}_{\mathrm{av},m}\mathfrak{t}_{\mathrm{av},m-1}^{-1} \leq N^{\e_{X,3}}$ immediately in addition to monotonicity in $m$ by slightly further inspection of the definition of $\mathfrak{t}_{\mathrm{av},m}$. To check the $\mathfrak{t}_{\mathrm{av},\mathfrak{m}_{+}}$-bound, we note $\beta_{\mathfrak{m}_{+}}-\frac12\beta_{\mathfrak{m}_{+}-1} = \frac12\beta_{\mathfrak{m}_{+}}+\frac12\e_{X,3} \leq \frac14+3\e_{X,3}$ and $\beta_{X}\geq\frac13$, and thus $-\beta_{X}-\frac12\beta_{\mathfrak{m}_{+}-1}+\beta_{\mathfrak{m}_{+}} \leq -\frac{1}{12}+3\e_{X,3}$.
\end{proof}
\begin{corollary}\label{corollary:D1B2B}
Fix deterministic data $\{\beta_{m}^{\sim}\}_{m=0}^{\infty}$ and $\{\mathfrak{t}_{\mathrm{av},m}^{\sim}\}_{m=1}^{\infty}$ and $\{\mathfrak{t}_{N,\e,m}^{\sim}\}_{m=1}^{\infty}$ and $\mathfrak{m}_{+}^{\sim} \in \Z_{>0}$ defined as follows.
\begin{itemize}[leftmargin=*]
\item Define $\beta_{0}^{\sim} \overset{\bullet}= 0$ and define $\beta_{m}^{\sim} \overset{\bullet}= \beta_{m-1}^{\sim} + \e_{X,2}$ for $\e_{X,2} \in \R_{>0}$ arbitrarily small but universal and $m \in \Z_{\geq1}$.
\item Define $\mathfrak{t}_{\mathrm{av},m}^{\sim} \overset{\bullet}= N^{-\frac54- \frac12 \beta_{m-1}^{\sim} + \beta_{m}^{\sim} + \e_{X,3}}$ for $m \in \Z_{\geq 1}$ and $\e_{X,3}>0$ arbitrarily small but uniformly bounded below and chosen so $\mathfrak{t}_{\mathrm{av},m+1}^{\sim}$ is an integer multiple of $\mathfrak{t}_{\mathrm{av},m}^{\sim}$. The integer-multiple-constraint will only be useful for convenience in \emph{Section \ref{section:KPZ2}}.
\item Take $\mathfrak{t}_{N,\e,m}^{\sim}\geq0$ to be any deterministic time satisfying $0\leq\mathfrak{t}_{N,\e,m}^{\sim}\lesssim N^{\e}\mathfrak{t}_{\mathrm{av},m}^{\sim}$ with $\e>0$ sufficiently small but universal. 
\item Define $\mathfrak{m}_{+}^{\sim} \in \Z_{\geq 0}$ as the smallest positive integer for which $\beta_{\mathfrak{m}_{+}^{\sim}}^{\sim} > \beta_{X} + 3\e_{X,2}$.
\end{itemize}
For some universal constant $\beta_{\mathrm{univ}} \in \R_{>0}$ and for any deterministic $0\leq\mathfrak{t}_{\mathfrak{s}}\leq2$, we have the estimate
\small\begin{align}
{\sup}_{m=1,\ldots,\mathfrak{m}_{+}^{\sim}}{\sup}_{\mathfrak{i}=1,2}{\sup}_{\mathfrak{l}=1,\ldots,N^{\beta_{X}}}\E\|\bar{\mathbf{H}}_{T,x}^{N}(N^{\beta_{X}}|\mathscr{C}_{\mathfrak{t}_{\mathrm{av},m}^{\sim};\beta_{m-1}^{\sim},\beta_{m}^{\sim}}^{\mathbf{T},+,\mathfrak{t}_{N,\e,m}^{\sim},\mathfrak{i}}(\wt{\mathfrak{g}}^{\mathfrak{l}}_{S+\mathfrak{t}_{\mathfrak{s}},y})|)\|_{2;\mathbb{T}_{N}} \ &\lesssim \ N^{-\beta_{\mathrm{univ}}}.
\end{align}\normalsize\normalsize
We have $\mathfrak{m}_{+}^{\sim}\lesssim1$ and $1 \leq \mathfrak{t}_{\mathrm{av},m}^{\sim}(\mathfrak{t}_{\mathrm{av},m-1}^{\sim})^{-1} \leq N^{\e_{X,2}}$ and $\mathfrak{t}_{\mathrm{av},\mathfrak{m}_{+}^{\sim}}^{\sim} \lesssim N^{-\frac{13}{12}+\e_{X,4}}$ for $\e_{X,4}>0$ arbitrarily small but universal.
\end{corollary}
\begin{proof}
This follows by Proposition \ref{prop:D1B2B} how Corollary \ref{corollary:D1B2A} follows by Proposition \ref{prop:D1B2A}.
 \end{proof}
With Corollary \ref{corollary:D1B2A}, we may control time-averages of $\mathscr{C}^{\mathbf{X},-}$ with respect to the ``maximal" time-scale $\mathfrak{t}_{\mathrm{av},\mathfrak{m}_{+}}\in\R_{\geq0}$ defined in the statement of Corollary \ref{corollary:D1B2A} by the multiscale procedure outlined in the would-be-proof of Pseudo-Proposition \ref{pprop:S4}. Thus, we replace the original spatial-average-with-cutoff $\mathscr{C}^{\mathbf{X},-}$ with its time-average on this time-scale $\mathfrak{t}_{\mathrm{av},\mathfrak{m}_{+}}\in\R_{\geq0}$ via step-by-step replacement as in the would-be-proof of Pseudo-Proposition \ref{pprop:S3}. We then estimate the errors with the following corollary.
\begin{corollary}\label{corollary:D1B1A}
 Consider a sequence $\{\mathfrak{t}_{\mathrm{av},\ell}\}_{\ell=0}^{\infty}$ and $\mathfrak{l}_{+}\in\Z_{\geq0}$ defined below, where $\e_{X,2}>0$ is arbitrarily small but universal.
\begin{itemize}[leftmargin=*]
\item Define $\mathfrak{t}_{\mathrm{av},\ell+1} = \lfloor N^{\e_{X,2}}\rfloor\mathfrak{t}_{\mathrm{av},\ell}\wedge\mathfrak{t}_{\mathrm{av},\mathfrak{m}_{+}}$ with $\mathfrak{t}_{\mathrm{av},\mathfrak{m}_{+}}$ in \emph{Corollary \ref{corollary:D1B2A}}. Let $\mathfrak{l}_{+}$ be the first index $\ell$ for which $\mathfrak{t}_{\mathrm{av},\ell}=\mathfrak{t}_{\mathrm{av},\mathfrak{m}_{+}}$.
\item Choose $\mathfrak{t}_{\mathrm{av},0}$ so $N^{-2}\lesssim\mathfrak{t}_{\mathrm{av},0}\lesssim N^{-2+\e_{X,2}}$ and $\mathfrak{t}_{\mathrm{av},\ell+1}\mathfrak{t}_{\mathrm{av},\ell}^{-1}\in\Z_{>0}$ for all $\ell$. Our choice of $\mathfrak{t}_{\mathrm{av},0}$ depends on $\e_{X,2},\mathfrak{t}_{\mathrm{av},\mathfrak{m}_{+}}$.
\end{itemize}
There exists $\beta_{\mathrm{univ}} \in \R_{>0}$ universal so that if $\e_{X,2} \in \R_{>0}$ is sufficiently small, we have uniform boundedness $\mathfrak{l}_{+} \lesssim_{\e_{X,2}} 1$ and
\small\begin{align}
{\sup}_{\ell=0,\ldots,\mathfrak{l}_{+}} \mathfrak{t}_{\mathrm{av},\ell+1}^{1/4} \E\|\bar{\mathbf{H}}_{T,x}^{N}(N^{\frac12}|\mathscr{A}_{\mathfrak{t}_{\mathrm{av},\ell}}^{\mathbf{T},+}\mathscr{C}_{N^{\beta_{X}}}^{\mathbf{X},-}(\mathfrak{g}_{S,y})|)\|_{1;\mathbb{T}_{N}} \ &\lesssim \ N^{-\beta_{\mathrm{univ}}}.
\end{align}\normalsize\normalsize
\end{corollary}
\begin{proof}
If we pick $\e_{X,2} \in \R_{>0}$ sufficiently small, we may replace the prefactor $\mathfrak{t}_{\mathrm{av},\ell+1}^{1/4} \to \mathfrak{t}_{\mathrm{av},\ell}^{1/4}$ for the $\ell$-quantity in the supremum while changing $\beta_{\mathrm{univ}}\in\R_{>0}$ on the RHS of the proposed estimate by a factor of $1/2$. At this point, we observe that all $\mathfrak{t}_{\mathrm{av},\ell} \in \R_{>0}$ in question satisfy constraints assumed for $\mathfrak{t}_{\mathrm{av}} \in \R_{>0}$ in Proposition \ref{prop:D1B1A}. Indeed, the times $\mathfrak{t}_{\mathrm{av},0},\mathfrak{t}_{\mathrm{av},\mathfrak{l}_{+}} \in \R_{>0}$ with maximal and minimal indices do by definition and by Corollary \ref{corollary:D1B2A}, and the times $\mathfrak{t}_{\mathrm{av},\ell} \in \R_{>0}$ are monotone in the index. Thus, the estimate on the sup follows by Proposition \ref{prop:D1B1A}. The bound $\mathfrak{l}_{+} \lesssim_{\e_{X,2}} 1$ follows by $\lfloor N^{\e_{X,2}}\rfloor\gtrsim N^{\e_{X,2}/2}$ and elementary considerations.
\end{proof}
\begin{remark}
Like Corollary \ref{corollary:D1B2A}, the integer-multiple-constraint for time-scales in Corollary \ref{corollary:D1B1A} is for later convenience.
\end{remark}
The next/last result follows by Proposition \ref{prop:D1B1B} how Corollary \ref{corollary:D1B1A} follows by Proposition \ref{prop:D1B1A}.
\begin{corollary}\label{corollary:D1B1B}
 Consider the sequence $\{\mathfrak{t}_{\mathrm{av},\ell}^{\sim}\}_{\ell=0}^{\infty}$ and $\mathfrak{l}_{+}^{\sim}\in\Z_{\geq0}$ defined below; $\e_{X,2}>0$ is arbitrarily small but universal.
\begin{itemize}[leftmargin=*]
\item Define $\mathfrak{t}_{\mathrm{av},\ell+1}^{\sim} = \lfloor N^{\e_{X,2}}\rfloor\mathfrak{t}_{\mathrm{av},\ell}^{\sim}\wedge\mathfrak{t}_{\mathrm{av},\mathfrak{m}_{+}^{\sim}}^{\sim}$ with $\mathfrak{t}_{\mathrm{av},\mathfrak{m}_{+}^{\sim}}^{\sim}$ from \emph{Corollary \ref{corollary:D1B2B}}. Let $\mathfrak{l}_{+}^{\sim}$ be the first index $\ell$ for which $\mathfrak{t}_{\mathrm{av},\ell}^{\sim}=\mathfrak{t}_{\mathrm{av},\mathfrak{m}_{+}^{\sim}}^{\sim}$.
\item Choose $\mathfrak{t}_{\mathrm{av},0}^{\sim}$ such that $N^{-2}\lesssim\mathfrak{t}_{\mathrm{av},0}^{\sim}\lesssim N^{-2+\e_{X,2}}$ and $\mathfrak{t}_{\mathrm{av},\ell+1}^{\sim}(\mathfrak{t}_{\mathrm{av},\ell}^{\sim})^{-1}\in\Z_{>0}$. Our choice of $\mathfrak{t}_{\mathrm{av},0}^{\sim}$ depends on $\e_{X,2},\mathfrak{t}_{\mathrm{av},\mathfrak{m}_{+}^{\sim}}^{\sim}$.
\end{itemize}
There exists $\beta_{\mathrm{univ}} \in \R_{>0}$ universal so that if $\e_{X,2} \in \R_{>0}$ is sufficiently small, we have uniform boundedness $\mathfrak{l}_{+}^{\sim} \lesssim_{\e_{X,2}} 1$ and
\small\begin{align}
{\sup}_{\ell=0,\ldots,\mathfrak{l}_{+}^{\sim}}{\sup}_{\mathfrak{l}=1,\ldots,N^{\beta_{X}}} (\mathfrak{t}_{\mathrm{av},\ell+1}^{\sim})^{1/4} \E\|\bar{\mathbf{H}}_{T,x}^{N}(N^{\beta_{X}}|\mathscr{A}_{\mathfrak{t}_{\mathrm{av},\ell}^{\sim}}^{\mathbf{T},+}(\wt{\mathfrak{g}}_{S,y}^{\mathfrak{l}})|)\|_{1;\mathbb{T}_{N}} \ &\lesssim \ N^{-\beta_{\mathrm{univ}}}.
\end{align}\normalsize\normalsize
\end{corollary}
The rest of this section is dedicated to proofs of Propositions \ref{prop:S1B}, \ref{prop:D1B1A}, \ref{prop:D1B2A}, \ref{prop:D1B1B}, and \ref{prop:D1B2B}. We again invite the reader to go to Section \ref{section:KPZ1} at least upon a first reading as the proofs of the aforementioned results are on the technical side. For proofs themselves, we replace mesoscopic space-time averages by ``local" versions by coupling arguments, and then we analyze these ``local" versions by a reduction to equilibrium estimates in Lemmas \ref{lemma:KV}, \ref{lemma:SpectralH-1}, and \ref{lemma:H-1SpectralPGF}. As for Proposition \ref{prop:S1B}, we will do the same but instead of Lemma \ref{lemma:KV}, Lemma \ref{lemma:SpectralH-1}, and Lemma \ref{lemma:H-1SpectralPGF} we will apply static LDP estimates in Lemma \ref{lemma:LDP} and Corollary \ref{corollary:LDP}. Every result will also require some technical gymnastics we explain below; see Lemma \ref{lemma:D1B2A1}, for example. We will \emph{not} present the proofs in the written order. We start with proofs of Proposition \ref{prop:D1B2A} and Proposition \ref{prop:D1B2B}; these proofs contain almost all ideas necessary to get the rest. We then move to proofs for Propositions \ref{prop:D1B1A} and \ref{prop:D1B1B}. We finish with a proof of Proposition \ref{prop:S1B}.
\subsection{Proof of Proposition \ref{prop:D1B2A}}
We will adopt notation from the statement of Proposition \ref{prop:D1B2A} throughout this subsection. We delay the proofs of preliminary lemmas until the last subsection of this section to avoid obscuring the proof of Proposition \ref{prop:D1B2A}.

We will depend on several preliminary ingredients starting with aforementioned technical gymnastics. Roughly speaking, it allows us to replace the heat operator within the LHS of the proposed estimate in Proposition \ref{prop:D1B2A} with averaging against the constant function on $[0,\mathfrak{t}_{\mathfrak{f}}]\times\mathbb{T}_{N}$ free of short-time singularities of the heat kernel. We do this by ``direct" convolution bound.
\begin{lemma}\label{lemma:D1B2A1}
 Take any possibly random function $\varphi_{S,y}:\R_{\geq0}\times\mathbb{T}_{N}\to\R$ and suppose $\|\varphi\|_{\infty;\mathbb{T}_{N}}\lesssim N^{-\beta_{-}}$ with a universal implied constant. Given any $\e>0$ and $\mathfrak{t}_{\mathfrak{s}},\mathfrak{t}_{\mathfrak{f}}\geq0$, we have the following deterministic estimate with universal implied constant:
\small\begin{align}
\|\bar{\mathbf{H}}_{T,x}^{N}(N^{\frac12}|\varphi_{S+\mathfrak{t}_{\mathfrak{s}},y}|)\|_{\mathfrak{t}_{\mathfrak{f}};\mathbb{T}_{N}} \ &\lesssim \ N^{1+\e_{\mathrm{cpt}}+\frac12\e-\frac12\beta_{-}}\int_{0}^{\mathfrak{t}_{\mathfrak{f}}}\wt{\sum}_{y\in\mathbb{T}_{N}}|\varphi_{S+\mathfrak{t}_{\mathfrak{s}},y}| \ \d S \ + \ N^{-\e}.
\end{align}\normalsize\normalsize
\end{lemma}
From this last gymnastics-estimate, we get the following. It roughly follows by taking expectation of the RHS in the bound of Lemma \ref{lemma:D1B2A1} with $\varphi$ the iterated cutoff $\mathscr{C}^{\mathbf{T},+}\mathscr{C}^{\mathbf{X},-}$ inside the heat operator on the LHS of the proposed estimate in Proposition \ref{prop:D1B2A}. The iterated expectation on the RHS of the proposed estimate in this next result Lemma \ref{lemma:D1B2A2} follows by decomposing the expectation of a dynamic functional in terms of an expectation over the path-space measure and then an expectation over the initial configuration, the latter of which we then average over the integral-sum of the RHS of the estimate in Lemma \ref{lemma:D1B2A1}.
\begin{lemma}\label{lemma:D1B2A2}
 Recall \emph{Definition \ref{definition:AvMeasures}}. Provided any $\e\in\R_{>0}$ and $\mathfrak{t}_{\mathfrak{f}}\in\R_{\geq0}$ and $\mathfrak{i}\in\{1,2\}$, we have, with notation explained after,
\small\begin{align}
 \E\|\bar{\mathbf{H}}_{T,x}^{N}(N^{\frac12}|\mathscr{C}_{\mathfrak{t}_{\mathrm{av}};\beta_{-},\beta_{+}}^{\mathbf{T},+,\mathfrak{t}_{N,\e},\mathfrak{i}}\mathscr{C}_{N^{\beta_{X}}}^{\mathbf{X},-}(\mathfrak{g}_{S+\mathfrak{t}_{\mathfrak{s}},y})|)\|_{\mathfrak{t}_{\mathfrak{f}};\mathbb{T}_{N}} \ &\lesssim \ N^{1+\e_{\mathrm{cpt}}+\frac12\e-\frac12\beta_{-}}\mathfrak{t}_{\mathfrak{f}} \E^{\mu_{0,\Z}}\bar{\mathfrak{f}}_{\mathfrak{t}_{\mathfrak{s}},\mathfrak{t}_{\mathfrak{f}},\mathbb{T}_{N}}\E_{\eta_{0}}^{\mathrm{path}}|\mathscr{C}_{\mathfrak{t}_{\mathrm{av}};\beta_{-},\beta_{+}}^{\mathbf{T},+,\mathfrak{t}_{N,\e},\mathfrak{i}}\mathscr{C}_{N^{\beta_{X}}}^{\mathbf{X},-}(\mathfrak{g}_{0,0})| \ + \ N^{-\e}. \nonumber
\end{align}\normalsize
Above the expectation $\E_{\eta_{0}}^{\mathrm{path}}$ is an expectation with respect to the path-space measure induced by the original $\Omega_{\Z}$-valued exclusion process with an initial configuration $\eta_{0}$ that is then being averaged via $\E^{\mu_{0,\Z}}$ against the space-time averaged density $\bar{\mathfrak{f}}_{\mathfrak{t}_{\mathfrak{s}},\mathfrak{t}_{\mathfrak{f}},\mathbb{T}_{N}}$.
\end{lemma}
We observe the path-space expectation of the iterated cutoff $\mathscr{C}^{\mathbf{T},+}\mathscr{C}^{\mathbf{X},-}$ is a \emph{global} statistic. It depends on the entire configuration $\eta_{0}\in\Omega_{\Z}$. However, because we integrate the mesoscopic spatial-cutoff $\mathscr{C}^{\mathbf{X},-}$ on a mesoscopic time-scale $\mathfrak{t}_{N,\e}\in\R_{\geq0}$, it ``looks like" a mesoscopic statistic. The next ingredient localizes these space-time averages of interest in Lemma \ref{lemma:D1B2A2}. To state it, we require some notation. First let us recall the $\mathfrak{I}$-local exclusion processes introduced in Definition \ref{definition:PeriodicGenerator}.
\begin{definition}\label{definition:tAvBlock}
For $\mathfrak{t}_{\mathrm{av}} \in \R_{\geq0}$, we set 
\begin{align}
\mathfrak{l}_{\mathfrak{t}_{\mathrm{av}}} \overset{\bullet}= (N(2\mathfrak{t}_{\mathrm{av}}+2\mathfrak{t}_{N,\e})^{\frac12} + N^{\frac32}(2\mathfrak{t}_{\mathrm{av}}+2\mathfrak{t}_{N,\e}) + N^{\beta_{X}})\log^{100}N \quad\mathrm{and}\quad \mathfrak{I}_{\mathfrak{t}_{\mathrm{av}}} \overset{\bullet}= \llbracket-\mathfrak{l}_{\mathfrak{t}_{\mathrm{av}}},\mathfrak{l}_{\mathfrak{t}_{\mathrm{av}}}\rrbracket \subseteq \Z.
\end{align}\normalsize
\end{definition}
\begin{remark}\label{remark:ch2SoP}
Consider the particle random walk for the $\mathfrak{S}^{N,!!}$-process on $\Omega_{\Z}$, and suppose it starts in the ``middle" $N^{\beta_{X}}$ of $\mathfrak{I}_{\mathfrak{t}_{\mathrm{av}}}$. The probability this random walk exits $\mathfrak{I}_{\mathfrak{t}_{\mathrm{av}}}$ before time $2\mathfrak{t}_{\mathrm{av}}+2\mathfrak{t}_{N,\e}$ is order $\exp(-\log^{100}N)\lesssim_{C}N^{-C}$ for any $C\in\R_{>0}$. To prove this, we may first remove the drift because $|\mathfrak{I}_{\mathfrak{t}_{\mathrm{av}}}|$ is much larger than the drift on time-scale $2\mathfrak{t}_{\mathrm{av}}+2\mathfrak{t}_{N,\e}$. For the symmetric part of the particle random walk, we use its sub-Gaussian property and a Doob maximal inequality to control its maximal process.
\end{remark}
The following construction is a localization of space-time averages evaluated on $\mathfrak{I}_{\mathfrak{t}_{\mathrm{av}}}$-exclusion processes. Everything here is basically Definition \ref{definition:S1B}, Definition \ref{definition:D1B1A}, and Definition \ref{definition:D1B2A} but for $\mathfrak{I}_{\mathfrak{t}_{\mathrm{av}}}$-local processes; it is blind to anything outside $\mathfrak{I}_{\mathfrak{t}_{\mathrm{av}}}\subseteq\Z$.
\begin{definition}\label{definition:averagesAlocal}
Consider any $\mathfrak{t}_{\mathrm{av}}\in\R_{\geq0}$. We define $\eta_{\mathfrak{t}_{\mathrm{av}};\bullet}$ as the $\Omega_{\mathfrak{I}_{\mathfrak{t}_{\mathrm{av}}}}$-valued $\mathfrak{I}_{\mathfrak{t}_{\mathrm{av}}}$-local exclusion process defined in Definition \ref{definition:PeriodicGenerator} with $\mathfrak{I}_{\mathfrak{t}_{\mathrm{av}}}\subseteq\Z$ defined earlier and the initial configuration $\eta_{\mathfrak{t}_{\mathrm{av}};0}= \Pi_{\mathfrak{I}_{\mathfrak{t}_{\mathrm{av}}}}\eta_{0} \in \Omega_{\mathfrak{I}_{\mathfrak{t}_{\mathrm{av}}}}$, namely the projection on $\mathfrak{I}_{\mathfrak{t}_{\mathrm{av}}}\subseteq\Z$ of the initial configuration of the \emph{original} $\Omega_{\Z}$-exclusion process. We now define local versions of Definitions \ref{definition:S1B}, \ref{definition:D1B1A}, and \ref{definition:D1B2A}.
\begin{itemize}[leftmargin=*]
\item Define $\mathfrak{g}_{\mathfrak{t}_{\mathrm{av}};S,0} \overset{\bullet}= \mathfrak{g}_{0,0}(\eta_{\mathfrak{t}_{\mathrm{av}};S})$ as the pseudo-gradient in Proposition \ref{prop:Duhamel} but evaluated at the $\mathfrak{I}_{\mathfrak{t}_{\mathrm{av}}}$-local process $\eta_{\mathfrak{t}_{\mathrm{av}};\bullet}$ at $S\in\R_{\geq0}$. Observe this is well-defined because $\mathfrak{I}_{\mathfrak{t}_{\mathrm{av}}}$ contains the support of $\mathfrak{g}_{0,0}$ if $N\in\Z_{>0}$ is sufficiently large but universal.
\item Recall $\beta_{X}=\frac13+\e_{X,1}$ in Proposition \ref{prop:Duhamel}. We define the event $\mathbf{1}[\mathscr{E}_{\mathfrak{t}_{\mathrm{av}};S}^{\mathbf{X},-}] = \mathbf{1}[|\mathscr{A}_{N^{\beta_{X}}}^{\mathbf{X},-}(\mathfrak{g}_{\mathfrak{t}_{\mathrm{av}};S,0})|\lesssim N^{-\frac12\beta_{X}+\frac12\e_{X,1}}]$ where
\small\begin{align}
\mathscr{A}_{N^{\beta_{X}}}^{\mathbf{X},-}(\mathfrak{g}_{\mathfrak{t}_{\mathrm{av}};S,0}) \ &= \ \wt{\sum}_{\mathfrak{l}=1,\ldots,N^{\beta_{X}}}\tau_{-7\mathfrak{l}\mathfrak{m}}\mathfrak{g}_{\mathfrak{t}_{\mathrm{av}};S,0}.
\end{align}\normalsize\normalsize
We clarify that $\tau_{-7\mathfrak{l}\mathfrak{m}}\mathfrak{g}_{\mathfrak{t}_{\mathrm{av}};S,0}$ corresponds to shifting the $\eta_{\mathfrak{t}_{\mathrm{av}};S}$-configuration by $-7\mathfrak{l}\mathfrak{m}$, which yields a particle configuration in $\Omega_{\mathfrak{I}_{\mathfrak{t}_{\mathrm{av}}}-7\mathfrak{l}\mathfrak{m}}$, and then evaluating $\mathfrak{g}_{0,0}$ at the particle configuration obtained by projecting the shifted configuration in $\Omega_{\mathfrak{I}_{\mathfrak{t}_{\mathrm{av}}}-7\mathfrak{l}\mathfrak{m}}$ onto the support of $\mathfrak{g}_{0,0}$ which is a \emph{fixed} subset in $\Z$. Indeed, observe the support of $\mathfrak{g}_{0,0}$ is contained properly in $\mathfrak{I}_{\mathfrak{t}_{\mathrm{av}}}-7\mathfrak{l}\mathfrak{m}\subseteq\Z$ for all $\mathfrak{l}\in\llbracket0,N^{\beta_{X}}\rrbracket$ for $N\in\Z_{>0}$ sufficiently large depending only on $\mathfrak{m}\in\Z_{>0}$ because of definition of $\mathfrak{I}_{\mathfrak{t}_{\mathrm{av}}}\subseteq\Z$. We emphasize the supports of the summands defining the spatial-average in the last display are disjoint subsets of $\mathfrak{I}_{\mathfrak{t}_{\mathrm{av}}}\subseteq\Z$ as was the case in Proposition \ref{prop:Duhamel}. We additionally define the following cutoff of the previous spatial-average via CLT considerations:
\small\begin{align}
\mathscr{C}_{N^{\beta_{X}}}^{\mathbf{X},-}(\mathfrak{g}_{\mathfrak{t}_{\mathrm{av}};S,0}) \ &\overset{\bullet}= \ \mathscr{A}_{N^{\beta_{X}}}^{\mathbf{X},-}(\mathfrak{g}_{\mathfrak{t}_{\mathrm{av}};S,0}) \cdot \mathbf{1}[\mathscr{E}_{\mathfrak{t}_{\mathrm{av}};S}^{\mathbf{X},-}].
\end{align}\normalsize\normalsize
\item We first recall the time-scale $\mathfrak{t}_{\mathrm{av}}\in\R_{\geq0}$. Define the following time-averages acting on both spatial-averages and their cutoffs of the pseudo-gradient functional $\mathfrak{g}_{\mathfrak{t}_{\mathrm{av}};}$ from the first bullet point that is localized to the $\mathfrak{I}_{\mathfrak{t}_{\mathrm{av}}}$-local process:
\begin{subequations}
\small\begin{align}
\mathscr{A}_{\mathfrak{t}_{\mathrm{av}}}^{\mathbf{T},+}\mathscr{A}_{N^{\beta_{X}}}^{\mathbf{X},-}(\mathfrak{g}_{\mathfrak{t}_{\mathrm{av}};S,0}) \ &\overset{\bullet}= \ \mathbf{1}_{\mathfrak{t}_{\mathrm{av}}\neq0}\mathfrak{t}_{\mathrm{av}}^{-1}\int_{0}^{\mathfrak{t}_{\mathrm{av}}}\mathscr{A}_{N^{\beta_{X}}}^{\mathbf{X},-}(\mathfrak{g}_{\mathfrak{t}_{\mathrm{av}};S+\mathfrak{r},0})\d\mathfrak{r} \ + \ \mathbf{1}_{\mathfrak{t}_{\mathrm{av}}=0}\mathscr{A}_{N^{\beta_{X}}}^{\mathbf{X},-}(\mathfrak{g}_{\mathfrak{t}_{\mathrm{av}};S,0}) \\
\mathscr{A}_{\mathfrak{t}_{\mathrm{av}}}^{\mathbf{T},+}\mathscr{C}_{N^{\beta_{X}}}^{\mathbf{X},-}(\mathfrak{g}_{\mathfrak{t}_{\mathrm{av}};S,0}) \ &\overset{\bullet}= \ \mathbf{1}_{\mathfrak{t}_{\mathrm{av}}\neq0}\mathfrak{t}_{\mathrm{av}}^{-1}\int_{0}^{\mathfrak{t}_{\mathrm{av}}}\mathscr{C}_{N^{\beta_{X}}}^{\mathbf{X},-}(\mathfrak{g}_{\mathfrak{t}_{\mathrm{av}};S+\mathfrak{r},0})\d\mathfrak{r} \ + \ \mathbf{1}_{\mathfrak{t}_{\mathrm{av}}=0}\mathscr{C}_{N^{\beta_{X}}}^{\mathbf{X},-}(\mathfrak{g}_{\mathfrak{t}_{\mathrm{av}};S,0}).
\end{align}\normalsize\normalsize
\end{subequations}
The above space-time averages only care about what happens/are functionals along the $\eta_{\mathfrak{t}_{\mathrm{av}}}$-exclusion process. Let us define the following pair of events for maximal processes and time-averages. The same comment about maximal processes that we made in Definition \ref{definition:D1B2A} applies here upon replacing $\mathfrak{g}_{S,0}$ of the original $\Omega_{\Z}$-valued process by $\mathfrak{g}_{\mathfrak{t}_{\mathrm{av}};S,0}$ of the $\mathfrak{I}_{\mathfrak{t}_{\mathrm{av}}}$-local process:
\begin{subequations}
\small\begin{align}
\mathbf{1}[\mathscr{G}_{\mathfrak{t}_{\mathrm{av}};\lesssim;S,0}^{\beta_{-}}] \ &\overset{\bullet}= \ \mathbf{1}\left({\sup}_{0\leq\mathfrak{t}\leq\mathfrak{t}_{\mathrm{av}}}\mathfrak{t}\cdot\mathfrak{t}_{\mathrm{av}}^{-1}|\mathscr{A}_{\mathfrak{t}}^{\mathbf{T},+}\mathscr{C}_{N^{\beta_{X}}}^{\mathbf{X},-}(\mathfrak{g}_{\mathfrak{t}_{\mathrm{av}};S,0})|\lesssim N^{-\beta_{-}}\right) \\
\mathbf{1}[\mathscr{G}_{\mathfrak{t}_{\mathrm{av}};\geq;S,0}^{\beta_{+}}]  \ &\overset{\bullet}= \ \mathbf{1}\left({\sup}_{0\leq\mathfrak{t}\leq\mathfrak{t}_{\mathrm{av}}}\mathfrak{t}\cdot\mathfrak{t}_{\mathrm{av}}^{-1}|\mathscr{A}_{\mathfrak{t}}^{\mathbf{T},+}\mathscr{C}_{N^{\beta_{X}}}^{\mathbf{X},-}(\mathfrak{g}_{\mathfrak{t}_{\mathrm{av}};S,0})|\geq N^{-\beta_{+}}\right).
\end{align}\normalsize\normalsize
\end{subequations}
where the exponents $\beta_{-},\beta_{+}\in\R_{\geq0}$ are universal and satisfy $\beta_{-}\leq\beta_{+}$. Recalling $\mathfrak{t}_{N,\e}\geq0$ from the statement of Proposition \ref{prop:D1B2A}, we then define the following two-sided cutoffs of the previous $\mathfrak{I}_{\mathfrak{t}_{\mathrm{av}}}$-local space-time averages:
\begin{subequations}
\small\begin{align}
\mathscr{C}_{\mathfrak{t}_{\mathrm{av}};\beta_{-},\beta_{+}}^{\mathbf{T},+,\mathfrak{t}_{N,\e},1}\mathscr{C}_{N^{\beta_{X}}}^{\mathbf{X},-}(\mathfrak{g}_{\mathfrak{t}_{\mathrm{av}};S,0}) \ &\overset{\bullet}= \ \mathscr{A}_{\mathfrak{t}_{\mathrm{av}}}^{\mathbf{T},+}\mathscr{C}_{N^{\beta_{X}}}^{\mathbf{X},-}(\mathfrak{g}_{\mathfrak{t}_{\mathrm{av}};S+\mathfrak{t}_{N,\e},0})  \cdot \mathbf{1}[\mathscr{G}_{\mathfrak{t}_{\mathrm{av}};\lesssim;S+\mathfrak{t}_{N,\e},0}^{\beta_{-}}]\mathbf{1}[\mathscr{G}_{\mathfrak{t}_{\mathrm{av}};\geq;S,0}^{\beta_{+}}] \\
\mathscr{C}_{\mathfrak{t}_{\mathrm{av}};\beta_{-},\beta_{+}}^{\mathbf{T},+,\mathfrak{t}_{N,\e},2}\mathscr{C}_{N^{\beta_{X}}}^{\mathbf{X},-}(\mathfrak{g}_{\mathfrak{t}_{\mathrm{av}};S,0}) \ &\overset{\bullet}= \  \mathscr{A}_{\mathfrak{t}_{\mathrm{av}}}^{\mathbf{T},+}\mathscr{C}_{N^{\beta_{X}}}^{\mathbf{X},-}(\mathfrak{g}_{\mathfrak{t}_{\mathrm{av}};S,0})  \cdot \mathbf{1}[\mathscr{G}_{\mathfrak{t}_{\mathrm{av}};\lesssim;S,0}^{\beta_{-}}]\mathbf{1}[\mathscr{G}_{\mathfrak{t}_{\mathrm{av}};\geq;S+\mathfrak{t}_{N,\e},0}^{\beta_{+}}].
\end{align}\normalsize\normalsize
\end{subequations}
The clarifying remarks we made for cutoffs of space-time averages in Definition \ref{definition:D1B2A} hold for the above $\mathfrak{I}_{\mathfrak{t}_{\mathrm{av}}}$-local averages.
\end{itemize}
\end{definition}
By speed-of-propagation estimates and a coupling very similar to that used in the proof of Lemma \ref{lemma:EProdSoP} we get the following replacement of mesoscopic dynamic functionals of the original $\Omega_{\Z}$-valued exclusion process with their local versions.
\begin{lemma}\label{lemma:SoP}
 We have the following for $C \in \R_{>0}$ arbitrarily large but universal, $\mathfrak{i}\in\{1,2\}$, and any deterministic $\mathfrak{t}_{\mathrm{av}} \lesssim N^{10}$:
\small\begin{align}
\E_{\eta_{0}}^{\mathrm{path}}|\mathscr{A}_{\mathfrak{t}_{\mathrm{av}}}^{\mathbf{T},+}\mathscr{C}_{N^{\beta_{X}}}^{\mathbf{X},-}(\mathfrak{g}_{0,0})|^{2} \ &= \ \E_{\Pi_{\mathfrak{I}_{\mathfrak{t}_{\mathrm{av}}}}\eta_{0}}^{\mathrm{loc}}|\mathscr{A}_{\mathfrak{t}_{\mathrm{av}}}^{\mathbf{T},+}\mathscr{C}_{N^{\beta_{X}}}^{\mathbf{X},-}(\mathfrak{g}_{\mathfrak{t}_{\mathrm{av}};0,0})|^{2} \ + \ \mathscr{O}_{C}(N^{-C}) \label{eq:SoP1} \\
\E_{\eta_{0}}^{\mathrm{path}}|\mathscr{C}_{\mathfrak{t}_{\mathrm{av}};\beta_{-},\beta_{+}}^{\mathbf{T},+,\mathfrak{t}_{N,\e},\mathfrak{i}}\mathscr{C}_{N^{\beta_{X}}}^{\mathbf{X},-}(\mathfrak{g}_{0,0})| \ &= \ \E_{\Pi_{\mathfrak{I}_{\mathfrak{t}_{\mathrm{av}}}}\eta_{0}}^{\mathrm{loc}}|\mathscr{C}_{\mathfrak{t}_{\mathrm{av}};\beta_{-},\beta_{+}}^{\mathbf{T},+,\mathfrak{t}_{N,\e},\mathfrak{i}}\mathscr{C}_{N^{\beta_{X}}}^{\mathbf{X},-}(\mathfrak{g}_{\mathfrak{t}_{\mathrm{av}};0,0})| \ + \ \mathscr{O}_{C}(N^{-C}). \label{eq:SoP2}
\end{align}\normalsize\normalsize
%
\begin{itemize}[leftmargin=*]
\item The expectation $\E_{\eta_{0}}^{\mathrm{path}}$ is with respect to the path-space measure induced by the $\mathfrak{S}^{N,!!}$-dynamic on $\Omega_{\Z}$ with initial spins $\eta_{0}$.
\item The expectation $\E_{\Pi_{\mathfrak{I}_{\mathfrak{t}_{\mathrm{av}}}}\eta_{0}}^{\mathrm{loc}}$ is with respect to the path-space measure of the $\mathfrak{I}_{\mathfrak{t}_{\mathrm{av}}}$-process on $\Omega_{\mathfrak{I}_{\mathfrak{t}_{\mathrm{av}}}}$ with initial spins $\Pi_{\mathfrak{I}_{\mathfrak{t}_{\mathrm{av}}}}\eta_{0}$.
\item The functionals inside expectations depend only on a finite time-interval, but the two expectations are defined for infinite time.
\end{itemize}
\end{lemma}
We take advantage of localization via local equilibrium Lemma \ref{lemma:LE} and equilibrium calculations in Lemmas \ref{lemma:KV} and \ref{lemma:SpectralH-1}.
\begin{lemma}\label{lemma:D1B2A3}
 With a universal implied constant, we have the following estimate in which $\e_{X}=\max_{i}\e_{X,i}$:
\small\begin{align}
N^{1+\e_{\mathrm{cpt}}+\frac12\e-\frac12\beta_{-}}\mathfrak{t}_{\mathfrak{f}} \E^{\mu_{0,\Z}}\bar{\mathfrak{f}}_{\mathfrak{t}_{\mathfrak{s}},\mathfrak{t}_{\mathfrak{f}},\mathbb{T}_{N}}^{\mathfrak{I}_{\mathfrak{t}_{\mathrm{av}}}} \E_{\Pi_{\mathfrak{I}_{\mathfrak{t}_{\mathrm{av}}}}\eta_{0}}^{\mathrm{loc}}|\mathscr{C}_{\mathfrak{t}_{\mathrm{av}};\beta_{-},\beta_{+}}^{\mathbf{T},+,\mathfrak{t}_{N,\e},\mathfrak{i}}\mathscr{C}_{N^{\beta_{X}}}^{\mathbf{X},-}(\mathfrak{g}_{\mathfrak{t}_{\mathrm{av}};0,0})| \ &\lesssim \ N^{-\frac34\beta_{-}+10\e_{X}}+N^{-\frac12+10\e_{X}}+N^{-\beta_{\mathrm{univ}}}\mathfrak{t}_{\mathfrak{f}}.
\end{align}\normalsize\normalsize
\end{lemma}
\begin{proof}[Proof of \emph{Proposition \ref{prop:D1B2A}}]
We use Lemma \ref{lemma:D1B2A2} to get the first line and the second bound in Lemma \ref{lemma:SoP} to get the second line:
\small\begin{align}
 \E\|\bar{\mathbf{H}}_{T,x}^{N}(N^{\frac12}|\mathscr{C}_{\mathfrak{t}_{\mathrm{av}};\beta_{-},\beta_{+}}^{\mathbf{T},+,\mathfrak{t}_{N,\e},\mathfrak{i}}\mathscr{C}_{N^{\beta_{X}}}^{\mathbf{X},-}(\mathfrak{g}_{S+\mathfrak{t}_{\mathfrak{s}},y})|)\|_{\mathfrak{t}_{\mathfrak{f}};\mathbb{T}_{N}} \ &\lesssim \ N^{1+\e_{\mathrm{cpt}}+\frac12\e-\frac12\beta_{-}}\mathfrak{t}_{\mathfrak{f}} \E^{\mu_{0,\Z}}\bar{\mathfrak{f}}_{\mathfrak{t}_{\mathfrak{s}},\mathfrak{t}_{\mathfrak{f}},\mathbb{T}_{N}} \E_{\eta_{0}}^{\mathrm{path}}|\mathscr{C}_{\mathfrak{t}_{\mathrm{av}};\beta_{-},\beta_{+}}^{\mathbf{T},+,\mathfrak{t}_{N,\e},\mathfrak{i}}\mathscr{C}_{N^{\beta_{X}}}^{\mathbf{X},-}(\mathfrak{g}_{0,0})| \ + \ N^{-\e} \nonumber\\
 &\lesssim \ N^{1+\e_{\mathrm{cpt}}+\frac12\e-\frac12\beta_{-}}\mathfrak{t}_{\mathfrak{f}} \E^{\mu_{0,\Z}}\bar{\mathfrak{f}}_{\mathfrak{t}_{\mathfrak{s}},\mathfrak{t}_{\mathfrak{f}},\mathbb{T}_{N}}\E_{\Pi_{\mathfrak{I}_{\mathfrak{t}_{\mathrm{av}}}}\eta_{0}}^{\mathrm{loc}}|\mathscr{C}_{\mathfrak{t}_{\mathrm{av}};\beta_{-},\beta_{+}}^{\mathbf{T},+,\mathfrak{t}_{N,\e},\mathfrak{i}}\mathscr{C}_{N^{\beta_{X}}}^{\mathbf{X},-}(\mathfrak{g}_{\mathfrak{t}_{\mathrm{av}};0,0})| \ + \ N^{-\e} \nonumber \\
 &= \ N^{1+\e_{\mathrm{cpt}}+\frac12\e-\frac12\beta_{-}}\mathfrak{t}_{\mathfrak{f}} \E^{\mu_{0,\Z}}\bar{\mathfrak{f}}_{\mathfrak{t}_{\mathfrak{s}},\mathfrak{t}_{\mathfrak{f}},\mathbb{T}_{N}}^{\mathfrak{I}_{\mathfrak{t}_{\mathrm{av}}}} \E_{\Pi_{\mathfrak{I}_{\mathfrak{t}_{\mathrm{av}}}}\eta_{0}}^{\mathrm{loc}}|\mathscr{C}_{\mathfrak{t}_{\mathrm{av}};\beta_{-},\beta_{+}}^{\mathbf{T},+,\mathfrak{t}_{N,\e},\mathfrak{i}}\mathscr{C}_{N^{\beta_{X}}}^{\mathbf{X},-}(\mathfrak{g}_{\mathfrak{t}_{\mathrm{av}};0,0})| \ + \ N^{-\e}. \label{eq:D1B2A1}
\end{align}\normalsize\normalsize
The last line above follows by realizing that after the replacement-by-local expectation via Lemma \ref{lemma:SoP}, the remaining inner expectation depends only on the particle configuration in the subset $\mathfrak{I}_{\mathfrak{t}_{\mathrm{av}}}\subseteq\Z$. In particular, we may project the density $\bar{\mathfrak{f}}$ from the line before onto this block in the outer expectation. It then suffices to apply Lemma \ref{lemma:D1B2A3} to the first term in \eqref{eq:D1B2A1}.
\end{proof}
\subsection{Proof of Proposition \ref{prop:D1B2B}}
Proof of Proposition \ref{prop:D1B2B} follows by almost identical considerations as the proof of Proposition \ref{prop:D1B2A}. We list here a few differences. First we note Proposition \ref{prop:D1B2B} is concerned with order $N^{\beta_{X}}$ terms as opposed to order $N^{1/2}$ terms in Proposition \ref{prop:D1B2A}. This changes the relevant power-counting for the proof of Proposition \ref{prop:D1B2B} ``in our favor". Second, the $\wt{\mathfrak{g}}^{\mathfrak{l}}$-terms in Proposition \ref{prop:D1B2B} only have time-averages with cutoff; they are not spatially-averaged via $\mathscr{A}^{\mathbf{X},-}$ operators. This works ``against our favor" as we cannot leverage improved a priori estimates that we depended on in the proof of Proposition \ref{prop:D1B2A}. These two changes cooperate well enough for our power-counting to give us the upper bound in the proposed estimate. We clarify that in our local equilibrium calculations for time-averages, the support of $\wt{\mathfrak{g}}^{\mathfrak{l}}$-terms grows too fast to only apply Lemma \ref{lemma:SpectralH-1} to estimate Sobolev norms. We also require Lemma \ref{lemma:H-1SpectralPGF} to reduce Sobolev estimates of $\wt{\mathfrak{g}}^{\mathfrak{l}}$ to those of a pseudo-gradient factor with uniformly bounded support. We now inherit notation of Proposition \ref{prop:D1B2B} for this subsection.

Similar to the proof for Proposition \ref{prop:D1B2A}, we will present preliminary ingredients, use them to get Proposition \ref{prop:D1B2B}, and then defer the proof of the preliminary ingredients to the last subsection. The first preliminary is an analog of gymnastics in Lemma \ref{lemma:D1B2A1} and Lemma \ref{lemma:D1B2A2} with an almost identical proof. Note the below recovers Lemma \ref{lemma:D1B2A1} if $\beta_{X}$ were actually equal to $\frac12$.
\begin{lemma}\label{lemma:D1B2B1}
 Provided the setting of \emph{Lemma \ref{lemma:D1B2A1}}, for any $\mathfrak{i}\in\{1,2\}$, we have deterministic and expectation estimates: 
\small\begin{align}
\|\bar{\mathbf{H}}_{T,x}^{N}(N^{\beta_{X}}|\varphi_{S+\mathfrak{t}_{\mathfrak{s}},y}|)\|_{\mathfrak{t}_{\mathfrak{f}};\mathbb{T}_{N}} \ &\lesssim \ N^{\frac32\beta_{X}+\frac14+\e_{\mathrm{cpt}}+\frac12\e-\frac12\beta_{-}}\int_{0}^{\mathfrak{t}_{\mathfrak{f}}}\wt{\sum}_{y\in\mathbb{T}_{N}}|\varphi_{S+\mathfrak{t}_{\mathfrak{s}},y}| \ \d S \ + \ N^{-\e} \\
\E\|\bar{\mathbf{H}}_{T,x}^{N}(N^{\beta_{X}}|\mathscr{C}_{\mathfrak{t}_{\mathrm{av}};\beta_{-},\beta_{+}}^{\mathbf{T},+,\mathfrak{t}_{N,\e},\mathfrak{i}}(\wt{\mathfrak{g}}_{S+\mathfrak{t}_{\mathfrak{s}},y}^{\mathfrak{l}})|)\|_{\mathfrak{t}_{\mathfrak{f}};\mathbb{T}_{N}} \ &\lesssim \ N^{\frac32\beta_{X}+\frac14+\e_{\mathrm{cpt}}+\frac12\e-\frac12\beta_{-}}\mathfrak{t}_{\mathfrak{f}} \E^{\mu_{0,\Z}}\bar{\mathfrak{f}}_{\mathfrak{t}_{\mathfrak{s}},\mathfrak{t}_{\mathfrak{f}},\mathbb{T}_{N}} \E_{\eta_{0}}^{\mathrm{path}}|\mathscr{C}_{\mathfrak{t}_{\mathrm{av}};\beta_{-},\beta_{+}}^{\mathbf{T},+,\mathfrak{t}_{N,\e},\mathfrak{i}}(\wt{\mathfrak{g}}_{0,0}^{\mathfrak{l}})| \ + \ N^{-\e}.
\end{align}\normalsize\normalsize
The expectation $\E_{\eta_{0}}^{\mathrm{path}}$ is with respect to the path-space measure induced by the original $\Omega_{\Z}$-valued exclusion process with initial configuration $\eta_{0}\in\Omega_{\Z}$ that is then being averaged via $\E^{\mu_{0,\Z}}$ against the space-time averaged density $\bar{\mathfrak{f}}_{\mathfrak{t}_{\mathfrak{s}},\mathfrak{t}_{\mathfrak{f}},\mathbb{T}_{N}}$.
\end{lemma}
Parallel to the proof of Proposition \ref{prop:D1B2A}, we will now localize the inner expectation on the RHS in the second estimate in Lemma \ref{lemma:D1B2B1} to reduce our estimates for it to local equilibrium calculations. The localization mechanism follows similar constructions and procedure that we used in Lemma \ref{lemma:SoP}. This starts with the following local versions of Definition \ref{definition:D1B1B} and Definition \ref{definition:D1B2B}.
\begin{definition}\label{definition:averagesBlocal}
We recall the $\mathfrak{I}_{\mathfrak{t}_{\mathrm{av}}}$ block in Definition \ref{definition:tAvBlock} and the $\eta_{\mathfrak{t}_{\mathrm{av}};\bullet}$-process of the first paragraph of Definition \ref{definition:averagesAlocal}.
\begin{itemize}[leftmargin=*]
\item Define $\wt{\mathfrak{g}}_{\mathfrak{t}_{\mathrm{av}};S,0}^{\mathfrak{l}} \overset{\bullet}= \wt{\mathfrak{g}}_{0,0}^{\mathfrak{l}}(\eta_{\mathfrak{t}_{\mathrm{av}};S})$ with pseudo-gradient factor as in Proposition \ref{prop:Duhamel} but evaluated along the $\mathfrak{I}_{\mathfrak{t}_{\mathrm{av}}}$-local process.
\item We now define the following $\mathfrak{I}_{\mathfrak{t}_{\mathrm{av}}}$-local time-average that only sees data of the $\eta_{\mathfrak{t}_{\mathrm{av}};\bullet}$-exclusion process:
\small\begin{align}
\mathscr{A}_{\mathfrak{t}_{\mathrm{av}}}^{\mathbf{T},+}(\wt{\mathfrak{g}}_{\mathfrak{t}_{\mathrm{av}};S,0}^{\mathfrak{l}}) \ &\overset{\bullet}= \ \mathbf{1}_{\mathfrak{t}_{\mathrm{av}}\neq0}\mathfrak{t}_{\mathrm{av}}^{-1}\int_{0}^{\mathfrak{t}_{\mathrm{av}}} \wt{\mathfrak{g}}_{\mathfrak{t}_{\mathrm{av}};S+\mathfrak{r},0}^{\mathfrak{l}}\d\mathfrak{r} \ + \ \mathbf{1}_{\mathfrak{t}_{\mathrm{av}}=0}\wt{\mathfrak{g}}_{\mathfrak{t}_{\mathrm{av}};S,0}^{\mathfrak{l}}.
\end{align}\normalsize\normalsize
\item Let us define the following events for maximal processes and time-averages. We will again adopt the same comment about maximal processes we gave in Definition \ref{definition:D1B2B} upon the replacement of the functional $\wt{\mathfrak{g}}^{\mathfrak{l}}$ of the original $\Omega_{\Z}$-valued exclusion process by the functional $\wt{\mathfrak{g}}^{\mathfrak{l}}_{\mathfrak{t}_{\mathrm{av}};S,0}$ of the $\mathfrak{I}_{\mathfrak{t}_{\mathrm{av}}}$-local exclusion process:
\begin{subequations}
\small\begin{align}
\mathbf{1}[\mathscr{F}_{\mathfrak{t}_{\mathrm{av}};\lesssim;S,0}^{\beta_{-}}] \ &\overset{\bullet}= \ \mathbf{1}\left({\sup}_{0\leq\mathfrak{t}\leq\mathfrak{t}_{\mathrm{av}}}\mathfrak{t}\cdot\mathfrak{t}_{\mathrm{av}}^{-1}|\mathscr{A}_{\mathfrak{t}}^{\mathbf{T},+}(\wt{\mathfrak{g}}_{\mathfrak{t}_{\mathrm{av}};S,0}^{\mathfrak{l}})|\lesssim N^{-\beta_{-}}\right) \\
\mathbf{1}[\mathscr{F}_{\mathfrak{t}_{\mathrm{av}};\geq;S,0}^{\beta_{+}}] \ &\overset{\bullet}= \  \mathbf{1}\left({\sup}_{0\leq\mathfrak{t}\leq\mathfrak{t}_{\mathrm{av}}}\mathfrak{t}\cdot\mathfrak{t}_{\mathrm{av}}^{-1}|\mathscr{A}_{\mathfrak{t}}^{\mathbf{T},+}(\wt{\mathfrak{g}}_{\mathfrak{t}_{\mathrm{av}};S,0}^{\mathfrak{l}})|\geq N^{-\beta_{+}}\right).
\end{align}\normalsize\normalsize
\end{subequations}
Above the exponents $\beta_{-},\beta_{+}\in\R_{\geq0}$ are deterministic, universal, and assumed to satisfy $\beta_{-}\leq\beta_{+}$. Recalling $\mathfrak{t}_{N,\e}\in\R_{\geq0}$ from the statement of Proposition \ref{prop:D1B2B}, we define the following two-sided cutoffs of the previous $\mathfrak{I}_{\mathfrak{t}_{\mathrm{av}}}$-local time-average:
\begin{subequations}
\small\begin{align}
\mathscr{C}_{\mathfrak{t}_{\mathrm{av}};\beta_{-},\beta_{+}}^{\mathbf{T},+,\mathfrak{t}_{N,\e},1}(\wt{\mathfrak{g}}_{\mathfrak{t}_{\mathrm{av}};S,0}^{\mathfrak{l}}) \ &\overset{\bullet}= \ \mathscr{A}_{\mathfrak{t}_{\mathrm{av}}}^{\mathbf{T},+}(\wt{\mathfrak{g}}_{\mathfrak{t}_{\mathrm{av}};S+\mathfrak{t}_{N,\e},0}^{\mathfrak{l}})\cdot\mathbf{1}[\mathscr{F}_{\mathfrak{t}_{\mathrm{av}};\lesssim;S+\mathfrak{t}_{N,\e},0}^{\beta_{-}}]\mathbf{1}[\mathscr{F}_{\mathfrak{t}_{\mathrm{av}};\geq;S,0}^{\beta_{+}}] \\
\mathscr{C}_{\mathfrak{t}_{\mathrm{av}};\beta_{-},\beta_{+}}^{\mathbf{T},+,\mathfrak{t}_{N,\e},2}(\wt{\mathfrak{g}}_{\mathfrak{t}_{\mathrm{av}};S,0}^{\mathfrak{l}}) \ &\overset{\bullet}= \ \mathscr{A}_{\mathfrak{t}_{\mathrm{av}}}^{\mathbf{T},+}(\wt{\mathfrak{g}}_{\mathfrak{t}_{\mathrm{av}};S,0}^{\mathfrak{l}})\cdot\mathbf{1}[\mathscr{F}_{\mathfrak{t}_{\mathrm{av}};\lesssim;S,0}^{\beta_{-}}]\mathbf{1}[\mathscr{F}_{\mathfrak{t}_{\mathrm{av}};\geq;S+\mathfrak{t}_{N,\e},0}^{\beta_{+}}].
\end{align}\normalsize\normalsize
\end{subequations}
\end{itemize}
\end{definition}
The following preliminary ingredient to the proof of Proposition \ref{prop:D1B2B} is the parallel to Lemma \ref{lemma:SoP} but for time-averages of $\wt{\mathfrak{g}}^{\mathfrak{l}}$ and their local versions. Its proof is exactly the same as proof of Lemma \ref{lemma:SoP} because we are still comparing a time-average on the time-scale $\mathfrak{t}_{\mathrm{av}}\in\R_{\geq0}$ of a functional with support of size $N^{\beta_{X}}$ times universal factors to its $\mathfrak{I}_{\mathfrak{t}_{\mathrm{av}}}$-local version. We write the following localization-via-coupling bound precisely but it is just Lemma \ref{lemma:SoP} upon replacing $\mathscr{C}^{\mathbf{X},-}(\mathfrak{g})$-content by $\wt{\mathfrak{g}}^{\mathfrak{l}}$-content.
\begin{lemma}\label{lemma:SoP2}
 We have the following for $C \in \R_{>0}$ arbitrarily large but universal, $\mathfrak{i}\in\{1,2\}$, and any deterministic $\mathfrak{t}_{\mathrm{av}} \lesssim N^{10}$:
\small\begin{align}
\E_{\eta_{0}}^{\mathrm{path}}|\mathscr{A}_{\mathfrak{t}_{\mathrm{av}}}^{\mathbf{T},+}(\wt{\mathfrak{g}}_{0,0}^{\mathfrak{l}})|^{2} \ &= \ \E_{\Pi_{\mathfrak{I}_{\mathfrak{t}_{\mathrm{av}}}}\eta_{0}}^{\mathrm{loc}}|\mathscr{A}_{\mathfrak{t}_{\mathrm{av}}}^{\mathbf{T},+}(\wt{\mathfrak{g}}_{\mathfrak{t}_{\mathrm{av}};0,0}^{\mathfrak{l}})|^{2} \ + \ \mathscr{O}_{C}(N^{-C}) \label{eq:SoP21} \\
\E_{\eta_{0}}^{\mathrm{path}}|\mathscr{C}_{\mathfrak{t}_{\mathrm{av}};\beta_{-},\beta_{+}}^{\mathbf{T},+,\mathfrak{t}_{N,\e},\mathfrak{i}}(\wt{\mathfrak{g}}_{0,0}^{\mathfrak{l}})| \ &= \ \E_{\Pi_{\mathfrak{I}_{\mathfrak{t}_{\mathrm{av}}}}\eta_{0}}^{\mathrm{loc}}|\mathscr{C}_{\mathfrak{t}_{\mathrm{av}};\beta_{-},\beta_{+}}^{\mathbf{T},+,\mathfrak{t}_{N,\e},\mathfrak{i}}(\wt{\mathfrak{g}}_{\mathfrak{t}_{\mathrm{av}};0,0}^{\mathfrak{l}})| \ + \ \mathscr{O}_{C}(N^{-C}). \label{eq:SoP22}
\end{align}\normalsize\normalsize
%
\begin{itemize}[leftmargin=*]
\item The expectation $\E_{\eta_{0}}^{\mathrm{path}}$ is with respect to the path-space measure induced by the $\mathfrak{S}^{N,!!}$-dynamic on $\Omega_{\Z}$ with initial spins $\eta_{0}$.
\item The expectation $\E_{\Pi_{\mathfrak{I}_{\mathfrak{t}_{\mathrm{av}}}}\eta_{0}}^{\mathrm{loc}}$ is with respect to the path-space measure of the $\mathfrak{I}_{\mathfrak{t}_{\mathrm{av}}}$-process on $\Omega_{\mathfrak{I}_{\mathfrak{t}_{\mathrm{av}}}}$ with initial spins $\Pi_{\mathfrak{I}_{\mathfrak{t}_{\mathrm{av}}}}\eta_{0}$.
\item The functionals inside expectations depend only on a finite time-interval, but the two expectations are defined for infinite time.
\end{itemize}
\end{lemma}
Parallel to Lemma \ref{lemma:D1B2A3} we used in the proof of Proposition \ref{prop:D1B2A}, we take advantage of localization in Lemma \ref{lemma:SoP2} with the following which ultimately follows by reduction to local equilibrium via Lemma \ref{lemma:LE} and local equilibrium calculations.
\begin{lemma}\label{lemma:D1B2B3}
 With a universal implied constant, we have the following uniformly in $\mathfrak{l}\in\llbracket1,N^{\beta_{X}}\rrbracket$ in which $\e_{X}=\max_{i}\e_{X,i}$:
\small\begin{align}
N^{\frac32\beta_{X}+\frac14+\e_{\mathrm{cpt}}+\frac12\e-\frac12\beta_{-}}\mathfrak{t}_{\mathfrak{f}} \E^{\mu_{0,\Z}}\bar{\mathfrak{f}}_{\mathfrak{t}_{\mathfrak{s}},\mathfrak{t}_{\mathfrak{f}},\mathbb{T}_{N}}^{\mathfrak{I}_{\mathfrak{t}_{\mathrm{av}}}}\E_{\Pi_{\mathfrak{I}_{\mathfrak{t}_{\mathrm{av}}}}\eta_{0}}^{\mathrm{loc}}|\mathscr{C}_{\mathfrak{t}_{\mathrm{av}};\beta_{-},\beta_{+}}^{\mathbf{T},+,\mathfrak{t}_{N,\e},\mathfrak{i}}(\wt{\mathfrak{g}}_{\mathfrak{t}_{\mathrm{av}};0,0}^{\mathfrak{l}})| \ &\lesssim \ N^{-\frac18-\frac34\beta_{-}+10\e_{X}} + N^{-\frac12+10\e_{X}} + \mathfrak{t}_{\mathfrak{f}}N^{-\beta_{\mathrm{univ}}}.
\end{align}\normalsize\normalsize
\end{lemma}
\begin{proof}[Proof of \emph{Proposition \ref{prop:D1B2B}}]
Following the proof of Proposition \ref{prop:D1B2A}, we use Lemma \ref{lemma:D1B2B1} and the second bound in Lemma \ref{lemma:SoP2}:
\small\begin{align}
\E\|\bar{\mathbf{H}}_{T,x}^{N}(N^{\frac12}|\mathscr{C}_{\mathfrak{t}_{\mathrm{av}};\beta_{-},\beta_{+}}^{\mathbf{T},+,\mathfrak{t}_{N,\e},\mathfrak{i}}(\wt{\mathfrak{g}}_{S+\mathfrak{t}_{\mathfrak{s}},y}^{\mathfrak{l}})|)\|_{\mathfrak{t}_{\mathfrak{f}};\mathbb{T}_{N}} \ &\lesssim \ N^{\frac32\beta_{X}+\frac14+\e_{\mathrm{cpt}}+\frac12\e-\frac12\beta_{-}}\mathfrak{t}_{\mathfrak{f}} \E^{\mu_{0,\Z}}\bar{\mathfrak{f}}_{\mathfrak{t}_{\mathfrak{s}},\mathfrak{t}_{\mathfrak{f}},\mathbb{T}_{N}} \E_{\eta_{0}}^{\mathrm{path}}|\mathscr{C}_{\mathfrak{t}_{\mathrm{av}};\beta_{-},\beta_{+}}^{\mathbf{T},+,\mathfrak{t}_{N,\e},\mathfrak{i}}(\wt{\mathfrak{g}}_{0,0}^{\mathfrak{l}})| \ + \ N^{-\e} \nonumber \\
&\lesssim \ N^{\frac32\beta_{X}+\frac14+\e_{\mathrm{cpt}}+\frac12\e-\frac12\beta_{-}}\mathfrak{t}_{\mathfrak{f}} \E^{\mu_{0,\Z}}\bar{\mathfrak{f}}_{\mathfrak{t}_{\mathfrak{s}},\mathfrak{t}_{\mathfrak{f}},\mathbb{T}_{N}}^{\mathfrak{I}_{\mathfrak{t}_{\mathrm{av}}}}\E_{\Pi_{\mathfrak{I}_{\mathfrak{t}_{\mathrm{av}}}}\eta_{0}}^{\mathrm{loc}}|\mathscr{C}_{\mathfrak{t}_{\mathrm{av}};\beta_{-},\beta_{+}}^{\mathbf{T},+,\mathfrak{t}_{N,\e},\mathfrak{i}}(\wt{\mathfrak{g}}_{\mathfrak{t}_{\mathrm{av}};0,0}^{\mathfrak{l}})| \ + \ N^{-\e}. \label{eq:D1B2B1}
\end{align}\normalsize\normalsize
We clarify \eqref{eq:D1B2B1} also requires the observation that the inner expectation therein depends only on the initial particle configuration on the block $\mathfrak{I}_{\mathfrak{t}_{\mathrm{av}}}$. This allows us to project the probability density $\bar{\mathfrak{f}}$ in \eqref{eq:D1B2B1} on its marginal distribution on the block $\mathfrak{I}_{\mathfrak{t}_{\mathrm{av}}}$. It now suffices to apply the estimate in Lemma \ref{lemma:D1B2B3} to the first term in \eqref{eq:D1B2B1} to finish the proof.
\end{proof}
\subsection{Proof of Proposition \ref{prop:D1B1A}}
The proof of Proposition \ref{prop:D1B1A} will be easier because of the small factor $\mathfrak{t}_{\mathrm{av}}^{1/4}$ on the LHS of the proposed bound. However, it follows the same general scheme of the proof of Proposition \ref{prop:D1B2A}, for example. Precisely, we use technical gymnastics to remove the singularity in the heat kernel in the heat operator on the LHS of the proposed estimate in Proposition \ref{prop:D1B1A}. We then localize mesoscopic space-time averages, apply local equilibrium via Lemma \ref{lemma:LE}, and perform equilibrium calculations via Lemmas \ref{lemma:KV} and \ref{lemma:SpectralH-1}. Let us now inherit notation in the statement of Proposition \ref{prop:D1B1A}.

The following first ingredient is slightly different than its analog Lemma \ref{lemma:D1B2A1}. For example, it squares the test function.
\begin{lemma}\label{lemma:D1B1A1}
 Take any $\varphi:\R_{\geq0}\times\mathbb{T}_{N}\to\R$. We have a deterministic and expectation bound with universal implied constants:
\small\begin{align}
\left(\|\bar{\mathbf{H}}_{T,x}^{N}(N^{\frac12}|\varphi_{S,y}|)\|_{1;\mathbb{T}_{N}} \right)^{2} \ &\lesssim \ N^{\frac54+\e_{\mathrm{cpt}}}\int_{0}^{1}\wt{\sum}_{y\in\mathbb{T}_{N}}|\varphi_{S,y}|^{2} \ \d S \\
\left(\E\|\bar{\mathbf{H}}_{T,x}^{N}(N^{\frac12}|\mathscr{A}_{\mathfrak{t}_{\mathrm{av}}}^{\mathbf{T},+}\mathscr{C}_{N^{\beta_{X}}}^{\mathbf{X},-}(\mathfrak{g}_{S,y})|)\|_{1;\mathbb{T}_{N}}\right)^{2} \ &\lesssim \ N^{\frac54+\e_{\mathrm{cpt}}} \E^{\mu_{0,\Z}}\bar{\mathfrak{f}}_{1,\mathbb{T}_{N}} \E_{\eta_{0}}^{\mathrm{path}}|\mathscr{A}_{\mathfrak{t}_{\mathrm{av}}}^{\mathbf{T},+}\mathscr{C}_{N^{\beta_{X}}}^{\mathbf{X},-}(\mathfrak{g}_{0,0})|^{2}.
\end{align}\normalsize\normalsize
\end{lemma}
Following the proof of Proposition \ref{prop:D1B2A}, we localize the path-space expectation on the RHS of the second estimate in Lemma \ref{lemma:D1B1A1}. The error is controlled already in Lemma \ref{lemma:SoP}. Thus, parallel to Lemma \ref{lemma:D1B2A3}, we take advantage of localization below.
\begin{lemma}\label{lemma:D1B1A3}
 Recall the $\mathfrak{I}_{\mathfrak{t}_{\mathrm{av}}}$ block in \emph{Definition \ref{definition:tAvBlock}}. Provided any $\e\in\R_{>0}$, we have
\small\begin{align}
N^{\frac54+\e_{\mathrm{cpt}}} \mathfrak{t}_{\mathrm{av}}^{\frac12}\E^{\mu_{0,\Z}}\bar{\mathfrak{f}}_{1,\mathbb{T}_{N}} ^{\mathfrak{I}_{\mathfrak{t}_{\mathrm{av}}}}\E_{\Pi_{\mathfrak{I}_{\mathfrak{t}_{\mathrm{av}}}}\eta_{0}}^{\mathrm{loc}}|\mathscr{A}_{\mathfrak{t}_{\mathrm{av}}}^{\mathbf{T},+}\mathscr{C}_{N^{\beta_{X}}}^{\mathbf{X},-}(\mathfrak{g}_{\mathfrak{t}_{\mathrm{av}};0,0})|^{2} \ &\lesssim_{\e} \ N^{\frac54+\e_{\mathrm{cpt}}}\mathfrak{t}_{\mathrm{av}}^{\frac12}N^{-\frac34-\beta_{X}+\e_{X,1}+\e}|\mathbb{T}_{N}|^{-1}|\mathfrak{I}_{\mathfrak{t}_{\mathrm{av}}}|^{3} \ + \ N^{\frac54+\e_{\mathrm{cpt}}}N^{-2-\beta_{X}}\mathfrak{t}_{\mathrm{av}}^{-\frac12}. \nonumber
\end{align}\normalsize
\end{lemma}
The $\mathfrak{t}_{\mathrm{av}}^{1/2}$-factor on the LHS comes by squaring the ``time-regularity factor" $\mathfrak{t}_{\mathrm{av}}^{1/4}$ in the proposed bound of Proposition \ref{prop:D1B1A}.
\begin{proof}[Proof of \emph{Proposition \ref{prop:D1B1A}}]
Following the proof of Proposition \ref{prop:D1B2A}, we employ the second estimate in Lemma \ref{lemma:D1B1A1} along with the first estimate in Lemma \ref{lemma:SoP} that we introduced in the proof of Proposition \ref{prop:D1B2A}. As $\mathfrak{t}_{\mathrm{av}}\in\R_{\geq0}$ is uniformly bounded, we get
\small\begin{align}
\left(\mathfrak{t}_{\mathrm{av}}^{\frac14}\E\|\bar{\mathbf{H}}_{T,x}^{N}(N^{\frac12}|\mathscr{A}_{\mathfrak{t}_{\mathrm{av}}}^{\mathbf{T},+}\mathscr{C}_{N^{\beta_{X}}}^{\mathbf{X},-}(\mathfrak{g}_{S,y})|)\|_{1;\mathbb{T}_{N}}\right)^{2} \ &\lesssim \ N^{\frac54+\e_{\mathrm{cpt}}}\mathfrak{t}_{\mathrm{av}}^{\frac12}\E^{\mu_{0,\Z}}\bar{\mathfrak{f}}_{1,\mathbb{T}_{N}} \E_{\eta_{0}}^{\mathrm{path}}|\mathscr{A}_{\mathfrak{t}_{\mathrm{av}}}^{\mathbf{T},+}\mathscr{C}_{N^{\beta_{X}}}^{\mathbf{X},-}(\mathfrak{g}_{0,0})|^{2} \\
&\lesssim \ N^{\frac54+\e_{\mathrm{cpt}}}\mathfrak{t}_{\mathrm{av}}^{\frac12}\E^{\mu_{0,\Z}}\bar{\mathfrak{f}}_{1,\mathbb{T}_{N}} ^{\mathfrak{I}_{\mathfrak{t}_{\mathrm{av}}}}\E_{\Pi_{\mathfrak{I}_{\mathfrak{t}_{\mathrm{av}}}}\eta_{0}}^{\mathrm{loc}}|\mathscr{A}_{\mathfrak{t}_{\mathrm{av}}}^{\mathbf{T},+}\mathscr{C}_{N^{\beta_{X}}}^{\mathbf{X},-}(\mathfrak{g}_{\mathfrak{t}_{\mathrm{av}};0,0})|^{2} \ + \ N^{-100}. \label{eq:D1B1A1}
\end{align}\normalsize\normalsize
The same clarifying remarks we made after the estimate of \eqref{eq:D1B2A1} from the proof for Proposition \ref{prop:D1B2A} are in order here as well. Thus, it remains to estimate the first expectation from the estimate \eqref{eq:D1B1A1}. We do this by applying Lemma \ref{lemma:D1B1A3} to this first term in \eqref{eq:D1B1A1}. The proof for Proposition \ref{prop:D1B1A} then follows by power-counting upon recalling the constraint $N^{-2}\lesssim\mathfrak{t}_{\mathrm{av}}\lesssim N^{-1}$ as well as the length-scale $|\mathfrak{I}_{\mathfrak{t}_{\mathrm{av}}}| \lesssim_{\e}N^{1+\e}\mathfrak{t}_{\mathrm{av}}^{1/2} + N^{3/2+\e}\mathfrak{t}_{\mathrm{av}}+ N^{\beta_{X}+\e}$ given any $\e\in\R_{>0}$ from Definition \ref{definition:tAvBlock} and $|\mathbb{T}_{N}|\gtrsim N^{5/4+\e_{\mathrm{cpt}}}$ from the beginning of Section \ref{section:Ctify}. This power-counting is ultimately elementary and thus omitted.
\end{proof}
\subsection{Proof of Proposition \ref{prop:D1B1B}}
Let us again inherit all notation and framework of Proposition \ref{prop:D1B1B}. The proof of Proposition \ref{prop:D1B1B} resembles strongly the proof of Proposition \ref{prop:D1B1A}. We apply the following preliminary gymnastics parallel to Lemma \ref{lemma:D1B1A1} to avoid the heat kernel singularity. Again, if $\beta_{X}$ were actually $\frac12$, the first estimate below would recover that of Lemma \ref{lemma:D1B1A1}.
\begin{lemma}\label{lemma:D1B1B1}
 Take any $\varphi:\R_{\geq0}\times\mathbb{T}_{N}\to\R$. We have a deterministic and expectation bound with universal implied constants:
\small\begin{align}
\|\bar{\mathbf{H}}_{T,x}^{N}(N^{\beta_{X}}|\varphi_{S,y}|)\|_{1;\mathbb{T}_{N}} \ &\lesssim \ \left(N^{2\beta_{X}+\frac14+\e_{\mathrm{cpt}}}\int_{0}^{1}\wt{\sum}_{y\in\mathbb{T}_{N}}|\varphi_{S,y}|^{2} \ \d S \right)^{\frac12} \\
\left( \E\|\bar{\mathbf{H}}_{T,x}^{N}(N^{\beta_{X}}|\mathscr{A}_{\mathfrak{t}_{\mathrm{av}}}^{\mathbf{T},+}(\wt{\mathfrak{g}}_{S,y}^{\mathfrak{l}})|)\|_{1;\mathbb{T}_{N}}\right)^{2} \ &\lesssim \ N^{2\beta_{X}+\frac14+\e_{\mathrm{cpt}}} \E^{\mu_{0,\Z}}\bar{\mathfrak{f}}_{1,\mathbb{T}_{N}} \E_{\eta_{0}}^{\mathrm{path}}|\mathscr{A}_{\mathfrak{t}_{\mathrm{av}}}^{\mathbf{T},+}(\wt{\mathfrak{g}}_{0,0}^{\mathfrak{l}})|^{2}.
\end{align}\normalsize\normalsize
\end{lemma}
From Lemma \ref{lemma:D1B1B1} by following the proof of Proposition \ref{prop:D1B2B}, we localize the inner path-space expectation on the RHS of the second bound in Lemma \ref{lemma:D1B1B1}. Similar to the proof of Proposition \ref{prop:D1B1A}, this is actually already done in Lemma \ref{lemma:SoP2}. We now take advantage of localization via the following parallel to Lemma \ref{lemma:D1B2B3} and Lemma \ref{lemma:D1B1A3}.
\begin{lemma}\label{lemma:D1B1B3}
Provided any $\e\in\R_{>0}$, uniformly in $\mathfrak{l}\in\llbracket1,N^{\beta_{X}}\rrbracket$ we have
\small\begin{align}
N^{2\beta_{X}+\frac14+\e_{\mathrm{cpt}}}\mathfrak{t}_{\mathrm{av}}^{\frac12} \E^{\mu_{0,\Z}}\bar{\mathfrak{f}}_{1,\mathbb{T}_{N}}^{\mathfrak{I}_{\mathfrak{t}_{\mathrm{av}}}} \E_{\Pi_{\mathfrak{I}_{\mathfrak{t}_{\mathrm{av}}}}\eta_{0}}^{\mathrm{loc}}|\mathscr{A}_{\mathfrak{t}_{\mathrm{av}}}^{\mathbf{T},+}(\wt{\mathfrak{g}}_{\mathfrak{t}_{\mathrm{av}};0,0}^{\mathfrak{l}})|^{2} \ &\lesssim_{\e} \ N^{2\beta_{X}+\frac14+\e_{\mathrm{cpt}}}\mathfrak{t}_{\mathrm{av}}^{\frac12} N^{-\frac34+\e}|\mathbb{T}_{N}|^{-1}|\mathfrak{I}_{\mathfrak{t}_{\mathrm{av}}}|^{3} \ + \ N^{2\beta_{X}+\frac14+\e_{\mathrm{cpt}}}N^{-2}\mathfrak{t}_{\mathrm{av}}^{-\frac12}. \nonumber
\end{align}\normalsize
\end{lemma}
\begin{proof}[Proof of \emph{Proposition \ref{prop:D1B1B}}]
We follow somewhat the proof of Proposition \ref{prop:D1B1A}. To start, we employ the gymnastics of Lemma \ref{lemma:D1B1B1} and the localization estimate \eqref{eq:SoP21} of Lemma \ref{lemma:SoP2}. This is in parallel to the first step in the proof of Proposition \ref{prop:D1B1A}:
\small\begin{align}
\left(\mathfrak{t}_{\mathrm{av}}^{\frac14}\E\|\bar{\mathbf{H}}_{T,x}^{N}(N^{\beta_{X}}|\mathscr{A}_{\mathfrak{t}_{\mathrm{av}}}^{\mathbf{T},+}(\wt{\mathfrak{g}}_{S,y}^{\mathfrak{l}})|)\|_{1;\mathbb{T}_{N}}\right)^{2} \ &\lesssim \ N^{2\beta_{X}+\frac14+\e_{\mathrm{cpt}}}\mathfrak{t}_{\mathrm{av}}^{\frac12}\E^{\mu_{0,\Z}}\bar{\mathfrak{f}}_{1,\mathbb{T}_{N}} \E_{\eta_{0}}^{\mathrm{path}}|\mathscr{A}_{\mathfrak{t}_{\mathrm{av}}}^{\mathbf{T},+}(\wt{\mathfrak{g}}_{0,0}^{\mathfrak{l}})|^{2} \\
&\lesssim \ N^{2\beta_{X}+\frac14+\e_{\mathrm{cpt}}}\mathfrak{t}_{\mathrm{av}}^{\frac12}\E^{\mu_{0,\Z}}\bar{\mathfrak{f}}_{1,\mathbb{T}_{N}}^{\mathfrak{I}_{\mathfrak{t}_{\mathrm{av}}}} \E_{\Pi_{\mathfrak{I}_{\mathfrak{t}_{\mathrm{av}}}}\eta_{0}}^{\mathrm{loc}}|\mathscr{A}_{\mathfrak{t}_{\mathrm{av}}}^{\mathbf{T},+}(\wt{\mathfrak{g}}_{\mathfrak{t}_{\mathrm{av}};0,0}^{\mathfrak{l}})|^{2} \ + \ N^{-100}. \label{eq:D1B1B1}
\end{align}\normalsize\normalsize
The same clarifying remarks we made after the estimate of \eqref{eq:D1B2B1} from the proof of Proposition \ref{prop:D1B2B} are in order here as well. Like the proof of Proposition \ref{prop:D1B1A}, we apply the estimate in Lemma \ref{lemma:D1B1B3} to the first term in \eqref{eq:D1B1B1}. The proof for Proposition \ref{prop:D1B1B} then follows by power-counting in $N$ and recalling $N^{-2}\lesssim\mathfrak{t}_{\mathrm{av}}\lesssim N^{-1}$, again like the end of the proof of Proposition \ref{prop:D1B1A}.
\end{proof}
\subsection{Proof of Proposition \ref{prop:S1B}}
We turn to the proof for the static estimate in Proposition \ref{prop:S1B}. The idea behind the proof for this estimate in Proposition \ref{prop:S1B} is the same as every argument we have written thus far. By this we mean gymnastics to remove the heat kernel singularity in the heat operator in Proposition \ref{prop:S1B} followed by a reduction to local equilibrium and equilibrium calculations. However, there are no time-averages in Proposition \ref{prop:S1B}, so we will apply large-deviations estimates for canonical ensembles in Lemma \ref{lemma:LDP} and Corollary \ref{corollary:LDP}. We reemphasize that intuitively, Proposition \ref{prop:S1B} follows from the negligibility of the replacement $\mathscr{A}^{\mathbf{X},-}\to\mathscr{C}^{\mathbf{X},-}$ for canonical ensembles \emph{at large-deviations scale}, from which we reduce from non-stationary measures to stationary measures by local equilibrium in Lemma \ref{lemma:LE}. We begin with ``gymnastics" that resembles Lemma \ref{lemma:D1B1A1}.
\begin{lemma}\label{lemma:S1B1}
Consider the block $\mathfrak{I} = \llbracket-100\mathfrak{m}N^{\beta_{X}}, 100\mathfrak{m}N^{\beta_{X}}\rrbracket$. We have the following with universal implied constant:
\small\begin{align}
\left(\E\|\bar{\mathbf{H}}_{T,x}^{N}(N^{\frac12}|\mathscr{A}_{N^{\beta_{X}}}^{\mathbf{X},-}(\mathfrak{g}_{S,y}) - \mathscr{C}_{N^{\beta_{X}}}^{\mathbf{X},-}(\mathfrak{g}_{S,y})|)\|_{1;\mathbb{T}_{N}}\right)^{2} \ &\lesssim \ N^{\frac54+\e_{\mathrm{cpt}}}\E^{\mu_{0,\Z}}\bar{\mathfrak{f}}_{1,\mathbb{T}_{N}}^{\mathfrak{I}} |\mathscr{A}_{N^{\beta_{X}}}^{\mathbf{X},-}(\mathfrak{g}_{0,0}) - \mathscr{C}_{N^{\beta_{X}}}^{\mathbf{X},-}(\mathfrak{g}_{0,0})|^{2}. 
\end{align}\normalsize\normalsize
\end{lemma}
We observe the RHS of the estimate in Lemma \ref{lemma:S1B1} is the expectation of a $\mathfrak{I}$-local statistic against the averaged law $\bar{\mathfrak{f}}$ of the particle system. The block $\mathfrak{I}$ is that in Lemma \ref{lemma:S1B1}. Thus, we may again employ Lemma \ref{lemma:LE} to reduce estimation on the RHS of the estimate in Lemma \ref{lemma:S1B1} to a problem of equilibrium calculations. Ultimately, we get the following.
\begin{lemma}\label{lemma:S1B2}
With a universal implied constant, we have the following for the same block $\mathfrak{I} = \llbracket-100\mathfrak{m}N^{\beta_{X}}, 100\mathfrak{m}N^{\beta_{X}}\rrbracket$:
\small\begin{align}
N^{\frac54+\e_{\mathrm{cpt}}}\E^{\mu_{0,\Z}}\bar{\mathfrak{f}}_{1,\mathbb{T}_{N}}^{\mathfrak{I}} |\mathscr{A}_{N^{\beta_{X}}}^{\mathbf{X},-}(\mathfrak{g}_{0,0}) - \mathscr{C}_{N^{\beta_{X}}}^{\mathbf{X},-}(\mathfrak{g}_{0,0})|^{2} \ &\lesssim \ N^{-\beta_{\mathrm{univ}}}.
\end{align}\normalsize\normalsize
\end{lemma}
\begin{proof}[Proof of \emph{Proposition \ref{prop:S1B}}]
It suffices to combine the estimates of Lemma \ref{lemma:S1B1} and Lemma \ref{lemma:S1B2}.
\end{proof}
\subsection{Proofs of Technical Estimates}
We spend the rest of this section on proofs of all preliminary technical ingredients. 
\begin{proof}[Proof of \emph{Lemma \ref{lemma:D1B2A1}}]
Consider any $(T,x)\in[0,\mathfrak{t}_{\mathfrak{f}}]\times\mathbb{T}_{N}$ and define $\mathfrak{t}_{N}\overset{\bullet}=T-N^{-\frac12+\beta_{-}-\e}$ and $\mathfrak{t}_{N,+} = \mathfrak{t}_{N}\vee0$. Upon redefining $\varphi$ by a $\mathfrak{t}_{\mathfrak{s}}$ time-shift, it suffices to assume $\mathfrak{t}_{\mathfrak{s}}=0$. We first write
\small\begin{align}
\bar{\mathbf{H}}_{T,x}^{N}(N^{\frac12}|\varphi_{S,y}|) \ = \ \bar{\mathbf{H}}_{T,x}^{N}(N^{\frac12}|\varphi_{S,y}|\mathbf{1}_{\mathfrak{t}_{N,+}\leq S\leq T}) \ + \ \bar{\mathbf{H}}_{T,x}^{N}(N^{\frac12}|\varphi_{S,y}|\mathbf{1}_{0\leq S\leq\mathfrak{t}_{N,+}}) \ \overset{\bullet}= \ \Phi_{1} + \Phi_{2}. \label{eq:D1B2A10}
\end{align}\normalsize\normalsize
We will control $\Phi_{1}$ and $\Phi_{2}$ uniformly on the space-time set $[0,\mathfrak{t}_{\mathfrak{f}}]\times\mathbb{T}_{N}$. We bound $\Phi_{1}$ by the convolution estimate \eqref{eq:HKEConvolution} for the heat operator which controls the heat operator by the supremum of whatever test function that it acts on times the length of the time-interval that the integral in the heat operator is defined on:
\small\begin{align}
\bar{\mathbf{H}}_{T,x}^{N}(N^{\frac12}|\varphi_{S,y}|\mathbf{1}_{\mathfrak{t}_{N,+}\leq S\leq T}) \ \lesssim \ N^{\frac12}\|\varphi\|_{\infty;\mathbb{T}_{N}} |T-\mathfrak{t}_{N,+}| \ \lesssim \ N^{\frac12}\|\varphi\|_{\infty;\mathbb{T}_{N}}N^{-\frac12+\beta_{-}-\e} \ \lesssim \ N^{-\e}. \label{eq:D1B2A11}
\end{align}\normalsize\normalsize
The final estimate on the far RHS of \eqref{eq:D1B2A11} follows from recalling the a priori estimate $|\varphi|\lesssim N^{-\beta_{-}}$. To estimate $\Phi_{2}$, we use the on-diagonal heat kernel estimate implied by \eqref{eq:HKENash}. Observe that on the set $[0,\mathfrak{t}_{N,+}]$ we have the bound $\mathfrak{s}_{S,T}\gtrsim N^{-1/2+\beta_{-}-\e}$. This gives a deterministic bound that uses the aforementioned lower bound to control the short-time singularity in the heat kernel:
\small\begin{align}
\bar{\mathbf{H}}_{T,x}^{N}(N^{\frac12}|\varphi_{S,y}|\mathbf{1}_{0\leq S\leq\mathfrak{t}_{N,+}}) \ \lesssim \ N^{-1}\int_{0}^{\mathfrak{t}_{N,+}}\mathfrak{s}_{S,T}^{-1/2}{\sum}_{y\in\mathbb{T}_{N}} N^{\frac12}|\varphi_{S,y}| \ \d S \ &\lesssim \ N^{-1}N^{\frac14-\frac12\beta_{-}+\frac12\e}N^{\frac12}\int_{0}^{\mathfrak{t}_{N,+}}{\sum}_{y\in\mathbb{T}_{N}}|\varphi_{S,y}| \ \d S. \label{eq:D1B2A12a}
\end{align}\normalsize\normalsize
From the previous display, we multiply and divide by $|\mathbb{T}_{N}|$ on the far RHS to get a spatial average over $\mathbb{T}_{N}$; recall $|\mathbb{T}_{N}|\lesssim N^{5/4+\e_{\mathrm{cpt}}}$ from the beginning of Section \ref{section:Ctify}. Moreover, we may additionally extend the domain of integration $[0,\mathfrak{t}_{N,+}]\to[0,\mathfrak{t}_{\mathfrak{f}}]$ after we multiply and divide by $|\mathbb{T}_{N}|$. This gives the following estimate which identifies the sum over $\mathbb{T}_{N}$ by $|\mathbb{T}_{N}|$ times the average over $\mathbb{T}_{N}$ and then performs power-counting for the host of $N$-dependent factors that appear in the display below:
\small\begin{align}
N^{-1}N^{\frac14-\frac12\beta_{-}+\frac12\e}N^{\frac12}\int_{0}^{\mathfrak{t}_{N,+}}{\sum}_{y\in\mathbb{T}_{N}}|\varphi_{S,y}| \ \d S \ &\lesssim \ N^{-1}N^{\frac14-\frac12\beta_{-}+\frac12\e}N^{\frac12}|\mathbb{T}_{N}|\int_{0}^{\mathfrak{t}_{N,+}}\wt{\sum}_{y\in\mathbb{T}_{N}}|\varphi_{S,y}| \ \d S \\
&\lesssim \ N^{1-\frac12\beta_{-}+\frac12\e+\e_{\mathrm{cpt}}}\int_{0}^{\mathfrak{t}_{\mathfrak{f}}}\wt{\sum}_{y\in\mathbb{T}_{N}}|\varphi_{S,y}| \ \d S. \label{eq:D1B2A12b}
\end{align}\normalsize\normalsize
We reiterate that the final bound follows from power-counting on the RHS of the previous line and then extending the domain of integration $[0,\mathfrak{t}_{N,+}]\to[0,\mathfrak{t}_{\mathfrak{f}}]$. We also note that \eqref{eq:D1B2A11} and \eqref{eq:D1B2A12a} and \eqref{eq:D1B2A12b} provide estimates for $\Phi_{1}$ and $\Phi_{2}$ uniformly on $[0,\mathfrak{t}_{\mathfrak{f}}]\times\mathbb{T}_{N}$ so by the triangle inequality for $\|\|_{\mathfrak{t}_{\mathfrak{f}};\mathbb{T}_{N}}$ and by \eqref{eq:D1B2A10}, \eqref{eq:D1B2A11}, \eqref{eq:D1B2A12a} and \eqref{eq:D1B2A12b}, we finish the proof.
\end{proof}
\begin{proof}[Proof of \emph{Lemma \ref{lemma:D1B2A2}}]
Observe the $\mathscr{C}^{\mathbf{T},+,\mathfrak{t}_{N,\e},\mathfrak{i}}$ term in the heat operator on the LHS of the proposed estimate has an upper bound of order $N^{-\beta_{-}}$ by construction in Definition \ref{definition:D1B2A}. Thus, we may use Lemma \ref{lemma:D1B2A1} for $\varphi$ equal to this $\mathscr{C}^{\mathbf{T},+,\mathfrak{t}_{N,\e},\mathfrak{i}}$-term to get
\small\begin{align}
\E\|\bar{\mathbf{H}}_{T,x}^{N}(N^{\frac12}|\mathscr{C}_{\mathfrak{t}_{\mathrm{av}};\beta_{-},\beta_{+}}^{\mathbf{T},+,\mathfrak{t}_{N,\e},\mathfrak{i}}\mathscr{C}_{N^{\beta_{X}}}^{\mathbf{X},-}(\mathfrak{g}_{S+\mathfrak{t}_{\mathfrak{s}},y})|)\|_{\mathfrak{t}_{\mathfrak{f}};\mathbb{T}_{N}} \ &\lesssim \  N^{1+\e_{\mathrm{cpt}}+\frac12\e-\frac12\beta_{-}}\int_{0}^{\mathfrak{t}_{\mathfrak{f}}}\wt{\sum}_{y\in\mathbb{T}_{N}}\E|\mathscr{C}_{\mathfrak{t}_{\mathrm{av}};\beta_{-},\beta_{+}}^{\mathbf{T},+,\mathfrak{t}_{N,\e},\mathfrak{i}}\mathscr{C}_{N^{\beta_{X}}}^{\mathbf{X},-}(\mathfrak{g}_{S+\mathfrak{t}_{\mathfrak{s}},y})| \ \d S \ + \ N^{-\e}. \label{eq:D1B2A21}
\end{align}\normalsize\normalsize
We now make the next observation with explanation/justification after in which we use notation to be explained as well:
\small\begin{align}
\E|\mathscr{C}_{\mathfrak{t}_{\mathrm{av}};\beta_{-},\beta_{+}}^{\mathbf{T},+,\mathfrak{t}_{N,\e},\mathfrak{i}}\mathscr{C}_{N^{\beta_{X}}}^{\mathbf{X},-}(\mathfrak{g}_{S+\mathfrak{t}_{\mathfrak{s}},y})| \ &= \ \E^{\mu_{0,\Z}}\mathfrak{f}_{S+\mathfrak{t}_{\mathfrak{s}},N} \E_{\eta_{0}}^{\mathrm{path}}|\mathscr{C}_{\mathfrak{t}_{\mathrm{av}};\beta_{-},\beta_{+}}^{\mathbf{T},+,\mathfrak{t}_{N,\e},\mathfrak{i}}\mathscr{C}_{N^{\beta_{X}}}^{\mathbf{X},-}(\mathfrak{g}_{0,y})|. \label{eq:D1B2A22}
\end{align}\normalsize\normalsize
The previous statement \eqref{eq:D1B2A22} follows by decomposing the expectation of the \emph{path-space functional} given by the $\mathscr{C}^{\mathbf{T},+,\mathfrak{t}_{N,\e},\mathfrak{i}}$-term into the expectation with respect to the path-space measure corresponding to the original $\Omega_{\Z}$-valued process after time $S+\mathfrak{t}_{\mathfrak{s}}$ with an initial particle configuration given by the particle configuration at time $S+\mathfrak{t}_{\mathfrak{s}}$ and afterwards taking expectation over this initial configuration for the path-space expectation with respect to the law of the particle system $\mathfrak{f}_{S+\mathfrak{t}_{\mathfrak{s}},N}\d\mu_{0,\Z}$ at time $S+\mathfrak{t}_{\mathfrak{s}}$. However, this is the same thing as sampling the initial configuration for the path-space expectation according to the time-0 configuration if we pick the initial measure of $\mathfrak{f}_{S+\mathfrak{t}_{\mathfrak{s}},N}\d\mu_{0,\Z}$. Here, we emphasize that the path-space expectation after we shift the initial measure to $\mathfrak{f}_{S+\mathfrak{t}_{\mathfrak{s}},N}\d\mu_{0,\Z}$ is with respect to the path-space measure of the exclusion process starting at time 0 rather than time $S+\mathfrak{t}_{\mathfrak{s}}$. However, the law of the particle system dynamic is invariant under any time-shift as long as we pick the same initial configuration. We also get the following, which we explain afterwards, where $(\tau_{y}\eta)_{Z} = \eta_{y+Z}$ is the shift-map on $\Omega_{\Z}$:
\small\begin{align}
\E^{\mu_{0,\Z}}\mathfrak{f}_{S+\mathfrak{t}_{\mathfrak{s}},N} \E_{\eta_{0}}^{\mathrm{path}}|\mathscr{C}_{\mathfrak{t}_{\mathrm{av}};\beta_{-},\beta_{+}}^{\mathbf{T},+,\mathfrak{t}_{N,\e},\mathfrak{i}}\mathscr{C}_{N^{\beta_{X}}}^{\mathbf{X},-}(\mathfrak{g}_{0,y})| \ &= \ \E^{\mu_{0,\Z}}\mathfrak{f}_{S+\mathfrak{t}_{\mathfrak{s}},N} \E_{\tau_{-y}\eta_{0}}^{\mathrm{path}}|\mathscr{C}_{\mathfrak{t}_{\mathrm{av}};\beta_{-},\beta_{+}}^{\mathbf{T},+,\mathfrak{t}_{N,\e},\mathfrak{i}}\mathscr{C}_{N^{\beta_{X}}}^{\mathbf{X},-}(\mathfrak{g}_{0,0})| \nonumber \\
&= \ \E^{\mu_{0,\Z}}\tau_{y}\mathfrak{f}_{S+\mathfrak{t}_{\mathfrak{s}},N} \E_{\eta_{0}}^{\mathrm{path}}|\mathscr{C}_{\mathfrak{t}_{\mathrm{av}};\beta_{-},\beta_{+}}^{\mathbf{T},+,\mathfrak{t}_{N,\e},\mathfrak{i}}\mathscr{C}_{N^{\beta_{X}}}^{\mathbf{X},-}(\mathfrak{g}_{0,0})|. \nonumber
\end{align}\normalsize\normalsize
The first identity follows by the shift-invariance of the path-space expectation. In particular, on the far LHS we take the path-space expectation of particle system data supported near $y\in\Z$. This is identical to path-space expectation with the same data but shifted to be centered at $0\in\Z$,  because the path-space measure is invariant under spatial shifts if the initial configuration is kept the same. Such a ``shifted" path-space expectation is what middle of the above line encodes. The second identity follows from the invariance of the measure $\mu_{0,\Z}$ under shift-maps $\tau_{y}:\Omega_{\Z}\to\Omega_{\Z}$ and that the composition $\tau_{y}\tau_{-y}$ is the identity on $\Omega_{\Z}$. We clarify the inner expectation in the middle of the last display is a function of the shifted configuration $\tau_{-y}\eta_{0}$ which we take expectation of against $\mathfrak{f}_{S+\mathfrak{t}_{\mathfrak{s}},N}\d\mu_{0,\Z}$. By the last display, \eqref{eq:D1B2A22} and Fubini's theorem, as in the one-block step in \cite{GPV} we get the following where the fourth line follows by the path expectation in the line before being independent of integration/summation variables:
\small\begin{align}
\int_{0}^{\mathfrak{t}_{\mathfrak{f}}}\wt{\sum}_{y\in\mathbb{T}_{N}}\E|\mathscr{C}_{\mathfrak{t}_{\mathrm{av}};\beta_{-},\beta_{+}}^{\mathbf{T},+,\mathfrak{t}_{N,\e},\mathfrak{i}}\mathscr{C}_{N^{\beta_{X}}}^{\mathbf{X},-}(\mathfrak{g}_{S+\mathfrak{t}_{\mathfrak{s}},y})| \ \d S \ &= \ \int_{0}^{\mathfrak{t}_{\mathfrak{f}}}\wt{\sum}_{y\in\mathbb{T}_{N}}\E^{\mu_{0,\Z}}\mathfrak{f}_{S+\mathfrak{t}_{\mathfrak{s}},N} \E_{\eta_{0}}^{\mathrm{path}}|\mathscr{C}_{\mathfrak{t}_{\mathrm{av}};\beta_{-},\beta_{+}}^{\mathbf{T},+,\mathfrak{t}_{N,\e},\mathfrak{i}}\mathscr{C}_{N^{\beta_{X}}}^{\mathbf{X},-}(\mathfrak{g}_{0,y})| \ \d S \\
&= \ \int_{0}^{\mathfrak{t}_{\mathfrak{f}}}\wt{\sum}_{y\in\mathbb{T}_{N}}\E^{\mu_{0,\Z}}\tau_{y}\mathfrak{f}_{S+\mathfrak{t}_{\mathfrak{s}},N} \E_{\eta_{0}}^{\mathrm{path}}|\mathscr{C}_{\mathfrak{t}_{\mathrm{av}};\beta_{-},\beta_{+}}^{\mathbf{T},+,\mathfrak{t}_{N,\e},\mathfrak{i}}\mathscr{C}_{N^{\beta_{X}}}^{\mathbf{X},-}(\mathfrak{g}_{0,0})| \ \d S \\
&= \ \E^{\mu_{0,\Z}}\left(\int_{0}^{\mathfrak{t}_{\mathfrak{f}}}\wt{\sum}_{y\in\mathbb{T}_{N}}\tau_{y}\mathfrak{f}_{S+\mathfrak{t}_{\mathfrak{s}},N} \E_{\eta_{0}}^{\mathrm{path}}|\mathscr{C}_{\mathfrak{t}_{\mathrm{av}};\beta_{-},\beta_{+}}^{\mathbf{T},+,\mathfrak{t}_{N,\e},\mathfrak{i}}\mathscr{C}_{N^{\beta_{X}}}^{\mathbf{X},-}(\mathfrak{g}_{0,0})| \ \d S\right) \\
&= \ \E^{\mu_{0,\Z}}\left(\left(\int_{0}^{\mathfrak{t}_{\mathfrak{f}}}\wt{\sum}_{y\in\mathbb{T}_{N}}\tau_{y}\mathfrak{f}_{S+\mathfrak{t}_{\mathfrak{s}},N} \ \d S\right)\E_{\eta_{0}}^{\mathrm{path}}|\mathscr{C}_{\mathfrak{t}_{\mathrm{av}};\beta_{-},\beta_{+}}^{\mathbf{T},+,\mathfrak{t}_{N,\e},\mathfrak{i}}\mathscr{C}_{N^{\beta_{X}}}^{\mathbf{X},-}(\mathfrak{g}_{0,0})|\right). \label{eq:D1B2A23}
\end{align}\normalsize\normalsize
The average in parentheses in \eqref{eq:D1B2A23} is $\mathfrak{t}_{\mathfrak{f}}\bar{\mathfrak{f}}_{\mathfrak{t}_{\mathfrak{s}},\mathfrak{t}_{\mathfrak{f}},\mathbb{T}_{N}}$ given prior to Lemma \ref{lemma:LE}, so \eqref{eq:D1B2A21} and \eqref{eq:D1B2A23} finish the proof.
\end{proof}
\begin{proof}[Proof of \emph{Lemma \ref{lemma:SoP}}]
We first construct a two-species process, one evolving on $\Omega_{\Z}$ and another on $\Omega_{\mathfrak{I}_{\mathfrak{t}_{\mathrm{av}}}}$. 
\begin{itemize}[leftmargin=*]
\item Consider the original process $\eta_{\bullet}$ with the infinitesimal generator $\mathfrak{S}^{N,!!}$ evolving on $\Omega_{\Z}$. This is the first species, and we specify the symmetric component of the particle random walks is given pathwise by Poisson clocks attached to \emph{bonds}. In particular, a ringing of a bond indicates swapping spins at the two points along that bond. All the other dynamics are pieces of the totally asymmetric part of $\mathfrak{S}^{N,!!}$ and use any Poisson clocks. This construction was used in the proof of Lemma \ref{lemma:EProdSoP} for example.
\item Consider the $\mathfrak{I}_{\mathfrak{t}_{\mathrm{av}}}$-local process $\eta_{\mathfrak{I}_{\mathfrak{t}_{\mathrm{av}}};\bullet,\bullet}$ evolving on $\Omega_{\mathfrak{I}_{\mathfrak{t}_{\mathrm{av}}}}$. This is the second species. We will similarly specify the symmetric component of the particle random walks is given pathwise by Poisson clocks attached to bonds \emph{in} $\mathfrak{I}_{\mathfrak{t}_{\mathrm{av}}}\subseteq\Z$; note in particular $
\inf\mathfrak{I}_{\mathfrak{t}_{\mathrm{av}}}$ and $\sup\mathfrak{I}_{\mathfrak{t}_{\mathrm{av}}}$ are connected with the associated speed $\alpha_{1} \in \R_{>0}$. We now couple symmetric dynamics of the two species by coupling any spin-swap dynamics along any shared bonds. In particular, we will use the same Poisson clock for bonds that appear in both of $\Z$ and $\mathfrak{I}_{\mathfrak{t}_{\mathrm{av}}}$. Moreover, all Poisson clocks corresponding to totally asymmetric jumps across bonds in $\mathfrak{I}_{\mathfrak{t}_{\mathrm{av}}}$ are coupled to Poisson clocks from the first species above that are associated to the same totally asymmetric jump across the same bond in $\Z$ if such a bond is in $\Z$. In other words, we employ the basic coupling for all totally asymmetric jumps that appear in both dynamics. All other Poisson clocks, such as symmetric spin-swap clocks attached to bonds not in the first $\Z$-process and totally asymmetric jump clocks also not in the first $\Z$-process, are constructed in any arbitrary fashion.
\end{itemize}
Assume these two species above exhibit the same spins in $\mathfrak{I}_{\mathfrak{t}_{\mathrm{av}}} \subseteq \Z$ initially. We will show the probability of seeing a discrepancy between these two species inside the neighborhood of radius $N^{\beta_{X}}\log^{50}N$ around $0\in \Z$ before the time $2\mathfrak{t}_{\mathrm{av}}+2\mathfrak{t}_{N,\e}$ is bounded above by $N^{-C}$ up to $C$-dependent factors for any $C \in \R_{>0}$. Before we actually show the previous claim, observe that the claim would provide both of \eqref{eq:SoP1} and \eqref{eq:SoP2}. Indeed, the path-space functionals of interest in both \eqref{eq:SoP1} and \eqref{eq:SoP2} only depend on the trajectories of spins inside the support of $\mathscr{C}_{N^{\beta_{X}}}^{\mathbf{X},-}(\mathfrak{g})$, which is contained in the aforementioned neighborhood around $0 \in \Z$, until time $2\mathfrak{t}_{\mathrm{av}} + 2\mathfrak{t}_{N,\e}$. We clarify the paths taken by these spins in the aforementioned support certainly depend upon the global configuration of spins. Regardless, if the two species agree on these spins until time $2\mathfrak{t}_{\mathrm{av}} + 2\mathfrak{t}_{N,\e}$, then the two functionals agree on this path of the two-species process. If not, we bound the expectations via $|\mathfrak{g}|\lesssim1$ and the $N^{-C}$-probability of such an event.

We now prove the claim about discrepancies between the two species appearing in the neighborhood of radius $N^{\beta_{X}}\log^{50}N$ around $0 \in \Z$ before time $2\mathfrak{t}_{\mathrm{av}} + 2\mathfrak{t}_{N,\e}$. We will argue via large deviations for random walks as in the proof of Lemma \ref{lemma:EProdSoP}. First, for convenience we will let $\mathfrak{r}_{N}=N^{\beta_{X}}\log^{50}N$ denote the radius of the aforementioned neighborhood. Observe $|\mathfrak{I}_{\mathfrak{t}_{\mathrm{av}}}|\gg\mathfrak{r}_{N}$.
\begin{itemize}[leftmargin=*]
\item Initially, there are no discrepancies in $\mathfrak{I}_{\mathfrak{t}_{\mathrm{av}}}\subseteq\Z$ between the two species by construction. Note there certainly might be some outside this block. By construction of the coupling between the two species, discrepancies cannot be made by shared Poisson clocks, but only moved or killed. In particular, any discrepancy in $\mathfrak{I}_{\mathfrak{t}_{\mathrm{av}}}$ must be created by ringing some Poisson clock in at least one of the two species that is located within $\mathfrak{m}$ of the boundary of $\mathfrak{I}_{\mathfrak{t}_{\mathrm{av}}}$ as each particle performs range-$\mathfrak{m}$ random walks. First, we conclude any discrepancy that appears inside the neighborhood of radius $\mathfrak{r}_{N}$ around $0\in \Z$ at any time must originally be a discrepancy that is within $\mathfrak{m}$ of the boundary of $\mathfrak{I}_{\mathfrak{t}_{\mathrm{av}}}$ that propagates into $\mathfrak{r}_{N}$ of $0\in\Z$. Second, the number of discrepancies in $\mathfrak{I}_{\mathfrak{t}_{\mathrm{av}}}$ that can be made until time $2\mathfrak{t}_{\mathrm{av}}+2\mathfrak{t}_{N,\e}$ is bounded by the number of times Poisson clocks at sites within $\mathfrak{m}$ of the boundary of $\mathfrak{I}_{\mathfrak{t}_{\mathrm{av}}}$ can ring before the time $2\mathfrak{t}_{\mathrm{av}}+2\mathfrak{t}_{N,\e} \lesssim N^{20}$. As these Poisson clocks have speed $N^{2}$, this is at most $N^{C_{1}}$ outside an event of exponentially low probability $\mathscr{O}(\exp(-N^{\kappa_{C_{1}}}))$ for any $C_{1}$ arbitrarily large but universal and for $\kappa_{C_{1}}>0$.
\item Take any discrepancy in the previous bullet point, let it start within $\mathfrak{m}$ of the boundary of $\mathfrak{I}_{\mathfrak{t}_{\mathrm{av}}}$, and wait until it is within $\frac15|\mathfrak{I}_{\mathfrak{t}_{\mathrm{av}}}|$ of a radius-$\mathfrak{r}_{N}$ neighborhood of $0\in \Z$. As in the proof of Lemma \ref{lemma:EProdSoP}, the dynamics of such discrepancy is a random walk with a free unsuppressed symmetric component of speed $N^{2}$ and the maximal jump-length $\mathfrak{m} \lesssim 1$, with an environment-dependent drift of order $N^{3/2}$ of maximal jump-length $\mathfrak{m}\lesssim1$, and a random killing mechanism, \emph{at least} as long as such discrepancy stays within $\frac15|\mathfrak{I}_{\mathfrak{t}_{\mathrm{av}}}|$ of the radius-$\mathfrak{r}_{N}$ neighborhood of $0\in\Z$. We emphasize that this free and unsuppressed nature of the symmetric component of the discrepancy random walks comes from coupling bonds and symmetric spin-swaps as opposed to the basic coupling. As $|\mathfrak{I}_{\mathfrak{t}_{\mathrm{av}}}|\gg \mathfrak{r}_{N}$, the distance between the edge of the $\frac15|\mathfrak{I}_{\mathfrak{t}_{\mathrm{av}}}|$-neighborhood and the $\mathfrak{r}_{N}$-neighborhood is order $|\mathfrak{I}_{\mathfrak{t}_{\mathrm{av}}}|$, so the probability the discrepancy arrives within $\mathfrak{r}_{N}$ of $0 \in \Z$ is controlled by the probability that the aforementioned random walk exits a ball of radius order $|\mathfrak{I}_{\mathfrak{t}_{\mathrm{av}}}|$, both these statements restricted to the time-interval $[0,2\mathfrak{t}_{\mathrm{av}}+2\mathfrak{t}_{N,\e}]$. Indeed, we wait for such discrepancy to land within $\mathfrak{r}_{N}$ of $0\in\Z$ through the previous random walk dynamics, or we wait for it to get beyond $\frac15|\mathfrak{I}_{\mathfrak{t}_{\mathrm{av}}}|$ of this neighborhood. If it does the latter, we wait for it to get within $\frac15|\mathfrak{I}_{\mathfrak{t}_{\mathrm{av}}}|$ of this neighborhood again and repeat. But if it does the former, then it must have traveled distance at least order $|\mathfrak{I}_{\mathfrak{t}_{\mathrm{av}}}|$ via the aforementioned random walk dynamics.
\item Continuing with the above bullet point, the probability that this discrepancy lands in the radius-$\mathfrak{r}_{N}$ neighborhood of $0 \in \Z$ is then order $\kappa_{C}N^{-2C}$ for any $C \in \R_{>0}$ courtesy of standard random walk/large deviations estimates because of the additional log-factors in $|\mathfrak{I}_{\mathfrak{t}_{\mathrm{av}}}|$; see Remark \ref{remark:ch2SoP}. We take union bound over all $N^{C_{1}}$ possible discrepancies and choose $C \geq C_{1}$ arbitrarily large but universal to complete the proof of the high-probability discrepancy claim.
\end{itemize}
This completes the proof.
\end{proof}
\begin{proof}[Proof of \emph{Lemma \ref{lemma:D1B2A3}}]
The inner expectation on the LHS of the proposed estimate is a functional on $\Omega_{\mathfrak{I}_{\mathfrak{t}_{\mathrm{av}}}}$ because dependence on the $\mathfrak{I}_{\mathfrak{t}_{\mathrm{av}}}$-stochastic process is only through its initial configuration; we have taken expectation over the associated $\mathfrak{I}_{\mathfrak{t}_{\mathrm{av}}}$-local path-space measure conditioning on its initial configuration. Thus, we may apply the entropy inequality from Lemma \ref{lemma:LE} with $\varphi$ equal to the inner expectation on the LHS of the proposed bound with support $\mathfrak{I}_{\mathfrak{t}_{\mathrm{av}}}$ and with the choice of constant $\kappa = N^{\beta_{-}}$. We reemphasize that we choose $\mathfrak{I}_{N} = \mathbb{T}_{N}$ in Lemma \ref{lemma:LE} which is okay because $\mathbb{T}_{N}$ satisfies the constraints we assumed for the block $\mathfrak{I}_{N}$ stated in Proposition \ref{prop:EProd}. In our application of Lemma \ref{lemma:LE}, we also choose $\mathfrak{t}_{0}=\mathfrak{t}_{\mathfrak{s}}$ here and $T = \mathfrak{t}_{\mathfrak{f}}$. This all gives
\small\begin{align}
\mathfrak{t}_{\mathfrak{f}} \E^{\mu_{0,\Z}}\bar{\mathfrak{f}}_{\mathfrak{t}_{\mathfrak{s}},\mathfrak{t}_{\mathfrak{f}},\mathbb{T}_{N}}^{\mathfrak{I}_{\mathfrak{t}_{\mathrm{av}}}} \E_{\Pi_{\mathfrak{I}_{\mathfrak{t}_{\mathrm{av}}}}\eta_{0}}^{\mathrm{loc}}|\mathscr{C}_{\mathfrak{t}_{\mathrm{av}};\beta_{-},\beta_{+}}^{\mathbf{T},+,\mathfrak{t}_{N,\e},\mathfrak{i}}\mathscr{C}_{N^{\beta_{X}}}^{\mathbf{X},-}(\mathfrak{g}_{\mathfrak{t}_{\mathrm{av}};0,0})| \ &\lesssim \ N^{-\beta_{-}-\frac34+\e}|\mathbb{T}_{N}|^{-1}|\mathfrak{I}_{\mathfrak{t}_{\mathrm{av}}}|^{3} \ + \ \mathfrak{t}_{\mathfrak{f}} {\sup}_{\varrho\in\R} \Psi_{\varrho} \label{eq:D1B2A32}
\end{align}\normalsize\normalsize
where $\e\in\R_{>0}$ is arbitrarily small but universal and for convenience we defined the following equilibrium log-exp term:
\small\begin{align}
\Psi_{\varrho} \ &\overset{\bullet}= \ N^{-\beta_{-}} \log \E^{\mu_{\varrho,\mathfrak{I}_{\mathfrak{t}_{\mathrm{av}}}}^{\mathrm{can}}}\exp\left(N^{\beta_{-}} \E_{\Pi_{\mathfrak{I}_{\mathfrak{t}_{\mathrm{av}}}}\eta_{0}}^{\mathrm{loc}}|\mathscr{C}_{\mathfrak{t}_{\mathrm{av}};\beta_{-},\beta_{+}}^{\mathbf{T},+,\mathfrak{t}_{N,\e},\mathfrak{i}}\mathscr{C}_{N^{\beta_{X}}}^{\mathbf{X},-}(\mathfrak{g}_{\mathfrak{t}_{\mathrm{av}};0,0})|\right).
\end{align}\normalsize\normalsize
The first term on the RHS of \eqref{eq:D1B2A32} is bounded by power-counting. We do this later and focus on the more interesting term $\Psi_{\varrho}$. First observe that by definition of $\mathscr{C}^{\mathbf{T},+}$, in the exponential defining $\Psi_{\varrho}$, the $\E^{\mathrm{loc}}$-term is at most $N^{-\beta_{-}}$ times universal factors; see Definition \ref{definition:averagesAlocal} for  the definition of the $\mathscr{C}^{\mathbf{T},+}$-term. Thus the quantity inside the exponential is uniformly bounded. Using the inequalities $\exp(|x|)\leq 1+\kappa|x|$ for $\kappa\geq0$ and $x\in\R$ both uniformly bounded and $\log(1+|x|)\leq|x|$, we get
\small\begin{align}
\Psi_{\varrho} \ &\leq \ N^{-\beta_{-}}\log\E^{\mu_{\varrho,\mathfrak{I}_{\mathfrak{t}_{\mathrm{av}}}}^{\mathrm{can}}}\left( 1 + \kappa \cdot N^{\beta_{-}}\E_{\Pi_{\mathfrak{I}_{\mathfrak{t}_{\mathrm{av}}}}\eta_{0}}^{\mathrm{loc}}|\mathscr{C}_{\mathfrak{t}_{\mathrm{av}};\beta_{-},\beta_{+}}^{\mathbf{T},+,\mathfrak{t}_{N,\e},\mathfrak{i}}\mathscr{C}_{N^{\beta_{X}}}^{\mathbf{X},-}(\mathfrak{g}_{\mathfrak{t}_{\mathrm{av}};0,0})|\right) \\
&= \ N^{-\beta_{-}}\log\left( 1 + \kappa \cdot N^{\beta_{-}}\E^{\mu_{\varrho,\mathfrak{I}_{\mathfrak{t}_{\mathrm{av}}}}^{\mathrm{can}}}\E_{\Pi_{\mathfrak{I}_{\mathfrak{t}_{\mathrm{av}}}}\eta_{0}}^{\mathrm{loc}}|\mathscr{C}_{\mathfrak{t}_{\mathrm{av}};\beta_{-},\beta_{+}}^{\mathbf{T},+,\mathfrak{t}_{N,\e},\mathfrak{i}}\mathscr{C}_{N^{\beta_{X}}}^{\mathbf{X},-}(\mathfrak{g}_{\mathfrak{t}_{\mathrm{av}};0,0})|\right) \\
&\lesssim \ \E^{\mu_{\varrho,\mathfrak{I}_{\mathfrak{t}_{\mathrm{av}}}}^{\mathrm{can}}}\E_{\Pi_{\mathfrak{I}_{\mathfrak{t}_{\mathrm{av}}}}\eta_{0}}^{\mathrm{loc}}|\mathscr{C}_{\mathfrak{t}_{\mathrm{av}};\beta_{-},\beta_{+}}^{\mathbf{T},+,\mathfrak{t}_{N,\e},\mathfrak{i}}\mathscr{C}_{N^{\beta_{X}}}^{\mathbf{X},-}(\mathfrak{g}_{\mathfrak{t}_{\mathrm{av}};0,0})|. \label{eq:D1B2A33}
\end{align}\normalsize\normalsize
To estimate the double expectation in \eqref{eq:D1B2A33}, let us first ease notation and let $\wt{\E}$ denote the iterated expectation in \eqref{eq:D1B2A33}. We now recall definition of the $\mathscr{C}^{\mathbf{T},+}$-term in \eqref{eq:D1B2A33} within Definition \ref{definition:averagesAlocal} to be the time-average $\mathscr{A}^{\mathbf{T},+}$ multiplied by two indicator functions. For the sake of getting an upper bound, we drop the upper-bound indicator function and apply the Cauchy-Schwarz inequality to get the following. Below, we emphasize that because the canonical ensemble initial measure we take in definition of $\wt{\E}$ is an invariant measure for the $\mathfrak{I}_{\mathfrak{t}_{\mathrm{av}}}$-local process we are taking expectation with respect to in $\wt{\E}$, this tells us that time-shifts do not change the law of functionals of the $\mathfrak{I}_{\mathfrak{t}_{\mathrm{av}}}$-local process as long as the $\mathfrak{I}_{\mathfrak{t}_{\mathrm{av}}}$-local process starts with any canonical ensemble initial measure. We deduce this ``invariance-under-time-shift" upon writing the $\wt{\E}$-expectation of any time-shifted functional as an expectation with respect to the path-space measure induced by the $\mathfrak{I}_{\mathfrak{t}_{\mathrm{av}}}$-local process with an initial measure given by the law of the particle system after evolving for a time equal to the time-shift. We can do this with the Markov property for the $\mathfrak{I}_{\mathfrak{t}_{\mathrm{av}}}$-local process. But canonical ensembles are invariant measures, so the ``time-shifted" initial measure is the same canonical ensemble. This lets us forget the time-shift $\mathfrak{t}_{N,\e}\in\R_{\geq0}$ in $\wt{\E}$-expectations below until \eqref{eq:D1B2A34}. More precisely, we note the following.
\begin{itemize}[leftmargin=*]
\item The laws of $\mathbf{1}[\mathscr{G}_{\mathfrak{t}_{\mathrm{av}};\geq;S,0}^{\beta_{+}}]$ and $\mathscr{A}_{\mathfrak{t}_{\mathrm{av}}}^{\mathbf{T},+}\mathscr{C}_{N^{\beta_{X}}}^{\mathbf{X},-}(\mathfrak{g}_{\mathfrak{t}_{\mathrm{av}};S,0})$ are $S$-independent under $\wt{\E}$. See Definition \ref{definition:averagesAlocal} for their definitions.
\end{itemize}
Thus at the level of $\wt{\E}$, the time-shift $\mathfrak{t}_{N,\e}$ is irrelevant. For any $\mathfrak{i}\in\{1,2\}$, we get the following with justification given after:
\small\begin{align}
\wt{\E}|\mathscr{C}_{\mathfrak{t}_{\mathrm{av}};\beta_{-},\beta_{+}}^{\mathbf{T},+,\mathfrak{t}_{N,\e},\mathfrak{i}}\mathscr{C}_{N^{\beta_{X}}}^{\mathbf{X},-}(\mathfrak{g}_{\mathfrak{t}_{\mathrm{av}};0,0})| \ &\lesssim \ \left(\wt{\E}|\mathscr{A}_{\mathfrak{t}_{\mathrm{av}}}^{\mathbf{T},+}\mathscr{C}_{N^{\beta_{X}}}^{\mathbf{X},-}(\mathfrak{g}_{\mathfrak{t}_{\mathrm{av}};0,0})|^{2}\right)^{1/2}\left(\wt{\E}\mathbf{1}[\mathscr{G}_{\mathfrak{t}_{\mathrm{av}};\geq;0,0}^{\beta_{+}}] \right)^{1/2} \\
&\lesssim \ \left(\wt{\E}|\mathscr{A}_{\mathfrak{t}_{\mathrm{av}}}^{\mathbf{T},+}\mathscr{C}_{N^{\beta_{X}}}^{\mathbf{X},-}(\mathfrak{g}_{\mathfrak{t}_{\mathrm{av}};0,0})|^{2}\right)^{1/2}\left(N^{2\beta_{+}}\wt{\E}{\sup}_{0\leq\mathfrak{t}\leq\mathfrak{t}_{\mathrm{av}}} \mathfrak{t}^{2}\mathfrak{t}_{\mathrm{av}}^{-2}|\mathscr{A}_{\mathfrak{t}}^{\mathbf{T},+}\mathscr{C}_{N^{\beta_{X}}}^{\mathbf{X},-}(\mathfrak{g}_{\mathfrak{t}_{\mathrm{av}};0,0})|^{2}\right)^{1/2} \\
&\lesssim \ N^{\beta_{+}}\wt{\E}{\sup}_{0\leq\mathfrak{t}\leq\mathfrak{t}_{\mathrm{av}}} \mathfrak{t}^{2}\mathfrak{t}_{\mathrm{av}}^{-2}|\mathscr{A}_{\mathfrak{t}}^{\mathbf{T},+}\mathscr{C}_{N^{\beta_{X}}}^{\mathbf{X},-}(\mathfrak{g}_{\mathfrak{t}_{\mathrm{av}};0,0})|^{2}. \label{eq:D1B2A34}
\end{align}\normalsize\normalsize
To get \eqref{eq:D1B2A34} from the previous line, we control a process by its running supremum (take $\mathfrak{t}=\mathfrak{t}_{\mathrm{av}}$ on the RHS of this next bound):
\small\begin{align}
|\mathscr{A}_{\mathfrak{t}_{\mathrm{av}}}^{\mathbf{T},+}\mathscr{C}_{N^{\beta_{X}}}^{\mathbf{X},-}(\mathfrak{g}_{\mathfrak{t}_{\mathrm{av}};0,0})|^{2} \ &\leq \ {\sup}_{0\leq\mathfrak{t}\leq\mathfrak{t}_{\mathrm{av}}} \mathfrak{t}^{2}\mathfrak{t}_{\mathrm{av}}^{-2}|\ \mathscr{A}_{\mathfrak{t}}^{\mathbf{T},+}\mathscr{C}_{N^{\beta_{X}}}^{\mathbf{X},-}(\mathfrak{g}_{\mathfrak{t}_{\mathrm{av}};0,0})|^{2} \label{eq:D1B2A34.5}
\end{align}\normalsize\normalsize
The bound preceding \eqref{eq:D1B2A34} follows from the Chebyshev inequality for second moments applied to the probability of $\mathscr{G}_{\mathfrak{t}_{\mathrm{av}};\geq;0,0}^{\beta_{+}}$. To estimate the last supremum we would like to apply the Kipnis-Varadhan time-average estimate in Lemma \ref{lemma:KV}. Indeed, this would be a major advantage of reducing our estimates to equilibrium calculations. However, the functional that we are time-averaging via $\mathscr{A}^{\mathbf{T},+}$, which we emphasize has no indicator functions attached to it anymore in \eqref{eq:D1B2A34}, is not a spatial-average of pseudo-gradients with disjoint support. The obstruction to this being true is the cutoff present in $\mathscr{C}^{\mathbf{X},-}$. We will remove this cutoff. First, observe the following courtesy of definition of $\mathscr{C}^{\mathbf{X},-}$ from Definition \ref{definition:averagesAlocal} and the uniform boundedness of $\mathfrak{g}$:
\small\begin{align}
|\mathscr{C}^{\mathbf{X},-}_{N^{\beta_{X}}}(\mathfrak{g}_{\mathfrak{t}_{\mathrm{av}};\mathfrak{r},0}) - \mathscr{A}^{\mathbf{X},-}_{N^{\beta_{X}}}(\mathfrak{g}_{\mathfrak{t}_{\mathrm{av}};\mathfrak{r},0})| \ &\lesssim \ \mathbf{1}[(\mathscr{E}_{\mathfrak{t}_{\mathrm{av}};\mathfrak{r}}^{\mathbf{X},-})^{C}]. \label{eq:D1B2A35}
\end{align}\normalsize\normalsize
We will use \eqref{eq:D1B2A35} to replace the $\mathscr{C}^{\mathbf{X},-}$ term \eqref{eq:D1B2A34} by the spatial-average $\mathscr{A}^{\mathbf{X},-}$ without cutoff. In particular, we will employ the following estimate with universal implied constant uniformly over $0\leq\mathfrak{t}\leq\mathfrak{t}_{\mathrm{av}}$. We will use notation we define afterwards:
\begin{align*}
\mathfrak{t}^{2}\mathfrak{t}_{\mathrm{av}}^{-2}|\mathscr{A}_{\mathfrak{t}}^{\mathbf{T},+}\mathscr{C}_{N^{\beta_{X}}}^{\mathbf{X},-}(\mathfrak{g}_{\mathfrak{t}_{\mathrm{av}};0,0})|^{2} \ &= \ \mathfrak{t}^{2}\mathfrak{t}_{\mathrm{av}}^{-2}|\mathscr{A}_{\mathfrak{t}}^{\mathbf{T},+}\mathscr{A}_{N^{\beta_{X}}}^{\mathbf{X},-}(\mathfrak{g}_{\mathfrak{t}_{\mathrm{av}};0,0}) \ + \ \mathscr{A}_{\mathfrak{t}}^{\mathbf{T},+}\mathscr{C}_{N^{\beta_{X}}}^{\mathbf{X},-}(\mathfrak{g}_{\mathfrak{t}_{\mathrm{av}};0,0}) - \mathscr{A}_{\mathfrak{t}}^{\mathbf{T},+}\mathscr{A}_{N^{\beta_{X}}}^{\mathbf{X},-}(\mathfrak{g}_{\mathfrak{t}_{\mathrm{av}};0,0})|^{2} \\
&\lesssim \ \mathfrak{t}^{2}\mathfrak{t}_{\mathrm{av}}^{-2}|\mathscr{A}_{\mathfrak{t}}^{\mathbf{T},+}\mathscr{A}_{N^{\beta_{X}}}^{\mathbf{X},-}(\mathfrak{g}_{\mathfrak{t}_{\mathrm{av}};0,0})|^{2} \ + \ \mathfrak{t}^{2}\mathfrak{t}_{\mathrm{av}}^{-2}|\mathscr{A}_{\mathfrak{t}}^{\mathbf{T},+}\mathscr{C}_{N^{\beta_{X}}}^{\mathbf{X},-}(\mathfrak{g}_{\mathfrak{t}_{\mathrm{av}};0,0}) - \mathscr{A}_{\mathfrak{t}}^{\mathbf{T},+}\mathscr{A}_{N^{\beta_{X}}}^{\mathbf{X},-}(\mathfrak{g}_{\mathfrak{t}_{\mathrm{av}};0,0})|^{2} \\
&\lesssim \ \mathfrak{t}^{2}\mathfrak{t}_{\mathrm{av}}^{-2}|\mathscr{A}_{\mathfrak{t}}^{\mathbf{T},+}\mathscr{A}_{N^{\beta_{X}}}^{\mathbf{X},-}(\mathfrak{g}_{\mathfrak{t}_{\mathrm{av}};0,0})|^{2} \ + \ \mathfrak{t}^{2}\mathfrak{t}_{\mathrm{av}}^{-2}|\mathscr{A}_{\mathfrak{t}}^{\mathbf{T},+}(\mathbf{1}[(\mathscr{E}_{\mathfrak{t}_{\mathrm{av}};0}^{\mathbf{X},-})^{C}])|^{2}.
\end{align*}
The second bound follows by the elementary inequality $|a+b|^{2} \lesssim |a|^{2}+|b|^{2}$ for all $a,b\in\R$. The last bound follows by \eqref{eq:D1B2A35}, where in the last bound the term $\mathscr{A}_{\mathfrak{t}}^{\mathbf{T},+}(\mathbf{1}[(\mathscr{E}_{\mathfrak{t}_{\mathrm{av}};\mathfrak{r}}^{\mathbf{X},-})^{C}])$ is the time-average over $[0,\mathfrak{t}]$ of the random indicator function $\mathbf{1}[(\mathscr{E}_{\mathfrak{t}_{\mathrm{av}};\mathfrak{r}}^{\mathbf{X},-})^{C}]$. To bound the second term in the last line, we use the Cauchy-Schwarz inequality with respect to the time \emph{integral} $\mathfrak{t}\mathscr{A}_{\mathfrak{t}}^{\mathbf{T},+}$:
\begin{align*}
\mathfrak{t}^{2}\mathfrak{t}_{\mathrm{av}}^{-2}|\mathscr{A}_{\mathfrak{t}}^{\mathbf{T},+}(\mathbf{1}[(\mathscr{E}_{\mathfrak{t}_{\mathrm{av}};0}^{\mathbf{X},-})^{C}])|^{2} \ = \ \mathfrak{t}_{\mathrm{av}}^{-2}|\left(\int_{0}^{\mathfrak{t}}\mathbf{1}[(\mathscr{E}_{\mathfrak{t}_{\mathrm{av}};\mathfrak{r}}^{\mathbf{X},-})^{C}] \ \d\mathfrak{r}\right)|^{2} \ \leq \ \mathfrak{t}_{\mathrm{av}}^{-2}\mathfrak{t}\int_{0}^{\mathfrak{t}}\mathbf{1}[(\mathscr{E}_{\mathfrak{t}_{\mathrm{av}};\mathfrak{r}}^{\mathbf{X},-})^{C}] \ \d\mathfrak{r} \ &\leq \ \mathfrak{t}_{\mathrm{av}}^{-1}\int_{0}^{\mathfrak{t}_{\mathrm{av}}}\mathbf{1}[(\mathscr{E}_{\mathfrak{t}_{\mathrm{av}};\mathfrak{r}}^{\mathbf{X},-})^{C}] \ \d\mathfrak{r}.
\end{align*}
The first identity follows by cancelling the $\mathfrak{t}^{-2}$ factor in the square of the time-average on the far LHS with the $\mathfrak{t}^{2}$-factor therein. The second estimate follows from the Cauchy-Schwarz inequality with respect to the time-integral. The third estimate follows by extending the integration domain $[0,\mathfrak{t}]\to[0,\mathfrak{t}_{\mathrm{av}}]$ as $0\leq\mathfrak{t}\leq\mathfrak{t}_{\mathrm{av}}$ and thus noting $\mathfrak{t}_{\mathrm{av}}^{-1}\mathfrak{t}\leq1$. The last two displays give
\small\begin{align}
\wt{\E}{\sup}_{0\leq\mathfrak{t}\leq\mathfrak{t}_{\mathrm{av}}} \mathfrak{t}^{2}\mathfrak{t}_{\mathrm{av}}^{-2}|\mathscr{A}_{\mathfrak{t}}^{\mathbf{T},+}\mathscr{C}_{N^{\beta_{X}}}^{\mathbf{X},-}(\mathfrak{g}_{\mathfrak{t}_{\mathrm{av}};0,0})|^{2} \ &\lesssim \ \wt{\E}{\sup}_{0\leq\mathfrak{t}\leq\mathfrak{t}_{\mathrm{av}}} \mathfrak{t}^{2}\mathfrak{t}_{\mathrm{av}}^{-2}|\mathscr{A}_{\mathfrak{t}}^{\mathbf{T},+}\mathscr{A}_{N^{\beta_{X}}}^{\mathbf{X},-}(\mathfrak{g}_{\mathfrak{t}_{\mathrm{av}};0,0})|^{2} \ + \ \mathfrak{t}_{\mathrm{av}}^{-1}\int_{0}^{\mathfrak{t}_{\mathrm{av}}}\wt{\E}\mathbf{1}[(\mathscr{E}_{\mathfrak{t}_{\mathrm{av}};\mathfrak{r}}^{\mathbf{X},-})^{C}]\d\mathfrak{r}. \label{eq:D1B2A36}
\end{align}\normalsize\normalsize
We observe the final expectation in the second term of \eqref{eq:D1B2A36} is the probability of an event in $\Omega_{\mathfrak{I}_{\mathfrak{t}_{\mathrm{av}}}}$ under a canonical ensemble. Indeed, the event therein depends on only statistics at time $\mathfrak{r}$ under the $\mathfrak{I}_{\mathfrak{t}_{\mathrm{av}}}$-local exclusion process, and the canonical ensemble initial measure is invariant for this exclusion process. In particular, it is the probability under canonical ensemble measure that the spatial-average of $N^{\beta_{X}}$-many pseudo-gradients with disjoint support exceeds its natural ``CLT" bound. Combining \eqref{eq:D1B2A36} with this observation and the canonical ensemble estimate of Lemma \ref{lemma:LDP} with $\varphi_{j} = \tau_{-7j\mathfrak{m}}\mathfrak{g}_{0,0}$ and $J = N^{\beta_{X}}$ gives the bound
\small\begin{align}
\wt{\E}{\sup}_{0\leq\mathfrak{t}\leq\mathfrak{t}_{\mathrm{av}}} \mathfrak{t}^{2}\mathfrak{t}_{\mathrm{av}}^{-2}|\mathscr{A}_{\mathfrak{t}}^{\mathbf{T},+}\mathscr{C}_{N^{\beta_{X}}}^{\mathbf{X},-}(\mathfrak{g}_{\mathfrak{t}_{\mathrm{av}};0,0})|^{2} \ &\lesssim \ \wt{\E}{\sup}_{0\leq\mathfrak{t}\leq\mathfrak{t}_{\mathrm{av}}} \mathfrak{t}^{2}\mathfrak{t}_{\mathrm{av}}^{-2}|\mathscr{A}_{\mathfrak{t}}^{\mathbf{T},+}\mathscr{A}_{N^{\beta_{X}}}^{\mathbf{X},-}(\mathfrak{g}_{\mathfrak{t}_{\mathrm{av}};0,0})|^{2} \ + \ N^{-100}. \label{eq:D1B2A37}
\end{align}\normalsize\normalsize
The error term $N^{-100}$ in \eqref{eq:D1B2A37} can be replaced by something exponentially small if desired. Anyway, the first term on the RHS of \eqref{eq:D1B2A37} is now a space-time average of pseudo-gradients with disjoint support free of any indicator functions or cutoffs. Thus may use Lemma \ref{lemma:KV} with $\varphi$ the $\mathscr{A}^{\mathbf{X},-}$ spatial average and Lemma \ref{lemma:SpectralH-1} with $J=N^{\beta_{X}}$ and $\varphi_{j}=\tau_{-7j\mathfrak{m}}\mathfrak{g}_{0,0}$ with support $\mathfrak{I}_{j}$; note that in the following estimate, we use the uniform boundedness of $|\mathfrak{I}_{j}|$-cardinalities and of the shifted pseudo-gradients $\varphi_{j}$:
\small\begin{align}
\wt{\E}\sup_{0\leq\mathfrak{t}\leq\mathfrak{t}_{\mathrm{av}}} \mathfrak{t}^{2}\mathfrak{t}_{\mathrm{av}}^{-2}|\mathscr{A}_{\mathfrak{t}}^{\mathbf{T},+}\mathscr{A}_{N^{\beta_{X}}}^{\mathbf{X},-}(\mathfrak{g}_{\mathfrak{t}_{\mathrm{av}};0,0})|^{2} \ \lesssim \ \mathfrak{t}_{\mathrm{av}}^{-1} \|\mathscr{A}_{N^{\beta_{X}}}^{\mathbf{X},-}(\mathfrak{g}_{\mathfrak{t}_{\mathrm{av}};0,0})\|_{\dot{\mathbf{H}}^{-1}_{\varrho,\mathfrak{I}_{\mathfrak{t}_{\mathrm{av}}}}}^{2} \ &\lesssim \ \mathfrak{t}_{\mathrm{av}}^{-1}N^{-\beta_{X}} \sup_{j=1,\ldots,J}\sup_{\varrho_{j}\in\R}\|\varphi_{j}\|_{\dot{\mathbf{H}}^{-1}_{\varrho_{j},\mathfrak{I}_{j}}}^{2} \nonumber \\
&\lesssim \ N^{-2-\beta_{X}}\mathfrak{t}_{\mathrm{av}}^{-1}. \label{eq:D1B2A38}
\end{align}\normalsize\normalsize
We combine \eqref{eq:D1B2A32}, \eqref{eq:D1B2A33}, \eqref{eq:D1B2A34}, \eqref{eq:D1B2A37}, and \eqref{eq:D1B2A38} to get the following in which we use $\mathfrak{t}_{\mathfrak{f}}N^{\beta_{+}}N^{-100}\lesssim N^{-97}$:
\small\begin{align}
\mathfrak{t}_{\mathfrak{f}} \E^{\mu_{0,\Z}}\bar{\mathfrak{f}}_{\mathfrak{t}_{\mathfrak{s}},\mathfrak{t}_{\mathfrak{f}},\mathbb{T}_{N}}^{\mathfrak{I}_{\mathfrak{t}_{\mathrm{av}}}} \E_{\Pi_{\mathfrak{I}_{\mathfrak{t}_{\mathrm{av}}}}\eta_{0}}^{\mathrm{loc}}|\mathscr{C}_{\mathfrak{t}_{\mathrm{av}};\beta_{-},\beta_{+}}^{\mathbf{T},+,\mathfrak{t}_{N,\e},\mathfrak{i}}\mathscr{C}_{N^{\beta_{X}}}^{\mathbf{X},-}(\mathfrak{g}_{\mathfrak{t}_{\mathrm{av}};0,0})| \ &\lesssim \ N^{-\beta_{-}-\frac34+\e}|\mathbb{T}_{N}|^{-1}|\mathfrak{I}_{\mathfrak{t}_{\mathrm{av}}}|^{3} \ + \ \mathfrak{t}_{\mathfrak{f}}N^{-2-\beta_{X}+\beta_{+}}\mathfrak{t}_{\mathrm{av}}^{-1} \ + \ N^{-97}.
\end{align}\normalsize\normalsize
The LHS of this last estimate is the LHS of the proposed estimate in the current lemma but without the $N$-dependent prefactor in the proposed estimate. Putting this $N$-dependent prefactor back in the above and power-counting completes the proof upon recalling that $|\mathfrak{I}_{\mathfrak{t}_{\mathrm{av}}}| \lesssim_{\e}N^{1+\e}\mathfrak{t}_{\mathrm{av}}^{1/2} + N^{3/2+\e}\mathfrak{t}_{\mathrm{av}}+ N^{\beta_{X}+\e}$ for any $\e\in\R_{>0}$ and recalling that $|\mathbb{T}_{N}|\gtrsim N^{5/4+\e_{\mathrm{cpt}}}$. We appeal to Definition \ref{definition:tAvBlock} and the beginning of Section \ref{section:Ctify} for these length-scale estimates, respectively.
\end{proof}
\begin{proof}[Proof of \emph{Lemma \ref{lemma:D1B2B1}}]
The first estimate follows from the proof for Lemma \ref{lemma:D1B2A1} but choosing $\mathfrak{t}_{N}\overset{\bullet}=T-N^{-\beta_{X}+\beta_{-}-\e}$ therein. To prove the second estimate, we follow the proof of Lemma \ref{lemma:D1B2A2} upon replacing $\mathscr{C}^{\mathbf{X},-}(\mathfrak{g}) \to \wt{\mathfrak{g}}^{\mathfrak{l}}$ and using the same argument. We note that in following the argument of Lemma \ref{lemma:D1B2A2}, we start with the $\|\|_{\mathfrak{t}_{\mathfrak{f}};\mathbb{T}_{N}}$-bound obtained in the first estimate of the current lemma with $\varphi$ equal to the $\mathscr{C}^{\mathbf{T},+,\mathfrak{t}_{N,\e},\mathfrak{i}}$-term, that certainly satisfies the constraints for $\varphi$ in the first estimate of the current lemma, instead of the estimate of Lemma \ref{lemma:D1B2A1}. Precisely, after doing this and then taking expectation we get, for any $\e\in\R_{>0}$,
\small\begin{align}
\E\|\bar{\mathbf{H}}_{T,x}^{N}(N^{\beta_{X}}|\mathscr{C}_{\mathfrak{t}_{\mathrm{av}};\beta_{-},\beta_{+}}^{\mathbf{T},+,\mathfrak{t}_{N,\e},\mathfrak{i}}(\wt{\mathfrak{g}}_{S+\mathfrak{t}_{\mathfrak{s}},y}^{\mathfrak{l}})|)\|_{\mathfrak{t}_{\mathfrak{f}};\mathbb{T}_{N}} \ &\lesssim \ N^{\frac32\beta_{X}+\frac14+\e_{\mathrm{cpt}}+\frac12\e-\frac12\beta_{-}}\int_{0}^{\mathfrak{t}_{\mathfrak{f}}}\wt{\sum}_{y\in\mathbb{T}_{N}}\E|\mathscr{C}_{\mathfrak{t}_{\mathrm{av}};\beta_{-},\beta_{+}}^{\mathbf{T},+,\mathfrak{t}_{N,\e},\mathfrak{i}}(\wt{\mathfrak{g}}_{S+\mathfrak{t}_{\mathfrak{s}},y}^{\mathfrak{l}})| \ \d S \ + \ N^{-\e}.
\end{align}\normalsize\normalsize
At this point, we use the same procedure of decomposing expectations via path-space measure expectations and initial configuration expectations along with the space-time-shift invariance of the law of the path-space measure.
\end{proof}
\begin{proof}[Proof of \emph{Lemma \ref{lemma:SoP2}}]
This follows by the same coupling construction as was used to establish Lemma \ref{lemma:SoP}. Indeed, note the support of $\wt{\mathfrak{g}}^{\mathfrak{l}}$ is also bounded above in size by $N^{\beta_{X}}$ times universal factors, so in the proof of Lemma \ref{lemma:SoP2} we are also interested here in looking at when discrepancies in the coupling from the proof of Lemma \ref{lemma:SoP} get within $N^{\beta_{X}}\log^{50}N$ of $x\in\Z$.
\end{proof}
\begin{proof}[Proof of \emph{Lemma \ref{lemma:D1B2B3}}]
Observe the inner expectation on the LHS of the proposed estimate is a functional on the configuration space $\Omega_{\mathfrak{I}_{\mathfrak{t}_{\mathrm{av}}}}$ since dependence on the path-space measure induced by the $\mathfrak{I}_{\mathfrak{t}_{\mathrm{av}}}$-local process goes away given we take expectation over this path-space measure \emph{conditioning} on the initial particle configuration. This is actually the same observation we made at the beginning of the proof of Lemma \ref{lemma:D1B2A3}. The outer expectation on the LHS of the proposed estimate is thus an expectation of a $\mathfrak{I}_{\mathfrak{t}_{\mathrm{av}}}$-statistic with respect to the $\bar{\mathfrak{f}}$ probability density which we recall is a space-time averaged probability density. We may therefore apply the local equilibrium estimate/entropy inequality from Lemma \ref{lemma:LE} with $\kappa = N^{\beta_{-}}$, with $\mathfrak{I}_{N}=\mathbb{T}_{N}$, and with $\varphi$ the inner expectation on the LHS of the proposed estimate. We emphasize that choosing $\mathfrak{I}_{N}=\mathbb{T}_{N}$ in this application of Lemma \ref{lemma:LE} is okay as $\mathbb{T}_{N}$ satisfies assumptions of Proposition \ref{prop:EProd}. We also emphasize our choice of $\varphi$ has support $\mathfrak{I}_{\mathfrak{t}_{\mathrm{av}}}$. We finally note that in applying Lemma \ref{lemma:LE} we take $\mathfrak{t}_{0}$ therein to be $\mathfrak{t}_{\mathfrak{s}}$ here and $T$ therein to be $\mathfrak{t}_{\mathfrak{f}}$ here. We ultimately get
\small\begin{align}
\mathfrak{t}_{\mathfrak{f}} \E^{\mu_{0,\Z}}\bar{\mathfrak{f}}_{\mathfrak{t}_{\mathfrak{s}},\mathfrak{t}_{\mathfrak{f}},\mathbb{T}_{N}}^{\mathfrak{I}_{\mathfrak{t}_{\mathrm{av}}}}\E_{\Pi_{\mathfrak{I}_{\mathfrak{t}_{\mathrm{av}}}}\eta_{0}}^{\mathrm{loc}}|\mathscr{C}_{\mathfrak{t}_{\mathrm{av}};\beta_{-},\beta_{+}}^{\mathbf{T},+,\mathfrak{t}_{N,\e},\mathfrak{i}}(\wt{\mathfrak{g}}_{\mathfrak{t}_{\mathrm{av}};0,0}^{\mathfrak{l}})| \ &\lesssim \ N^{-\beta_{-}-\frac34+\e}|\mathbb{T}_{N}|^{-1}|\mathfrak{I}_{\mathfrak{t}_{\mathrm{av}}}|^{3} \ + \ \mathfrak{t}_{\mathfrak{f}}\sup_{\varrho\in\R}\wt{\Psi}_{\varrho} \label{eq:D1B2B32}
\end{align}\normalsize\normalsize
where again $\e\in\R_{>0}$ is arbitrarily small but universal and we defined the equilibrium log-exp term
\small\begin{align}
\wt{\Psi}_{\varrho} \ &\overset{\bullet}= \ N^{-\beta_{-}}\log \E^{\mu_{\varrho,\mathfrak{I}_{\mathfrak{t}_{\mathrm{av}}}}^{\mathrm{can}}}\exp\left(N^{\beta_{-}}\E_{\Pi_{\mathfrak{I}_{\mathfrak{t}_{\mathrm{av}}}}\eta_{0}}^{\mathrm{loc}}|\mathscr{C}_{\mathfrak{t}_{\mathrm{av}};\beta_{-},\beta_{+}}^{\mathbf{T},+,\mathfrak{t}_{N,\e},\mathfrak{i}}(\wt{\mathfrak{g}}_{\mathfrak{t}_{\mathrm{av}};0,0}^{\mathfrak{l}})|\right).
\end{align}\normalsize\normalsize
The first term on the RHS of \eqref{eq:D1B2B32} is handled again by power-counting which we do at the end of the proof. The analysis of the second term on the RHS of \eqref{eq:D1B2B32} is addressed in similar fashion as the $\Psi_{\varrho}$-terms from the proof of Proposition \ref{prop:D1B2A} except now it is actually a little easier as we illustrate. We recall from Definition \ref{definition:averagesBlocal} that the expectation inside the exponential defining $\wt{\Psi}$-quantities is bounded by $N^{-\beta_{-}}$ times universal factors by definition. In particular, the term inside of the exponential defining $\wt{\Psi}_{\varrho}$ is uniformly bounded. This is identical to what we observed for $\Psi_{\varrho}$-terms in the proof for Proposition \ref{prop:D1B2A}. Thus, we use the exponential and logarithm bounds given prior to \eqref{eq:D1B2A33} in the proof of Proposition \ref{prop:D1B2A} to get, with universal implied constant,
\small\begin{align}
\wt{\Psi}_{\varrho} \ &\lesssim \ \E^{\mu_{\varrho,\mathfrak{I}_{\mathfrak{t}_{\mathrm{av}}}}^{\mathrm{can}}}\E_{\Pi_{\mathfrak{I}_{\mathfrak{t}_{\mathrm{av}}}}\eta_{0}}^{\mathrm{loc}}|\mathscr{C}_{\mathfrak{t}_{\mathrm{av}};\beta_{-},\beta_{+}}^{\mathbf{T},+,\mathfrak{t}_{N,\e},\mathfrak{i}}(\wt{\mathfrak{g}}_{\mathfrak{t}_{\mathrm{av}};0,0}^{\mathfrak{l}})|. \label{eq:D1B2B33}
\end{align}\normalsize\normalsize
At this we point we will perform equilibrium calculations like with the proof for Lemma \ref{lemma:D1B2A3} after \eqref{eq:D1B2A33}. First, we follow the proof of \eqref{eq:D1B2A34} via the Cauchy-Schwarz inequality to estimate the last expectation from the RHS of \eqref{eq:D1B2B33} in terms of a second moment of the time-average $\mathscr{A}^{\mathbf{T},+}$ without any cutoff and the probability of the lower-bound-cutoff-event defining $\mathscr{F}^{\beta_{+}}$ from Definition \ref{definition:averagesBlocal}. Indeed observe that $\mathscr{C}^{\mathbf{T},+,\mathfrak{t}_{N,\e},\mathfrak{i}}$ time-averaging is, in absolute value, controlled by $\mathscr{A}^{\mathbf{T},+}$ time-averaging times the indicator function of this last $\mathscr{F}^{\beta_{+}}$-event. To estimate the probability of this lower-bound-cutoff-event $\mathscr{F}^{\beta_{+}}$ we employ the Chebyshev inequality for a second moment and observe that at the level of expectations the time-shift $\mathfrak{t}_{N,\e}\in\R_{\geq0}$ is irrelevant. This is because the canonical ensemble initial measure in our iterated expectation within the RHS of \eqref{eq:D1B2B33} is invariant for the $\mathfrak{I}_{\mathfrak{t}_{\mathrm{av}}}$-local process, so the time-shift $\mathfrak{t}_{N,\e}$ does not affect the law of $\mathscr{A}^{\mathbf{T},+}$-terms. If $\wt{\E}$ is the double expectation in\eqref{eq:D1B2B33}, we obtain the following analog of the estimate \eqref{eq:D1B2A34} from the proof of Lemma \ref{lemma:D1B2A3}:
\small\begin{align}
\wt{\E}|\mathscr{C}_{\mathfrak{t}_{\mathrm{av}};\beta_{-},\beta_{+}}^{\mathbf{T},+,\mathfrak{t}_{N,\e},\mathfrak{i}}(\wt{\mathfrak{g}}_{\mathfrak{t}_{\mathrm{av}};0,0}^{\mathfrak{l}})| \ &\lesssim \ \left(\wt{\E}|\mathscr{A}_{\mathfrak{t}_{\mathrm{av}}}^{\mathbf{T},+}(\wt{\mathfrak{g}}_{\mathfrak{t}_{\mathrm{av}};0,0}^{\mathfrak{l}})|^{2}\right)^{\frac12}\left(\wt{\E}\mathbf{1}[\mathscr{F}_{\mathfrak{t}_{\mathrm{av}};\geq;0,0}^{\beta_{+}}]\right)^{\frac12} \\
&\lesssim \ \left(\wt{\E}|\mathscr{A}_{\mathfrak{t}_{\mathrm{av}}}^{\mathbf{T},+}(\wt{\mathfrak{g}}_{\mathfrak{t}_{\mathrm{av}};0,0}^{\mathfrak{l}})|^{2}\right)^{\frac12}\left(N^{2\beta_{+}}\wt{\E}\sup_{0\leq\mathfrak{t}\leq\mathfrak{t}_{\mathrm{av}}}\mathfrak{t}^{2}\mathfrak{t}_{\mathrm{av}}^{-2}|\mathscr{A}_{\mathfrak{t}}^{\mathbf{T},+}(\wt{\mathfrak{g}}_{\mathfrak{t}_{\mathrm{av}};0,0}^{\mathfrak{l}})|^{2}\right)^{\frac12} \\
&\lesssim \ N^{\beta_{+}}\wt{\E}\sup_{0\leq\mathfrak{t}\leq\mathfrak{t}_{\mathrm{av}}}\mathfrak{t}^{2}\mathfrak{t}_{\mathrm{av}}^{-2}|\mathscr{A}_{\mathfrak{t}}^{\mathbf{T},+}(\wt{\mathfrak{g}}_{\mathfrak{t}_{\mathrm{av}};0,0}^{\mathfrak{l}})|^{2}. \label{eq:D1B2B34}
\end{align}\normalsize\normalsize
To get \eqref{eq:D1B2B34} from the the preceding line, as with the proof of \eqref{eq:D1B2A34} we used a bound for a process by a running supremum:
\small\begin{align}
|\mathscr{A}_{\mathfrak{t}_{\mathrm{av}}}^{\mathbf{T},+}(\wt{\mathfrak{g}}_{\mathfrak{t}_{\mathrm{av}};0,0}^{\mathfrak{l}})|^{2} \ \lesssim \ \sup_{0\leq\mathfrak{t}\leq\mathfrak{t}_{\mathrm{av}}}\mathfrak{t}^{2}\mathfrak{t}_{\mathrm{av}}^{-2}|\mathscr{A}_{\mathfrak{t}}^{\mathbf{T},+}(\wt{\mathfrak{g}}_{\mathfrak{t}_{\mathrm{av}};0,0}^{\mathfrak{l}})|^{2}. \label{eq:D1B2B34.5}
\end{align}\normalsize\normalsize
We now estimate the remaining $\wt{\E}$-expectation in \eqref{eq:D1B2B34}. First we apply the Kipnis-Varadhan inequality of Lemma \ref{lemma:KV} with $\varphi$ therein equal to $\wt{\mathfrak{g}}^{\mathfrak{l}}$ here and $\mathfrak{I}$ therein equal to $\mathfrak{I}_{\mathfrak{t}_{\mathrm{av}}}$ here. This controls the last $\wt{\E}$-expectation in terms of the Sobolev norm of $\wt{\mathfrak{g}}^{\mathfrak{l}}$. To control this Sobolev norm, we employ Lemma \ref{lemma:H-1SpectralPGF} to bound it by the Sobolev norm of the pseudo-gradient factor inside $\wt{\mathfrak{g}}^{\mathfrak{l}}$. This last Sobolev norm of the pseudo-gradient factor in $\wt{\mathfrak{g}}^{\mathfrak{l}}$, which is uniformly bounded and has uniformly bounded support, is then controlled by its uniform bound and a speed factor $N^{-2}$ by Lemma \ref{lemma:SpectralH-1}. We ultimately get
\small\begin{align}
\wt{\E}\sup_{0\leq\mathfrak{t}\leq\mathfrak{t}_{\mathrm{av}}}\mathfrak{t}^{2}\mathfrak{t}_{\mathrm{av}}^{-2}|\mathscr{A}_{\mathfrak{t}}^{\mathbf{T},+}(\wt{\mathfrak{g}}_{\mathfrak{t}_{\mathrm{av}};0,0}^{\mathfrak{l}})|^{2} \ \lesssim \ \mathfrak{t}_{\mathrm{av}}^{-1}\|\wt{\mathfrak{g}}^{\mathfrak{l}}\|_{\dot{\mathbf{H}}_{\varrho,\mathfrak{I}_{\mathfrak{t}_{\mathrm{av}}}}^{-1}}^{2} \ \lesssim \ N^{-2}\mathfrak{t}_{\mathrm{av}}^{-1}. \label{eq:D1B2B35}
\end{align}\normalsize\normalsize
We combine \eqref{eq:D1B2B33}, \eqref{eq:D1B2B34}, and \eqref{eq:D1B2B35} to bound $\wt{\Psi}$-terms. Combining the resulting $\wt{\Psi}$-bound with \eqref{eq:D1B2B32} then gives
\small\begin{align}
\mathfrak{t}_{\mathfrak{f}} \E^{\mu_{0,\Z}}\bar{\mathfrak{f}}_{\mathfrak{t}_{\mathfrak{s}},\mathfrak{t}_{\mathfrak{f}},\mathbb{T}_{N}}^{\mathfrak{I}_{\mathfrak{t}_{\mathrm{av}}}}\E_{\Pi_{\mathfrak{I}_{\mathfrak{t}_{\mathrm{av}}}}\eta_{0}}^{\mathrm{loc}}|\mathscr{C}_{\mathfrak{t}_{\mathrm{av}};\beta_{-},\beta_{+}}^{\mathbf{T},+,\mathfrak{t}_{N,\e},\mathfrak{i}}(\wt{\mathfrak{g}}_{\mathfrak{t}_{\mathrm{av}};0,0}^{\mathfrak{l}})| \ &\lesssim \ N^{-\beta_{-}-\frac34+\e}|\mathbb{T}_{N}|^{-1}|\mathfrak{I}_{\mathfrak{t}_{\mathrm{av}}}|^{3} \ + \ \mathfrak{t}_{\mathfrak{f}} N^{-2+\beta_{+}}\mathfrak{t}_{\mathrm{av}}^{-1}. \label{eq:D1B2B36}
\end{align}\normalsize\normalsize
We recall $|\mathbb{T}_{N}|\gtrsim N^{5/4+\e_{\mathrm{cpt}}}$ from Section \ref{section:Ctify} and the length-scale $|\mathfrak{I}_{\mathfrak{t}_{\mathrm{av}}}| \lesssim_{\e}N^{1+\e}\mathfrak{t}_{\mathrm{av}}^{1/2} + N^{3/2+\e}\mathfrak{t}_{\mathrm{av}}+ N^{\beta_{X}+\e}$; power-counting gives the following bound for the RHS of \eqref{eq:D1B2B36} in which we recall $\e_{X}=\max_{i}\e_{X,i}$ from Lemma \ref{lemma:D1B2B3}, for example:
\small\begin{align}
N^{\frac32\beta_{X}+\frac14+\e_{\mathrm{cpt}}+\frac12\e-\frac12\beta_{-}}\mathfrak{t}_{\mathfrak{f}} \E^{\mu_{0,\Z}}\bar{\mathfrak{f}}_{\mathfrak{t}_{\mathfrak{s}},\mathfrak{t}_{\mathfrak{f}},\mathbb{T}_{N}}^{\mathfrak{I}_{\mathfrak{t}_{\mathrm{av}}}}\E_{\Pi_{\mathfrak{I}_{\mathfrak{t}_{\mathrm{av}}}}\eta_{0}}^{\mathrm{loc}}|\mathscr{C}_{\mathfrak{t}_{\mathrm{av}};\beta_{-},\beta_{+}}^{\mathbf{T},+,\mathfrak{t}_{N,\e},\mathfrak{i}}(\wt{\mathfrak{g}}_{\mathfrak{t}_{\mathrm{av}};0,0}^{\mathfrak{l}})| \ &\lesssim \ N^{-\frac18-\frac34\beta_{-}+10\e_{X}} + N^{-\frac12+10\e_{X}} + \mathfrak{t}_{\mathfrak{f}}N^{-\beta_{\mathrm{univ}}}. \label{eq:D1B2B37}
\end{align}\normalsize\normalsize
This last estimate is the proposed estimate of the current lemma as it holds uniformly in $\mathfrak{l}\in\llbracket1,N^{\beta_{X}}\rrbracket$, so we are done.
\end{proof}
\begin{proof}[Proof of \emph{Lemma \ref{lemma:D1B1A1}}]
Take any $T\in[0,1]$ and $x\in\mathbb{T}_{N}$. To get the first bound, we employ the Cauchy-Schwarz inequality with respect to the space-time ``integral" against the heat operator and the fact that the heat kernel is a probability measure on $\mathbb{T}_{N}$:
\small\begin{align}
\bar{\mathbf{H}}_{T,x}^{N}(N^{\frac12}|\varphi_{S,y}|) \ = \ \int_{0}^{T}\mathfrak{s}_{S,T}^{-\frac14}\mathfrak{s}_{S,T}^{\frac14}\sum_{y\in\mathbb{T}_{N}}\bar{\mathbf{H}}_{S,T,x,y}^{N} \cdot N^{\frac12}|\varphi_{S,y}| \ \d S \ &\lesssim \ \left(\int_{0}^{T}\mathfrak{s}_{S,T}^{-\frac12}\d S\right)^{\frac12}\left(\int_{0}^{T}\mathfrak{s}_{S,T}^{\frac12}\sum_{y\in\mathbb{T}_{N}}\bar{\mathbf{H}}_{S,T,x,y}^{N} \cdot N|\varphi_{S,y}|^{2} \ \d S \right)^{\frac12} \nonumber \\
&\lesssim \ \left(\int_{0}^{T}\mathfrak{s}_{S,T}^{\frac12}\sum_{y\in\mathbb{T}_{N}}\bar{\mathbf{H}}_{S,T,x,y}^{N} \cdot N|\varphi_{S,y}|^{2} \ \d S \right)^{\frac12}. \label{eq:D1B1A11}
\end{align}\normalsize\normalsize
The final bound \eqref{eq:D1B1A11} follows by integrating the first integral in the preceding bound. We now use the on-diagonal heat kernel estimate implied by \eqref{eq:HKENash}. As $|\mathbb{T}_{N}|\lesssim N^{5/4+\e_{\mathrm{cpt}}}$, we get what follows; the last line follows by extending the integration-domain:
\small\begin{align}
&\left(\int_{0}^{T}\mathfrak{s}_{S,T}^{\frac12}\sum_{y\in\mathbb{T}_{N}}\bar{\mathbf{H}}_{S,T,x,y}^{N} \cdot N|\varphi_{S,y}|^{2} \ \d S \right)^{\frac12} \ \lesssim \ \left(N^{-1}\int_{0}^{T}\mathfrak{s}_{S,T}^{\frac12}\mathfrak{s}_{S,T}^{-\frac12}{\sum}_{y\in\mathbb{T}_{N}}N|\varphi_{S,y}|^{2} \ \d S \right)^{\frac12} \\
&\lesssim \ \left(|\mathbb{T}_{N}|\int_{0}^{T}\wt{\sum}_{y\in\mathbb{T}_{N}}|\varphi_{S,y}|^{2} \ \d S \right)^{\frac12} \ \lesssim \ \left(N^{\frac54+\e_{\mathrm{cpt}}}\int_{0}^{T}\wt{\sum}_{y\in\mathbb{T}_{N}} |\varphi_{S,y}|^{2} \ \d S \right)^{\frac12} \\
&\lesssim \ \left(N^{\frac54+\e_{\mathrm{cpt}}}\int_{0}^{1}\wt{\sum}_{y\in\mathbb{T}_{N}} |\varphi_{S,y}|^{2} \ \d S \right)^{\frac12}. \label{eq:D1B1A12}
\end{align}\normalsize\normalsize
We combine \eqref{eq:D1B1A11} and \eqref{eq:D1B1A12} to get a bound for the LHS of \eqref{eq:D1B1A11} that is uniform on $[0,1]\times\mathbb{T}_{N}$ and thus we deduce the first estimate in the current lemma. To get the second estimate in the current lemma, we employ the first estimate therein to get
\small\begin{align}
\left(\E\|\bar{\mathbf{H}}_{T,x}^{N}(N^{\frac12}|\mathscr{A}_{\mathfrak{t}_{\mathrm{av}}}^{\mathbf{T},+}\mathscr{C}_{N^{\beta_{X}}}^{\mathbf{X},-}(\mathfrak{g}_{S,y})|)\|_{1;\mathbb{T}_{N}}\right)^{2} \ &\lesssim \ \left(\E\left(N^{\frac54+\e_{\mathrm{cpt}}}\int_{0}^{1}\wt{\sum}_{y\in\mathbb{T}_{N}} |\mathscr{A}_{\mathfrak{t}_{\mathrm{av}}}^{\mathbf{T},+}\mathscr{C}_{N^{\beta_{X}}}^{\mathbf{X},-}(\mathfrak{g}_{S,y})|^{2} \ \d S \right)^{1/2}\right)^{2} \\
&\lesssim \ N^{\frac54+\e_{\mathrm{cpt}}}\int_{0}^{1}\wt{\sum}_{y\in\mathbb{T}_{N}} \E|\mathscr{A}_{\mathfrak{t}_{\mathrm{av}}}^{\mathbf{T},+}\mathscr{C}_{N^{\beta_{X}}}^{\mathbf{X},-}(\mathfrak{g}_{S,y})|^{2} \ \d S. \label{eq:D1B1A13}
\end{align}\normalsize\normalsize
The last estimate \eqref{eq:D1B1A13} follows by the Cauchy-Schwarz inequality with respect to the expectation $\E$. At this point, we follow the proof of Lemma \ref{lemma:D1B2A2} via decomposing the last expectation in \eqref{eq:D1B1A13} in terms of an expectation with respect to a path-space measure with an initial configuration that is then sampled according to the space-time averaged law of the particle system. In particular, the replacement of $\mathscr{C}^{\mathbf{T},+}\to\mathscr{A}^{\mathbf{T},+}$ and the presence of the square does not change the property of the term within the expectation in \eqref{eq:D1B1A13} being a functional of the original exclusion process on $\Omega_{\Z}$ with generator $\mathfrak{S}^{N,!!}$ which is all we needed to decompose this last expectation in \eqref{eq:D1B1A13} in terms of path-space expectation and expectation over initial configuration.
\end{proof}
\begin{proof}[Proof of \emph{Lemma \ref{lemma:D1B1A3}}]
We follow the proof of Lemma \ref{lemma:D1B2A3}. Observe the inner expectation on the LHS of the proposed estimate is a $\mathfrak{I}_{\mathfrak{t}_{\mathrm{av}}}$-local functional that is integrated against the space-time averaged law $\bar{\mathfrak{f}}$ of the particle system. We may therefore again apply local equilibrium to the $\mathfrak{I}_{\mathfrak{t}_{\mathrm{av}}}$-local inner expectation on the LHS of the proposed estimate of the current lemma via Lemma \ref{lemma:LE} and perform equilibrium calculations. To be precise, we employ Lemma \ref{lemma:LE} with $\mathfrak{I}_{N} = \mathbb{T}_{N}$, which we recall is okay because $\mathbb{T}_{N}$ satisfies constraints of Proposition \ref{prop:EProd}, with $\kappa = N^{\beta_{X}-\e_{X,1}}$, and with $\varphi$ equal to the inner expectation within the LHS of the proposed estimate of the current lemma with support $\mathfrak{I}_{\mathfrak{t}_{\mathrm{av}}}$. In this application of Lemma \ref{lemma:LE}, we will take $\mathfrak{t}_{\mathfrak{s}}$ therein to be 0 and we take $\mathfrak{t}_{\mathfrak{f}}$ to be 1. This gives the following where $\e\in\R_{>0}$ is arbitrarily small but universal:
\small\begin{align}
\E^{\mu_{0,\Z}}\bar{\mathfrak{f}}_{1,\mathbb{T}_{N}} ^{\mathfrak{I}_{\mathfrak{t}_{\mathrm{av}}}}\E_{\Pi_{\mathfrak{I}_{\mathfrak{t}_{\mathrm{av}}}}\eta_{0}}^{\mathrm{loc}}|\mathscr{A}_{\mathfrak{t}_{\mathrm{av}}}^{\mathbf{T},+}\mathscr{C}_{N^{\beta_{X}}}^{\mathbf{X},-}(\mathfrak{g}_{\mathfrak{t}_{\mathrm{av}};0,0})|^{2} \ &\lesssim \ N^{-\frac34-\beta_{X}+\e_{X,1}+\e}|\mathbb{T}_{N}|^{-1}|\mathfrak{I}_{\mathfrak{t}_{\mathrm{av}}}|^{3} \ + \ {\sup}_{\varrho\in\R}\Upsilon_{\varrho}, \label{eq:D1B1A32}
\end{align}\normalsize\normalsize
where we have introduced the following equilibrium log-exp quantities:
\small\begin{align}
\Upsilon_{\varrho} \ &\overset{\bullet}= \ N^{-\beta_{X}+\e_{X,1}}\log\E^{\mu_{\varrho,\mathfrak{I}_{\mathfrak{t}_{\mathrm{av}}}}^{\mathrm{can}}}\exp\left(N^{\beta_{X}-\e_{X,1}}\E_{\Pi_{\mathfrak{I}_{\mathfrak{t}_{\mathrm{av}}}}\eta_{0}}^{\mathrm{loc}}|\mathscr{A}_{\mathfrak{t}_{\mathrm{av}}}^{\mathbf{T},+}\mathscr{C}_{N^{\beta_{X}}}^{\mathbf{X},-}(\mathfrak{g}_{\mathfrak{t}_{\mathrm{av}};0,0})|^{2}\right).
\end{align}\normalsize\normalsize
For the $\Upsilon_{\varrho}$-terms, recall the $\mathscr{C}^{\mathbf{X},-}$-term in the exponential is bounded above by $N^{-\beta_{X}/2+\e_{X,1}/2}$ deterministically by construction in Definition \ref{definition:averagesAlocal}. Its time-average $\mathscr{A}^{\mathbf{T},+}\mathscr{C}^{\mathbf{X},-}$, therefore, satisfies the same deterministic a priori estimate, and thus the quantity in the exponential defining $\Upsilon_{\varrho}$ is uniformly bounded deterministically. Thus, the exponential and logarithm inequalities given prior to \eqref{eq:D1B2A33} in the proof of Proposition \ref{prop:D1B2A} yield the following in similar fashion as how we obtained \eqref{eq:D1B2A33}:
\small\begin{align}
\Upsilon_{\varrho} \ &\lesssim \ \E^{\mu_{\varrho,\mathfrak{I}_{\mathfrak{t}_{\mathrm{av}}}}^{\mathrm{can}}}\E_{\Pi_{\mathfrak{I}_{\mathfrak{t}_{\mathrm{av}}}}\eta_{0}}^{\mathrm{loc}}|\mathscr{A}_{\mathfrak{t}_{\mathrm{av}}}^{\mathbf{T},+}\mathscr{C}_{N^{\beta_{X}}}^{\mathbf{X},-}(\mathfrak{g}_{\mathfrak{t}_{\mathrm{av}};0,0})|^{2}. \label{eq:D1B1A33}
\end{align}\normalsize\normalsize
To bound the iterated expectation in \eqref{eq:D1B1A33}, we again employ equilibrium calculations. We have actually done this in the proof of Lemma \ref{lemma:D1B2A3}. Precisely, we follow \eqref{eq:D1B2A34.5}, \eqref{eq:D1B2A37}, and \eqref{eq:D1B2A38} in the proof for Lemma \ref{lemma:D1B2A3} to get the following in which $N^{-100}$ is ultimately negligible for the conclusions of the current lemma:
\small\begin{align}
\E^{\mu_{\varrho,\mathfrak{I}_{\mathfrak{t}_{\mathrm{av}}}}^{\mathrm{can}}}\E_{\Pi_{\mathfrak{I}_{\mathfrak{t}_{\mathrm{av}}}}\eta_{0}}^{\mathrm{loc}}|\mathscr{A}_{\mathfrak{t}_{\mathrm{av}}}^{\mathbf{T},+}\mathscr{C}_{N^{\beta_{X}}}^{\mathbf{X},-}(\mathfrak{g}_{\mathfrak{t}_{\mathrm{av}};0,0})|^{2} \ &\lesssim \ N^{-2-\beta_{X}}\mathfrak{t}_{\mathrm{av}}^{-1}+N^{-100}. \label{eq:D1B1A34} 
\end{align}\normalsize\normalsize
We now combine \eqref{eq:D1B1A32}, \eqref{eq:D1B1A33}, and \eqref{eq:D1B1A34} to deduce the desired estimate in Lemma \ref{lemma:D1B1A3} and thus complete its proof.
\end{proof}
\begin{proof}[Proof of \emph{Lemma \ref{lemma:D1B1B1}}]
The first bound follows by the proof of the first estimate in Lemma \ref{lemma:D1B1A1} by replacing $N^{1/2}$ with $N^{\beta_{X}}$ and accounting for the resulting elementary changes in power-counting. This may be screened without a calculation if we pretend $\beta_{X} = \frac12$ in which case we match the exponent $2\beta_{X}+\frac14+\e_{\mathrm{cpt}}$ in the first estimate of Lemma \ref{lemma:D1B1B1} with the exponent $\frac54+\e_{\mathrm{cpt}}$ in the first estimate of Lemma \ref{lemma:D1B1A1}. To prove the second bound in Lemma \ref{lemma:D1B1B1}, we similarly follow the proof of the second estimate in Lemma \ref{lemma:D1B1A1} after the following replacement and the same minor and elementary adjustments in power-counting which we alluded to in the proof of the first estimate of the current Lemma \ref{lemma:D1B1B1}:
\small\begin{align}
\left(N^{\frac12}, \mathscr{A}_{\mathfrak{t}_{\mathrm{av}}}^{\mathbf{T},+}\mathscr{C}_{N^{\beta_{X}}}^{\mathbf{X},-}(\mathfrak{g}_{S,y})\right) \ \to \ \left(N^{\beta_{X}}, \mathscr{A}_{\mathfrak{t}_{\mathrm{av}}}^{\mathbf{T},+}(\wt{\mathfrak{g}}_{S,y}^{\mathfrak{l}})\right).
\end{align}\normalsize\normalsize
This completes the proof.
\end{proof}
\begin{proof}[Proof of \emph{Lemma \ref{lemma:D1B1B3}}]
Looking at the beginning of the proof of Lemma \ref{lemma:D1B2B3}, we first observe the inner expectation on the LHS of the proposed estimate is a $\mathfrak{I}_{\mathfrak{t}_{\mathrm{av}}}$-local statistic, so that the iterated expectation within the LHS of the proposed estimate is that of a $\mathfrak{I}_{\mathfrak{t}_{\mathrm{av}}}$-local statistic against the space-time averaged Radon-Nikodym derivative $\bar{\mathfrak{f}}$. We may thus employ the local equilibrium estimate of Lemma \ref{lemma:LE} to reduce the problem of estimating the iterated expectation within the LHS of the proposed estimate in terms of equilibrium calculations. In particular, we employ Lemma \ref{lemma:LE} with $\mathfrak{I}_{N} = \mathbb{T}_{N}$, with $\varphi$ the inner path-space expectation on the LHS of the proposed estimate which has support $\mathfrak{I}_{\mathfrak{t}_{\mathrm{av}}}$, and with $\kappa = 1$. We clarify that this choice of $\kappa = 1$ means we will not depend on an a priori large-deviations estimate for the inner expectation that we are taking to be our local statistic $\varphi$ in this application of Lemma \ref{lemma:LE}. We also emphasize this choice $\mathfrak{I}_{N}=\mathbb{T}_{N}$ in using Lemma \ref{lemma:LE} is okay because $\mathbb{T}_{N}$ satisfies constraints on the block $\mathfrak{I}_{N}$ in Proposition \ref{prop:EProd}. Lastly, we clarify that we apply Lemma \ref{lemma:LE} with $\mathfrak{t}_{\mathfrak{s}}$ therein to be 0 and $\mathfrak{t}_{\mathfrak{f}}$ therein to be 1:
\small\begin{align}
\E^{\mu_{0,\Z}}\bar{\mathfrak{f}}_{1,\mathbb{T}_{N}}^{\mathfrak{I}_{\mathfrak{t}_{\mathrm{av}}}} \E_{\Pi_{\mathfrak{I}_{\mathfrak{t}_{\mathrm{av}}}}\eta_{0}}^{\mathrm{loc}}|\mathscr{A}_{\mathfrak{t}_{\mathrm{av}}}^{\mathbf{T},+}(\wt{\mathfrak{g}}_{\mathfrak{t}_{\mathrm{av}};0,0}^{\mathfrak{l}})|^{2} \ &\lesssim_{\e} \ N^{-\frac34+\e}|\mathbb{T}_{N}|^{-1}|\mathfrak{I}_{\mathfrak{t}_{\mathrm{av}}}|^{3} \ + \ \sup_{\varrho\in\R}\wt{\Upsilon}_{\varrho}. \label{eq:D1B1B32}
\end{align}\normalsize\normalsize
Above $\e\in\R_{>0}$ is arbitrarily small but universal. We have also introduced the following equilibrium quantities:
\small\begin{align}
\wt{\Upsilon}_{\varrho} \ &\overset{\bullet}= \ \log\E^{\mu_{\varrho,\mathfrak{I}_{\mathfrak{t}_{\mathrm{av}}}}^{\mathrm{can}}}\exp\left(\E_{\Pi_{\mathfrak{I}_{\mathfrak{t}_{\mathrm{av}}}}\eta_{0}}^{\mathrm{loc}}|\mathscr{A}_{\mathfrak{t}_{\mathrm{av}}}^{\mathbf{T},+}(\wt{\mathfrak{g}}_{\mathfrak{t}_{\mathrm{av}};0,0}^{\mathfrak{l}})|^{2}\right).
\end{align}\normalsize\normalsize
We focus on local equilibrium calculations relevant to estimating $\wt{\Upsilon}$-terms. Recall in Proposition \ref{prop:Duhamel} that the $\wt{\mathfrak{g}}^{\mathfrak{l}}$-terms are terms that admit pseudo-gradient factors, so they are uniformly bounded. The exponential and logarithm inequalities we wrote prior to the estimate \eqref{eq:D1B2A33} from the proof of Lemma \ref{lemma:D1B2A3} then give
\small\begin{align}
\wt{\Upsilon}_{\varrho} \ \lesssim \ \E^{\mu_{\varrho,\mathfrak{I}_{\mathfrak{t}_{\mathrm{av}}}}^{\mathrm{can}}}\E_{\Pi_{\mathfrak{I}_{\mathfrak{t}_{\mathrm{av}}}}\eta_{0}}^{\mathrm{loc}}|\mathscr{A}_{\mathfrak{t}_{\mathrm{av}}}^{\mathbf{T},+}(\wt{\mathfrak{g}}_{\mathfrak{t}_{\mathrm{av}};0,0}^{\mathfrak{l}})|^{2} \ \lesssim \ N^{-2}\mathfrak{t}_{\mathrm{av}}^{-1}. \label{eq:D1B1B33}
\end{align}\normalsize\normalsize
The final estimate \eqref{eq:D1B1B33} follows from the bound that is immediately to its left via \eqref{eq:D1B2B34.5} and \eqref{eq:D1B2B35} from the proof of Lemma \ref{lemma:D1B2B3}. We now combine \eqref{eq:D1B1B32} and \eqref{eq:D1B1B33} to deduce the desired estimate of Lemma \ref{lemma:D1B1B3}.
\end{proof}
\begin{proof}[Proof of \emph{Lemma \ref{lemma:S1B1}}]
We may follow the proof of the second estimate in the statement of Lemma \ref{lemma:D1B1A1} using the first estimate therein but with the function $\varphi_{S,y} = \mathscr{A}_{N^{\beta_{X}}}^{\mathbf{X},-}(\mathfrak{g}_{S,y}) - \mathscr{C}_{N^{\beta_{X}}}^{\mathbf{X},-}(\mathfrak{g}_{S,y})$. Technically this gives the desired identity but with $\bar{\mathfrak{f}}$ on the RHS of the proposed identity, not its projection onto $\mathfrak{I}$. But the inner expectation on the RHS depends only on spins in $\mathfrak{I}$.
\end{proof}
\begin{proof}[Proof of \emph{Lemma \ref{lemma:S1B2}}]
Let us observe the expectation on the LHS of the proposed bound is expectation is of a local statistic given by the squared difference against the space-time averaged probability density $\bar{\mathfrak{f}}$ of the law of the particle system. We may thus employ local equilibrium Lemma \ref{lemma:LE} with $\mathfrak{I}_{N}=\mathbb{T}_{N}$, which again is okay as $\mathbb{T}_{N}$ satisfies constraints of Proposition \ref{prop:EProd}, with $\varphi$ equal to the square within the LHS of the estimate in Lemma \ref{lemma:S1B2} with support contained in $\mathfrak{I}_{\varphi} = \mathfrak{I} = \llbracket-100\mathfrak{m}N^{\beta_{X}}, 100\mathfrak{m}N^{\beta_{X}}\rrbracket$ from Lemma \ref{lemma:S1B1}, and with the choice of constant $\kappa = N^{\beta_{X}-3\e_{X,1}}$. Again, in this application of Lemma \ref{lemma:LE} we also take $\mathfrak{t}_{\mathfrak{s}}$ therein to be 0 and $\mathfrak{t}_{\mathfrak{f}}$ therein to be 1. For $\e\in\R_{>0}$ arbitrarily small but universal, this gives
\small\begin{align}
\E^{\mu_{0,\Z}}\bar{\mathfrak{f}}_{1,\mathbb{T}_{N}}^{\mathfrak{I}} |\mathscr{A}_{N^{\beta_{X}}}^{\mathbf{X},-}(\mathfrak{g}_{0,0}) - \mathscr{C}_{N^{\beta_{X}}}^{\mathbf{X},-}(\mathfrak{g}_{0,0})|^{2} \ &\lesssim \ N^{-\frac34+\e}N^{-\beta_{X}+3\e_{X,1}}|\mathbb{T}_{N}|^{-1}|\mathfrak{I}|^{3} \ + \ \sup_{\varrho\in\R} \Gamma_{\varrho}. \label{eq:S1B1}
\end{align}\normalsize\normalsize
We have introduced the following equilibrium log-exp expectations in \eqref{eq:S1B1} above:
\small\begin{align}
\Gamma_{\varrho} \ &\overset{\bullet}= \ N^{-\beta_{X}+3\e_{X,1}} \log \E^{\mu_{\varrho,\mathfrak{I}}^{\mathrm{can}}}\exp\left(N^{\beta_{X}-3\e_{X,1}}|\mathscr{A}_{N^{\beta_{X}}}^{\mathbf{X},-}(\mathfrak{g}_{0,0}) - \mathscr{C}_{N^{\beta_{X}}}^{\mathbf{X},-}(\mathfrak{g}_{0,0})|^{2}\right).
\end{align}\normalsize\normalsize
The first term on the RHS of \eqref{eq:S1B1} is again addressed by elementary power-counting in $N\in\Z_{>0}$ that we come back to at the end of this proof. We focus on $\Gamma$-terms on the RHS of \eqref{eq:S1B1} first. To this end, by definitions in Definition \ref{definition:S1B}, we get
\small\begin{align}
\exp\left(N^{\beta_{X}-3\e_{X,1}}|\mathscr{A}_{N^{\beta_{X}}}^{\mathbf{X},-}(\mathfrak{g}_{0,0}) - \mathscr{C}_{N^{\beta_{X}}}^{\mathbf{X},-}(\mathfrak{g}_{0,0})|^{2}\right) \ &= \ \mathbf{1}[\mathscr{E}_{0,0}^{\mathbf{X},-}] \ + \ \exp\left(N^{\beta_{X}-3\e_{X,1}}|\mathscr{A}_{N^{\beta_{X}}}^{\mathbf{X},-}(\mathfrak{g}_{0,0})|^{2}\right)\mathbf{1}[(\mathscr{E}_{0,0}^{\mathbf{X},-})^{C}]. \label{eq:S1B2}
\end{align}\normalsize\normalsize
The superscript $C$ on the RHS of \eqref{eq:S1B2} denotes taking the complement of the $\mathscr{E}^{\mathbf{X},-}$ event we defined in Definition \ref{definition:S1B}. Indeed, on this $\mathscr{E}^{\mathbf{X},-}$-event, the $\mathscr{A}^{\mathbf{X},-}$ and $\mathscr{C}^{\mathbf{X},-}$ operators/terms agree by definition, and thus the exponential of the difference is equal to 1, while outside this $\mathscr{E}^{\mathbf{X},-}$-event we note that $\mathscr{C}^{\mathbf{X},-}$ vanishes. Thus, taking expectations of \eqref{eq:S1B2} we get the consequence
\small\begin{align}
\E^{\mu_{\varrho,\mathfrak{I}}^{\mathrm{can}}}\exp\left(N^{\beta_{X}-3\e_{X,1}}|\mathscr{A}_{N^{\beta_{X}}}^{\mathbf{X},-}(\mathfrak{g}_{0,0}) - \mathscr{C}_{N^{\beta_{X}}}^{\mathbf{X},-}(\mathfrak{g}_{0,0})|^{2}\right) \ &= \ \E^{\mu_{\varrho,\mathfrak{I}}^{\mathrm{can}}}\mathbf{1}[\mathscr{E}_{0,0}^{\mathbf{X},-}] \ + \ \E^{\mu_{\varrho,\mathfrak{I}}^{\mathrm{can}}}\exp\left(N^{\beta_{X}-3\e_{X,1}}|\mathscr{A}_{N^{\beta_{X}}}^{\mathbf{X},-}(\mathfrak{g}_{0,0})|^{2}\right)\mathbf{1}[(\mathscr{E}_{0,0}^{\mathbf{X},-})^{C}] \nonumber \\
&\leq \ 1+\E^{\mu_{\varrho,\mathfrak{I}}^{\mathrm{can}}}\exp\left(N^{\beta_{X}-3\e_{X,1}}|\mathscr{A}_{N^{\beta_{X}}}^{\mathbf{X},-}(\mathfrak{g}_{0,0})|^{2}\right)\mathbf{1}[(\mathscr{E}_{0,0}^{\mathbf{X},-})^{C}]. \label{eq:S1B3}
\end{align}\normalsize\normalsize
The estimate \eqref{eq:S1B3} is trivial as $\mathbf{1}[\mathscr{E}^{\mathbf{X},-}_{0,0}]\leq1$. We proceed by using the inequality $\log(1+|x|) \leq |x|$, with $|x|$ equal to the expectation in \eqref{eq:S1B3}, with \eqref{eq:S1B3} and definition of $\Gamma_{\varrho}$; we get the following with universal implied constant:
\small\begin{align}
\Gamma_{\varrho} \ &\lesssim \ N^{-\beta_{X}+3\e_{X,1}} \E^{\mu_{\varrho,\mathfrak{I}}^{\mathrm{can}}}\exp\left(N^{\beta_{X}-3\e_{X,1}}|\mathscr{A}_{N^{\beta_{X}}}^{\mathbf{X},-}(\mathfrak{g}_{0,0})|^{2}\right)\mathbf{1}[(\mathscr{E}_{0,0}^{\mathbf{X},-})^{C}]. \label{eq:S1B4}
\end{align}\normalsize\normalsize
To estimate this exponential expectation, we recall by definition that $\mathscr{A}^{\mathbf{X},-}(\mathfrak{g}_{0,0})$ is the average of spatial-shifts of the pseudo-gradient $\mathfrak{g}_{0,0}$ which all have mutually disjoint support; see Proposition \ref{prop:Duhamel} for the disjoint property of supports and Definition \ref{definition:S1B} for $\mathscr{A}^{\mathbf{X},-}$. Thus, we use Corollary \ref{corollary:LDP} with $J = N^{\beta_{X}}$, with $\varphi_{j} = \tau_{-7j\mathfrak{m}}\mathfrak{g}_{0,0}$, with $\e_{1} = 3\e_{X,1}$, and with $\e_{2} = \frac12\e_{X,1}$ to get
\small\begin{align}
\E^{\mu_{\varrho,\mathfrak{I}}^{\mathrm{can}}}\exp\left(N^{\beta_{X}-3\e_{X,1}}|\mathscr{A}_{N^{\beta_{X}}}^{\mathbf{X},-}(\mathfrak{g}_{0,0})|^{2}\right)\mathbf{1}[(\mathscr{E}_{0,0}^{\mathbf{X},-})^{C}] \ &\lesssim_{C} \ N^{-C}. \label{eq:S1B5}
\end{align}\normalsize\normalsize
The constant $C\in\R_{>0}$ in \eqref{eq:S1B5} is arbitrarily large but universal. We combine this with \eqref{eq:S1B4} to control $\Gamma_{\varrho}$ by an arbitrarily large but universal \emph{negative} power of $N\in\Z_{>0}$. We combine \eqref{eq:S1B1}, \eqref{eq:S1B4}, and \eqref{eq:S1B5} to get
\small\begin{align}
N^{\frac54+\e_{\mathrm{cpt}}}\E^{\mu_{0,\Z}}\bar{\mathfrak{f}}_{1,\mathbb{T}_{N}}^{\mathfrak{I}} |\mathscr{A}_{N^{\beta_{X}}}^{\mathbf{X},-}(\mathfrak{g}_{0,0}) - \mathscr{C}_{N^{\beta_{X}}}^{\mathbf{X},-}(\mathfrak{g}_{0,0})|^{2} \ &\lesssim_{C} \ N^{\frac54+\e_{\mathrm{cpt}}}N^{-\frac34+\e}N^{-\beta_{X}+3\e_{X,1}}|\mathbb{T}_{N}|^{-1}|\mathfrak{I}|^{3} \ + \ N^{\frac54+\e_{\mathrm{cpt}}}N^{-C}. \label{eq:S1B6}
\end{align}\normalsize\normalsize
We recall $\mathfrak{I}\subseteq\Z$ within the statement of Lemma \ref{lemma:S1B2}. Choosing $C\in\R_{>0}$ sufficiently large and elementary power-counting for the first term on the RHS of \eqref{eq:S1B6}, recalling $|\mathbb{T}_{N}|\gtrsim N^{5/4+\e_{\mathrm{cpt}}}$ from the beginning of Section \ref{section:Ctify}, completes the proof.
\end{proof}
%
%
%
\section{Stochastic Time-Regularity Estimates}\label{section:KPZ1}
As pointed out in the would-be-proof of Pseudo-Proposition \ref{pprop:S3}, to perform time-replacement in the statement of Pseudo-Proposition \ref{pprop:S3} we need time-regularity of the heat kernel $\bar{\mathbf{H}}^{N}$ and of the compactified microscopic Cole-Hopf transform $\bar{\mathbf{Z}}^{N}$. The former is given in Lemma \ref{lemma:HKE} and is deterministic. The latter is the current focus. Let us first introduce some notation.
\begin{definition}
Provided any $\mathfrak{t}\in\R$, let us define $\grad_{\mathfrak{t}}^{\mathbf{T}}\varphi_{\bullet} = \varphi_{\bullet+\mathfrak{t}} - \varphi_{\bullet}$ for functions $\varphi:\R_{\geq0}\to\R$. Like with data of the particle system, if $\varphi$ is evaluated at negative times we instead take its value at time 0.
\end{definition}
\begin{prop}\label{prop:TRGTProp}
 Consider any time-scale $\mathfrak{t}_{\mathrm{reg}} \in \R_{\geq0}$ satisfying $N^{-100}\lesssim\mathfrak{t}_{\mathrm{reg}}\lesssim N^{-1}$ and any $\e,C \in \R_{>0}$. Provided a random time $\mathfrak{t}_{\mathrm{st}} \in \R_{\geq0}$ satisfying $\mathfrak{t}_{\mathrm{st}} \leq 1$ with probability 1, we have the following outside an event of probability at most $\kappa_{\e,C}N^{-C}$:
\small\begin{align}
\|\grad_{-\mathfrak{t}_{\mathrm{reg}}}^{\mathbf{T}}\bar{\mathbf{Z}}^{N}\|_{\mathfrak{t}_{\mathrm{st}};\mathbb{T}_{N}} \ &\lesssim_{\e} \ N^{-\frac12+\e} + N^{\e}\mathfrak{t}_{\mathrm{reg}}^{\frac14} + N^{-\frac12+\e}\|\bar{\mathbf{Z}}^{N}\|_{\mathfrak{t}_{\mathrm{st}};\mathbb{T}_{N}}^{2} + N^{\e}\mathfrak{t}_{\mathrm{reg}}^{\frac14}\|\bar{\mathbf{Z}}^{N}\|_{\mathfrak{t}_{\mathrm{st}};\mathbb{T}_{N}}^{2}.
\end{align}\normalsize\normalsize
\end{prop}
We explain Proposition \ref{prop:TRGTProp}. Assuming Theorem \ref{theorem:KPZ}, the space-time process $\mathbf{Z}^{N}$ looks like the SHE solution. The comparison in Proposition \ref{prop:Ctify} says the same is true for $\bar{\mathbf{Z}}^{N}$. The time-regularity of $\bar{\mathbf{Z}}^{N}$ is then that of the SHE. This is Holder regularity with exponent $(\frac14)^{-}$. Forgetting $N^{-1/2+\e}$-terms and $N^{\e}$-factors on the RHS of the bound in Proposition \ref{prop:TRGTProp} then gives the result.

To prove Proposition \ref{prop:TRGTProp}, we use the defining stochastic integral equation for $\bar{\mathbf{Z}}^{N}$ introduced in Section \ref{section:Ctify} and then estimate time-regularity in the same way that we would estimate time-regularity of the solution of SHE through its mild Duhamel form. Let us clarify that the $N^{\e}$-factors, which are sub-optimal but sufficient for our purposes, are here to guarantee that such a strong time-regularity estimate proposed in Proposition \ref{prop:TRGTProp} holds with the required high-probability. The $N^{-1/2+\e}$-errors that are not present in time-regularity estimates of the SHE solution come from the observation that the jumps in $\bar{\mathbf{Z}}^{N}$ are order $N^{-1/2}$. We lastly remark that the $N^{\e}$-factors in the proposed upper bound in the statement of Proposition \ref{prop:TRGTProp} will be harmless because we will only use Proposition \ref{prop:TRGTProp} when we compute power-saving estimates for $\bar{\Phi}^{N,2}$-related terms, so we will only use Proposition \ref{prop:TRGTProp} when we have a factor of $N^{-\beta_{\mathrm{univ}}}$ for $\beta_{\mathrm{univ}}\in\R_{>0}$ universal. Choosing $\e\in\R_{>0}$ in these $N^{\e}$-factors small then dwarfs them.

We additionally comment on the quadratic dependence of the $\|\|$-norm on the RHS of the proposed estimate in Proposition \ref{prop:TRGTProp}. This is mostly a technical consequence of obtaining time-regularity estimates for the stochastic-integral-type quantity in the stochastic equation for $\bar{\mathbf{Z}}^{N}$ introduced in Section \ref{section:Ctify}. In particular, we will need to apply martingale estimates to control the time-regularity of this quantity; a direct bound on $\bar{\mathbf{Z}}^{N}$ in this stochastic integral type term would not let us employ martingale inequalities, so we need to take a little more care, and this gives quadratic dependence of the $\|\|$-norm.

Given that the proof of Proposition \ref{prop:TRGTProp} looks like the proof of time-regularity estimates for the solution of SHE as we noted above, the reader is invited to skip or skim the remainder of the section. All we do is follow the proof of Proposition 3.2 in \cite{DT}, namely the proof of (3.14) therein, with some technical adjustments to account for the $\|\|_{\mathfrak{t}_{\mathrm{st}};\mathbb{T}_{N}}$-norms.
\subsection{Proof of Proposition \ref{prop:TRGTProp}}
The first ingredient we use in the proof of Proposition \ref{prop:TRGTProp} is simplification of Proposition \ref{prop:TRGTProp}. In particular, we will control the time-gradient of $\bar{\mathbf{Z}}^{N}$ first on a discretization of the semi-discrete set $[0,\mathfrak{t}_{\mathrm{st}}]\times\mathbb{T}_{N}$. This first step is easier since we can use a union bound as in the proof of Proposition \ref{prop:Ctify}. We also specialize to deterministic times $\mathfrak{t}_{\mathrm{st}}\in\R_{\geq0}$.
\begin{lemma}\label{lemma:TRGTProp1}
 We admit the setting in \emph{Proposition \ref{prop:TRGTProp}} and assume that $\mathfrak{t}_{\mathrm{st}} \in \R_{\geq0}$ is deterministic. Outside an event with probability at most $\kappa_{\e,C}N^{-C}$, we have the following estimate for the fully discretized norm defined in the statement of \emph{Proposition \ref{prop:Ctify}}:
\small\begin{align}
[\grad_{-\mathfrak{t}_{\mathrm{reg}}}^{\mathbf{T}}\bar{\mathbf{Z}}^{N}]_{\mathfrak{t}_{\mathrm{st}};\mathbb{T}_{N}} \ &\lesssim_{\e} \ N^{-\frac12+\e} + N^{\e}\mathfrak{t}_{\mathrm{reg}}^{\frac14} + N^{-\frac12+\e}\|\bar{\mathbf{Z}}^{N}\|_{\mathfrak{t}_{\mathrm{st}};\mathbb{T}_{N}}^{2} + N^{\e}\mathfrak{t}_{\mathrm{reg}}^{\frac14}\|\bar{\mathbf{Z}}^{N}\|_{\mathfrak{t}_{\mathrm{st}};\mathbb{T}_{N}}^{2}. \label{eq:TRGTProp1}
\end{align}\normalsize\normalsize
\end{lemma}
We defer proof of Lemma \ref{lemma:TRGTProp1} and other ingredients in the proof of Proposition \ref{prop:TRGTProp} to future subsections. We now bootstrap to the semi-discrete norm $\|\|_{\mathfrak{t}_{\mathrm{st}};\mathbb{T}_{N}}$ and to possibly random times $\mathfrak{t}_{\mathrm{st}}\in\R_{\geq0}$ through a stochastic continuity estimate.
\begin{lemma}\label{lemma:TRGTProp3}
 Consider any possibly random time $\mathfrak{t}_{\mathrm{st}} \in \R_{\geq0}$ satisfying $\mathfrak{t}_{\mathrm{st}} \leq 1$ with probability 1. Outside some event of probability at most $\kappa_{\e,C}N^{-C}$ with $C \in \R_{>0}$ arbitrarily large but universal and $\e \in \R_{>0}$ arbitrarily small but universal, we have
\small\begin{align}
\|\grad_{-\mathfrak{t}_{\mathrm{reg}}}^{\mathbf{T}}\bar{\mathbf{Z}}^{N}\|_{\mathfrak{t}_{\mathrm{st}};\mathbb{T}_{N}} \ &\lesssim \ [\grad_{-\mathfrak{t}_{\mathrm{reg}}}^{\mathbf{T}}\bar{\mathbf{Z}}^{N}]_{\mathfrak{t}_{\mathrm{st}};\mathbb{T}_{N}} \ + \ N^{-\frac12+\e} + N^{-\frac12+\e}\|\bar{\mathbf{Z}}^{N}\|_{\mathfrak{t}_{\mathrm{st}};\mathbb{T}_{N}}^{2}. \label{eq:TRGTProp3I}
\end{align}\normalsize\normalsize
We define the random time $\wt{\mathfrak{t}}_{\mathrm{st}} \geq0$ by mapping $\mathfrak{t}_{\mathrm{st}}\geq0$ to the random closest element of the set $\{\mathfrak{t}_{\mathfrak{j}}\}_{\mathfrak{j}=0}^{3N^{100}}$, where $\mathfrak{t}_{\mathfrak{j}} \overset{\bullet}= \mathfrak{j}N^{-100}$. Outside an event of probability at most $\kappa_{\e,C}N^{-C}$ with $C$ arbitrarily large and $\e>0$ arbitrarily small but both universal,
\small\begin{align}
\|\grad_{-\mathfrak{t}_{\mathrm{reg}}}^{\mathbf{T}}\bar{\mathbf{Z}}^{N}\|_{\mathfrak{t}_{\mathrm{st}};\mathbb{T}_{N}} \ &\lesssim \ \|\grad_{-\mathfrak{t}_{\mathrm{reg}}}^{\mathbf{T}}\bar{\mathbf{Z}}^{N}\|_{\wt{\mathfrak{t}}_{\mathrm{st}};\mathbb{T}_{N}} \ + \ N^{-\frac12+\e} + N^{-\frac12+\e}\|\bar{\mathbf{Z}}^{N}\|_{\mathfrak{t}_{\mathrm{st}};\mathbb{T}_{N}}^{2}. \label{eq:TRGTProp3II}.
\end{align}\normalsize\normalsize
\end{lemma}
\begin{proof}[Proof of \emph{Proposition \ref{prop:TRGTProp}}]
If $\mathfrak{t}_{\mathrm{st}}\in\R_{\geq0}$ is deterministic, we apply the first estimate \eqref{eq:TRGTProp3I} of Lemma \ref{lemma:TRGTProp3} followed by Lemma \ref{lemma:TRGTProp1}. It remains to study possibly random times $\mathfrak{t}_{\mathrm{st}}\in\R_{\geq0}$. Adopting notation of Lemma \ref{lemma:TRGTProp3}, by \eqref{eq:TRGTProp3II} it suffices to prove the claim for the random time $\wt{\mathfrak{t}}_{\mathrm{st}}\in\R_{\geq0}$ associated to $\mathfrak{t}_{\mathrm{st}}\in\R_{\geq0}$. Note $\wt{\mathfrak{t}}_{\mathrm{st}}\in\R_{\geq0}$ can take at most $5N^{100}$-many possible values. Using a union bound over such values, we may assume $\wt{\mathfrak{t}}_{\mathrm{st}}$ is deterministic and multiply the complement probability $\kappa_{\e,C}N^{-C}$ by $5N^{100}$. Thus, we get the proposed bound with probability at least $1-5\kappa_{\e,C}N^{-C+100}$, so the proof is complete upon redefining $C\in\R_{\geq0}$.
\end{proof}
\subsection{Proof of Lemma \ref{lemma:TRGTProp1}}
We start with time-regularity for the stochastic-integral-type term in the defining equation for $\bar{\mathbf{Z}}^{N}$.
\begin{lemma}\label{lemma:C0SIProp}
 We admit the setting of \emph{Proposition \ref{prop:TRGTProp}} and we assume $\mathfrak{t}_{\mathrm{st}} \in \R_{\geq0}$ is deterministic. Consider any space-time process $\bar{\mathbf{V}}^{N}$ adapted to the canonical filtration of the particle system. Outside an event of probability at most $\kappa_{\e,C}N^{-C}$, we have
\small\begin{align}
[\grad_{-\mathfrak{t}_{\mathrm{reg}}}^{\mathbf{T}}\bar{\mathbf{H}}^{N}(\bar{\mathbf{V}}^{N}\d\xi^{N})]_{\mathfrak{t}_{\mathrm{st}};\mathbb{T}_{N}} \ &\lesssim_{\e} \ N^{-\frac12+\e} + N^{\e}\mathfrak{t}_{\mathrm{reg}}^{\frac14} + N^{-\frac12+\e}\|\bar{\mathbf{V}}^{N}\|_{\mathfrak{t}_{\mathrm{st}};\mathbb{T}_{N}}^{2} + N^{\e}\mathfrak{t}_{\mathrm{reg}}^{\frac14}\|\bar{\mathbf{V}}^{N}\|_{\mathfrak{t}_{\mathrm{st}};\mathbb{T}_{N}}^{2}.
\end{align}\normalsize\normalsize
\end{lemma}
\begin{proof}[Proof of \emph{Lemma \ref{lemma:TRGTProp1}}]
For convenience we recall the following stochastic equation for $\bar{\mathbf{Z}}^{N}$ with notation from Definition \ref{definition:ChiTorus}:
\small\begin{align}
\bar{\mathbf{Z}}_{T,x}^{N} \ &\overset{\bullet}= \ \bar{\mathbf{H}}_{T,x}^{N,\mathbf{X}}(\chi_{\bullet}\mathbf{Z}_{0,\bullet}^{N}) + \bar{\mathbf{H}}_{T,x}^{N}(\bar{\mathbf{Z}}^{N}\d\xi^{N}) + \bar{\mathbf{H}}_{T,x}^{N}(\bar{\Phi}^{N,2}) + \bar{\mathbf{H}}_{T,x}^{N}(\bar{\Phi}^{N,3}). 
\end{align}\normalsize\normalsize
We follow the proof of (3.14) in Proposition 3.2 in \cite{DT} with minor adjustments. Let us also assume $T\geq\mathfrak{t}_{\mathrm{reg}}$. For times $T\geq0$ that do not satisfy this lower bound the time-gradient of $\bar{\mathbf{Z}}^{N}$ on scale $\mathfrak{t}_{\mathrm{reg}}$ can be written as the time-gradient on a scale $\mathfrak{t}\leq\mathfrak{t}_{\mathrm{reg}}\wedge T$ because we extend $\bar{\mathbf{Z}}^{N}$ to negative times by its value at time 0. Thus we get $\mathfrak{t}$-versions of $\mathfrak{t}_{\mathrm{reg}}$-estimates below which are certainly controlled by such $\mathfrak{t}_{\mathrm{reg}}$-estimates below as $\mathfrak{t}\leq\mathfrak{t}_{\mathrm{reg}}$. First, we use a Chapman-Kolmogorov equation and a priori near-stationary moment bounds as in proof of (3.14) in Proposition 3.2 in \cite{DT} to get, for $\delta\in\R_{>0}$ arbitrarily small but universal and $p\in\R_{\geq1}$, the following bound where we remark $\bar{\mathbf{Z}}_{0,\bullet}^{N}=\chi_{\bullet}\mathbf{Z}_{0,\bullet}^{N}$ admits near-stationary-type regularity estimates with respect to geodesic distance on $\mathbb{T}_{N}$ since $\chi$ itself has regularity and lets us ignore the boundary condition on the torus $\mathbb{T}_{N}$; see Remark \ref{remark:ch2ChiT=0}:
\small\begin{align}
\|\grad_{-\mathfrak{t}_{\mathrm{reg}}}^{\mathbf{T}}\bar{\mathbf{H}}_{T,x}^{N,\mathbf{X}}\bar{\mathbf{Z}}_{0,\bullet}^{N}\|_{\omega;2p} \ \lesssim_{p,\delta} \ N^{-\frac12+\delta} + N^{2\delta}\mathfrak{t}_{\mathrm{reg}}^{1/4-\delta} \ \lesssim \ N^{-\frac12+\delta} + N^{200\delta}\mathfrak{t}_{\mathrm{reg}}^{1/4}.
\end{align}\normalsize\normalsize
Indeed, the last estimate comes from $\mathfrak{t}_{\mathrm{reg}}\gtrsim N^{-100}$. As we later explain in the proof of Lemma \ref{lemma:C0SIProp}, a union bound and Chebyshev inequality give, with required high probability, the following discretized bound that is $N^{100\delta}$-worse than the previous bound:
\small\begin{align}
[\mathbf{1}_{T\geq\mathfrak{t}_{\mathrm{reg}}}\grad_{-\mathfrak{t}_{\mathrm{reg}}}^{\mathbf{T}}\bar{\mathbf{H}}_{T,x}^{N,\mathbf{X}}\bar{\mathbf{Z}}_{0,\bullet}^{N}]_{\mathfrak{t}_{\mathrm{st}};\mathbb{T}_{N}} \ &\lesssim_{\delta} \ N^{-\frac12+101\delta} + N^{300\delta}\mathfrak{t}_{\mathrm{reg}}^{1/4}. \label{eq:TRGTProp11}
\end{align}\normalsize\normalsize
We already treated the second term in the equation for $\bar{\mathbf{Z}}^{N}$ so we move to the $\bar{\Phi}^{N,2}$-term. Note $\|\bar{\Phi}^{N,2}\|_{\mathfrak{t}_{\mathrm{st}};\mathbb{T}_{N}}\lesssim N^{1/2}\|\bar{\mathbf{Z}}^{N}\|_{\mathfrak{t}_{\mathrm{st}};\mathbb{T}_{N}}$; we apply the time-regularity estimate \eqref{eq:HKETR1} for the heat operator to get the following deterministic estimate for any $\delta\in\R_{>0}$:
\small\begin{align}
\|\mathbf{1}_{T\geq\mathfrak{t}_{\mathrm{reg}}}\grad_{-\mathfrak{t}_{\mathrm{reg}}}^{\mathbf{T}}\bar{\mathbf{H}}_{T,x}^{N}(\bar{\Phi}^{N,2})\|_{\mathfrak{t}_{\mathrm{st}};\mathbb{T}_{N}} \ \lesssim_{\delta} \ N^{2\delta}\|\bar{\Phi}^{N,2}\|_{\mathfrak{t}_{\mathrm{st}};\mathbb{T}_{N}}\mathfrak{t}_{\mathrm{reg}} \ \lesssim \ N^{\frac12+2\delta}\mathfrak{t}_{\mathrm{reg}}\|\bar{\mathbf{Z}}^{N}\|_{\mathfrak{t}_{\mathrm{st}};\mathbb{T}_{N}} \ \lesssim \ \mathfrak{t}^{1/4}_{\mathrm{reg}}\|\bar{\mathbf{Z}}^{N}\|_{\mathfrak{t}_{\mathrm{st}};\mathbb{T}_{N}}. \label{eq:TRGTProp12}
\end{align}\normalsize\normalsize
The final estimate giving the far RHS of \eqref{eq:TRGTProp12} follows by recalling that $\mathfrak{t}_{\mathrm{reg}}\lesssim N^{-1}$, and thus $\mathfrak{t}_{\mathrm{reg}} \lesssim \mathfrak{t}_{\mathrm{reg}}^{3/4}\mathfrak{t}_{\mathrm{reg}}^{1/4} \lesssim N^{-3/4}\mathfrak{t}_{\mathrm{reg}}^{1/4}$.

We now move to estimating the time-regularity of the heat operator term in the fixed-point equation for $\bar{\mathbf{Z}}^{N}$ corresponding to $\bar{\Phi}^{N,3}$. For the first term in $\bar{\Phi}^{N,3}$ we can use \eqref{eq:HKETR1} as we did to establish \eqref{eq:TRGTProp12} for $\bar{\Phi}^{N,2}$. However, the second term in $\bar{\Phi}^{N,3}$ has a gradient that we need to take advantage of as a naive bound of order $N$ would not suffice. To this end, we will instead use the space-time gradient bound \eqref{eq:HKETRXR1}. Recalling all coefficients hitting $\bar{\mathbf{Z}}^{N}$ in $\bar{\Phi}^{N,3}$ are uniformly bounded, for any $\delta\in\R_{>0}$ we have 
\small\begin{align}
\|\mathbf{1}_{T\geq\mathfrak{t}_{\mathrm{reg}}}\grad_{-\mathfrak{t}_{\mathrm{reg}}}^{\mathbf{T}}\bar{\mathbf{H}}_{T,x}^{N}(\bar{\Phi}^{N,3})\|_{\mathfrak{t}_{\mathrm{st}};\mathbb{T}_{N}} \ \lesssim_{\delta} \ \mathfrak{t}_{\mathrm{reg}}^{\frac14-\delta}\|\bar{\mathbf{Z}}^{N}\|_{\mathfrak{t}_{\mathrm{st}};\mathbb{T}_{N}} \ \lesssim \ N^{200\delta}\mathfrak{t}_{\mathrm{reg}}^{\frac14}\|\bar{\mathbf{Z}}^{N}\|_{\mathfrak{t}_{\mathrm{st}};\mathbb{T}_{N}}. \label{eq:TRGTProp13}
\end{align}\normalsize\normalsize
We now combine \eqref{eq:TRGTProp11}, \eqref{eq:TRGTProp12}, \eqref{eq:TRGTProp13}, and Lemma \ref{lemma:C0SIProp} along with the fixed-point equation for $\bar{\mathbf{Z}}^{N}$ to complete the proof upon allowing an additional term on the RHS of the proposed estimate from the statement of Lemma \ref{lemma:TRGTProp1} of type $N^{200\delta}\mathfrak{t}_{\mathrm{reg}}^{1/4}\|\bar{\mathbf{Z}}^{N}\|_{\mathfrak{t}_{\mathrm{st}};\mathbb{T}_{N}}$. In particular, we get something with the correct $\mathfrak{t}_{\mathrm{reg}}$-dependence but that is linear in $\|\bar{\mathbf{Z}}^{N}\|_{\mathfrak{t}_{\mathrm{st}};\mathbb{T}_{N}}$. However, to resolve such minor issue we use $\|\bar{\mathbf{Z}}^{N}\|_{\mathfrak{t}_{\mathrm{st}};\mathbb{T}_{N}} \lesssim 1 + \|\bar{\mathbf{Z}}^{N}\|_{\mathfrak{t}_{\mathrm{st}};\mathbb{T}_{N}}^{2}$ to get rid of this cosmetic difference, and this completes the proof.
\end{proof}
\begin{proof}[Proof of \emph{Lemma \ref{lemma:C0SIProp}}]
Define $\wt{\mathbf{V}} \overset{\bullet}= 1 + \|\bar{\mathbf{V}}^{N}\|_{\mathfrak{t}_{\mathrm{st}};\mathbb{T}_{N}}^{2}$. We claim that it suffices to establish the following for all $p\in\R_{\geq1}$ and $\e\in\R_{>0}$ for reasons we explain afterwards that are related to the moment estimates/union bounds in the proof of Proposition \ref{prop:Ctify}:
\small\begin{align}
\sup_{0\leq T\leq\mathfrak{t}_{\mathrm{st}}}\sup_{x\in\mathbb{T}_{N}}\left\|\wt{\mathbf{V}}^{-1}|\grad_{-\mathfrak{t}_{\mathrm{reg}}}^{\mathbf{T}}\bar{\mathbf{H}}_{T,x}^{N}(\bar{\mathbf{V}}^{N}\d\xi^{N})|\right\|_{\omega;2p} \ &\lesssim_{p,\e} \ N^{\frac12\e}\mathfrak{t}_{\mathrm{reg}}^{\frac14} + N^{-\frac12+\frac12\e}. \label{eq:C0SIProp1}
\end{align}\normalsize\normalsize
We have an additional $N^{-\e/2}$-factor in  \eqref{eq:C0SIProp1}, so by Chebyshev's inequality for $p$-th moments we deduce the desired bound in Lemma \ref{lemma:C0SIProp} \emph{for any fixed} $0\leq T\leq\mathfrak{t}_{\mathrm{st}}$ and $x\in\mathbb{T}_{N}$ outside an event of probability $\kappa_{p,\e}N^{-p\e/2}$. A union bound over all points in the discretization $\mathfrak{I}_{\mathfrak{t}_{\mathrm{st}}}\times\mathbb{T}_{N}$ then gives the proposed estimate uniformly over $\mathfrak{I}_{\mathfrak{t}_{\mathrm{st}}}\times\mathbb{T}_{N}$ outside an event of probability $\kappa_{p,\e}N^{-p\e/2+200}$ as in the proof for Proposition \ref{prop:Ctify}. Choosing $p\gtrsim_{\e}1$ sufficiently large depending only on $\e\in\R_{>0}$ completes the proof. To prove \eqref{eq:C0SIProp1}, take $T\geq\mathfrak{t}_{\mathrm{reg}}$ and then take $\mathfrak{t}_{\mathrm{reg}}\geq N^{-2}$. We start with the following decomposition of the time-gradient on the LHS of \eqref{eq:C0SIProp1} into a short-time/length-$\mathfrak{t}_{\mathrm{reg}}$-integral and an integral of a time-gradient of the heat kernel. First:
\begin{itemize}[leftmargin=*]
\item Decompose the time-integration domain $[0,T]$ into $I_{1}\cup I_{2}$ where $I_{1}$ looks at times until $T-\mathfrak{t}_{\mathrm{reg}}$ and $I_{2}$ looks at the rest.
\end{itemize}
An elementary calculation with the time-gradient of the heat operator, like with the proof for time-regularity in Proposition 3.2 in \cite{DT}, gives control on the time-gradient of $\bar{\mathbf{H}}^{N}$ by that of the integration-domain in $\bar{\mathbf{H}}^{N}$ and that of the heat kernel therein:
\small\begin{align}
|\grad_{-\mathfrak{t}_{\mathrm{reg}}}^{\mathbf{T}}\bar{\mathbf{H}}_{T,x}^{N}(\bar{\mathbf{V}}^{N}\d\xi^{N})| \ &\leq \ |\bar{\mathbf{H}}_{T,x}^{N}(\bar{\mathbf{V}}_{S,y}^{N}\d\xi_{S,y}^{N}\mathbf{1}_{S\in I_{2}})| \ + \ |\int_{0}^{T-\mathfrak{t}_{\mathrm{reg}}}{\sum}_{y\in\mathbb{T}_{N}}\grad_{-\mathfrak{t}_{\mathrm{reg}}}^{\mathbf{T}}\mathbf{H}_{S,T,x,y}^{N}\bar{\mathbf{V}}_{S,y}^{N}\d\xi_{S,y}^{N}| \ \overset{\bullet}= \ \Phi_{1}+\Phi_{2}.  \label{eq:C0SIProp0}
\end{align}\normalsize
We now consider the following decomposition for both $\Phi_{1},\Phi_{2}$ in terms of level sets of the ``squared-norm"-type quantity $\wt{\mathbf{V}}\geq1$:
\small\begin{align}
|\Phi_{\mathfrak{i}}| \ &\leq \ {\sum}_{\mathfrak{l}=1}^{\infty}|\Phi_{\mathfrak{i}}|\cdot\mathbf{1}_{\mathfrak{l}\leq(\wt{\mathbf{V}})^{\frac12}\leq\mathfrak{l}+1}.
\end{align}\normalsize\normalsize
We observe the product of any indicator function in the summation on the RHS of the above inequality with any other of these indicator functions can only avoid vanishing when multiplied by at most 2 other of these indicator functions. Thus, binomial expansion and this ``almost-orthogonality" of the indicator functions give the following multiscale estimate for $p\in\R_{\geq1}$:
\small\begin{align}
\|\wt{\mathbf{V}}^{-1}|\Phi_{\mathfrak{i}}|\|_{\omega;2p}^{2p} \ \lesssim_{p} \ \sum_{\mathfrak{l}=1}^{\infty} \|\wt{\mathbf{V}}^{-1}|\Phi_{\mathfrak{i}}|\mathbf{1}_{\mathfrak{l}\leq(\mathbf{V})^{\frac12}\leq\mathfrak{l}+1}\|_{\omega;2p}^{2p} \ &\lesssim \ \sum_{\mathfrak{l}=1}^{\infty} \mathfrak{l}^{-4p} \| \Phi_{\mathfrak{i}}\mathbf{1}_{\mathfrak{l}\leq(\mathbf{V})^{\frac12}\leq\mathfrak{l}+1}\|_{\omega;2p}^{2p} \ \lesssim \ \sup_{\mathfrak{l} \in \Z_{\geq1}} \mathfrak{l}^{-2p}\|\Phi_{\mathfrak{i}}\mathbf{1}_{(\mathbf{V})^{\frac12}\leq\mathfrak{l}+1}\|_{\omega;2p}^{2p}\sum_{\mathfrak{l}=1}^{\infty}\mathfrak{l}^{-2p} \nonumber \\
&\lesssim \ \sup_{\mathfrak{l} \in \Z_{\geq1}} \mathfrak{l}^{-2p}\|\Phi_{\mathfrak{i}}\mathbf{1}_{(\mathbf{V})^{\frac12}\leq\mathfrak{l}+1}\|_{\omega;2p}^{2p}. \label{eq:C0SIProp2}
\end{align}\normalsize\normalsize
At this point we will proceed with moment estimates almost exactly as in the proof of Proposition 3.2 in \cite{DT}. We will only use the indicator function hitting $\Phi_{\mathfrak{i}}$ in \eqref{eq:C0SIProp2} to get an a priori estimate; we afterwards quickly remove this indicator function. We now specialize the index $\mathfrak{i}\in\{1,2\}$; we start with $\mathfrak{i}=1$. In particular, we will estimate $\Phi_{1}$ starting with the estimate \eqref{eq:C0SIProp2}. To this end, we first define the uniformly bounded process $\mathbf{R}^{N} = \mathfrak{l}^{-1}\bar{\mathbf{V}}^{N}\mathbf{1}[|\mathfrak{l}^{-1}\bar{\mathbf{V}}^{N}| \lesssim 1]$. Concluding with the martingale bound Lemma \ref{lemma:MG}, which is a version of Lemma 3.1 in \cite{DT} and additionally needs uniform boundedness/adaptedness of $\mathbf{R}^{N}$, we have
\small\begin{align}
&\mathfrak{l}^{-2p}\|\Phi_{1}\mathbf{1}_{(\mathbf{V})^{\frac12}\leq\mathfrak{l}+1}\|_{\omega;2p}^{2p} \ \leq \ \E|\int_{T-\mathfrak{t}_{\mathrm{reg}}}^{T}\sum_{y\in\mathbb{T}_{N}}\bar{\mathbf{H}}_{S,T,x,y}^{N} \cdot \mathfrak{l}^{-1} \bar{\mathbf{V}}_{S,y}^{N}\d\xi_{S,y}^{N}|^{2p} \mathbf{1}_{\|\bar{\mathbf{V}}\|_{\mathfrak{t}_{\mathrm{st}};\mathbb{T}_{N}} \lesssim \mathfrak{l}+1} \\
&\leq \ \E|\int_{T-\mathfrak{t}_{\mathrm{reg}}}^{T}\sum_{y\in\mathbb{T}_{N}}\bar{\mathbf{H}}_{S,T,x,y}^{N} \cdot \mathbf{R}_{S,y}^{N}\d\xi_{S,y}^{N}|^{2p} \ \lesssim_{p} \ \left( \int_{T-\mathfrak{t}_{\mathrm{reg}}}^{T}N\sum_{y\in\mathbb{T}_{N}}|\bar{\mathbf{H}}_{S,T,x,y}^{N}|^{2} \ \d S \right)^{p} \ \overset{\bullet}= \ \Psi_{1}. \label{eq:C0SIProp3}
\end{align}\normalsize\normalsize
Indeed if $\|\bar{\mathbf{V}}\|_{\mathfrak{t}_{\mathrm{st}};\mathbb{T}_{N}} \lesssim \mathfrak{l}+1 \leq 2\mathfrak{l}$, then $\mathfrak{l}^{-1}\bar{\mathbf{V}}^{N}$ is uniformly bounded above; in this case the cutoff $\mathbf{R}^{N}$ does nothing and we get $\mathbf{R}^{N} = \mathfrak{l}^{-1}\bar{\mathbf{V}}^{N}$. We note that $\mathbf{R}^{N}$ is adapted because $\bar{\mathbf{V}}^{N}$ is adapted and therefore so is the indicator function in $\mathbf{R}^{N}$ while $\mathfrak{l} \in \Z_{\geq1}$ is deterministic. We now estimate $\Phi_{2}$ and we start with \eqref{eq:C0SIProp2} with the index $\mathfrak{i}=2$. By Lemma \ref{lemma:MG}, we similarly get
\small\begin{align}
\mathfrak{l}^{-2p}\|\Phi_{2}\mathbf{1}_{(\mathbf{V})^{\frac12}\leq\mathfrak{l}+1}\|_{\omega;2p}^{2p} \ &\leq \ \E|\int_{0}^{T-\mathfrak{t}_{\mathrm{reg}}}\sum_{y\in\mathbb{T}_{N}}\grad_{-\mathfrak{t}_{\mathrm{reg}}}^{\mathbf{T}}\bar{\mathbf{H}}_{S,T,x,y}^{N}\cdot \mathbf{R}_{S,y}^{N}\d\xi_{S,y}^{N}|^{2p} \mathbf{1}_{\|\bar{\mathbf{V}}\|_{\mathfrak{t}_{\mathrm{st}};\mathbb{T}_{N}} \lesssim \mathfrak{l}+1} \\
&\lesssim_{p} \ \left(\int_{0}^{T-\mathfrak{t}_{\mathrm{reg}}}N\sum_{y\in\mathbb{T}_{N}}|\grad_{-\mathfrak{t}_{\mathrm{reg}}}^{\mathbf{T}}\bar{\mathbf{H}}_{S,T,x,y}^{N}|^{2} \ \d S\right)^{p}\ \overset{\bullet}= \ \Psi_{2}. \label{eq:C0SIProp6}
\end{align}\normalsize\normalsize
The second line requires dropping the indicator from the bound before. Combining \eqref{eq:C0SIProp0}, \eqref{eq:C0SIProp2}, \eqref{eq:C0SIProp3}, and \eqref{eq:C0SIProp6} provides the following upper bound which is familiar from the proof of time-regularity in Proposition 3.2 in \cite{DT}:
\small\begin{align}
\left\|\wt{\mathbf{V}}^{-1}|\grad_{-\mathfrak{t}_{\mathrm{reg}}}^{\mathbf{T}}\bar{\mathbf{H}}_{T,x}^{N}(\bar{\mathbf{V}}^{N}\d\xi^{N})|\right\|_{\omega;2p}^{2p} \ \lesssim_{p} \ \sum_{\mathfrak{i}=1,2}\|\wt{\mathbf{V}}^{-1}|\Phi_{\mathfrak{i}}|\|_{\omega;2p}^{2p} \ &\lesssim_{p} \ \sum_{\mathfrak{i}=1,2}\sup_{\mathfrak{l} \in \Z_{\geq1}} \mathfrak{l}^{-2p}\|\Phi_{\mathfrak{i}}\mathbf{1}_{(\mathbf{V})^{\frac12}\leq\mathfrak{l}+1}\|_{\omega;2p}^{2p} \ \lesssim_{p} \ \sum_{\mathfrak{i}=1,2}\Psi_{\mathfrak{i}}. \label{eq:C0SIProp7}
\end{align}\normalsize\normalsize
We now follow the proof of (3.14) in Proposition 3.2 in \cite{DT} to bound $\Psi_{\mathfrak{i}}$-terms from \eqref{eq:C0SIProp3} and \eqref{eq:C0SIProp6} and use the on-diagonal heat kernel bounds in Corollary A.2 in \cite{DT} that hold for the $\bar{\mathbf{H}}^{N}$-heat kernel for order 1 times; see Remark \ref{remark:ch2HKE}. For any $\delta \in \R_{>0}$, we obtain the following final estimate in this proof:
\small\begin{align}
\left\|\wt{\mathbf{V}}^{-1}|\grad_{-\mathfrak{t}_{\mathrm{reg}}}^{\mathbf{T}}\bar{\mathbf{H}}_{T,x}^{N}(\bar{\mathbf{V}}^{N}\d\xi^{N})|\right\|_{\omega;2p} \ &\lesssim_{p,\delta} \ \mathfrak{t}_{\mathrm{reg}}^{\frac14-\delta} \ \lesssim \ N^{2\delta} \mathfrak{t}_{\mathrm{reg}}^{\frac14}.
\end{align}\normalsize\normalsize
The final bound follows by the assumption $\mathfrak{t}_{\mathrm{reg}}\geq N^{-2}$. This completes the proof for $\mathfrak{t}_{\mathrm{reg}}\geq N^{-2}$. In the case $\mathfrak{t}_{\mathrm{reg}}\leq N^{-2}$, we use the same strategy, but Lemma \ref{lemma:MG} provides, instead of the $\Psi_{\mathfrak{i}}$-terms in \eqref{eq:C0SIProp7}, the upper bound of $N^{-1/2}$ times $p$-dependent factors so we are done in the case $T\geq\mathfrak{t}_{\mathrm{reg}}$. To remove this final assumption for general times we can write the $\mathfrak{t}_{\mathrm{reg}}$-time-gradient as a time-gradient on a time-scale $\mathfrak{t}$ which is at most $\mathfrak{t}_{\mathrm{reg}}$ because we extend everything to negative times by their value at time 0. The previous argument then gives $\mathfrak{t}$-versions of the estimates above which are certainly controlled by the $\mathfrak{t}_{\mathrm{reg}}$-estimates above since $\mathfrak{t}\leq\mathfrak{t}_{\mathrm{reg}}$. This was the same observation that we made in the proof of Lemma \ref{lemma:TRGTProp1}. This completes the proof.
\end{proof}
\subsection{Proof of Lemma \ref{lemma:TRGTProp3}}
The differences between $\|\|_{\mathfrak{t}_{\mathrm{st}};\mathbb{T}_{N}}$ and $[]_{\mathfrak{t}_{\mathrm{st}};\mathbb{T}_{N}}$ come via behavior between times in the discretization $\mathfrak{I}_{\mathfrak{t}_{\mathrm{st}}}$ defining $[]_{\mathfrak{t}_{\mathrm{st}};\mathbb{T}_{N}}$. To apply this to $\bar{\mathbf{Z}}^{N}$, we note the following SDE-type formulation for $\bar{\mathbf{Z}}^{N}$-dynamics and explain it after:
\small\begin{align}
\d\bar{\mathbf{Z}}_{\mathfrak{t},x}^{N} \ &= \ \bar{\mathscr{L}}^{!!}\bar{\mathbf{Z}}_{\mathfrak{t},x}^{N}\d\mathfrak{t} \ + \ \bar{\mathbf{Z}}_{\mathfrak{t},x}^{N}\d\xi_{\mathfrak{t},x}^{N} \ + \ \bar{\Phi}^{N,2}_{\mathfrak{t},x}\d\mathfrak{t} \ + \ \bar{\Phi}^{N,3}_{\mathfrak{t},x}\d\mathfrak{t}. \label{eq:TRGTProp3SDE}
\end{align}\normalsize\normalsize
By SDE-type we mean $\d\xi^{N}$ is a discrete approximation to Brownian motions. We recall the notation in this equation in Section \ref{section:Ctify}. We emphasize \eqref{eq:TRGTProp3SDE} can be checked by taking the time-differential of the defining linear equation for $\bar{\mathbf{Z}}^{N}$ in Definition \ref{definition:ChiTorus}.

We observe the SDE-type equation \eqref{eq:TRGTProp3SDE} admits a continuous part corresponding to the first, third, and fourth terms on the RHS of \eqref{eq:TRGTProp3SDE} as well as the continuous ``compensation" part of the compensated Poisson process $\d\xi^{N}$. This latter continuous part in $\d\xi^{N}$ is a random drift with speed bounded by $N^{2}$ times universal factors. If \eqref{eq:TRGTProp3SDE} only had these continuous parts, then \eqref{eq:TRGTProp3SDE} would be a random \emph{linear} ODE. Short-time estimates would then follow by standard procedure for linear ODE.

As for jumps in \eqref{eq:TRGTProp3SDE}, by definition of $\d\xi^{N}$ from Proposition \ref{prop:Duhamel}/(2.4) from \cite{DT}, any jump in the SDE-type equation \eqref{eq:TRGTProp3SDE} is order $N^{-1/2}$. We will use this and control on the number of jumps in short times via tail bounds for the Poisson distribution.
\begin{lemma}\label{lemma:TRGTProp2}
 Outside an event of probability at most $\kappa_{\e,C}N^{-C}$ for $C$ arbitrarily large and $\e>0$ arbitrarily small but both universal, we have the following uniformly over $0\leq T\leq2$ and $x\in\mathbb{T}_{N}$ simultaneously on this same high-probability event:
\small\begin{align}
\sup_{0 \leq \mathfrak{s} \leq N^{-99}}|\grad_{\pm\mathfrak{s}}^{\mathbf{T}}\bar{\mathbf{Z}}_{T,x}^{N}| \ &\lesssim_{\e} \ N^{-\frac12+\e}\left(1+\|\bar{\mathbf{Z}}^{N}\|_{T;\mathbb{T}_{N}}^{2}\right) \ = \ N^{-\frac12+\e} + N^{-\frac12+\e} \|\bar{\mathbf{Z}}^{N}\|_{T;\mathbb{T}_{N}}^{2}. \label{eq:TRGTProp2}
\end{align}\normalsize\normalsize
\end{lemma}
\begin{proof}[Proof of \emph{Lemma \ref{lemma:TRGTProp3}}]
We will first prove \eqref{eq:TRGTProp3I}. Provided any $\mathfrak{t} \in [0,\mathfrak{t}_{\mathrm{st}}]$, we first employ the following decomposition in which $\mathfrak{t}_{\mathfrak{j}}\in\R_{\geq0}$ denotes the time \emph{in the discretization} $\mathfrak{I}_{\mathfrak{t}_{\mathrm{st}}} = \{\mathfrak{t}_{\mathrm{st}} \cdot \mathfrak{j}N^{-100}\}_{j=1}^{N^{100}}$ that is closest to $\mathfrak{t}$, and $x\in\mathbb{T}_{N}$ is any arbitrary point:
\small\begin{align}
|\grad_{-\mathfrak{t}_{\mathrm{reg}}}\bar{\mathbf{Z}}_{\mathfrak{t},x}^{N}| \ = \ |\bar{\mathbf{Z}}_{\mathfrak{t}-\mathfrak{t}_{\mathrm{reg}},x}^{N} - \bar{\mathbf{Z}}_{\mathfrak{t},x}^{N}| \ &\leq \ |\bar{\mathbf{Z}}_{\mathfrak{t}-\mathfrak{t}_{\mathrm{reg}},x}^{N} - \bar{\mathbf{Z}}_{\mathfrak{t}_{\mathfrak{j}}-\mathfrak{t}_{\mathrm{reg}},x}^{N}| + |\bar{\mathbf{Z}}_{\mathfrak{t}_{\mathfrak{j}}-\mathfrak{t}_{\mathrm{reg}},x}^{N} - \bar{\mathbf{Z}}_{\mathfrak{t}_{\mathfrak{j}},x}^{N}| + |\bar{\mathbf{Z}}_{\mathfrak{t}_{\mathfrak{j}},x}^{N} - \bar{\mathbf{Z}}_{\mathfrak{t},x}^{N}| \\
&\leq \ |\bar{\mathbf{Z}}_{\mathfrak{t}-\mathfrak{t}_{\mathrm{reg}},x}^{N} - \bar{\mathbf{Z}}_{\mathfrak{t}_{\mathfrak{j}}-\mathfrak{t}_{\mathrm{reg}},x}^{N}| + |\bar{\mathbf{Z}}_{\mathfrak{t}_{\mathfrak{j}},x}^{N} - \bar{\mathbf{Z}}_{\mathfrak{t},x}^{N}| + [\grad_{-\mathfrak{t}_{\mathrm{reg}}}^{\mathbf{T}}\bar{\mathbf{Z}}_{T,x}^{N}]_{\mathfrak{t}_{\mathrm{st}};\mathbb{T}_{N}}. \label{eq:TRGTProp3I1}
\end{align}\normalsize\normalsize
Observe $|\mathfrak{t}-\mathfrak{t}_{\mathfrak{j}}| \leq N^{-100}$, as the points in $\mathfrak{I}_{\mathfrak{t}_{\mathrm{st}}}$ are separated by at most $N^{-100}$. Thus, estimating the first two terms in \eqref{eq:TRGTProp3I1} with the desired high probability via Lemma \ref{lemma:TRGTProp2} with $T = \mathfrak{t}-\mathfrak{t}_{\mathrm{reg}},\mathfrak{t}$ and $\mathfrak{s} = |\mathfrak{t}-\mathfrak{t}_{\mathfrak{j}}|$ gives \eqref{eq:TRGTProp3I} as $\mathfrak{t}-\mathfrak{t}_{\mathrm{reg}},\mathfrak{t}\leq\mathfrak{t}_{\mathrm{st}}$ so we may bound the RHS of \eqref{eq:TRGTProp2} for these choices of $T$ with $1+\|\bar{\mathbf{Z}}^{N}\|_{\mathfrak{t};\mathbb{T}_{N}}^{2},1+\|\bar{\mathbf{Z}}^{N}\|_{\mathfrak{t}-\mathfrak{t}_{\mathrm{reg}};\mathbb{T}_{N}}^{2}\leq1+\|\bar{\mathbf{Z}}^{N}\|_{\mathfrak{t}_{\mathrm{st}};\mathbb{T}_{N}}^{2}$.

We now show \eqref{eq:TRGTProp3II}. If $\mathfrak{t}_{\mathrm{st}} \leq \wt{\mathfrak{t}}_{\mathrm{st}}$, there is nothing to do, so assume $\wt{\mathfrak{t}}_{\mathrm{st}} \leq \mathfrak{t}_{\mathrm{st}}$. Take any $\mathfrak{t} \in [\wt{\mathfrak{t}}_{\mathrm{st}},\mathfrak{t}_{\mathrm{st}}]$ and any $x\in\mathbb{T}_{N}$. As above,
\small\begin{align}
|\grad_{-\mathfrak{t}_{\mathrm{reg}}}\bar{\mathbf{Z}}_{\mathfrak{t},x}^{N}| \ = \ |\bar{\mathbf{Z}}_{\mathfrak{t}-\mathfrak{t}_{\mathrm{reg}},x}^{N}-\bar{\mathbf{Z}}_{\mathfrak{t},x}^{N}| \ &\leq \ |\bar{\mathbf{Z}}_{\mathfrak{t}-\mathfrak{t}_{\mathrm{reg}},x}^{N} - \bar{\mathbf{Z}}_{\wt{\mathfrak{t}}_{\mathrm{st}}-\mathfrak{t}_{\mathrm{reg}},x}^{N}| + |\bar{\mathbf{Z}}_{\wt{\mathfrak{t}}_{\mathrm{st}}-\mathfrak{t}_{\mathrm{reg}},x}^{N} - \bar{\mathbf{Z}}_{\wt{\mathfrak{t}}_{\mathrm{st}},x}^{N}| + |\bar{\mathbf{Z}}_{\wt{\mathfrak{t}}_{\mathrm{st}},x}^{N} - \bar{\mathbf{Z}}_{\mathfrak{t},x}^{N}| \\
&\leq \ |\bar{\mathbf{Z}}_{\mathfrak{t}-\mathfrak{t}_{\mathrm{reg}},x}^{N} - \bar{\mathbf{Z}}_{\wt{\mathfrak{t}}_{\mathrm{st}}-\mathfrak{t}_{\mathrm{reg}},x}^{N}| + |\bar{\mathbf{Z}}_{\wt{\mathfrak{t}}_{\mathrm{st}},x}^{N} - \bar{\mathbf{Z}}_{\mathfrak{t},x}^{N}| + \|\grad_{-\mathfrak{t}_{\mathrm{reg}}}^{\mathbf{T}}\bar{\mathbf{Z}}_{T,x}^{N}\|_{\wt{\mathfrak{t}}_{\mathrm{st}};\mathbb{T}_{N}}. \label{eq:TRGTProp3II1}
\end{align}\normalsize\normalsize
Observe that $|\mathfrak{t}-\wt{\mathfrak{t}}_{\mathrm{st}}| = |\mathfrak{t}-\mathfrak{t}_{\mathrm{reg}} - (\wt{\mathfrak{t}}_{\mathrm{st}}-\mathfrak{t}_{\mathrm{reg}})| \leq N^{-100}$ because $|\mathfrak{t}_{\mathrm{st}}-\wt{\mathfrak{t}}_{\mathrm{st}}| \leq N^{-100}$ by construction. Thus, as with deducing \eqref{eq:TRGTProp3I} by \eqref{eq:TRGTProp3I1}, we get \eqref{eq:TRGTProp3II} with the required high-probability from \eqref{eq:TRGTProp3II1} via Lemma \ref{lemma:TRGTProp2} with $T = \mathfrak{t}-\mathfrak{t}_{\mathrm{reg}},\mathfrak{t}$ and $\mathfrak{s} = |\wt{\mathfrak{t}}_{\mathrm{st}}-\mathfrak{t}|$.
\end{proof}
\begin{proof}[Proof of \emph{Lemma \ref{lemma:TRGTProp2}}]
We prove \eqref{eq:TRGTProp2} for forwards time-gradient, so $+\mathfrak{s}$. The proof for $-\mathfrak{s}$ amounts to the same argument but backwards in time. Define $\mathfrak{I}_{\mathfrak{j}} \overset{\bullet}= [\mathfrak{t}_{\mathfrak{j}},\mathfrak{t}_{\mathfrak{j}+1}]$ with $\mathfrak{t}_{\mathfrak{j}} = \mathfrak{j}N^{-99}$ and $\mathfrak{j}\in\llbracket0,3N^{99}\rrbracket$. We restrict to the following event denoted by $\mathscr{E}$.
\begin{itemize}[leftmargin=*]
\item The number of ringings among all Poisson clocks on $\mathbb{T}_{N}$ in any $\mathfrak{I}_{\mathfrak{j}}$-interval is at most $N^{\delta}$ with $\delta \in \R_{>0}$ small but universal.
\end{itemize}
We first observe that if $\mathscr{E}^{C}$ is the complement event to $\mathscr{E}$, we have $\mathbf{P}[\mathscr{E}^{C}] \lesssim_{\delta,C} N^{-C}$. Indeed, the sum of all Poisson clocks over all points in $\mathbb{T}_{N}$ is another Poisson clock of rate at most $N^{10}$ times uniformly bounded factors upon considering there are at most $N^{5/4+\e_{\mathrm{cpt}}}$-many clocks each of rate at most $N^{2}$, both again times uniformly bounded factors. Thus, the probability that there are more than $N^{\delta}$-many ringings in any interval $\mathfrak{I}_{\mathfrak{j}}$ of time-scale $N^{-99}$ is exponentially small in $N \in \Z_{>0}$. The estimate for $\mathbf{P}[\mathscr{E}^{C}]$ then follows by taking a union bound of these \emph{exponentially small} probabilities over all \emph{polynomially many} $\mathfrak{j}\in\llbracket0,3N^{99}\rrbracket$.

We proceed to show that on the event $\mathscr{E}$ above, the proposed estimate \eqref{eq:TRGTProp2} holds. We make the following observations.
\begin{itemize}[leftmargin=*]
\item Observe with probability 1 on the event $\mathscr{E}$, for any $0 \leq T \leq 2$ the total number of Poisson clocks ringing in the time-window $[T,T+N^{-99}]$ is at most $N^{2\delta}$, as this time-window is covered by a uniformly bounded number of intervals $\mathfrak{I}_{\mathfrak{j}}$ for $\mathfrak{j}\in\llbracket0,3N^{99}\rrbracket$.
\item Between jumps \eqref{eq:TRGTProp3SDE} gives $\partial_{T}\bar{\mathbf{Z}}_{T,x}^{N} = \mathscr{Q}\bar{\mathbf{Z}}_{T,x}^{N}$, where $\mathscr{Q}$ is a linear operator satisfying the following with $\kappa\in\R_{\geq0}$ universal:
\small\begin{align}
{\sup}_{x\in\mathbb{T}_{N}}|\mathscr{Q}\bar{\mathbf{Z}}_{T,x}^{N}| \ &\leq \ \kappa N^{2}{\sup}_{x\in\mathbb{T}_{N}}|\bar{\mathbf{Z}}_{T,x}^{N}|.
\end{align}\normalsize\normalsize
This can be checked directly. Thus, provided any $T \in [0,\mathfrak{t}_{\mathrm{st}}]$, for all $T+\mathfrak{s} \in \R_{\geq0}$ before the first jump time after $T \in \R_{\geq0}$, assuming $\mathfrak{s}\leq N^{-99}$ we have the following by expanding the relevant exponential matrix/operator below in a power series:
\small\begin{align}
|\grad_{\mathfrak{s}}^{\mathbf{T}}\bar{\mathbf{Z}}_{T,x}^{N}| \ = \ |\exp(\mathfrak{s}\mathscr{Q})\bar{\mathbf{Z}}^{N}_{T,x}- \bar{\mathbf{Z}}_{T,x}^{N}| \ \leq \ \sum_{\mathfrak{l}=1}^{\infty} \mathfrak{s}^{\mathfrak{l}}(\mathfrak{l}!)^{-1}|\mathscr{Q}^{\mathfrak{l}}\bar{\mathbf{Z}}_{T,x}^{N}| \ \leq \ \sum_{\mathfrak{l}=1}^{\infty}N^{-99\mathfrak{l}}(\mathfrak{l}!)^{-1}\kappa^{\mathfrak{l}}N^{2\mathfrak{l}} \|\bar{\mathbf{Z}}^{N}\|_{T;\mathbb{T}_{N}} \ &\lesssim \ N^{-97}\|\bar{\mathbf{Z}}^{N}\|_{T;\mathbb{T}_{N}}. \label{eq:TRGT4}
\end{align}\normalsize\normalsize
\item The jumps of $\bar{\mathbf{Z}}^{N}$ at any space-time point have sizes that are bounded above by $N^{-1/2}$ times the value of $\bar{\mathbf{Z}}^{N}$ at that space-time point times uniformly bounded factors. This follows by definition of the martingale differential $\d\xi^{N}$ within (2.4) of \cite{DT}.
\end{itemize}
Iterating the last two bullet points at most $N^{2\delta}$-many times, given the first bullet point in that list, yields the proof if we replace $1+\|\bar{\mathbf{Z}}_{T,x}^{N}\|_{T;\mathbb{T}_{N}}^{2}$ with $\|\bar{\mathbf{Z}}_{T,x}^{N}\|_{T;\mathbb{T}_{N}}$ in the proposed bound. But $\|\bar{\mathbf{Z}}_{T,x}^{N}\|_{T;\mathbb{T}_{N}}\lesssim1+\|\bar{\mathbf{Z}}_{T,x}^{N}\|_{T;\mathbb{T}_{N}}^{2}$ by the Cauchy-Schwarz inequality, so we deduce the proposed estimate as written as well.
\end{proof}
%
%
%
\section{Key Estimates for Pseudo-Gradients and $\Phi^{N,2}$}\label{section:KPZ2}
We will control $\bar{\mathbf{H}}^{N}(\bar{\Phi}^{N,2})$ in the defining equation for $\bar{\mathbf{Z}}^{N}$ via the following version of Pseudo-Proposition \ref{pprop:S1}.
\begin{prop}\label{prop:KPZNL}
 Consider a random time $\mathfrak{t}_{\mathrm{st}} \in \R_{\geq0}$ with $\mathfrak{t}_{\mathrm{st}}\leq1$ with probability 1. There exist universal constants $\beta_{\mathrm{univ},1}\in\R_{>0}$ and $\beta_{\mathrm{univ},2}\in\R_{>0}$ such that outside an event of probability at most $N^{-\beta_{\mathrm{univ},1}}$ times uniformly bounded factors, we have
\small\begin{align}
\|\bar{\mathbf{H}}^{N}(\bar{\Phi}^{N,2})\|_{\mathfrak{t}_{\mathrm{st}};\mathbb{T}_{N}} \ &\lesssim \ N^{-\beta_{\mathrm{univ},2}} + N^{-\beta_{\mathrm{univ},2}}\|\bar{\mathbf{Z}}^{N}\|_{\mathfrak{t}_{\mathrm{st}};\mathbb{T}_{N}}^{2}. \label{eq:KPZNL}
\end{align}\normalsize\normalsize
We recall $\bar{\Phi}^{N,2}$ from the defining equation \emph{\eqref{eq:QBarEquation}} for $\bar{\mathbf{Z}}^{N}$. It contains the relevant pseudo-gradient data for $\bar{\mathbf{Z}}^{N}$-dynamics.
\end{prop}
We invite the reader to take Proposition \ref{prop:KPZNL} for granted and proceed to the next section to see how we use it to get Theorem \ref{theorem:KPZ}, at least in a first reading. The proof of Proposition \ref{prop:KPZNL} is quite technical. It consists of the following three ingredients, which we overviewed in the would-be-proof of Pseudo-Proposition \ref{pprop:S1}.
\begin{itemize}[leftmargin=*]
\item We replace the pseudo-gradient-spatial average in $\bar{\Phi}^{N,2}$ with a cutoff via Proposition \ref{prop:S1B} in view of Pseudo-Proposition \ref{pprop:S2}.
\item We then replace the spatial-average of pseudo-gradients, now with a cutoff given the previous bullet point, with its dynamic average on the time-scale $\mathfrak{t}_{\mathrm{av},\mathfrak{m}_{+}} \in \R_{>0}$ from the multiscale estimate Corollary \ref{corollary:D1B2A}. We do this by replacement with dynamic averages on progressively larger time-scales. Here, time-regularity for the heat kernel $\bar{\mathbf{H}}^{N}$ and for $\bar{\mathbf{Z}}^{N}$, the latter established in Proposition \ref{prop:TRGTProp}, will be crucial. The multiscale estimate in Corollary \ref{corollary:D1B1A} will be crucial too. We do something similar for the term in $\bar{\Phi}^{N,2}$ of order $N^{\beta_{X}}$ but via Corollary \ref{corollary:D1B1B} instead of Corollary \ref{corollary:D1B1A}. This step resembles Pseudo-Proposition \ref{pprop:S3}.
\item The third step is to estimate the pseudo-gradient content of $\bar{\Phi}^{N,2}$ with their dynamic averages from the previous bullet point. We employ another multiscale strategy via replacements by progressively sharper cutoff on progressively bigger time-scale as outlined in the would-be-proof of Pseudo-Proposition \ref{pprop:S4}. Corollary \ref{corollary:D1B2A} and Corollary \ref{corollary:D1B2B} are important for this step.
\end{itemize}
We now introduce some useful notation which is sometimes referred to as a squared version of the ``Japanese bracket" in PDE.
\begin{definition}
For $\mathfrak{t}_{\mathrm{st}}\in\R_{\geq0}$, define $\langle\rangle_{\mathfrak{t}_{\mathrm{st}};\mathbb{T}_{N}} \overset{\bullet}= 1 + \|\|_{\mathfrak{t}_{\mathrm{st}};\mathbb{T}_{N}}^{2}$ acting on $\mathfrak{f}:\R_{\geq0}\times\mathbb{T}_{N} \to \R$. Via the Cauchy-Schwarz inequality, we have the following estimate, which we will frequently use in this section, that compares this bracket to the $\|\|_{\mathfrak{t}_{\mathrm{st}};\mathbb{T}_{N}}$-norm:
\small\begin{align}
\|\mathfrak{f}\|_{\mathfrak{t}_{\mathrm{st}};\mathbb{T}_{N}} \ &\leq \ \langle\mathfrak{f}\rangle_{\mathfrak{t}_{\mathrm{st}};\mathbb{T}_{N}}. \label{eq:KPZNLCS}
\end{align}\normalsize\normalsize
\end{definition}
A significant portion of our upcoming analysis for order- $N^{1/2}$ terms in $\bar{\Phi}^{N,2}$ holds equally well for order-$N^{\beta_{X}}$ terms in $\bar{\Phi}^{N,2}$ with a possible exception of elementary but ultimately negligible adjustments in power-counting and application of Corollary \ref{corollary:D1B2B}/Corollary \ref{corollary:D1B1B} instead of Corollary \ref{corollary:D1B2A}/Corollary \ref{corollary:D1B1A}. So, we make the following ``formal replacement" occasionally:
\small\begin{align}
N^{\frac12} \ \to \ N^{\beta_{X}} \quad \mathrm{and} \quad \mathscr{C}_{N^{\beta_{X}}}^{\mathbf{X},-}(\mathfrak{g}) \ \to \ \wt{\sum}_{\mathfrak{l}=1,\ldots,N^{\beta_{X}}}\wt{\mathfrak{g}}^{\mathfrak{l}}. \label{eq:KPZNLReplace}
\end{align}\normalsize\normalsize
%
\subsection{Proof of Proposition \ref{prop:KPZNL}}
First we provide the key inputs each of which correspond to a bullet point in the list following Proposition \ref{prop:KPZNL}. We then combine them to prove Proposition \ref{prop:KPZNL}. The rest of the section is dedicated to proofs of each input. 

We first introduce some notation for the aforementioned key inputs which capture the errors in the first two bullet points in the outline above. Again, these errors come from introducing cutoff into a spatial-average and replacements-by-time-average.
\begin{itemize}[leftmargin=*]
\item Recall the spatial-average of pseudo-gradients $\mathscr{A}_{N^{\beta_{X}}}^{\mathbf{X},-}(\mathfrak{g})$ in Proposition \ref{prop:Duhamel} and its cutoff $\mathscr{C}_{N^{\beta_{X}}}^{\mathbf{X},-}(\mathfrak{g})$ in Definition \ref{definition:S1B}. We set
\small\begin{align}
\Psi^{N,1} \ &\overset{\bullet}= \ N^{\frac12}\mathscr{A}_{N^{\beta_{X}}}^{\mathbf{X},-}(\mathfrak{g}) - N^{\frac12}\mathscr{C}_{N^{\beta_{X}}}^{\mathbf{X},-}(\mathfrak{g}).
\end{align}\normalsize\normalsize
\item Taking $\mathfrak{t}_{\mathrm{av},\mathfrak{l}_{+}} \in \R_{>0}$ as the ``maximal time" defined in Corollary \ref{corollary:D1B1A}, we define the difference-via-time-average-replacement below in which we recall the space-time averages in Definition \ref{definition:D1B1A} and the cutoff-spatial-average in Definition \ref{definition:S1B}:
\small\begin{align}
\Psi^{N,2} \ \overset{\bullet}= \ N^{\frac12}\mathscr{A}_{\mathfrak{t}_{\mathrm{av},\mathfrak{l}_{+}}}^{\mathbf{T},+}\mathscr{C}_{N^{\beta_{X}}}^{\mathbf{X},-}(\mathfrak{g}) - N^{\frac12}\mathscr{C}_{N^{\beta_{X}}}^{\mathbf{X},-}(\mathfrak{g}).
\end{align}\normalsize\normalsize
Taking $\mathfrak{t}_{\mathrm{av},\mathfrak{l}_{+}^{\sim}}^{\sim} \in \R_{>0}$ as the ``maximal time" defined in Corollary \ref{corollary:D1B1B}, we define the difference-via-time-average-replacement below in which we recall the time-average in Definition \ref{definition:D1B1B} of the functional-with-pseudo-gradient-factor:
\small\begin{align}
\wt{\Psi}^{N,2} \ \overset{\bullet}= \ N^{\beta_{X}}\wt{\sum}_{\mathfrak{l}=1,\ldots,N^{\beta_{X}}}\left(\mathscr{A}_{\mathfrak{t}_{\mathrm{av},\mathfrak{l}_{+}^{\sim}}^{\sim}}^{\mathbf{T},+}(\wt{\mathfrak{g}}^{\mathfrak{l}}) - \wt{\mathfrak{g}}^{\mathfrak{l}}\right).
\end{align}\normalsize\normalsize
\end{itemize}
The first technical ingredient corresponds to the first bullet point from the list following the statement of Proposition \ref{prop:KPZNL}. The proof of Lemma \ref{lemma:KPZNLS1B} below is a direct consequence of the Markov inequality and Proposition \ref{prop:S1B}, in short.
\begin{lemma}\label{lemma:KPZNLS1B}
 Consider a random time $\mathfrak{t}_{\mathrm{st}} \in \R_{\geq0}$ with $\mathfrak{t}_{\mathrm{st}}\leq1$ with probability 1. There exist universal constants $\beta_{\mathrm{univ},1}\in\R_{>0}$ and $\beta_{\mathrm{univ},2}\in\R_{>0}$ such that outside an event of probability at most $N^{-\beta_{\mathrm{univ},1}}$ times uniformly bounded factors, we have
\small\begin{align}
\|\bar{\mathbf{H}}_{T,x}^{N}(\Psi^{N,1}\bar{\mathbf{Z}}^{N})\|_{\mathfrak{t}_{\mathrm{st}};\mathbb{T}_{N}} \ &\lesssim \ N^{-\beta_{\mathrm{univ},2}}\langle\bar{\mathbf{Z}}^{N}\rangle_{\mathfrak{t}_{\mathrm{st}};\mathbb{T}_{N}}. \label{eq:KPZNLS1B}
\end{align}\normalsize\normalsize
\end{lemma}
The second technical ingredient we require is the replacement of $\mathscr{C}^{\mathbf{X},-}$ in $\Psi^{N,1}$ within the previous Lemma \ref{lemma:KPZNLS1B} with its time-average on mesoscopic time-scale that corresponds to the second bullet point in the list following the statement of Proposition \ref{prop:KPZNL}. Roughly speaking we may replace local functionals with their time-averages inside the heat operator $\bar{\mathbf{H}}^{N}$ up to error terms controlled by the time-regularity of the heat kernel and of $\bar{\mathbf{Z}}^{N}$. We apply Proposition \ref{prop:TRGTProp} and time-regularity of the heat kernel in Lemma \ref{lemma:HKE} to control these errors. We give details behind this heuristic after our proof of Proposition \ref{prop:KPZNL}.
\begin{lemma}\label{lemma:KPZNLD1B1A}
 Consider a random time $\mathfrak{t}_{\mathrm{st}} \in \R_{\geq0}$ with $\mathfrak{t}_{\mathrm{st}}\leq1$ with probability 1. There exist universal constants $\beta_{\mathrm{univ},1}\in\R_{>0}$ and $\beta_{\mathrm{univ},2}\in\R_{>0}$ such that outside an event of probability at most $N^{-\beta_{\mathrm{univ},1}}$ times uniformly bounded factors, we have
\small\begin{align}
\|\bar{\mathbf{H}}_{T,x}^{N}(\Psi^{N,2}\bar{\mathbf{Z}}^{N})\|_{\mathfrak{t}_{\mathrm{st}};\mathbb{T}_{N}} \ &\lesssim \ N^{-\beta_{\mathrm{univ},2}}\langle\bar{\mathbf{Z}}^{N}\rangle_{\mathfrak{t}_{\mathrm{st}};\mathbb{T}_{N}}. \label{eq:KPZNLD1B1A}
\end{align}\normalsize\normalsize
\end{lemma}
\begin{lemma}\label{lemma:KPZNLD1B1B}
 Consider a random time $\mathfrak{t}_{\mathrm{st}} \in \R_{\geq0}$ with $\mathfrak{t}_{\mathrm{st}}\leq1$ with probability 1. There exist universal constants $\beta_{\mathrm{univ},1}\in\R_{>0}$ and $\beta_{\mathrm{univ},2}\in\R_{>0}$ such that outside an event of probability at most $N^{-\beta_{\mathrm{univ},1}}$ times uniformly bounded factors, we have
\small\begin{align}
\|\bar{\mathbf{H}}_{T,x}^{N}(\wt{\Psi}^{N,2}\bar{\mathbf{Z}}^{N})\|_{\mathfrak{t}_{\mathrm{st}};\mathbb{T}_{N}} \ &\lesssim \ N^{-\beta_{\mathrm{univ},2}}\langle\bar{\mathbf{Z}}^{N}\rangle_{\mathfrak{t}_{\mathrm{st}};\mathbb{T}_{N}}. \label{eq:KPZNLD1B1B}
\end{align}\normalsize\normalsize
\end{lemma}
The third ingredient is the following pair of estimates for time-averages on maximal time-scales $\mathfrak{t}_{\mathrm{av},\mathfrak{m}_{+}},\mathfrak{t}_{\mathrm{av},\mathfrak{m}_{+}^{\sim}}^{\sim}$ from Corollary \ref{corollary:D1B2A} and Corollary \ref{corollary:D1B2B}. This corresponds to the third bullet point from the list following Proposition \ref{prop:KPZNL}. As alluded to therein, the following is a consequence of a scheme that avoids controlling time-averages on the aforementioned maximal time-scales $\mathfrak{t}_{\mathrm{av},\mathfrak{m}_{+}},\mathfrak{t}_{\mathrm{av},\mathfrak{m}_{+}^{\sim}}^{\sim}$ directly and instead controls time-averages on smaller time-scales, glues these shorter-time estimates into longer time-estimates, and iterates until we arrive at the maximal time-scales from Corollary \ref{corollary:D1B2A} and Corollary \ref{corollary:D1B2B}.
\begin{lemma}\label{lemma:Step3}
 Consider a random time $\mathfrak{t}_{\mathrm{st}} \in \R_{\geq0}$ with $\mathfrak{t}_{\mathrm{st}}\leq1$ with probability 1. There exist universal constants $\beta_{\mathrm{univ},1}\in\R_{>0}$ and $\beta_{\mathrm{univ},2}\in\R_{>0}$ such that outside an event of probability at most $N^{-\beta_{\mathrm{univ},1}}$ times uniformly bounded factors, we have
\small\begin{align}
\|\bar{\mathbf{H}}_{T,x}^{N}(N^{\frac12}\mathscr{A}^{\mathbf{T},+}_{\mathfrak{t}_{\mathrm{av},\mathfrak{m}_{+}}}\mathscr{C}_{N^{\beta_{X}}}^{\mathbf{X},-}(\mathfrak{g})\cdot\bar{\mathbf{Z}}^{N})\|_{\mathfrak{t}_{\mathrm{st}};\mathbb{T}_{N}} \ &\lesssim \ N^{-\beta_{\mathrm{univ},2}}\langle\bar{\mathbf{Z}}^{N}\rangle_{\mathfrak{t}_{\mathrm{st}};\mathbb{T}_{N}}. \label{eq:Step3I}
\end{align}\normalsize\normalsize
The time-scale $\mathfrak{t}_{\mathrm{av},\mathfrak{m}_{+}} \in \R_{>0}$ is the ``maximal" time-scale in \emph{Corollary \ref{corollary:D1B2A}}, which we emphasize is equal to the ``maximal" time-scale $\mathfrak{t}_{\mathrm{av},\mathfrak{l}_{+}} \in \R_{>0}$ in \emph{Corollary \ref{corollary:D1B1A}}. Moreover, if instead we take the time-scales $\mathfrak{t}_{\mathrm{av},\mathfrak{m}_{+}^{\sim}}^{\sim},\mathfrak{t}_{\mathrm{av},\mathfrak{l}_{+}^{\sim}}^{\sim}$ in \emph{Corollary \ref{corollary:D1B2B}} and \emph{Corollary \ref{corollary:D1B1B}}, which are also equal to each other, we have the following with at least the same probability and same $\beta_{\mathrm{univ},2}>0$:
\small\begin{align}
\|\bar{\mathbf{H}}_{T,x}^{N}\left(N^{\beta_{X}}\left(\wt{\sum}_{\mathfrak{l}=1,\ldots,N^{\beta_{X}}}\mathscr{A}^{\mathbf{T},+}_{\mathfrak{t}_{\mathrm{av},\mathfrak{m}_{+}^{\sim}}^{\sim}}(\wt{\mathfrak{g}}^{\mathfrak{l}})\right) \cdot \bar{\mathbf{Z}}^{N}\right)\|_{\mathfrak{t}_{\mathrm{st}};\mathbb{T}_{N}} \ &\lesssim \ N^{-\beta_{\mathrm{univ},2}}\langle\bar{\mathbf{Z}}^{N}\rangle_{\mathfrak{t}_{\mathrm{st}};\mathbb{T}_{N}}. \label{eq:Step3II}
\end{align}\normalsize\normalsize
\end{lemma}
\begin{proof}[Proof of \emph{Proposition \ref{prop:KPZNL}}]
Throughout we will make statements which hold with high probability. Here high probability means that the complement event has probability at most $N^{-\beta_{\mathrm{univ},1}}$ times uniformly bounded constants, with $\beta_{\mathrm{univ},1} \in \R_{>0}$ universal. Because we make a uniformly bounded number of such statements, the conclusion holds with high-probability too. 

By definition of $\bar{\Phi}^{N,2}$ in the stochastic equation for $\bar{\mathbf{Z}}^{N}$ in Definition \ref{definition:ChiTorus} and the triangle inequality for $\|\|_{\mathfrak{t}_{\mathrm{st}};\mathbb{T}_{N}}$, we first have
\small\begin{align}
\|\bar{\mathbf{H}}^{N}(\bar{\Phi}^{N,2})\|_{\mathfrak{t}_{\mathrm{st}};\mathbb{T}_{N}} \ &\leq \ \|\bar{\mathbf{H}}_{T,x}^{N}(N^{\frac12}\mathscr{A}_{N^{\beta_{X}}}^{\mathbf{X},-}(\mathfrak{g}) \cdot \bar{\mathbf{Z}}^{N})\|_{\mathfrak{t}_{\mathrm{st}};\mathbb{T}_{N}} \ + \ \|\bar{\mathbf{H}}_{T,x}^{N}\left(N^{\beta_{X}}\wt{\sum}_{\mathfrak{l}=1,\ldots,N^{\beta_{X}}}\wt{\mathfrak{g}}^{\mathfrak{l}} \cdot \bar{\mathbf{Z}}^{N}\right)\|_{\mathfrak{t}_{\mathrm{st}};\mathbb{T}_{N}}  \ + \ \Xi, \label{eq:KPZNL1}
\end{align}\normalsize\normalsize
where the last quantity corresponds to gradient terms inside $\bar{\Phi}^{N,2}$; precisely, the $\Xi$ term is 
\small\begin{align}
\Xi \ &\overset{\bullet}= \ \wt{\sum}_{\mathfrak{l}=1,\ldots,N^{\beta_{X}}} N^{-\frac12}\|\bar{\mathbf{H}}_{T,x}^{N}(\bar{\grad}_{-7\mathfrak{l}\mathfrak{m}}^{!}(\mathfrak{b}^{\mathfrak{l}}\bar{\mathbf{Z}}^{N}))\|_{\mathfrak{t}_{\mathrm{st}};\mathbb{T}_{N}}.
\end{align}\normalsize\normalsize
Recall $\mathfrak{m} \in \Z_{>0}$ is the maximal jump-length in the particle random walks in the model, and $\bar{\grad}^{!} \overset{\bullet}= N\bar{\grad}$ is the continuum-scaling of the discrete gradient acting on functions $\mathbb{T}_{N}\to\R$ on the torus. We denote the first and second terms on the RHS of \eqref{eq:KPZNL1} by $\Xi_{1}$ and $\Xi_{2}$, respectively. We adopt the notation in Lemmas \ref{lemma:KPZNLS1B}, \ref{lemma:KPZNLD1B1A}, and \ref{lemma:Step3} with $\mathfrak{t}_{\mathrm{av},\mathfrak{m}_{+}}$ from Corollary \ref{corollary:D1B2A}/\ref{corollary:D1B1A} to first get
\small\begin{align}
\Xi_{1} \ &\leq \ \|\bar{\mathbf{H}}_{T,x}^{N}(\Psi^{N,1}\bar{\mathbf{Z}}^{N})\|_{\mathfrak{t}_{\mathrm{st}};\mathbb{T}_{N}} \ + \ \|\bar{\mathbf{H}}_{T,x}^{N}(\Psi^{N,2}\bar{\mathbf{Z}}^{N})\|_{\mathfrak{t}_{\mathrm{st}};\mathbb{T}_{N}} \ + \ \|\bar{\mathbf{H}}_{T,x}^{N}(N^{\frac12}\mathscr{A}^{\mathbf{T},+}_{\mathfrak{t}_{\mathrm{av},\mathfrak{m}_{+}}}\mathscr{C}_{N^{\beta_{X}}}^{\mathbf{X},-}(\mathfrak{g})\cdot \bar{\mathbf{Z}}^{N})\|_{\mathfrak{t}_{\mathrm{st}};\mathbb{T}_{N}}
\end{align}\normalsize\normalsize
Combining the last bound with Lemmas \ref{lemma:KPZNLS1B}, \ref{lemma:KPZNLD1B1A}, and \ref{lemma:Step3}, with the required high probability we obtain $\Xi_{1} \lesssim N^{-\beta_{\mathrm{univ},2}}\langle\bar{\mathbf{Z}}^{N}\rangle_{\mathfrak{t}_{\mathrm{st}};\mathbb{T}_{N}}$. We now study $\Xi_{2}$. Adopting the same notation and also that of Lemma \ref{lemma:KPZNLD1B1B} but with $\mathfrak{t}_{\mathrm{av},\mathfrak{m}_{+}^{\sim}}^{\sim}$ from Corollary \ref{corollary:D1B2B}/\ref{corollary:D1B1B}, we get
\small\begin{align}
\Xi_{2} \ &\leq \ \|\bar{\mathbf{H}}_{T,x}^{N}(\wt{\Psi}^{N,2}\bar{\mathbf{Z}}^{N})\|_{\mathfrak{t}_{\mathrm{st}};\mathbb{T}_{N}} \ + \ \|\bar{\mathbf{H}}_{T,x}^{N}\left(N^{\beta_{X}}\left(\wt{\sum}_{\mathfrak{l}=1,\ldots,N^{\beta_{X}}}\mathscr{A}^{\mathbf{T},+}_{\mathfrak{t}_{\mathrm{av},\mathfrak{m}_{+}^{\sim}}^{\sim}}(\wt{\mathfrak{g}}^{\mathfrak{l}})\right) \cdot \bar{\mathbf{Z}}^{N}\right)\|_{\mathfrak{t}_{\mathrm{st}};\mathbb{T}_{N}}.
\end{align}\normalsize\normalsize
We now apply Lemma \ref{lemma:KPZNLD1B1B} and Lemma \ref{lemma:Step3} to deduce $\Xi_{2} \lesssim N^{-\beta_{\mathrm{univ},2}}\langle\bar{\mathbf{Z}}^{N}\rangle_{\mathfrak{t}_{\mathrm{st}};\mathbb{T}_{N}}$ also with the required high-probability. It remains to estimate $\Xi$. To this end, we estimate each term inside the averaged summation defining $\Xi$ given right after the first triangle inequality estimate \eqref{eq:KPZNL1}. Estimating all of these terms amounts to analytic estimates in Lemma \ref{lemma:HKE}. In particular, we do not use randomness and we only require the uniform bound $|\mathfrak{b}^{\mathfrak{l}}|\lesssim1$. To be precise, we use the regularity estimate \eqref{eq:HKEKPZNL} to get
\small\begin{align}
\Xi \ \overset{\bullet}= \ \wt{\sum}_{\mathfrak{l}=1}^{N^{\beta_{X}}} N^{-\frac12}\|\bar{\mathbf{H}}_{T,x}^{N}(\bar{\grad}_{-7\mathfrak{l}\mathfrak{m}}^{!}(\mathfrak{b}^{\mathfrak{l}}\bar{\mathbf{Z}}^{N}))\|_{\mathfrak{t}_{\mathrm{st}};\mathbb{T}_{N}} \ &\lesssim_{\mathfrak{m}} \ \wt{\sum}_{\mathfrak{l}=1}^{N^{\beta_{X}}} N^{-\frac12} |\mathfrak{l}|\|\bar{\mathbf{Z}}^{N}\|_{\mathfrak{t}_{\mathrm{st}};\mathbb{T}_{N}} \nonumber \\
&\lesssim \ N^{-\frac12+\beta_{X}}\|\bar{\mathbf{Z}}^{N}\|_{\mathfrak{t}_{\mathrm{st}};\mathbb{T}_{N}} \ \leq \ N^{-\frac16+\e_{X,1}}\langle\bar{\mathbf{Z}}^{N}\rangle_{\mathfrak{t}_{\mathrm{st}};\mathbb{T}_{N}}. \label{eq:KPZNL4}
\end{align}\normalsize\normalsize
Indeed, the far RHS of \eqref{eq:KPZNL4} follows by the middle upon recalling $\beta_{X} = \frac13+\e_{X,1}$ and applying the bracket inequality \eqref{eq:KPZNLCS}. We now use \eqref{eq:KPZNL4} with \eqref{eq:KPZNL1} and the previous $\Xi_{1}+\Xi_{2} \lesssim N^{-\beta_{\mathrm{univ},2}}\langle\bar{\mathbf{Z}}^{N}\rangle_{\mathfrak{t}_{\mathrm{st}};\mathbb{T}_{N}}$ for the first two terms on the RHS of \eqref{eq:KPZNL1}.
\end{proof}
We now prove ingredients used in the proof of Proposition \ref{prop:KPZNL}. For all of Lemmas \ref{lemma:KPZNLS1B}, \ref{lemma:KPZNLD1B1A}, \ref{lemma:KPZNLD1B1B}, and \ref{lemma:Step3}, we give ingredients for their proofs, combine them, and then defer their technical proofs to the end of the section to avoid obscuring key ideas.
\subsection{Proof of Lemma \ref{lemma:KPZNLS1B}}
The proof of replacement-by-cutoff we give here provides a good template for later arguments in this section. As $\mathfrak{t}_{\mathrm{st}}\leq1$ with probability 1, for any $T \leq \mathfrak{t}_{\mathrm{st}} \leq 1$ we have $\|\|_{\mathfrak{t}_{\mathrm{st}};\mathbb{T}_{N}}\leq\|\|_{1;\mathbb{T}_{N}}$ as norms, and the deterministic bound
\small\begin{align}
\|\bar{\mathbf{H}}_{T,x}^{N}(\Psi^{N,1}\bar{\mathbf{Z}}^{N})\|_{\mathfrak{t}_{\mathrm{st}};\mathbb{T}_{N}} \ &\leq \ \|\bar{\mathbf{Z}}^{N}\|_{\mathfrak{t}_{\mathrm{st}};\mathbb{T}_{N}} \|\bar{\mathbf{H}}_{T,x}^{N}(N^{\frac12}|\mathscr{A}_{N^{\beta_{X}}}^{\mathbf{X},-}(\mathfrak{g}) - \mathscr{C}_{N^{\beta_{X}}}^{\mathbf{X},-}(\mathfrak{g})|)\|_{1;\mathbb{T}_{N}} \ \leq \ \langle\bar{\mathbf{Z}}^{N}\rangle_{\mathfrak{t}_{\mathrm{st}};\mathbb{T}_{N}} \|\bar{\mathbf{H}}_{T,x}^{N}(N^{\frac12}|\mathscr{A}_{N^{\beta_{X}}}^{\mathbf{X},-}(\mathfrak{g}) - \mathscr{C}_{N^{\beta_{X}}}^{\mathbf{X},-}(\mathfrak{g})|)\|_{1;\mathbb{T}_{N}}. \nonumber
\end{align}\normalsize\normalsize
The second bound follows from applying \eqref{eq:KPZNLCS} to the $\bar{\mathbf{Z}}^{N}$-norm in the first bound. By Proposition \ref{prop:S1B} and the Markov inequality, the $\|\|_{1;\mathbb{T}_{N}}$-norm in this last estimate is at most $N^{-\beta_{\mathrm{univ},2}}$ times uniformly bounded factors with the required high probability for $\beta_{\mathrm{univ},2} \in \R_{>0}$ universal. Precisely, by the Markov inequality and expectation bound in Proposition \ref{prop:S1B}, we deduce
\small\begin{align}
\mathbf{P}\left(\|\bar{\mathbf{H}}_{T,x}^{N}(N^{\frac12}|\mathscr{A}_{N^{\beta_{X}}}^{\mathbf{X},-}(\mathfrak{g}) - \mathscr{C}_{N^{\beta_{X}}}^{\mathbf{X},-}(\mathfrak{g})|)\|_{1;\mathbb{T}_{N}} \geq N^{-\beta_{\mathrm{univ},2}}\right) \ &\leq \ N^{\beta_{\mathrm{univ},2}}\E\|\bar{\mathbf{H}}_{T,x}^{N}(N^{\frac12}|\mathscr{A}_{N^{\beta_{X}}}^{\mathbf{X},-}(\mathfrak{g}) - \mathscr{C}_{N^{\beta_{X}}}^{\mathbf{X},-}(\mathfrak{g})|)\|_{1;\mathbb{T}_{N}} \nonumber \\
&\lesssim \ N^{-\beta_{\mathrm{univ}}+\beta_{\mathrm{univ},2}}. \nonumber
\end{align}\normalsize\normalsize
We emphasize that $\beta_{\mathrm{univ}}\in\R_{>0}$ is universal and $\beta_{\mathrm{univ},2}\in\R_{>0}$ is a universal constant of our choosing. Choosing $\beta_{\mathrm{univ},2} = \frac12\beta_{\mathrm{univ}}$ guarantees the event in the probability on the LHS fails with the required high-probability. Outside such an event, the proposed estimate in Lemma \ref{lemma:KPZNLS1B} holds courtesy of the first display in this proof. This completes the proof of Lemma \ref{lemma:KPZNLS1B}. \qed
\subsection{Proof of Lemma \ref{lemma:KPZNLD1B1A}}
The first ingredient towards the proof for Lemma \ref{lemma:KPZNLD1B1A} is the following inductive estimate that we will iterate to prove Lemma \ref{lemma:KPZNLD1B1A}. This identifies a cost of replacing a time-average in the heat operator by the same time-average but on a slightly larger time-scale. This cost is ultimately controlled by time-regularity of the heat kernel $\bar{\mathbf{H}}^{N}$ and of $\bar{\mathbf{Z}}^{N}$, both on the larger of the two time-scales that we time-average with respect to, times the time-average on the shorter time-scale.
\begin{lemma}\label{lemma:KPZNLStep2}
 Consider any possibly random time $\mathfrak{t}_{\mathrm{st}}\in\R_{\geq0}$ satisfying $\mathfrak{t}_{\mathrm{st}}\leq1$ with probability 1. Consider any times $\mathfrak{t}_{1},\mathfrak{t}_{2}\in\R_{\geq0}$ satisfying $N^{-100} \lesssim \mathfrak{t}_{1}\leq\mathfrak{t}_{2}\lesssim N^{-1}$ and $\mathfrak{t}_{2}\mathfrak{t}_{1}^{-1}\in\Z_{\geq0}$. There exist universal constants $\beta_{\mathrm{univ},1},\beta_{\mathrm{univ},2}>0$ such that outside an event of probability at most $N^{-\beta_{\mathrm{univ},1}}$ times uniformly bounded factors, we have the following estimate for any $\e>0$:
\small\begin{align}
\|\bar{\mathbf{H}}_{T,x}^{N}(\Psi^{N,2,\mathfrak{t}_{1},\mathfrak{t}_{2}}\bar{\mathbf{Z}}^{N})\|_{\mathfrak{t}_{\mathrm{st}};\mathbb{T}_{N}} \ &\lesssim_{\e} \ N^{-\beta_{\mathrm{univ},2}}\langle\bar{\mathbf{Z}}^{N}\rangle_{\mathfrak{t}_{\mathrm{st}};\mathbb{T}_{N}} \ + \ \langle\bar{\mathbf{Z}}^{N}\rangle_{\mathfrak{t}_{\mathrm{st}};\mathbb{T}_{N}} N^{\e}\mathfrak{t}_{2}^{\frac14}\|\bar{\mathbf{H}}_{T,x}^{N}(N^{\frac12}|\mathscr{A}_{\mathfrak{t}_{1}}^{\mathbf{T},+}\mathscr{C}_{N^{\beta_{X}}}^{\mathbf{X},-}(\mathfrak{g})|)\|_{1;\mathbb{T}_{N}}. \label{eq:KPZNLStep2I}
\end{align}\normalsize\normalsize
Above we recalled the space-time average of pseudo-gradients in \emph{Definition \ref{definition:D1B1A}}, and we also introduced the following quantity:
\small\begin{align}
\Psi^{N,2,\mathfrak{t}_{1},\mathfrak{t}_{2}} \ &\overset{\bullet}= \ N^{\frac12}\mathscr{A}_{\mathfrak{t}_{2}}^{\mathbf{T},+}\mathscr{C}_{N^{\beta_{X}}}^{\mathbf{X},-}(\mathfrak{g}) - N^{\frac12}\mathscr{A}_{\mathfrak{t}_{1}}^{\mathbf{T},+}\mathscr{C}_{N^{\beta_{X}}}^{\mathbf{X},-}(\mathfrak{g}).
\end{align}\normalsize\normalsize
\end{lemma}
We will require another version of Lemma \ref{lemma:KPZNLStep2} but when the initial time-scale is $\mathfrak{t}_{1} = 0$, so the following result estimates the error in introducing a time-average in the first place that will kickstart the inductive procedure referenced at the beginning of this subsection that is based on Lemma \ref{lemma:KPZNLStep2}. The proof of this upcoming estimate first replaces a time-0 time-average by a time-average on a very sub-microscopic time-scale. Because in expectation we expect nothing to happen on a very sub-microscopic time-scale, we directly get a little ``regularity" of the local functional we are time-averaging for this initial replacement-by-time-average. In particular, we do not use regularity of the heat kernel or $\bar{\mathbf{Z}}^{N}$ for the initial step. After, we use Lemma \ref{lemma:KPZNLStep2}.
\begin{lemma}\label{lemma:KPZNLStep2b}
 Assume the setting and notation in \emph{Lemma \ref{lemma:KPZNLStep2}} but for $\mathfrak{t}_{1} = 0$ and $\mathfrak{t}_{2} = \mathfrak{t}_{\mathrm{av},0}$ in \emph{Corollary \ref{corollary:D1B1A}}. There are universal constants $\beta_{\mathrm{univ},1},\beta_{\mathrm{univ},2}\in\R_{>0}$ such that outside an event of probability at most $N^{-\beta_{\mathrm{univ},1}}$ times uniformly bounded factors,
\small\begin{align}
\|\bar{\mathbf{H}}_{T,x}^{N}(\Psi^{N,2,0,\mathfrak{t}_{\mathrm{av},0}}\bar{\mathbf{Z}}^{N})\|_{\mathfrak{t}_{\mathrm{st}};\mathbb{T}_{N}} \ &\lesssim \  N^{-\beta_{\mathrm{univ},2}}\langle\bar{\mathbf{Z}}^{N}\rangle_{\mathfrak{t}_{\mathrm{st}};\mathbb{T}_{N}}. \label{eq:KPZNLStep2b}
\end{align}\normalsize\normalsize
\end{lemma}
\begin{proof}[Proof of \emph{Lemma \ref{lemma:KPZNLD1B1A}}]
We make the same note about ``high-probability" as in the beginning of the proof of Proposition \ref{prop:KPZNL}.

Recall the times $\mathfrak{t}_{\mathrm{av},\ell}\in\R_{>0}$ in Corollary \ref{corollary:D1B1A}, including $\mathfrak{t}_{\mathrm{av},0} = \mathscr{O}(N^{-2+\e_{X,2}})$. We will iterate the estimate \eqref{eq:KPZNLStep2I} in Lemma \ref{lemma:KPZNLStep2} to replace the time-average on time-scale $\mathfrak{t}_{\mathrm{av},\ell}$ by the time-average on time-scale $\mathfrak{t}_{\mathrm{av},\ell+1}$ until we reach the maximal time-scale $\mathfrak{t}_{\mathrm{av},\mathfrak{l}_{+}}\in\R_{>0}$. We start this off by replacing the time-0 average, namely the identity operator, by a time-$\mathfrak{t}_{\mathrm{av},0}$ average. Estimating the errors in this first step of introducing a time-$\mathfrak{t}_{\mathrm{av},0}$-average and the subsequent ``multiscale" steps then completes the proof.

We now make the previous paragraph precise. Let us first observe the following telescoping sum that follows by definition. It replaces the time-0 average with a time-$\mathfrak{t}_{\mathrm{av},\mathfrak{l}_{+}}$ average via intermediate steps. Recall $\Psi^{N,2}$ defined before/for Lemma \ref{lemma:KPZNLD1B1A}:
\small\begin{align}
\bar{\mathbf{H}}_{T,x}^{N}(\Psi^{N,2}\bar{\mathbf{Z}}^{N}) \ = \ \bar{\mathbf{H}}_{T,x}^{N}(\Psi^{N,2,0,\mathfrak{t}_{\mathrm{av},\mathfrak{l}_{+}}}\bar{\mathbf{Z}}^{N}) \ &= \ \bar{\mathbf{H}}_{T,x}^{N}(\Psi^{N,2,0,\mathfrak{t}_{\mathrm{av},0}}\bar{\mathbf{Z}}^{N}) \ + \ {\sum}_{\ell=0}^{\mathfrak{l}_{+}-1} \bar{\mathbf{H}}_{T,x}^{N}(\Psi^{N,2,\mathfrak{t}_{\mathrm{av},\ell},\mathfrak{t}_{\mathrm{av},\ell+1}}\bar{\mathbf{Z}}^{N}). \label{eq:KPZNLD1B1A1}
\end{align}\normalsize\normalsize
We first address the first term on the RHS of \eqref{eq:KPZNLD1B1A1}. By Lemma \ref{lemma:KPZNLStep2b} with high-probability we get the estimate:
\small\begin{align}
\|\bar{\mathbf{H}}_{T,x}^{N}(\Psi^{N,2,0,\mathfrak{t}_{\mathrm{av},0}}\bar{\mathbf{Z}}^{N})\|_{\mathfrak{t}_{\mathrm{st}};\mathbb{T}_{N}}\  &\lesssim \ N^{-\beta_{\mathrm{univ},2}}\langle\bar{\mathbf{Z}}^{N}\rangle_{\mathfrak{t}_{\mathrm{st}};\mathbb{T}_{N}}. \label{eq:KPZNLD1B1A6}
\end{align}\normalsize\normalsize
For the second term on the RHS of \eqref{eq:KPZNLD1B1A1}, we use the triangle inequality for $\|\|_{\mathfrak{t}_{\mathrm{st}};\mathbb{T}_{N}}$ and the uniform boundedness for $\mathfrak{l}_{+}\in\Z_{\geq0}$ from Corollary \ref{corollary:D1B1A} to get the following which controls a sum by the number of terms in the sum times a supremum:
\small\begin{align}
\|{\sum}_{\ell=0}^{\mathfrak{l}_{+}-1} \bar{\mathbf{H}}_{T,x}^{N}(\Psi^{N,2,\mathfrak{t}_{\mathrm{av},\ell},\mathfrak{t}_{\mathrm{av},\ell+1}}\bar{\mathbf{Z}}^{N})\|_{\mathfrak{t}_{\mathrm{st}};\mathbb{T}_{N}} \ &\lesssim \ \sup_{\ell=0,\ldots,\mathfrak{l}_{+}-1}\|\bar{\mathbf{H}}_{T,x}^{N}(\Psi^{N,2,\mathfrak{t}_{\mathrm{av},\ell},\mathfrak{t}_{\mathrm{av},\ell+1}}\bar{\mathbf{Z}}^{N})\|_{\mathfrak{t}_{\mathrm{st}};\mathbb{T}_{N}}. \label{eq:KPZNLD1B1A2}
\end{align}\normalsize\normalsize
By \eqref{eq:KPZNLStep2I} in Lemma \ref{lemma:KPZNLStep2}, we estimate each term in the supremum on the RHS of \eqref{eq:KPZNLD1B1A2} with high probability \emph{uniformly} in $\ell$:
\small\begin{align}
\|\bar{\mathbf{H}}_{T,x}^{N}(\Psi^{N,2,\mathfrak{t}_{\mathrm{av},\ell},\mathfrak{t}_{\mathrm{av},\ell+1}}\bar{\mathbf{Z}}^{N})\|_{\mathfrak{t}_{\mathrm{st}};\mathbb{T}_{N}} \ &\lesssim \ N^{-\beta_{\mathrm{univ},2}}\langle\bar{\mathbf{Z}}^{N}\rangle_{\mathfrak{t}_{\mathrm{st}};\mathbb{T}_{N}} \ + \ \langle\bar{\mathbf{Z}}^{N}\rangle_{\mathfrak{t}_{\mathrm{st}};\mathbb{T}_{N}} N^{\e} \mathfrak{t}_{\mathrm{av},\ell+1}^{1/4}\|\bar{\mathbf{H}}_{T,x}^{N}(N^{\frac12}|\mathscr{A}_{\mathfrak{t}_{\mathrm{av},\ell}}^{\mathbf{T},+}\mathscr{C}_{N^{\beta_{X}}}^{\mathbf{X},-}(\mathfrak{g})|)\|_{1;\mathbb{T}_{N}}. \label{eq:KPZNLD1B1A3}
\end{align}\normalsize\normalsize
We may now apply the expectation estimate in Corollary \ref{corollary:D1B1A} along with the Markov inequality to deduce the following bound with high-probability for the second term on the RHS of \eqref{eq:KPZNLD1B1A3} in which $\beta_{\mathrm{univ},2}\in\R_{>0}$ is possibly updated from \eqref{eq:KPZNLD1B1A6}/\eqref{eq:KPZNLD1B1A3}:
\small\begin{align}
\sup_{\ell=0,\ldots,\mathfrak{l}_{+}-1}N^{\e} \mathfrak{t}_{\mathrm{av},\ell+1}^{1/4}\|\bar{\mathbf{H}}_{T,x}^{N}(N^{\frac12}|\mathscr{A}_{\mathfrak{t}_{\mathrm{av},\ell}}^{\mathbf{T},+}\mathscr{C}_{N^{\beta_{X}}}^{\mathbf{X},-}(\mathfrak{g})|)\|_{1;\mathbb{T}_{N}} \ &\lesssim \ N^{-\beta_{\mathrm{univ},2}}. \label{eq:KPZNLD1B1A4}
\end{align}\normalsize\normalsize
We now combine \eqref{eq:KPZNLD1B1A2}, \eqref{eq:KPZNLD1B1A3}, and \eqref{eq:KPZNLD1B1A4} to get the following high-probability bound:
\small\begin{align}
\|{\sum}_{\ell=0}^{\mathfrak{l}_{+}-1} \bar{\mathbf{H}}_{T,x}^{N}(\Psi^{N,2,\mathfrak{t}_{\mathrm{av},\ell},\mathfrak{t}_{\mathrm{av},\ell+1}}\bar{\mathbf{Z}}^{N})\|_{\mathfrak{t}_{\mathrm{st}};\mathbb{T}_{N}} \ &\lesssim \ N^{-\beta_{\mathrm{univ},2}}\langle\bar{\mathbf{Z}}^{N}\rangle_{\mathfrak{t}_{\mathrm{st}};\mathbb{T}_{N}}. \label{eq:KPZNLD1B1A5}
\end{align}\normalsize\normalsize
Combining \eqref{eq:KPZNLD1B1A1}, \eqref{eq:KPZNLD1B1A6}, and \eqref{eq:KPZNLD1B1A5} with the triangle inequality for $\|\|_{\mathfrak{t}_{\mathrm{st}};\mathbb{T}_{N}}$ completes the proof.
\end{proof}
\subsection{Proof of Lemma \ref{lemma:KPZNLD1B1B}}
The proof of Lemma \ref{lemma:KPZNLD1B1B} is near identical to the proof of Lemma \ref{lemma:KPZNLD1B1A}. We just make the replacement \eqref{eq:KPZNLReplace} and adjust power-counting in $N\in\Z_{>0}$ in a fashion that is ultimately negligible. This starts with an analog of Lemma \ref{lemma:KPZNLStep2}.
\begin{lemma}\label{lemma:KPZNLStep22}
Retain the notation of \emph{Lemma \ref{lemma:KPZNLStep2}}. There exist universal constants $\beta_{\mathrm{univ},1},\beta_{\mathrm{univ},2}\in\R_{>0}$ such that outside an event of probability at most $N^{-\beta_{\mathrm{univ},1}}$ times uniformly bounded factors, we have the following estimate for any $\e\in\R_{>0}$:
\small\begin{align}
\|\bar{\mathbf{H}}_{T,x}^{N}(\wt{\Psi}^{N,2,\mathfrak{t}_{1},\mathfrak{t}_{2}}\bar{\mathbf{Z}}^{N})\|_{\mathfrak{t}_{\mathrm{st}};\mathbb{T}_{N}} \ &\lesssim_{\e} \ N^{-\beta_{\mathrm{univ},2}}\langle\bar{\mathbf{Z}}^{N}\rangle_{\mathfrak{t}_{\mathrm{st}};\mathbb{T}_{N}} \ + \ \langle\bar{\mathbf{Z}}^{N}\rangle_{\mathfrak{t}_{\mathrm{st}};\mathbb{T}_{N}} N^{\e}\mathfrak{t}_{2}^{\frac14}\left(\wt{\sum}_{\mathfrak{l}=1}^{N^{\beta_{X}}} \|\bar{\mathbf{H}}_{T,x}^{N}(N^{\beta_{X}}|\mathscr{A}_{\mathfrak{t}_{1}}^{\mathbf{T},+}(\wt{\mathfrak{g}}^{\mathfrak{l}})|)\|_{1;\mathbb{T}_{N}}\right). \label{eq:KPZNLStep2II}
\end{align}\normalsize\normalsize
We recall the time-average introduced in \emph{Definition \ref{definition:D1B1B}} of the functionals $\mathfrak{g}^{\mathfrak{l}}$ with pseudo-gradient factor. Also,
\small\begin{align}
\wt{\Psi}^{N,2,\mathfrak{t}_{1},\mathfrak{t}_{2}} \ &\overset{\bullet}= \ N^{\beta_{X}} \wt{\sum}_{\mathfrak{l}=1,\ldots,N^{\beta_{X}}}\left(\mathscr{A}_{\mathfrak{t}_{2}}^{\mathbf{T},+}(\wt{\mathfrak{g}}^{\mathfrak{l}}) -\mathscr{A}_{\mathfrak{t}_{1}}^{\mathbf{T},+}(\wt{\mathfrak{g}}^{\mathfrak{l}})\right).
\end{align}\normalsize\normalsize
\end{lemma}
We emphasize that the only difference between this and Lemma \ref{lemma:KPZNLStep2} is the replacement \eqref{eq:KPZNLReplace}. In particular, as we suggested at the beginning of this subsection its proof is similar to that of Lemma \ref{lemma:KPZNLStep2}. The details of what we time-average in Lemma \ref{lemma:KPZNLStep2} and Lemma \ref{lemma:KPZNLStep22} are unimportant, just that we are replacing a time-average by another on a larger time-scale.

The second preliminary ingredient we will need towards the proof of Lemma \ref{lemma:KPZNLD1B1B} is another parallel to a step in the proof for Lemma \ref{lemma:KPZNLD1B1A}. This second ingredient is an analog of Lemma \ref{lemma:KPZNLStep2b}. It serves to introduce a preliminary time-average that kicks off the inductive procedure driven by the previous estimate in Lemma \ref{lemma:KPZNLStep22}. This is similar to the purpose of Lemma \ref{lemma:KPZNLStep2b}. 
\begin{lemma}\label{lemma:KPZNLStep22b}
 Assume the setting and notation in \emph{Lemma \ref{lemma:KPZNLStep2}} but for $\mathfrak{t}_{1} = 0$ and $\mathfrak{t}_{2} = \mathfrak{t}_{\mathrm{av},0}^{\sim}$ in \emph{Corollary \ref{corollary:D1B1B}}. There are universal constants $\beta_{\mathrm{univ},1},\beta_{\mathrm{univ},2}\in\R_{>0}$ such that outside an event of probability at most $N^{-\beta_{\mathrm{univ},1}}$ times uniformly bounded factors, 
\small\begin{align}
\|\bar{\mathbf{H}}_{T,x}^{N}(\wt{\Psi}^{N,2,0,\mathfrak{t}_{\mathrm{av},0}^{\sim}}\bar{\mathbf{Z}}^{N})\|_{\mathfrak{t}_{\mathrm{st}};\mathbb{T}_{N}} \ &\lesssim \  N^{-\beta_{\mathrm{univ},2}}\langle\bar{\mathbf{Z}}^{N}\rangle_{\mathfrak{t}_{\mathrm{st}};\mathbb{T}_{N}}. \label{eq:KPZNLStep22b}
\end{align}\normalsize\normalsize
\end{lemma}
\begin{proof}[Proof of \emph{Lemma \ref{lemma:KPZNLD1B1B}}]
It suffices to follow the proof of Lemma \ref{lemma:KPZNLD1B1A} verbatim, make the replacement \eqref{eq:KPZNLReplace}, use the bound \eqref{eq:KPZNLStep2II} in place of \eqref{eq:KPZNLStep2I}, use time-scales in Corollary \ref{corollary:D1B1B} in place of those in Corollary \ref{corollary:D1B1A}, and use the estimate in Corollary \ref{corollary:D1B1B} in place of that in Corollary \ref{corollary:D1B1A}. This also means using Lemma \ref{lemma:KPZNLStep22}/\ref{lemma:KPZNLStep22b} instead of Lemma \ref{lemma:KPZNLStep2}/\ref{lemma:KPZNLStep2b}.
\end{proof}
\subsection{Proof of Lemma \ref{lemma:Step3}}
We have cutoff the space-average of the pseudo-gradient in $\bar{\Phi}^{N,2}$ in Lemma \ref{lemma:KPZNLS1B}, and we replaced this spatial-average-with-cutoff by a time-average on a mesoscopic scale in Lemma \ref{lemma:KPZNLD1B1A}. In Lemma \ref{lemma:KPZNLD1B1B}, we replaced the functionals $\wt{\mathfrak{g}}^{\mathfrak{l}}$ with pseudo-gradient factors with mesoscopic time-averages, too. For Lemma \ref{lemma:Step3}, we introduce the following inductive scheme transferring a priori estimates, of main interest in Propositions \ref{prop:D1B2A} and \ref{prop:D1B2B}, and Corollaries \ref{corollary:D1B2A} and \ref{corollary:D1B2B}, between time-averages on progressively larger scales. This induction is a precise implementation of the would-be-proof of Pseudo-Proposition \ref{pprop:S4}.
\begin{lemma}\label{lemma:Step3A}
 Recall the time-scales $\{\mathfrak{t}_{\mathrm{av},m}\}_{m=1}^{\mathfrak{m}_{+}}$ and exponents $\{\beta_{m}\}_{m=0}^{\mathfrak{m}_{+}}$ in \emph{Corollary \ref{corollary:D1B2A}} and the space-time averages with cutoffs in \emph{Definition \ref{definition:D1B2A}}. Given any $m \in \Z_{\geq 1}$, define the following double averaging which groups scale-$\mathfrak{t}_{\mathrm{av},m}$ time-averages with a priori upper bounds into blocks of time-scale $\mathfrak{t}_{\mathrm{av},m+1}$, and these blocks, composed of scale-$\mathfrak{t}_{\mathrm{av},m}$ time-averages-with-upper-bound-cutoff, are grouped into the ``maximal" time-scale $\mathfrak{t}_{\mathrm{av},\mathfrak{m}_{+}} \in \R_{>0}$. Precisely, we define the following double average in which scale-$\mathfrak{t}_{\mathrm{av},m}$ time-averages with the upper-bound-cutoff of order $N^{-\beta_{m}}$ are defined afterwards/in the second line below:
\small\begin{align}
\Gamma_{T,x}^{N,m,m+1} \ &\overset{\bullet}= \ \wt{\sum}_{\mathfrak{l}=0}^{\mathfrak{t}_{\mathrm{av},\mathfrak{m}_{+}}\mathfrak{t}_{\mathrm{av},m+1}^{-1}-1}\Gamma_{T+\mathfrak{l}\mathfrak{t}_{\mathrm{av},m+1},x}^{N,\mathfrak{t}_{\mathrm{av},m}} \ \overset{\bullet}= \ \wt{\sum}_{\mathfrak{l}=0}^{\mathfrak{t}_{\mathrm{av},\mathfrak{m}_{+}}\mathfrak{t}_{\mathrm{av},m+1}^{-1}-1}\wt{\sum}_{\mathfrak{j}=0}^{\mathfrak{t}_{\mathrm{av},m+1}\mathfrak{t}_{\mathrm{av},m}^{-1}-1}\Gamma_{T+\mathfrak{l}\mathfrak{t}_{\mathrm{av},m+1},x}^{N,\mathfrak{t}_{\mathrm{av},m},\mathfrak{j}} \\
\Gamma_{T,x}^{N,\mathfrak{t}_{\mathrm{av},m},\mathfrak{j}} \ &\overset{\bullet}= \ \mathscr{A}_{\mathfrak{t}_{\mathrm{av},m}}^{\mathbf{T},+}\mathscr{C}_{N^{\beta_{X}}}^{\mathbf{X},-}(\mathfrak{g}_{T+\mathfrak{j}\mathfrak{t}_{\mathrm{av},m},x})\cdot\mathbf{1}\left(\sup_{0\leq\mathfrak{t}\leq\mathfrak{t}_{\mathrm{av},m}}\mathfrak{t}\cdot\mathfrak{t}_{\mathrm{av},m}^{-1}|\mathscr{A}_{\mathfrak{t}}^{\mathbf{T},+}\mathscr{C}_{N^{\beta_{X}}}^{\mathbf{X},-}(\mathfrak{g}_{T+\mathfrak{j}\mathfrak{t}_{\mathrm{av},m},x})| \leq N^{-\beta_{m}}\right).
\end{align}\normalsize\normalsize
Take any random time $\mathfrak{t}_{\mathrm{st}} \in \R_{\geq0}$ satisfying $\mathfrak{t}_{\mathrm{st}}\leq1$ with probability 1. There exists a pair of universal constants $\beta_{\mathrm{univ},1},\beta_{\mathrm{univ},2}\in\R_{>0}$ so that outside an event of probability at most $N^{-\beta_{\mathrm{univ},1}}$ times uniformly bounded constants, simultaneously in $m \in\llbracket1,\mathfrak{m}_{+}-2\rrbracket$ we have the following which replaces an average of scale-$\mathfrak{t}_{\mathrm{av},m}$ time-averages with upper bound cutoffs with an average of scale-$\mathfrak{t}_{\mathrm{av},m+1}$ time-averages with slightly improved upper bound cutoffs:
\small\begin{align}
\|\bar{\mathbf{H}}_{T,x}^{N}(N^{\frac12}\Gamma^{N,m,m+1}\bar{\mathbf{Z}}^{N})\|_{\mathfrak{t}_{\mathrm{st}};\mathbb{T}_{N}} \ &\lesssim \ \|\bar{\mathbf{H}}_{T,x}^{N}(N^{\frac12}\Gamma^{N,m+1,m+2}\bar{\mathbf{Z}}^{N})\|_{\mathfrak{t}_{\mathrm{st}};\mathbb{T}_{N}} \ + \ N^{-\beta_{\mathrm{univ},2}}\|\bar{\mathbf{Z}}^{N}\|_{\mathfrak{t}_{\mathrm{st}};\mathbb{T}_{N}}.
\end{align}\normalsize\normalsize
The same is true after replacement \emph{\eqref{eq:KPZNLReplace}} if we use time-scales and exponents in \emph{Corollary \ref{corollary:D1B2B}} instead of those in \emph{Corollary \ref{corollary:D1B2A}}.
\end{lemma}
Lemma \ref{lemma:Step3A} is difficult to justify with words; it roughly amounts to comparing $\Gamma^{N,m,m+1}$ and $\Gamma^{N,m+1,m+2}$ with bootstrapping and improving a priori bounds on time-scale $\mathfrak{t}_{\mathrm{av},m}$ into a priori bounds on time-scale $\mathfrak{t}_{\mathrm{av},m+1}$. One important observation for the proof of Lemma \ref{lemma:Step3A} is that time-averages on larger time-scales are averages of time-averages on shorter time-scales. This lets us transfer between time-scales in the proof for Lemma \ref{lemma:Step3A}. We emphasize that error terms we get when we compare $\Gamma^{N,m,m+1}$ and $\Gamma^{N,m+1,m+2}$ and when we bootstrap a priori bounds to larger time-scales can be controlled by Corollaries \ref{corollary:D1B2A} and \ref{corollary:D1B2B}.

We proceed with another preparatory lemma which starts the inductive procedure that Lemma \ref{lemma:Step3A} propagates.
\begin{lemma}\label{lemma:Step3B}
 We recap setting of \emph{Lemma \ref{lemma:Step3A}}. Recall the sequences $\{\mathfrak{t}_{\mathrm{av},m}\}_{m=1}^{\mathfrak{m}_{+}}$ and $\{\beta_{m}\}_{m=0}^{\mathfrak{m}_{+}}$ in \emph{Corollary \ref{corollary:D1B2A}}, and define
\small\begin{align}
\Gamma_{T,x}^{N,1,2} \ &\overset{\bullet}= \ \wt{\sum}_{\mathfrak{l}=0}^{\mathfrak{t}_{\mathrm{av},\mathfrak{m}_{+}}\mathfrak{t}_{\mathrm{av},2}^{-1}-1}\Gamma_{T+\mathfrak{l}\mathfrak{t}_{\mathrm{av},2},x}^{N,\mathfrak{t}_{\mathrm{av},1}} \ \overset{\bullet}= \ \wt{\sum}_{\mathfrak{l}=0}^{\mathfrak{t}_{\mathrm{av},\mathfrak{m}_{+}}\mathfrak{t}_{\mathrm{av},2}^{-1}-1}\wt{\sum}_{\mathfrak{j}=0}^{\mathfrak{t}_{\mathrm{av},2}\mathfrak{t}_{\mathrm{av},1}^{-1}-1}\Gamma_{T+\mathfrak{l}\mathfrak{t}_{\mathrm{av},2},x}^{N,\mathfrak{t}_{\mathrm{av},1},\mathfrak{j}} \\
\Gamma_{T,x}^{N,\mathfrak{t}_{\mathrm{av},1},\mathfrak{j}} \ &\overset{\bullet}= \ \mathscr{A}_{\mathfrak{t}_{\mathrm{av},1}}^{\mathbf{T},+}\mathscr{C}_{N^{\beta_{X}}}^{\mathbf{X},-}(\mathfrak{g}_{T+\mathfrak{j}\mathfrak{t}_{\mathrm{av},1},x}) \cdot \mathbf{1}\left(\sup_{0\leq\mathfrak{t}\leq\mathfrak{t}_{\mathrm{av},1}}\mathfrak{t}\cdot\mathfrak{t}_{\mathrm{av},1}^{-1}|\mathscr{A}_{\mathfrak{t}}^{\mathbf{T},+}\mathscr{C}_{N^{\beta_{X}}}^{\mathbf{X},-}(\mathfrak{g}_{T+\mathfrak{j}\mathfrak{t}_{\mathrm{av},1},x})| \leq N^{-\beta_{1}} \right).
\end{align}\normalsize\normalsize
Consider any random time $\mathfrak{t}_{\mathrm{st}} \in \R_{\geq0}$ satisfying $\mathfrak{t}_{\mathrm{st}}\leq1$ with probability 1. There exist universal constants $\beta_{\mathrm{univ},1},\beta_{\mathrm{univ},2}\in\R_{>0}$ such that outside an event of probability at most $N^{-\beta_{\mathrm{univ},1}}$ times uniformly bounded constants, we have
\small\begin{align}
\|\bar{\mathbf{H}}_{T,x}^{N}(N^{\frac12}\mathscr{A}_{\mathfrak{t}_{\mathrm{av},\mathfrak{m}_{+}}}^{\mathbf{T},+}\mathscr{C}_{N^{\beta_{X}}}^{\mathbf{X},-}(\mathfrak{g}) \cdot \bar{\mathbf{Z}}^{N})\|_{\mathfrak{t}_{\mathrm{st}};\mathbb{T}_{N}} \ &\lesssim \ \|\bar{\mathbf{H}}_{T,x}^{N}(N^{\frac12}\Gamma^{N,1,2}\bar{\mathbf{Z}}^{N})\|_{\mathfrak{t}_{\mathrm{st}};\mathbb{T}_{N}} \ + \ N^{-\beta_{\mathrm{univ},2}}\|\bar{\mathbf{Z}}^{N}\|_{\mathfrak{t}_{\mathrm{st}};\mathbb{T}_{N}}.
\end{align}\normalsize\normalsize
The same is true after replacement \emph{\eqref{eq:KPZNLReplace}} if we use time-scales and exponents in \emph{Corollary \ref{corollary:D1B2B}} instead of those in \emph{Corollary \ref{corollary:D1B2A}}.
\end{lemma}
\begin{proof}[Proof of \emph{Lemma \ref{lemma:Step3}}]
We first make the same disclaimer concerning high-probability statements and events as we made at the beginning of the proof of Proposition \ref{prop:KPZNL}. We also focus first on showing the first estimate \eqref{eq:Step3I} from the statement of Lemma \ref{lemma:Step3}. We use the first bound for order $N^{1/2}$ quantities in Lemma \ref{lemma:Step3B} and iterate the first estimate for order $N^{1/2}$ terms in Lemma \ref{lemma:Step3A}. As $\mathfrak{m}_{+} \lesssim 1$ by Corollary \ref{corollary:D1B2A}, we iterate Lemma \ref{lemma:Step3A} a uniformly bounded number of times, so with high probability
\small\begin{align}
\|\bar{\mathbf{H}}_{T,x}^{N}(N^{\frac12}\mathscr{A}^{\mathbf{T},+}_{\mathfrak{t}_{\mathrm{av},\mathfrak{m}_{+}}}\mathscr{C}_{N^{\beta_{X}}}^{\mathbf{X},-}(\mathfrak{g}) \cdot \bar{\mathbf{Z}}^{N})\|_{\mathfrak{t}_{\mathrm{st}};\mathbb{T}_{N}} \ &\lesssim \ \|\bar{\mathbf{H}}_{T,x}^{N}(N^{\frac12}\Gamma^{N,\mathfrak{m}_{+}-1,\mathfrak{m}_{+}}\bar{\mathbf{Z}}^{N})\|_{\mathfrak{t}_{\mathrm{st}};\mathbb{T}_{N}} \ + \ N^{-\beta_{\mathrm{univ},2}}\|\bar{\mathbf{Z}}^{N}\|_{\mathfrak{t}_{\mathrm{st}};\mathbb{T}_{N}}. \label{eq:Step31}
\end{align}\normalsize\normalsize
By definition in Lemma \ref{lemma:Step3A} $\Gamma^{N,\mathfrak{m}_{+}-1,\mathfrak{m}_{+}}$ is an average of terms that are $\mathrm{O}(N^{-\beta_{\mathfrak{m}_{+}-1}})$. As $\beta_{\mathfrak{m}_{+}-1} = \beta_{\mathfrak{m}_{+}}-\e_{X,3} \geq \frac12 + 2\e_{X,3}$, by the space-time convolution estimate \eqref{eq:HKEConvolution}, we control the first term on the RHS of \eqref{eq:Step31} directly/deterministically and get 
\small\begin{align}
\|\bar{\mathbf{H}}_{T,x}^{N}(N^{\frac12}\Gamma^{N,\mathfrak{m}_{+}-1,\mathfrak{m}_{+}}\bar{\mathbf{Z}}^{N})\|_{\mathfrak{t}_{\mathrm{st}};\mathbb{T}_{N}} \ \lesssim \ N^{\frac12}\|\Gamma^{N,\mathfrak{m}_{+}-1,\mathfrak{m}_{+}}\bar{\mathbf{Z}}^{N}\|_{\mathfrak{t}_{\mathrm{st}};\mathbb{T}_{N}} \ \lesssim \ N^{\frac12-\beta_{\mathfrak{m}_{+}-1}}\|\bar{\mathbf{Z}}^{N}\|_{\mathfrak{t}_{\mathrm{st}};\mathbb{T}_{N}} \ &\leq \ N^{-2\e_{X,3}}\langle\bar{\mathbf{Z}}^{N}\rangle_{\mathfrak{t}_{\mathrm{st}};\mathbb{T}_{N}}. \label{eq:Step32}
\end{align}\normalsize\normalsize
Combining \eqref{eq:Step31} and \eqref{eq:Step32} yields the first estimate \eqref{eq:Step3I}. The bound \eqref{eq:Step3II} follows via the same argument except that for our application of Lemmas \ref{lemma:Step3A} and \ref{lemma:Step3B}, and also for \eqref{eq:Step31} and \eqref{eq:Step32} above, we make the replacement \eqref{eq:KPZNLReplace}, and we adopt the time-scales $\{\mathfrak{t}_{\mathrm{av},m}^{\sim}\}_{m=1}^{\mathfrak{m}_{+}^{\sim}}$, exponents $\{\beta_{m}^{\sim}\}_{m=0}^{\mathfrak{m}_{+}^{\sim}}$, and estimates in Corollary \ref{corollary:D1B2B} instead of in Corollary \ref{corollary:D1B2A}.
\end{proof}
\subsection{Proofs of Technical Estimates}
We now give proofs of all technical ingredients in the order that they are presented.
\begin{proof}[Proof of \emph{Lemma \ref{lemma:KPZNLStep2}}]
Recall the notation $\Psi^{N,2,\mathfrak{t}_{1},\mathfrak{t}_{2}}$ in Lemma \ref{lemma:KPZNLStep2}. The architecture of our proof will be computing explicitly the heat-operator term inside the $\|\|_{\mathfrak{t}_{\mathrm{st}};\mathbb{T}_{N}}$-norm on the LHS of the proposed estimate and then estimating each term in our resulting calculations. This starts with the following elementary consideration in which we decompose a scale-$\mathfrak{t}_{2}$ time-average in terms of time-shifted scale-$\mathfrak{t}_{1}$ time-averages and then we remove these time-shifts to get the $\Phi_{\mathfrak{j}}$-errors below. We clarify this after:
\small\begin{align}
\bar{\mathbf{H}}_{T,x}^{N}(N^{\frac12}\mathscr{A}_{\mathfrak{t}_{2}}^{\mathbf{T},+}\mathscr{C}_{N^{\beta_{X}}}^{\mathbf{X},-}(\mathfrak{g}_{S,y}) \cdot \bar{\mathbf{Z}}^{N}_{S,y}) \ &= \ \bar{\mathbf{H}}_{T,x}^{N}\left(N^{\frac12}\wt{\sum}_{\mathfrak{j}=0}^{\mathfrak{t}_{2}\mathfrak{t}_{1}^{-1}-1} \mathscr{A}_{\mathfrak{t}_{1}}^{\mathbf{T},+}\mathscr{C}_{N^{\beta_{X}}}^{\mathbf{X},-}(\mathfrak{g}_{S+\mathfrak{j}\mathfrak{t}_{1},y}) \cdot \bar{\mathbf{Z}}^{N}_{S,y}\right) \\
&= \ \bar{\mathbf{H}}_{T,x}^{N}(N^{\frac12}\mathscr{A}_{\mathfrak{t}_{1}}^{\mathbf{T},+}\mathscr{C}_{N^{\beta_{X}}}^{\mathbf{X},-}(\mathfrak{g}_{S,y}) \cdot \bar{\mathbf{Z}}_{S,y}^{N}) \ + \ \wt{\sum}_{\mathfrak{j}=0}^{\mathfrak{t}_{2}\mathfrak{t}_{1}^{-1}-1} \Phi_{\mathfrak{j}}. \label{eq:KPZNLStep2I1}
\end{align}\normalsize\normalsize
The first line follows from the observation that a time-average with respect to the time-scale $\mathfrak{t}_{2}$ is an average of suitably shifted time-averages with respect to the time-scale $\mathfrak{t}_{1}$. We recall from the statement of Lemma \ref{lemma:KPZNLStep2} that $\mathfrak{t}_{2}$ is a positive integer multiple of $\mathfrak{t}_{1}$, so that the sum and its limits of indices are all well-defined. For example, the previous decomposition yielding \eqref{eq:KPZNLStep2I1} is a relative of the fact that an average of 10 terms may be written as the average of 5 other terms, each of which is the average of 2 ``adjacent" terms. We have also introduced the following $\Phi_{\mathfrak{j}}$-errors obtained after removing the $\mathfrak{j}\mathfrak{t}_{1}$-time shift to get \eqref{eq:KPZNLStep2I1}:
\small\begin{align}
\Phi_{\mathfrak{j}} \ &\overset{\bullet}= \ \bar{\mathbf{H}}_{T,x}^{N}\left(N^{\frac12}\mathscr{A}_{\mathfrak{t}_{1}}^{\mathbf{T},+}\mathscr{C}_{N^{\beta_{X}}}^{\mathbf{X},-}(\mathfrak{g}_{S+\mathfrak{j}\mathfrak{t}_{1},y}) \cdot \bar{\mathbf{Z}}^{N}_{S,y}\right) \ - \ \bar{\mathbf{H}}_{T,x}^{N}\left(N^{\frac12}\mathscr{A}_{\mathfrak{t}_{1}}^{\mathbf{T},+}\mathscr{C}_{N^{\beta_{X}}}^{\mathbf{X},-}(\mathfrak{g}_{S,y}) \cdot \bar{\mathbf{Z}}_{S,y}^{N}\right) \ = \ \bar{\mathbf{H}}_{T,x}^{N}\left(N^{\frac12}\grad_{\mathfrak{j}\mathfrak{t}_{1}}^{\mathbf{T}}\left(\mathscr{A}_{\mathfrak{t}_{1}}^{\mathbf{T},+}\mathscr{C}_{N^{\beta_{X}}}^{\mathbf{X},-}(\mathfrak{g}_{S,y})\right)\bar{\mathbf{Z}}_{S,y}^{N}\right). \nonumber
\end{align}\normalsize\normalsize
From \eqref{eq:KPZNLStep2I1}, we subtract the first term on the RHS from both sides. By definition of $\Psi^{N,2,\mathfrak{t}_{1},\mathfrak{t}_{2}}$, we get the following formula:
\small\begin{align}
\bar{\mathbf{H}}_{T,x}^{N}(\Psi^{N,2,\mathfrak{t}_{1},\mathfrak{t}_{2}}\bar{\mathbf{Z}}^{N}) \ = \ \bar{\mathbf{H}}_{T,x}^{N}(N^{\frac12}\mathscr{A}_{\mathfrak{t}_{2}}^{\mathbf{T},+}\mathscr{C}_{N^{\beta_{X}}}^{\mathbf{X},-}(\mathfrak{g}_{S,y}) \cdot \bar{\mathbf{Z}}^{N}_{S,y}) - \bar{\mathbf{H}}_{T,x}^{N}(N^{\frac12}\mathscr{A}_{\mathfrak{t}_{1}}^{\mathbf{T},+}\mathscr{C}_{N^{\beta_{X}}}^{\mathbf{X},-}(\mathfrak{g}_{S,y}) \cdot \bar{\mathbf{Z}}_{S,y}^{N}) \ &= \ \wt{\sum}_{\mathfrak{j}=0}^{\mathfrak{t}_{2}\mathfrak{t}_{1}^{-1}-1} \Phi_{\mathfrak{j}}. \label{eq:KPZNLStep2I2}
\end{align}\normalsize\normalsize
We proceed to rewrite $\Phi_{\mathfrak{j}}$ to transform the time-gradient of the space-time average defining $\Phi_{\mathfrak{j}}$-terms into time-gradients of the heat operator and the microscopic Cole-Hopf transform. We will achieve this by a time-discrete-version of integration-by-parts for the time-integral in the heat operator $\bar{\mathbf{H}}^{N}$. To make this precise, we first use the following time-discrete-type Leibniz rule:
\small\begin{align}
\grad_{\mathfrak{j}\mathfrak{t}_{1}}^{\mathbf{T}}\left(\mathscr{A}_{\mathfrak{t}_{1}}^{\mathbf{T},+}\mathscr{C}_{N^{\beta_{X}}}^{\mathbf{X},-}(\mathfrak{g}_{S,y})\right)\bar{\mathbf{Z}}_{S,y}^{N} \ &= \ \grad_{\mathfrak{j}\mathfrak{t}_{1}}^{\mathbf{T}}\left(\mathscr{A}_{\mathfrak{t}_{1}}^{\mathbf{T},+}\mathscr{C}_{N^{\beta_{X}}}^{\mathbf{X},-}(\mathfrak{g}_{S,y})\bar{\mathbf{Z}}_{S-\mathfrak{j}\mathfrak{t}_{1},y}^{N}\right) + \mathscr{A}_{\mathfrak{t}_{1}}^{\mathbf{T},+}\mathscr{C}_{N^{\beta_{X}}}^{\mathbf{X},-}(\mathfrak{g}_{S,y})\grad_{-\mathfrak{j}\mathfrak{t}_{1}}^{\mathbf{T}}\bar{\mathbf{Z}}_{S,y}^{N}. \label{eq:LeibnizKPZNL}
\end{align}\normalsize\normalsize
We will now plug the previous display into the heat operator in the definition of $\Phi_{\mathfrak{j}}$/the far RHS of the display preceding \eqref{eq:KPZNLStep2I2} to get the following representation of $\Phi_{\mathfrak{j}}$ in which the time-gradient below that acts on the space-time pseudo-gradient average times $\bar{\mathbf{Z}}^{N}$ will be transferred to the heat kernel in $\bar{\mathbf{H}}^{N}$ via an integration-by-parts-type consideration as we will soon see:
\small\begin{align}
\Phi_{\mathfrak{j}} \ &= \ \bar{\mathbf{H}}_{T,x}^{N}\left(N^{\frac12}\grad_{\mathfrak{j}\mathfrak{t}_{1}}^{\mathbf{T}}\left(\mathscr{A}_{\mathfrak{t}_{1}}^{\mathbf{T},+}\mathscr{C}_{N^{\beta_{X}}}^{\mathbf{X},-}(\mathfrak{g}_{S,y})\bar{\mathbf{Z}}_{S-\mathfrak{j}\mathfrak{t}_{1},y}^{N}\right)\right) + \bar{\mathbf{H}}_{T,x}^{N}\left(N^{\frac12}\mathscr{A}_{\mathfrak{t}_{1}}^{\mathbf{T},+}\mathscr{C}_{N^{\beta_{X}}}^{\mathbf{X},-}(\mathfrak{g}_{S,y})\grad_{-\mathfrak{j}\mathfrak{t}_{1}}^{\mathbf{T}}\bar{\mathbf{Z}}_{S,y}^{N}\right) \ \overset{\bullet}= \ \Phi_{\mathfrak{j},1}+\Phi_{\mathfrak{j},2}. \label{eq:KPZNLStep2I3}
\end{align}\normalsize\normalsize
The identities \eqref{eq:KPZNLStep2I2} and \eqref{eq:KPZNLStep2I3} combined with the triangle inequality yield the following estimate:
\small\begin{align}
\|\bar{\mathbf{H}}_{T,x}^{N}(\Psi^{N,2,\mathfrak{t}_{1},\mathfrak{t}_{2}}\bar{\mathbf{Z}}^{N})\|_{\mathfrak{t}_{\mathrm{st}};\mathbb{T}_{N}} \ \leq \ \wt{\sum}_{\mathfrak{j}=0}^{\mathfrak{t}_{2}\mathfrak{t}_{1}^{-1}-1}\|\Phi_{\mathfrak{j}}\|_{\mathfrak{t}_{\mathrm{st}};\mathbb{T}_{N}} \ \leq \ \wt{\sum}_{\mathfrak{j}=0}^{\mathfrak{t}_{2}\mathfrak{t}_{1}^{-1}-1}{\sum}_{\mathfrak{k}=1}^{2}\|\Phi_{\mathfrak{j},\mathfrak{k}}\|_{\mathfrak{t}_{\mathrm{st}};\mathbb{T}_{N}} \ = \ 2\wt{\sum}_{\mathfrak{j}=0}^{\mathfrak{t}_{2}\mathfrak{t}_{1}^{-1}-1}\wt{\sum}_{\mathfrak{k}=1}^{2}\|\Phi_{\mathfrak{j},\mathfrak{k}}\|_{\mathfrak{t}_{\mathrm{st}};\mathbb{T}_{N}}. \label{eq:KPZNLStep2I5}
\end{align}\normalsize\normalsize
To control the average of the norms on the far RHS of \eqref{eq:KPZNLStep2I5}, it suffices to estimate each summand in \eqref{eq:KPZNLStep2I5} uniformly in both indices on the same high-probability event. We do this by heat operator/time-regularity estimates. First note $\mathfrak{j}\mathfrak{t}_{1}\lesssim\mathfrak{t}_{2}\lesssim N^{-1}$.
\begin{itemize}[leftmargin=*]
\item We bound $\Phi_{\mathfrak{j},1}$ in \eqref{eq:KPZNLStep2I3} by deterministic analytic means. In particular we move the time-gradient onto the heat operator and take advantage of smoothness for the heat operator in time with a sufficiently integrable short-time singularity. This is done in Lemma \ref{lemma:HKEB} which yields the following given any $\delta\in\R_{>0}$ arbitrarily small but universal for which we recall $\mathfrak{j}\mathfrak{t}_{1}\lesssim\mathfrak{t}_{2}\lesssim N^{-1}$ and boundedness of $\mathscr{A}^{\mathbf{T},+}\mathscr{C}^{\mathbf{X},-}(\mathfrak{g})$. Note $\Phi_{\mathfrak{j},1}$ in \eqref{eq:KPZNLStep2I3} sees $\bar{\mathbf{Z}}^{N}$ until time $\mathfrak{t}_{\mathrm{st}}$, which is why its $\|\|_{\mathfrak{t}_{\mathrm{st}};\mathbb{T}_{N}}$-norm appears below. We also clarify such application of Lemma \ref{lemma:HKEB} is done with $\bar{\varphi}$ equal to the space-time average times the time-shifted $\bar{\mathbf{Z}}^{N}$-process in the time-gradient in $\Phi_{\mathfrak{j},1}$ and with the time-scale $\mathfrak{t}_{\mathrm{reg}}=\mathfrak{j}\mathfrak{t}_{1}$ as well as $\mathfrak{t}_{\mathrm{st}}$ therein equal to $\mathfrak{t}_{\mathrm{st}}$ here:
\small\begin{align}
\|\Phi_{\mathfrak{j},1}\|_{\mathfrak{t}_{\mathrm{st}};\mathbb{T}_{N}} \ \lesssim_{\delta} \ N^{\frac12+2\delta}\mathfrak{j}\mathfrak{t}_{1}\|\bar{\mathbf{Z}}^{N}\|_{\mathfrak{t}_{\mathrm{st}};\mathbb{T}_{N}} \ \lesssim \ N^{\frac12+2\delta}\mathfrak{t}_{2}\|\bar{\mathbf{Z}}^{N}\|_{\mathfrak{t}_{\mathrm{st}};\mathbb{T}_{N}} \ \lesssim \ N^{-\frac12+2\delta}\|\bar{\mathbf{Z}}^{N}\|_{\mathfrak{t}_{\mathrm{st}};\mathbb{T}_{N}} \ \leq \ N^{-\frac12+2\delta}\langle\bar{\mathbf{Z}}^{N}\rangle_{\mathfrak{t}_{\mathrm{st}};\mathbb{T}_{N}}. \label{eq:KPZNLStep2I6}
\end{align}\normalsize\normalsize
\item It remains to estimate $\Phi_{\mathfrak{j},2}$. Observe that all of our steps/bounds have been deterministic thus far. We will now use a random high-probability estimate for $\Phi_{\mathfrak{j},2}$. It will be convenient to first use the notation $\kappa_{\mathfrak{j},\mathfrak{t}_{1},\e} = N^{-2+\e} + N^{\e}\mathfrak{j}\mathfrak{t}_{1}$. We think of $\e\in\R_{>0}$ as arbitrarily small but universal in such notation. We also emphasize $\Phi_{\mathfrak{j},2}$ is the heat operator acting on the space-time average of the pseudo-gradient times the time-gradient of $\bar{\mathbf{Z}}^{N}$ before time $\mathfrak{t}_{\mathrm{st}}\in\R_{\geq0}$.

Via Proposition \ref{prop:TRGTProp}, with probability at least $1-\kappa_{\e,C}N^{-C}$ for any $\e,C\in\R_{>0}$, where in this application of Proposition \ref{prop:TRGTProp} we choose $\mathfrak{t}_{\mathrm{st}}\in\R_{\geq0}$ therein equal to $\mathfrak{t}_{\mathrm{st}}\in\R_{\geq0}$ here and we choose $\mathfrak{t}_{\mathrm{reg}} = \mathfrak{j}\mathfrak{t}_{1}$, we may control the time-gradients of $\bar{\mathbf{Z}}^{N}$ in $\Phi_{\mathfrak{j},2}$ uniformly in the integral on the same high-probability event. We then establish the following estimate upon observing again the $\bar{\mathbf{Z}}^{N}$-process is only evaluated before time $\mathfrak{t}_{\mathrm{st}}\in\R_{\geq0}$ in the integral/heat operator $\Phi_{\mathfrak{j},2}$. We clarify after:
\small\begin{align}
|\Phi_{\mathfrak{j},2}| \ &\lesssim \ \|\grad_{-\mathfrak{j}\mathfrak{t}_{1}}^{\mathbf{T}}\bar{\mathbf{Z}}^{N}\|_{\mathfrak{t}_{\mathrm{st}};\mathbb{T}_{N}}\bar{\mathbf{H}}_{T,x}^{N}(N^{\frac12}|\mathscr{A}_{\mathfrak{t}_{1}}^{\mathbf{T},+}\mathscr{C}_{N^{\beta_{X}}}^{\mathbf{X},-}(\mathfrak{g}_{S,y})|) \ \lesssim \ \kappa_{\mathfrak{j},\mathfrak{t}_{1},\e}^{1/4}\langle\bar{\mathbf{Z}}^{N}\rangle_{\mathfrak{t}_{\mathrm{st}};\mathbb{T}_{N}} \|\bar{\mathbf{H}}_{T,x}^{N}(N^{\frac12}|\mathscr{A}_{\mathfrak{t}_{1}}^{\mathbf{T},+}\mathscr{C}_{N^{\beta_{X}}}^{\mathbf{X},-}(\mathfrak{g})|)\|_{1;\mathbb{T}_{N}}. \label{eq:KPZNLStep2I9}
\end{align}\normalsize\normalsize
Let us emphasize \eqref{eq:KPZNLStep2I9} holds uniformly over $[0,\mathfrak{t}_{\mathrm{st}}]\times\mathbb{T}_{N}$ with high-probability. We clarify that the second estimate, again, follows by estimating the time-gradient in $\Phi_{\mathfrak{j},2}$ via the Holder regularity bound of $\kappa_{\mathfrak{j},\mathfrak{t}_{1},\e}^{1/4}$ times the bracket-norm of $\bar{\mathbf{Z}}^{N}$, which holds uniformly in integration variables in the heat operator $\bar{\mathbf{H}}^{N}$ in $\Phi_{\mathfrak{j},2}$ with high probability by Proposition \ref{prop:TRGTProp}.

We recall \eqref{eq:KPZNLStep2I9} holds uniformly over $[0,\mathfrak{t}_{\mathrm{st}}]\times\mathbb{T}_{N}$. Thus, upon recalling $\kappa_{\mathfrak{j},\mathfrak{t}_{1},\e}$ defined at the beginning of this bullet point, we get the following high-probability estimate for which we provide a little more explanation afterwards as well:
\small\begin{align}
\|\Phi_{\mathfrak{j},2}\|_{\mathfrak{t}_{\mathrm{st}};\mathbb{T}_{N}} \ &\lesssim \ N^{-1/2+\e/4}\langle\bar{\mathbf{Z}}^{N}\rangle_{\mathfrak{t}_{\mathrm{st}};\mathbb{T}_{N}} \|\bar{\mathbf{H}}_{T,x}^{N}(N^{1/2}|\mathscr{A}_{\mathfrak{t}_{1}}^{\mathbf{T},+}\mathscr{C}_{N^{\beta_{X}}}^{\mathbf{X},-}(\mathfrak{g})|)\|_{1;\mathbb{T}_{N}} \ + \ N^{\e/4}\mathfrak{j}^{1/4}\mathfrak{t}_{1}^{1/4}\langle\bar{\mathbf{Z}}^{N}\rangle_{\mathfrak{t}_{\mathrm{st}};\mathbb{T}_{N}} \|\bar{\mathbf{H}}_{T,x}^{N}(N^{1/2}|\mathscr{A}_{\mathfrak{t}_{1}}^{\mathbf{T},+}\mathscr{C}_{N^{\beta_{X}}}^{\mathbf{X},-}(\mathfrak{g})|)\|_{1;\mathbb{T}_{N}} \nonumber \\
&\lesssim \ N^{-1/6+\e/4}\langle\bar{\mathbf{Z}}^{N}\rangle_{\mathfrak{t}_{\mathrm{st}};\mathbb{T}_{N}} \ + \ N^{\e/4}\mathfrak{j}^{1/4}\mathfrak{t}_{1}^{1/4}\langle\bar{\mathbf{Z}}^{N}\rangle_{\mathfrak{t}_{\mathrm{st}};\mathbb{T}_{N}} \|\bar{\mathbf{H}}_{T,x}^{N}(N^{1/2}|\mathscr{A}_{\mathfrak{t}_{1}}^{\mathbf{T},+}\mathscr{C}_{N^{\beta_{X}}}^{\mathbf{X},-}(\mathfrak{g})|)\|_{1;\mathbb{T}_{N}}. \label{eq:KPZNLStep2I10}
\end{align}\normalsize\normalsize
The line \eqref{eq:KPZNLStep2I10} follows from the line above by the convolution bound \eqref{eq:HKEConvolution} applied to the first term in the line before and the bound $|\mathscr{A}^{\mathbf{T},+}\mathscr{C}^{\mathbf{X},-}(\mathfrak{g})| \lesssim N^{-\beta_{X}/2-\e_{X,1}/2} \leq N^{-1/6}$ via time-averaging the a priori $\mathscr{C}^{\mathbf{X},-}$-bound; see Definition \ref{definition:S1B}.
\end{itemize}
We observe the previous two bullet points provide high-probability estimates for $\|\Phi_{\mathfrak{j},\mathfrak{k}}\|_{\mathfrak{t}_{\mathrm{st}};\mathbb{T}_{N}}$ \emph{for each} of the indices $\mathfrak{j},\mathfrak{k}\in\Z_{\geq0}$. By ``high-probability" in this claim, we mean an event whose complement has probability $\kappa_{C}N^{-C}$ for any $C\in\R_{\geq0}$. With a union bound over the complement of all events over all indices $\mathfrak{j},\mathfrak{k}$, we deduce the event on which all of our $\|\Phi_{\mathfrak{j},\mathfrak{k}}\|_{\mathfrak{t}_{\mathrm{st}};\mathbb{T}_{N}}$-estimates hold \emph{simultaneously} has complement with probability at most $\kappa_{C}N^{-C}\mathfrak{t}_{2}\mathfrak{t}_{1}^{-1} \lesssim \kappa_{C}N^{-C+100}$ with universal implied constant. This final probability bound follows from the assumed constraints $N^{-100}\lesssim\mathfrak{t}_{1},\mathfrak{t}_{2}\lesssim N^{-1}$ in the statement of Lemma \ref{lemma:KPZNLStep2}. Combining these simultaneous high-probability $\|\Phi_{\mathfrak{j},\mathfrak{k}}\|_{\mathfrak{t}_{\mathrm{st}};\mathbb{T}_{N}}$-estimates with \eqref{eq:KPZNLStep2I5} yields the desired estimate.
\end{proof}
\begin{proof}[Proof of \emph{Lemma \ref{lemma:KPZNLStep2b}}]
Let us first define the intermediate sub-microscopic time-scale of $\mathfrak{t}_{1,2} \overset{\bullet}= N^{-100}$ between $\mathfrak{t}_{1},\mathfrak{t}_{2}\in\R_{\geq0}$. The triangle inequality gives the following in which we first replace a scale-$0$ time-average by the intermediate time-scale $\mathfrak{t}_{1,2}$:
\small\begin{align}
\|\bar{\mathbf{H}}_{T,x}^{N}(\Psi^{N,2,0,\mathfrak{t}_{\mathrm{av},0}}\bar{\mathbf{Z}}^{N})\|_{\mathfrak{t}_{\mathrm{st}};\mathbb{T}_{N}} \ &\leq \ \|\bar{\mathbf{H}}_{T,x}^{N}(\Psi^{N,2,0,\mathfrak{t}_{1,2}}\bar{\mathbf{Z}}^{N})\|_{\mathfrak{t}_{\mathrm{st}};\mathbb{T}_{N}}  + \|\bar{\mathbf{H}}_{T,x}^{N}(\Psi^{N,2,\mathfrak{t}_{1,2},\mathfrak{t}_{\mathrm{av},0}}\bar{\mathbf{Z}}^{N})\|_{\mathfrak{t}_{\mathrm{st}};\mathbb{T}_{N}}. \label{eq:KPZNLStep2b1}  
\end{align}\normalsize\normalsize
Let us now explain the utility of this intermediate time-scale $\mathfrak{t}_{1,2} = N^{-100}$. The first term on the RHS of \eqref{eq:KPZNLStep2b1} can be estimated via first moment after pulling the $\bar{\mathbf{Z}}^{N}$-factor outside the heat operator. This first moment bound would yield a ``time-regularity" bound for local functionals of the particle system. It says at times $\mathfrak{t}_{1,2}=N^{-100}$, we do not expect to see any of the Poisson clocks ring, and thus local functionals are ``constant" or ``smooth" on the time-scale $\mathfrak{t}_{1,2}=N^{-100}$. The second term on the RHS of \eqref{eq:KPZNLStep2b1} is then bounded by \eqref{eq:KPZNLStep2I}, as we have introduced a preliminary time-average $\mathfrak{t}_{1,2}=N^{-100}$ for this second term.

We make the above precise. Pulling $\|\bar{\mathbf{Z}}^{N}\|_{\mathfrak{t}_{\mathrm{st}};\mathbb{T}_{N}}\leq\langle\bar{\mathbf{Z}}^{N}\rangle_{\mathfrak{t}_{\mathrm{st}};\mathbb{T}_{N}}$ from the heat operator in the first term on the RHS of \eqref{eq:KPZNLStep2b1},
\small\begin{align}
\|\bar{\mathbf{H}}_{T,x}^{N}(\Psi^{N,2,0,\mathfrak{t}_{1,2}}\bar{\mathbf{Z}}^{N})\|_{\mathfrak{t}_{\mathrm{st}};\mathbb{T}_{N}} \ &\lesssim \ \langle\bar{\mathbf{Z}}^{N}\rangle_{\mathfrak{t}_{\mathrm{st}};\mathbb{T}_{N}}\|\left(\int_{0}^{T}{\sum}_{y\in\mathbb{T}_{N}}|\Psi^{N,2,0,\mathfrak{t}_{1,2}}_{S,y}| \ \d S\right)\|_{\mathfrak{t}_{\mathrm{st}};\mathbb{T}_{N}} \nonumber \\
&\lesssim \ \langle\bar{\mathbf{Z}}^{N}\rangle_{\mathfrak{t}_{\mathrm{st}};\mathbb{T}_{N}} \int_{0}^{1}{\sum}_{y\in\mathbb{T}_{N}}|\Psi^{N,2,0,\mathfrak{t}_{1,2}}_{S,y}| \ \d S. \label{eq:KPZNLStep2b3}
\end{align}\normalsize\normalsize
The estimate before the far RHS of \eqref{eq:KPZNLStep2b3}/in the middle of the above display follows by a maximum principle $\mathbf{1}_{T\geq S}\bar{\mathbf{H}}_{S,T,x,y}^{N}\leq1$. The far RHS of \eqref{eq:KPZNLStep2b3} follows by extending the integration-domain in the integral of a non-negative term from $[0,T]\to[0,1]$.

To get a high-probability estimate for the RHS, by the Markov inequality it suffices to bound expectation of the RHS of \eqref{eq:KPZNLStep2b3} \emph{without} the bracket factor in front. To this end, we pull said expectation past the integral and summation so it hits the $\Psi$-factor. To bound expectation of this $\Psi$-factor, recall by definition this $\Psi$-term is the difference between the spatial-average-with-cutoff $\mathscr{C}^{\mathbf{X},-}(\mathfrak{g})$ and its time-average on scale-$\mathfrak{t}_{1,2}=N^{-100}$. As $\mathscr{C}^{\mathbf{X},-}(\mathfrak{g})$ is uniformly bounded with support with size of order $N^{\beta_{X}}$, this $\Psi$-term, in absolute value, is controlled by the number of times any of the Poisson clocks in the support of $\mathscr{C}^{\mathbf{X},-}(\mathfrak{g})$, all speed at most $N^{2}$ times uniformly bounded factors, ring in the time-scale $\mathfrak{t}_{1,2} = N^{-100}$. As a sum of independent Poissons is Poisson via additive rates, this is asking for the number of times a Poisson clock of speed $N^{2+\beta_{X}-100}$ rings in order 1 time. The probability that such Poisson clock rings at all is controlled by its speed if its speed is uniformly bounded. Thus, as $\beta_{X} = \frac13+\e_{X,1} \leq 1$,
\small\begin{align}
\E|\Psi^{N,2,0,\mathfrak{t}_{1,2}}_{S,y}| \ \lesssim \ N^{2+\beta_{X}-100} \ &\lesssim \ N^{-97}. \label{eq:KPZNLStep2b5}
\end{align}\normalsize\normalsize
As a consequence \eqref{eq:KPZNLStep2b3}, the proof of the Markov inequality gives the required estimate for the first term on the RHS of \eqref{eq:KPZNLStep2b1} outside an event of probability bounded by the following in which we employ \eqref{eq:KPZNLStep2b5} along with the Fubini theorem and control on an unaveraged sum/integral by the length of the sum/integration set times a supremum; recall $|\mathbb{T}_{N}|\lesssim N^{2}$ from Section \ref{section:Ctify}:
\small\begin{align}
N^{\beta_{\mathrm{univ},2}} \E\int_{0}^{1}{\sum}_{y\in\mathbb{T}_{N}}|\Psi^{N,2,0,\mathfrak{t}_{1,2}}_{S,y}| \ \d S \ = \ N^{\beta_{\mathrm{univ},2}}\int_{0}^{1}{\sum}_{y\in\mathbb{T}_{N}}\E|\Psi^{N,2,0,\mathfrak{t}_{1,2}}_{S,y}| \ \d S \ &\lesssim \ N^{\beta_{\mathrm{univ},2}}|\mathbb{T}_{N}|N^{-97} \ \lesssim \ N^{-95+\beta_{\mathrm{univ},2}}. \label{eq:KPZNLStep2b6}
\end{align}\normalsize\normalsize
The last bound follows from \eqref{eq:KPZNLStep2b5}. Choosing $\beta_{\mathrm{univ},2} \in \R_{>0}$ sufficiently small but still universal gives that with high-probability, the first term from the RHS of \eqref{eq:KPZNLStep2b1} is controlled by the RHS of the proposed estimate in Lemma \ref{lemma:KPZNLStep2b}. By \eqref{eq:KPZNLStep2b1}, the proposed estimate in Lemma \ref{lemma:KPZNLStep2b} follows if we control the second term on the RHS of \eqref{eq:KPZNLStep2b1} with high probability.

To control the second term from the RHS of \eqref{eq:KPZNLStep2b1} with required high probability, we use Lemma \ref{lemma:KPZNLStep2} with $\mathfrak{t}_{1} = \mathfrak{t}_{1,2} = N^{-100}$ and $\mathfrak{t}_{2} = \mathfrak{t}_{\mathrm{av},0} \lesssim N^{-2+\e_{X,2}}$ for $\e_{X,2}\in\R_{>0}$ arbitrarily small but universal. We also recall $|\mathscr{A}^{\mathbf{T},+}\mathscr{C}^{\mathbf{X},-}(\mathfrak{g})|\lesssim N^{-1/6}$. Ultimately, with the required high probability we have the following estimate for the second term on the RHS of \eqref{eq:KPZNLStep2b1}:
\small\begin{align}
\|\bar{\mathbf{H}}_{T,x}^{N}(\Psi^{N,2,\mathfrak{t}_{1,2},\mathfrak{t}_{\mathrm{av},0}}\bar{\mathbf{Z}}^{N})\|_{\mathfrak{t}_{\mathrm{st}};\mathbb{T}_{N}} \ &\lesssim_{\e} \ N^{-\beta_{\mathrm{univ},2}}\langle\bar{\mathbf{Z}}^{N}\rangle_{\mathfrak{t}_{\mathrm{st}};\mathbb{T}_{N}} \ + \ \langle\bar{\mathbf{Z}}^{N}\rangle_{\mathfrak{t}_{\mathrm{st}};\mathbb{T}_{N}} N^{\e}\mathfrak{t}_{\mathrm{av},0}^{1/4}\|\bar{\mathbf{H}}_{T,x}^{N}(N^{\frac12}|\mathscr{A}_{\mathfrak{t}_{1,2}}^{\mathbf{T},+}\mathscr{C}_{N^{\beta_{X}}}^{\mathbf{X},-}(\mathfrak{g})|)\|_{1;\mathbb{T}_{N}} \\ 
&\lesssim \ N^{-\beta_{\mathrm{univ},2}}\langle\bar{\mathbf{Z}}^{N}\rangle_{\mathfrak{t}_{\mathrm{st}};\mathbb{T}_{N}} \ + \ \langle\bar{\mathbf{Z}}^{N}\rangle_{\mathfrak{t}_{\mathrm{st}};\mathbb{T}_{N}} N^{\frac12+\e}N^{-\frac16}\mathfrak{t}_{\mathrm{av},0}^{1/4}. \label{eq:KPZNLStep2b7}
\end{align}\normalsize\normalsize
Elementary power-counting with $\mathfrak{t}_{\mathrm{av},0}\lesssim N^{-2+\e_{X,2}}$ bounds the second term on the RHS of \eqref{eq:KPZNLStep2b1}. This completes the proof.
\end{proof}
\begin{proof}[Proof of \emph{Lemma \ref{lemma:KPZNLStep22}}]
We observe that Lemma \ref{lemma:KPZNLStep22} differs from the earlier result Lemma \ref{lemma:KPZNLStep2} just through the replacement \eqref{eq:KPZNLReplace}. The proof of Lemma \ref{lemma:KPZNLStep22} amounts to verbatim following proof of Lemma \ref{lemma:KPZNLStep2} upon making the replacement \eqref{eq:KPZNLReplace}. In particular, the only important thing is that we replace the time-scale of time-averaging from $\mathfrak{t}_{1}\rightsquigarrow\mathfrak{t}_{2}$. The details of what we time-average are irrelevant. Actually, the replacement of \eqref{eq:KPZNLReplace} in the proof of Lemma \ref{lemma:KPZNLStep2} changes power-counting in the proof of Lemma \ref{lemma:KPZNLStep2} but in ultimately irrelevant ways if instead of the bound $|\mathscr{C}^{\mathbf{X},-}(\mathfrak{g})|\lesssim N^{-\frac16}$, we instead use $|\wt{\mathfrak{g}}^{\mathfrak{l}}|\lesssim1$ in this power-counting.
\end{proof}
\begin{proof}[Proof of \emph{Lemma \ref{lemma:KPZNLStep22b}}]
We follow the proof of Lemma \ref{lemma:KPZNLStep2b} with the minor changes mentioned in the proof of Lemma \ref{lemma:KPZNLStep22}.
\end{proof}
\begin{proof}[Proof of \emph{Lemma \ref{lemma:Step3A}}]
The architecture behind this argument is to relate $\Gamma^{N,m,m+1}$ and $\Gamma^{N,m+1,m+2}$ from the statement of Lemma \ref{lemma:Step3A} via exact identities. We then estimate the errors obtained when relating $\Gamma^{N,m,m+1}$ and $\Gamma^{N,m+1,m+2}$ in these exact identities. Before we start, we invite the reader to refer to Lemma \ref{lemma:Step3A} to recall definitions of $\Gamma$-terms in this proof. First, notation.
\begin{itemize}[leftmargin=*]
\item Take $m,\mathfrak{l},\mathfrak{k},\mathfrak{j}\in\Z_{\geq0}$. For $T\geq0$ let $T^{m}_{\mathfrak{l},\mathfrak{k},\mathfrak{j}}\overset{\bullet}=T+\mathfrak{l}\mathfrak{t}_{\mathrm{av},m+2}+\mathfrak{k}\mathfrak{t}_{\mathrm{av},m+1}+\mathfrak{j}\mathfrak{t}_{\mathrm{av},m}$ be time-shifts at 3 $m$-dependent time-scales. In particular each subscript in this notation indicates shifting at one of the three time-scales $\mathfrak{t}_{\mathrm{av},m+2},\mathfrak{t}_{\mathrm{av},m+1},\mathfrak{t}_{\mathrm{av},m}$ in this order.
\item By definition the term $\Gamma^{N,m,m+1}$ from the LHS of the proposed estimate is an average of $\mathfrak{t}_{\mathrm{av},\mathfrak{m}_{+}}\mathfrak{t}_{\mathrm{av},m+1}^{-1}$-many time-shifted terms $\Gamma^{N,\mathfrak{t}_{\mathrm{av},m}}$ each of which is the average of scale-$\mathfrak{t}_{\mathrm{av},m}$ time-averages equipped with a priori upper bound cutoff of $N^{-\beta_{m}}$. To relate this to $\Gamma^{N,m+1,m+2}$, we will group each of the $\mathfrak{t}_{\mathrm{av},\mathfrak{m}_{+}}\mathfrak{t}_{\mathrm{av},m+1}^{-1}$-many time-shifted terms $\Gamma^{N,\mathfrak{t}_{\mathrm{av},m}}$ into $\mathfrak{t}_{\mathrm{av},\mathfrak{m}_{+}}\mathfrak{t}_{\mathrm{av},m+2}^{-1}$-many groups. We will then relate each of these $\mathfrak{t}_{\mathrm{av},\mathfrak{m}_{+}}\mathfrak{t}_{\mathrm{av},m+2}^{-1}$-many groups to $\Gamma^{N,\mathfrak{t}_{\mathrm{av},m+1}}$-terms. Because the term $\Gamma^{N,m,m+1}$ is the average of these $\mathfrak{t}_{\mathrm{av},\mathfrak{m}_{+}}\mathfrak{t}_{\mathrm{av},m+2}^{-1}$-many $\Gamma^{N,\mathfrak{t}_{\mathrm{av},m+1}}$-terms, making such a relation and estimating error terms will suffice. Precisely, the first step is to compute $\Gamma^{N,m,m+1}$ by decompositions with respect to $\mathfrak{t}_{\mathrm{av},m+2}$-, $\mathfrak{t}_{\mathrm{av},m+1}$-, and $\mathfrak{t}_{\mathrm{av},m}$-time scales. We collect all $\mathfrak{t}_{\mathrm{av},\mathfrak{m}_{+}}\mathfrak{t}_{\mathrm{av},m+1}^{-1}$-many length-$\mathfrak{t}_{\mathrm{av},m+1}$ $\Gamma^{N,\mathfrak{t}_{\mathrm{av},m}}$-terms inside $\Gamma^{N,m,m+1}$ into length-$\mathfrak{t}_{\mathrm{av},m+2}$ sums, each with $\mathfrak{t}_{\mathrm{av},m+2}\mathfrak{t}_{\mathrm{av},m+1}^{-1}$-many summands. Equivalently, we rewrite the sum in $\Gamma^{N,m,m+1}$ with time-scale increments given by multiples of $\mathfrak{t}_{\mathrm{av},m+1}$ by grouping time-scales by multiples of $\mathfrak{t}_{\mathrm{av},m+2}$. For example a sum of 10 terms is a sum of 5 other terms which are each sums of 2 terms with neighboring indices:
\small\begin{align}
\Gamma_{T,x}^{N,m,m+1} \ &= \ \wt{\sum}_{\mathfrak{l}=0}^{\mathfrak{t}_{\mathrm{av},\mathfrak{m}_{+}}\mathfrak{t}_{\mathrm{av},m+2}^{-1}-1}\wt{\sum}_{\mathfrak{k}=0}^{\mathfrak{t}_{\mathrm{av},m+2}\mathfrak{t}_{\mathrm{av},m+1}^{-1}-1} \Gamma_{T+\mathfrak{l}\mathfrak{t}_{\mathrm{av},m+2}+\mathfrak{k}\mathfrak{t}_{\mathrm{av},m+1},x}^{N,\mathfrak{t}_{\mathrm{av},m}} \nonumber \\
&= \ \wt{\sum}_{\mathfrak{l}=0}^{\mathfrak{t}_{\mathrm{av},\mathfrak{m}_{+}}\mathfrak{t}_{\mathrm{av},m+2}^{-1}-1}\wt{\sum}_{\mathfrak{k}=0}^{\mathfrak{t}_{\mathrm{av},m+2}\mathfrak{t}_{\mathrm{av},m+1}^{-1}-1} \Gamma_{T_{\mathfrak{l},\mathfrak{k},0}^{m},x}^{N,\mathfrak{t}_{\mathrm{av},m}}. \label{eq:Step3A0}
\end{align}\normalsize\normalsize
Recall in Corollary \ref{corollary:D1B2A} that $\mathfrak{t}_{\mathrm{av},m}$ are integer multiples of each other. Our goal is now to relate the $\mathfrak{k}$-sum in \eqref{eq:Step3A0} to $\Gamma^{N,\mathfrak{t}_{\mathrm{av},m+1}}$.
\item Building off the previous bullet point, we recall $\Gamma^{N,\mathfrak{t}_{\mathrm{av},m+1}}$ is an average of time-shifted time-averages of $\mathscr{C}^{\mathbf{X},-}(\mathfrak{g})$ each on the time-scale $\mathfrak{t}_{\mathrm{av},m+1}$ with upper bound cutoffs of $N^{-\beta_{m+1}}$. The first step we take towards getting such $\Gamma^{N,\mathfrak{t}_{\mathrm{av},m+1}}$-terms out of the inner summations from the RHS of \eqref{eq:Step3A0} is to first equip the scale-$\mathfrak{t}_{\mathrm{av},m+1}$ inner summations in \eqref{eq:Step3A0} with an upper bound cutoff of $N^{-\beta_{m}}$ for a scale-$\mathfrak{t}_{\mathrm{av},m+1}$ time-average which we later upgrade to $N^{-\beta_{m+1}}$. Precisely, with explanation after,
\small\begin{align}
\Gamma_{T,x}^{N,m,m+1} \ = \ \wt{\sum}_{\mathfrak{l}=0}^{\mathfrak{t}_{\mathrm{av},\mathfrak{m}_{+}}\mathfrak{t}_{\mathrm{av},m+2}^{-1}-1}\wt{\sum}_{\mathfrak{k}=0}^{\mathfrak{t}_{\mathrm{av},m+2}\mathfrak{t}_{\mathrm{av},m+1}^{-1}-1} \Gamma_{\mathfrak{l},\mathfrak{k},1} + \wt{\sum}_{\mathfrak{l}=0}^{\mathfrak{t}_{\mathrm{av},\mathfrak{m}_{+}}\mathfrak{t}_{\mathrm{av},m+2}^{-1}-1}\wt{\sum}_{\mathfrak{k}=0}^{\mathfrak{t}_{\mathrm{av},m+2}\mathfrak{t}_{\mathrm{av},m+1}^{-1}-1} \Gamma_{\mathfrak{l},\mathfrak{k},2} \ &\overset{\bullet}= \ \Phi_{1} + \Phi_{2}. \label{eq:Step3A1}
\end{align}\normalsize\normalsize
Here $\Gamma_{\mathfrak{l},\mathfrak{k},i} \overset{\bullet}= \Gamma_{T_{\mathfrak{l},\mathfrak{k},0}^{m},x}^{N,\mathfrak{t}_{\mathrm{av},m}}\mathbf{1}[\mathscr{E}_{i}]$ with $\mathscr{E}_{2}$ the complement of $\mathscr{E}_{1}$ below and thus $\Gamma_{T_{\mathfrak{l},\mathfrak{k},0}^{m},x}^{N,\mathfrak{t}_{\mathrm{av},m}}=\Gamma_{T_{\mathfrak{l},\mathfrak{k},0}^{m},x}^{N,\mathfrak{t}_{\mathrm{av},m}}\mathbf{1}[\mathscr{E}_{1}]+\Gamma_{T_{\mathfrak{l},\mathfrak{k},0}^{m},x}^{N,\mathfrak{t}_{\mathrm{av},m}}\mathbf{1}[\mathscr{E}_{2}]=\Gamma_{\mathfrak{l},\mathfrak{k},1}+\Gamma_{\mathfrak{l},\mathfrak{k},2}$. In particular, plugging this last identity into the far RHS of \eqref{eq:Step3A0} produces \eqref{eq:Step3A1}. The $\mathscr{E}_{1}$ event is defined as
\small\begin{align}
\mathbf{1}[\mathscr{E}_{1}] \ \overset{\bullet}= \ \mathbf{1}\left({\sup}_{0\leq\mathfrak{t}\leq\mathfrak{t}_{\mathrm{av},m+1}}\mathfrak{t}\cdot\mathfrak{t}_{\mathrm{av},m+1}^{-1}|\mathscr{A}_{\mathfrak{t}}^{\mathbf{T},+}\mathscr{C}_{N^{\beta_{X}}}^{\mathbf{X},-}(\mathfrak{g}_{T_{\mathfrak{l},\mathfrak{k},0}^{m},x})| \leq N^{-\beta_{m}} \right).
\end{align}\normalsize\normalsize
The $\mathscr{E}_{1}$-event gives a cutoff for time-averages on scale $\mathfrak{t}_{\mathrm{av},m+1}$. The $\Phi_{2}$-term from the far RHS of \eqref{eq:Step3A1} will be treated as an error term at the end of this proof. We unfold the $\Phi_{1}$-term by unfolding each $\Gamma_{\mathfrak{l},\mathfrak{k},1}$-summand. This summand is a time-shifted $\Gamma^{N,\mathfrak{t}_{\mathrm{av},m}}$-term with the a priori upper bound cutoff in $\mathscr{E}_{1}$. Each of the $\Gamma^{N,\mathfrak{t}_{\mathrm{av},m}}$-terms is an average of time-shifted time-averages of $\mathscr{C}^{\mathbf{X},-}(\mathfrak{g})$ each with an a priori upper bound cutoff of $N^{-\beta_{m}}$. If we take away these upper bound cutoffs for the time-shifted time-averages inside $\Gamma^{N,\mathfrak{t}_{\mathrm{av},m}}$, then the resulting $\Gamma^{N,\mathfrak{t}_{\mathrm{av},m}}$-term would be the average of time-shifted time-averages of $\mathscr{C}^{\mathbf{X},-}(\mathfrak{g})$ on adjacent time-intervals of length $\mathfrak{t}_{\mathrm{av},m}$ which would glue together into a time-average of $\mathscr{C}^{\mathbf{X},-}(\mathfrak{g})$ on the time-scale $\mathfrak{t}_{\mathrm{av},m+1}$. The indicator function $\mathscr{E}_{1}$ would then grant the resulting scale-$\mathfrak{t}_{\mathrm{av},m+1}$ average an a priori upper bound cutoff of $N^{-\beta_{m}}$. Thus, the next step we will take is to adjust $\Gamma_{\mathfrak{l},\mathfrak{k},1}$ by keeping the $\mathbf{1}[\mathscr{E}_{1}]$-factor but removing the a priori upper bound cutoffs inside the time-averages $\Gamma^{N,\mathfrak{t}_{\mathrm{av},m},\mathfrak{j}}$ defining the $\Gamma^{N,\mathfrak{t}_{\mathrm{av},m}}$-factor in $\Gamma_{\mathfrak{l},\mathfrak{k},1}$. Precisely, recalling $\Gamma_{\mathfrak{l},\mathfrak{k},1}$ defined immediately after \eqref{eq:Step3A1} and looking at definitions of $\Gamma^{N,\mathfrak{t}_{\mathrm{av},m}}$ and $\Gamma^{N,\mathfrak{t}_{\mathrm{av},m},\mathfrak{j}}$ in Lemma \ref{lemma:Step3A}, we get, with explanation to come, a scale-$\mathfrak{t}_{\mathrm{av},m}$/$\mathfrak{t}_{\mathrm{av},m+1}$ expansion below in which we treat each of the scale-$\mathfrak{t}_{\mathrm{av},m}$ $\Gamma^{N,\mathfrak{t}_{\mathrm{av},m},\mathfrak{j}}$-terms in $\Gamma^{N,\mathfrak{t}_{\mathrm{av},m}}$ by removing their indicator functions to get an honest scale-$\mathfrak{t}_{\mathrm{av},m}$ time-average. This is the first term on the far RHS below. We get an error of the same scale-$\mathfrak{t}_{\mathrm{av},m}$ average although with lower bound cutoff complement to the upper bound event in $\Gamma^{N,\mathfrak{t}_{\mathrm{av},m},\mathfrak{j}}$. This is the second term on the far RHS below:
\small\begin{align}
\Gamma_{\mathfrak{l},\mathfrak{k},1} \ = \ \mathbf{1}[\mathscr{E}_{1}]\Gamma_{T_{\mathfrak{l},\mathfrak{k},0}^{m},x}^{N,\mathfrak{t}_{\mathrm{av},m}} \ = \ \wt{\sum}_{\mathfrak{j}=0}^{\mathfrak{t}_{\mathrm{av},m+1}\mathfrak{t}_{\mathrm{av},m}-1}\mathbf{1}[\mathscr{E}_{1}]\Gamma_{T_{\mathfrak{l},\mathfrak{k},0}^{m},x}^{N,\mathfrak{t}_{\mathrm{av},m},\mathfrak{j}} \ &= \ \wt{\sum}_{\mathfrak{j}=0}^{\mathfrak{t}_{\mathrm{av},m+1}\mathfrak{t}_{\mathrm{av},m}^{-1}-1}\Gamma_{\mathfrak{l},\mathfrak{k},\mathfrak{j},1,1} \ + \ \wt{\sum}_{\mathfrak{j}=0}^{\mathfrak{t}_{\mathrm{av},m+1}\mathfrak{t}_{\mathrm{av},m}^{-1}-1} \Gamma_{\mathfrak{l},\mathfrak{k},\mathfrak{j},1,2}. \label{eq:Step3A2}
\end{align}\normalsize\normalsize
The terms in the two-term decomposition within the RHS of \eqref{eq:Step3A2} are defined below with an event defined afterwards. The first term is a scale-$\mathfrak{t}_{\mathrm{av},m}$ average over a window contained inside the length-$\mathfrak{t}_{\mathrm{av},m+1}$ block starting at the time $T_{\mathfrak{l},\mathfrak{k},0}^{m}$ associated to $\Gamma_{\mathfrak{l},\mathfrak{k},1}$, equipped with the event $\mathscr{E}_{1}$ which is defined at scale $\mathfrak{t}_{\mathrm{av},m+1}$/depends only on $\mathfrak{l},\mathfrak{k}$ variables. The second term below is the same but now with an additional indicator function $\mathbf{1}[\mathscr{F}]$ obtained by removing the indicator function in $\Gamma^{N,\mathfrak{t}_{\mathrm{av},m},\mathfrak{j}}$-terms:
\small\begin{align}
\Gamma_{\mathfrak{l},\mathfrak{k},\mathfrak{j},1,1} \ &\overset{\bullet}= \ \mathscr{A}_{\mathfrak{t}_{\mathrm{av},m}}^{\mathbf{T},+}\mathscr{C}_{N^{\beta_{X}}}^{\mathbf{X},-}(\mathfrak{g}_{T_{\mathfrak{l},\mathfrak{k},\mathfrak{j}}^{m},x}) \cdot \mathbf{1}[\mathscr{E}_{1}] \quad \mathrm{and} \quad \Gamma_{\mathfrak{l},\mathfrak{k},\mathfrak{j},1,2} \ \overset{\bullet}= \ -\mathscr{A}_{\mathfrak{t}_{\mathrm{av},m}}^{\mathbf{T},+}\mathscr{C}_{N^{\beta_{X}}}^{\mathbf{X},-}(\mathfrak{g}_{T_{\mathfrak{l},\mathfrak{k},\mathfrak{j}}^{m},x}) \cdot \mathbf{1}[\mathscr{E}_{1}] \mathbf{1}[\mathscr{F}]. \nonumber
\end{align}\normalsize\normalsize
We introduced the following complement of the event in the indicator function defining $\Gamma^{N,\mathfrak{t}_{\mathrm{av},m},\mathfrak{j}}$. Indeed, it appears in \eqref{eq:Step3A2} by removing said indicator function/observing the indicator function in $\Gamma^{N,\mathfrak{t}_{\mathrm{av},m},\mathfrak{j}}$ is equal to $1-\mathbf{1}[\mathscr{F}]$: 
\small\begin{align}
\mathbf{1}[\mathscr{F}] \ &\overset{\bullet}= \ \mathbf{1}\left({\sup}_{0\leq\mathfrak{t}\leq\mathfrak{t}_{\mathrm{av},m}}\mathfrak{t}\cdot\mathfrak{t}_{\mathrm{av},m}^{-1}|\mathscr{A}_{\mathfrak{t}}^{\mathbf{T},+}\mathscr{C}_{N^{\beta_{X}}}^{\mathbf{X},-}(\mathfrak{g}_{T_{\mathfrak{l},\mathfrak{k},\mathfrak{j}}^{m},x})| > N^{-\beta_{m}}\right).
\end{align}\normalsize\normalsize
Thus, averaging the previous expansion \eqref{eq:Step3A2} over $\mathfrak{l},\mathfrak{k}$-indices gives the decomposition
\small\begin{align}
\Phi_{1} \ &= \ \wt{\sum}_{\mathfrak{l}=0}^{\mathfrak{t}_{\mathrm{av},\mathfrak{m}_{+}}\mathfrak{t}_{\mathrm{av},m+2}^{-1}-1}\wt{\sum}_{\mathfrak{k}=0}^{\mathfrak{t}_{\mathrm{av},m+2}\mathfrak{t}_{\mathrm{av},m+1}^{-1}-1}\wt{\sum}_{\mathfrak{j}=0}^{\mathfrak{t}_{\mathrm{av},m+1}\mathfrak{t}_{\mathrm{av},m}^{-1}-1} \Gamma_{\mathfrak{l},\mathfrak{k},\mathfrak{j},1,1} \nonumber \\
&\quad + \ \wt{\sum}_{\mathfrak{l}=0}^{\mathfrak{t}_{\mathrm{av},\mathfrak{m}_{+}}\mathfrak{t}_{\mathrm{av},m+2}^{-1}-1}\wt{\sum}_{\mathfrak{k}=0}^{\mathfrak{t}_{\mathrm{av},m+2}\mathfrak{t}_{\mathrm{av},m+1}^{-1}-1}\wt{\sum}_{\mathfrak{j}=0}^{\mathfrak{t}_{\mathrm{av},m+1}\mathfrak{t}_{\mathrm{av},m}^{-1}-1} \Gamma_{\mathfrak{l},\mathfrak{k},\mathfrak{j},1,2}. \label{eq:Step3A3}
\end{align}\normalsize\normalsize
We denote the first sum on the RHS of \eqref{eq:Step3A3} by $\Phi_{3}$ and the second by $\Phi_{4}$. We treat $\Phi_{4}$ as an error like $\Phi_{2}$. We now study $\Phi_{3}$. We first take the sum over the index $\mathfrak{j}$. We observe the event $\mathscr{E}_{1}$ defined after \eqref{eq:Step3A1} is independent of this sum variable since it is defined by a constraint at scale $\mathfrak{t}_{\mathrm{av},m+1}$/with respect to $\mathfrak{l},\mathfrak{k}$-indices. By definition of $\Gamma_{\mathfrak{l},\mathfrak{k},\mathfrak{j},1,1}$ given after \eqref{eq:Step3A2}, we sum over the $\mathfrak{j}$-variable/the finest $\mathfrak{t}_{\mathrm{av},m}$ time-scale in $\Phi_{3}$ and glue such scale-$\mathfrak{t}_{\mathrm{av},m}$ terms into scale-$\mathfrak{t}_{\mathrm{av},m+1}$/$\mathfrak{l},\mathfrak{k}$-terms:
\small\begin{align}
\wt{\sum}_{\mathfrak{j}=0}^{\mathfrak{t}_{\mathrm{av},m+1}\mathfrak{t}_{\mathrm{av},m}^{-1}-1} \Gamma_{\mathfrak{l},\mathfrak{k},\mathfrak{j},1,1} \ &= \ \mathbf{1}[\mathscr{E}_{1}] \cdot \wt{\sum}_{\mathfrak{j}=0}^{\mathfrak{t}_{\mathrm{av},m+1}\mathfrak{t}_{\mathrm{av},m}^{-1}-1}\mathscr{A}_{\mathfrak{t}_{\mathrm{av},m}}^{\mathbf{T},+}\mathscr{C}_{N^{\beta_{X}}}^{\mathbf{X},-}(\mathfrak{g}_{T_{\mathfrak{l},\mathfrak{k},\mathfrak{j}}^{m},x}) \ = \ \mathbf{1}[\mathscr{E}_{1}] \cdot \mathscr{A}_{\mathfrak{t}_{\mathrm{av},m+1}}^{\mathbf{T},+}\mathscr{C}_{N^{\beta_{X}}}^{\mathbf{X},-}(\mathfrak{g}_{T_{\mathfrak{l},\mathfrak{k},0}^{m},x}). \label{eq:Step3A4}
\end{align}\normalsize\normalsize
Above, \eqref{eq:Step3A4} follows by observing the $\mathfrak{t}_{\mathrm{av},m+1}$-scale time average is an average of the preceding time-shifted time averages at the smaller $\mathfrak{t}_{\mathrm{av},m}$-scale. We emphasize \eqref{eq:Step3A2} and \eqref{eq:Step3A4} are a precise implementation of the statement previous to \eqref{eq:Step3A2} concerning adjustment of $\Gamma_{\mathfrak{l},\mathfrak{k},1}$ to turn it into a time-average of $\mathscr{C}^{\mathbf{X},-}(\mathfrak{g})$ on time-scale $\mathfrak{t}_{\mathrm{av},m+1}$ with an a priori upper bound cutoff of $N^{-\beta_{m}}$. By definition of $\Phi_{3}$ as the first triple sum on the RHS of \eqref{eq:Step3A3}, from \eqref{eq:Step3A4} we get the following representation of $\Phi_{3}$ as an average of time-shifted time-averages of $\mathscr{C}^{\mathbf{X},-}(\mathfrak{g})$ on scales $\mathfrak{t}_{\mathrm{av},m+1}$ with upper bound cutoff $N^{-\beta_{m}}$:
\small\begin{align}
\Phi_{3} \ &= \ \wt{\sum}_{\mathfrak{l}=0}^{\mathfrak{t}_{\mathrm{av},\mathfrak{m}_{+}}\mathfrak{t}_{\mathrm{av},m+2}^{-1}-1}\wt{\sum}_{\mathfrak{k}=0}^{\mathfrak{t}_{\mathrm{av},m+2}\mathfrak{t}_{\mathrm{av},m+1}^{-1}-1}\mathbf{1}[\mathscr{E}_{1}] \cdot \mathscr{A}_{\mathfrak{t}_{\mathrm{av},m+1}}^{\mathbf{T},+}\mathscr{C}_{N^{\beta_{X}}}^{\mathbf{X},-}(\mathfrak{g}_{T_{\mathfrak{l},\mathfrak{k},0}^{m},x}). \label{eq:Step3A4.5}
\end{align}\normalsize\normalsize
\item We build off the previous bullet point starting with \eqref{eq:Step3A4.5}. Let us observe that $\Phi_{3}$ is \emph{almost} equal to $\Gamma^{N,m+1,m+2}$. This would give us an explicit formula to relate $\Gamma^{N,m,m+1}$ and $\Gamma^{N,m+1,m+2}$ because $\Phi_{3}$ is related to $\Phi_{1}$ with an error term $\Phi_{4}$ via \eqref{eq:Step3A3}, and $\Phi_{1}$ is related to $\Gamma^{N,m,m+1}$ through an error term $\Phi_{2}$ via \eqref{eq:Step3A1}. The ``almost"-ness of the equality between $\Phi_{3}$ and $\Gamma^{N,m+1,m+2}$ is because the term in \eqref{eq:Step3A4} or equivalently the summand in \eqref{eq:Step3A4.5} is \emph{almost} equal to the time-average of $\mathscr{C}^{\mathbf{X},-}$ on time-scale $\mathfrak{t}_{\mathrm{av},m+1}$ with upper bound cutoff of $N^{-\beta_{m+1}}$, and such a time-average with $N^{-\beta_{m+1}}$-cutoff is equal to $\Gamma^{N,\mathfrak{t}_{\mathrm{av},m+1},\mathfrak{k}}$ by definition from the statement of Lemma \ref{lemma:Step3A}, and similarly by definition $\Gamma^{N,m+1,m+2}$ averages these $\Gamma^{N,\mathfrak{t}_{\mathrm{av},m+1},\mathfrak{k}}$-terms. However, the cutoff in $\mathscr{E}_{1}$ is the upper bound of $N^{-\beta_{m}}$, not $N^{-\beta_{m+1}}$. Thus, the next step we take is to upgrade the cutoff $N^{-\beta_{m}}$ in $\mathscr{E}_{1}$ to $N^{-\beta_{m+1}}$. Precisely,
\small\begin{align}
\mathbf{1}[\mathscr{E}_{1}] \cdot \mathscr{A}_{\mathfrak{t}_{\mathrm{av},m+1}}^{\mathbf{T},+}\mathscr{C}_{N^{\beta_{X}}}^{\mathbf{X},-}(\mathfrak{g}_{T_{\mathfrak{l},\mathfrak{k},0}^{m},x}) \ &= \ \left(\mathbf{1}[\mathscr{G}_{1}] + \mathbf{1}[\mathscr{E}_{1}]\mathbf{1}[\mathscr{G}_{2}]\right)\cdot\mathscr{A}_{\mathfrak{t}_{\mathrm{av},m+1}}^{\mathbf{T},+}\mathscr{C}_{N^{\beta_{X}}}^{\mathbf{X},-}(\mathfrak{g}_{T_{\mathfrak{l},\mathfrak{k},0}^{m},x}) \label{eq:Step3A5}
\end{align}\normalsize\normalsize
where $\mathscr{G}_{1}$ is the following upgrade of $\mathscr{E}_{1}$ that is certainly contained in $\mathscr{E}_{1}$, so $\mathbf{1}[\mathscr{G}_{1}]\mathbf{1}[\mathscr{E}_{1}]=\mathbf{1}[\mathscr{G}_{1}]$, and $\mathscr{G}_{2}$ is its complement:
\small\begin{align}
\mathbf{1}[\mathscr{G}_{1}] \ &\overset{\bullet}= \ \mathbf{1}\left({\sup}_{0\leq\mathfrak{t}\leq\mathfrak{t}_{\mathrm{av},m+1}}\mathfrak{t}\cdot\mathfrak{t}_{\mathrm{av},m+1}^{-1}|\mathscr{A}_{\mathfrak{t}}^{\mathbf{T},+}\mathscr{C}_{N^{\beta_{X}}}^{\mathbf{X},-}(\mathfrak{g}_{T_{\mathfrak{l},\mathfrak{k},0}^{m},x})| \leq N^{-\beta_{m+1}}\right).
\end{align}\normalsize\normalsize
By \eqref{eq:Step3A4.5}, \eqref{eq:Step3A5}, and definition of $\Gamma^{N,\mathfrak{t}_{\mathrm{av},m+1}}$ and $\Gamma^{N,\mathfrak{t}_{\mathrm{av},m+1},\mathfrak{k}}$ in the statement of Lemma \ref{lemma:Step3A}, we average over $\mathfrak{l},\mathfrak{k}$-indices:
\small\begin{align}
\Phi_{3} \ &= \ \wt{\sum}_{\mathfrak{l}=0}^{\mathfrak{t}_{\mathrm{av},\mathfrak{m}_{+}}\mathfrak{t}_{\mathrm{av},m+2}^{-1}-1}\wt{\sum}_{\mathfrak{k}=0}^{\mathfrak{t}_{\mathrm{av},m+2}\mathfrak{t}_{\mathrm{av},m+1}^{-1}-1}\mathbf{1}[\mathscr{E}_{1}] \cdot \mathscr{A}_{\mathfrak{t}_{\mathrm{av},m+1}}^{\mathbf{T},+}\mathscr{C}_{N^{\beta_{X}}}^{\mathbf{X},-}(\mathfrak{g}_{T_{\mathfrak{l},\mathfrak{k},0}^{m},x}) \\
&= \ \wt{\sum}_{\mathfrak{l}=0}^{\mathfrak{t}_{\mathrm{av},\mathfrak{m}_{+}}\mathfrak{t}_{\mathrm{av},m+2}^{-1}-1}\wt{\sum}_{\mathfrak{k}=0}^{\mathfrak{t}_{\mathrm{av},m+2}\mathfrak{t}_{\mathrm{av},m+1}^{-1}-1} \mathbf{1}[\mathscr{G}_{1}] \cdot \mathscr{A}_{\mathfrak{t}_{\mathrm{av},m+1}}^{\mathbf{T},+}\mathscr{C}_{N^{\beta_{X}}}^{\mathbf{X},-}(\mathfrak{g}_{T_{\mathfrak{l},\mathfrak{k},0}^{m},x}) \ + \ \Phi_{5} \\
&= \ \wt{\sum}_{\mathfrak{l}=0}^{\mathfrak{t}_{\mathrm{av},\mathfrak{m}_{+}}\mathfrak{t}_{\mathrm{av},m+2}^{-1}-1}\wt{\sum}_{\mathfrak{k}=0}^{\mathfrak{t}_{\mathrm{av},m+2}\mathfrak{t}_{\mathrm{av},m+1}^{-1}-1} \Gamma_{T+\mathfrak{l}\mathfrak{t}_{\mathrm{av},m+2},x}^{N,\mathfrak{t}_{\mathrm{av},m+1},\mathfrak{k}} \ + \ \Phi_{5} \\
&= \ \wt{\sum}_{\mathfrak{l}=0}^{\mathfrak{t}_{\mathrm{av},\mathfrak{m}_{+}}\mathfrak{t}_{\mathrm{av},m+2}^{-1}-1}\Gamma^{N,\mathfrak{t}_{\mathrm{av},m+1}}_{T+\mathfrak{l}\mathfrak{t}_{\mathrm{av},m+2},x} \ + \ \Phi_{5} \ = \ \Gamma_{T,x}^{N,m+1,m+2} \ + \ \Phi_{5}. \label{eq:Step3A6}
\end{align}\normalsize\normalsize
The $\Phi_{5}$-term introduced in the derivation of \eqref{eq:Step3A6} above is defined below:
\small\begin{align}
\Phi_{5} \ &\overset{\bullet}= \ \wt{\sum}_{\mathfrak{l}=0}^{\mathfrak{t}_{\mathrm{av},\mathfrak{m}_{+}}\mathfrak{t}_{\mathrm{av},m+2}^{-1}-1}\wt{\sum}_{\mathfrak{k}=0}^{\mathfrak{t}_{\mathrm{av},m+2}\mathfrak{t}_{\mathrm{av},m+1}^{-1}-1}\mathbf{1}[\mathscr{E}_{1}]\mathbf{1}[\mathscr{G}_{2}]\cdot\mathscr{A}_{\mathfrak{t}_{\mathrm{av},m+1}}^{\mathbf{T},+}\mathscr{C}_{N^{\beta_{X}}}^{\mathbf{X},-}(\mathfrak{g}_{T_{\mathfrak{l},\mathfrak{k},0}^{m},x}).
\end{align}\normalsize\normalsize
We repeat that the first two identities in the calculation to get \eqref{eq:Step3A6} follow by \eqref{eq:Step3A4.5}, \eqref{eq:Step3A5}, respectively. The rest follow by definitions in the statement of the current Lemma \ref{lemma:Step3A}. We explain this as follows in case of interest/clarity. In the following explanation all of the references to the definitions of $\Gamma$-terms can be found in the statement of Lemma \ref{lemma:Step3A}. In the calculation to get \eqref{eq:Step3A6}, the third identity follows from observing the summand in the second line in that calculation is a time-shifted time-average of $\mathscr{C}^{\mathbf{X},-}(\mathfrak{g})$ on time-scale $\mathfrak{t}_{\mathrm{av},m+1}$ with an a priori upper bound cutoff of $N^{-\beta_{m+1}}$ in $\mathscr{G}_{1}$, and this is precisely the $\Gamma^{N,\mathfrak{t}_{\mathrm{av},m+1},\mathfrak{k}}$-type term in the statement of the current Lemma \ref{lemma:Step3A}. The fourth line follows by recalling $\Gamma^{N,\mathfrak{t}_{\mathrm{av},m+1}}$ collects-via-averaging $\mathfrak{t}_{\mathrm{av},m+2}\mathfrak{t}_{\mathrm{av},m+1}^{-1}$-many time-shifted versions of $\Gamma^{N,\mathfrak{t}_{\mathrm{av},m+1},\mathfrak{k}}$-terms by definition as well. The fifth identity in \eqref{eq:Step3A6} follows by similarly recalling $\Gamma^{N,m+1,m+2}$ collects $\mathfrak{t}_{\mathrm{av},\mathfrak{m}_{+}}\mathfrak{t}_{\mathrm{av},m+2}^{-1}$-many time-shifted versions of $\Gamma^{N,\mathfrak{t}_{\mathrm{av},m+1}}$.
\end{itemize}
By \eqref{eq:Step3A1}, \eqref{eq:Step3A3}, and \eqref{eq:Step3A6}, we have a transfer-of-scales with errors $\Phi_{2},\Phi_{4},\Phi_{5}$ we are left to bound after multiplying by $N^{1/2}$:
\small\begin{align}
\Gamma_{T,x}^{N,m,m+1} \ &= \ \Gamma_{T,x}^{N,m+1,m+2} + \Phi_{2} + \Phi_{4} + \Phi_{5}. \label{eq:Step3A7}
\end{align}\normalsize\normalsize
Each of the error terms $\Phi_{2},\Phi_{4},\Phi_{5}$, by their respective definitions, come from manipulating time-averages of $\mathscr{C}^{\mathbf{X},-}(\mathfrak{g})$ on time-scales $\mathfrak{t}_{\mathrm{av},m}$ or $\mathfrak{t}_{\mathrm{av},m+1}$ by putting on indicator functions which all give either an upper bound cutoff or a lower bound cutoff for said time-averages. Thus, as we make precise shortly, these error terms will all eventually be treated via Corollary \ref{corollary:D1B2A}.

We now move to analysis. We first define $\Psi_{\mathfrak{u}} \overset{\bullet}= \bar{\mathbf{H}}^{N}(\Phi_{\mathfrak{u}}\bar{\mathbf{Z}}^{N})$ and observe, via the proof of the Markov inequality, the bound
\small\begin{align}
\mathbf{P}\left(\|\bar{\mathbf{Z}}^{N}\|_{\mathfrak{t}_{\mathrm{st}};\mathbb{T}_{N}}^{-1}{\sum}_{\mathfrak{u}=2,4,5}\|\Psi_{\mathfrak{u}}\|_{\mathfrak{t}_{\mathrm{st}};\mathbb{T}_{N}} \gtrsim N^{-\frac12-\beta_{\mathrm{univ},2}}\right) \ &\lesssim \ N^{\frac12+\beta_{\mathrm{univ},2}}{\sum}_{\mathfrak{u} = 2,4,5} \E[\|\bar{\mathbf{Z}}^{N}\|_{\mathfrak{t}_{\mathrm{st}};\mathbb{T}_{N}}^{-1}\|\Psi_{\mathfrak{u}}\|_{\mathfrak{t}_{\mathrm{st}};\mathbb{T}_{N}}]. \label{eq:Step3A8}
\end{align}\normalsize\normalsize
Given \eqref{eq:Step3A7}, we treat each of the three expectations on the RHS of \eqref{eq:Step3A8}. Let us first reemphasize that all three $\Psi_{\mathfrak{u}}$-terms will be turned into space-time averages of pseudo-gradients with both upper and lower bound cutoffs. We then use Corollary \ref{corollary:D1B2A}.
\begin{itemize}[leftmargin=*]
\item We first treat $\mathfrak{u}=2$. Recalling $\Psi_{2} = \bar{\mathbf{H}}^{N}(\Phi_{2}\cdot\bar{\mathbf{Z}}^{N})$ and $\Phi_{2}$ in \eqref{eq:Step3A1}, via the triangle inequality we get the deterministic bound
\small\begin{align}
|\Psi_{2}|  \ &\leq \ \wt{\sum}_{\mathfrak{l}=0}^{\mathfrak{t}_{\mathrm{av},\mathfrak{m}_{+}}\mathfrak{t}_{\mathrm{av},m+2}^{-1}-1}\wt{\sum}_{\mathfrak{k}=0}^{\mathfrak{t}_{\mathrm{av},m+2}\mathfrak{t}_{\mathrm{av},m+1}^{-1}-1} \bar{\mathbf{H}}_{T,x}^{N}(|\Gamma_{\mathfrak{l},\mathfrak{k},2}|) \cdot \|\bar{\mathbf{Z}}^{N}\|_{\mathfrak{t}_{\mathrm{st}};\mathbb{T}_{N}}. \label{eq:Step3A9}
\end{align}\normalsize\normalsize
Thus, from \eqref{eq:Step3A9} and pushing all $\|\|_{\mathfrak{t}_{\mathrm{st}};\mathbb{T}_{N}}$-norms to $\|\|_{1;\mathbb{T}_{N}}$, we get the following in which we bound averages by suprema:
\small\begin{align}
\E[\|\bar{\mathbf{Z}}^{N}\|_{\mathfrak{t}_{\mathrm{st}};\mathbb{T}_{N}}^{-1}\|\Psi_{2}\|_{\mathfrak{t}_{\mathrm{st}};\mathbb{T}_{N}}] \ &\leq \ \wt{\sum}_{\mathfrak{l}=0}^{\mathfrak{t}_{\mathrm{av},\mathfrak{m}_{+}}\mathfrak{t}_{\mathrm{av},m+2}^{-1}-1}\wt{\sum}_{\mathfrak{k}=0}^{\mathfrak{t}_{\mathrm{av},m+2}\mathfrak{t}_{\mathrm{av},m+1}^{-1}-1} \E \|\bar{\mathbf{H}}_{T,x}^{N}(|\Gamma_{\mathfrak{l},\mathfrak{k},2}|)\|_{\mathfrak{t}_{\mathrm{st}};\mathbb{T}_{N}} \\
&\leq \ {\sup}_{\mathfrak{l}=0,\ldots,\mathfrak{t}_{\mathrm{av},\mathfrak{m}_{+}}\mathfrak{t}_{\mathrm{av},m+2}^{-1}-1}{\sup}_{\mathfrak{k}=0,\ldots,\mathfrak{t}_{\mathrm{av},m+2}\mathfrak{t}_{\mathrm{av},m+1}^{-1}-1}\E \|\bar{\mathbf{H}}_{T,x}^{N}(|\Gamma_{\mathfrak{l},\mathfrak{k},2}|)\|_{1;\mathbb{T}_{N}}. \label{eq:Step3A10}
\end{align}\normalsize\normalsize
We control these expectations uniformly over $\mathfrak{l},\mathfrak{k}$. We first make the following observation regarding the event $\mathscr{E}_{2}$ which we defined earlier as the complement of the $\mathscr{E}_{1}$-event introduced after \eqref{eq:Step3A1}. This effectively takes an a priori lower bound for a scale-$\mathfrak{t}_{\mathrm{av},m+1}$ time-average from $\mathscr{E}_{2}$ and provides a priori lower bounds for \emph{at least one} of its scale-$\mathfrak{t}_{\mathrm{av},m}$ ``pieces":
\small\begin{align}
\mathbf{1}[\mathscr{E}_{2}] \ \leq \ {\sum}_{\mathfrak{n}=0}^{\mathfrak{t}_{\mathrm{av},m+1}\mathfrak{t}_{\mathrm{av},m}^{-1}-1} \mathbf{1}\left({\sup}_{0\leq\mathfrak{t}\leq\mathfrak{t}_{\mathrm{av},m}}\mathfrak{t}\cdot\mathfrak{t}_{\mathrm{av},m}^{-1}|\mathscr{A}_{\mathfrak{t}}^{\mathbf{T},+}\mathscr{C}_{N^{\beta_{X}}}^{\mathbf{X},-}(\mathfrak{g}_{T_{\mathfrak{l},\mathfrak{k},\mathfrak{n}}^{m},x})| \geq N^{-\beta_{m}} \right) \ &\overset{\bullet}= \ {\sum}_{\mathfrak{n}=0}^{\mathfrak{t}_{\mathrm{av},m+1},\mathfrak{t}_{\mathrm{av},m}^{-1}-1} \mathbf{1}[\mathscr{E}_{2,\mathfrak{n},T}^{\mathfrak{l},\mathfrak{k}}]. \label{eq:Step3A11}
\end{align}\normalsize\normalsize
Indeed, observe that the dynamic-average $\mathscr{A}^{\mathbf{T},+}\mathscr{C}^{\mathbf{X},-}$ with respect to time-scales $\mathfrak{t} \leq \mathfrak{t}_{\mathrm{av},m+1}$ is an average of $\mathfrak{t}_{\mathrm{av},m+1}\mathfrak{t}_{\mathrm{av},m}^{-1}$-many suitably time-shifted time-averages $\mathscr{A}^{\mathbf{T},+}\mathscr{C}^{\mathbf{X},-}$ that are each defined with respect to time-scales $\mathfrak{t}\leq\mathfrak{t}_{\mathrm{av},m}$. Thus, if the former is bigger than $N^{-\beta_{m}}$, at least one of the pieces that it is averaging over must be bigger than $N^{-\beta_{m}}$ as well. The summation over $\mathfrak{n}=0,\ldots,\mathfrak{t}_{\mathrm{av},m+1}\mathfrak{t}_{\mathrm{av},m}^{-1}-1$ is a union bound that accounts for which of the aforementioned time-scale $\mathfrak{t}_{\mathrm{av},m}\in\R_{>0}$ pieces does the job, of which there certainly may be multiple. By the definition of $\Gamma_{\mathfrak{l},\mathfrak{k},2}$ in \eqref{eq:Step3A1}, from \eqref{eq:Step3A11} we trade in the $\mathbf{1}[\mathscr{E}_{2}]$-factor in $\Gamma_{\mathfrak{l},\mathfrak{k},2}$ for the sum of the $\mathbf{1}[\mathscr{E}_{2,\mathfrak{n},T}^{\mathfrak{l},\mathfrak{k}}]$-indicator functions which we then rewrite as an average of these latter indicator functions upon inserting the normalization factor back. Using this and the triangle inequality, we get the following upper bound which controls $\Gamma_{\mathfrak{l},\mathfrak{k},2}$ by controlling the $\mathbf{1}[\mathscr{E}_{2}]$-factor therein by the sum of the events $\mathbf{1}[\mathscr{E}_{2,\mathfrak{n},T}^{\mathfrak{l},\mathfrak{k}}]$:
\small\begin{align}
\E \|\bar{\mathbf{H}}_{T,x}^{N}(|\Gamma_{\mathfrak{l},\mathfrak{k},2}|)\|_{1;\mathbb{T}_{N}} \ &\leq \ \mathfrak{t}_{\mathrm{av},m+1}\mathfrak{t}_{\mathrm{av},m}^{-1}\wt{\sum}_{\mathfrak{n}=0}^{\mathfrak{t}_{\mathrm{av},m+1}\mathfrak{t}_{\mathrm{av},m}^{-1}-1} \E\|\bar{\mathbf{H}}_{T,x}^{N}(|\Gamma_{S_{\mathfrak{l},\mathfrak{k},0}^{m},y}^{N,\mathfrak{t}_{\mathrm{av},m}}| \cdot \mathbf{1}[\mathscr{E}_{2,\mathfrak{n},S}^{\mathfrak{l},\mathfrak{k}}])\|_{1;\mathbb{T}_{N}}. \label{eq:Step3A12}
\end{align}\normalsize\normalsize
Observe that $\mathfrak{t}_{\mathrm{av},m+1}\mathfrak{t}_{\mathrm{av},m}^{-1} \lesssim N^{\e}$ with $\e \in \R_{>0}$ arbitrarily small but universal as any of these time-scales with adjacent indices are constructed as arbitrarily small but universal powers of $N\in\Z_{>0}$ times each other. This may be verified by an elementary calculation with the definition of these times in Corollary \ref{corollary:D1B2A}. Recalling $\Gamma^{N,\mathfrak{t}_{\mathrm{av},m}}$ as an average over $\Gamma^{N,\mathfrak{t}_{\mathrm{av},m},\mathfrak{j}}$, \eqref{eq:Step3A12} gives
\small\begin{align}
\E \|\bar{\mathbf{H}}_{T,x}^{N}(|\Gamma_{\mathfrak{l},\mathfrak{k},2}|)\|_{1;\mathbb{T}_{N}} \ &\lesssim \ N^{\e} {\sup}_{\mathfrak{n}=0,\ldots,\mathfrak{t}_{\mathrm{av},m+1}\mathfrak{t}_{\mathrm{av},m}^{-1}-1}{\sup}_{\mathfrak{j}=0,\ldots,\mathfrak{t}_{\mathrm{av},m+1}\mathfrak{t}_{\mathrm{av},m}^{-1}-1} \E\|\bar{\mathbf{H}}_{T,x}^{N}(|\Gamma_{S_{\mathfrak{l},\mathfrak{k},0}^{m},y}^{N,\mathfrak{t}_{\mathrm{av},m},\mathfrak{j}}|\cdot\mathbf{1}[\mathscr{E}_{2,\mathfrak{n},S}^{\mathfrak{l},\mathfrak{k}}])\|_{1;\mathbb{T}_{N}}. \label{eq:Step3A13}
\end{align}\normalsize\normalsize
In this final expectation, we first observe that $\Gamma^{N,\mathfrak{t}_{\mathrm{av},m},\mathfrak{j}}$-term associated to indices $\mathfrak{l},\mathfrak{k},\mathfrak{j}$ is the time-average $\mathscr{A}^{\mathbf{T},+}\mathscr{C}^{\mathbf{X},-}$ evaluated at time $S+\mathfrak{l}\mathfrak{t}_{\mathrm{av},m+2}+\mathfrak{k}\mathfrak{t}_{\mathrm{av},m+1}+\mathfrak{j}\mathfrak{t}_{\mathrm{av},m}$ by definition of $\Gamma^{N,\mathfrak{t}_{\mathrm{av},m},\mathfrak{j}}$ in the statement of Lemma \ref{lemma:Step3A}. It also carries the upper bound cutoff of $N^{-\beta_{m}}$ which therefore induces an upper bound cutoff of $N^{-\beta_{m-1}}$ for trivial reasons. Observe now that courtesy of the $\mathscr{E}^{\mathfrak{l},\mathfrak{k}}_{2,\mathfrak{n},S}$-event hitting this $\Gamma^{N,\mathfrak{t}_{\mathrm{av},m},\mathfrak{j}}$-term it also carries a lower bound cutoff $N^{-\beta_{m}}$ for the same time-average but now evaluated at the time $S+\mathfrak{l}\mathfrak{t}_{\mathrm{av},m+2}+\mathfrak{k}\mathfrak{t}_{\mathrm{av},m+1}+\mathfrak{n}\mathfrak{t}_{\mathrm{av},m}$. In particular, observe these times are separated by $|\mathfrak{j}-\mathfrak{n}|\mathfrak{t}_{\mathrm{av},m} \lesssim \mathfrak{t}_{\mathrm{av},m+1}\lesssim N^{\e}\mathfrak{t}_{\mathrm{av},m}$, because $\mathfrak{j},\mathfrak{n} \in \Z_{\geq 0}$ satisfy $\mathfrak{j},\mathfrak{n}\leq \mathfrak{t}_{\mathrm{av},m+1}\mathfrak{t}_{\mathrm{av},m}^{-1}-1$. We may thus use Corollary \ref{corollary:D1B2A} for $\mathfrak{t}_{\mathfrak{s}} = \mathfrak{l}\mathfrak{t}_{\mathrm{av},m+2}+\mathfrak{k}\mathfrak{t}_{\mathrm{av},m+1}+(\mathfrak{j}\wedge\mathfrak{n})\mathfrak{t}_{\mathrm{av},m}$ with $\mathfrak{t}_{N,\e,m} = (\mathfrak{j}\vee\mathfrak{n})\mathfrak{t}_{\mathrm{av},m}-(\mathfrak{j}\wedge\mathfrak{n})\mathfrak{t}_{\mathrm{av},m}$ to get the following for $\beta_{\mathrm{univ}} \in \R_{>0}$ a universal constant. Our application of Corollary \ref{corollary:D1B2A} is used with $\mathfrak{i}=2$ if the lower-bound cutoff in $\mathbf{1}[\mathscr{E}_{2,\mathfrak{n},S}^{\mathfrak{l},\mathfrak{k}}]$ is shifted ahead of the time-average $\Gamma^{N,\mathfrak{t}_{\mathrm{av},m},\mathfrak{j}}$ so that $\mathfrak{n}\geq\mathfrak{j}$ and otherwise $\mathfrak{i}=1$; the index $\mathfrak{i}\in\{1,2\}$ here is that in the superscript of the time-average-with-cutoff $\mathscr{C}^{\mathbf{T},+,\mathfrak{t}_{N,\e,m},\mathfrak{i}}$ in Corollary \ref{corollary:D1B2A}:
\small\begin{align}
\E\|\bar{\mathbf{H}}_{T,x}^{N}(|\Gamma_{S_{\mathfrak{l},\mathfrak{k},0}^{m},y}^{N,\mathfrak{t}_{\mathrm{av},m},\mathfrak{j}}| \cdot \mathbf{1}[\mathscr{E}_{2,\mathfrak{n},S}^{\mathfrak{l},\mathfrak{k}}])\|_{1;\mathbb{T}_{N}} \ &\lesssim \ N^{-1/2-\beta_{\mathrm{univ}}}. \label{eq:Step3A14}
\end{align}\normalsize\normalsize
Combining \eqref{eq:Step3A10}, \eqref{eq:Step3A13}, and \eqref{eq:Step3A14} gives the following if we choose $2\e\leq\beta_{\mathrm{univ}}$ sufficiently small but universal:
\small\begin{align}
\E[\|\bar{\mathbf{Z}}^{N}\|_{\mathfrak{t}_{\mathrm{st}};\mathbb{T}_{N}}^{-1}\|\Psi_{2}\|_{\mathfrak{t}_{\mathrm{st}};\mathbb{T}_{N}}] \ &\lesssim \ N^{-\frac12-\frac12\beta_{\mathrm{univ}}}. \label{eq:Step3Ai=2}
\end{align}\normalsize\normalsize
\item We move to $\mathfrak{u}=4$. Following the observation \eqref{eq:Step3A10} made for $\mathfrak{u}=2$ and recalling $\Psi_{4} = \bar{\mathbf{H}}^{N}(\Phi_{4}\cdot\bar{\mathbf{Z}}^{N})$ with $\Phi_{4}$ from \eqref{eq:Step3A3}, we get the following estimate in which we remove the $\bar{\mathbf{Z}}^{N}$-process from $\Psi_{4}$ via pulling out its $\|\|_{\mathfrak{t}_{\mathrm{st}};\mathbb{T}_{N}}$-norm and dividing it out to get just the $\Gamma_{\mathfrak{l},\mathfrak{k},\mathfrak{j},1,2}$-terms in $\Phi_{4}$ from \eqref{eq:Step3A3}; we additionally push the remaining $\|\|_{\mathfrak{t}_{\mathrm{st}};\mathbb{T}_{N}}$-norm to the $\|\|_{1;\mathbb{T}_{N}}$-norm:
\small\begin{align}
\E[\|\bar{\mathbf{Z}}^{N}\|_{\mathfrak{t}_{\mathrm{st}};\mathbb{T}_{N}}^{-1}\|\Psi_{\mathfrak{4}}\|_{\mathfrak{t}_{\mathrm{st}};\mathbb{T}_{N}}] \ &\leq \ \sup_{\mathfrak{l}=0,\ldots,\mathfrak{t}_{\mathrm{av},\mathfrak{m}_{+}}\mathfrak{t}_{\mathrm{av},m+2}^{-1}-1}\sup_{\mathfrak{k}=0,\ldots,\mathfrak{t}_{\mathrm{av},m+2}\mathfrak{t}_{\mathrm{av},m+1}^{-1}-1}\sup_{\mathfrak{j}=0,\ldots,\mathfrak{t}_{\mathrm{av},m+1}\mathfrak{t}_{\mathrm{av},m}^{-1}-1} \E\|\bar{\mathbf{H}}_{T,x}^{N}(|\Gamma_{\mathfrak{l},\mathfrak{k},\mathfrak{j},1,2}|)\|_{1;\mathbb{T}_{N}}. \label{eq:Step3A15}
\end{align}\normalsize\normalsize
To bound expectations on the RHS of \eqref{eq:Step3A15}, recall the definition of $\Gamma_{\mathfrak{l},\mathfrak{k},\mathfrak{j},1,2}$. This leads us to studying the event $\mathscr{E}_{1}$ introduced after \eqref{eq:Step3A1}. We make the following observation. We will explain it in detail shortly but it roughly transfers a priori estimates in $\mathscr{E}_{1}$ after \eqref{eq:Step3A1} for scale-$\mathfrak{t}_{\mathrm{av},m+1}$ time-integrals to scale-$\mathfrak{t}_{\mathrm{av},m}$ time-integrals by controlling the latter ``smaller-scale" integrals by integrals on intervals of the form $T_{\mathfrak{l},\mathfrak{k},0}^{m}+[0,\mathfrak{t}]$ with $\mathfrak{t}\leq\mathfrak{t}_{\mathrm{av},m+1}$ for which we have integral estimates in $\mathscr{E}_{1}$. Such a ``transfer" is performed by a close relative of the fact that we may control an integral on $[1,2]$ by integrals on $[0,2]$ and $[0,1]$. Roughly, we control integrals on scale-$\mathfrak{t}_{\mathrm{av},m}$ intervals by referring to the scale-$\mathfrak{t}_{\mathrm{av},m+1}$ block/$\mathfrak{k}$-index the scale-$\mathfrak{t}_{\mathrm{av},m}$ intervals live in:
\small\begin{align}
\mathbf{1}[\mathscr{E}_{1}] \ &\leq \ {\prod}_{\mathfrak{n}=0}^{\mathfrak{t}_{\mathrm{av},m+1}\mathfrak{t}_{\mathrm{av},m}^{-1}-1} \mathbf{1}\left({\sup}_{0\leq\mathfrak{t}\leq\mathfrak{t}_{\mathrm{av},m}}\mathfrak{t}\cdot\mathfrak{t}_{\mathrm{av},m+1}^{-1}|\mathscr{A}_{\mathfrak{t}}^{\mathbf{T},+}\mathscr{C}_{N^{\beta_{X}}}^{\mathbf{X},-}(\mathfrak{g}_{T_{\mathfrak{l},\mathfrak{k},\mathfrak{n}}^{m},x})| \lesssim N^{-\beta_{m}}\right) \label{eq:Step3A16} \\
&\leq \ {\prod}_{\mathfrak{n}=0}^{\mathfrak{t}_{\mathrm{av},m+1}\mathfrak{t}_{\mathrm{av},m}^{-1}-1} \mathbf{1}\left({\sup}_{0\leq\mathfrak{t}\leq\mathfrak{t}_{\mathrm{av},m}}\mathfrak{t}\cdot\mathfrak{t}_{\mathrm{av},m}^{-1}|\mathscr{A}_{\mathfrak{t}}^{\mathbf{T},+}\mathscr{C}_{N^{\beta_{X}}}^{\mathbf{X},-}(\mathfrak{g}_{T_{\mathfrak{l},\mathfrak{k},\mathfrak{n}}^{m},x})| \lesssim N^{-\beta_{m-1}}\right). \label{eq:Step3A17}
\end{align}\normalsize\normalsize
To get \eqref{eq:Step3A16}, the $\mathfrak{t}\mathscr{A}^{\mathbf{T},+}\mathscr{C}^{\mathbf{X},-}$-term in the sup in the indicator function on the RHS of \eqref{eq:Step3A16} is an \emph{unnormalized} integral of $\mathscr{C}^{\mathbf{X},-}$ over $T_{\mathfrak{l},\mathfrak{k},\mathfrak{n}}^{m}+[0,\mathfrak{t}]$. We control it by integrals on $[T_{\mathfrak{l},\mathfrak{k},0}^{m},T_{\mathfrak{l},\mathfrak{k},\mathfrak{n}}^{m}]$ and $[T_{\mathfrak{l},\mathfrak{k},0}^{m},T_{\mathfrak{l},\mathfrak{k},\mathfrak{n}}^{m}+\mathfrak{t}]$, whose set difference is $T_{\mathfrak{l},\mathfrak{k},\mathfrak{n}}^{m}+[0,\mathfrak{t}]$, via linearity of the integral with respect to domain of integration. Uniformly in $0\leq\mathfrak{t}\leq\mathfrak{t}_{\mathrm{av},m}$ these last two integrals are bounded on the event $\mathscr{E}_{1}$ by definition of $\mathscr{E}_{1}$ if $0\leq\mathfrak{n}\leq\mathfrak{t}_{\mathrm{av},m+1}\mathfrak{t}_{\mathrm{av},m}^{-1}-1$ because these two domains of integration $[T_{\mathfrak{l},\mathfrak{k},0}^{m},T_{\mathfrak{l},\mathfrak{k},\mathfrak{n}}^{m}]$ and $[T_{\mathfrak{l},\mathfrak{k},0}^{m},T_{\mathfrak{l},\mathfrak{k},\mathfrak{n}}^{m}+\mathfrak{t}]$ are both of the form $T_{\mathfrak{l},\mathfrak{k},0}^{m}+[0,\mathfrak{s}]$ with $\mathfrak{s}\leq\mathfrak{t}_{\mathrm{av},m+1}$ on which we have integral bounds for $\mathscr{C}^{\mathbf{X},-}$ on $\mathscr{E}_{1}$. Thus the $N^{-\beta_{m}}$-cutoff in $\mathscr{E}_{1}$ gets transferred to the terms inside the indicator functions on the RHS of \eqref{eq:Step3A16} \emph{on the event} $\mathscr{E}_{1}$. On the other hand, \eqref{eq:Step3A17} follows by \eqref{eq:Step3A16} and $\mathfrak{t}_{\mathrm{av},m}^{-1}\lesssim\mathfrak{t}_{\mathrm{av},m+1}^{-1}N^{-\beta_{m-1}+\beta_{m}}$ which can be checked using the definitions in Corollary \ref{corollary:D1B2A}. In particular, we may trade the $\mathfrak{t}_{\mathrm{av},m+1}^{-1}$-factor in the indicator function on the RHS of \eqref{eq:Step3A16} for the worse $\mathfrak{t}_{\mathrm{av},m}^{-1}$-factor if we loosen the upper bound $N^{-\beta_{m}}$ on the RHS of \eqref{eq:Step3A16} to $N^{-\beta_{m-1}}$ by the previous inequality for $\mathfrak{t}_{\mathrm{av},m+1}^{-1}$ and $\mathfrak{t}_{\mathrm{av},m}^{-1}$.

Looking at the definition of $\Gamma_{\mathfrak{l},\mathfrak{k},\mathfrak{j},1,2}$ given after \eqref{eq:Step3A2}, observe that it is a time-average $\mathscr{A}^{\mathbf{T},+}\mathscr{C}^{\mathbf{X},-}$ with the a priori lower-bound cutoff of $N^{-\beta_{m}}$. Moreover, by taking the $\mathfrak{n}=\mathfrak{j}$ factor within \eqref{eq:Step3A17}, we may impose the upper bound cutoff of $N^{-\beta_{m-1}}$ for the same time average, therefore controlling $\Gamma_{\mathfrak{l},\mathfrak{k},\mathfrak{j},1,2}$ by a $\mathscr{C}^{\mathbf{T},+}\mathscr{C}^{\mathbf{X},-}$-term. Thus, Corollary \ref{corollary:D1B2A} implies the following for $\beta_{\mathrm{univ}}>0$ universal upon choosing $\mathfrak{t}_{\mathfrak{s}} = \mathfrak{l}\mathfrak{t}_{\mathrm{av},m+2}+\mathfrak{k}\mathfrak{t}_{\mathrm{av},m+1}+\mathfrak{j}\mathfrak{t}_{\mathrm{av},m}$ and $\mathfrak{t}_{N,\e,m} = 0$ and for either $\mathfrak{i}\in\{1,2\}$ as $\mathfrak{t}_{N,\e,m} = 0$:
\small\begin{align}
\E\|\bar{\mathbf{H}}_{T,x}^{N}(|\Gamma_{\mathfrak{l},\mathfrak{k},\mathfrak{j},1,2}|)\|_{1;\mathbb{T}_{N}} \ &\lesssim \ N^{-1/2-\beta_{\mathrm{univ}}}.
\end{align}\normalsize\normalsize
Combining this with \eqref{eq:Step3A15} yields the expectation estimate
\small\begin{align}
\E[\|\bar{\mathbf{Z}}^{N}\|_{\mathfrak{t}_{\mathrm{st}};\mathbb{T}_{N}}^{-1}\|\Psi_{4}\|_{\mathfrak{t}_{\mathrm{st}};\mathbb{T}_{N}}] \ &\lesssim \ N^{-\frac12-\beta_{\mathrm{univ}}}. \label{eq:Step3Ai=4}
\end{align}\normalsize\normalsize
\item We are left with $\mathfrak{u}=5$ on the RHS of \eqref{eq:Step3A8}. As with \eqref{eq:Step3A10} and \eqref{eq:Step3A15}, by definition of $\Phi_{5}$ given after \eqref{eq:Step3A6}, we get the next bound again upon removing the $\bar{\mathbf{Z}}^{N}$-process from $\Psi_{5} = \bar{\mathbf{H}}^{N}(\Phi_{5}\cdot\bar{\mathbf{Z}}^{N})$ by pulling out its $\|\|_{\mathfrak{t}_{\mathrm{st}};\mathbb{T}_{N}}$-norm and dividing it out:
\small\begin{align}
\E[\|\bar{\mathbf{Z}}^{N}\|_{\mathfrak{t}_{\mathrm{st}};\mathbb{T}_{N}}^{-1}\|\Psi_{5}\|_{\mathfrak{t}_{\mathrm{st}};\mathbb{T}_{N}}] \ &\leq \ {\sup}_{\substack{\mathfrak{l}=0,\ldots,\mathfrak{t}_{\mathrm{av},\mathfrak{m}_{+}}\mathfrak{t}_{\mathrm{av},m+2}^{-1}-1\\\mathfrak{k}=0,\ldots,\mathfrak{t}_{\mathrm{av},m+2}\mathfrak{t}_{\mathrm{av},m+1}^{-1}-1}}\E\|\bar{\mathbf{H}}_{T,x}^{N}(\mathbf{1}[\mathscr{E}_{1}]\mathbf{1}[\mathscr{G}_{2}]\cdot|\mathscr{A}_{\mathfrak{t}_{\mathrm{av},m+1}}^{\mathbf{T},+}\mathscr{C}_{N^{\beta_{X}}}^{\mathbf{X},-}(\mathfrak{g}_{S_{\mathfrak{l},\mathfrak{k},0}^{m},y})|)\|_{1;\mathbb{T}_{N}}. \label{eq:Step3A18}
\end{align}\normalsize\normalsize
The quantity that we are integrating in the heat-operator on the RHS of \eqref{eq:Step3A18}, per every pair of indices $\mathfrak{l},\mathfrak{k}$, is a time-average term $\mathscr{A}^{\mathbf{T},+}\mathscr{C}^{\mathbf{X},-}$ that, by definition of $\mathscr{E}_{1}$, carries an upper bound cutoff of $N^{-\beta_{m}}$ and, by definition of $\mathscr{G}_{2}$, carries a lower bound cutoff of $N^{-\beta_{m+1}}$. Thus, the quantity we are integrating in the heat operator on the RHS of \eqref{eq:Step3A18} is a $\mathscr{C}^{\mathbf{T},+}\mathscr{C}^{\mathbf{X},-}$ that is treated in Corollary \ref{corollary:D1B2A}. In particular, by Corollary \ref{corollary:D1B2A} with $\mathfrak{t}_{\mathfrak{s}} = \mathfrak{l}\mathfrak{t}_{\mathrm{av},m+2}+\mathfrak{k}\mathfrak{t}_{\mathrm{av},m+1}$ and $\mathfrak{t}_{N,\e,m} = 0$ and either $\mathfrak{i}\in\{1,2\}$, we get
\small\begin{align}
\E\|\bar{\mathbf{H}}_{T,x}^{N}(\mathbf{1}[\mathscr{E}_{1}]\mathbf{1}[\mathscr{G}_{2}]\cdot|\mathscr{A}_{\mathfrak{t}_{\mathrm{av},m+1}}^{\mathbf{T},+}\mathscr{C}_{N^{\beta_{X}}}^{\mathbf{X},-}(\mathfrak{g}_{S_{\mathfrak{l},\mathfrak{k},0}^{m},y})|)\|_{1;\mathbb{T}_{N}} \ \lesssim \ N^{-1/2-\beta_{\mathrm{univ}}}.
\end{align}\normalsize\normalsize
From this last estimate combined with \eqref{eq:Step3A18}, we get the expectation estimate
\small\begin{align}
\E[\|\bar{\mathbf{Z}}^{N}\|_{\mathfrak{t}_{\mathrm{st}};\mathbb{T}_{N}}^{-1}\|\Psi_{5}\|_{\mathfrak{t}_{\mathrm{st}};\mathbb{T}_{N}}] \ &\lesssim \ N^{-\frac12-\beta_{\mathrm{univ}}}. \label{eq:Step3Ai=5}
\end{align}\normalsize\normalsize
\end{itemize}
By \eqref{eq:Step3A8} and \eqref{eq:Step3Ai=2}, \eqref{eq:Step3Ai=4}, and \eqref{eq:Step3Ai=5}, with the required high-probability, for some $\beta_{\mathrm{univ},2}\in\R_{>0}$ universal, we get
\small\begin{align}
\|\bar{\mathbf{Z}}^{N}\|_{\mathfrak{t}_{\mathrm{st}};\mathbb{T}_{N}}^{-1}{\sum}_{\mathfrak{u}=2,4,5}\|\bar{\mathbf{H}}_{T,x}^{N}(\Phi_{\mathfrak{u}}\bar{\mathbf{Z}}^{N})\|_{\mathfrak{t}_{\mathrm{st}};\mathbb{T}_{N}} \ = \ \|\bar{\mathbf{Z}}^{N}\|_{\mathfrak{t}_{\mathrm{st}};\mathbb{T}_{N}}^{-1}{\sum}_{\mathfrak{u}=2,4,5}\|\Psi_{\mathfrak{u}}\|_{\mathfrak{t}_{\mathrm{st}};\mathbb{T}_{N}} \ \lesssim \ N^{-\frac12-\beta_{\mathrm{univ},2}}. \label{eq:Step3AFinal}
\end{align}\normalsize\normalsize
Indeed we recall $\Psi_{\mathfrak{u}} = \bar{\mathbf{H}}_{T,x}^{N}(\Phi_{\mathfrak{u}}\bar{\mathbf{Z}}^{N})$ by definition given right after the equation \eqref{eq:Step3A7}. Applying \eqref{eq:Step3AFinal} with \eqref{eq:Step3A7} finishes the proof for the proposed estimate that is explicitly written in the statement of Lemma \ref{lemma:Step3A}. To get the estimate after replacement \eqref{eq:KPZNLReplace} and a replacement of time-scales/exponents by those in Corollary \ref{corollary:D1B2B}, it suffices to write these replacements in formally and then employ an identical argument. In particular, the details of what we are time-averaging and cutting off are not relevant, as are the details of the exponents as long as we use the estimates in Corollary \ref{corollary:D1B2B} instead of those in Corollary \ref{corollary:D1B2A}.
\end{proof}
\begin{proof}[Proof of \emph{Lemma \ref{lemma:Step3B}}]
In principle we have already written a proof for Lemma \ref{lemma:Step3B} in the middle of the proof for Lemma \ref{lemma:Step3A}. However we provide details anyway for clarity. The architecture for the proof of Lemma \ref{lemma:Step3B} is similar to that of Lemma \ref{lemma:Step3A}. We first relate $\Gamma^{N,1,2}$ within the RHS of the proposed estimate to the space-time average $\mathscr{A}^{\mathbf{T},+}\mathscr{C}^{\mathbf{X},-}$ within the LHS of the same proposed estimate through exact identities, and then we estimate the error terms in such relation/the quantities in these exact identities.  The approach that we take towards relating these two quantities from the previous sentence/writing exact identities is again by upgrading in time-scale $\mathfrak{t}_{\mathrm{av},1} \to \mathfrak{t}_{\mathrm{av},2}$ and cutoffs then estimating errors which turn out to be time-averages of $\mathscr{C}^{\mathbf{X},-}(\mathfrak{g})$ with upper/lower bound cutoffs. First, as in the proof of Lemma \ref{lemma:Step3A} we introduce notation for time-shifts at two time-scales:
\begin{itemize}[leftmargin=*]
\item For any $\mathfrak{l},\mathfrak{j}\in\Z_{\geq0}$ and $T\geq0$, define a shift $T_{\mathfrak{l},\mathfrak{j}}\overset{\bullet}=T+\mathfrak{l}\mathfrak{t}_{\mathrm{av},2}+\mathfrak{j}\mathfrak{t}_{\mathrm{av},1}$ with respect to time-scales $\mathfrak{t}_{\mathrm{av},2},\mathfrak{t}_{\mathrm{av},1}$, respectively.
\end{itemize}
Recall the $\mathscr{A}^{\mathbf{T},+}\mathscr{C}^{\mathbf{X},-}(\mathfrak{g})$-term is a time-average with respect to the ``maximal" scale $\mathfrak{t}_{\mathrm{av},\mathfrak{m}_{+}}$. With more explanation after,
\small\begin{align}
\mathscr{A}_{\mathfrak{t}_{\mathrm{av},\mathfrak{m}_{+}}}^{\mathbf{T},+}\mathscr{C}_{N^{\beta_{X}}}^{\mathbf{X},-}(\mathfrak{g}_{T,x}) \ &= \ \wt{\sum}_{\mathfrak{l}=0}^{\mathfrak{t}_{\mathrm{av},\mathfrak{m}_{+}}\mathfrak{t}_{\mathrm{av},2}^{-1}-1}\wt{\sum}_{\mathfrak{j}=0}^{\mathfrak{t}_{\mathrm{av},2}\mathfrak{t}_{\mathrm{av},1}^{-1}-1} \mathscr{A}_{\mathfrak{t}_{\mathrm{av},1}}^{\mathbf{T},+}\mathscr{C}_{N^{\beta_{X}}}^{\mathbf{X},-}(\mathfrak{g}_{T_{\mathfrak{l},\mathfrak{j}},x}). \label{eq:Step3B1}
\end{align}\normalsize\normalsize
Indeed, we interpret the LHS of \eqref{eq:Step3B1} to be the average of suitably time-shifted dynamic-averages with respect to the smaller time-scale $\mathfrak{t}_{\mathrm{av},2}$ in Corollary \ref{corollary:D1B2A}, and we further decompose each of these time averages with respect to $\mathfrak{t}_{\mathrm{av},2}$ into more suitably time-shifted time averages with respect to the smallest time-scale $\mathfrak{t}_{\mathrm{av},1}$. As $\mathscr{C}^{\mathbf{X},-}$ has the upper bound cutoff $N^{-\beta_{0}}$ by definition, where $\beta_{0} = \frac12\beta_{X} - \e_{X,2}$ with $\e_{X,2} \geq \frac12\e_{X,1}$ arbitrarily small but universal in Corollary \ref{corollary:D1B2A}, for free by \eqref{eq:Step3B1} we get
\small\begin{align}
\mathscr{A}_{\mathfrak{t}_{\mathrm{av},\mathfrak{m}_{+}}}^{\mathbf{T},+}\mathscr{C}_{N^{\beta_{X}}}^{\mathbf{X},-}(\mathfrak{g}_{T,x}) \ &= \ \wt{\sum}_{\mathfrak{l}=0}^{\mathfrak{t}_{\mathrm{av},\mathfrak{m}_{+}}\mathfrak{t}_{\mathrm{av},2}^{-1}-1}\wt{\sum}_{\mathfrak{j}=0}^{\mathfrak{t}_{\mathrm{av},2}\mathfrak{t}_{\mathrm{av},1}^{-1}-1} \mathscr{A}_{\mathfrak{t}_{\mathrm{av},1}}^{\mathbf{T},+}\mathscr{C}_{N^{\beta_{X}}}^{\mathbf{X},-}(\mathfrak{g}_{T_{\mathfrak{l},\mathfrak{j}},x}) \cdot \mathbf{1}[\mathscr{E}_{\mathfrak{l},\mathfrak{j}}] \label{eq:Step3B2}
\end{align}\normalsize\normalsize
provided that we have defined the following event. We clarify that the supremum in the following event is the supremum over \emph{integrals} on time-scales $0\leq\mathfrak{t}\leq\mathfrak{t}_{\mathrm{av},1}$ of the cutoff spatial average $\mathscr{C}^{\mathbf{X},-}$ then reweighted by $\mathfrak{t}_{\mathrm{av},1}^{-1}$. In particular, every $\mathfrak{t}$-indexed term in the sup below is bounded by $|\mathscr{C}^{\mathbf{X},-}|\lesssim N^{-\beta_{0}}$ as every $\mathfrak{t}$-indexed term time-averages $\mathscr{C}^{\mathbf{X},-}$ on a time-scale $\mathfrak{t}\leq\mathfrak{t}_{\mathrm{av},1}$:
\small\begin{align}
\mathbf{1}[\mathscr{E}_{\mathfrak{l},\mathfrak{j}}] \ &\overset{\bullet}= \ \mathbf{1}\left({\sup}_{0\leq\mathfrak{t}\leq\mathfrak{t}_{\mathrm{av},1}}\mathfrak{t}\cdot\mathfrak{t}_{\mathrm{av},1}^{-1}|\mathscr{A}_{\mathfrak{t}}^{\mathbf{T},+}\mathscr{C}_{N^{\beta_{X}}}^{\mathbf{X},-}(\mathfrak{g}_{T_{\mathfrak{l},\mathfrak{j}},x})| \lesssim N^{-\beta_{0}} \right).
\end{align}\normalsize\normalsize
Observe that the $\mathfrak{l},\mathfrak{j}$-term in the sum from the RHS of \eqref{eq:Step3B2} is \emph{almost} the $\Gamma^{N,\mathfrak{t}_{\mathrm{av},1},\mathfrak{j}}$-term, except the cutoff defining $\mathscr{E}_{\mathfrak{l},\mathfrak{j}}$ is of order $N^{-\beta_{0}}$ rather than $N^{-\beta_{1}}$. To this end, we define the following event $\mathscr{G}_{\mathfrak{l},\mathfrak{j},1}$ with a stricter constraint giving us the cutoff we need:
\small\begin{align}
\mathbf{1}[\mathscr{G}_{\mathfrak{l},\mathfrak{j},1}] \ &\overset{\bullet}= \ \mathbf{1}\left({\sup}_{0\leq\mathfrak{t}\leq\mathfrak{t}_{\mathrm{av},1}}\mathfrak{t}\cdot\mathfrak{t}_{\mathrm{av},1}^{-1}|\mathscr{A}_{\mathfrak{t}}^{\mathbf{T},+}\mathscr{C}_{N^{\beta_{X}}}^{\mathbf{X},-}(\mathfrak{g}_{T_{\mathfrak{l},\mathfrak{j}},x})| \leq N^{-\beta_{1}} \right).
\end{align}\normalsize\normalsize
Let $\mathscr{G}_{\mathfrak{l},\mathfrak{j},2}$ denote the complement of $\mathscr{G}_{\mathfrak{l},\mathfrak{j},1}$. Observe the containment $\mathscr{G}_{\mathfrak{l},\mathfrak{j},1}\subseteq\mathscr{E}_{\mathfrak{l},\mathfrak{j}}$ that follows because $\beta_{1} \geq \beta_{0}$, so $\mathscr{G}_{\mathfrak{l},\mathfrak{j},1}$ imposes a stricter constraint than $\mathscr{E}_{\mathfrak{l},\mathfrak{j}}$. Thus $\mathbf{1}[\mathscr{E}_{\mathfrak{l},\mathfrak{j}}]=\mathbf{1}[\mathscr{G}_{\mathfrak{l},\mathfrak{j},1}]+\mathbf{1}[\mathscr{G}_{\mathfrak{l},\mathfrak{j},2}]\mathbf{1}[\mathscr{E}_{\mathfrak{l},\mathfrak{j}}]$. We apply this decomposition to \eqref{eq:Step3B2} and deduce
\small\begin{align}
\mathscr{A}_{\mathfrak{t}_{\mathrm{av},\mathfrak{m}_{+}}}^{\mathbf{T},+}\mathscr{C}_{N^{\beta_{X}}}^{\mathbf{X},-}(\mathfrak{g}_{T,x}) \ &= \ \wt{\sum}_{\mathfrak{l}=0}^{\mathfrak{t}_{\mathrm{av},\mathfrak{m}_{+}}\mathfrak{t}_{\mathrm{av},2}^{-1}-1}\wt{\sum}_{\mathfrak{j}=0}^{\mathfrak{t}_{\mathrm{av},2}\mathfrak{t}_{\mathrm{av},1}^{-1}-1}\Phi_{1,\mathfrak{l},\mathfrak{j}} + \wt{\sum}_{\mathfrak{l}=0}^{\mathfrak{t}_{\mathrm{av},\mathfrak{m}_{+}}\mathfrak{t}_{\mathrm{av},2}^{-1}-1}\wt{\sum}_{\mathfrak{j}=0}^{\mathfrak{t}_{\mathrm{av},2}\mathfrak{t}_{\mathrm{av},1}^{-1}-1}\Phi_{2,\mathfrak{l},\mathfrak{j}} \ \overset{\bullet}= \ \Phi_{1} + \Phi_{2} \label{eq:Step3B3}
\end{align}\normalsize\normalsize
where $\Phi_{1,\mathfrak{l},\mathfrak{j}}$ replaces the free $\mathscr{E}_{\mathfrak{l},\mathfrak{j}}$-cutoff of $N^{-\beta_{0}}$ with the $\mathscr{G}_{\mathfrak{l},\mathfrak{j},1}$-cutoff of $N^{-\beta_{1}}$, and $\Phi_{2,\mathfrak{l},\mathfrak{j}}$ is the resulting error in this upgrade:
\small\begin{align}
\Phi_{1,\mathfrak{l},\mathfrak{j}} \ &\overset{\bullet}= \ \mathscr{A}_{\mathfrak{t}_{\mathrm{av},1}}^{\mathbf{T},+}\mathscr{C}_{N^{\beta_{X}}}^{\mathbf{X},-}(\mathfrak{g}_{T_{\mathfrak{l},\mathfrak{j}},x}) \cdot \mathbf{1}[\mathscr{G}_{\mathfrak{l},\mathfrak{j},1}] \quad \mathrm{and} \quad \Phi_{2,\mathfrak{l},\mathfrak{j}} \ \overset{\bullet}= \ \mathscr{A}_{\mathfrak{t}_{\mathrm{av},1}}^{\mathbf{T},+}\mathscr{C}_{N^{\beta_{X}}}^{\mathbf{X},-}(\mathfrak{g}_{T_{\mathfrak{l},\mathfrak{j}},x}) \cdot \mathbf{1}[\mathscr{E}_{\mathfrak{l},\mathfrak{j}}] \mathbf{1}[\mathscr{G}_{\mathfrak{l},\mathfrak{j},2}].
\end{align}\normalsize\normalsize
We observe that $\Phi_{1} = \Gamma^{N,1,2}$ just by its definition within the statement of Lemma \ref{lemma:Step3B} as it is an average of suitably time-shifted time-averages with respect to time-scale $\mathfrak{t}_{\mathrm{av},1}$ each with an upper bound cutoff of $N^{-\beta_{1}}$. By \eqref{eq:Step3B3} and this observation we get
\small\begin{align}
|\bar{\mathbf{H}}_{T,x}^{N}(N^{\frac12}\mathscr{A}_{\mathfrak{t}_{\mathrm{av},\mathfrak{m}_{+}}}^{\mathbf{T},+}\mathscr{C}_{N^{\beta_{X}}}^{\mathbf{X},-}(\mathfrak{g}) \cdot \bar{\mathbf{Z}}^{N})| \ &\leq \ |\bar{\mathbf{H}}_{T,x}^{N}(N^{\frac12}\Gamma^{N,1,2} \cdot \bar{\mathbf{Z}}^{N})| \ + \ \|\bar{\mathbf{Z}}^{N}\|_{\mathfrak{t}_{\mathrm{st}};\mathbb{T}_{N}} \bar{\mathbf{H}}_{T,x}^{N}(N^{\frac12}|\Phi_{2}|). \label{eq:Step3B4}
\end{align}\normalsize\normalsize
It suffices to show that with the required high-probability, we can control the second term in \eqref{eq:Step3B4}:
\small\begin{align}
\|\bar{\mathbf{H}}_{T,x}^{N}(N^{\frac12}|\Phi_{2}|)\|_{\mathfrak{t}_{\mathrm{st}};\mathbb{T}_{N}} \ &\lesssim \ N^{-\beta_{\mathrm{univ},2}}. \label{eq:Step3B!}
\end{align}\normalsize\normalsize
We will do this using a procedure that is a simplified or specialized version of estimates for $\Psi_{\mathfrak{u}}$ with $\mathfrak{u}\in\{2,4,5\}$ in the proof for Lemma \ref{lemma:Step3A}. In particular, we will control $|\Phi_{2}|$ in terms of quantities of the form $\mathscr{C}^{\mathbf{T},+}\mathscr{C}^{\mathbf{X},-}$ that are estimated in Corollary \ref{corollary:D1B2A}. Recall the $\mathfrak{l},\mathfrak{j}$-term in the sum defining $\Phi_{2}$ is $\Phi_{2,\mathfrak{l},\mathfrak{j}}$. To establish the previous high-probability statement, we employ the Markov inequality as with the estimate \eqref{eq:Step3A8} in the proof of Lemma \ref{lemma:Step3A} and then unfold the definition of $\Phi_{2}$. Since $\mathfrak{t}_{\mathrm{st}}\leq2$,
\small\begin{align}
\mathbf{P}\left(\|\bar{\mathbf{H}}_{T,x}^{N}(N^{\frac12}|\Phi_{2}|)\|_{\mathfrak{t}_{\mathrm{st}};\mathbb{T}_{N}} \gtrsim N^{-\beta_{\mathrm{univ},2}}\right) \ &\leq \ N^{\beta_{\mathrm{univ},2}} \E\|\bar{\mathbf{H}}_{T,x}^{N}(N^{\frac12}|\Phi_{2}|)\|_{2;\mathbb{T}_{N}} \\
&\leq \ \wt{\sum}_{\mathfrak{l}=0}^{\mathfrak{t}_{\mathrm{av},\mathfrak{m}_{+}}\mathfrak{t}_{\mathrm{av},2}^{-1}-1}\wt{\sum}_{\mathfrak{j}=0}^{\mathfrak{t}_{\mathrm{av},2}\mathfrak{t}_{\mathrm{av},1}^{-1}-1} N^{\beta_{\mathrm{univ},2}} \E\|\bar{\mathbf{H}}^{N}_{T,x}(N^{\frac12}|\Phi_{2,\mathfrak{l},\mathfrak{j}}|)\|_{2;\mathbb{T}_{N}} \\
&\leq \ {\sup}_{\mathfrak{l}=0,\ldots,\mathfrak{t}_{\mathrm{av},\mathfrak{m}_{+}}\mathfrak{t}_{\mathrm{av},2}^{-1}-1}{\sup}_{\mathfrak{j}=0,\ldots,\mathfrak{t}_{\mathrm{av},2}\mathfrak{t}_{\mathrm{av},1}^{-1}-1} N^{\beta_{\mathrm{univ},2}} \E\|\bar{\mathbf{H}}^{N}_{T,x}(N^{\frac12}|\Phi_{2,\mathfrak{l},\mathfrak{j}}|)\|_{2;\mathbb{T}_{N}}. \label{eq:Step3B5}
\end{align}\normalsize\normalsize
In view of \eqref{eq:Step3B5}, we want to control expectations in \eqref{eq:Step3B5} uniformly in $\mathfrak{l},\mathfrak{j}$ as this would provide the estimate \eqref{eq:Step3B!}. To this end, note that $\Phi_{2,\mathfrak{l},\mathfrak{j}}$, by its definition as the $(\mathfrak{l},\mathfrak{j})$-summand in the summation for $\Phi_{2}$ that we wrote immediately after the decomposition \eqref{eq:Step3B3}, is a space-time average $\mathscr{A}^{\mathbf{T},+}\mathscr{C}^{\mathbf{X},-}$ equipped with the upper bound cutoff $N^{-\beta_{0}}$ courtesy of the event $\mathscr{E}_{\mathfrak{l},\mathfrak{j}}$ inside $\Phi_{2,\mathfrak{l},\mathfrak{j}}$ as well as a lower bound cutoff of $N^{-\beta_{1}}$ courtesy of the $\mathscr{G}_{\mathfrak{l},\mathfrak{j},2}$-event inside $\Phi_{2,\mathfrak{l},\mathfrak{j}}$. Thus, we may use the $m=1$ bound in Corollary \ref{corollary:D1B2A} with $\mathfrak{t}_{\mathfrak{s}} = \mathfrak{l}\mathfrak{t}_{\mathrm{av},2}+\mathfrak{j}\mathfrak{t}_{\mathrm{av},1}$ and $\mathfrak{t}_{N,\e,m} = 0$ and either index $\mathfrak{i}\in\{1,2\}$ to get, for $\beta_{\mathrm{univ}}>0$ universal,
\small\begin{align}
{\sup}_{\mathfrak{l}=0,\ldots,\mathfrak{t}_{\mathrm{av},\mathfrak{m}_{+}}\mathfrak{t}_{\mathrm{av},2}^{-1}-1}{\sup}_{\mathfrak{j}=0,\ldots,\mathfrak{t}_{\mathrm{av},2}\mathfrak{t}_{\mathrm{av},1}^{-1}-1} N^{\beta_{\mathrm{univ},2}} \E\|\bar{\mathbf{H}}^{N}_{T,x}(N^{\frac12}|\Phi_{2,\mathfrak{l},\mathfrak{j}}|)\|_{2;\mathbb{T}_{N}} \ &\lesssim \ N^{-\beta_{\mathrm{univ}}+\beta_{\mathrm{univ},2}}.
\end{align}\normalsize\normalsize
Choosing $\beta_{\mathrm{univ},2} = \frac12\beta_{\mathrm{univ}}$ and plugging the previous estimate into \eqref{eq:Step3B5} establishes \eqref{eq:Step3B!} with the required high-probability and completes the proof of the proposed bound that is explicitly written in the statement of Lemma \ref{lemma:Step3B}. To prove the estimate \emph{after} the replacement \eqref{eq:KPZNLReplace} and the replacement of time-scales and exponents by those in Corollary \ref{corollary:D1B2B}, we follow the remarks given at the end of the proof of Lemma \ref{lemma:Step3A}. In particular, we make the aforementioned replacements \emph{in the proof} and use the estimates in Corollary \ref{corollary:D1B2B} as opposed to those in Corollary \ref{corollary:D1B2A}.
\end{proof}
%
%
%
\section{Proof of Theorem \ref{theorem:KPZ}}\label{section:KPZ3}
The proof of Theorem \ref{theorem:KPZ} will follow the strategy that is outlined in the bullet points below.
\begin{itemize}[leftmargin=*]
\item We compare the microscopic Cole-Hopf transform $\mathbf{Z}^{N}$ to an auxiliary space-time process we denote by $\mathbf{Y}^{N}$. This space-time process is defined by the same stochastic integral equation as $\bar{\mathbf{Z}}^{N}$ in Section \ref{section:Ctify} but tossing out $\bar{\Phi}^{N,2}$ and replacing all $\bar{\mathbf{Z}}^{N}$-terms by $\mathbf{Y}^{N}$. In particular, besides microscopic SHE terms, the $\mathbf{Y}^{N}$-equation only has data of weakly vanishing functionals.
\item After the comparison between $\mathbf{Z}^{N}$ and $\mathbf{Y}^{N}$, we prove that $\mathbf{Y}^{N}$ converges to the solution of SHE in the Skorokhod space $\mathbf{D}_{1}$. To this end, recall that $\mathbf{Y}^{N}$ is defined by a microscopic version of the SHE plus some extra weakly vanishing data. In particular, proof of convergence to SHE for $\mathbf{Y}^{N}$ will follow the analysis of weakly vanishing terms in \cite{DT}. Because most of the work is done in \cite{DT}, we only present necessary adjustments for the proof of convergence to SHE for $\mathbf{Y}^{N}$.
\end{itemize}
We first define the $\mathbf{Y}^{N}$-process; recall $\mathbb{T}_{N}$ in the beginning of Section \ref{section:Ctify} and $\bar{\grad}^{!} = N\bar{\grad}$ where $\bar{\grad}$ is gradient on the \emph{torus} $\mathbb{T}_{N}$.
\begin{definition}\label{definition:YProcess}
Define $\mathbf{Y}^{N}$ as the solution to the following stochastic equation on $\R_{\geq0}\times\mathbb{T}_{N}$; recall $\chi$ from Definition \ref{definition:ChiTorus}:
\small\begin{align}
\mathbf{Y}_{T,x}^{N} \ &= \ \bar{\mathbf{H}}_{T,x}^{N,\mathbf{X}}(\bar{\mathbf{Z}}_{0,\bullet}^{N}) + \bar{\mathbf{H}}_{T,x}^{N}(\mathbf{Y}^{N}\d\xi^{N}) + \bar{\mathbf{H}}_{T,x}^{N}(\bar{\Phi}^{N,4}) \ = \ \bar{\mathbf{H}}_{T,x}^{N,\mathbf{X}}(\chi_{\bullet}\mathbf{Z}_{0,\bullet}^{N}) + \bar{\mathbf{H}}_{T,x}^{N}(\mathbf{Y}^{N}\d\xi^{N}) + \bar{\mathbf{H}}_{T,x}^{N}(\bar{\Phi}^{N,4}).
\end{align}\normalsize\normalsize
As in Remark \ref{remark:ch2Jumps} $\mathbf{Y}^{N}\d\xi^{N}$ is the martingale differential corresponding to the Poisson process with jumps of $\d M$ in (2.4) in \cite{DT} then scaled by $\mathbf{Y}^{N}$ at the same space-time point. Also $\bar{\Phi}^{N,4}$ is the following analog of $\bar{\Phi}^{N,3}$ from Definition \ref{definition:ChiTorus}:
\small\begin{align}
\bar{\Phi}^{N,4}_{T,x} \ &\overset{\bullet}= \ \mathfrak{w}_{T,x}\mathbf{Y}_{T,x}^{N} + {\sum}_{k=-2\mathfrak{m}}^{2\mathfrak{m}}c_{k} \bar{\grad}_{k}^{!}(\mathfrak{w}_{T,x}^{k}\mathbf{Y}_{T,x}^{N}). 
\end{align}\normalsize\normalsize
\end{definition}
A function on $\mathbb{T}_{N}$ or $\R_{\geq0}\times\mathbb{T}_{N}$, if it is not already given by restricting a function on $\Z$ or $\R_{\geq0}\times\Z$ to the torus $\mathbb{T}_{N}$, is implicitly lifted to a function on $\Z$ or $\R_{\geq0}\times\Z$ by periodic extension outside $\mathbb{T}_{N}$. We then lift to $\R$ or $\R_{\geq0}\times\R$ by linear interpolation. This applies to $\mathbf{Y}^{N}$, for example. The main ingredients that we briefly overviewed above are stated precisely as follows.
\begin{prop}\label{prop:KPZ1}
 Consider any compact set $\mathbb{K}\subseteq\R$ and define its microscopic coordinates $\mathbb{K}_{N}\overset{\bullet}=N\mathbb{K}\cap\Z$. There exist universal constants $\beta_{\mathrm{univ},1},\beta_{\mathrm{univ},2}\in\R_{>0}$ such that outside an event of probability at most $N^{-\beta_{\mathrm{univ},1}}$ times uniformly bounded factors,
\small\begin{align}
\|\mathbf{Z}^{N}-\mathbf{Y}^{N}\|_{1;\mathbb{K}_{N}} \ &\lesssim_{\mathbb{K}} \ N^{-\beta_{\mathrm{univ},2}}. 
\end{align}\normalsize\normalsize
\end{prop}
\begin{prop}\label{prop:KPZ2}
 The process $\mathbf{Y}^{N}_{\bullet,N\bullet}$ is tight in $\mathbf{D}_{1}$ and all limit points are the law of the solution of \emph{SHE} with initial data $\mathbf{Z}_{0,\bullet}$.
\end{prop}
\begin{lemma}\label{lemma:KPZ3}
 Consider processes $\mathbf{X}^{N,1}$ and $\mathbf{X}^{N,2}$ in $\mathbf{D}_{1}$. Suppose $\mathbf{X}^{N,2}$ is tight in $\mathbf{D}_{1}$ with limit $\mathbf{X}^{\infty,2}$, and there are universal constants $\beta_{\mathrm{univ},1},\beta_{\mathrm{univ},2}\in\R_{>0}$ so that for any compact set $\mathbb{K}\subseteq\R$, outside an event with probability at most order $N^{-\beta_{\mathrm{univ},1}}$,
\small\begin{align}
\|\mathbf{X}^{N,1}-\mathbf{X}^{N,2}\|_{1;\mathbb{K}} \ \lesssim_{\mathbb{K}} \ N^{-\beta_{\mathrm{univ},2}}.
\end{align}\normalsize\normalsize
The sequence $\mathbf{X}^{N,1}$ is also tight and it converges to the same limit $\mathbf{X}^{\infty,2}$. All limits are as probability measures on $\mathbf{D}_{1}$.
\end{lemma}
\begin{proof}[Proof of \emph{Theorem \ref{theorem:KPZ}}]
Observe Proposition \ref{prop:KPZ1} gives a power-saving bound in $N\in\Z_{>0}$ on the difference of spatially-rescaled versions of $\mathbf{Z}^{N}$ and $\mathbf{Y}^{N}$ uniformly over a discretization of the compact space-time set $[0,1]\times\mathbb{K}$. Because both processes extend \emph{by linear interpolation} outside such a discretization to the whole space-time set $[0,1]\times\mathbb{K}$, the difference between the spatially rescaled versions of $\mathbf{Z}^{N}$ and $\mathbf{Y}^{N}$ on the set $[0,1]\times\mathbb{K}$ is controlled by the same difference but restricting to the aforementioned discretization. Thus we deduce Theorem \ref{theorem:KPZ} from the tightness and convergence of $\mathbf{Y}^{N}$ within Proposition \ref{prop:KPZ2} combined with the comparison result in Lemma \ref{lemma:KPZ3}. This completes the proof.
\end{proof}
The rest of the section is organized as follows. We will focus on a proof of Proposition \ref{prop:KPZ1} because it takes up the bulk of the section since it requires a few auxiliary lemmas and an elementary but detailed ``high-probability" pathwise argument. We will present the auxiliary lemmas, but we will defer their proofs to the end of this section in order to avoid obscuring key arguments and since their proofs are based on standard methods or are a lot of technical work. On the other hand, the proof of Proposition \ref{prop:KPZ2} is relatively short as it is the focus of \cite{DT}. Lemma \ref{lemma:KPZ3} is an elementary topological result proved at the end of this section.
\subsection{Proof of Proposition \ref{prop:KPZ1}}
Let us inherit the notation from Proposition \ref{prop:KPZ1} throughout this subsection. We observe that the intersection of any two events whose complements hold with probability at most $N^{-\beta_{\mathrm{univ},1}}$ times uniformly bounded factors also has complement holding with probability at most $N^{-\beta_{\mathrm{univ},1}}$ times uniformly bounded factors. Thus, Proposition \ref{prop:KPZ1} follows by the following pair of comparison estimates courtesy of this observation and the triangle inequality.
\begin{lemma}\label{lemma:KPZ12}
 Consider a compact set $\mathbb{K}\subseteq\R$ and define microscopic coordinates $\mathbb{K}_{N}\overset{\bullet}=N\mathbb{K}\cap\Z$. There exist universal $\beta_{\mathrm{univ},1},\beta_{\mathrm{univ},2}>0$ such that outside an event of probability at most $N^{-\beta_{\mathrm{univ},1}}$ times uniformly bounded factors, we have
\small\begin{align}
\|\bar{\mathbf{Z}}^{N}-\mathbf{Y}^{N}\|_{1;\mathbb{K}_{N}} \ &\lesssim_{\mathbb{K}} \ \|\bar{\mathbf{Z}}^{N}-\mathbf{Y}^{N}\|_{1;\mathbb{T}_{N}} \ \lesssim \ N^{-\beta_{\mathrm{univ},2}}. 
\end{align}\normalsize\normalsize
\end{lemma}
\begin{lemma}\label{lemma:KPZ11}
 \emph{Lemma \ref{lemma:KPZ12}} holds upon replacing $\mathbf{Y}^{N}$ by $\mathbf{Z}^{N}$ and forgetting the middle inequality/$\|\|_{1;\mathbb{T}_{N}}$-term.
\end{lemma}
Lemma \ref{lemma:KPZ11} is almost direct from Proposition \ref{prop:Ctify}. The difference is that Lemma \ref{lemma:KPZ11} uses $\|\|_{1;\mathbb{K}_{N}}$ and Proposition \ref{prop:Ctify} uses the discretized norm $[]_{1;\mathbb{K}_{N}}$. We will use short-time stochastic continuity to bootstrap norms similar to the proof of Lemma \ref{lemma:TRGTProp3}. We also use Lemma \ref{lemma:KPZ12} to prove Lemma \ref{lemma:KPZ11}; the proof of Lemma \ref{lemma:KPZ12} will be independent of Lemma \ref{lemma:KPZ11}.

The proof of Lemma \ref{lemma:KPZ11} will be deferred to a later section on technical estimates as it will require no preliminary ingredients. The proof of Lemma \ref{lemma:KPZ12}, however, is more involved and requires a few preliminaries. It is based on the following outline.
\begin{itemize}[leftmargin=*]
\item The first estimate in Lemma \ref{lemma:KPZ12} follows by noting $\mathbb{T}_{N}$ is centered at $0\in\Z$ and $|\mathbb{T}_{N}|\gg N$ and thus $\mathbb{K}_{N}\subseteq\mathbb{T}_{N}$ if $N\gtrsim_{\mathbb{K}}1$.
\item Consider Proposition \ref{prop:KPZNL} and suppose that $\bar{\mathbf{Z}}^{N}$ were bounded; a fortiori it resembles the continuous SHE solution. Since the stochastic equations for $\bar{\mathbf{Z}}^{N}$ and $\mathbf{Y}^{N}$ differ only in the $\bar{\Phi}^{N,2}$-term of interest in Proposition \ref{prop:KPZNL}, with high-probability we would deduce this term is negligible in the large-$N$ limit and Lemma \ref{lemma:KPZ12} would follow by linear theory for equations on $\mathbb{T}_{N}$.
\item The previous remark depends on the assumption $\bar{\mathbf{Z}}^{N}$ is uniformly bounded on $\mathbb{T}_{N}$. This is not clear, nor will we try to show it. Note it is actually enough for $\bar{\mathbf{Z}}^{N}$ to be bounded by small powers of $N$ with high-probability. We will show this for $\mathbf{Y}^{N}$ via moment estimates and union bound as in the proof of Proposition \ref{prop:TRGTProp}. As $\mathbf{Y}^{N}$ and $\mathbf{Z}^{N}$ are supposed to be close we expect this to hold true for $\mathbf{Z}^{N}$ as well. To implement this, we use a continuity argument. In what follows $\e>0$ is thought of as small.
\item Consider the first time $\mathfrak{t}_{2\e}$ before 1 that $\bar{\mathbf{Z}}^{N}$ exceeds $N^{2\e}$ in the $\|\|_{\mathfrak{t}_{2\e};\mathbb{T}_{N}}$-norm. Proposition \ref{prop:KPZNL} tells us with high-probability the $\bar{\Phi}^{N,2}$-term is negligible, so $\bar{\mathbf{Z}}^{N}\approx\mathbf{Y}^{N}$ until time $\mathfrak{t}_{2\e}$. With high-probability we know $\mathbf{Y}^{N}$ is at most $N^{\e}$ in $\|\|_{\mathfrak{t}_{2\e};\mathbb{T}_{N}}$-norm, thus the same is true for $\bar{\mathbf{Z}}^{N}$. This means we can wait \emph{after} $\mathfrak{t}_{2\e}$ for $\bar{\mathbf{Z}}^{N}$ to exceed $N^{2\e}$. This propagates $\mathfrak{t}_{2\e}\to1$.
\item We emphasize our estimates for the $\bar{\Phi}^{N,2}$-term in Proposition \ref{prop:KPZNL} and for $\mathbf{Y}^{N}$ will not depend on the random time $\mathfrak{t}_{2\e}$, so even though this scheme may take a very large number of iterations, our comparison estimates for $\bar{\mathbf{Z}}^{N}$ and $\mathbf{Y}^{N}$ stay the same. 
\end{itemize}
We now make the previous outline precise and prove Lemma \ref{lemma:KPZ12}. This begins with the following construction of random times.
\begin{definition}\label{definition:KPZ12}
Consider any $\e_{\mathrm{st}} \in \R_{>0}$ arbitrarily small but universal. Define the random time $\mathfrak{t}_{1,\mathrm{st}} \overset{\bullet}= \mathfrak{t}_{1,\mathrm{st},1}\wedge\mathfrak{t}_{1,\mathrm{st},2}$, where, for some universal constant $\beta_{\mathrm{univ}} \in \R_{>0}$, we define the following in which the implied constant for $\mathfrak{t}_{1,\mathrm{st},2}$ is large but universal:
\begin{subequations}
\small\begin{align}
\mathfrak{t}_{1,\mathrm{st},1} \ &\overset{\bullet}= \ \inf\left\{\mathfrak{t}\in[0,1]: \ \|\bar{\mathbf{Z}}^{N}\|_{\mathfrak{t};\mathbb{T}_{N}} \geq N^{2\e_{\mathrm{st}}} \right\} \wedge 1 \\
\mathfrak{t}_{1,\mathrm{st},2} \ &\overset{\bullet}= \ \inf\left\{\mathfrak{t}\in[0,1]: \ \|\bar{\mathbf{H}}_{T,x}^{N}(\bar{\Phi}^{N,2})\|_{\mathfrak{t};\mathbb{T}_{N}} \gtrsim N^{-\beta_{\mathrm{univ}}} + N^{-\beta_{\mathrm{univ}}}\|\bar{\mathbf{Z}}^{N}\|_{\mathfrak{t};\mathbb{T}_{N}}^{2} \right\} \wedge 1.
\end{align}\normalsize\normalsize
\end{subequations}
Set $\mathbf{D}^{N} \overset{\bullet}= \mathbf{Y}^{N} - \mathbf{Q}^{N}$ where $\mathbf{Q}^{N}$ solves the stochastic equation on $\R_{\geq0}\times\mathbb{T}_{N}$ below with $\bar{\mathbf{Z}}_{0,\bullet}^{N}=\chi_{\bullet}\mathbf{Z}_{0,\bullet}^{N}$ from Definition \ref{definition:ChiTorus}:
\small\begin{align}
\mathbf{Q}_{T,x}^{N} \ &= \ \bar{\mathbf{H}}_{T,x}^{N,\mathbf{X}}(\bar{\mathbf{Z}}^{N}_{0,\bullet}) + \bar{\mathbf{H}}_{T,x}^{N}(\mathbf{Q}^{N}\d\xi^{N}) + \bar{\mathbf{H}}_{T,x}^{N}(\bar{\Phi}^{N,2}_{S,y}\mathbf{1}_{S\leq\mathfrak{t}_{1,\mathrm{st}}}) + \bar{\mathbf{H}}_{T,x}^{N}(\bar{\Phi}^{N,5}).
\end{align}\normalsize\normalsize
See Remark \ref{remark:ch2Jumps} for definition of $\mathbf{Q}^{N}\d\xi^{N}$. Recall $\bar{\Phi}^{N,2}$ in Definition \ref{definition:ChiTorus}. We have also introduced the following adaptation $\bar{\Phi}^{N,5}$ of $\bar{\Phi}^{N,3}$ in the defining equation for $\bar{\mathbf{Z}}^{N}$ but catered to $\mathbf{Q}^{N}$. Recall $\bar{\grad}^{!} = N\bar{\grad}$ where $\bar{\grad}$ is gradient on the torus $\mathbb{T}_{N}$:
\small\begin{align}
\bar{\Phi}^{N,5}_{T,x} \ &\overset{\bullet}= \ \mathfrak{w}_{T,x}\mathbf{Q}_{T,x}^{N} + {\sum}_{k=-2\mathfrak{m}}^{2\mathfrak{m}} c_{k}\bar{\grad}_{k}^{!}(\mathfrak{w}_{T,x}^{k}\mathbf{Q}_{T,x}^{N}). 
\end{align}\normalsize\normalsize
\end{definition}
We will eventually pick $\beta_{\mathrm{univ}}\in\R_{>0}$ small depending only on universal constants including $\beta_{\mathrm{univ},2}\in\R_{>0}$ from Proposition \ref{prop:KPZNL}. We will then choose $\e_{\mathrm{st}}\in\R_{>0}$ sufficiently small depending only on $\beta_{\mathrm{univ}}\in\R_{>0}$. To be concrete, the reader is welcome to take $\e_{\mathrm{st}} = 99^{-99}\wedge99^{-99}\beta_{\mathrm{univ}}$ and $\beta_{\mathrm{univ}} = 99^{-99}\beta_{\mathrm{univ},2}$ with $\beta_{\mathrm{univ},2}\in\R_{>0}$ from the statement of Proposition \ref{prop:KPZNL}.

We first provide the following pathwise result which relates $\mathbf{Q}^{N}$ to $\bar{\mathbf{Z}}^{N}$ relative to the random time $\mathfrak{t}_{1,\mathrm{st}}$. Roughly speaking, it follows from the observation that $\bar{\mathbf{Z}}^{N}$ and $\mathbf{Q}^{N}$ solve the same defining stochastic linear equation before time $\mathfrak{t}_{1,\mathrm{st}}$.
\begin{lemma}\label{lemma:KPZ121}
 We have $\bar{\mathbf{Z}}^{N}_{\mathfrak{s},x}=\mathbf{Q}_{\mathfrak{s},x}^{N}$ for all $0\leq\mathfrak{s}\leq\mathfrak{t}_{1,\mathrm{st}}$ and $x\in\mathbb{T}_{N}$ with probability 1.
\end{lemma}
Roughly speaking, the result Lemma \ref{lemma:KPZ122} below provides $\|\|_{1;\mathbb{T}_{N}}$-estimates for the difference process $\mathbf{D}^{N}=\mathbf{Y}^{N}-\mathbf{Q}^{N}$. Note $\mathbf{D}^{N}$ solves a linear equation like that defining $\mathbf{Y}^{N}$ but with the extra $\bar{\Phi}^{N,2}$-contribution in the $\mathbf{Q}^{N}$ equation and vanishing initial condition since the initial data of $\mathbf{Y}^{N}$ and $\mathbf{Q}^{N}$ coincide. The proposed bound in the first result below will be shown by standard moment bounds to get pointwise control on $\mathbf{D}^{N}$ and a stochastic continuity argument like with the proof of Proposition \ref{prop:TRGTProp}. In the proof for Lemma \ref{lemma:KPZ122} we rely heavily on the fact that we have cutoff the $\bar{\Phi}^{N,2}$-contribution in $\mathbf{Q}^{N}$ by the upper bound on the RHS of the estimate in Proposition \ref{prop:KPZNL} so we have \emph{deterministic} control on the $\bar{\Phi}^{N,2}$-contribution in $\mathbf{Q}^{N}$. Afterwards, we record a second result Lemma \ref{lemma:KPZ123} that shows the $\mathfrak{t}_{1,\mathrm{st}}$-cutoff in the $\mathbf{Q}^{N}$-equation can be ignored with high-probability.
\begin{lemma}\label{lemma:KPZ122}
 For a $\beta_{\mathrm{univ},2}>0$ universal and any $C>0$, with probability at least $1 - \kappa_{C}N^{-C}$ we have $\|\mathbf{D}^{N}\|_{1;\mathbb{T}_{N}}\lesssim N^{-\beta_{\mathrm{univ},2}}$.
\end{lemma}
\begin{lemma}\label{lemma:KPZ123}
 There exists $\beta_{\mathrm{univ},1}>0$ universal so that with probability at least $1 - \kappa N^{-\beta_{\mathrm{univ},1}}$, we have $\mathfrak{t}_{1,\mathrm{st}} = 1$ provided $\beta_{\mathrm{univ}}>0$ in the definition of $\mathfrak{t}_{1,\mathrm{st}}$ is sufficiently small and $\e_{\mathrm{st}} \in \R_{>0}$ is sufficiently small depending only on our choice of $\beta_{\mathrm{univ}}$.
\end{lemma}
\begin{proof}[Proof of \emph{Lemma \ref{lemma:KPZ12}}]
Lemma \ref{lemma:KPZ122} gives Lemma \ref{lemma:KPZ12} if we replace $\bar{\mathbf{Z}}^{N}$ by $\mathbf{Q}^{N}$ in the statement of Lemma \ref{lemma:KPZ12}. But Lemma \ref{lemma:KPZ123} says that with the required high-probability, such replacement can be removed as $\bar{\mathbf{Z}}^{N}=\mathbf{Q}^{N}$ with high-probability on $[0,1]\times\mathbb{T}_{N}$ by pathwise identification in Lemma \ref{lemma:KPZ121}. This completes the proof of Lemma \ref{lemma:KPZ12} modulo proofs of Lemmas \ref{lemma:KPZ121}, \ref{lemma:KPZ122}, and \ref{lemma:KPZ123}. We defer these to the final subsection as they are on the technical side.
\end{proof}
\subsection{Proof of Proposition \ref{prop:KPZ2}}
Observe the defining stochastic linear equation for $\mathbf{Y}^{N}$ contains two terms which resemble a microscopic SHE and another that contains only weakly vanishing terms. Thus, the stochastic linear equation for $\mathbf{Y}^{N}$ resembles that of the microscopic Cole-Hopf transform within \cite{DT}. We are now in a position to follow the proof of Proposition 2.2 in \cite{DT} and thus the proof of Theorem 1.1 in \cite{DT}. We explain this in more detail below.
\begin{itemize}[leftmargin=*]
\item Tightness of the microscopic Cole-Hopf transform under the appropriate space-time scaling in \cite{DT} is stated in Proposition 1.4 in \cite{DT}. It uses two ingredients. The first is Corollary 3.3 in \cite{DT} and the second is boundedness of a Poisson clock speed. For the former, the list of moment estimates in Corollary 3.3 in \cite{DT} also holds for the $\mathbf{Y}^{N}$-process here because it only needs the defining linear equation (3.2) in \cite{DT} for the microscopic Cole-Hopf transform. The defining linear equation for $\mathbf{Y}^{N}$ has the same structure as that in (3.2) in \cite{DT}. We also use Lemma \ref{lemma:MG} instead of Lemma 3.1 in \cite{DT} here. The second ingredient concerning Poisson clocks speeds also holds in our case; in \cite{DT} the maximal jump-length did not play an important role.
\item To get convergence of the microscopic Cole-Hopf transform to the continuum SHE in \cite{DT}, the approach taken therein is via the martingale problem. To use ideas in \cite{DT}, we require two ingredients. The first is spatial regularity of $\mathbf{Y}^{N}$ on microscopic length-scales to treat the correlations in $\d\xi^{N}$ at different spatial points that come from the non-simple nature of the particle system. For the microscopic Cole-Hopf transform this was done via explicit formulas in terms of the particle system in \cite{DT}, but the regularity from Corollary 3.3 in \cite{DT}, which holds for $\mathbf{Y}^{N}$ as noted in the last bullet point, suffices for this as well. The second ingredient we require is a hydrodynamic limit argument for analysis of weakly vanishing functionals at macroscopic space-time scales; see Lemma 2.5 from \cite{DT}. Our version is Lemma \ref{lemma:KPZ21}.
\item The analysis in \cite{DT} is done on $\R_{\geq0}\times\Z$, while $\mathbf{Y}^{N}$ is defined on $\R_{\geq0}\times\mathbb{T}_{N}$. This does not matter as our heat kernel estimates in Lemma \ref{lemma:HKE} hold for $\bar{\mathbf{H}}^{N}$. We may also follow the proof of Proposition \ref{prop:Ctify} but for $\mathbf{Y}^{N}$ to replace $\mathbf{Y}^{N}$ by its $\R_{\geq0}\times\Z$-version.
\end{itemize}
\begin{lemma}\label{lemma:KPZ21}
 Consider a weakly vanishing functional $\mathfrak{w}:\Omega_{\Z}\to\R$. Provided any smooth test function $\varphi:\R\to\R$ with compact support, for $\mathfrak{u}\in\{1,2\}$ and for any fixed $T \in \R_{\geq0}$ independent of $N\in\Z_{>0}$, we have 
\small\begin{align}
{\lim}_{N\to\infty}\frac{1}{N}\int_{0}^{T}\sum_{y\in\Z}\varphi_{N^{-1}y} \cdot \mathfrak{w}_{S,y} \cdot |\mathbf{Y}_{S,y}^{N}|^{\mathfrak{u}} \ \d S \ = \ 0.
\end{align}\normalsize\normalsize
\end{lemma}
We defer the proof of Lemma \ref{lemma:KPZ21} to the last subsection along with other technical results. We remark that the proof of Lemma \ref{lemma:KPZ21} follows that of Lemma 2.5 in \cite{DT}, as we have spatial regularity of $\mathbf{Y}^{N}$ needed to replace $\mathfrak{w}$ by mesoscopic spatial average and we have the entropy production estimate in Proposition \ref{prop:EProd} needed to run the one-block/two-blocks steps in the proof for Lemma 2.5 in \cite{DT}. There is a caveat. In the proof of Lemma 2.5 in \cite{DT}, density fluctuations are controlled by regularity of the microscopic Cole-Hopf transform. Here it is also given by the \emph{original} microscopic Cole-Hopf transform $\mathbf{Z}^{N}$, not $\mathbf{Y}^{N}$, as only $\mathbf{Z}^{N}$ is defined via the particle system. On the other hand, what we need from $\mathbf{Z}^{N}$ can be estimated via $\mathbf{Y}^{N}$ by Proposition \ref{prop:KPZ1}.
\subsection{Proofs of Technical Estimates}
It will be helpful to have the following a priori moment estimate for $\mathbf{Y}^{N}$. Its proof was described in the subsection concerning the proof of Proposition \ref{prop:KPZ2} in the bullet point about tightness of $\mathbf{Y}^{N}$ so we omit it.
\begin{lemma}\label{lemma:APY}
 Consider any $p \in \R_{\geq 1}$. We have the moment estimate $\sup_{T\in[0,1]}\sup_{x\in\mathbb{T}_{N}} \|\mathbf{Y}_{T,x}^{N}\|_{\omega;2p}\lesssim_{p}1$.
\end{lemma}
\begin{corollary}\label{corollary:APY}
 Consider any $\e,C\in\R_{>0}$. Outside an event with probability at most $\kappa_{\e,C}N^{-C}$, we have $\|\mathbf{Y}^{N}\|_{1;\mathbb{T}_{N}}\lesssim N^{\e}$.
\end{corollary}
\begin{proof}
In SDE-type equation formulation, for times in $[0,1]$ we get $\d\mathbf{Y}^{N} = \mathscr{Q}\mathbf{Y}^{N}\d T$ with jumps of order $N^{-1/2}\|\mathbf{Y}^{N}\|_{1;\mathbb{T}_{N}}$ where $\mathscr{Q}$ satisfies the operator bounds for $\mathscr{Q}$ in the proof of Lemma \ref{lemma:TRGTProp2}. Indeed, the only difference between the two operators is a term that is multiplicative in the solution $\mathbf{Y}^{N}$ or $\bar{\mathbf{Z}}^{N}$ with coefficient $\bar{\Phi}^{N,2}$ of order $N^{1/2}$. Thus, the proof of Lemma \ref{lemma:TRGTProp2} applies to $\mathbf{Y}^{N}$ in place of $\bar{\mathbf{Z}}^{N}$. This controls $\|\|_{1;\mathbb{T}_{N}}$ by $[]_{1;\mathbb{T}_{N}}$ with the required high-probability as with Lemma \ref{lemma:TRGTProp3}: 
\small\begin{align}
\|\mathbf{Y}^{N}\|_{1;\mathbb{T}_{N}} \ &\lesssim \ [\mathbf{Y}^{N}]_{1;\mathbb{T}_{N}}. \label{eq:APY1}
\end{align}\normalsize\normalsize
Indeed, the only way for \eqref{eq:APY1} to fail is $\mathbf{Y}^{N}$ to have large change between times in the discretization $\{\mathfrak{j}N^{-100}\}_{\mathfrak{j}=0}^{N^{100}}$. Lemma \ref{lemma:TRGTProp2} though for $\mathbf{Y}^{N}$ prevents that with high-probability. Thus it is left to show Corollary \ref{corollary:APY} upon replacing $\|\mathbf{Y}^{N}\|_{1;\mathbb{T}_{N}}$ by $[\mathbf{Y}^{N}]_{1;\mathbb{T}_{N}}$. To this end, a union bound allows us to control the probability where $[\mathbf{Y}^{N}]_{1;\mathbb{T}_{N}} \gtrsim N^{\e}$ by the probability for which $\mathbf{Y}^{N}_{\mathfrak{t},x}\gtrsim N^{\e}$, then multiplied by $N^{105}$. The one-point probabilities are at most $\kappa_{\e,C'}N^{-C'}$ for any $C'>0$ by the Chebyshev inequality for high $p$-moments and the moment estimate in Lemma \ref{lemma:APY} as in the proofs of Proposition \ref{prop:Ctify}/Lemma \ref{lemma:C0SIProp}.
\end{proof}
\begin{proof}[Proof of \emph{Lemma \ref{lemma:KPZ11}}]
If $N$ is sufficiently large depending only on $\mathbb{K}\subseteq\R$, note $\mathbb{K}_{N}=N\mathbb{K}\cap\Z\subseteq\wt{\mathbb{T}}_{N}$, where $\wt{\mathbb{T}}_{N}$ is defined in Proposition \ref{prop:Ctify}. Thus, it suffices to get the following comparison with high-probability in which $\beta_{\mathrm{univ}}\in\R_{>0}$ is universal:
\small\begin{align}
\|\mathbf{Z}^{N}-\bar{\mathbf{Z}}^{N}\|_{1;\mathbb{K}_{N}} \ &\lesssim \ [\mathbf{Z}^{N}-\bar{\mathbf{Z}}^{N}]_{1;\mathbb{K}_{N}} \ + \ N^{-\beta_{\mathrm{univ}}}. \label{eq:KPZ111}
\end{align}\normalsize\normalsize
Indeed, Proposition \ref{prop:Ctify} tells us the first term on the RHS of \eqref{eq:KPZ111} is at most $\kappa_{C}N^{-C}$ with high-probability as $\mathbb{K}_{N}\subseteq\wt{\mathbb{T}}_{N}$. Observe the LHS of \eqref{eq:KPZ111} is controlled by the first term on the RHS of \eqref{eq:KPZ111} and behavior of $\mathbf{Z}^{N}$ and $\bar{\mathbf{Z}}^{N}$ between times in $\{\mathfrak{j}N^{-100}\}_{\mathfrak{j}=0}^{N^{100}}$. Precisely, we get a deterministic bound by the triangle inequality; recall time-gradients defined prior to Proposition \ref{prop:TRGTProp}:
\small\begin{align}
\|\mathbf{Z}^{N}-\bar{\mathbf{Z}}^{N}\|_{1;\mathbb{K}_{N}} \ &\lesssim \ [\mathbf{Z}^{N}-\bar{\mathbf{Z}}^{N}]_{1;\mathbb{K}_{N}} \ + \ \sup_{0\leq\mathfrak{t}\leq1}\sup_{0\leq\mathfrak{s}\leq N^{-100}}\sup_{x\in\mathbb{K}_{N}} |\grad_{\mathfrak{s}}^{\mathbf{T}}\mathbf{Z}_{\mathfrak{t},x}^{N}| \ + \ \sup_{0\leq\mathfrak{t}\leq1}\sup_{0\leq\mathfrak{s}\leq N^{-100}}\sup_{x\in\mathbb{K}_{N}} |\grad_{\mathfrak{s}}^{\mathbf{T}}\bar{\mathbf{Z}}_{\mathfrak{t},x}^{N}|. \label{eq:KPZ112}
\end{align}\normalsize\normalsize
To get \eqref{eq:KPZ111} with the required high-probability, we estimate the second and third terms on the RHS of \eqref{eq:KPZ112}.
\begin{itemize}[leftmargin=*]
\item We cite (2.3) in \cite{DT} for the following stochastic equation for the microscopic Cole-Hopf transform $\mathbf{Z}^{N}$ that is a starting point for the proof of the stochastic equation within Proposition \ref{prop:Duhamel} we have been employing until this point. We emphasize that this holds for  $\mathbf{Z}^{N}$ in this paper because the proof of equation (2.3) in \cite{DT} does not depend on the maximal jump-length:
\small\begin{align}
(\mathbf{Z}_{T,x}^{N})^{-1}\d\mathbf{Z}_{T,x}^{N} \ &= \ \mathfrak{f}_{T,x}^{N}\d T \ + \ \d\xi_{T,x}^{N}.
\end{align}\normalsize\normalsize
Above the functional $\mathfrak{f}^{N}$ satisfies $|\mathfrak{f}^{N}|\lesssim N^{2}$ and $\d\xi^{N}$ is the same martingale integrator that is in Proposition \ref{prop:Duhamel}. In words, the dynamic of the original microscopic Cole-Hopf transform $\mathbf{Z}^{N}$ is multiplicative  in $\mathbf{Z}^{N}$ \emph{at the same point} with a continuous part that has speed of order $N^{2}$ and jumps of speed $N^{2}$ that are of order $N^{-1/2}$. This is in contrast to $\bar{\mathbf{Z}}^{N}$, because the only dynamic equation we have for $\bar{\mathbf{Z}}^{N}$ has the discrete-type Laplacian $\bar{\mathscr{L}}^{!!}$ that relates the growth of $\bar{\mathbf{Z}}^{N}$ to its values at neighboring points. We use reasoning in the proof of Lemma \ref{lemma:TRGTProp2} to get the following high-probability short-time $\|\|_{1;\mathbb{K}_{N}}$-bound for any $\e\in\R_{>0}$:
\small\begin{align}
\sup_{0\leq\mathfrak{t}\leq1}\sup_{0\leq\mathfrak{s}\leq N^{-100}}\sup_{x\in\mathbb{K}_{N}} |\grad_{\mathfrak{s}}^{\mathbf{T}}\mathbf{Z}_{\mathfrak{t},x}^{N}| \ &\lesssim \ N^{-\frac12+\e}\|\mathbf{Z}^{N}\|_{1;\mathbb{K}_{N}}. \label{eq:KPZ113}
\end{align}\normalsize\normalsize
\item We now use Lemma \ref{lemma:TRGTProp2} as written to get, also with the required high-probability, a short-time bound now in terms of $\|\|_{1;\mathbb{T}_{N}}$:
\small\begin{align}
\sup_{0\leq\mathfrak{t}\leq1}\sup_{0\leq\mathfrak{s}\leq N^{-100}}\sup_{x\in\mathbb{K}_{N}} |\grad_{\mathfrak{s}}^{\mathbf{T}}\bar{\mathbf{Z}}_{\mathfrak{t},x}^{N}| \ &\lesssim \ N^{-\frac12+\e}\|\bar{\mathbf{Z}}^{N}\|_{1;\mathbb{T}_{N}}. \label{eq:KPZ114}
\end{align}\normalsize\normalsize
Via Lemma \ref{lemma:KPZ12} and Corollary \ref{corollary:APY} and the triangle inequality, we estimate $\bar{\mathbf{Z}}^{N}$ by comparing it to $\mathbf{Y}^{N}$ and then by estimating $\mathbf{Y}^{N}$. Ultimately, with the required high-probability we get the following for some universal constant $\beta_{\mathrm{univ},2}$:
\small\begin{align}
\|\bar{\mathbf{Z}}^{N}\|_{1;\mathbb{T}_{N}} \ \leq \ \|\bar{\mathbf{Z}}^{N}-\mathbf{Y}^{N}\|_{1;\mathbb{T}_{N}} \ + \ \|\mathbf{Y}^{N}\|_{1;\mathbb{T}_{N}} \ &\lesssim \ N^{-\beta_{\mathrm{univ},2}} \ + \ N^{\e}. \label{eq:KPZ115}
\end{align}\normalsize\normalsize
\end{itemize}
We now combine \eqref{eq:KPZ112}, \eqref{eq:KPZ113}, \eqref{eq:KPZ114}, and \eqref{eq:KPZ115} to deduce the following for which we give more explanation after:
\small\begin{align}
\|\mathbf{Z}^{N}-\bar{\mathbf{Z}}^{N}\|_{1;\mathbb{K}_{N}} \ &\lesssim \ [\mathbf{Z}^{N}-\bar{\mathbf{Z}}^{N}]_{1;\mathbb{K}_{N}} \ + \ N^{-\frac12+\e}\|\mathbf{Z}^{N}\|_{1;\mathbb{K}_{N}} \ + \ N^{-\frac12+\e}\|\bar{\mathbf{Z}}^{N}\|_{1;\mathbb{T}_{N}} \\
 &\lesssim \ [\mathbf{Z}^{N}-\bar{\mathbf{Z}}^{N}]_{1;\mathbb{K}_{N}} \ + \ N^{-\frac12+\e}\|\mathbf{Z}^{N}-\bar{\mathbf{Z}}^{N}\|_{1;\mathbb{K}_{N}} \ + \ N^{-\frac12+\e}\|\bar{\mathbf{Z}}^{N}\|_{1;\mathbb{T}_{N}} \\
 &\lesssim \ [\mathbf{Z}^{N}-\bar{\mathbf{Z}}^{N}]_{1;\mathbb{K}_{N}} \ + \ N^{-\frac12+\e}\|\mathbf{Z}^{N}-\bar{\mathbf{Z}}^{N}\|_{1;\mathbb{K}_{N}} \ + \ N^{-\frac12+\e}N^{-\beta_{\mathrm{univ},2}} \ + \ N^{-\frac12+2\e}. \label{eq:KPZ116}
\end{align}\normalsize\normalsize
The first line in the above calculation follows from directly applying \eqref{eq:KPZ113} and \eqref{eq:KPZ114} to \eqref{eq:KPZ112}. The second line follows from the triangle inequality $\|\mathbf{Z}^{N}\|_{1;\mathbb{K}_{N}} \leq \|\mathbf{Z}^{N}-\bar{\mathbf{Z}}^{N}\|_{1;\mathbb{K}_{N}} + \|\bar{\mathbf{Z}}^{N}\|_{1;\mathbb{K}_{N}}$ combined with the bound $\|\bar{\mathbf{Z}}^{N}\|_{1;\mathbb{K}_{N}}\leq\|\bar{\mathbf{Z}}^{N}\|_{1;\mathbb{T}_{N}}$. The last bound \eqref{eq:KPZ116} follows from employing \eqref{eq:KPZ115}. We now move the second term within \eqref{eq:KPZ116} to the LHS of the first line in this previous set of estimates and choose $\e\in\R_{>0}$ small to deduce \eqref{eq:KPZ111} and complete the proof.
\end{proof}
\begin{proof}[Proof of \emph{Lemma \ref{lemma:KPZ121}}]
Observe until time $\mathfrak{t}_{1,\mathrm{st}} \in \R_{\geq0}$, the difference $\Phi = \bar{\mathbf{Z}}^{N}-\mathbf{Q}^{N}$ solves the same equation as $\bar{\mathbf{Z}}^{N}$ upon replacing every factor of $\bar{\mathbf{Z}}^{N}$ with $\Phi$. We emphasize $\Phi$ has vanishing initial condition. Thus, because $\Phi$ solves a linear evolution equation, whose solutions are unique with probability 1, and a vanishing initial condition until $\mathfrak{t}_{1,\mathrm{st}}$, we have $\Phi = 0$ until $\mathfrak{t}_{1,\mathrm{st}}$.
\end{proof}
\begin{proof}[Proof of \emph{Lemma \ref{lemma:KPZ122}}]
The first step we take to control $\mathbf{D}^{N}$ is the following stochastic equation for $\mathbf{D}^{N}$:
\small\begin{align}
\mathbf{D}_{T,x}^{N} \ &= \ \bar{\mathbf{H}}_{T,x}^{N}(\mathbf{D}^{N}\d\xi^{N}) - \bar{\mathbf{H}}_{T,x}^{N}(\bar{\Phi}_{S,y}^{N,2}\mathbf{1}_{S\leq\mathfrak{t}_{1,\mathrm{st}}}) + \bar{\mathbf{H}}_{T,x}^{N}(\bar{\Phi}^{N,6}).
\end{align}\normalsize\normalsize
The last term in the $\mathbf{D}^{N}$-equation is the following relevant weakly-vanishing content from Proposition \ref{prop:Duhamel} but catered to $\mathbf{D}^{N}$:
\small\begin{align}
\bar{\Phi}_{T,x}^{N,6} \ &\overset{\bullet}= \ \mathfrak{w}_{T,x}\mathbf{D}_{T,x}^{N} + {\sum}_{k=-2\mathfrak{m}}^{2\mathfrak{m}} c_{k}\bar{\grad}_{k}^{!}(\mathfrak{w}_{T,x}^{k}\mathbf{D}_{T,x}^{N}).
\end{align}\normalsize\normalsize
This equation for $\mathbf{D}^{N}$ follows by subtracting stochastic equations for $\mathbf{Q}^{N}$ and $\mathbf{Y}^{N}$. We now proceed with the following steps.
\begin{itemize}[leftmargin=*]
\item Observe $\d\mathbf{D}^{N} = \mathscr{Q}\mathbf{D}^{N}\d T + \mathbf{1}_{T\leq\mathfrak{t}_{1,\mathrm{st}}}\bar{\Phi}^{N,2}\d T$ with jumps of order $N^{-1/2}\mathbf{D}^{N}$. Here the operator $\mathscr{Q}$ satisfies the same operator-estimates as that within the proof for Lemma \ref{lemma:TRGTProp2}. Indeed, the operator $\mathscr{Q}$ here is that from the proof of Corollary \ref{corollary:APY}, so the only difference between $\mathscr{Q}$ in here and that in the proof for Lemma \ref{lemma:TRGTProp2} is the lack of a multiplicative term in the solution with pseudo-gradient coefficients of order $N^{1/2}$. In particular, Lemma \ref{lemma:TRGTProp2} holds for $\mathbf{D}^{N}$ as well except now in short time-intervals we also get short-time contribution of $\bar{\Phi}^{N,2}$ for times before $\mathfrak{t}_{1,\mathrm{st}}$. However, before time $\mathfrak{t}_{1,\mathrm{st}}$, we control $\bar{\Phi}^{N,2}$ by $\bar{\mathbf{Z}}^{N}$ times small powers of $N$, which is then controlled via small powers of $N$. Thus, the contribution of $\bar{\Phi}^{N,2}$ before time $\mathfrak{t}_{1,\mathrm{st}}$ for short times is at most $N^{10}$ times the length $N^{-100}$ of these short time-windows. With the required high-probability, as in the proof of \eqref{eq:APY1} we have
\small\begin{align}
\|\mathbf{D}^{N}\|_{1;\mathbb{T}_{N}} \ &\lesssim \ [\mathbf{D}^{N}]_{1;\mathbb{T}_{N}} \ + \ N^{-90}. \label{eq:KPZ1221}
\end{align}\normalsize\normalsize
By \eqref{eq:KPZ1221} it is left to get Lemma \ref{lemma:KPZ122} upon replacing $\|\mathbf{D}^{N}\|_{1;\mathbb{T}_{N}}$ by $[\mathbf{D}^{N}]_{1;\mathbb{T}_{N}}$. As in proofs for Proposition \ref{prop:Ctify}/Corollary \ref{corollary:APY}, it suffices to bound pointwise moments of $\mathbf{D}^{N}$. Precisely, it suffices to show, for any $p\geq1$ and some $\beta>0$ universal, that
\small\begin{align}
{\sup}_{\mathfrak{t}\in[0,1]}{\sup}_{x\in\mathbb{T}_{N}}\|\mathbf{D}_{\mathfrak{t},x}^{N}\|_{\omega;2p} \ &\lesssim_{p} \ N^{-\beta}. \label{eq:KPZ1222}
\end{align}\normalsize\normalsize
\end{itemize}
We turn to the stochastic heat operator equation for $\mathbf{D}^{N}$ at the beginning of this proof. First, this stochastic equation gives
\small\begin{align}
\|\mathbf{D}^{N}_{T,x}\|_{\omega;2p}^{2} \ &\lesssim \ \| \bar{\mathbf{H}}_{T,x}^{N}(\bar{\Phi}_{S,y}^{N,2}\mathbf{1}_{S\leq\mathfrak{t}_{1,\mathrm{st}}})\|_{\omega;2p}^{2} \ + \ \|\bar{\mathbf{H}}_{T,x}^{N}(\mathbf{D}^{N}\d\xi^{N})\|_{\omega;2p}^{2} \ + \ \| \bar{\mathbf{H}}_{T,x}^{N}(\bar{\Phi}^{N,6})\|_{\omega;2p}^{2}. \label{eq:KPZ1223}
\end{align}\normalsize\normalsize
We now estimate the first term on the RHS of \eqref{eq:KPZ1223}. To this end, we apply the deterministic bound in Lemma \ref{lemma:HKEC} to get
\small\begin{align}
\|\bar{\mathbf{H}}_{T,x}^{N}(\bar{\Phi}_{S,y}^{N,2}\mathbf{1}_{S\leq\mathfrak{t}_{1,\mathrm{st}}})\|_{1;\mathbb{T}_{N}} \ \leq \ \|\bar{\mathbf{H}}_{T,x}^{N}(\bar{\Phi}_{S,y}^{N,2})\|_{\mathfrak{t}_{1,\mathrm{st}};\mathbb{T}_{N}} \ \lesssim \ N^{-\beta_{\mathrm{univ}}+2\e_{\mathrm{st}}} \ &\lesssim \ N^{-\frac12\beta_{\mathrm{univ}}}. \label{eq:KPZ1224}
\end{align}\normalsize\normalsize
The second bound in \eqref{eq:KPZ1224} follows by definition of $\mathfrak{t}_{1,\mathrm{st}}$ in Definition \ref{definition:KPZ12}. The third bound in \eqref{eq:KPZ1224} follows by recalling $\e_{\mathrm{st}}$ can be chosen small depending only on $\beta_{\mathrm{univ}}$, for example $\e_{\mathrm{st}}\leq99^{-99}\beta_{\mathrm{univ}}$. Using the deterministic bound \eqref{eq:KPZ1224} with \eqref{eq:KPZ1223},
\small\begin{align}
\|\mathbf{D}^{N}_{T,x}\|_{\omega;2p}^{2} \ &\lesssim \ N^{-\beta_{\mathrm{univ}}} \ + \ \|\bar{\mathbf{H}}_{T,x}^{N}(\mathbf{D}^{N}\d\xi^{N})\|_{\omega;2p}^{2} \ + \ \| \bar{\mathbf{H}}_{T,x}^{N}(\bar{\Phi}^{N,6})\|_{\omega;2p}^{2}. \label{eq:KPZ1226}
\end{align}\normalsize\normalsize
At this point, we follow the proof for (3.12) in Proposition 3.2 in \cite{DT}, as we did in the proof of Lemma \ref{lemma:CtifySFS1}, to get \eqref{eq:KPZ1222} from \eqref{eq:KPZ1226}; we again use the martingale inequality in Lemma \ref{lemma:MG} instead of Lemma 3.1 in \cite{DT}. This completes the proof.
\end{proof}
\begin{proof}[Proof of \emph{Lemma \ref{lemma:KPZ123}}]
We first start with constructing some random times in addition to the previous $\mathfrak{t}_{1,\mathrm{st}}$-random time.
\begin{definition}
Recall $\e_{\mathrm{st}}\in\R_{\geq0}$ from the construction of $\mathfrak{t}_{1,\mathrm{st}}$. Define the random time $\mathfrak{t}_{2,\mathrm{st}} = \mathfrak{t}_{2,\mathrm{st},1}\wedge\mathfrak{t}_{2,\mathrm{st},2}$, where
\begin{subequations}
\small\begin{align}
\mathfrak{t}_{2,\mathrm{st},1} \ &\overset{\bullet}= \ \inf\left\{\mathfrak{t}\in[0,1]: \ \|\mathbf{Y}^{N}\|_{\mathfrak{t};\mathbb{T}_{N}} \geq N^{\e_{\mathrm{st}}} \right\} \wedge 1 \\
\mathfrak{t}_{2,\mathrm{st},2} \ &\overset{\bullet}= \ \inf\left\{\mathfrak{t}\in[0,1]: \ \sup_{0 \leq \mathfrak{s} \leq N^{-100}}\sup_{0\leq\mathfrak{t}_{0}\leq\mathfrak{t}}\sup_{x\in\mathbb{T}_{N}}\left(1+\|\bar{\mathbf{Z}}^{N}\|_{\mathfrak{t}_{0};\mathbb{T}_{N}}^{2}\right)^{-1}|\grad_{\pm\mathfrak{s}}^{\mathbf{T}}\bar{\mathbf{Z}}_{\mathfrak{t}_{0},x}^{N}| \ \gtrsim \ N^{-\frac12+\e_{\mathrm{st}}} \right\} \wedge 1.
\end{align}\normalsize\normalsize
\end{subequations}
We additionally define the random time $\mathfrak{t}_{3,\mathrm{st}} \overset{\bullet}= \inf\{\mathfrak{t}\in[0,1]: \ \|\mathbf{D}^{N}\|_{\mathfrak{t};\mathbb{T}_{N}} \geq N^{\e_{\mathrm{st}}}\} \wedge 1$.
\end{definition}
The time $\mathfrak{t}_{2,\mathrm{st}}$ will both control $\mathbf{Y}^{N}$ and the maximal change in $\bar{\mathbf{Z}}^{N}$ over any short time-interval. The time $\mathfrak{t}_{3,\mathrm{st}}$ will control the difference between $\mathbf{Y}^{N}$ and $\mathbf{Q}^{N}$. Thus the minimum of $\mathfrak{t}_{2,\mathrm{st}}$ and $\mathfrak{t}_{3,\mathrm{st}}$ will control $\mathbf{Q}^{N}$. Moreover, with the pathwise identification in Lemma \ref{lemma:KPZ121}, this will control $\bar{\mathbf{Z}}^{N}$ until time $\mathfrak{t}_{1,\mathrm{st}}$. The short-time continuity estimate within $\mathfrak{t}_{2,\mathrm{st},2}$ will then help us to propagate this control on $\bar{\mathbf{Z}}^{N}$. We first show $\mathfrak{t}_{2,\mathrm{st}}$ and $\mathfrak{t}_{3,\mathrm{st}}$ are equal to 1 with high probability. We also do this for $\mathfrak{t}_{1,\mathrm{st},2}$ in Definition \ref{definition:KPZ12}.
\begin{itemize}[leftmargin=*]
\item For $\mathfrak{t}_{2,\mathrm{st}}$, we first note that Lemma \ref{lemma:TRGTProp2} guarantees $\mathbf{P}[\mathfrak{t}_{2,\mathrm{st},2}\neq1]\lesssim_{\e_{\mathrm{st}},C}N^{-C}$. Moreover, Corollary \ref{corollary:APY} guarantees the same is true for $\mathfrak{t}_{2,\mathrm{st},1}$. Thus, because the event $\mathfrak{t}_{2,\mathrm{st}}\neq1$ implies either $\mathfrak{t}_{2,\mathrm{st},1}\neq1$ or $\mathfrak{t}_{2,\mathrm{st},2}\neq1$ by definition, we get
\small\begin{align}
\mathbf{P}[\mathfrak{t}_{2,\mathrm{st}}\neq1] \ \leq \ \mathbf{P}[\mathfrak{t}_{2,\mathrm{st},1}\neq1] + \mathbf{P}[\mathfrak{t}_{2,\mathrm{st},2}\neq1] \ \lesssim_{\e_{\mathrm{st}},C} \ N^{-C}. \label{eq:KPZ1231}
\end{align}\normalsize\normalsize
\item Consequence of Lemma \ref{lemma:KPZ122}, we gain very high-probability control of $\|\mathbf{D}^{N}\|_{1;\mathbb{T}_{N}}$, and thus
\small\begin{align}
\mathbf{P}[\mathfrak{t}_{3,\mathrm{st}}\neq1] \ &\lesssim_{\e_{\mathrm{st}},C} \ N^{-C}. \label{eq:KPZ1232}
\end{align}\normalsize\normalsize
\item We use Proposition \ref{prop:KPZNL} with the random time $\mathfrak{t}_{\mathrm{st}} = \mathfrak{t}_{1,\mathrm{st},2}$. As $\mathfrak{t}_{1,\mathrm{st},2}\leq1$, if $\mathfrak{t}_{1,\mathrm{st},2}\neq1$ then $\mathfrak{t}_{1,\mathrm{st},2}<1$. If $\beta_{\mathrm{univ}}\in\R_{>0}$ in definition of $\mathfrak{t}_{1,\mathrm{st},2}$ is sufficiently small but still universal depending only on $\beta_{\mathrm{univ},2}\in\R_{>0}$ in Proposition \ref{prop:KPZNL} and $\mathfrak{t}_{1,\mathrm{st},2}<1$, then \eqref{eq:KPZNL} in Proposition \ref{prop:KPZNL} fails for $\mathfrak{t}_{\mathrm{st}}=\mathfrak{t}_{1,\mathrm{st},2}$, which happens with probability at most order $N^{-\beta_{\mathrm{univ},1}}$:
\small\begin{align}
\mathbf{P}[\mathfrak{t}_{1,\mathrm{st},2}\neq1] \ &\lesssim \ N^{-\beta_{\mathrm{univ},1}}. \label{eq:KPZ12321}
\end{align}\normalsize\normalsize
\end{itemize}
Thus, to prove Lemma \ref{lemma:KPZ123}, it suffices to show $\mathfrak{t}_{1,\mathrm{st},1} = 1$ conditioning on the event on which $\mathfrak{t}_{1,\mathrm{st},2},\mathfrak{t}_{2,\mathrm{st}},\mathfrak{t}_{3,\mathrm{st}}=1$. Precisely, because $\mathfrak{t}_{1,\mathrm{st}} = \mathfrak{t}_{1,\mathrm{st},1}\wedge\mathfrak{t}_{1,\mathrm{st},2}$, we get a union bound inequality that says if $\mathfrak{t}_{1,\mathrm{st}}\neq1$ then $\mathfrak{t}_{1,\mathrm{st},2}\neq1$ or if $\mathfrak{t}_{1,\mathrm{st},2}=1$ then $\mathfrak{t}_{1,\mathrm{st},1}\neq1$:
\small\begin{align}
\mathbf{P}[\mathfrak{t}_{1,\mathrm{st}}\neq1] \ &\leq \ \mathbf{P}[\mathfrak{t}_{1,\mathrm{st},2}\neq1] \ + \ \mathbf{P}[\mathfrak{t}_{1,\mathrm{st},2}=1, \mathfrak{t}_{1,\mathrm{st},1}\neq1]. \label{eq:KPZ1233}
\end{align}\normalsize\normalsize
We decompose the second probability on the RHS of \eqref{eq:KPZ1233} according to a similar observation for $\mathfrak{t}_{2,\mathrm{st}},\mathfrak{t}_{3,\mathrm{st}}$:
\small\begin{align}
\mathbf{P}[\mathfrak{t}_{1,\mathrm{st},2}=1, \mathfrak{t}_{1,\mathrm{st},1}\neq1] \ &\leq \ \mathbf{P}[\mathfrak{t}_{2,\mathrm{st}}\neq1] \ + \ \mathbf{P}[\mathfrak{t}_{3,\mathrm{st}}\neq1] \ + \ \mathbf{P}[\mathfrak{t}_{1,\mathrm{st},2},\mathfrak{t}_{2,\mathrm{st}},\mathfrak{t}_{3,\mathrm{st}}=1, \mathfrak{t}_{1,\mathrm{st},1}\neq1]. \label{eq:KPZ1234}
\end{align}\normalsize\normalsize
We combine the union bound inequalities \eqref{eq:KPZ1233} and \eqref{eq:KPZ1234} with \eqref{eq:KPZ1231}, \eqref{eq:KPZ1232}, and \eqref{eq:KPZ12321} to get
\small\begin{align}
\mathbf{P}[\mathfrak{t}_{1,\mathrm{st}}\neq1]  \ &\lesssim \ N^{-\beta_{\mathrm{univ},1}} \ + \ \mathbf{P}[\mathfrak{t}_{1,\mathrm{st},2},\mathfrak{t}_{2,\mathrm{st}},\mathfrak{t}_{3,\mathrm{st}}=1, \mathfrak{t}_{1,\mathrm{st},1}\neq1]. \label{eq:KPZ1235}
\end{align}\normalsize\normalsize
Thus, to prove Lemma \ref{lemma:KPZ123}/estimate the LHS of \eqref{eq:KPZ1235} above, it suffices to show the last probability is equal to 0 at least when $N$ is sufficiently large depending only on $\e_{\mathrm{st}},\beta_{\mathrm{univ},1}\in\R_{>0}$. We do this with a continuity argument that is classical in PDEs. 

Suppose $\mathfrak{t}_{1,\mathrm{st},1}\neq1$ while $\mathfrak{t}_{1,\mathrm{st},2},\mathfrak{t}_{2,\mathrm{st}},\mathfrak{t}_{3,\mathrm{st}}=1$. In this case $\mathfrak{t}_{1,\mathrm{st},1}<1$ since $\mathfrak{t}_{1,\mathrm{st},1}\leq1$ with probability 1. Thus, we can pick $0<\mathfrak{t}_{N}\leq N^{-100}$ such that $\mathfrak{t}_{1,\mathrm{st},1}+\mathfrak{t}_{N}\leq1$. We will show that $\|\bar{\mathbf{Z}}^{N}\|_{\mathfrak{t}_{1,\mathrm{st},1}+\mathfrak{t}_{N};\mathbb{T}_{N}} \lesssim N^{\e_{\mathrm{st}}}$ with universal implied constant with probability 1. If $N\in\Z_{>0}$ is sufficiently large depending only on $\e_{\mathrm{st}}>0$, then we get $\|\bar{\mathbf{Z}}^{N}\|_{\mathfrak{t}_{1,\mathrm{st},1}+\mathfrak{t}_{N};\mathbb{T}_{N}} \leq N^{3\e_{\mathrm{st}}/2}$. This would contradict the definition of $\mathfrak{t}_{1,\mathrm{st},1}$ conditioning on $\mathfrak{t}_{1,\mathrm{st},1}\neq1$; we could push $\mathfrak{t}_{1,\mathrm{st},1}$ past whatever we condition it to be if $\mathfrak{t}_{1,\mathrm{st},1}\neq1$ and $\mathfrak{t}_{1,\mathrm{st},2},\mathfrak{t}_{2,\mathrm{st}},\mathfrak{t}_{3,\mathrm{st}}=1$. This would show $\mathfrak{t}_{1,\mathrm{st},1}>\mathfrak{t}_{1,\mathrm{st},1}$, so $\mathfrak{t}_{1,\mathrm{st},1}\neq1$ and $\mathfrak{t}_{1,\mathrm{st},2},\mathfrak{t}_{2,\mathrm{st}},\mathfrak{t}_{3,\mathrm{st}}=1$ is impossible.

Because $\mathfrak{t}_{2,\mathrm{st}}=1$, from the definitions of $\mathfrak{t}_{2,\mathrm{st}} = 1$ and $\mathfrak{t}_{2,\mathrm{st},2}$ we get $\mathfrak{t}_{2,\mathrm{st},2}=1$ so $\bar{\mathbf{Z}}^{N}$ cannot change very much in short time, so given any $\mathfrak{s}\in[0,\mathfrak{t}_{N}]$ we establish the following continuity estimate. We recall $\mathfrak{t}_{N}\leq N^{-100}$ and $\e_{\mathrm{st}}\in\R_{>0}$ is arbitrarily small but universal. For example, we may pick $\e_{\mathrm{st}} = 99^{-99}\wedge99^{-99}\beta_{\mathrm{univ}}$ where $\beta_{\mathrm{univ}}>0$ is the universal exponent defining $\mathfrak{t}_{1,\mathrm{st},2}$:
\small\begin{align}
|\bar{\mathbf{Z}}^{N}_{\mathfrak{t}_{1,\mathrm{st},1}+\mathfrak{s},x}| \ &\lesssim \ |\bar{\mathbf{Z}}^{N}_{\mathfrak{t}_{1,\mathrm{st},1},x}| + N^{-\frac12+\e_{\mathrm{st}}} + N^{-\frac12+\e_{\mathrm{st}}}\|\bar{\mathbf{Z}}^{N}\|_{\mathfrak{t}_{1,\mathrm{st},1};\mathbb{T}_{N}}^{2} \nonumber \\
&\lesssim \ \|\bar{\mathbf{Z}}^{N}\|_{\mathfrak{t}_{1,\mathrm{st},1};\mathbb{T}_{N}}+N^{-\frac12+\e_{\mathrm{st}}} + N^{-\frac12+\e_{\mathrm{st}}}\|\bar{\mathbf{Z}}^{N}\|_{\mathfrak{t}_{1,\mathrm{st},1};\mathbb{T}_{N}}^{2}. \label{eq:KPZ1236}
\end{align}\normalsize\normalsize
The last bound getting the far RHS of \eqref{eq:KPZ1236} follows by trivially bounding the first term in the previous bound by its space-time supremum over $[0,\mathfrak{t}_{1,\mathrm{st},1}]\times\mathbb{T}_{N}$. We now observe that because $\mathfrak{t}_{1,\mathrm{st}} = \mathfrak{t}_{1,\mathrm{st},1}\wedge\mathfrak{t}_{1,\mathrm{st},2}$ while $\mathfrak{t}_{1,\mathrm{st},1}<1$ and $\mathfrak{t}_{1,\mathrm{st},2}=1$ by assumption on the event we are working on, we have $\mathfrak{t}_{1,\mathrm{st}}=\mathfrak{t}_{1,\mathrm{st},1}$. Thus the pathwise identification in Lemma \ref{lemma:KPZ121} gives
\small\begin{align}
\|\bar{\mathbf{Z}}^{N}\|_{\mathfrak{t}_{1,\mathrm{st},1};\mathbb{T}_{N}} \ = \ \|\mathbf{Q}^{N}\|_{\mathfrak{t}_{1,\mathrm{st},1};\mathbb{T}_{N}} \ \leq \ \|\mathbf{D}^{N}\|_{\mathfrak{t}_{1,\mathrm{st},1};\mathbb{T}_{N}} \ + \ \|\mathbf{Y}^{N}\|_{\mathfrak{t}_{1,\mathrm{st},1};\mathbb{T}_{N}} \ &\lesssim \ N^{\e_{\mathrm{st}}}. \label{eq:KPZ1237}
\end{align}\normalsize\normalsize
The first identity follows from Lemma \ref{lemma:KPZ121} and $\mathfrak{t}_{1,\mathrm{st}}=\mathfrak{t}_{1,\mathrm{st},1}$ as we noted earlier. The second statement follows from definition of $\mathbf{D}^{N} = \mathbf{Y}^{N}-\mathbf{Q}^{N}$ and the triangle inequality. The final estimate follows by the assumption that $\mathfrak{t}_{2,\mathrm{st}},\mathfrak{t}_{3,\mathrm{st}}=1$ which gives a priori control on $\mathbf{D}^{N}$ and $\mathbf{Y}^{N}$ in the $\|\|_{1;\mathbb{T}_{N}}$-norm; recall $\mathfrak{t}_{1,\mathrm{st},1}\leq1$. We now combine \eqref{eq:KPZ1236} and \eqref{eq:KPZ1237} to deduce
\small\begin{align}
|\bar{\mathbf{Z}}^{N}_{\mathfrak{t}_{1,\mathrm{st},1}+\mathfrak{s},x}| \ &\lesssim \ N^{\e_{\mathrm{st}}} + N^{-1/2+3\e_{\mathrm{st}}} \ \lesssim \ N^{\e_{\mathrm{st}}}. \label{eq:KPZ1238}
\end{align}\normalsize\normalsize
This final estimate \eqref{eq:KPZ1238} is deterministic and the implied constant is universal. Moreover, \eqref{eq:KPZ1236} and \eqref{eq:KPZ1237} are uniform over $\mathfrak{s}\in[0,\mathfrak{t}_{N}]$ and $x\in\mathbb{T}_{N}$, and thus so is the estimate \eqref{eq:KPZ1238}. Let us now observe that $\|\bar{\mathbf{Z}}^{N}\|_{\mathfrak{t}_{1,\mathrm{st},1}+\mathfrak{t}_{N};\mathbb{T}_{N}}$ is controlled by $\|\bar{\mathbf{Z}}^{N}\|_{\mathfrak{t}_{1,\mathrm{st},1};\mathbb{T}_{N}}$ plus the supremum over $\mathfrak{s}\in[0,\mathfrak{t}_{N}]$ and $x\in\mathbb{T}_{N}$ of the LHS of \eqref{eq:KPZ1238}, so $\|\bar{\mathbf{Z}}^{N}\|_{\mathfrak{t}_{1,\mathrm{st},1}+\mathfrak{t}_{N};\mathbb{T}_{N}} \lesssim N^{\e_{\mathrm{st}}}$ with universal implied constant; see \eqref{eq:KPZ1237} for control on $\|\bar{\mathbf{Z}}^{N}\|_{\mathfrak{t}_{1,\mathrm{st},1};\mathbb{T}_{N}}$. As noted before \eqref{eq:KPZ1236}, this completes the proof.
\end{proof}
\begin{proof}[Proof of \emph{Lemma \ref{lemma:KPZ21}}]
Like with the proof of Lemma 2.5 in \cite{DT}, we have the entropy production in Proposition \ref{prop:EProd} and spatial regularity of $\mathbf{Y}^{N}$, so the one-block, two-blocks scheme implies that it suffices to prove the following expectation estimate which is pointwise on compact space-time sets; the following LHS/term we bound resembles $\wt{U}_{13}$ in the proof of Lemma 2.5 in \cite{DT}:
\small\begin{align}
\sup_{\mathfrak{u}=1,2}\sup_{0\leq S\leq1}\sup_{|y| \lesssim_{\varphi} N}\E |\mathfrak{A}_{S,y,\delta}^{N}| |\mathbf{Y}_{S,y}^{N}|^{\mathfrak{u}} \ \overset{\bullet}= \ \sup_{\mathfrak{u}=1,2}\sup_{0\leq S\leq1}\sup_{|y| \lesssim_{\varphi} N}\E |\left(\wt{\sum}_{\mathfrak{j}=1,\ldots,\delta N^{1/2}}\eta_{S,y+\mathfrak{j}}\right)| |\mathbf{Y}_{S,y}^{N}|^{\mathfrak{u}} \ \lesssim_{\delta,\varphi} \ N^{-\beta_{\mathrm{univ}}}. \label{eq:KPZ1WV1}
\end{align}\normalsize\normalsize
All estimates in this argument will be allowed to depend on $\delta$ and $\varphi$. For $\sigma \in \R_{>0}$ arbitrarily small but universal, as $|\mathfrak{A}^{N}| \lesssim 1$, we consider the following decomposition which is similar to the $D_{1}$ and $D_{2}$ events from the proof of Lemma 2.5 in \cite{DT} but with refinements to address the relevance of both of $\mathbf{Z}^{N}$ and $\mathbf{Y}^{N}$ as well as refinements in the cutoff-exponents for technical reasons; see the final estimate \eqref{eq:KPZ1WV8} of this proof for why we will need to choose the cutoff-exponent $\sigma$ appearing below carefully:
\small\begin{align}
|\mathfrak{A}_{S,y,\delta}^{N}| |\mathbf{Y}_{S,y}^{N}|^{\mathfrak{u}} \ &\leq \ |\mathfrak{A}_{S,y,\delta}^{N}| |\mathbf{Y}_{S,y}^{N}|^{\mathfrak{u}} \mathbf{1}_{|\mathbf{Z}_{S,y}^{N}|\leq N^{-\sigma}}\mathbf{1}_{|\mathbf{Y}_{S,y}^{N}|\leq N^{-\sigma/2}} + |\mathfrak{A}_{S,y,\delta}^{N}| |\mathbf{Y}_{S,y}^{N}|^{\mathfrak{u}} \mathbf{1}_{|\mathbf{Z}_{S,y}^{N}|\leq N^{-\sigma}}\mathbf{1}_{|\mathbf{Y}_{S,y}^{N}|\geq N^{-\sigma/2}} + |\mathfrak{A}_{S,y,\delta}^{N}| |\mathbf{Y}_{S,y}^{N}|^{\mathfrak{u}} \mathbf{1}_{|\mathbf{Z}_{S,y}^{N}|\geq N^{-\sigma}} \nonumber \\
&\lesssim \ N^{-\frac12\sigma} + |\mathfrak{A}_{S,y,\delta}^{N}| |\mathbf{Y}_{S,y}^{N}|^{\mathfrak{u}} \mathbf{1}_{|\mathbf{Z}_{S,y}^{N}|\leq N^{-\sigma}}\mathbf{1}_{|\mathbf{Y}_{S,y}^{N}|\geq N^{-\sigma/2}} + |\mathfrak{A}_{S,y,\delta}^{N}| |\mathbf{Y}_{S,y}^{N}|^{\mathfrak{u}} \mathbf{1}_{|\mathbf{Z}_{S,y}^{N}|\geq N^{-\sigma}}.\label{eq:KPZ1WV!}
\end{align}\normalsize\normalsize
We first analyze the second term on the RHS of \eqref{eq:KPZ1WV!}. To this end, we first observe the following for $N \in \Z_{>0}$ sufficiently large:
\small\begin{align}
|\mathfrak{A}_{S,y,\delta}^{N}| |\mathbf{Y}_{S,y}^{N}|^{\mathfrak{u}} \mathbf{1}_{|\mathbf{Z}_{S,y}^{N}|\leq N^{-\sigma}}\mathbf{1}_{|\mathbf{Y}_{S,y}^{N}|\geq N^{-\sigma/2}} \ &\leq \ |\mathfrak{A}_{S,y,\delta}^{N}| |\mathbf{Y}_{S,y}^{N}|^{\mathfrak{u}} \mathbf{1}_{|\mathbf{Z}_{S,y}^{N}|\leq N^{-\sigma}}\mathbf{1}_{|\mathbf{Y}_{S,y}^{N}-\mathbf{Z}_{S,y}^{N}|\gtrsim N^{-\sigma/2}}. \label{eq:KPZ1WV?}
\end{align}\normalsize\normalsize
Indeed, if $|\mathbf{Z}^{N}| \leq N^{-\sigma}$ and $|\mathbf{Y}^{N}| \geq N^{-\sigma/2}$, then the difference, in absolute value, is $\gtrsim N^{-\sigma/2}$. As $|\mathfrak{A}^{N}|\lesssim1$, we get
\small\begin{align}
\E |\mathfrak{A}_{S,y,\delta}^{N}| |\mathbf{Y}_{S,y}^{N}|^{\mathfrak{u}} \mathbf{1}_{|\mathbf{Z}_{S,y}^{N}|\leq N^{-\sigma}}\mathbf{1}_{|\mathbf{Y}_{S,y}^{N}-\mathbf{Z}_{S,y}^{N}|\gtrsim N^{-\sigma/2}} \ &\lesssim \ \left( \E|\mathbf{Y}_{S,y}^{N}|^{2\mathfrak{u}} \right)^{1/2} \left( \mathbf{P}\left(|\mathbf{Y}_{S,y}^{N}-\mathbf{Z}_{S,y}^{N}| \gtrsim N^{-\frac12\sigma} \right)\right)^{1/2}. \label{eq:KPZ1WV??}
\end{align}\normalsize\normalsize
Lemma \ref{lemma:APY} shows the $\mathbf{Y}^{N}$-factor from the RHS of \eqref{eq:KPZ1WV??} is uniformly bounded. Proposition \ref{prop:KPZ1} applied using $\mathbb{K}$ any compact neighborhood of the support of our test function $\varphi$ controls the probability in the second factor on the RHS of \eqref{eq:KPZ1WV??}. Thus we get the following if $\sigma \in \R_{>0}$ is sufficiently small but still universal and for $\beta_{\mathrm{univ},1} \in \R_{>0}$ universal and $|y|\lesssim_{\varphi}N$:
\small\begin{align}
\left( \E|\mathbf{Y}_{S,y}^{N}|^{2\mathfrak{u}} \right)^{1/2} \left( \mathbf{P}\left(|\mathbf{Y}_{S,y}^{N}-\mathbf{Z}_{S,y}^{N}| \gtrsim N^{-\frac12\sigma} \right)\right)^{1/2} \ \lesssim \ \left( \mathbf{P}\left(|\mathbf{Y}_{S,y}^{N}-\mathbf{Z}_{S,y}^{N}| \gtrsim N^{-\frac12\sigma}\right)\right)^{1/2} \ &\lesssim_{\varphi} \ N^{-\frac12\beta_{\mathrm{univ},1}}. \label{eq:KPZ1WV???}
\end{align}\normalsize\normalsize
The implied constant in the final inequality of \eqref{eq:KPZ1WV???} depends on $\varphi$ through its support because the estimate in Proposition \ref{prop:KPZ1} depends on the compact set $\mathbb{K}\subseteq\R$ taken therein. Combining \eqref{eq:KPZ1WV?}, \eqref{eq:KPZ1WV??}, and \eqref{eq:KPZ1WV???} gives
\small\begin{align}
\E |\mathfrak{A}_{S,y,\delta}^{N}| |\mathbf{Y}_{S,y}^{N}|^{\mathfrak{u}} \mathbf{1}_{|\mathbf{Z}_{S,y}^{N}|\leq N^{-\sigma}}\mathbf{1}_{|\mathbf{Y}_{S,y}^{N}|\geq N^{-\sigma/2}} \ &\lesssim \ N^{-\frac12\beta_{\mathrm{univ},1}}. \label{eq:KPZ1WV????}
\end{align}\normalsize\normalsize
We now analyze the last term on the RHS of \eqref{eq:KPZ1WV!} in expectation. For $\beta \in \R_{>0}$ small but to be determined, we decompose
\begin{align}
|\mathfrak{A}_{S,y,\delta}^{N}| |\mathbf{Y}_{S,y}^{N}|^{\mathfrak{u}} \mathbf{1}_{|\mathbf{Z}_{S,y}^{N}|\geq N^{-\sigma}} \ &\leq \ |\mathfrak{A}_{S,y,\delta}^{N}| |\mathbf{Y}_{S,y}^{N}|^{\mathfrak{u}} \mathbf{1}_{|\mathbf{Z}_{S,y}^{N}|\geq N^{-\sigma}}\mathbf{1}_{|\mathbf{Y}_{S,\star}^{N}-\mathbf{Z}_{S,\star}^{N}| \geq N^{-\beta}} \label{eq:KPZ1WV!!} \\
&\quad + |\mathfrak{A}_{S,y,\delta}^{N}| |\mathbf{Y}_{S,y}^{N}|^{\mathfrak{u}} \mathbf{1}_{|\mathbf{Z}_{S,y}^{N}|\geq N^{-\sigma}}\mathbf{1}_{|\mathbf{Y}_{S,\star}^{N}-\mathbf{Z}_{S,\star}^{N}| \leq N^{-\beta}}. \nonumber
\end{align}\normalsize
In the newly introduced event $\star \in \{y,y+\delta N^{1/2}\}$ is allowed to take two values, so this event is defined by two constraints. Like the first term on the RHS of \eqref{eq:KPZ1WV!!} and the estimate \eqref{eq:KPZ1WV??}, for the first term on the RHS of \eqref{eq:KPZ1WV!!} we employ Lemma \ref{lemma:APY} and Proposition \ref{prop:KPZ1} with $\mathbb{K}$ any compact neighborhood of the support of $\varphi$, and we choose $\beta \in \R_{>0}$ sufficiently small to get
\small\begin{align}
\E|\mathfrak{A}_{S,y,\delta}^{N}| |\mathbf{Y}_{S,y}^{N}|^{\mathfrak{u}} \mathbf{1}_{|\mathbf{Z}_{S,y}^{N}|\geq N^{-\sigma}}\mathbf{1}_{|\mathbf{Y}_{S,\star}^{N}-\mathbf{Z}_{S,\star}^{N}| \geq N^{-\beta}} \ &\lesssim \ \E|\mathbf{Y}_{S,y}^{N}|^{\mathfrak{u}} \mathbf{1}_{|\mathbf{Y}_{S,\star}^{N}-\mathbf{Z}_{S,\star}^{N}| \geq N^{-\beta}} \\
&\lesssim \ \left( \E|\mathbf{Y}_{S,y}^{N}|^{2\mathfrak{u}} \right)^{1/2} \left( \mathbf{P}\left(|\mathbf{Y}_{S,\star}^{N}-\mathbf{Z}_{S,\star}^{N}| \geq N^{-\beta}\right)\right)^{1/2} \ \lesssim_{\varphi} \ N^{-\frac12\beta_{\mathrm{univ},1}}. \label{eq:KPZ1WV!!?}
\end{align}\normalsize\normalsize
The last bound in \eqref{eq:KPZ1WV!!?} follows by observing $\star\in N\mathbb{K}$ from which we obtain high-probability comparison via Proposition \ref{prop:KPZ1}.

We now address second term on the RHS of \eqref{eq:KPZ1WV!!}. By definition of $\mathbf{h}^{N}$ and $\mathbf{Z}^{N}$, as in the proof of Lemma 2.5 in \cite{DT} we get
\small\begin{align}
|\mathfrak{A}_{S,y,\delta}^{N}| \ &\lesssim \ \delta^{-1} |\log(1 + (\mathbf{Z}_{S,y}^{N})^{-1}\grad_{\delta N^{1/2}}\mathbf{Z}_{S,y}^{N})|. \label{eq:KPZ1WV2}
\end{align}\normalsize\normalsize
At this point, we differ slightly from the proof of Lemma 2.5 in \cite{DT}. In particular, we observe that the term $(\mathbf{Z}_{S,y}^{N})^{-1}\grad_{\delta N^{1/2}}\mathbf{Z}_{S,y}^{N}$ is at most the power series for the exponential evaluated at something that is at most $\delta$ and without the constant term. This is a consequence of Taylor expansion and definition of $\mathbf{Z}^{N}$. More precisely, we have the deterministic inequality
\small\begin{align}
|(\mathbf{Z}_{S,y}^{N})^{-1}\grad_{\delta N^{1/2}}\mathbf{Z}_{S,y}^{N}| \ \lesssim \ {\sum}_{\mathfrak{j}=1}^{\infty} |\lambda|^{\mathfrak{j}}(\mathfrak{j}!)^{-1}|\delta|^{\mathfrak{j}} \ &\lesssim_{\lambda} |\delta|. \label{eq:KPZ1WV3}
\end{align}\normalsize\normalsize
The second estimate \eqref{eq:KPZ1WV3} follows from choosing $\delta \in \R_{>0}$ sufficiently small but depending only on $\lambda \in \R$. We now note, like in the proof of Lemma 2.5 in \cite{DT}, that $|\log(1+x)| \lesssim_{\delta} |x|$ for $|x| \lesssim \delta$. Thus, combining this with \eqref{eq:KPZ1WV2} and \eqref{eq:KPZ1WV3}, we have
\small\begin{align}
|\mathfrak{A}_{S,y,\delta}^{N}| |\mathbf{Y}_{S,y}^{N}|^{\mathfrak{u}} \mathbf{1}_{|\mathbf{Z}_{S,y}^{N}|\geq N^{-\sigma}}\mathbf{1}_{|\mathbf{Y}_{S,\star}^{N}-\mathbf{Z}_{S,\star}^{N}| \leq N^{-\beta}} \ &\lesssim_{\lambda} \ \delta^{-1} |(\mathbf{Z}_{S,y}^{N})^{-1}\grad_{\delta N^{1/2}}\mathbf{Z}_{S,y}^{N}||\mathbf{Y}_{S,y}^{N}|^{\mathfrak{u}}\mathbf{1}_{|\mathbf{Z}_{S,y}^{N}|\geq N^{-\sigma}}\mathbf{1}_{|\mathbf{Y}_{S,\star}^{N}-\mathbf{Z}_{S,\star}^{N}| \leq N^{-\beta}} \\
&\lesssim \ \delta^{-1}N^{\sigma}|\grad_{\delta N^{1/2}}\mathbf{Z}_{S,y}^{N}||\mathbf{Y}_{S,y}^{N}|^{\mathfrak{u}}\mathbf{1}_{|\mathbf{Z}_{S,y}^{N}|\geq N^{-\sigma}}\mathbf{1}_{|\mathbf{Y}_{S,\star}^{N}-\mathbf{Z}_{S,\star}^{N}| \leq N^{-\beta}}. \label{eq:KPZ1WV4}
\end{align}\normalsize\normalsize
On the event corresponding to the second indicator function in \eqref{eq:KPZ1WV4}, the gradient in question of $\mathbf{Z}^{N}$ in \eqref{eq:KPZ1WV4} can be bounded in terms of the same gradient but of $\mathbf{Y}^{N}$ up to an admissible error. More precisely, we have
\small\begin{align}
|\grad_{\delta N^{1/2}}\mathbf{Z}_{S,y}^{N}|\mathbf{1}_{|\mathbf{Z}_{S,y}^{N}|\geq N^{-\sigma}}\mathbf{1}_{|\mathbf{Y}_{S,\star}^{N}-\mathbf{Z}_{S,\star}^{N}| \leq N^{-\beta}} \ &\lesssim \ |\grad_{\delta N^{1/2}}\mathbf{Y}_{S,y}^{N}|\mathbf{1}_{|\mathbf{Y}_{S,\star}^{N}-\mathbf{Z}_{S,\star}^{N}| \leq N^{-\beta}} \ + \ N^{-\beta}. \label{eq:KPZ1WV5}
\end{align}\normalsize\normalsize
This step is the precise implementation of transferring relevant $\mathbf{Z}^{N}$-data in the proof of Lemma \ref{lemma:KPZ21} into $\mathbf{Y}^{N}$-estimates that we discussed in the proof of Proposition \ref{prop:KPZ2}. By \eqref{eq:KPZ1WV4}, \eqref{eq:KPZ1WV5}, the Cauchy-Schwarz inequality, and Lemma \ref{lemma:APY}, we get
\small\begin{align}
\E|\mathfrak{A}_{S,y,\delta}^{N}| |\mathbf{Y}_{S,y}^{N}|^{\mathfrak{u}} \mathbf{1}_{|\mathbf{Z}_{S,y}^{N}|\geq N^{-\sigma}}\mathbf{1}_{|\mathbf{Y}_{S,\star}^{N}-\mathbf{Z}_{S,\star}^{N}| \leq N^{-\beta}} \ &\lesssim \ \delta^{-1}N^{\sigma}\E|\grad_{\delta N^{1/2}}\mathbf{Z}_{S,y}^{N}||\mathbf{Y}_{S,y}^{N}|^{\mathfrak{u}}\mathbf{1}_{|\mathbf{Z}_{S,y}^{N}|\geq N^{-\sigma}}\mathbf{1}_{|\mathbf{Y}_{S,\star}^{N}-\mathbf{Z}_{S,\star}^{N}| \leq N^{-\beta}} \nonumber\\
&\lesssim \ \delta^{-1}N^{\sigma}\E|\grad_{\delta N^{1/2}}\mathbf{Y}_{S,y}^{N}||\mathbf{Y}_{S,y}^{N}|^{\mathfrak{u}}\mathbf{1}_{|\mathbf{Y}_{S,\star}^{N}-\mathbf{Z}_{S,\star}^{N}| \leq N^{-\beta}} \ + \ \delta^{-1}N^{\sigma}N^{-\beta}\E|\mathbf{Y}_{S,y}^{N}|^{\mathfrak{u}} \nonumber\\
&\lesssim \ \delta^{-1}N^{\sigma}(\E|\grad_{\delta N^{1/2}}\mathbf{Y}_{S,y}^{N}|^{2}\mathbf{1}_{|\mathbf{Y}_{S,\star}^{N}-\mathbf{Z}_{S,\star}^{N}| \leq N^{-\beta}})^{\frac12}(\E|\mathbf{Y}_{S,y}^{N}|^{2\mathfrak{u}})^{\frac12} \ + \ \delta^{-1}N^{\sigma-\beta}\nonumber \\
&\lesssim \ \delta^{-1}N^{\sigma}(\E|\grad_{\delta N^{1/2}}\mathbf{Y}_{S,y}^{N}|^{2})^{\frac12} \ + \ \delta^{-1}N^{\sigma-\beta}. \label{eq:KPZ1WV6}
\end{align}\normalsize\normalsize
To estimate the first term, provided Lemma \ref{lemma:APY} note that the spatial-regularity estimate (3.13) from Proposition 3.2 of \cite{DT} still holds for $\mathbf{Y}^{N}$, again because the estimate therein for the $Z$-field therein comes just from its mild form (3.2) in \cite{DT}, which is sufficiently identical to the defining equation for $\mathbf{Y}^{N}$. We used this spatial regularity of $\mathbf{Y}^{N}$ in qualitative fashion earlier at the beginning to reduce proof of Lemma \ref{lemma:KPZ21} into proof of \eqref{eq:KPZ1WV1}. Ultimately we get the following Holder bound for any $\vartheta < 1$:
\small\begin{align}
\delta^{-1}N^{\sigma}(\E|\grad_{\delta N^{1/2}}\mathbf{Y}_{S,y}^{N}|^{2})^{\frac12} \ &\lesssim_{\vartheta} \ \delta^{-1}N^{\sigma} \cdot \delta^{\frac12\vartheta}N^{\frac14\vartheta} \cdot N^{-\frac12\vartheta}. \label{eq:KPZ1WV7}
\end{align}\normalsize\normalsize
We combine \eqref{eq:KPZ1WV6} and \eqref{eq:KPZ1WV7} to deduce
\small\begin{align}
\E|\mathfrak{A}_{S,y,\delta}^{N}| |\mathbf{Y}_{S,y}^{N}|^{\mathfrak{u}} \mathbf{1}_{|\mathbf{Z}_{S,y}^{N}|\geq N^{-\sigma}}\mathbf{1}_{|\mathbf{Y}_{S,\star}^{N}-\mathbf{Z}_{S,\star}^{N}| \leq N^{-\beta}} \ &\lesssim_{\vartheta} \ \delta^{-1+\frac12\vartheta}N^{\sigma-\frac14\vartheta} + \delta^{-1}N^{\sigma-\beta}. \label{eq:KPZ1WV8}
\end{align}\normalsize\normalsize
The estimate \eqref{eq:KPZ1WV1} follows from combining \eqref{eq:KPZ1WV!}, \eqref{eq:KPZ1WV????}, \eqref{eq:KPZ1WV!!}, \eqref{eq:KPZ1WV!!?}, and \eqref{eq:KPZ1WV8} with the last of these bounds employed having chosen $\sigma \in \R_{>0}$ sufficiently small depending on $\beta \in \R_{>0}$ and $\vartheta < 1$. This completes the proof.
\end{proof}
\begin{proof}[Proof of \emph{Lemma \ref{lemma:KPZ3}}]
We begin with the following general properties of $\mathbf{D}_{1}$ in \cite{Bil}.
\begin{itemize}[leftmargin=*]
\item The set $\mathbf{D}_{1}$ is metrizable with metric bounded by the uniform norm; in the definition (12.13) on page 124 in \cite{Bil} but replacing $\R$-valued paths with $\mathbf{C}(\mathbb{K})$-valued paths, we can choose the identity function in the infimum and recover the uniform metric.
\end{itemize}
Thus $\mathbf{X}^{N,1}-\mathbf{X}^{N,2}$ converges to the 0 process in probability. As $\mathbf{X}^{N,1} = \mathbf{X}^{N,2} + (\mathbf{X}^{N,1}-\mathbf{X}^{N,2})$ and $\mathbf{X}^{N,2}$ converges in law to $\mathbf{X}^{\infty,2}$ while $\mathbf{X}^{N,1}-\mathbf{X}^{N,2}$ converges \emph{in probability} to 0, standard probability on separable metric spaces finishes the proof.
\end{proof}
%
%
%
\appendix
\section{Heat Operator Estimates}
We collect heat kernel estimates for $\mathbf{H}^{N}$ and $\bar{\mathbf{H}}^{N}$ based on Proposition A.1/Corollary A.2 from \cite{DT}. The following estimates are effectively the heat kernel estimates that are used in \cite{DT} as well as versions for the compactified heat operator $\bar{\mathbf{H}}^{N}$ though with a few elementary additional ingredients.
\begin{lemma}\label{lemma:HKE}
 Recall \emph{Definition \ref{definition:Exps}} for the entire statement of this lemma. Provided any $\kappa\in\R_{\geq0}$, uniformly bounded $k\in\Z$, and uniformly bounded $S,T\in\R_{\geq0}$ with $S\leq T$, we have the following pointwise off-diagonal estimates for $\mathbf{H}^{N}$ and $\bar{\mathbf{H}}^{N}$:
\small\begin{align}
{\sup}_{x,y\in\Z}\left(e_{S,T,x,y}^{N,\kappa}\mathbf{H}_{S,T,x,y}^{N}\right) \ + \ {\sup}_{x,y\in\mathbb{T}_{N}}\left(e_{S,T,x,y}^{N,\kappa,\mathrm{cpt}}\bar{\mathbf{H}}_{S,T,x,y}^{N}\right) \ &\lesssim_{\kappa} \ N^{-1}\mathfrak{s}_{S,T}^{-\frac12} \wedge 1 \label{eq:HKENash} \\
{\sup}_{x,y\in\Z}\left(e_{S,T,x,y}^{N,\kappa}|\grad_{k}^{!}\mathbf{H}_{S,T,x,y}^{N}|\right) \ + \ {\sup}_{x,y\in\mathbb{T}_{N}}\left(e_{S,T,x,y}^{N,\kappa,\mathrm{cpt}}|\bar{\grad}_{k}^{!}\bar{\mathbf{H}}_{S,T,x,y}^{N}|\right) \ &\lesssim_{\kappa} \ |k| \left( N^{-1}\mathfrak{s}_{S,T}^{-1} \wedge N\right). \label{eq:HKEXR}
\end{align}\normalsize\normalsize
Provided $\varphi:\R_{\geq0}\times\Z\to\R$ and $\bar{\varphi}:\R_{\geq0}\times\mathbb{T}_{N}\to\R$, any subset $I\subseteq\R_{\geq0}$ with measure $|I|$, any uniformly bounded $\mathfrak{t}_{\mathrm{st}}\in\R_{\geq0}$, any uniformly bounded $\mathfrak{t}_{\mathrm{reg}}\in\R_{\leq0}$, any $\delta\in\R_{>0}$, any $k\in\Z$ uniformly bounded, and any $\mathfrak{l}\in\Z$, we have
\small\begin{align}
\|\varphi\|_{\mathfrak{t}_{\mathrm{st}};\Z}^{-1}\|\mathbf{H}^{N}(\varphi_{S,y}\mathbf{1}_{S\in I})\|_{\mathfrak{t}_{\mathrm{st}};\Z} \ + \ \|\bar{\varphi}\|_{\mathfrak{t}_{\mathrm{st}};\mathbb{T}_{N}}^{-1}\|\bar{\mathbf{H}}^{N}(\bar{\varphi}_{S,y}\mathbf{1}_{S\in I})\|_{\mathfrak{t}_{\mathrm{st}};\mathbb{T}_{N}} \ &\lesssim \ |I|\wedge\mathfrak{t}_{\mathrm{st}} \label{eq:HKEConvolution} \\
\|\varphi\|_{\mathfrak{t}_{\mathrm{st}};\Z}^{-1}\|\mathbf{1}_{T\geq|\mathfrak{t}_{\mathrm{reg}}|}\mathfrak{s}_{|\mathfrak{t}_{\mathrm{reg}}|,T}^{1-\delta}\grad_{\mathfrak{t}_{\mathrm{reg}}}^{\mathbf{T}}\mathbf{H}_{T,x}^{N,\mathbf{X}}(\varphi)\|_{\mathfrak{t}_{\mathrm{st}};\Z} \ + \ \|\bar{\varphi}\|_{\mathfrak{t}_{\mathrm{st}};\mathbb{T}_{N}}^{-1}\|\mathbf{1}_{T\geq|\mathfrak{t}_{\mathrm{reg}}|}\mathfrak{s}_{|\mathfrak{t}_{\mathrm{reg}}|,T}^{1-\delta}\grad_{\mathfrak{t}_{\mathrm{reg}}}^{\mathbf{T}}\bar{\mathbf{H}}_{T,x}^{N,\mathbf{X}}(\bar{\varphi})\|_{\mathfrak{t}_{\mathrm{st}};\mathbb{T}_{N}} \ &\lesssim_{\delta} \ N^{2\delta}|\mathfrak{t}_{\mathrm{reg}}|\label{eq:HKETR2} \\
\|\varphi\|_{\mathfrak{t}_{\mathrm{st}};\Z}^{-1}\|\mathbf{1}_{T\geq|\mathfrak{t}_{\mathrm{reg}}|}\grad_{\mathfrak{t}_{\mathrm{reg}}}^{\mathbf{T}}\mathbf{H}^{N}_{T,x}(\varphi)\|_{\mathfrak{t}_{\mathrm{st}};\Z} \ + \ \|\bar{\varphi}\|_{\mathfrak{t}_{\mathrm{st}};\mathbb{T}_{N}}^{-1}\|\mathbf{1}_{T\geq|\mathfrak{t}_{\mathrm{reg}}|}\grad_{\mathfrak{t}_{\mathrm{reg}}}^{\mathbf{T}}\bar{\mathbf{H}}^{N}_{T,x}(\bar{\varphi})\|_{\mathfrak{t}_{\mathrm{st}};\mathbb{T}_{N}} \ &\lesssim_{\delta} \ N^{2\delta}|\mathfrak{t}_{\mathrm{reg}}| \label{eq:HKETR1} \\
\|\varphi\|_{\mathfrak{t}_{\mathrm{st}};\Z}^{-1}\|\mathbf{1}_{T\geq|\mathfrak{t}_{\mathrm{reg}}|}\grad_{\mathfrak{t}_{\mathrm{reg}}}^{\mathbf{T}}\mathbf{H}^{N}_{T,x}(\grad_{k}^{!}\varphi)\|_{\mathfrak{t}_{\mathrm{st}};\Z} \ + \ \|\bar{\varphi}\|_{\mathfrak{t}_{\mathrm{st}};\mathbb{T}_{N}}^{-1}\|\mathbf{1}_{T\geq|\mathfrak{t}_{\mathrm{reg}}|}\grad_{\mathfrak{t}_{\mathrm{reg}}}^{\mathbf{T}}\bar{\mathbf{H}}^{N}_{T,x}(\bar{\grad}_{k}^{!}\bar{\varphi})\|_{\mathfrak{t}_{\mathrm{st}};\mathbb{T}_{N}} \ &\lesssim_{\delta} \ |\mathfrak{t}_{\mathrm{reg}}|^{\frac14-\delta} \label{eq:HKETRXR1} \\
\|\varphi\|_{\mathfrak{t}_{\mathrm{st}};\Z}^{-1}\|\mathbf{H}^{N}(\grad_{\mathfrak{l}}^{!}\varphi)\|_{\mathfrak{t}_{\mathrm{st}};\Z} \ + \ \|\bar{\varphi}\|_{\mathfrak{t}_{\mathrm{st}};\mathbb{T}_{N}}^{-1}\|\bar{\mathbf{H}}^{N}(\bar{\grad}_{\mathfrak{l}}^{!}\bar{\varphi})\|_{\mathfrak{t}_{\mathrm{st}};\mathbb{T}_{N}} \ &\lesssim \ |\mathfrak{l}|. \label{eq:HKEKPZNL}
\end{align}\normalsize\normalsize
\end{lemma}
\begin{proof}
We first discuss $\mathbf{H}^{N}$-estimates. The $\mathbf{H}^{N}$ estimates in \eqref{eq:HKENash} and \eqref{eq:HKEXR} follow from Proposition A.1 in \cite{DT}. The $\mathbf{H}^{N}$-bound in \eqref{eq:HKEConvolution} follows by the fact that the spatial heat operator $\mathbf{H}^{N,\mathbf{X}}$ is a contraction with respect to the supremum-norm on functions on $\Z$ and $\mathbf{H}^{N}$ integrates $\mathbf{H}^{N,\mathbf{X}}$-operators on $I\cap[0,\mathfrak{t}_{\mathrm{st}}]$ on the LHS of \eqref{eq:HKEConvolution}. The other $\mathbf{H}^{N}$-bounds start with the observation that by definition of $\mathbf{H}^{N}$, we have a fundamental-theorem-of-calculus estimate where $\delta\in\R_{>0}$ is arbitrarily small and $S+|\mathfrak{t}_{\mathrm{reg}}|\leq T$. Below, $\grad^{\mathbf{T}}$ on the far LHS below acts on the heat kernel via its forwards time-variable and $\|\varphi\|=\|\varphi\|_{T;\Z}$:
\small\begin{align}
\sum_{y\in\Z}|\grad_{\mathfrak{t}_{\mathrm{reg}}}^{\mathbf{T}}\mathbf{H}_{S,T,x,y}^{N}| |\varphi_{S,y}| \ \lesssim \ \|\varphi\|\int_{\mathfrak{t}_{\mathrm{reg}}}^{0}\sum_{y\in\Z}|\mathscr{L}^{!!}\mathbf{H}_{S,T+\mathfrak{r},x,y}^{N}| \d\mathfrak{r} \ \lesssim_{\mathfrak{m}} \ N^{2\delta}\|\varphi\|\int_{\mathfrak{t}_{\mathrm{reg}}}^{0}\mathfrak{s}_{S,T+\mathfrak{r}}^{-1+\delta}\d\mathfrak{r} \ \lesssim \ N^{2\delta}|\mathfrak{t}_{\mathrm{reg}}|\mathfrak{s}_{S+|\mathfrak{t}_{\mathrm{reg}}|,T}^{-1+\delta}\|\varphi\|. \label{eq:HKE1}
\end{align}\normalsize\normalsize
The second estimate above follows by (A.29) in Corollary A.2 in \cite{DT} for $v = 1-2\delta$, and the last estimate follows by observing the integrand is decreasing in $\mathfrak{r}\in[\mathfrak{t}_{\mathrm{reg}},0]$; the heat kernel here is scaled by $N^{2}$ in time but the heat kernel in \cite{DT} is not. 
\begin{itemize}[leftmargin=*]
\item As $T\leq\mathfrak{t}_{\mathrm{st}}$, the norm $\|\varphi\|$ in \eqref{eq:HKE1} can be extended to $\|\varphi\|_{\mathfrak{t}_{\mathrm{st}};\Z}$. We get from this and \eqref{eq:HKE1} for $S=0$ the $\mathbf{H}^{N}$-bound in \eqref{eq:HKETR2}.
\item To obtain the $\mathbf{H}^{N}$ estimate in \eqref{eq:HKETR1}, we decompose the space-time heat operator $\grad^{\mathbf{T}}\mathbf{H}^{N}$ into a time-integral of the $1$-norm on $\Z$ of the time-gradient of the heat kernel, which we denote by $\Phi_{1}$, and a short-time integral of the spatial heat operator, which we denote by $\Phi_{2}$. This was the decomposition used to get $I_{31}$ and $I_{32}$ in the proof of (3.14) in Proposition 3.2 in \cite{DT}. For the heat-kernel-time-gradient term $\Phi_{1}$, we integrate \eqref{eq:HKE1} in $S\in[0,T-|\mathfrak{t}_{\mathrm{reg}}|]$. For the short-time integral $\Phi_{2}$, we employ \eqref{eq:HKEConvolution} for $I=T+[\mathfrak{t}_{\mathrm{reg}},0]$. Again we refer to $I_{31}$ and $I_{32}$ terms and their estimates in the proof of (3.14) in Proposition 3.2 in \cite{DT}.
\item The proof of the $\mathbf{H}^{N}$-estimate in \eqref{eq:HKETRXR1} follows via the proof of the estimates for $I_{41}$ and $I_{42}$ in the proof of (3.14) in Proposition 3.2 in \cite{DT}. Roughly speaking, we move the gradient from $\varphi$ onto the heat kernel and employ space-time regularity estimates for the heat kernel given in (A.10) from Proposition A.1 in \cite{DT} and (A.27) from Corollary A.2 in \cite{DT}.
\item The proof of the $\mathbf{H}^{N}$ estimate in \eqref{eq:HKEKPZNL} follows by summation-by-parts to move the gradient from $\varphi$ to the heat kernel, pulling out the $\varphi$-factor in the resulting spatial gradient of the $\mathbf{H}^{N}$ heat operator via its $\|\|_{\mathfrak{t}_{\mathrm{st}};\Z}$-norm, applying (A.27) in Corollary A.2 in \cite{DT} with $v=1$  to control the $1$-norm/sum of the absolute values of the spatial gradient of the heat kernel, and integrating the resulting integrable singularity in time. The summation-by-parts is the same giving (3.2) in \cite{DT}, and the estimate noted in this last sentence is the same that gives the $I_{4}$ estimate after (3.23) in the proof of (3.12) in Proposition 3.2 in \cite{DT}.
\end{itemize}
The proof of the proposed $\bar{\mathbf{H}}^{N}$-bounds follows by identical considerations; see Remark \ref{remark:ch2HKE}. This completes the proof.
\end{proof}
We now collect a time-regularity estimate for $\bar{\mathbf{H}}^{N}$ which requires an additional integration-by-parts-type ingredient, though it is still an integrated time-smoothness estimate for the heat operator like \eqref{eq:HKETR1} and requires only a little more gymnastics.
\begin{lemma}\label{lemma:HKEB}
For any possibly random $\bar{\varphi}:\R_{\geq0}\times\mathbb{T}_{N}\to\R$, any $\mathfrak{t}_{\mathrm{st}},\mathfrak{t}_{\mathrm{reg}}\in\R_{\geq0}$ uniformly bounded such that $\mathfrak{t}_{\mathrm{st}}$ is possibly random and with probability 1 we have $\mathfrak{t}_{\mathrm{st}}\leq1$, and any $\delta\in\R_{>0}$, we have the following forwards-time-gradient estimate:
\small\begin{align}
\|\bar{\varphi}\|_{\mathfrak{t}_{\mathrm{st}}+\mathfrak{t}_{\mathrm{reg}};\mathbb{T}_{N}}^{-1}\|\bar{\mathbf{H}}_{T,x}^{N}(\grad_{\mathfrak{t}_{\mathrm{reg}}}^{\mathbf{T}}\bar{\varphi})\|_{\mathfrak{t}_{\mathrm{st}};\mathbb{T}_{N}} \ \lesssim_{\delta} \ N^{2\delta}|\mathfrak{t}_{\mathrm{reg}}|.
\end{align}\normalsize\normalsize
\end{lemma}
\begin{proof}
Similar to \eqref{eq:LeibnizKPZNL} we have the following time-discrete Leibniz rule in which $\grad^{\mathbf{T}}$ acts only on $S$-variables:
\small\begin{align}
\bar{\mathbf{H}}_{S,T,x,y}^{N}\mathbf{1}_{S\in[0,T]}\grad^{\mathbf{T}}_{\mathfrak{t}_{\mathrm{reg}}}\bar{\varphi}_{S,y} \ &= \ \grad_{\mathfrak{t}_{\mathrm{reg}}}^{\mathbf{T}}\left(\bar{\mathbf{H}}_{S-\mathfrak{t}_{\mathrm{reg}},T,x,y}^{N}\mathbf{1}_{S-\mathfrak{t}_{\mathrm{reg}}\in[0,T]}\bar{\varphi}_{S,y}\right) + \bar{\varphi}_{S,y}\grad_{-\mathfrak{t}_{\mathrm{reg}}}^{\mathbf{T}}\left(\bar{\mathbf{H}}_{S,T,x,y}^{N}\mathbf{1}_{S\in[0,T]}\right). \label{eq:HKEB1}
\end{align}\normalsize\normalsize
The identity \eqref{eq:HKEB1} is \eqref{eq:LeibnizKPZNL} after replacing the $\bar{\mathbf{Z}}^{N}$-process with the product of the $\bar{\mathbf{H}}^{N}$-heat kernel and the indicator function. We proceed with expansion of the time-gradient in the second term on the RHS of \eqref{eq:HKEB1} which also resembles a time-discrete Leibniz rule similar to \eqref{eq:HKEB1}. The following can be checked by an elementary calculation; again $\grad^{\mathbf{T}}$ acts only on $S$-variables:
\small\begin{align}
\bar{\varphi}_{S,y}\grad_{-\mathfrak{t}_{\mathrm{reg}}}^{\mathbf{T}}\left(\bar{\mathbf{H}}_{S,T,x,y}^{N}\mathbf{1}_{S\in[0,T]}\right) \ &= \ \bar{\varphi}_{S,y}\left(\grad_{-\mathfrak{t}_{\mathrm{reg}}}^{\mathbf{T}}\bar{\mathbf{H}}_{S,T,x,y}^{N}\right)\mathbf{1}_{S\in[0,T]} + \bar{\varphi}_{S,y}\bar{\mathbf{H}}_{S-\mathfrak{t}_{\mathrm{reg}},T,x,y}^{N}\left(\grad_{-\mathfrak{t}_{\mathrm{reg}}}^{\mathbf{T}}\mathbf{1}_{S\in[0,T]}\right) \ \overset{\bullet}= \ \Phi_{1}+\Phi_{2}. \label{eq:HKEB2}
\end{align}\normalsize\normalsize
We will now sum each term on the LHS/RHS of \eqref{eq:HKEB1} over $y\in\mathbb{T}_{N}$ and additionally integrate over $S\in\R$.
\begin{itemize}[leftmargin=*]
\item Summation over $y\in\mathbb{T}_{N}$ and integration over $S\in\R$ of the LHS of \eqref{eq:HKEB1} gives the heat operator $\bar{\mathbf{H}}^{N}_{T,x}(\grad^{\mathbf{T}}_{\mathfrak{t}_{\mathrm{reg}}}\bar{\varphi})$ by definition.
\item Integrating the first term on the RHS of \eqref{eq:HKEB1} over $S\in\R$ yields 0 because the integral of the time-discrete gradient of any compactly supported function vanishes; note that the first term on the RHS of \eqref{eq:HKEB1} is compactly supported because of the indicator function. This general observation about integrating time-discrete gradients is analogous to the fact that summing a compactly supported discrete gradient on $\Z$ gives 0, or that integrating a compactly supported derivative on $\R$ gives 0.
\item We now treat the second term on the RHS of \eqref{eq:HKEB1} by \eqref{eq:HKEB2}. We may let the $\grad^{\mathbf{T}}$-operator in $\Phi_{1}$ act instead on the forward time-variable of the heat kernel if we change $-\mathfrak{t}_{\mathrm{reg}}\to\mathfrak{t}_{\mathrm{reg}}$, since the heat kernel is time-homogeneous. Elementary calculation lets us \emph{then} replace $\mathfrak{t}_{\mathrm{reg}}\to-\mathfrak{t}_{\mathrm{reg}}$ if we replace $T\to T+\mathfrak{t}_{\mathrm{reg}}$ and introduce a factor of $-1$. With these changes,
\small\begin{align}
\|\left(\int_{\R}\sum_{y\in\mathbb{T}_{N}}\left(\grad_{-\mathfrak{t}_{\mathrm{reg}}}^{\mathbf{T}}\bar{\mathbf{H}}_{S,T+\mathfrak{t}_{\mathrm{reg}},x,y}^{N}\right)\mathbf{1}_{S\in[0,T]}\bar{\varphi}_{S,y} \d S\right)\|_{\mathfrak{t}_{\mathrm{st}};\mathbb{T}_{N}} \ \leq \ \|\bar{\varphi}\|_{\mathfrak{t}_{\mathrm{st}}+\mathfrak{t}_{\mathrm{reg}};\mathbb{T}_{N}}\|\left(\int_{0}^{T}\sum_{y\in\mathbb{T}_{N}}|\grad_{-\mathfrak{t}_{\mathrm{reg}}}^{\mathbf{T}}\bar{\mathbf{H}}_{S,T+\mathfrak{t}_{\mathrm{reg}},x,y}^{N}| \ \d S \right)\|_{\mathfrak{t}_{\mathrm{st}};\mathbb{T}_{N}}. \label{eq:HKEB3}
\end{align}\normalsize\normalsize
Note the $\mathfrak{t}_{\mathrm{reg}}$-parameter in \eqref{eq:HKE1} is non-positive while $\mathfrak{t}_{\mathrm{reg}}$ here is non-negative. So we may use \eqref{eq:HKE1}, with $\mathfrak{t}_{\mathrm{reg}}$ there equal to $-\mathfrak{t}_{\mathrm{reg}}$ here and $T$ therein equal to $T+\mathfrak{t}_{\mathrm{reg}}$ here, to the RHS of \eqref{eq:HKEB3}. Recalling the term in the integration/summation on the LHS of \eqref{eq:HKEB3} is $\Phi_{1}$ from \eqref{eq:HKEB2} by definition, after applying \eqref{eq:HKE1} as just described to \eqref{eq:HKEB3} we get, for any $\delta\in\R_{>0}$,
\small\begin{align}
\|\left(\int_{\R}{\sum}_{y\in\mathbb{T}_{N}}\Phi_{1}\d S\right)\|_{\mathfrak{t}_{\mathrm{st}};\mathbb{T}_{N}} \ &\lesssim \ N^{2\delta}|\mathfrak{t}_{\mathrm{reg}}|\|\bar{\varphi}\|_{\mathfrak{t}_{\mathrm{st}}+\mathfrak{t}_{\mathrm{reg}};\mathbb{T}_{N}}\|\left(\int_{0}^{T}\mathfrak{s}_{S+|\mathfrak{t}_{\mathrm{reg}}|,T+|\mathfrak{t}_{\mathrm{reg}}|}^{-1+\delta}\d S\right)\|_{\mathfrak{t}_{\mathrm{st}};\mathbb{T}_{N}} \ \lesssim_{\delta} \ N^{2\delta}|\mathfrak{t}_{\mathrm{reg}}|\|\bar{\varphi}\|_{\mathfrak{t}_{\mathrm{st}}+\mathfrak{t}_{\mathrm{reg}};\mathbb{T}_{N}}. \label{eq:HKEB4}
\end{align}\normalsize\normalsize
The last estimate in \eqref{eq:HKEB4} follows by integrating the integral in the middle term in \eqref{eq:HKEB4} to get a space-time uniform bound.
\item Lastly for the $\Phi_{2}$-term in \eqref{eq:HKEB2}, the time-gradient therein is controlled by the indicator function of two intervals of length $\mathfrak{t}_{\mathrm{reg}}$. Thus $\Phi_{2}$ gives a space-time heat operator acting on $\bar{\varphi}$ times the indicator function $\mathbf{1}_{S\in I}$ with $|I|\lesssim\mathfrak{t}_{\mathrm{reg}}$. Technically, when we integrate $\Phi_{2}$ in $S\in\R$ we may employ change-of-variables $S-\mathfrak{t}_{\mathrm{reg}}\to S$ to remove the cosmetic time-shift $-\mathfrak{t}_{\mathrm{reg}}$ in the heat kernel in $\Phi_{2}$. Thus we employ \eqref{eq:HKEConvolution} to deduce the following estimate where the $\|\bar{\varphi}\|_{\mathfrak{t}_{\mathrm{st}}+\mathfrak{t}_{\mathrm{reg}};\mathbb{T}_{N}}$-factor comes from the fact that $\Phi_{2}$ and every other ``heat-term" in this proof only sees the test function $\bar{\varphi}$ until the time $\mathfrak{t}_{\mathrm{st}}+\mathfrak{t}_{\mathrm{reg}}$:
\small\begin{align}
\|\left(\int_{\R}{\sum}_{y\in\mathbb{T}_{N}}\Phi_{2}\d S\right)\|_{\mathfrak{t}_{\mathrm{st}};\mathbb{T}_{N}} \ &\lesssim \ \|\bar{\varphi}\|_{\mathfrak{t}_{\mathrm{st}}+\mathfrak{t}_{\mathrm{reg}};\mathbb{T}_{N}}|I| \ \lesssim \ \|\bar{\varphi}\|_{\mathfrak{t}_{\mathrm{st}}+\mathfrak{t}_{\mathrm{reg}};\mathbb{T}_{N}}|\mathfrak{t}_{\mathrm{reg}}|.
\end{align}\normalsize\normalsize
\end{itemize}
The proof of Lemma \ref{lemma:HKEB} follows immediately upon combining the previous bullet points with \eqref{eq:HKEB1} and \eqref{eq:HKEB2}.
\end{proof}
\begin{lemma}\label{lemma:HKEC}
Consider any $\bar{\varphi}:\R_{\geq0}\times\mathbb{T}_{N}\to\R$ and possibly random time $\mathfrak{t}_{1,\mathrm{st}}\geq0$ satisfying $\mathfrak{t}_{1,\mathrm{st}}\leq1$ with probability 1. We emphasize $\bar{\varphi}$ may be random. We have the following deterministic estimate for a space-time heat operator with time-cutoff:
\small\begin{align}
\|\bar{\mathbf{H}}_{T,x}^{N}(\bar{\varphi}_{S,y}\mathbf{1}_{S\leq\mathfrak{t}_{1,\mathrm{st}}})\|_{1;\mathbb{T}_{N}} \ &\leq \ \|\bar{\mathbf{H}}_{T,x}^{N}(\bar{\varphi}_{S,y})\|_{\mathfrak{t}_{1,\mathrm{st}};\mathbb{T}_{N}}.
\end{align}\normalsize\normalsize
\end{lemma}
\begin{proof}
We consider the following two cases distinguished by $\mathfrak{t}_{1,\mathrm{st}}$. In what follows we always assume $T\in\R_{\geq0}$ satisfies $T\leq1$.
\begin{itemize}[leftmargin=*]
\item Take any $T\leq\mathfrak{t}_{1,\mathrm{st}}$. We have the following two bounds which certainly suffice for this case. The first bound is straightforward and the second bound follows from the cutoff $\mathbf{1}_{S\leq\mathfrak{t}_{1,\mathrm{st}}}$ being already built into the heat operator $\bar{\mathbf{H}}^{N}$ for $T\leq\mathfrak{t}_{1,\mathrm{st}}$:
\small\begin{align}
|\bar{\mathbf{H}}_{T,x}^{N}(\bar{\varphi}_{S,y}\mathbf{1}_{S\leq\mathfrak{t}_{1,\mathrm{st}}})| \ \leq \ \|\bar{\mathbf{H}}_{T,x}^{N}(\bar{\varphi}_{S,y}\mathbf{1}_{S\leq\mathfrak{t}_{1,\mathrm{st}}})\|_{\mathfrak{t}_{1,\mathrm{st}};\mathbb{T}_{N}} \ = \ \|\bar{\mathbf{H}}_{T,x}^{N}(\bar{\varphi}_{S,y})\|_{\mathfrak{t}_{1,\mathrm{st}};\mathbb{T}_{N}}.
\end{align}\normalsize\normalsize
\item Take $T\geq\mathfrak{t}_{1,\mathrm{st}}$. By definition of the space-time heat operator and by the Chapman-Kolmogorov equation for the heat kernel, 
\small\begin{align}
\bar{\mathbf{H}}_{T,x}^{N}(\bar{\varphi}_{S,y}\mathbf{1}_{S\leq\mathfrak{t}_{1,\mathrm{st}}}) \ = \ \int_{0}^{\mathfrak{t}_{1,\mathrm{st}}}{\sum}_{y\in\mathbb{T}_{N}}\bar{\mathbf{H}}_{S,T,x,y}^{N}\bar{\varphi}_{S,y} \ \d S \ &= \ {\sum}_{w\in\mathbb{T}_{N}}\bar{\mathbf{H}}_{\mathfrak{t}_{1,\mathrm{st}},T,x,w}^{N}\left(\int_{0}^{\mathfrak{t}_{1,\mathrm{st}}}{\sum}_{y\in\mathbb{T}_{N}}\bar{\mathbf{H}}_{S,\mathfrak{t}_{1,\mathrm{st}},w,y}^{N}\bar{\varphi}_{S,y} \ \d S\right).\label{eq:HKEC1}
\end{align}\normalsize\normalsize
We observe the parenthetical term on the RHS of \eqref{eq:HKEC1} is a space-time heat operator acting on $\bar{\varphi}$ evaluated at time $\mathfrak{t}_{1,\mathrm{st}}$. Thus the far RHS of \eqref{eq:HKEC1} is the \emph{spatial} heat operator acting on this time-$\mathfrak{t}_{1,\mathrm{st}}$ space-time heat operator. Because the spatial heat operator is a contraction with respect to the supremum-norm on functions on $\mathbb{T}_{N}$, we get from \eqref{eq:HKEC1} the bound
\small\begin{align}
\|\mathbf{1}_{T\geq\mathfrak{t}_{1,\mathrm{st}}}\bar{\mathbf{H}}_{T,x}^{N}(\bar{\varphi}_{S,y}\mathbf{1}_{S\leq\mathfrak{t}_{1,\mathrm{st}}})\|_{1;\mathbb{T}_{N}} \ &\leq \ \|\bar{\mathbf{H}}_{T,x}^{N}(\bar{\varphi}_{S,y})\|_{\mathfrak{t}_{1,\mathrm{st}};\mathbb{T}_{N}}.
\end{align}\normalsize\normalsize
\end{itemize}
Combining the previous two bullet points gives the required bound for all $T\leq1$ so we are done.
\end{proof}
%
%
%
\section{Martingale Inequalities}
We give here a generalization of the martingale estimate in Lemma 3.1 in \cite{DT}; the proof of the result in \cite{DT} is specific to the microscopic Cole-Hopf transform while for this paper we will need it for other processes as well. Roughly speaking, Lemma 3.1 in \cite{DT} requires a short-time bound for the microscopic Cole-Hopf transform that holds at the level of $\|\|_{\omega;2p}$-norms; results like Lemma \ref{lemma:TRGTProp2}, which holds for processes we are interested in, are for very a different norm and thus not engineered for this.
\begin{lemma}\label{lemma:MG}
Retain the setting of \emph{Lemma 3.1} in \cite{DT} and consider fundamental solutions $\mathbf{J}^{N},\bar{\mathbf{J}}^{N}$ defined prior to the statement of \emph{Proposition \ref{prop:CtifySFS}}. Recall \emph{Definition \ref{definition:tShortweights}} and consider the following in \emph{Lemma 3.1} in \cite{DT} in which $0\leq\mathfrak{t}_{1}\leq\mathfrak{t}_{2}$ are fixed. In what follows, let $\lfloor R\rfloor$ be the largest element in $N^{-2}\Z$ that is bounded above by $R$ to account for the $N^{2}$ time-scaling.
\begin{itemize}[leftmargin=*]
\item We define the following local quadratic function of $\varphi:\R_{\geq0}\times\Z^{2}\to\R$ where $\mathfrak{m}$ is the uniformly bounded maximal jump-length:
\small\begin{align}
\wt{\varphi}_{R,x,w}^{\mathfrak{t}_{1},\mathfrak{t}_{2}} \ &\overset{\bullet}= \ {\sup}_{\mathfrak{r}'\in[\mathfrak{t}_{1},\mathfrak{t}_{2}): \ \lfloor\mathfrak{r}'\rfloor=\lfloor R\rfloor}{\sup}_{|\mathfrak{j}|\leq\mathfrak{m}} |\varphi_{\mathfrak{r}',x,w+\mathfrak{j}}\varphi_{\mathfrak{r}',x,w}|.
\end{align}\normalsize\normalsize
\end{itemize}
Given any $p\geq1$ and $\mathfrak{t}_{1}\leq\mathfrak{t}_{2}$ and $\varphi:\R_{\geq0}\times\Z^{2}\to\R$, all deterministic, we have, again recalling notation in \emph{Definition \ref{definition:tShortweights}},
\small\begin{align}
\|\int_{\mathfrak{t}_{1}}^{\mathfrak{t}_{2}}\sum_{w\in\Z}\varphi_{R,x,w} \cdot \mathbf{J}_{S,T+R,w,y}^{N} \d\xi_{R,w}^{N}\|_{\omega;2p}^{2} \ &\lesssim_{p,\mathfrak{m}} \ \int_{\mathfrak{t}_{1}}^{\mathfrak{t}_{2}}N\left(\sup_{\mathfrak{t}_{1}\leq\mathfrak{r}\leq\mathfrak{t}_{2}}\sup_{w\in\Z}e_{N,w,y}^{2}\|\mathbf{J}_{S,T+\mathfrak{r},w,y}^{N}\|_{\omega;2p}^{2}\right)\sum_{w\in\Z}\wt{\varphi}_{R,x,w}^{\mathfrak{t}_{1},\mathfrak{t}_{2}}e_{N,w,y}^{-2} \ \d R. \label{eq:MGII}
\end{align}\normalsize\normalsize
We have the same bound after replacing $\mathbf{J}^{N}$ by $\bar{\mathbf{J}}^{N}$, $\Z$ by $\mathbb{T}_{N}$, $e_{N,w,y}$ by $e_{N,w,y}^{\mathrm{cpt}}$, $\varphi$ by $\bar{\varphi}: \R_{\geq0}\times\mathbb{T}_{N}^{2}\to\R$, and $\wt{\varphi}^{\mathfrak{t}_{1},\mathfrak{t}_{2}}$ by
\small\begin{align}
\wt{\bar{\varphi}}_{R,x,w}^{\mathfrak{t}_{1},\mathfrak{t}_{2}} \ &\overset{\bullet}= \ {\sup}_{\mathfrak{r}'\in[\mathfrak{t}_{1},\mathfrak{t}_{2}): \ \lfloor\mathfrak{r}'\rfloor=\lfloor R\rfloor}{\sup}_{|\mathfrak{j}|\leq\mathfrak{m}} \mathbf{1}_{w+\mathfrak{j}\in\mathbb{T}_{N}}|\bar{\varphi}_{\mathfrak{r}',x,w+\mathfrak{j}}\bar{\varphi}_{\mathfrak{r}',x,w}|.
\end{align}\normalsize\normalsize
Recall $\mathbf{Y}^{N}$ from the beginning of \emph{Section \ref{section:KPZ3}} and recall $\mathbf{D}^{N}$ in \emph{Definition \ref{definition:KPZ12}}. We also have the following estimates for $\mathbf{X}^{N}=\mathbf{Y}^{N},\mathbf{D}^{N}$ and for any uniformly bounded adapted process $\mathbf{R}^{N}:\R_{\geq0}\times\mathbb{T}_{N}\to\R$:
\begin{subequations}
\small\begin{align}
\|\int_{\mathfrak{t}_{1}}^{\mathfrak{t}_{2}}\sum_{w\in\mathbb{T}_{N}}\bar{\varphi}_{R,x,w} \cdot \mathbf{X}_{R,w}^{N}\d\xi_{R,w}^{N}\|_{\omega;2p}^{2} \ &\lesssim_{p,\mathfrak{m}} \ \int_{\mathfrak{t}_{1}}^{\mathfrak{t}_{2}} N\left(\sup_{w\in\mathbb{T}_{N}}\|\mathbf{X}_{\lfloor R\rfloor,w}^{N}\|_{\omega;2p}^{2}\right)\sum_{w\in\mathbb{T}_{N}}\wt{\bar{\varphi}}_{R,x,w}^{\mathfrak{t}_{1},\mathfrak{t}_{2}} \ \d R \label{eq:MGIII} \\
\|\int_{\mathfrak{t}_{1}}^{\mathfrak{t}_{2}}\sum_{w\in\mathbb{T}_{N}}\bar{\varphi}_{R,x,w} \cdot \mathbf{R}_{R,w}^{N}\d\xi_{R,w}^{N}\|_{\omega;2p}^{2} \ &\lesssim_{p,\mathfrak{m}} \ \int_{\mathfrak{t}_{1}}^{\mathfrak{t}_{2}} N \sum_{w\in\mathbb{T}_{N}}\wt{\bar{\varphi}}_{R,x,w}^{\mathfrak{t}_{1},\mathfrak{t}_{2}} \ \d R. \label{eq:MGIV}
\end{align}\normalsize\normalsize
\end{subequations}
\end{lemma}
We first define the kernel for the total mass of a branching random walk to study short-time behavior of processes in Lemma \ref{lemma:MG}. We will first recall the discrete approximations of Laplacians given by $\mathscr{L}^{!!}$ in Definition \ref{definition:HEAT} and $\bar{\mathscr{L}}^{!!}$ in Definition \ref{definition:HEATCPT}.
\begin{definition}
Define the kernel $\mathbf{L}^{N}$ as the function on $\R_{\geq0}^{2}\times\Z^{2}$ satisfying $\mathbf{L}^{N}_{S,S,x,y} = \mathbf{1}_{x=y}$ and the \emph{deterministic} equation
\small\begin{align}
\partial_{T}\mathbf{L}_{S,T,x,y}^{N} \ &= \ \mathscr{L}^{!!}\mathbf{L}_{S,T,x,y}^{N}+N^{\frac32}{\sum}_{|k|\leq N^{\beta_{X}}}\mathbf{L}_{S,T,x+k,y}^{N}. \label{eq:MGL}
\end{align}\normalsize\normalsize
We define $\bar{\mathbf{L}}^{N}$ to be the solution to the same equation after replacing $\mathbf{L}^{N}$ by $\bar{\mathbf{L}}^{N}$ and $\mathscr{L}^{!!}$ by $\bar{\mathscr{L}}^{!!}$ and $\Z$ by $\mathbb{T}_{N}$. We will additionally define the maximal processes $\mathbf{L}^{N,\max}_{S,T,x,y} \overset{\bullet}= \sup_{0\leq\mathfrak{r}\leq N^{-2}}\mathbf{L}^{N}_{S,T+\mathfrak{r},x,y}$ and $\bar{\mathbf{L}}^{N,\max}_{S,T,x,y} \overset{\bullet}= \sup_{0\leq\mathfrak{r}\leq N^{-2}}\bar{\mathbf{L}}^{N}_{S,T+\mathfrak{r},x,y}$.
\end{definition}
The following auxiliary estimate for the two maximal kernels defined immediately above will help control their short-time behavior and, as we later explain in the proof of Lemma \ref{lemma:MG}, the short-time behavior of the processes of interest in the statement of Lemma \ref{lemma:MG} too. The proof of the following is an elementary deterministic idea for linear ODEs; it is a substitute for short-time estimates on the microscopic Cole-Hopf transform which are key to the martingale estimate in Lemma 3.1 in \cite{DT} but which are inaccessible without a formula in terms of the particle system (there is such formula for the microscopic Cole-Hopf transform).
\begin{lemma}\label{lemma:MG2}
Recall \emph{Definition \ref{definition:tShortweights}} and $\beta_{X}=\frac13+\e_{X,1}$. Define $\wt{e}_{N,x,y}=\exp(N^{-\beta_{X}}\mathbf{d}_{x,y})$ and $\wt{e}_{N,x,y}^{\mathrm{cpt}}=\exp(N^{-\beta_{X}}\mathbf{d}_{x,y}^{\mathrm{cpt}})$. We have
\small\begin{align}
{\sum}_{y\in\Z}\mathbf{L}^{N,\max}_{0,0,x,y}e_{N,x,y}^{2} \ + \ {\sum}_{y\in\mathbb{T}_{N}}\bar{\mathbf{L}}^{N,\max}_{0,0,x,y}(e_{N,x,y}^{\mathrm{cpt}})^{2} \ \leq \ {\sum}_{y\in\Z}\mathbf{L}^{N,\max}_{0,0,x,y}\wt{e}_{N,x,y}^{2} \ + \ {\sum}_{y\in\mathbb{T}_{N}}\bar{\mathbf{L}}^{N,\max}_{0,0,x,y}(\wt{e}_{N,x,y}^{\mathrm{cpt}})^{2} \ \lesssim_{\mathfrak{m}} \ 1.
\end{align}\normalsize\normalsize
\end{lemma}
\begin{proof}
The first bound follows via $e_{N,x,y}\leq\wt{e}_{N,x,y}$, and the same for $\mathrm{cpt}$-versions, as $N^{-3/4}\leq N^{-\beta_{X}}$; see Definition \ref{definition:tShortweights}. We first prove the claimed bound for $\mathbf{L}_{0,\mathfrak{t},x,y}^{N}$ for $\mathfrak{t}\leq N^{-2}$ instead of $\mathbf{L}^{N,\max}$. We use the PDE for $\mathbf{L}^{N}$ to get, with explanation after,
\small\begin{align}
{\sum}_{y\in\Z}|\partial_{T}\mathbf{L}_{0,T,x,y}^{N}|\wt{e}_{N,x,y}^{2} \ &\lesssim \ N^{2}{\sum}_{y\in\Z}\left({\sum}_{|k|\leq\mathfrak{m}}\mathbf{L}_{0,T,x+k,y}^{N}\wt{e}_{N,x+k,y}^{2} + N^{-\frac12}{\sum}_{|k|\leq N^{\beta_{X}}}\mathbf{L}_{0,T,x+k,y}^{N}\wt{e}_{N,x+k,y}^{2}\right) \nonumber \\
&\lesssim_{\mathfrak{m}} \ N^{2}(1+N^{-\frac12+\beta_{X}}){\sup}_{x\in\Z}{\sum}_{y\in\Z}\mathbf{L}_{0,T,x,y}^{N}\wt{e}_{N,x,y}^{2} \ \lesssim \ N^{2}{\sup}_{x\in\Z}{\sum}_{y\in\Z}\mathbf{L}_{0,T,x,y}^{N}\wt{e}_{N,x,y}^{2} \ \overset{\bullet}= \ N^{2}\Phi_{T}. \nonumber
\end{align}\normalsize\normalsize
We used the inequality $\wt{e}_{N,x,y}\lesssim\wt{e}_{N,x+k,y}$ for any $|k|\leq N^{\beta_{X}}$ as well as $N^{\beta_{X}}\leq N^{1/2}$. As the last display is uniform in $x\in\Z$,
\small\begin{align}
\Phi_{\mathfrak{t}} \ \lesssim \ \Phi_{0}+\sup_{x\in\Z}\int_{0}^{\mathfrak{t}}|\left({\sum}_{y\in\Z}\partial_{T}\mathbf{L}_{0,T,x,y}^{N}\wt{e}_{N,x,y}^{2}\right)| \ \d T \ \leq \ \Phi_{0}+\sup_{x\in\Z}\int_{0}^{\mathfrak{t}}{\sum}_{y\in\Z}|\partial_{T}\mathbf{L}_{0,T,x,y}^{N}|\wt{e}_{N,x,y}^{2} \ \d T \ \lesssim_{\mathfrak{m}} \ \Phi_{0}+N^{2}\int_{0}^{\mathfrak{t}}\Phi_{T} \ \d T. \nonumber
\end{align}\normalsize\normalsize
Above we bounded each term in the sup defining $\Phi_{\mathfrak{t}}$ by initial data plus the integral of the absolute value of its time-derivative, and we took a sup over $x\in\Z$ of that bound to control $\Phi_{\mathfrak{t}}$ itself. If $\mathfrak{t}\leq N^{-2}$, the Gronwall inequality and the simple observation $\Phi_{0}\lesssim1$ (note $\mathbf{L}_{0,0,x,y}^{N}=\mathbf{1}_{x=y}$) give the claim for $\mathbf{L}_{0,\mathfrak{t},x,y}^{N}$. For $\mathbf{L}^{N,\max}$, we have, with explanation after, a similar integral bound:
\small\begin{align}
|\mathbf{L}_{0,0,x,y}^{N,\max}| \ &\lesssim \ \mathbf{1}_{x=y} + {\sup}_{0\leq\mathfrak{t}\leq N^{-2}}\left(N^{2}\int_{0}^{\mathfrak{t}}{\sum}_{|k|\leq\mathfrak{m}}|\mathbf{L}_{0,\mathfrak{r},x+k,y}^{N}| \d\mathfrak{r}+N^{\frac32}\int_{0}^{\mathfrak{t}}{\sum}_{|k|\leq N^{\beta_{X}}}|\mathbf{L}_{0,\mathfrak{r},x+k,y}^{N}|\d\mathfrak{r}\right) \label{eq:MG21} \\
&\leq \ \mathbf{1}_{x=y} + N^{2}\int_{0}^{N^{-2}}{\sum}_{|k|\leq\mathfrak{m}}|\mathbf{L}_{0,\mathfrak{r},x+k,y}^{N}| \d\mathfrak{r} + N^{\frac32}\int_{0}^{N^{-2}}{\sum}_{|k|\leq N^{\beta_{X}}}|\mathbf{L}_{0,\mathfrak{r},x+k,y}^{N}|\d\mathfrak{r}. \label{eq:MG22}
\end{align}\normalsize\normalsize
The first bound follows by writing the $\mathbf{L}^{N}$-kernel as initial data plus the integral of the absolute value of its time-derivative and then applying its PDE. The second line follows from noting the integrals on the RHS of the first line are of non-negative terms. Let us now recall $\wt{e}_{N,x,y}\lesssim\wt{e}_{N,x+k,y}$ given any $|k|\leq N^{\beta_{X}}$. Thus, we multiply \eqref{eq:MG22} and the LHS of \eqref{eq:MG21} by these $\wt{e}_{N,x,y}^{2}$-weights, sum over $y\in\Z$, apply the result for $\mathbf{L}^{N}$ instead of $\mathbf{L}^{N,\max}$ established in the above paragraph, and then integrate-in-time on the interval $[0,N^{-2}]$ to get the desired bound for $\mathbf{L}^{N,\max}$. A similar argument works to prove the $\bar{\mathbf{L}}^{N,\max}$-estimates.
\end{proof}
\begin{proof}[Proof of \emph{Lemma \ref{lemma:MG}}]
Let us first prove \eqref{eq:MGII}. As $\mathbf{J}^{N}$ is adapted in the forwards time-variable, we can follow the proof of Lemma 3.1 in \cite{DT} starting with (3.5) except we replace the supremum over $|\mathfrak{j}|\leq\mathfrak{m}$ with a sum, and we estimate each term in this sum. We now find a replacement for (3.7) in \cite{DT} for short-time control on $\mathbf{J}^{N}$. Recall $\mathbf{J}^{N}$ from before Proposition \ref{prop:CtifySFS}.
\begin{itemize}[leftmargin=*]
\item The process $\mathbf{J}^{N}$ solves a linear stochastic equation with a discrete-type Laplacian $\mathscr{L}^{!!}$, some gradient terms with coefficients that are at most order $N^{3/2}$, a random multiplicative potential of order $N^{3/2}$, and jumps of order $N^{-1/2}$. Note this equation is time-inhomogeneous. However, we may still use a Feynman-Kac representation for $\mathbf{J}^{N}$ where the potential is the sum of the aforementioned order $N^{3/2}$-potential and order $N^{-1/2}$ jumps, while the underlying random walk is the one whose generator is $\mathscr{L}^{!!}$ plus gradient-operators in the $\mathbf{J}^{N}$-equation. It is easy to see the transition probability of this random walk is bounded by the $\mathbf{L}^{N}$-kernel, since gradient terms in the $\mathbf{J}^{N}$-equation are controlled by non-$\mathscr{L}^{!!}$-operator terms in the $\mathbf{L}^{N}$-equation. 
\item We make the bullet point above precise. First, we let $\mathfrak{y}^{\star}$ denote the random walk in the previous bullet point independent of the particle system and we denote expectation with respect to $\mathfrak{y}^{\star}$ by $\mathbb{E}$. We also define the accumulation of order $N^{-1/2}$-jumps and the order $N^{3/2}$-potential in the $\mathbf{J}^{N}$-equation along a sample $\mathfrak{y}^{\star}$-path below in which $\kappa\in\R_{>0}$ is a universal constant:
\small\begin{align}
\mathbf{J}_{\lfloor R\rfloor,\mathfrak{t}} \ \overset{\bullet}= \ |\left(\int_{0}^{\mathfrak{t}} \d\xi_{T+\lfloor R\rfloor+\mathfrak{t}-\mathfrak{s},\mathfrak{y}^{\star}_{T+\lfloor R\rfloor+\mathfrak{s}}}^{N}\right)| + \kappa N^{3/2}\mathfrak{t}.
\end{align}\normalsize\normalsize
The integral on the RHS above, which is with respect to $\mathfrak{s}$, is a sum of jumps plus the integral of the order $N^{3/2}$ compensation drift in $\d\xi^{N}$; see (2.4) in \cite{DT}. By the Feynman-Kac formula given by Proposition 7.1 in Appendix 1 of \cite{KL}, which holds for time-inhomogeneous processes after straightforward replacements in its proof, for $\mathfrak{t}\leq N^{-2}$ we have, with explanation after,
\small\begin{align}
(\mathbf{J}_{S,T+\lfloor R\rfloor+\mathfrak{t},w,y}^{N})^{2} \ &\lesssim \ \left(\mathbb{E}_{\mathfrak{y}_{T+\lfloor R\rfloor}^{\star}=w}\exp\left(\mathbf{J}_{\lfloor R\rfloor,\mathfrak{t}}\right)\mathbf{J}_{S,T+\lfloor R\rfloor,\mathfrak{y}_{T+\lfloor R\rfloor+\mathfrak{t}}^{\star},y}^{N}\right)^{2} \nonumber \\
&\leq \ \left(\mathbb{E}_{\mathfrak{y}_{T+\lfloor R\rfloor}^{\star}=w}\exp\left(2\mathbf{J}_{\lfloor R\rfloor,\mathfrak{t}}\right)\right)\left(\mathbb{E}_{\mathfrak{y}_{T+\lfloor R\rfloor}^{\star}=w}|\mathbf{J}_{S,T+\lfloor R\rfloor,\mathfrak{y}_{T+\lfloor R\rfloor+\mathfrak{t}}^{\star},y}^{N}|^{2}\right) \nonumber \\
&\lesssim \ \left(\mathbb{E}_{\mathfrak{y}_{T+\lfloor R\rfloor}^{\star}=w}\exp\left(2\mathbf{J}_{\lfloor R\rfloor,\mathfrak{t}}\right)\right){\sum}_{z\in\Z}\mathbf{L}_{0,0,w,z}^{N,\max}|\mathbf{J}_{S,T+\lfloor R\rfloor,z,y}^{N}|^{2}. \label{eq:MG1}
\end{align}\normalsize\normalsize
The first estimate is the Feynman-Kac formula because $\mathbf{J}_{\lfloor R\rfloor,\mathfrak{t}}$ bounds all jumps and potentials in the $\mathbf{J}^{N}$-equation. The second bound follows by the Cauchy-Schwarz inequality. The third/final bound follows by recalling the transition probability of $\mathfrak{y}^{\star}$ is controlled by the $\mathbf{L}^{N}$-kernel, as noted in the first bullet point, which is then controlled by the $\mathbf{L}^{N,\max}$-kernel. This will be our replacement for (3.7) in \cite{DT} because $\mathbf{L}^{N,\max}$ is deterministic so beyond the first exponential factor in\eqref{eq:MG1} the randomness in \eqref{eq:MG1} is $\lfloor R\rfloor$-measurable. We also have the following conditional expectation estimate, where $\mathfrak{n}_{\lfloor R\rfloor,\mathfrak{t},w}$ denotes the number of times that any Poisson clock across $w\in\Z$ rings in the interval $T+\lfloor R\rfloor+[0,\mathfrak{t}]$ and $\E_{\lfloor R\rfloor}$ conditions on $\lfloor R\rfloor$-data. Here $\kappa\in\R_{>0}$ is, again, universal, and $\mathfrak{t}\leq N^{-2}$, and $p\in\R_{\geq1}$. We give more explanation for each bound below after:
\small\begin{align}
\E_{\lfloor R\rfloor}|\mathfrak{n}_{\lfloor R\rfloor,\mathfrak{t},w}|^{2p}\mathbb{E}_{\mathfrak{y}_{T+\lfloor R\rfloor}^{\star}=w}\exp\left(2p\mathbf{J}_{\lfloor R\rfloor,\mathfrak{t}}\right) \ &\lesssim \ {\sum}_{\mathfrak{k}\in\Z_{\geq1}} \mathrm{e}^{-\mathfrak{k}} \E_{\lfloor R\rfloor}\left(|\mathfrak{n}_{\lfloor R\rfloor,\mathfrak{t},w}|^{2p}\exp\left(2p\kappa N^{-1/2}{\sum}_{z\in w+\llbracket-10N^{\beta_{X}}\mathfrak{k},10N^{\beta_{X}}\mathfrak{k}\rrbracket}\mathfrak{n}_{\lfloor R\rfloor,\mathfrak{t},z}\right)\right) \nonumber \\
&\lesssim_{p} \ \mathfrak{t}{\sum}_{\mathfrak{k}\in\Z_{\geq1}}\mathrm{e}^{-\frac12\mathfrak{k}} \ \lesssim \ \mathfrak{t}. \label{eq:MG2}
\end{align}\normalsize\normalsize
First recall the $\mathfrak{y}^{\star}$-walk has transition probability bounded by $\mathbf{L}^{N,\max}$. The first inequality follows by conditioning on $\mathfrak{y}^{\star}$ to hit $w+\llbracket-10N^{\beta_{X}}\mathfrak{k},10N^{\beta_{X}}\mathfrak{k}\rrbracket\setminus\llbracket-10N^{\beta_{X}}(\mathfrak{k}-1),10N^{\beta_{X}}(\mathfrak{k}-1)\rrbracket$, which happens with probability $\mathscr{O}(\mathrm{e}^{-\mathfrak{k}})$ because of the exponential tail estimate on $\mathbf{L}^{N,\max}$ in Lemma \ref{lemma:MG2}, but stay in $w+\llbracket-10N^{\beta_{X}}\mathfrak{k},10N^{\beta_{X}}\mathfrak{k}\rrbracket$, and then bounding $\mathbf{J}_{\lfloor R\rfloor,\mathfrak{t}}$ for this $\mathfrak{y}^{\star}$-walk-trajectory by the total number of ringings $\mathfrak{n}_{\lfloor R\rfloor,\mathfrak{t},z}$ across all points $z\in w+\llbracket-10N^{\beta_{X}}\mathfrak{k},10N^{\beta_{X}}\mathfrak{k}\rrbracket$. The second bound follows by elementary Poisson calculations as in (3.8) in \cite{DT}. We note the exponential in $\E_{\lfloor R\rfloor}$, per-sum-index $\mathfrak{k}\in\Z_{\geq1}$, is the exponential of a Poisson variable, with speed of order $N^{\beta_{X}}\mathfrak{k}$, scaled by order $N^{-1/2}$. The Laplace transform of such a Poisson variable is exponential in $\mathscr{O}_{p}(N^{-1/2+\beta_{X}}\mathfrak{k})$, which is then dwarfed by the $\mathrm{e}^{-\mathfrak{k}}$-factors in the $\mathfrak{k}\in\Z_{\geq1}$-sum.
\end{itemize}
We proceed after (3.5) in \cite{DT} but having replaced the supremum over $|\mathfrak{j}|\leq\mathfrak{m}$ with a sum and skipping (3.6) in \cite{DT}. Namely, by \eqref{eq:MG1}, the conditional expectation estimate \eqref{eq:MG2}, and the $\lfloor R\rfloor$-measurable nature of the second factor in \eqref{eq:MG1}, we show the following parallel to the bound after (3.8) in \cite{DT}; the $N$-factor comes by change-of-variables to account for the $N^{2}$-speed:
\small\begin{align}
\|\int_{\mathfrak{t}_{1}}^{\mathfrak{t}_{2}}\sum_{w\in\Z}\varphi_{R,x,w} \cdot \mathbf{J}_{S,T+R,w,y}^{N} \d\xi_{R,w}^{N}\|_{\omega;2p}^{2} \ &\lesssim_{p,\mathfrak{m}} \ \sum_{|\mathfrak{j}|\leq\mathfrak{m}}\int_{\mathfrak{t}_{1}}^{\mathfrak{t}_{2}}N\sum_{w\in\Z}\wt{\varphi}_{R,x,w}^{\mathfrak{t}_{1},\mathfrak{t}_{2}} \cdot \|\sum_{z\in\Z}\mathbf{L}_{0,0,w+\mathfrak{j},z}^{N,\max}|\mathbf{J}_{S,T+\lfloor R\rfloor,z,y}^{N}|^{2}\|_{\omega;p} \ \d R \\
&\lesssim_{p,\mathfrak{m}} \ \sum_{|\mathfrak{j}|\leq\mathfrak{m}}\int_{\mathfrak{t}_{1}}^{\mathfrak{t}_{2}}N\sum_{w\in\Z}\wt{\varphi}_{R,x,w}^{\mathfrak{t}_{1},\mathfrak{t}_{2}} \cdot \sum_{z\in\Z}\mathbf{L}_{0,0,w+\mathfrak{j},z}^{N,\max}\|\mathbf{J}_{S,T+\lfloor R\rfloor,z,y}^{N}\|_{\omega;2p}^{2} \ \d R \label{eq:MGII4} \\
&\lesssim \ \sum_{|\mathfrak{j}|\leq\mathfrak{m}}\int_{\mathfrak{t}_{1}}^{\mathfrak{t}_{2}}N\sum_{w\in\Z}\wt{\varphi}_{R,x,w}^{\mathfrak{t}_{1},\mathfrak{t}_{2}} e_{N,w,y}^{-2} \cdot e_{N,w,y}^{2}\sum_{z\in\Z}\mathbf{L}_{0,0,w+\mathfrak{j},z}^{N,\max}\|\mathbf{J}_{S,T+\lfloor R\rfloor,z,y}^{N}\|_{\omega;2p}^{2} \ \d R. \label{eq:MGII5}
\end{align}\normalsize\normalsize
The bound \eqref{eq:MGII4} follows from the previous line combined with the triangle inequality for the $\|\|_{\omega;2p}$-norm; for this step recall the $\mathbf{L}^{N,\max}$-kernel is deterministic. The final bound \eqref{eq:MGII5} follows from multiplying by $e_{N,w,y}^{-2}$ and $e_{N,w,y}^{2}$. At this point, it suffices to use the following in which we observe $e_{N,w,y}^{2}\lesssim e_{N,w,z}^{2}e_{N,z,y}^{2}\lesssim e_{N,w+\mathfrak{j},z}^{2}e_{N,z,y}^{2}$ for any $|\mathfrak{j}|\leq\mathfrak{m}$; this follows by definition of these exponential weights in Definition \ref{definition:tShortweights}, the bound $\mathfrak{m}\lesssim1$, and elementary triangle inequality. We deduce
\small\begin{align}
e_{N,w,y}^{2}{\sum}_{z\in\Z}\mathbf{L}_{0,0,w+\mathfrak{j},z}^{N,\max}\cdot\|\mathbf{J}_{S,T+\lfloor R\rfloor,z,y}^{N}\|_{\omega;2p}^{2} \ &\lesssim \ {\sum}_{z\in\Z}\mathbf{L}_{0,0,w+\mathfrak{j},z}^{N,\max}e_{N,w+\mathfrak{j},z}^{2} \cdot e_{N,z,y}^{2}\|\mathbf{J}_{S,T+\lfloor R\rfloor,z,y}^{N}\|_{\omega;2p}^{2} \nonumber \\
&\lesssim \ {\sup}_{\substack{\mathfrak{t}_{1}\leq \mathfrak{r}\leq\mathfrak{t}_{2} \\ z \in \Z}}e_{N,z,y}^{2}\|\mathbf{J}_{S,T+\mathfrak{r},z,y}^{N}\|_{\omega;2p}^{2}. \nonumber
\end{align}\normalsize\normalsize
The second estimate in the last display follows via pulling out the space-time supremum of the $\|\|^{2}$-factor in the $\Z$-summation, and then estimating the remaining $\Z$-summation above using Lemma \ref{lemma:MG2}. If we plug the above display into \eqref{eq:MGII5} and multiply by $3\mathfrak{m}$ to bound the sum over $|\mathfrak{j}|\leq\mathfrak{m}$, then we get \eqref{eq:MGII}. To establish the estimate \eqref{eq:MGII} but with compactification replacements, the same argument works if we replace $\Z$ by $\mathbb{T}_{N}$, $\mathscr{L}^{!!}$ by $\bar{\mathscr{L}}^{!!}$, gradients on $\Z$ by those on $\mathbb{T}_{N}$, and $\mathbf{L}^{N,\max}$ by $\bar{\mathbf{L}}^{N,\max}$.

To establish \eqref{eq:MGIII}, observe that $\mathbf{X}^{N}$ solves a similar stochastic equation as $\bar{\mathbf{J}}^{N}$, so that \eqref{eq:MGII4} holds upon replacing $\mathbf{J}^{N}$ by $\mathbf{X}^{N}$ including the necessary replacements $\wt{\varphi}$ by $\wt{\bar{\varphi}}$ and $\Z$ by $\mathbb{T}_{N}$ and $\mathbf{L}^{N,\max}$ by $\bar{\mathbf{L}}^{N,\max}$. At that point we pull out the supremum over $z\in\Z$ of $\|\mathbf{X}^{N}_{\lfloor R\rfloor,z}\|_{\omega;2p}^{2}$ and estimate the 1-norm of the $\bar{\mathbf{L}}^{N,\max}$-kernel with Lemma \ref{lemma:MG2}. This provides \eqref{eq:MGIII} after again multiplying the resulting estimate by $3\mathfrak{m}$. To get \eqref{eq:MGIV}, we actually instead follow the proof for Lemma 3.1 in \cite{DT}, but with a replacement of $\Z$ by $\mathbb{T}_{N}$ until (3.5) therein, so we have (3.5) therein but replacing the $Z$-process with $\mathbf{R}$. At that point, we drop the $\mathbf{R}$-factors because they are uniformly bounded and proceed like in the rest of the proof for Lemma 3.1 in \cite{DT} starting at (3.8), though for our version of (3.8) here we can do without the conditioning. This proves \eqref{eq:MGIV} and therefore finishes the proof.
\end{proof}
%
%
%
\section{Notation}
\begin{itemize}[leftmargin=*]
\item ``Arbitrarily small/large but universal" constants are taken as small/large as we want but independent of $N$. Also, we take $N$ sufficiently large depending only on some universal constants to make certain statements true, such as $100 \leq N^{\e}$.
\item We always let $\mathfrak{m}\in\Z_{>0}\cup\{+\infty\}$ denote the maximal jump-length in the particle system.
\item Provided any $a,b\in\R$, we will define the discretized interval $\llbracket a,b\rrbracket \overset{\bullet}= [a,b]\cap\Z$.
\item We use big-Oh notation $\mathscr{O}$, and $\lesssim$ is synonymous with $\mathscr{O}$. Dependences of implied constants in $\lesssim$ and $\mathscr{O}$ will be recorded as subscripts. For example $\lesssim_{I}$ means the implied constant depends only on a set $I$. We also use the pre-factor $\kappa_{I}$ synonymously with $\lesssim_{I}$ or $\mathscr{O}_{I}$ so that the constants $\kappa_{I}$ and $\kappa_{i_{1},\ldots,i_{\mathfrak{n}}}$ depend only on the set $I=\{i_{1},\ldots,i_{\mathfrak{n}}\}$.
\item For any $X\in\Z$ define the shift-map $\tau_{X}:\Omega_{\Z}\to\Omega_{\Z}$ via $(\tau_{X}\eta)_{Z} = \eta_{Z+X}$. For any $\varphi:\Omega_{\Z}\to\R$, set $\tau_{X}\varphi = \varphi \circ \tau_{X}:\Omega_{\Z}\to\R$.
\item For any $\varphi:\Omega_{\Z}\to\R$, its support is defined as the smallest subset $\mathfrak{I}\subseteq\Z$ such that $\varphi$ depends only on spins $\eta_{x}$ for $x\in\mathfrak{I}$.
\item Provided any \emph{finite set} $I$, we define the averaged sum $\wt{{{\sum}}}_{i\in I} = |I|^{-1}{{\sum}}_{i\in I}$ where $|I|$ denotes the cardinality/size of $I$.
\item For any operator $\mathscr{T}$, each superscript $!$ denotes scaling by $N$. In particular, define the discrete gradient $\grad_{k}\varphi_{x} = \varphi_{x+k}-\varphi_{x}$ for any $k\in\Z$ and the discrete Laplacian $\Delta_{k}\varphi_{x} = \varphi_{x+k}+\varphi_{x-k}-2\varphi_{x}$ for any $k\in\Z$. Define $\grad_{k}^{!} = N\grad_{k}$ and $\Delta_{k}^{!!} = N^{2}\Delta_{k}$.
\item Define the \emph{torus} $\mathbb{T}_{N}\overset{\bullet}=\llbracket-N^{5/4+\e_{\mathrm{cpt}}},N^{5/4+\e_{\mathrm{cpt}}}\rrbracket\subseteq\Z$ with $\e_{\mathrm{cpt}}\in\R_{>0}$ arbitrarily small but universal.
\item For $k\in\Z$ define $\bar{\grad}_{k}\varphi_{x} = \varphi_{x\wt{+}k}-\varphi_{x}$ and $\bar{\Delta}_{k}\varphi_{x} = \varphi_{x\wt{+}k}+\varphi_{x\wt{-}k}-2\varphi_{x}$. Here $\wt{+}$/$\wt{-}$ denote the induced addition/subtraction on the \emph{torus} $\mathbb{T}_{N}$. This addition/subtraction notation is specific to this bullet point. Define $\bar{\grad}_{k}^{!}=N\bar{\grad}_{k}$ and $\bar{\Delta}^{!!}_{k}=N^{2}\bar{\Delta}_{k}$.
\item The macroscopically smooth cutoff function $\chi$ is catered to the torus $\mathbb{T}_{N}$; it is defined at the beginning of Section \ref{section:Ctify}.
\item For any $\mathfrak{t}\in\R_{\geq0}$ and any subset $\mathbb{X}\subseteq\R$ we define a discretization $\mathfrak{I}_{\mathfrak{t}}\overset{\bullet}=\{\mathfrak{t}\cdot\mathfrak{j}N^{-100}\}_{\mathfrak{j}=0}^{N^{100}}$ and the norms/norm-type quantities
\begin{align*}
\|\varphi\|_{\mathfrak{t};\mathbb{X}}\ &\overset{\bullet}= \ \sup_{0\leq\mathfrak{s}\leq\mathfrak{t}}\sup_{x\in\mathbb{X}}|\varphi_{\mathfrak{s},x}| \ \quad\mathrm{and}\quad [\varphi]_{\mathfrak{t};\mathbb{X}}\ \overset{\bullet}= \ \sup_{\mathfrak{s}\in\mathfrak{I}_{\mathfrak{t}}}\sup_{x\in\mathbb{X}}|\varphi_{\mathfrak{s},x}| \quad\mathrm{and}\quad \langle\varphi\rangle_{\mathfrak{t};\mathbb{X}} \ \overset{\bullet}= \ 1+\|\varphi\|_{\mathfrak{t};\mathbb{X}}^{2}.
\end{align*}
\item Define $\mathbf{d}_{x,y}=|x-y|$ as the usual absolute value distance for $x,y\in\Z$ and $\mathbf{d}^{\mathrm{cpt}}_{x,y}$ as geodesic distance on the torus $\mathbb{T}_{N}$.
\item For any $p\in\R_{\geq1}$, define $\|\varphi\|_{\omega;p}$ as the $p$-norm \emph{over all randomness in the exclusion process} including the initial measure.
\item Provided any $S,T\in\R_{\geq0}$, we define $\mathfrak{s}_{S,T} \overset{\bullet}= |T-S|$ usually as an integrable singularity in a time-integral.
\end{itemize}
%
%
%


\end{document}